\documentclass[a4paper]{amsart}
\usepackage{amsmath, amssymb, amsfonts, amscd}
\usepackage{longtable}
\usepackage{young}
\def\pdf{TF}
\iffalse
\usepackage[dvipdfm]{color}\usepackage[dvipdfm,%
 bookmarks=true,bookmarksnumbered=true,bookmarkstype=toc%
,colorlinks=true,linkcolor=blue,citecolor=blue,filecolor=blue,pagecolor=blue,urlcolor=blue%
 ]{hyperref}
\usepackage[dvipdfm]{graphicx}
\else
\usepackage{graphicx, color}
\usepackage[hypertex]{hyperref}
\fi
\if\pdf
\usepackage[all,ps,dvips]{xy}
\else
\usepackage[all]{xy}
\fi
\numberwithin{equation}{section}
\theoremstyle{plain}
 \newtheorem{thm}{Theorem}[section]
 \newtheorem{cor}[thm]{Corollary}
 \newtheorem{lem}[thm]{Lemma}
 \newtheorem{prop}[thm]{Proposition}
\theoremstyle{definition}
 \newtheorem{defn}[thm]{Definition}
 \newtheorem{exmp}[thm]{Example}
 
 \newtheorem{conj}[thm]{Conjecture}
\theoremstyle{remark}
 \newtheorem{rem}[thm]{Remark}

\DeclareMathOperator{\ord}{ord}
\DeclareMathOperator{\RE}{Re}
\DeclareMathOperator{\IM}{Im}

\DeclareMathOperator{\rank}{rank}

\DeclareMathOperator{\Ridx}{Ridx}
\DeclareMathOperator{\idx}{idx}
\DeclareMathOperator{\End}{End}

\DeclareMathOperator{\Ad}{Ad}
\DeclareMathOperator{\Red}{R}
\DeclareMathOperator{\Adei}{Adei}
\DeclareMathOperator{\Ade}{Ade}
\DeclareMathOperator{\RAdei}{RAdei}
\DeclareMathOperator{\RAde}{RAde}
\DeclareMathOperator{\RAd}{RAd}
\DeclareMathOperator{\RAdeiL}{RAdeiL}
\DeclareMathOperator{\RAdeL}{RAdeL}
\DeclareMathOperator{\RAdL}{RAdL}
\DeclareMathOperator{\AdL}{AdL}
\DeclareMathOperator{\AdeL}{AdeL}
\DeclareMathOperator{\AdeiL}{AdeiL}
\DeclareMathOperator{\AdV}{AdV}
\DeclareMathOperator{\RAdV}{RAdV}

\DeclareMathOperator{\Lap}{L}

\DeclareMathOperator{\id}{id}
\DeclareMathOperator{\Pidx}{Pidx}
\DeclareMathOperator{\spc}{spc}
\DeclareMathOperator{\supp}{supp}
\DeclareMathOperator{\Top}{Top}
\DeclareMathOperator{\Trace}{Trace}
\DeclareMathOperator{\MC}{MC}
\DeclareMathOperator{\Mon}{Mon}
\DeclareMathOperator{\Arg}{Arg}
\DeclareMathOperator{\trace}{trace}
\DeclareMathOperator{\codim}{codim}
\def\p{\partial}
\title[Weyl algebra and Fuchsian differential equations]{Fractional calculus of Weyl algebra and Fuchsian differential equations}%
\author{Toshio Oshima}
\thanks{
\hspace*{12pt}Supported by Grant-in-Aid for Scientific Researches (A), 
No.\ 20244008, Japan Society of Promotion of Science\\
\hspace*{12pt}Toshio Oshima, Graduate School of Mathematical Sciences,
University of Tokyo, 7-3-1, Komaba, Meguro-ku, Tokyo 153-8914, Japan\\
\hspace*{12pt}{\sl e-mail address}\,:\,\texttt{oshima@ms.u-tokyo.ac.jp}}
\makeindex
\begin{document}
\begin{abstract}
We give a unified interpretation of confluences, contiguity 
relations and Katz's middle convolutions for linear ordinary differential 
equations with polynomial coefficients and their generalization 
to partial differential equations.  
The integral representations and series expansions of their solutions 
are also within our interpretation.
As an application to Fuchsian differential equations on
the Riemann sphere, we construct a universal model of
Fuchsian differential equations with a given spectral type,
in particular, we construct single ordinary differential 
equations without apparent singularities corresponding to
the rigid local systems, whose existence was an open problem 
presented by Katz.
Furthermore we obtain an explicit solution to the 
connection problem for the rigid Fuchsian differential equations
and the necessary and sufficient condition for their irreducibility.
We give many examples calculated by our fractional calculus.
\end{abstract}
\maketitle
\tableofcontents
\section{Introduction}
Gauss hypergeometric functions and the functions in their family, 
such as Bessel functions, Whittaker functions, Hermite functions, 
Legendre polynomials and Jacobi polynomials etc.\ are the most fundamental 
and important special functions (cf.~\cite{EMO, Wa, WW}).  
Many formulas related to the family have been studied and clarified 
together with the theory of ordinary differential equations, the theory 
of holomorphic functions and relations with other fields.
They have been extensively used in various fields of mathematics,
mathematical physics and engineering.

Euler studied the hypergeometric equation
\begin{equation}\label{eq:G}
 x(1-x)y''+\bigl(c-(a+b+1)x\bigr)y'-aby=0
\end{equation}
with constant complex numbers $a$, $b$ and $c$ and he got the solution
\begin{equation}\label{eq:GaussSeries}
  F(a,b,c;x) := \sum_{k=0}^\infty\frac{a(a+1)\cdots(a+k-1)\cdot
   b(b+1)\cdots(b+k-1)}{c(c+1)\cdots(c+k-1)\cdot k!}x^k.
\end{equation}
\index{hypergeometric equation/function!Gauss}%
The series $F(a,b,c;x)$ is now called Gauss hypergeometric series
or function and Gauss proved the Gauss summation formula
\begin{equation}\label{eq:Gausssum}
 F(a,b,c;1) = \frac{\Gamma(c)\Gamma(c-a-b)}{\Gamma(c-a)\Gamma(c-b)}
\end{equation}
when the real part of $c$ is sufficiently large.
Then in the study of this function an important concept was introduced by 
Riemann. That is the Riemann scheme 
\begin{equation}\label{eq:RSGauss}
 \begin{Bmatrix}
  x = 0 & 1 & \infty\\
      0 & 0 & a&\!\!;\ x\\
      1-c&c-a-b & b
 \end{Bmatrix}
\end{equation}
which describes the property of singularities of the function
and Riemann proved that this property characterizes the Gauss 
hypergeometric function.

The equation \eqref{eq:G} is a second order Fuchsian differential 
equation on the Riemann sphere with the three singular points $\{0,1,\infty\}$.
One of the main purpose of this paper is to generalize these results to
the general Fuchsian differential equation on the Riemann sphere.
In fact, our study will be applied to the following three kinds of 
generalizations. 

One of the generalizations of the Gauss hypergeometric family is the
hypergeometric family containing the generalized 
hypergeometric function ${}_nF_{n-1}(\alpha,\beta;x)$  
or the solutions of Jordan-Pochhammer equations.
Some of their global structures are concretely described as in the
case of the Gauss hypergeometric family.

The second generalization is a class of Fuchsian differential 
equations such as the Heun equation which is of order 2 and has 4 singular 
points in the Riemann sphere.
In this case, there appear \textsl{accessory parameters}.
The global structure of the generic solution is quite transcendental
and the Painlev\'e equation which describes
the deformations preserving the monodromies of solutions of
the equations with an apparent singular point is interesting 
and has been quite deeply studied and now it becomes an important 
field of mathematics.

The third generalization is a class of hypergeometric functions 
of several variables, such as Appell's hypergeometric functions
(cf.~\cite{AK}), Gelfand's generalized hypergeometric functions 
(cf.~\cite{Ge}) and Heckman-Opdam's hypergeometric functions
(cf.~\cite{HO}).
The author and Shimeno \cite{OS} studied the ordinary differential
equations satisfied by the restrictions of Heckman-Opdam's 
hypergeometric function on singular lines through
the origin and we found that some of the equations belong to the even 
family classified by Simpson \cite{Si}, which is now called a class of
\textsl{rigid} differential equations and belongs to the first 
generalization in the above.

The author's original motivation related to the study in this note
is a generalization of Gauss summation formula, namely, to calculate 
a connection coefficient for a solution of this even family, 
which is solved in \S\ref{sec:C} as a direct 
consequence of the general formula \eqref{eq:Icon} 
of certain connection coefficients described in Theorem~\ref{thm:c}.
This paper is the author's first step to a unifying approach
for these generalizations and the recent development
in general Fuchsian differential equations described below
with the aim of getting concrete and computable results.
In this paper, we will avoid intrinsic arguments and results
if possible and hence the most results can be implemented in
computer programs.
Moreover the arguments in this paper will be 
understood without referring to other papers.


Rigid differential equations are the differential equations which are 
uniquely determined by the data describing the local structure of 
their solutions at the singular points.
From the point of view of the monodromy of the solutions, the rigid
systems are the local systems which are uniquely determined by local
monodromies around the singular points and Katz \cite{Kz} 
studied rigid local systems by defining and using the operations called 
\textsl{middle convolutions} and \textsl{additions}, which enables us to construct and analyze 
all the rigid local systems.
In fact, he proved that any irreducible rigid local system is 
transformed into a trivial equation $\frac{du}{dz}=0$ by successive application
of the operations.  In another word, any irreducible rigid local system
is obtained by successive applications
of the operations to the trivial equation because the operations are 
invertible.

The arguments there are rather intrinsic by using perverse sheaves.
Dettweiler-Reiter \cite{DR, DR2} interprets Katz's operations on 
monodromy generators and those on the systems of Fuchsian differential 
equations of Schlesinger canonical form
\begin{equation}
  \frac{du}{dx} = \sum_{j=1}^p\frac{A_j}{x-c_j}u
\end{equation}
with constant square matrices $A_1,\dots,A_p$.

Here $A_j$ are called the residue matrices of the system at the singular points
$x=c_j$, which describe the local structure of the solutions.
For example, the eigenvalues of the monodromy generator at $x=c_j$
are $e^{2\pi\sqrt{-1}\lambda_1},\dots, e^{2\pi\sqrt{-1}\lambda_n}$,
where $\lambda_1,\dots,\lambda_n$ are eigenvalues of $A_j$.
The residue matrix of the system at $x=\infty$ equals $A_0:=-(A_1+\cdots+A_p)$.
These operations are useful also for non-rigid Fuchsian systems.

\index{Deligne-Simpson problem}
Related to the Riemann-Hilbert problem, there is a natural problem to
determine the condition on matrices $B_0,B_1,\dots,B_p$ of 
Jordan canonical form such that there exists an irreducible system 
of Schlesinger canonical form with the residue matrices $A_j$ conjugate to 
$B_j$ for $j=0,\ldots,p$, respectively.
An obvious necessary condition is the equality $\sum_{j=0}^p\Trace B_j=0$.
A similar problem for monodromy generators, namely its multiplicative version, 
is equally formulated.
The latter is called a \textsl{mutiplicative} version and the former is called
an \textsl{additive} version.
Kostov \cite{Ko, Ko2} called them Deligne-Simpson problems and 
gave an answer under a certain genericity condition.
We note that the addition is a kind of a gauge transformation 
\[u(x)\mapsto (x-c)^\lambda u(x)\]
and the middle convolution is essentially an Euler transformation or a transformation by 
an Riemann-Liouville integral 
\[u(x)\mapsto \frac1{\Gamma(\mu)}\int_c^x u(t)(x-t)^{\mu-1}dt\]
or a fractional derivation.

Crawley-Boevey \cite{CB} found a relation between the Deligne-Simpson problem
and representations of certain quivers and gave an explicit
answer for the additive Deligne-Simpson problem in terms of a 
Kac-Moody root system.

Yokoyama \cite{Yo2} defined operations called extensions and restrictions on
the systems of Fuchsian ordinary differential equations 
of Okubo normal form
\begin{equation}
  \bigl(x - T\bigr)\frac{du}{dx} = Au.
\end{equation}
Here $A$ and $T$ are constant square matrices such that $T$ are 
diagonalizable.  
He proved that the irreducible rigid system of Okubo 
normal form is transformed into a trivial equation $\frac{du}{dz}=0$ by 
successive applications of his operations if the characteristic exponents 
are generic.

The relation between Katz's operations and Yokoyama's operations 
is clarified by \cite{O4} and it is proved there that their algorithms 
of reductions of Fuchsian systems are equivalent and so are those of 
the constructions of the systems.

These operations are quite powerful and
in fact if we fix the number of accessory parameters of the systems, 
they are connected into a finite number of fundamental systems
(cf.~\cite[Proposition~8.1 and Theorem~10.2]{O3} and 
Proposition~\ref{prop:Bineq}), which is a generalization of the fact 
that the irreducible rigid Fuchsian system is connected to the trivial equation.

Hence it is quite useful to understand how does the property of the 
solutions transform under these operations.
In this point of view, the system of the equations,  
the integral representation and the monodromy of the solutions
are studied by \cite{DR, DR2, HY} in the case of the Schlesinger canonical 
form.  Moreover the equation describing the deformation preserving 
the monodromy of the solutions doesn't change, which is proved
by \cite{HF}. 
In the case of the Okubo normal form the corresponding 
transformation of the systems, that of the integral representations of 
the solutions and that of their connection 
coefficients are studied by \cite{Yo2}, \cite{Ha} and \cite{Yo3}, 
respectively.
These operation are explicit and hence it will be expected 
to have explicit results in general Fuchsian systems.

To avoid the specific forms of the differential equations, such as
Schlesinger canonical form or Okubo normal form and moreover to
make explicit calculations easier under the transformations, we 
introduce certain operations on differential operators with 
polynomial coefficients in \S\ref{sec:frac}.
The operations in \S\ref{sec:frac} enables us 
to equally handle equations with irregular singularities or systems
of equations with several variables.

The ring of differential operators with polynomial coefficients
is called a \textsl{Weyl algebra} and denoted by $W[x]$ in this paper.
The endomorphisms of $W[x]$ do not give a wide class of operations 
and Dixmier \cite{Dix} conjectured that they are the automorphisms of 
$W[x]$.
But when we localize coordinate $x$, namely in the ring $W(x)$ 
of differential operators with coefficients in rational functions, 
we have a wider class of operations.  

For example, the transformation of the pair $(x, \frac{d}{dx})$ 
into $(x, \frac{d}{dx} - h(x))$ with any rational function $h(x)$ 
induces an automorphism of
$W(x)$.  
This operation is called a \textsl{gauge transformation}.
The addition in \cite{DR, DR2} corresponds to this operation
with $h(x) = \frac\lambda{x-c}$ and $\lambda,\, c\in\mathbb C$, which
is denoted by $\Ad\bigl((x-c)^\lambda\bigr)$.

The transformation of the pair $(x,\frac{d}{dx})$ into $(-\frac d{dx},x)$ 
defines an important automorphism $\Lap$ of $W[x]$, which
is called a \textsl{Laplace transformation}.
In some cases the Fourier transformation is introduced and it is 
a similar transformation. 
Hence we may also localize $\frac {d}{dx}$ and introduce the operators
such as $\lambda(\frac{d}{dx} - c)^{-1}$ and then the transformation of
the pair $(x,\frac{d}{dx})$ into 
$(x-\lambda(\frac{d}{dx})^{-1}, \frac{d}{dx})$
defines an endomorphism in this localized ring, which
corresponds to the middle convolution or an Euler transformation or
a fractional derivation and is denoted by $\Ad(\p^{-\lambda})$ or $mc_\lambda$.
But the simultaneous localizations of $x$ and $\frac{d}{dx}$ produce 
the operator
$(\frac{d}{dx})^{-1}\circ x^{-1}=\sum_{k=0}^\infty 
k!{x^{-k-1}}(\frac{d}{dx})^{-k-1}$ which is not algebraic in our sense
and hence we will not introduce such a microdifferential operator
in this paper and we will not allow the simultaneous localizations
of the operators.

Since our equation $Pu=0$ studied in this paper is defined on the 
Riemann sphere, we may replace the operator $P$ in $W(x)$ 
by a suitable representative $\tilde P\in \mathbb C(x)P\cap W[x]$ with
the minimal degree with respect to $x$ and we put $\Red P=\tilde P$.
Combining these operations including this replacement gives
a wider class of operations on the Weyl algebra $W[x]$.
In particular, the operator corresponding to the addition is 
$\RAd\bigl((x-c)^\lambda\bigr)$ and that corresponding to the middle 
convolution is $\RAd(\p^{-\mu})$ in our notation.
The operations introduced in \S\ref{sec:frac} correspond to certain 
transformations of solutions of the differential equations defined by
elements of Weyl algebra and we call the calculation using these 
operations \textsl{fractional calculus of Weyl algebra}.

To understand our operations, we show that, in Example~\ref{ex:midconv},
our operations enables us to construct Gauss hypergeometric equations, 
the equations satisfied by airy functions and Jordan-Pochhammer 
equations and to give integral representations of their solutions.

In this paper we mainly study ordinary differential equations and since
any ordinary differential equation is \textsl{cyclic}, 
namely, it is isomorphic to a single differential operator
$Pu=0$ (cf.~\S\ref{sec:ODE}),  we study a single ordinary 
differential equation $Pu=0$ with $P\in W[x]$.
In many cases, we are interested in a specific function $u(x)$ 
which is characterized by differential equations
and if $u(x)$ is a function with the single variable $x$, the 
differential operators $P\in W(x)$ 
satisfying $Pu(x)=0$ are generated by a single 
operator and hence it is naturally a single differential 
equation. 
A relation between our fractional calculus and Katz's middle convolution
is briefly explained in \S\ref{sec:DR}.

In \S\ref{sec:reg} we review fundamental results on Fuchsian ordinary 
differential equations.
Our Weyl algebra $W[x]$ is allowed to have some parameters $\xi_1,\ldots$ 
and in this case the algebra is denoted by $W[x;\xi]$.
The position of singular points of the equations and 
the characteristic exponents there are usually the parameters
and the analytic continuation of the parameters
naturally leads the confluence of \text{additions}
(cf.~\S\ref{sec:VAd}).


Combining this with our construction of equations
leads the confluence of the equations.  In the case of
Jordan-Pochhammer equations, we have versal Jordan-Pochhammer 
equations.  
In the case of Gauss hypergeometric equation, 
we have a unified expression of Gauss hypergeometric equation,
Kummer equation and Hermite-Weber equation and get 
a unified integral representation of their solutions
(cf.~Example~\ref{ex:VGHG}).
After this section in this paper, we mainly study
single Fuchsian differential equations on the Riemann sphere.
Equations with irregular singularities will be discussed elsewhere.

In \S\ref{sec:series} and \S\ref{sec:contig} we examine the transformation
of series expansions and contiguity relations of the solutions
of Fuchsian differential equations under our operations.

The Fuchsian equation satisfied by the generalized hypergeometric 
series
\index{hypergeometric equation/function!generalized}
\begin{equation}\label{eq:IGHG}\index{00Fn@${}_nF_{n-1}$}%
\index{000gammak@$(\gamma)_k,\ (\mu)_\nu$}
 \begin{gathered}
 {}_nF_{n-1}(\alpha_1,\dots,\alpha_n,\beta_1,\dots,\beta_{n-1};x)
  =\sum_{k=0}^\infty
  \frac{(\alpha_1)_k\dots(\alpha_n)_k}{
  (\beta_1)_k\dots(\beta_{n-1})_{n-1}k!}x^k\\
 \text{with\quad}(\gamma)_k :=\gamma(\gamma+1)\cdots(\gamma+k-1)
 \end{gathered}
\end{equation}
is characterized by the fact that it has $(n-1)$-dimensional 
local holomorphic solutions at $x=1$, which is more precisely as 
follows. 
The set of characteristic exponents of the equation at $x=1$ 
equals $\{0,1,\dots,n-1,-\beta_n\}$ with 
$\alpha_1+\dots+\alpha_n=\beta_1+\dots+\beta_n$
and those at $0$ and $\infty$ are
$\{1-\beta_1,\dots,1-\beta_{n-1},0\}$ and $\{\alpha_1,\dots,\alpha_n\}$, 
respectively.
Then if $\alpha_i$ and $\beta_j$ are generic, the Fuchsian differential
equation  $Pu=0$ is uniquely characterized by the fact that it has 
the above set of  characteristic exponents at each singular point 
$0$ or $1$ or $\infty$ and the monodromy generator around the point is 
\textsl{semisimple}, namely, the local solution around the singular 
point has no logarithmic term.
We express this condition by the (generalized) Riemann scheme
\begin{equation}\label{eq:GRSGHG}
 \begin{Bmatrix}
    x = 0 & 1 & \infty\\
     1-\beta_1    & [0]_{(n-1)}   & \alpha_1\\
     \vdots       &     & \vdots &;\,x\\
     1-\beta_{n-1}&     & \alpha_{n-1}\\
     0 &   -\beta_n    & \alpha_n
    \end{Bmatrix},\quad
 [\lambda]_{(k)}:=
 \begin{pmatrix}
 \lambda\\
 \lambda+1\\
 \vdots\\
 \lambda+k-1
 \end{pmatrix}.
\end{equation}
In particular, when $n=3$, the (generalized) Riemann scheme is
\[
 \begin{Bmatrix}
 x = 0 & 1 & \infty\\
 \begin{matrix}
 1 - \beta_1\\
 1 - \beta_2
 \end{matrix} &
 \begin{pmatrix}
 0\\
 1
 \end{pmatrix} &
 \begin{matrix}
 \alpha_1\\
 \alpha_2
 \end{matrix} &
 \!\!\!\begin{matrix}
 \\ &\!\!;\, x
 \end{matrix}\\
 0 & -\beta_3 & \alpha_3
\end{Bmatrix}.
\]
The corresponding usual Riemann scheme is obtained from the generalized
Riemann scheme by eliminating $\Big($ and $\Big)$.
Here $[0]_{(n-1)}$ in the above Riemann scheme means 
the characteristic exponents $0,1,\dots,n-2$ but it also
indicates that the corresponding monodromy generator is semisimple 
in spite of integer differences of the characteristic exponents.  
Thus the set of (generalized) characteristic exponents 
$\{[0]_{(n-1)},-\beta_n\}$ at $x=1$ is defined.
Here we remark that the coefficients of the Fuchsian differential
operator $P$ which is uniquely determined by the generalized Riemann scheme
for generic $\alpha_i$ and $\beta_j$ are polynomial functions of 
$\alpha_i$ and $\beta_j$ and hence $P$ is naturally defined for any 
$\alpha_i$ and $\beta_j$ as is given by \eqref{eq:GHP}.
Similarly the Riemann scheme of Jordan-Pochhammer equation 
of order $p$ is
\begin{equation}
\begin{gathered}
 \begin{Bmatrix}
 x = c_0  & c_1 &\cdots &c_{p-1} & \infty\\
    [0]_{(p-1)} & [0]_{(p-1)} & \cdots&[0]_{(p-1)}& [\lambda'_p]_{(p-1)}
     &;\,x\\
    \lambda_0 & \lambda_1 & \cdots &\lambda_{p-1}& \lambda_p
 \end{Bmatrix},\\
 \lambda_0+\cdots+\lambda_{p-1}+\lambda_p+(p-1)\lambda'_p=p-1.
\end{gathered}
\end{equation}
The last equality in the above is called \textsl{Fuchs relation}.

In \S\ref{sec:index} 
we define the set of generalized characteristic exponents at a
regular singular point of a differential equation $Pu=0$.
In fact, when the order of $P$ is $n$, it is the set 
$\{[\lambda_1]_{(m_1)},\dots,[\lambda_k]_{(m_k)}\}$
with a partition $n=m_1+\cdots+m_k$ and complex numbers 
$\lambda_1,\dots,\lambda_k$.
It means that the set of characteristic exponents at the point
equals $\{\lambda_j+\nu\,;\,\nu=0,\dots,m_j-1\text{ and }j=1,\dots,k\}$
and the corresponding monodromy generator is semisimple if
$\lambda_i-\lambda_j\not\in\mathbb Z$ for $1\le i<j\le k$.
In \S\ref{sec:Gexp} we define the set of generalized characteristic exponents
without the assumption $\lambda_i-\lambda_j\not\in\mathbb Z$ for 
$1\le i<j\le k$.  Here we only remark that 
when $\lambda_i=\lambda_1$ for $i=1,\dots,k$, it is
also characterized by the fact that the Jordan normal form of
the monodromy generator is defined by the dual partition of
$n=m_1+\cdots+m_k$ together with the usual characteristic exponents.

Thus for a single Fuchsian differential equation $Pu=0$ on the Riemann sphere
which has $p+1$ regular singular points $c_0,\dots,c_p$, 
we define a (generalized) Riemann scheme
\begin{equation}\label{eq:IGRS}
  \begin{Bmatrix}
   x = c_0 & c_1 & \cdots & c_p\\
  [\lambda_{0,1}]_{(m_{0,1})} & [\lambda_{1,1}]_{(m_{1,1})}&\cdots
    &[\lambda_{p,1}]_{(m_{p,1})}\\
  \vdots & \vdots & \vdots & \vdots&;\,x\\
    [\lambda_{0,n_0}]_{(m_{0,n_0})} & [\lambda_{1,n_1}]_{(m_{1,n_1})}&\cdots
    &[\lambda_{p,n_p}]_{(m_{p,n_p})}
  \end{Bmatrix}.
\end{equation}
Here $n=m_{j,1}+\cdots+m_{j,n_j}$ for $j=0,\dots,p$ and
$n$ is the order of $P$ and $\lambda_{j,\nu}\in\mathbb C$.
The $(p+1)$-tuple of partitions of $n$, which is denoted
by $\mathbf m=\bigl(m_{j,\nu}\bigr)_{\substack{j=0,\dots,p\\\nu=1,\dots,n_j}}$,
is called the \textsl{spectral type} of $P$ and the Riemann scheme 
\eqref{eq:IGRS}.  
Here we note that the Riemann scheme \eqref{eq:IGRS} should always satisfy 
the Fuchs relation
\begin{align}\label{eq:IFuchs}
 |\{\lambda_{\mathbf m}\}|&:=\sum_{j=0}^p\sum_{\nu=1}^{n_p}
  m_{j,\nu}\lambda_{j,\nu}- \ord\mathbf m + \tfrac12\idx\mathbf m
 =0,\\
 \idx\mathbf m&:=\sum_{j=0}^p\sum_{\nu=1}^{n_p}m_{j,\nu}^2
 -(p-1)\ord\mathbf m.
\end{align}
Here $\idx\mathbf m$ coincides with the \textsl{index of rigidity}
introduced by \cite{Kz}.

In \S\ref{sec:index}, after introducing certain
representatives of conjugacy classes of matrices and 
some notation and concepts related to tuples of partitions,
we define that the tuple $\mathbf m$ is \textsl{realizable} if 
there exists a Fuchsian differential operator $P$ with the Riemann 
scheme \eqref{eq:IGRS} for generic complex numbers $\lambda_{j,\nu}$ 
under the condition \eqref{eq:IFuchs}.  
Furthermore, if there exists such an operator $P$
so that $Pu=0$ is irreducible, we define that $\mathbf m$ is 
\textsl{irreducibly realizable}.

Lastly in \S\ref{sec:index}, we examine the generalized Riemann schemes
of the \textsl{product} of Fuchsian differential operators and the \textsl{dual} 
operators.

In \S\ref{sec:reduction} 
we examine the transformations of the Riemann scheme
under our operations corresponding to the additions and 
the middle convolutions, which define transformations within Fuchsian 
differential operators.
The operations induce transformations of spectral types of Fuchsian 
differential operators, which keep the indices of 
rigidity invariant but change the orders in general.
Looking at the spectral types, 
we see that the combinatorial aspect of the reduction of Fuchsian 
differential operators is parallel to that of systems of Schlesinger 
canonical form.
In this section, 
we also examine the combination of these transformation and the \textsl{fractional
linear transformations.}

As our interpretation of Deligne-Simpson problem introduced by Kostov,
we examine the condition for the existence of a given Riemann scheme
in \S\ref{sec:DS}.
We determine the conditions on $\mathbf m$ such that $\mathbf m$ is realizable
and irreducibly realizable, respectively, 
in Theorem~\ref{thm:univmodel}.
Moreover if $\mathbf m$ is realizable, Theorem~\ref{thm:univmodel} gives
an explicit construction of the \textsl{universal Fuchsian
differential operator} 
\begin{equation}
\begin{gathered}
 P_{\mathbf m}
 =\Bigl(\prod_{j=1}^p(x-c_j)^n\Bigr)\frac{d^n}{dx^n}
  +\sum_{k=0}^{n-1} a_k(x,\lambda, g)\frac{d^k}{dx^k},\\
 \lambda=\bigl(\lambda_{j,\nu}\bigr)
 _{\substack{j=0,\dots,p\\ \nu=1,\dots,n_j}},\quad
 g=(g_1,\dots,g_N)\in\mathbb C^N
\end{gathered}
\end{equation}
with the Riemann scheme \eqref{eq:IGRS},
which has the following properties.

For fixed complex numbers $\lambda_{j,\nu}$ satisfying \eqref{eq:IFuchs}
the operator with the Riemann scheme \eqref{eq:IGRS} satisfying 
$c_0=\infty$ equals $P_{\mathbf m}$ for a suitable $g\in\mathbb C^N$ up 
to a left multiplication by an element of $\mathbb C(x)$ if
$\lambda_{j,\nu}$ are ``generic", namely,
\begin{equation}
(\Lambda(\lambda)|\alpha)\notin
\bigl\{-1,-2,\dots,1-(\alpha|\alpha_{\mathbf m})\bigr\}\text{  for any }
\alpha\in\Delta(\mathbf m)
\text{ with } (\alpha|\alpha_{\mathbf m})>1
\end{equation}
under the notation used in \eqref{eq:KacIrr},
or $\mathbf m$ is fundamental or \textsl{simply reducible} 
(cf.~Definition \ref{def:fund} and \S\ref{sec:simpred}), etc.
Here $g_1,\dots,g_N$ are called \textsl{accessory parameters}
and if $\mathbf m$ is irreducibly realizable, $N=1-\frac12\idx\mathbf m$.
In particular, if there is an irreducible and \textsl{locally non-degenerate} 
(cf.~Definition~\ref{def:locnondeg}) operator $P$  with the Riemann scheme 
\eqref{eq:IGRS}, then $\lambda_{j,\nu}$ are ``generic".

The coefficients $a_k(x,\lambda, g)$ of the differential operator 
$P_{\mathbf m}$ are polynomials of the variables 
$x$, $\lambda$ and $g$.
The coefficients satisfy $\frac{\p^2 a_k}{\p g_\nu^2}=0$ and 
furthermore $g_\nu$ can be equal to suitable $a_{i_\nu,j_\nu}$ under 
the expression
$
 P_{\mathbf m}=
 \sum a_{i,j}x^i\frac{d^j}{dx^j}
$
and the pairs $(i_\nu,j_\nu)$ for $\nu=1,\dots,N$ are explicitly 
given in the theorem.

The universal operator $P_\mathbf m$ is a classically well-known
operator in the case of Gauss hypergeometric
equation, Jordan-Pochhammer equation or Heun's equation etc.\ and
the theorem assures the existence of such a good operator for
any realizable tuple $\mathbf m$.
We define the tuple $\mathbf m$ is \textsl{rigid}
if $\mathbf m $ is irreducibly realizable and 
moreover $N=0$, namely, $P_{\mathbf m}$ is free from accessory parameters.

In particular, the theorem gives the affirmative answer for the 
following question.
Katz asked a question in the introduction in the book \cite{Kz}
whether a rigid local system is realized by a single Fuchsian 
differential equation $Pu=0$ without apparent singularities
(cf.~Corollary~\ref{cor:irred} iii)).

It is a natural problem to examine the Fuchsian differential equation 
$P_{\mathbf m}u=0$ with an irreducibly realizable spectral type $\mathbf m$ 
which cannot be reduced to an equation with a lower order by additions
and middle convolutions.  The tuple $\mathbf m$ with this condition is
called \textsl{fundamental}.

The equation $P_{\mathbf m}u=0$ with an irreducibly realizable spectral type 
$\mathbf m$ can be transformed by the operation $\p_{max}$ 
(cf.~Definition~\ref{def:pell}) into a Fuchsian equation 
$P_{\mathbf m'}v=0$ with a fundamental spectral type $\mathbf m'$.
Namely, there exists a non-negative integer $K$ such that 
$P_{\mathbf m'}=\p_{\max}^KP_{\mathbf m}$ 
and we define $f\mathbf m:=\mathbf m'$.
Then it turns out that a realizable tuple $\mathbf m$ is rigid if and only if
the order of $f\mathbf m$, which is the order of $P_{f\mathbf m}$ by definition, 
equals 1.  Note that the operator $\p_{\max}$ is essentially a product of 
suitable operators $\RAd\bigl((x-c_j)^{\lambda_j}\bigr)$ and
$\RAd\bigl(\p^{-\mu}\bigr)$.

In this paper we study the transformations of several properties 
of the Fuchsian differential equation $P_{\mathbf m}u=0$ under the 
additions and middle convolutions.  If they are understood well, 
the study of the properties are reduced to those of the equation 
$P_{f\mathbf m}v=0$, which are of order 1 if $\mathbf m$ is rigid.
We note that there are many rigid spectral types $\mathbf m$ and for example
there are 187 different rigid spectral types $\mathbf m$ with 
$\ord\mathbf m\le 8$ as are given in \S\ref{sec:rigidEx}.

As in the case of the systems of Schlesinger canonical form
studied by \cite{CB}, 
the combinatorial aspect of transformations of the spectral type 
$\mathbf m$ of the Fuchsian differential operator $P$ induced from
our fractional operations is
described in \S\ref{sec:KacM} by using the terminology of a Kac-Moody
root system $(\Pi,W_{\!\infty})$.
Here $\Pi$ is the fundamental system of a Kac-Moody root system
with the following star-shaped Dynkin diagram and $W_{\!\infty}$ is the Weyl group
generated by the \textsl{simple reflections} $s_\alpha$ for 
$\alpha\in\Pi$.  The elements of $\Pi$ are called \textsl{simple roots}.

Associated to a tuple $\mathbf m$ of $(p+1)$ partitions of a positive integer $n$,
we define an element $\alpha_{\mathbf m}$ in the positive root 
lattice (cf.~\S\ref{sec:KM}, \eqref{eq:PIKac}):
\begin{equation}
\begin{split} \Pi&:=\{\alpha_0,\,\alpha_{j,\nu}\,;\,
 j=0,1,\ldots,\ \nu=1,2,\ldots\},\\
 W_{\!\infty}&:=\langle s_\alpha\,;\,\alpha\in\Pi\rangle,\\
\alpha_{\mathbf m}&:=n\alpha_0+\sum_{j=0}^p\sum_{\nu=1}^{n_j-1}
\Bigl(\sum_{i=\nu+1}^{n_j}m_{j,i}\Bigr)\alpha_{j,\nu},\\
(\alpha_{\mathbf m}&|\alpha_{\mathbf m})=\idx\mathbf m,
\end{split}\qquad
\begin{xy}
\ar@{-}               *++!D{\text{$\alpha_0$}}  *\cir<4pt>{}="O";
             (10,0)   *+!L!D{\text{$\alpha_{1,1}$}} *\cir<4pt>{}="A",
\ar@{-} "A"; (20,0)   *+!L!D{\text{$\alpha_{1,2}$}} *\cir<4pt>{}="B",
\ar@{-} "B"; (30,0)   *{\cdots}, 
\ar@{-} "O"; (10,-7)  *+!L!D{\text{$\alpha_{2,1}$}} *\cir<4pt>{}="C",
\ar@{-} "C"; (20,-7)  *+!L!D{\text{$\alpha_{2,2}$}} *\cir<4pt>{}="E",
\ar@{-} "E"; (30,-7)  *{\cdots}
\ar@{-} "O"; (10,8)   *+!L!D{\text{$\alpha_{0,1}$}} *\cir<4pt>{}="D",
\ar@{-} "D"; (20,8)   *+!L!D{\text{$\alpha_{0,2}$}} *\cir<4pt>{}="F",
\ar@{-} "F"; (30,8)   *{\cdots}
\ar@{-} "O"; (10,-13) *+!L!D{\text{$\alpha_{3,1}$}} *\cir<4pt>{}="G",
\ar@{-} "G"; (20,-13) *+!L!D{\text{$\alpha_{3,2}$}} *\cir<4pt>{}="H",
\ar@{-} "H"; (30,-13) *{\cdots},
\ar@{-} "O"; (7, -13),
\ar@{-} "O"; (4, -13),
\end{xy}
\end{equation}
We can define a fractional operation on $P_{\mathbf m}$ 
which is compatible with the action of $w\in W_{\!\infty}$ on the root lattice
(cf.~Theorem~\ref{thm:KatzKac}):
\begin{equation}\label{eq:mcKacdg}
 \begin{matrix}
\bigl\{P_{\mathbf m}:\,\text{Fuchsian differential operators}\bigr\}
  &\rightarrow
   &\bigl\{(\Lambda(\lambda),\alpha_{\mathbf m})\,;\,\alpha_{\mathbf m}\in\overline\Delta_+\bigr\}
 \\
 \\
  \downarrow \text{fractional operations}
  &\circlearrowright&\quad \downarrow {W_{\!\infty}\text{-action},\ 
  +\tau\Lambda^0_{0,j}}\\
 \\
\bigl\{P_{\mathbf m}:\,\text{Fuchsian differential operators}\bigr\}
   & \rightarrow
   & \bigl\{(\Lambda(\lambda), \alpha_{\mathbf m})\,;\,\alpha_{\mathbf m}
 \in\overline\Delta_+\bigr\}.
 \end{matrix}
\end{equation}
Here $\tau\in\mathbb C$ and
\begin{equation}
 \begin{split}
 \Lambda^0&:=\alpha_0+\sum_{\nu=1}^\infty(1+\nu)\alpha_{0,\nu}+
   \sum_{j=1}^p\sum_{\nu=1}^\infty(1-\nu)\alpha_{j,\nu},\\
 \Lambda^0_{i,j}&:=\sum_{\nu=1}^\infty \nu(\alpha_{i,\nu}-\alpha_{j,\nu}),\\
 \Lambda_0&:=\frac12\alpha_0
  +\frac12\sum_{j=0}^p\sum_{\nu=1}^\infty(1-\nu)\alpha_{j,\nu},\\
 \Lambda(\lambda)&:=-\Lambda_0-\sum_{j=0}^p\sum_{\nu=1}^\infty
  \Bigl(\sum_{i=1}^\nu\lambda_{j,i}
  \Bigr)\alpha_{j,\nu}
 \end{split}
\end{equation}
and these linear combinations of infinite simple roots are 
identified with each other if their differences are in $\mathbb C\Lambda^0$.
We note that 
\begin{equation}
 |\{\lambda_{\mathbf m}\}|
=(\Lambda(\lambda)+\tfrac12\alpha_{\mathbf m}|\alpha_{\mathbf m}).
\end{equation}

The realizable tuples exactly correspond to the elements of the set 
$\overline\Delta_+$ of positive integer multiples of the positive roots 
of the Kac-Moody root system whose support contains $\alpha_0$
and the rigid tuples exactly correspond to the positive real 
roots whose support contain $\alpha_0$.
For an element $w\in W_{\!\infty}$ and an element $\alpha\in\overline\Delta_+$ 
we do not consider $w\alpha$ in the commutative diagram \eqref{eq:mcKacdg}
when $w\alpha\notin \overline\Delta_+$.

Hence the fact that any irreducible rigid Fuchsian equation $P_{\mathbf m}u=0$ 
is transformed into the trivial equation $\frac{dv}{dx}=0$ by our 
invertible fractional operations corresponds to the fact that there 
exists $w\in W_{\!\infty}$ such that $w\alpha_{\mathbf m}=\alpha_0$ because 
$\alpha_{\mathbf m}$ is a positive real root. 
The \textsl{monotone} fundamental tuples of partitions correspond to
$\alpha_0$ or the positive imaginary roots $\alpha$ in the closed 
negative Weyl chamber which are indivisible or satisfies $(\alpha|\alpha)<0$.
A tuple of partitions $\mathbf m=\bigl(m_{j,\nu}\bigr)_{\substack{j=0,\dots,p\\\nu=1,\dots,n_j}}$ is said to be monotone if
$m_{j,1}\ge m_{j,2}\ge \cdots\ge m_{j,n_j}$ for $j=0,\dots,p$.
For example, we prove the exact estimate
\begin{equation}
 \ord\mathbf m\le 3|\idx\mathbf m|+6
\end{equation}
for any fundamental tuple $\mathbf m$ in \S\ref{sec:basic}.
Since we may assume
\begin{equation}
 p\le\tfrac12|\idx\mathbf m|+3
\end{equation}
for a fundamental tuple $\mathbf m$, there exist only finite number
of monotone fundamental tuples with a fixed index of rigidity.
We list the fundamental tuples of the index of rigidity $0$ or $-2$ in 
Remark~\ref{rem:bas0} or Proposition~\ref{prop:bas2}, respectively.

Our results in \S\ref{sec:series}, \S\ref{sec:reduction} and 
\S\ref{sec:DS} give an integral expression and a
power series expression of a local solution
of the universal equation $P_{\mathbf m}u=0$ corresponding to the
characteristic exponent whose multiplicity is free in the local 
monodromy.  These expressions are in \S\ref{sec:exp}.

In \S\ref{sec:MM} we review the monodromy of solutions of
a Fuchsian differential equation from the view point of our operations.
The theorems in this section are given by \cite{DR, DR2, Kz, Ko}.
In \S\ref{sec:Scott} we review Scott's lemma \cite{Sc} and related results
with their proofs, which are elementary but important for the study of 
the irreducibility of the monodromy.

In \S\ref{sec:reddirect} we examine the condition for
the decomposition $P_{\mathbf m}=P_{\mathbf m'}P_{\mathbf m''}$
of universal operators with or without fixing the exponents
$\{\lambda_{j,\nu}\}$, 
which implies the reducibility of the 
equation $P_{\mathbf m}u=0$.
In \S\ref{sec:redred} we study the value of spectral parameters which 
makes the equation reducible and obtain Theorem~\ref{thm:irred}.
In particular we have a necessary and sufficient condition on characteristic
exponents so that the monodromy of the solutions of the equation 
$P_{\mathbf m}u=0$ with a rigid spectral type $\mathbf m$ is irreducible,
which is given in Corollary \ref{cor:irred} or Theorem \ref{thm:irrKac}.
When $m_{j,1}\ge m_{j,2}\ge \cdots$ for any $j\ge0$, the condition equals
\begin{equation}\label{eq:KacIrr}
  (\Lambda(\lambda)|\alpha)\notin\mathbb Z\quad(\forall\alpha\in\Delta(\mathbf m)).
\end{equation}
Here $\Delta(\mathbf m)$ denotes the totality of real positive roots $\alpha$
such that $w_{\mathbf m}\alpha$ are negative
and $w_{\mathbf m}$ is the element of $W_{\!\infty}$ with the minimal length 
so that $\alpha_0=w_{\mathbf m}\alpha_{\mathbf m}$ 
(cf.~Definition~\ref{def:wm} and Proposition~\ref{prop:wm} v)).
The number of elements of $\Delta(\mathbf m)$ equals the length of 
$w_{\mathbf m}$, which is the minimal length of the expressions of 
$w_{\mathbf m}$ as products of simple reflections $s_\alpha$ with 
$\alpha\in\Pi$.

In \S\ref{sec:shift} we construct shift operators between rigid Fuchsian 
differential equations with the same spectral type such that the
differences of the corresponding characteristic exponents are integers.
Theorem~\ref{thm:shifm1} gives a recurrence relation of certain
solutions of the rigid Fuchsian equations, which is a generalization of 
the formula
\begin{equation}
 c\bigl(F(a,b+1,c;x) - F(a,b,c;x)\bigr) = axF(a+1,b+1,c+1;x)
\end{equation}
and moreover gives relations between the universal operators
and the shift operators
in Theorem~\ref{thm:shifm1} and Theorem~\ref{thm:sftUniv}.
In particular, Thorem~\ref{thm:sftUniv} gives a condition which assures that a 
universal operator is this shift operator.

The shift operators are useful for the study of Fuchsian differential equations
when they are reducible because of special values of the characteristic 
exponents.
Theorem~\ref{thm:isom} give a necessary condition and a sufficient 
condition so that the shift operator is bijective.
In many cases we get a necessary and sufficient condition by this theorem.
As an application of a shift operator we examine polynomial solutions
of a rigid Fuchsian differential equation of Okubo type in \S\ref{sec:polyn}.

In \S\ref{sec:C1} we study a connection problem of 
the Fuchsian differential equation  $P_{\mathbf m}u=0$.
First we give Lemma~\ref{lem:conn} which describes
the transformation of a connection coefficient under an addition and
a middle convolution.
In particular, for the equation $P_{\mathbf m}u=0$
satisfying $m_{0,n_0}=m_{1,n_1}=1$, Theorem~\ref{thm:GC} says that
the connection coefficient 
$c(c_0\!:\!\lambda_{0,n_0}\!\rightsquigarrow\!c_1\!:\!\lambda_{1,n_1})$ from 
the local solution corresponding to the exponent $\lambda_{0,n_0}$ to that 
corresponding to $\lambda_{1,n_1}$ in the Riemann scheme \eqref{eq:IGRS}
equals the connection coefficient of the reduced equation $P_{f\mathbf m}v=0$
up to the gamma factors which are explicitly calculated.

In particular, if the equation is rigid, Theorem~\ref{thm:c} gives
the connection coefficient as a quotient of products of gamma functions 
and an easier non-zero term. 
For example, when $p=2$, the easier term doesn't appear and the 
connection coefficient has the universal formula
\begin{equation}\label{eq:Icon}
 c(c_0\!:\!\lambda_{0,n_0}\!\rightsquigarrow\!c_1\!:\!\lambda_{1,n_1})
 =\frac
  {\displaystyle\prod_{\nu=1}^{n_0-1} 
    \Gamma\bigl(\lambda_{0,n_0}-\lambda_{0,\nu}+1\bigr)
   \cdot\prod_{\nu=1}^{n_1-1}
    \Gamma\bigl(\lambda_{1,\nu}-\lambda_{1,n_1}\bigr)
  }
  {\displaystyle\prod_{\substack{\mathbf m'\oplus\mathbf m''=\mathbf m\\
                    m'_{0,n_0}=m''_{1,n_1}=1}}
    \Gamma\bigl(|\{\lambda_{\mathbf m'}\}|\bigr)
  }.
\end{equation}
Here the notation \eqref{eq:IFuchs} is used and 
$\mathbf m=\mathbf m'\oplus\mathbf m''$ means that
$\mathbf m=\mathbf m'+\mathbf m''$ with rigid tuples
$\mathbf m'$ and $\mathbf m''$.
Moreover the number of gamma factors in the above denominator is equals
to that of the numerator.
The author conjectured this formula in 2007 and proved it in 2008
(cf.~\cite{O3}).  The proof in \S\ref{sec:C1} is different from the 
original proof, which is explained in \S\ref{sec:C2}.

Suppose $p=2$, $\ord\mathbf m=2$, $m_{j,\nu}=1$ for $0\le j\le 2$ and $1\le\nu\le 2$,
Then \eqref{eq:Icon} equals
\begin{equation}
 \frac{\Gamma(\lambda_{0,2}-\lambda_{0,1}+1)\Gamma(\lambda_{1,2}-\lambda_{1,1})}
      {\Gamma(\lambda_{0,1}+\lambda_{1,2}+\lambda_{2,1})
      \Gamma(\lambda_{0,1}+\lambda_{1,2}+\lambda_{2,2})},
\end{equation}
which implies \eqref{eq:Gausssum} under \eqref{eq:RSGauss}.

The hypergeometric series $F(a,b,c;x)$ satisfies
$\lim_{k\to+\infty}F(a,b,c+k;x)=1$ if $|x|\le 1$, which obviously 
implies $\lim_{k\to+\infty}F(a,b,c+k;1)=1$.
Gauss proves the summation formula \eqref{eq:Gausssum} by this limit formula
and the recurrence relation $F(a,b,c;1)=\frac{(c-a)(c-b)}{c(c-a-b)}F(a,b,c+1;1)$.
We have $\lim_{k\to+\infty}
c(c_0\!:\!\lambda_{0,n_0}+k\!\rightsquigarrow\!c_1\!:\!\lambda_{1,n_1}-k)=1$ 
in the connection formula \eqref{eq:Icon} (cf.~Corollary~\ref{cor:C}). 
This suggests a similar limit formula for a local solution of a general 
Fuchsian differential equation, which is given in \S\ref{sec:estimate}.

In \S\ref{sec:C2} we propose a procedure to calculate the connection coefficient
(cf.~Remark~\ref{rem:Cproc}), which is based on the calculation of its zeros 
and poles.
This procedure is different from the proof of Theorem~\ref{thm:c} in 
\S\ref{sec:C1} and useful to calculate a certain connection coefficient 
between local solutions with multiplicities in eigenvalues of local 
monodromies. The coefficient is defined in Definition~\ref{def:GC}.

In \S\ref{sec:ex} we show many examples which explain our fractional
calculus in this paper and also give concrete results of the calculus.
In \S\ref{sec:basicEx} we list all the fundamental tuples whose indices 
of rigidity are not smaller than $-6$ and in \S\ref{sec:rigidEx} we list
all the rigid tuples whose orders are not larger than 8, most of which
are calculated by a computer program \texttt{okubo} explained in 
\S\ref{sec:okubo}.
In \S\ref{sec:PoEx} and \S\ref{sec:GHG} we apply our fractional 
calculus to Jordan-Pochhammer equations and the hypergeometric family, 
respectively, which helps us to understand our unifying study of rigid Fuchsian 
differential equations.
In \S\ref{sec:EOEx} we apply our fractional 
calculus to the even/odd family classified by \cite{Si} and most of the
results there have been first obtained by the calculus.

In \S\ref{sec:4Ex}, \S\ref{sec:ord6Ex} and \S\ref{sec:Rob} we study the
rigid Fuchsian differential equations of order not larger than 4
and those of order 5 or 6 and the equations 
belonging to 12 submaximal series classified by \cite{Ro}, respectively.
Note that these 12 maximal series contain 
Yokoyama's list \cite{Yo}.
In \S\ref{sec:RobEx}, we explain how we read the condition of irreducibility,
connection coefficients, shift operators etc.\ of the corresponding differential 
equation from the data given in \S\ref{sec:4Ex}--\S\ref{sec:Rob}.
In \S\ref{sec:TriEx}, we show some interesting identities of 
trigonometric functions as a consequence of the concrete value \eqref{eq:Icon}
of connection coefficients.
We examine Appell's hypergeometric equations in \S\ref{sec:ApEx} by our 
fractional calculus, which will be further discussed in another paper.

In \S\ref{sec:prob} we give some problems to be studied related to the
results in this paper. 

In \S\ref{sec:appendix} a theorem on Coxeter groups is given, 
which was proved by K.~Nuida through a private communication between
the author and Nuida.
The theorem is useful for the study of the difference of 
various reductions of Fuchsian differential equations 
(cf.~Proposition~\ref{prop:wm} v)).
The author greatly thanks Nuida for allowing the author to
put the theorem with its proof in this paper.

The author express his sincere gratitude to Kazuo Okamoto and Yoshishige 
Haraoka for the guidance to the subjects developed in this paper and 
to Kazuki Hiroe for reading the manuscript of this paper.
\section{Fractional operations}\label{sec:frac}
\subsection{Weyl algebra}
In this section we define several operations on a Weyl algebra.
The operations are elementary or well-known but their combinations 
will be important.

Let $\mathbb C[x_1,\dots,x_n]$ denote the polynomial ring of $n$
independent variables $x_1,\dots,x_n$ over $\mathbb C$ and 
let $\mathbb C(x_1,\dots,x_n)$ denote the quotient field of
$\mathbb C[x_1,\dots,x_n]$.
\index{Weyl algebra}
The \textsl{Weyl algebra} $W[x_1,\dots,x_n]$ of $n$ variables $x_1,\dots,x_n$
is the algebra over $\mathbb C$ generated by $x_1,\dots,x_n$ and 
$\frac{\p}{\p x_1},\dots,\frac{\p}{\p x_n}$ with the 
fundamental relation
\begin{equation}
 [x_i,x_j]=[\tfrac{\p}{\p x_i},\tfrac{\p}{\p x_j}]=0, \quad
 [\tfrac{\p}{\p x_i},x_j]=\delta_{i,j}\qquad(1\le i,\,j\le n).
\end{equation}
We introduce a Weyl algebra $W[x_1,\dots,x_n][\xi_1,\dots,\xi_n]$
with parameters $\xi_1,\dots,\xi_N$ by 
\begin{equation*}
 W[x_1,\dots,x_n][\xi_1,\dots,\xi_N]:=
 \mathbb C[\xi_1,\dots,\xi_N]\underset{\mathbb C}\otimes W[x_1,\dots,x_n]
\end{equation*}
and put
\begin{align*}
 W[x_1,\dots,x_n;\xi_1,\dots,\xi_N]
  &:= \mathbb C(\xi_1,\dots,\xi_N)\underset{\mathbb C}\otimes W[x_1,\dots,x_n],\\
 W(x_1,\dots,x_n;\xi_1,\dots,\xi_N)
  &:= \mathbb C(x_1,\dots,x_n,\xi_1,\dots,\xi_N)
 \underset{\mathbb C[x_1,\dots,x_n]}\otimes W[x_1,\dots,x_n].
\end{align*}
Here we have
\begin{align}
[x_i,\xi_\nu]&=[\tfrac{\p}{\p{x_i}},\xi_\nu]=0
\quad(1\le i\le n,\ 1\le\nu\le N),\\
\begin{split}
\Bigl[\frac{\p}{\p{x_i}},\frac gf\Bigr]
    &=\frac{\p}{\p x_i}\left(\frac gf\right)\\
    &=\frac{\frac{\p g}{\p x_i}\cdot f
      -g\cdot\frac{\p f}{\p{x_i}}}{f^2}
\quad(f,\ g\in\mathbb C[x_1,\dots,x_n,\xi_1,\dots,\xi_N])
\end{split}
\end{align}
\index{00Wxxi@$W[x;\xi],\ W(x;\xi)$}
and $[\frac{\p}{\p x_i},f]=\frac{\p f}{\p{x_i}}\in 
\mathbb C[x_1,\dots,x_n,\xi_1,\dots,\xi_N]$.

For simplicity we put $x=(x_1,\dots,x_n)$ and $\xi=(\xi_1,\dots,\xi_N)$
and the algebras $\mathbb C[x_1,\dots,x_n]$, \ $\mathbb C(x_1,\dots,x_n)$, \ 
$W[x_1,\dots,x_n][\xi_1,\dots,\xi_N]$, \ 
$W[x_1,\dots,x_n;\xi_1,\dots,\xi_N]$,
$W(x_1,\dots,x_n;\xi_1,\dots,\xi_N)$ etc.\ are also denoted
by $\mathbb C[x]$, $\mathbb C(x)$, $W[x][\xi]$, $W[x;\xi]$,
$W(x;\xi)$ etc., respectively.
Then
\begin{equation}
 \mathbb C[x,\xi]\subset W[x][\xi]\subset W[x;\xi]\subset W(x;\xi).
\end{equation}

The element $P$ of $W(x;\xi)$ is uniquely written by
\begin{align}
 P = \sum_{\alpha=(\alpha_1,\dots,\alpha_n)\in {\mathbb Z}_{\ge 0}^n}
     p_\alpha(x,\xi)
     \frac{\p^{\alpha_1+\cdots+\alpha_n}}{\p x_1^{\alpha_1}
        \cdots \p x_n^{\alpha_n}}\qquad(p_\alpha(x,\xi)\in\mathbb C(x,\xi)).
\end{align}
\index{00Z@$\mathbb Z_{\ge 0},\ \mathbb Z_{>0}$}%
Here $\mathbb Z_{\ge0}=\{0,1,2,\dots\}$.
Similar we will denote the set of positive integers by $\mathbb Z_{>0}$.
If $P\in W(x;\xi)$ is not zero,
the maximal integer $\alpha_1+\dots+\alpha_n$ satisfying 
$p_\alpha(x,\xi)\ne 0$ is called the \textsl{order} of $P$ and 
denoted by $\ord P$.
\index{differential equation/operator!order}\index{00ord@$\ord$}
If $P\in W[x;\xi]$, $p_\alpha(x,\xi)$ are polynomials of $x$ with
coefficients in $\mathbb C(\xi)$ and the maximal degree of 
$p_\alpha(x,\xi)$ as polynomials of $x$ is called the \textsl{degree}
of $P$ and denoted by $\deg P$.
\index{differential equation/operator!degree}

\medskip
\subsection{Laplace and gauge transformations and reduced representatives}
First we will define some fundamental operations on $W[x;\xi]$.
\begin{defn}\label{def:Lap}
\textrm{i) }\index{00R@$\Red$}
For a non-zero element $P\in W(x;\xi)$ we choose an element
$(\mathbb C(x,\xi)\setminus\{0\})P\cap W[x;\xi]$ with the minimal degree and
denote it by $\Red P$ and call it a \textsl{reduced representative}
of $P$.  If $P=0$, we put $\Red P=0$.
Note that $\Red P$ is determined up to multiples by non-zero elements of
$\mathbb C(\xi)$. 
\index{reduced representative}

\textrm{ii)\ }\index{00L@$\Lap$}
For a subset $I$ of $\{1,\dots,n\}$ we define an automorphism
$\Lap_I$ of $W[x;\xi]$:
\begin{align}
  \Lap_I(\tfrac{\p}{\p x_i})=\begin{cases}
       x_i&(i\in I)\\
       \tfrac{\p}{\p x_i}&(i\not\in I)
               \end{cases},\ \  
  \Lap_I(x_i) =\begin{cases}
       -\tfrac{\p}{\p x_i}&(i\in I)\\
       x_i&(i\not\in I)
               \end{cases}
  \text{ and }\Lap_I(\xi_\nu)=\xi_\nu.
\end{align}
We put $\Lap=\Lap_{\{1,\dots,n\}}$ and call $\Lap$ the 
\textsl{Laplace transformation} of $W[x;\xi]$.
\index{Laplace transformation}

\textrm{iii)\ }
Let $W\!_L(x;\xi)$ be the algebra isomorphic to $W(x;\xi)$
which is defined by the Laplace transformation
\begin{equation}\label{eq:WL}
  \Lap: W(x;\xi)\ \overset{\sim}\to\ W\!_L(x;\xi) 
            \ \overset{\sim}\to\ W(x;\xi).
\end{equation}
For an element $P\in W\!_L(x;\xi)$ we define
\begin{equation}
    \Red_L(P):= \Lap^{-1}\circ \Red \circ \Lap(P).
\end{equation} 
\end{defn}

Note that the element of $W\!_L(x;\xi)$ is a finite sum of 
products of elements of $\mathbb C[x]$ and rational functions of 
$(\frac{\p}{\p x_1},\dots,\frac{\p}{\p x_n},\xi_1,\dots,\xi_N)$.

We will introduce an automorphism of $W(x;\xi)$.
\begin{defn}[gauge transformation]\label{def:gauge}\index{gauge transformation}
\index{00Ad@$\Ad$, $\Adei$}
Fix an element $(h_1,\dots,h_n)\in\mathbb C(x,\xi)^n$ satisfying
\begin{equation}
   \frac{\p h_i}{\p x_j} = \frac{\p h_j}{\p x_i}
   \qquad(1\le i,\ j\le n).
\end{equation}
We define an automorphism $\Adei(h_1,\dots,h_n)$ of $W(x;\xi)$
by
\begin{equation}
\begin{aligned}
  \Adei(h_1,\dots,h_n)(x_i)&=x_i&&(i=1,\dots,n),\\
  \Adei(h_1,\dots,h_n)(\tfrac{\p}{\p x_i})
  &=\tfrac{\p}{\p x_i}-h_i&&(i=1,\dots,n),\\
  \Adei(h_1,\dots,h_n)(\xi_\nu)&=\xi_\nu
   &&(\nu=1,\dots,N).
\end{aligned}
\end{equation}
Choose functions $f$ and $g$ satisfying
$\frac{\p g}{\p x_i}=h_i$ for $i=1,\dots,n$ and put
$f=e^g$ and
\begin{align}
  \Ad(f)&=\Ade(g)=\Adei(h_1,\dots,h_n).
\end{align}
\end{defn}

We will define a homomorphism of $W(x;\xi)$.
\index{coordinate transformation}
\begin{defn}[coordinate transformation]\label{def:coord}
\index{coordinate transformation}
\index{00T@$T_{\phi}^*$}
Let $\phi=(\phi_1,\dots,\phi_n)$ be an element of 
$\mathbb C(x_1,\dots,x_m,\xi)^n$ such that the rank of the matrix
\begin{equation}
 \Phi:=\Bigl(\frac{\p \phi_j}{\p x_i}
 \Bigr)_{\substack{1\le i\le m
\\ 1\le j\le n}}
\end{equation}
equals $n$ for a generic point $(x,\xi)\in\mathbb C^{m+N}$.
Let $\Psi=\bigl(\psi_{i,j}(x,\xi)\bigr)_{\substack{1\le i\le n\\1\le j\le m}}$
be an left inverse of $\Phi$, namely,
$\Psi\Phi$ is an identity matrix of size $n$
and $m\ge n$.
Then a homomorphism $T^*_{\phi}$ from 
$W(x_1,\dots,x_n;\xi)$ to 
$W(x_1,\dots,x_m;\xi)$ is defined by
\begin{equation}
 \begin{aligned}
  T^*_\phi(x_i)&=\phi_i(x)&&(1\le i\le n),\\
  T^*_\phi(\tfrac{\p}{\p x_i})&=\sum_{j=1}^m
  \psi_{i,j}(x,\xi)\tfrac{\p}{\p x_j}
  &&(1\le i\le n).
 \end{aligned}
\end{equation}
If $m>n$, we choose linearly independent elements
$h_\nu=(h_{\nu,1},\dots,h_{\nu,m})$ of $\mathbb C(x,\xi)^m$ for
$\nu=1,\dots,m-n$ such that $\psi_{i,1}h_{\nu,1}+\cdots+\psi_{i,m}h_{\nu,m}=0$
for $i=1,\dots,n$ and $\nu=1,\dots,m-n$ and put
\begin{equation}
 \mathcal K^*(\phi):=\sum_{\nu=1}^{m-n}\mathbb C(x,\xi)\sum_{j=1}^m 
 h_{\nu,j}\tfrac{\p}{\p x_j}\in W(x;\xi).
\end{equation}
\end{defn}
The meaning of these operations are clear as follows.
\begin{rem}
Let $P$ be an element of 
$W(x;\xi)$ and let $u(x)$ be an analytic solution of the equation 
$Pu=0$ with a parameter $\xi$.
Then under the notation in Definitions~\ref{def:Lap}--\ref{def:gauge}, we have
$(\Red P)u(x)=\bigl(\Ad(f)(P)\bigr)(f(x)u(x))=0$.
Note that $\Red P$ is defined up to the multiplications of non-zero
elements of $\mathbb C(\xi)$.

If a Laplace transform\index{Laplace transformation}
\begin{equation}
 (\mathcal R_k u)(x) 
 =\int_C e^{-x_1t_1-\cdots-x_kt_k}u(t_1,\dots,t_k,x_{k+1},\dots,x_n)
  dt_1\cdots dt_k
\end{equation}
of $u(x)$ is suitably defined, then 
$\bigl(\Lap_{\{1,\dots,k\}}(\Red P)\bigr)(\mathcal R_ku)=0$,
which follows from the equalities
$\frac{\p \mathcal R_ku}{\p x_i}=\mathcal R_k(-x_iu)$
and $0=\int_C\frac{\p}{\p t_i}\bigl(e^{-x_1t_1-\cdots-x_kt_k}
u(t,x_{k+1},\ldots)\bigr)dt
=-x_i\mathcal R_ku+\mathcal R_k(\frac{\p u}{\p t_i})$
for $i=1,\dots,k$.
Moreover we have 
\[
 f(x)\mathcal R_k \Red Pu=f(x)\bigl(L_{\{1,\dots,k\}}(\Red P)\bigr)(\mathcal R_k u)
 =\bigl(\Ad(f)L_{\{1,\dots,k\}}(\Red P)\bigr)\bigl(f(x)\mathcal R_k u\bigr).
\]
Under the notation of Definition~\ref{def:coord}, we have
$T_\phi^*(P) u(\phi_1(x),\dots,\phi_n(x))=0$ and $Q u\bigl(\phi_1(x),\dots,\phi_n(x)\bigr)=0$ for $Q\in\mathcal K^*(\phi)$.

Another transformation of $W[x;\xi]$ based on an integral transformation frequently used will be given in Proposition~\ref{prop:1-tx}.
\end{rem}

We introduce some notation for combinations of operators we have defined.
\begin{defn}
\index{00Rad@$\RAd$, $\RAdL$}
Retain the notation in Definition \ref{def:Lap}--\ref{def:coord}
and recall that $f=e^g$ and $h_i=\frac{\p g}{\p x_i}$.
\begin{align}
  \RAd(f)
   &\phantom{:}=
    \RAde(g)=\RAdei(h_1,\dots,h_n):=\Red\circ\Adei(h_1,\dots,h_n),\\
\begin{split}
  \AdL(f)&\phantom{:}=\AdeL(h)=\AdeiL(h_1,\dots,h_n)\\
   &:=\Lap^{-1}\circ\Adei(h_1,\dots,h_n)\circ\Lap,
\end{split}
\\
\begin{split}
  \RAdL(f)&\phantom{:}=\RAdeL(h)=\RAdeiL(h_1,\dots,h_n)\\
   &:=\Lap^{-1}\circ\RAdei(h_1,\dots,h_n)\circ\Lap,
\end{split}
\\
 \Ad(\p_{x_i}^{\mu})&:=\Lap^{-1}\circ\Ad(x_i^\mu)\circ \Lap,
\\
 \RAd(\p_{x_i}^{\mu})&:=\Lap^{-1}\circ\RAd(x_i^\mu)\circ \Lap.
\end{align}
Here $\mu$ is a complex number or an element of $\mathbb C(\xi)$
and $\Ad(\p_{x_i}^{\mu})$ defines an endomorphism of 
$W_L(x;\xi)$.
\end{defn}

We will sometimes denote $\frac{\p}{\p x_i}$ by $\p_{x_i}$ or $\p_i$
for simplicity. If $n=1$, we usually denote $x_1$ by $x$ and
$\frac{\p}{\p x_1}$ by 
$\frac d{dx}$ or $\p_x$ or $\partial$.
We will give some examples.

Since the calculation $\Ad\bigl(x^{-\mu})\p=x^{-\mu}\circ \p\circ x^{\mu}
=x^{-\mu}(x^{\mu}\p + \mu x^{\mu-1})=\p + \mu x^{-1}$ is allowed, 
the following calculation is justified by the isomorphism \eqref{eq:WL}:
\begin{align*}
 \Ad(\p^{-\mu})x^m&=\p^{-\mu}\circ x^m\circ \p^\mu\\
 &=(x^m\p^{-\mu}+\tfrac{(-\mu)m}{1!}x^{m-1}\p^{-\mu-1}+
\tfrac{(-\mu)(-\mu-1)m(m-1)}{2!}x^{m-2}\p^{-\mu-2}\\
 &\qquad{}
   +\cdots+\tfrac{(-\mu)(-\mu-1)\cdots(-\mu-m+1)m!}{m!}\p^{-\mu-m})\p^\mu\\
 &=\sum_{\nu=0}^m(-1)^\nu(\mu)_\nu\binom{m}{\nu}x^{m-\nu}\p^{-\nu}.
\end{align*}
This calculation is in a ring of certain pseudo-differential operators
according to Leibniz's rule.
In general, we may put $\Ad(\p^{-\mu})P=\p^{-\mu}\circ P\circ \p^\mu$
for $P\in W[x;\xi]$ under Leibniz's rule.
Here $m$ is a positive integer and we use the notation
\index{000gammak@$(\gamma)_k,\ (\mu)_\nu$}
\begin{equation}
  (\mu)_\nu:=\prod_{i=0}^{\nu-1}(\mu+i),\quad
  \binom{m}{\nu}:=\frac{\Gamma(m+1)}{\Gamma(m-\nu+1)\Gamma(\nu+1)}
=\frac{m!}{(m-\nu)!\nu!}.
\end{equation}
\subsection{Examples of ordinary differential operators}
In this paper we mainly study ordinary differential operators.
We give examples of the operations we have defined, which are related
to classical differential equations.
\begin{exmp}
[$n=1$]
For a rational function $h(x,\xi)$ of $x$ with a parameter $\xi$
we denote by $\int h(x,\xi)dx$ the function $g(x,\xi)$ satisfying 
$\frac{d}{dx}g(x,\xi)=h(x,\xi)$. 
Put  $f(x,\xi)=e^{g(x,\xi)}$ and define
\index{000thetaz@$\vartheta$}
\begin{equation}
 \vartheta := x\frac{d}{dx}.
\end{equation}
Then we have the following identities.
\begin{align}
\Adei(h)\p &= \p - h = \Ad(e^{\int h(x)dx})\p
           = e^{\int h(x)dx}\circ \p \circ  e^{-\int h(x)dx},
\allowdisplaybreaks\\
\Ad(f)x &=x,\quad \AdL(f)\p = \p,\\
\Ad(\lambda f)&=\Ad(f)\quad
\AdL(\lambda f)=\AdL(f),\\
\Ad(f)\p&=\p - h(x,\xi)\ \Rightarrow\ \AdL(f)x=x + h(\p,\xi),
\allowdisplaybreaks\\
\Ad\bigl((x-c)^\lambda\bigr)
    &=\Ade\bigl(\lambda\log(x-c)\bigr)
    =\Adei\bigl(\tfrac{\lambda}{x-c}\bigr),
\allowdisplaybreaks\\
\Ad\bigl((x-c)^\lambda\bigr)x&=x,\quad 
 \Ad\bigl((x-c)^\lambda\bigr)\p=\p - \tfrac{\lambda}{x-c},
\allowdisplaybreaks\\
\RAd\bigl((x-c)^\lambda\bigr)\p
 &=\Ad\bigl((x-c)^\lambda\bigr)\bigl((x-c)\p\bigr)
 =(x-c)\p - \lambda,
\allowdisplaybreaks\\
\begin{split}
  \RAdL\bigl((x-c)^\lambda\bigr)x&=
  L^{-1}\circ \RAd\bigl((x-c)^\lambda\bigr)(-\p)\\
  &= L^{-1}\bigl((x-c)(-\p)+\lambda\bigr)\\
  &=
  (\p-c)x+\lambda=x\p - cx +1+\lambda,
\end{split}
\allowdisplaybreaks\\
\RAdL\bigl((x-c)^\lambda\bigr)\p&=\p,\quad 
  \RAdL\bigl((x-c)^\lambda\bigr)\bigl((\p-c)x\bigr)
  =(\p-c)x+\lambda,
\allowdisplaybreaks\\
\Ad(\p^{\lambda})\vartheta &= \AdL(x^{\lambda})\vartheta=\vartheta+\lambda,
\allowdisplaybreaks\index{000deltas@$\p^{\lambda}$}\\
\Ad\bigl(e^{\tfrac {\lambda(x-c)^m}{m}}\bigr)x &= x,\quad 
\Ad\bigl(e^{\tfrac {\lambda(x-c)^m}{m}}\bigr)\p = \p - \lambda(x-c)^{m-1},
\allowdisplaybreaks\\
\RAdL\bigl(e^{\frac {\lambda(x-c)^m}{m}}\bigr)x&=
\begin{cases}
 x+\lambda(\p-c)^{m-1}&(m\ge1),\\
 (\p-c)^{1-m}x + \lambda&(m\le -1),
\end{cases}\allowdisplaybreaks\\
T_{(x-c)^m}^*(x)&=(x-c)^m,\quad 
T_{(x-c)^m}^*(\p)=\tfrac1m (x-c)^{1-m}\p.
\end{align}
Here $m$ is a non-zero integer and $\lambda$ is a non-zero complex
number.
\end{exmp}

Some operations are related to Katz's operations
defined by \cite{Kz}.
The operation $\RAd\bigl((x-c)^\mu\bigr)$ corresponds to
the \textsl{addition} given in \cite{DR} and the operator
\begin{equation}\index{00mc@$mc_\mu$}
 mc_\mu:=\RAd(\p^{-\mu})=\RAdL(x^{-\mu})
\end{equation}
corresponds to Katz's \textsl{middle convolution} and the Euler transformation
or the Riemann-Liouville integral
\index{Euler transformation}\index{Riemann-Liouville integral}
\index{fractional derivation}
(cf.~\cite[\S5.1]{Kh}) or the fractional derivation
\index{00Ic@$I_c^\mu,\ \tilde I_c^\mu$}
\begin{equation}\label{eq:fracdif}
  (I^\mu_c(u))(x) = \frac1{\Gamma(\mu)}\int_c^x u(t)(x-t)^{\mu-1}dt.
\end{equation}
\index{addition}\index{middle convolution}%
Here $c$ is suitably chosen.  In most cases, $c$ is a singular point
of the multi-valued holomorphic function $u(x)$.  The integration may be understood
through an analytic continuation with respect to a parameter or in the
sense of generalized functions.  
When $u(x)$ is a multi-valued holomorphic function on the punctured disk
around $c$, we can define the complex integral
\begin{equation}\label{eq:EuPh}
 (\tilde I_c^\mu (u))(x)
 := \int^{(x+,c+,x-,c-)}\!\!\!\!u(z)(x-z)^{\mu-1}dz\quad \ \ 
\raisebox{5pt}{\begin{xy}
(-2,-7) *{c}="c", (32,-7) *{x}="x",
(0,0)  *{\times} *\cir<3.0mm>{l_d} *\cir<4.5mm>{l_dr},
(30,0) *{\times} *\cir<3.0mm>{r_u} *\cir<4.5mm>{ur_lu},
(12,-7)  *{\text{starting point}}
\ar@{<-{*}} (2.8,0) ; (14,0);
\ar@{-}  (14,0); (27.2,0);
\ar@{-} (0,-2.85); (13.5,-2,85);
\ar@{->}  (13.5,-2.85); (27,-2.85);
\ar@{<-} (30,2.85) ; (3,2.85);
\ar@{->} (0,-4.3) ; (24,-4.3):
\ar@{-} (24,-4.3) ; (27,3);
\ar@{.>} "c";(0,0),
\ar@{.>} "x";(30,0),
\ar@{<.} (2.8,0);(1,-7)
\end{xy}}
\end{equation}
through Pochhammer contour\index{Pochhammer contour} 
$(x+,c+,x-,c-)$ along a double loop circuit (cf.\ \cite[12.43]{WW}).
If $(z-c)^{-\lambda}u(z)$ is a meromorphic function in a neighborhood of the 
point $c$, we have
\begin{equation}
 (\tilde I_c^\mu(u))(x)=\bigl(1-e^{2\pi\lambda\sqrt{-1}}\bigr)
 \bigl(1-e^{2\pi\mu\sqrt{-1}}\bigr)
 \int_c^xu(t)(x-t)^{\mu-1}dt.
\end{equation}
For example, we have
\begin{align}
 \begin{split}
 I_c^\mu\bigl((x-c)^\lambda\bigr)
  &=\frac1{\Gamma(\mu)}\int_c^x (t-c)^\lambda(x-t)^{\mu-1}dt\\
  &=\frac{(x-c)^{\lambda+\mu}}{\Gamma(\mu)}\int_0^1 s^\lambda(1-s)^{\mu-1}ds
   \quad(x-t=(1-s)(x-c))\\
  &=\frac{\Gamma(\lambda+1)}{\Gamma(\lambda+\mu+1)}(x-c)^{\lambda+\mu},
 \end{split}\\
 \tilde I_c^\mu\bigl((x-c)^\lambda\bigr)
  &=\frac{4\pi^2e^{\pi(\lambda+\mu)\sqrt{-1}}}
    {\Gamma(-\lambda)\Gamma(1-\mu)\Gamma(\lambda+\mu+1)}(x-c)^{\lambda+\mu+1}.
\end{align}
For $k\in\mathbb Z_{\ge0}$ we have
\begin{align}
 \tilde I_c^\mu\bigl((x-c)^k\log(x-c)\bigr)=
 \frac{-4\pi^2k!e^{\pi\lambda\sqrt{-1}}}{\Gamma(1-\mu)\Gamma(\mu+k+1)} (x-c)^{\mu+k+1}.
\end{align}

We note that since
\begin{align*}
  \tfrac{d}{dt}\big(u(t)(x-t)^{\mu-1}\bigr)
 &=u'(t)(x-t)^{\mu-1}
   -\tfrac{d}{dx}\bigl(u(t)(x-t)^{\mu-1}\bigr)\\
\intertext{and}
 \tfrac{d}{dt}\big(u(t)(x-t)^{\mu}\bigr)
 &=u'(t)(x-t)^{\mu}-u(t)\tfrac{d}{dx}(x-t)^{\mu}\\
 &=xu'(t)(x-t)^{\mu-1} - tu'(t)(x-t)^{\mu-1}
   -\mu u(t)(x-t)^{\mu-1},
\end{align*}
we have
\begin{equation}\label{eq:Icp}
 \begin{split}
   I_c^\mu(\p u)&=\p I_c^\mu(u),\\
   I_c^\mu(\vartheta u)&=(\vartheta - \mu)I_c^\mu(u).
 \end{split}
\end{equation}

\begin{rem}\label{rem:defIc}
{\rm i)} \ 
The integral \eqref{eq:fracdif} is naturally well-defined and
the equalities \eqref{eq:Icp} are valid if $\RE\lambda > 1$ and 
$\lim_{x\to c}x^{-1}u(x)=0$.
Depending on the definition of $I_c^\lambda$, they are also 
valid in many cases, which can be usually proved in this 
paper by analytic continuations with respect to certain parameters
(for example, cf.~\eqref{eq:IcP}).
Note that \eqref{eq:Icp} is valid if $I_c^\mu$ is replaced
by $\tilde I_c^\mu$ defined by \eqref{eq:EuPh}.

{\rm ii)}
Let $\epsilon$ be a positive number and let $u(x)$ be a holomorphic function 
on  
\[
  U^+_{\epsilon,\theta}:=\{x\in\mathbb C\,;\,|x-c|<\epsilon\text{ and }
  e^{-i\theta}(x-c)\notin(-\infty,0]\}.
\]
Suppose that there exists a positive number $\delta$ 
such that $|u(x)(x-c)^{-k}|$ is bounded on
$\{x\in U_{\epsilon.\theta}^+\,;\,|\Arg(x-c)-\theta|<\delta\}$ 
for any $k>0$.
Note that the function $Pu(x)$ also satisfies this estimate for $P\in W[x]$.
Then the integration \eqref{eq:fracdif} is defined along a suitable 
path $C:\gamma(t)$ $(0\le t\le 1$) such that $\gamma(0)=c$, $\gamma(1)=x$ and
$|\Arg\bigl(\gamma(t)-c\bigr)-\theta|<\delta$ for $0<t<\frac12$
and the equalities \eqref{eq:Icp} are valid.
\end{rem}

\begin{exmp}\label{ex:midconv}
We apply additions, middle convolutions and Laplace transformations 
to the trivial ordinary differential equation
\begin{equation}
 \frac{du}{dx} = 0,
\end{equation}
which has the solution $u(x)\equiv 1$.

\textrm{i)\ }(Gauss hypergeometric equation).%
\index{hypergeometric equation/function!Gauss}
Put
\begin{equation}
 \begin{split}
  P_{\lambda_1,\lambda_2,\mu}:\!&=
  \RAd\bigl(\p^{-\mu}\bigr)\circ\RAd\bigl(
   x^{\lambda_1}(1-x)^{\lambda_2}\bigr)\p\\
 &= \RAd(\p^{-\mu}\bigr)\circ\Red
   (\p-\tfrac{\lambda_1}{x}+\tfrac{\lambda_2}{1-x})\\
 &=\RAd(\p^{-\mu}\bigr)\bigl(x(1-x)\p
     -\lambda_1(1-x)+\lambda_2 x\bigr)\\
 &=\RAd(\p^{-\mu}\bigr)
  \bigl((\vartheta-\lambda_1)-x(\vartheta-\lambda_1-\lambda_2)\bigr)\\
 &=\Ad(\p^{-\mu}\bigr)
  \bigl((\vartheta+1-\lambda_1)\p-(\vartheta+1)(\vartheta-\lambda_1-\lambda_2)\bigr)\\
 &=(\vartheta+1-\lambda_1-\mu)\p-(\vartheta+1-\mu)(\vartheta-\lambda_1-\lambda_2-\mu)\\
 &=(\vartheta+\gamma)\p-(\vartheta+\beta)(\vartheta+\alpha)\\
 &=x(1-x)\p^2+\bigl(\gamma-(\alpha+\beta+1)x\bigr)\p-\alpha\beta
 \end{split}
\end{equation}
with
\begin{equation}
\begin{cases}
 \alpha=-\lambda_1-\lambda_2-\mu, \\
 \beta=1-\mu,\\
 \gamma=1-\lambda_1-\mu.
\end{cases}
\end{equation}
We have a solution
\begin{equation}\label{eq:Gmid}
 \begin{split}
   u(x)&=I_0^\mu(x^{\lambda_1}(1-x)^{\lambda_2})\\
       &=\frac1{\Gamma(\mu)}
         \int_0^x t^{\lambda_1}(1-t)^{\lambda_2}(x-t)^{\mu-1}dt\\
       &=\frac{x^{\lambda_1+\mu}}{\Gamma(\mu)}
         \int_0^1 s^{\lambda_1}(1-s)^{\mu-1}(1-xs)^{\lambda_2}ds
         \quad(t=xs)\\
       &=\frac{\Gamma(\lambda_1+1)x^{\lambda_1+\mu}}{\Gamma(\lambda_1+\mu+1)}
         F(-\lambda_2,\lambda_1+1,\lambda_1+\mu+1;x)\\
       &=\frac{\Gamma(\lambda_1+1)x^{\lambda_1+\mu}
         (1-x)^{\lambda_2+\mu}}{\Gamma(\lambda_1+\mu+1)}
         F(\mu,\lambda_1+\lambda_2+\mu,\lambda_1+\mu+1;x)\\
       &=\frac{\Gamma(\lambda_1+1)x^{\lambda_1+\mu}(1-x)^{-\lambda_2}}
         {\Gamma(\lambda_1+\mu+1)}
         F(\mu,-\lambda_2,\lambda_1+\mu+1;\frac x{x-1})
 \end{split}
\end{equation}
of the Gauss hypergeometric equation 
$P_{\lambda_1,\lambda_2,\mu}u=0$ with the Riemann scheme
\begin{equation}
  \begin{Bmatrix}
    x=0 & 1 & \infty\\
      0 &0&1-\mu&\!\!;\ x\\
      \lambda_1+\mu &\lambda_2+\mu&-\lambda_1-\lambda_2-\mu
  \end{Bmatrix},
\end{equation}
which is transformed by the middle convolution $mc_\mu$ from
the Riemann scheme\index{Riemann scheme}
\[
 \begin{Bmatrix}
   x= 0 & 1 & \infty\\
    \lambda_1 & \lambda_2 & -\lambda_1-\lambda_2&\!\!;x
 \end{Bmatrix}
\]
of $x^{\lambda_1}(1-x)^{\lambda_2}$.
Here using Riemann's $P$ symbol, we note that
\begin{align*}
  &P\begin{Bmatrix}
    x=0 & 1 & \infty\\
      0 &0&1-\mu&\!\!;\ x\\
      \lambda_1+\mu &\lambda_2+\mu&-\lambda_1-\lambda_2-\mu
  \end{Bmatrix}\allowdisplaybreaks\\
 &=x^{\lambda_1+\mu}
 P\begin{Bmatrix}
    x=0 & 1 & \infty\\
      -\lambda_1-\mu &0&\lambda_1+1&\!\!;\ x\\
       0 &\lambda_2+\mu&-\lambda_2
  \end{Bmatrix}\allowdisplaybreaks\\
 &=x^{\lambda_1+\mu}(1-x)^{\lambda_2+\mu}
 P\begin{Bmatrix}
    x=0 & 1 & \infty\\
      -\lambda_1-\mu &-\lambda_2-\mu&\lambda_1+\lambda_2+\mu+1&\!\!;\ x\\
       0 &0&\mu
  \end{Bmatrix}\allowdisplaybreaks\\
 &=x^{\lambda_1+\mu}
 P\begin{Bmatrix}
    x=0 & 1 & \infty\\
      -\lambda_1-\mu &\lambda_1+1&0&\!\!;\ \dfrac x{x-1}\\
       0 &-\lambda_2 &\lambda_2+\mu
  \end{Bmatrix}\allowdisplaybreaks\\
 &=x^{\lambda_1+\mu}(1-x)^{-\lambda_2}
 P\begin{Bmatrix}
    x=0 & 1 & \infty\\
      -\lambda_1-\mu &\lambda_1+\lambda_2+1&
      -\lambda_2&\!\!;\ \dfrac x{x-1}\\
       0 & 0 &\mu
  \end{Bmatrix}.
\end{align*}
In general, the Riemann scheme and its relation to $mc_\mu$ will be studied
in \S\ref{sec:index} and the symbol `$P$' will be omitted for simplicity.

The function $u(x)$ defined by \eqref{eq:Gmid} 
corresponds to the characteristic exponent $\lambda_1+\mu$ at the
origin and depends meromorphically on the parameters
$\lambda_1,\lambda_2$ and $\mu$.
The local solutions corresponding to the characteristic exponents
$\lambda_2+\mu$ at $1$ and $-\lambda_1-\lambda_2-\mu$ at $\infty$
are obtained by replacing $I_0^\mu$ by $I_1^\mu$ and $I_\infty^\mu$,
respectively.  

When we apply $\Ad(x^{\lambda_1'}(x-1)^{\lambda_2'})$ to 
$P_{\lambda_1,\lambda_2,\mu}$, the resulting Riemann scheme is
\begin{equation}
  \begin{Bmatrix}
    x=0 & 1 & \infty\\
      \lambda_1' &\lambda_2'&1-\lambda_1'-\lambda_2'-\mu&\!\!;\ x\\
      \lambda_1+\lambda_1'+\mu &\lambda_2+\lambda_2'+\mu&-\lambda_1-\lambda_2-\lambda_1'-\lambda_2'-\mu,
  \end{Bmatrix}.
\end{equation}
Putting $\lambda_{1,1}=\lambda_1'$, 
$\lambda_{1,2}=\lambda_1+\lambda_1'+\mu$,
$\lambda_{2,1}=\lambda_2'$, 
$\lambda_{2,2}=\lambda_2+\lambda_2'+\mu$,
$\lambda_{0,1}=1-\lambda_1'-\lambda_2'-\mu$
and $\lambda_{0,2}=-\lambda_1-\lambda_2-\lambda_1'-\lambda_2'-\mu$,
we have the Fuchs relation 
\begin{equation}
 \lambda_{0,1}+\lambda_{0,2}+\lambda_{1,1}+\lambda_{1,2}+\lambda_{2,1}
 +\lambda_{2,2}=1
\end{equation}
and the corresponding operator
\begin{equation}\label{eq:GH2}
\begin{split}
 P_{\lambda}&=x^2(x-1)^2\partial^2+x(x-1)
 \bigl((\lambda_{0,1}+\lambda_{0,2}+1)x+\lambda_{1,1}+\lambda_{1,2}-1\bigr)
 \partial\\
&{}\quad+\lambda_{0,1}\lambda_{0,2}x^2+
  (\lambda_{2,1}\lambda_{2,2}-\lambda_{0,1}\lambda_{0,2}-\lambda_{1,1}\lambda_{1,2}   )x+\lambda_{1,1}\lambda_{1,2}
\end{split}\end{equation}
has the Riemann scheme
\begin{equation}\label{eq:GRSGG}
  \begin{Bmatrix}
    x=0 & 1 & \infty\\
      \lambda_{0,1} &\lambda_{1,1}&\lambda_{2,1}&\!\!;\ x\\
      \lambda_{0,2} &\lambda_{1,2}&\lambda_{2,2}
  \end{Bmatrix}.
\end{equation}
By the symmetry of the transposition 
$\lambda_{j,1}$ and $\lambda_{j,2}$ for each $j$, we have integral 
representations of other local solutions.

\textrm{ii)\ }(Airy equations).\index{Airy equation}
For a positive integer $m$ we put
\begin{equation}
\begin{split}
 P_m &:= \Lap\circ\Ad(e^\frac{x^{m+1}}{m+1})\p\\
     &\phantom{:}= \Lap(\p - x^m) = x - (-\p)^m.
\end{split}
\end{equation}
Thus the equation
\begin{equation}
  \frac{d^mu}{dx^m}-(-1)^mxu=0
\end{equation}
has a solution
\begin{equation}
 u_j(x) = \int_{C_j}\exp\left(\frac{z^{m+1}}{m+1}-xz\right)
 dz\qquad(0\le j\le m),
\end{equation}
where the path $C_j$ of the integration is
\begin{equation*}
 C_j: z(t)= e^{\frac{(2j-1)\pi\sqrt{-1}}{m+1}-t}
          + e^{\frac{(2j+1)\pi\sqrt{-1}}{m+1}+t}
\quad(-\infty<t<\infty). 
\end{equation*}
Here we note that $u_0(x)+\cdots+u_m(x)=0$.
The equation has the symmetry under the rotation 
$x\mapsto e^{\frac{2\pi\sqrt{-1}}{m+1}}x$.

\textrm{iii)\ }(Jordan-Pochhammer equation).\index{Jordan-Pochhammer}
For $\{c_1,\dots,c_p\}\in\mathbb C\setminus\{0\}$ put
\begin{align*}
 P_{\lambda_1,\dots,\lambda_p,\mu} :\!&=
 \RAd(\p^{-\mu})\circ\RAd\Bigl(\prod_{j=1}^p (1-c_jx)^{\lambda_j}\Bigr)\p\\
 &=\RAd(\p^{-\mu})\circ
  \Red\Bigl(\p+\sum_{j=1}^p\frac{c_j\lambda_j}{1-c_jx}\Bigr)\\
 &=\RAd(\p^{-\mu})\Bigl(p_0(x)\p+q(x)\Bigr)\\
 &=\p^{-\mu+p-1}\Bigl(p_0(x)\p+q(x)\Bigr)\p^{\mu}
  =\sum_{k=0}^pp_k(x)\p^{p-k}
\intertext{with}
 p_0(x)&=\prod_{j=1}^p(1-c_jx),\quad
 q(x)=p_0(x)\sum_{j=1}^p\frac{c_j\lambda_j}{1-c_jx},\\
 p_k(x)&=\binom{-\mu+p-1}{k}
  p_0^{(k)}(x) + 
  \binom{-\mu+p-1}{k-1}
  q^{(k-1)}(x),\\
  \binom{\alpha}{\beta}
 :\!&=
  \frac{\Gamma(\alpha+1)}{\Gamma(\beta+1)\Gamma(\alpha-\beta+1)}
 \quad(\alpha,\beta\in\mathbb C).
\end{align*}
We have solutions
\[
 u_j(x)=\frac{1}{\Gamma(\mu)}\int_{\frac1{c_j}}^x\prod_{\nu=1}^p
   (1-c_\nu t)^{\lambda_\nu}(x-t)^{\mu-1}dt
 \quad(j=0,1,\dots,p,\ c_0=0)
\]
of the Jordan-Pochhammer equation $P_{\lambda_1,\dots,\lambda_p,\mu}u=0$ with
the Riemann scheme
\index{000lambda@$[\lambda]_{(k)}$}
\begin{equation}
 \begin{Bmatrix}
  x=\frac1{c_1} & \cdots & \frac1{c_p} & \infty\\
  [0]_{(p-1)} & \cdots & [0]_{(p-1)}& [1-\mu]_{(p-1)}&\!\!;\,x\\
  \lambda_1+\mu&\cdots & \lambda_p+\mu & -\lambda_1-\dots-\lambda_p-\mu
 \end{Bmatrix}.
\end{equation}

Here and hereafter we use the notation
\begin{equation}\label{eq:mult}
 [\lambda]_{(k)}:=
 \begin{pmatrix}
   \lambda \\ \lambda+1 \\ \vdots \\ \lambda+k-1
 \end{pmatrix}
\end{equation}
for a complex number $\lambda$ and a non-negative integer $k$.
If the component $[\lambda]_{(k)}$ is appeared in a Riemann scheme,
it means the corresponding local solutions with the exponents
$\lambda+\nu$ for $\nu=0,\dots,k-1$ have a semisimple local 
monodromy when $\lambda$ is generic.
\end{exmp}
\subsection{Ordinary differential equations}\label{sec:ODE}
We will study the ordinary differential equation
\begin{equation}\label{eq:M}
  \mathcal M : Pu = 0
\end{equation}
with an element $P\in W(x;\xi)$ in this paper.
The solution $u(x,\xi)$ of $\mathcal M$ is at least locally 
defined for $x$ and $\xi$ and holomorphically or meromorphically
depends on $x$ and $\xi$.
Hence we may replace $P$ by $\Red P$ and we similarly 
choose $P$ in $W[x;\xi]$.

We will identify $\mathcal M$ with the left $W(x;\xi)$-module
$W(x;\xi)/W(x;\xi)P$. 
Then we may consider \eqref{eq:M} as the fundamental relation of the generator 
$u$ of the module $\mathcal M$.

The results in this subsection are standard and well-known but for
our convenience we briefly review them.  First note that $W(x;\xi)$
is a (left) Euclidean ring:

Let $P$, $Q\in W(x;\xi)$ with $P\ne0$.
Then there uniquely exists $R$, $S\in W(x;\xi)$ such that
\begin{equation}\label{eq:Div}
 Q = SP + R\qquad(\ord R<\ord P).
\end{equation}
Hence we note that 
$\dim_{\mathbb C(x,\xi)}\bigl(W(x;\xi)/W(x;\xi)P\bigr)=\ord P$.
We get $R$ and $S$ in \eqref{eq:Div} by a simple algorithm as follows.
Put 
\begin{equation}\label{eq:PQ}
 P=a_n\p^n+\cdots+a_1\p+a_0 \text{ \ and \ } Q=b_m\p^m+\cdots+b_1\p+b_0
\end{equation}
with $a_n\ne 0$, $b_m\ne 0$.  Here $a_n$, $b_m\in\mathbb C(x,\xi)$.
The division \eqref{eq:Div} is obtained by the induction on $\ord Q$.  
If $\ord P>\ord Q$, \eqref{eq:Div} is trivial with $S=0$.  
If $\ord P \le \ord Q$, \eqref{eq:Div} is reduced to the equality
$Q'=S'P+R$ with $Q'=Q - a_n^{-1}b_m\p^{m-n}P$ and 
$S'=S-a_n^{-1}b_m\p^{m-n}$ and then we have $S'$ and $R$ satisfying
$Q'=S'P+R$ by the induction because $\ord Q'<\ord Q$.
The uniqueness of \eqref{eq:Div} is clear 
by comparing the highest order terms of \eqref{eq:Div} in the case when $Q=0$.

By the standard Euclid algorithm using the division 
\eqref{eq:Div} we have $M$, $N\in W(x;\xi)$ 
such that
\begin{equation}\label{eq:Euc}
 MP+NQ=U,\ P\in W(x;\xi)U \text{ \ and \ } Q\in W(x;\xi)U.
\end{equation}
Hence in particular any left ideal of $ W(x;\xi)$ is 
generated by a single element of $W[x;\xi]$, namely, $W(x;\xi)$
is a principal ideal domain.
\begin{defn}
The operators $P$ and $Q$ in $W(x;\xi)$ are defined to be 
\textsl{mutually prime} if one of the following equivalent conditions is valid.
\index{differential equation/operator!mutually prime}
\begin{align}
 &W(x;\xi)P + W(x;\xi)Q = W(x;\xi),\\
 &\text{there exists }R\in W(x;\xi)\text{ satisfying }
 RQu=u\text{ for the equation }Pu=0,\\
   &\begin{cases}
    \text{the simultaneous equation }Pu=Qu=0\text{ has not a non-zero solution}\\
    \text{for a generic value of }\xi.
    \end{cases}
\end{align}
\end{defn}

Moreover we have the following.
\index{differential equation/operator!cyclic}
\begin{equation}\label{eq:Wsing}
\text{Any left $W(x;\xi)$-module $\mathcal R$ with 
 $\dim_{\mathbb C(x,\xi)}\mathcal R<\infty$ is \textsl{cyclic},}
\end{equation}
namely, it is generated by a single element.
Hence any system of ordinary differential equations is
isomorphic to a single differential equation under the algebra 
$W(x;\xi)$.

To prove \eqref{eq:Wsing} it is sufficient to show that the direct sum 
$\mathcal M\oplus\mathcal N$ 
of $\mathcal M: Pu=0$ and $\mathcal N: Qv=0$ is cyclic.
In fact $M\oplus\mathcal N=W(x;\xi)w$ with 
$w=u+(x-c)^nv\in\mathcal M\oplus\mathcal N$ and $n=\ord P$ if 
$c\in\mathbb C$ is generic.
For the proof we have only to show
$\dim_{\mathbb C(x,\xi)}W(x;\xi)w\ge m+n$ and
we may assume that $P$ and $Q$ are in $W[x;\xi]$ and they are of
the form \eqref{eq:PQ}.
Fix $\xi$ generically and 
we choose $c\in\mathbb C$ such that $a_n(c)b_m(c)\ne 0$.
Since the function space 
$V=\{\phi(x)+(x-c)^n\varphi(x)\,;\,P\phi(x)=Q\varphi(x)=0\}$
is of dimension $m+n$ in a neighborhood of $x=c$, 
$\dim_{W(x;\xi)}W(x;\xi)w\ge m+n$
because the relation $Rw=0$ for an operator 
$R\in W(x;\xi)$ implies $R\psi(x)=0$ for $\psi\in V$.

Thus we have the following standard definition.
\begin{defn}\label{def:irred}
Fix $P\in W(x;\xi)$ with $\ord P>0$.
The equation \eqref{eq:M} is \textsl{irreducible} 
if and only if one of the following equivalent conditions is valid.
\index{differential equation/operator!irreducible}
\begin{align}
&\text{The left $W(x;\xi)$-module $\mathcal M$ is simple.}
\label{eq:irrsimp}\allowdisplaybreaks\\
&\text{The left $W(x;\xi)$-ideal $W(x;\xi)P$ is maximal.}\allowdisplaybreaks\\
&P=QR\text{ with }Q,\,R\in W(x;\xi)\text{ implies }
 \ord Q\cdot\ord R=0.\label{eq:QR}\allowdisplaybreaks\\
&\forall Q\not\in W(x;\xi)P,\ 
\exists M,\ N\in W(x;\xi)\text{ satisfying }MP+NQ=1.
\label{eq:irrEuc}\allowdisplaybreaks\\
 &\begin{cases}
   ST\in W(x;\xi)P\text{ with }S, T\in W(x;\xi)
   \text{ and }\ord S<\ord P \\
  \Rightarrow S=0\text{ or }T\in W(x;\xi)P.
  \end{cases}\label{eq:ST}
\end{align}
The equivalence of the above conditions is standard and easily proved.
The last condition may be a little non-trivial.

Suppose \eqref{eq:ST} and $P=QR$ and $\ord Q\cdot\ord R\ne 0$.
Then $R\notin W(x;\xi)P$ and therefore $Q=0$, which contradicts
to $P=QR$.  Hence \eqref{eq:ST} implies \eqref{eq:QR}.

Suppose \eqref{eq:irrsimp}, \eqref{eq:irrEuc}, 
$ST\in W(x;\xi)P$ and $T\notin W(x;\xi)P$.
Then there exists $P'$ such that 
$\{J\in W(x;\xi)\,;\,JT\in W(x;\xi)P\}= 
W(x;\xi)P'$,  $\ord P'=\ord P$ and moreover $P'v=0$ is also
simple. Since $Sv=0$ with $\ord S<\ord P'$, we have $S$=0.

In general, a system of ordinary differential equations
is defined to be irreducible if it is simple as a left 
$W(x;\xi)$-module.
\end{defn}

\begin{rem}
Suppose the equation $\mathcal M$ given in \eqref{eq:M} is irreducible.

{\rm i) }
Let $u(x,\xi)$ be a non-zero solution  of 
$\mathcal M$, which is locally defined for the variables $x$ and $\xi$
and meromorphically depends on $(x,\xi)$. 
If $S\in W[x;\xi]$ satisfies $Su(x,\xi)=0$, then
$S\in W(x;\xi)P$.
Therefore $u(x,\xi)$ determines $\mathcal M$.

{\rm ii) }
Suppose $\ord P>1$.
Fix $R\in W(x;\xi)$ such that $\ord R<\ord P$ and $R\ne 0$.
For $Q\in W(x;\xi)$ and a positive integer $m$,
the condition $R^mQu=0$ is equivalent to $Qu=0$.  
Hence for example, if $Q_1u+\p^mQ_2u=0$ with certain 
$Q_j\in W(x;\xi)$, we will 
allow the expression $\p^{-m}Q_1u+Q_2u=0$ and
$\p^{-m}Q_1u(x,\xi)+Q_2u(x,\xi)=0$. 

{\rm iii) }
For $T\not\in W(x;\xi)P$  we construct a differential
equation $Qv=0$ satisfied by $v=Tu$ as follows.  
Put $n=\ord P$.
We have $R_j\in  W(x;\xi)$ such that 
$\p^jTu=R_ju$ with $\ord R_j<\ord P$.
Then there exist $b_0,\dots,b_n\in\mathbb C(x,\xi)$  such that
$b_nR_n+\cdots+b_1R_1+b_0R_0=0$. Then $Q=b_n\p^n+\cdots+b_1\p+b_0$. 
\end{rem}
\subsection{Okubo normal form and Schlesinger canonical form}\label{sec:DR}
\index{Schlesinger canonical form}\index{Okubo normal form}

In this subsection we briefly explain the interpretation of Katz's middle 
convolution (cf.~\cite{Kz}) by \cite{DR} and its relation to our fractional 
operations.

For constant square matrices $T$ and $A$ of size $n'$, the ordinary
differential equation
\begin{equation}\label{eq:ONF}
  (xI_{n'}-T)\frac{du}{dx}=Au
\end{equation} 
is called \textsl{Okubo normal form} of Fuchsian system
when $T$ is a diagonal matrix.  
\index{Okubo normal form}
Then
\begin{equation}
  mc_\mu\bigl((xI_{n'}-T)\p -A\bigr)=(xI_{n'}-T)\p -(A+\mu I_{n'})
\end{equation}
for generic $\mu\in\mathbb C$, namely, the system is transformed into
\begin{equation}\label{eq:mcONF}
  (xI_{n'}-T)\frac{du_\mu}{dx}=\bigl(A+\mu I_{n'}\bigr)u_\mu
\end{equation} 
by the operation $mc_\mu$.
Hence for a solution $u(x)$ of \eqref{eq:ONF}, the 
Euler transformation $u_\mu(x)=I_c^\mu(u)$ of $u(x)$ 
satisfies \eqref{eq:mcONF}.

For constant square matrices $A_j$ of size $m$
and the \textsl{Schlesinger canonical form}
\begin{equation}\label{eq:SCF}
  \frac{d v}{dx} = \sum_{j=1}^p\frac{A_j}{x-c_j}v
\end{equation}
of a Fuchsian system of the Riemann sphere, we have\\[-20pt]
\begin{equation}\label{eq:S2O}
  \frac{du}{dx} = \sum_{j=1}^p
  \frac{\tilde A_j}{x-c_j}u,\quad
  \tilde A_j:= \bordermatrix{
      & & \cr
      & & \cr
      j)&A_1&\cdots&A_p\cr
      & & \cr}
  \text{ \ and \ }
  u:=\begin{pmatrix}
         \frac{v}{x-c_1}\\\vdots\\\frac{v}{x-c_p}
        \end{pmatrix}.
\end{equation}
Here $\tilde A_j$ are square matrices of size $pm$.
The addition $\Ad\bigl((x-c_k)^{\mu_k}\bigr)$ 
transforms $A_j$ into $A_j+\mu_k\delta_{j,k}I_m$ for $j=1,\dots,p$
in the system \eqref{eq:SCF}.
Putting 
\[A=\tilde A_1+\cdots+\tilde A_p\text{ \ and \ }
 T=\left(\begin{smallmatrix}
    c_1I_n\\&\ddots\\&&c_pI_n
   \end{smallmatrix}\right),
\]
the equation \eqref{eq:S2O} is equivalent to
\eqref{eq:ONF} with $n'=pm$.
Define square matrices of size $n'$ by 
\begin{align}
 \tilde A&:=
  \begin{pmatrix}
   A_1\\
    & \ddots\\
    && A_p
 \end{pmatrix},\\[-5pt]
 \tilde A_j(\mu)
 &:= \bordermatrix{
      &    &     &        & \underset{\smile}{j} \cr
      & & \cr
      j)&A_1&\cdots&A_{j-1}&A_j+\mu&A_{j+1}&\cdots& A_p\cr
      & & \cr}.\label{eq:conSch}
\end{align}
Then $\ker\tilde A$ and $\ker(A+\mu)$ are invariant
under $\tilde A_j(\mu)$ for $j=1,\dots,p$
and therefore $\tilde A_j(\mu)$ induce 
endomorphisms of 
$V:=\mathbb C^{pm}/\bigl(\ker\tilde A+\ker (A+\mu)\bigr)$,
which correspond to square matrices of size $N:=\dim V$, which
we put $\bar A_j(\mu)$,
respectively, under a fixed basis
of $V$.  Then the middle convolution $mc_\mu$ of \eqref{eq:SCF}
is the system
\begin{equation}\label{eq:mcS}
  \frac{dw}{dx} = \sum_{j=1}^p \frac{\bar A_j(\mu)}{x-c_j}w
\end{equation}
of rank $N$, which is defined and studied by \cite{DR, DR2}.
Here $\ker\tilde A \cap\ker(A+\mu)=\{0\}$ if $\mu\ne0$.

We define another realization of the middle convolution as in
\cite[\S2]{O2}. Suppose $\mu\ne0$.
The square matrices of size $n'$
\begin{align}
 A_j^\vee(\mu)
 &:= 
\bordermatrix{
         &     &&   \underset{\smile}{j} && \cr
      & && A_1\cr
      & && \vdots\cr
    j\,{\text{\tiny$)$}} &&& A_j+\mu && \cr
      & && \vdots\cr
      & && A_p
  }\text{ \ and \ }
 A^\vee(\mu) := A^\vee_1(\mu)+\cdots+A^\vee_p(\mu)
\end{align}
satisfy
\begin{align}
 \tilde A(A+\mu I_{n'}) &= A^\vee(\mu)\tilde A
 =
 \Bigl(A_iA_j+\mu\delta_{i,j}A_i\Bigr)
_{\substack{1\le i\le p \\ 1\le j\le p}}\in M(n',\mathbb C),\\
\tilde A (A+\mu I_{n'})\tilde A_j(\mu)
&=  A_j^\vee(\mu)\tilde A(A+\mu I_{n'}).
\end{align}
Hence $w^\vee :=\tilde A (A+\mu I_{n'}) u$ satisfies
\begin{align}
  \frac{d w^\vee}{dx} 
   &= \sum_{j=1}^p\frac{A_j^\vee(\mu)}{x-c_j}w^\vee,\\
  \sum_{j=1}^p\frac{A_j^\vee(\mu)}{x-c_j}
   &=\biggl(\frac{A_i+\mu\delta_{i,j}I_m}{x-c_j}\biggr)
     _{\substack{ 1\le i\le p,\\1\le j\le p}}\notag
\end{align}
and $\tilde A(A+\mu I_{n'})$ 
induces the isomorphism
\begin{equation}
\tilde A(A+\mu I_{n'}):
 V=\mathbb C^{n'}/(\mathcal K+\mathcal L_\mu)
 \ \overset{\sim}\to \ V^\vee:=\IM \tilde A(A+\mu I_{n'})
 \subset\mathbb C^{n'}.
\end{equation}
Hence putting $\bar A_j^\vee(\mu):=A_j^\vee(\mu)|_{V^\vee}$,
the system \eqref{eq:mcS} is isomorphic to the system
\begin{equation}\label{eq:SCFmcv}
  \frac{d w^\vee}{dx} = \sum_{j=1}^p\frac{\bar A_j^\vee(\mu)}{x-c_j}
  w^\vee
\end{equation}
of rank $N$,
which can be  regarded as a middle convolution $mc_\mu$ of 
\eqref{eq:SCF}.  Here
\begin{equation}
 w^\vee=\begin{pmatrix}w^\vee_1\\ \vdots\\ w_p^\vee\end{pmatrix},\quad
  w^\vee_j = \sum_{\nu=1}^p(A_jA_\nu + \mu\delta_{j,\nu})
            (u_\mu)_\nu 
            \quad(j=1,\dots,p)
\end{equation}
and if $v(x)$ is a solution of \eqref{eq:SCF}, then 
\begin{equation}
 w^\vee(x)=\biggl(\sum_{\nu=1}^p(A_jA_\nu+\mu \delta_{j,\nu})
 I_c^\mu\Bigl(\frac{v(x)}{x-c_\nu}\Bigr)\biggr)_{j=1,\dots,p}
\end{equation}
satisfies \eqref{eq:SCFmcv}.

Since any non-zero homomorphism between irreducible 
$W(x)$-modules is an isomorphism, we have the following
remark (cf.~\S\ref{sec:ODE} and \S\ref{sec:contig}).
\begin{rem}\label{rem:SCFmc}
Suppose that the systems \eqref{eq:SCF} and \eqref{eq:SCFmcv} are irreducible.
Moreover suppose the system \eqref{eq:SCF} is isomorphic to
a single Fuchsian differential equation $P\tilde u=0$ 
as left $W(x)$-modules and the equation
$mc_\mu(P)\tilde w=0$ is also irreducible.
Then the system \eqref{eq:SCFmcv} is isomorphic to the single equation
$mc_\mu(P)\tilde w=0$ because the differential equation satisfied by 
$I_c^\mu(\tilde u(x))$ is isomorphic to that of 
$I_c^\mu(Q\tilde u(x))$ for a non-zero solution $v(x)$ of $P\tilde u=0$ 
and an operator $Q\in W(x)$ with $Q\tilde u(x)\ne0$
(cf.~\S\ref{sec:contig}, Remark~\ref{rem:midisom} iii) and Proposition~\ref{prop:irred}).

In particular if the systems are rigid and their spectral parameters
are generic, all the assumptions here are satisfied
(cf.~Remark~\ref{rem:generic} ii) and Corollary~\ref{cor:irred}).
\end{rem}

Yokoyama \cite{Yo2} defines extension and restriction operations among
the systems of differential equations of Okubo normal form.
The relation of Yokoyama's operations to Katz's operations
is clarified by \cite{O4}, which shows that they are equivalent
from the view point of the construction and the reduction of systems
of Fuchsian differential equations.

\section{Confluences}
\subsection{Regular singularities}\label{sec:reg}
\index{regular singularity}
In this subsection we review fundamental facts related to the 
regular singularities of the ordinary differential equations.
\subsubsection{Characteristic exponents}\index{characteristic exponent}
The ordinary differential equation
\begin{equation}
  a_n(x)\tfrac{d^n u}{dx^n}+a_{n-1}(x)\tfrac{d^{n-1} u}{dx^{n-1}}+
  \cdots + a_1(x)\tfrac{d u}{dx}+a_0(x)u=0\label{eq:ode}
\end{equation}
of order $n$ with meromorphic functions $a_j(x)$ defined in a neighborhood
of $c\in\mathbb C$ has a singularity at $x=c$ if the function 
$\frac{a_j(x)}{a_n(x)}$ has a pole at $x=c$ for a certain $j$.
The singular point $x=c$ of the equation is a \textsl{regular singularity} 
if it is a removable singularity of 
the functions $b_j(x):=(x-c)^{n-j}a_j(x)a_n(x)^{-1}$ for $j=0,\dots,n$.
In this case $b_j(c)$ are complex numbers and the $n$ roots of 
the \textsl{indicial equation}\index{indicial equation}
\begin{equation}
  \sum_{j=0}^n b_j(c)s(s-1)\cdots(s-j+1)=0
\end{equation}
are called the \textsl{charactersitic exponents} of \eqref{eq:ode} at $c$.

Let $\{\lambda_1,\dots,\lambda_n\}$ be the set of these characteristic 
exponents at $c$.

If $\lambda_j-\lambda_1\notin\mathbb Z_{>0}$ for 
$1<j\le n$, then \eqref{eq:ode} has a unique solution 
$(x-c)^{\lambda_1}\phi_1(x)$
with a holomorphic function $\phi_1(x)$ in a neighborhood of $c$ satisfying
$\phi_1(c)=1$.
\begin{defn}\label{def:exp}
The regular singularity and the characteristic exponents
for the differential operator
\begin{equation}\label{eq:ODP}
  P=a_n(x)\tfrac{d^n}{dx^n}+a_{n-1}(x)\tfrac{d^{n-1}}{dx^{n-1}}+
  \cdots+a_1(x)\tfrac d{dx}+a_0(x)
\end{equation}
are defined by those of the equation \eqref{eq:ode}, respectively.
Suppose $P$ has a regular singularity at $c$. We say $P$ is 
\textsl{normalized at $c$} if $a_n(x)$ is holomorphic at $c$ and 
\index{regular singularity!normalized}
\begin{equation}
  a_n(c)=a_n^{(1)}(c)=\cdots=a_n^{(n-1)}(c)=0\text{ \ and \ }
  a_n^{(n)}(c)\ne 0.
\end{equation}
In this case $a_j(x)$ are analytic and have zeros
of order at least $j$ at $x=c$ for $j=0,\dots,n-1$.

\subsubsection{Local solutions}
\index{regular singularity!local solution}
\index{00O@$\mathcal O,\ \hat{\mathcal O},\ \mathcal O_c,\ \mathcal O_c(\mu,m)$}
The ring of convergent power series at $x=c$ is denoted by $\mathcal O_c$
and for a complex number $\mu$ and a non-negative integer $m$ we put
\begin{equation}\label{def:Ocm}
  \mathcal O_c(\mu,m) :=\bigoplus_{\nu=0}^m
  \mathcal (x-c)^\mu\log^\nu(x-c)\mathcal O_c.
\end{equation}
\end{defn}
Let $P$ be a differential operator of order $n$ which has a regular 
singularity at $x=c$ and let $\{\lambda_1,\cdots,\lambda_n\}$ be 
the corresponding characteristic exponents.
Suppose $P$ is normalized at $c$.
If a complex number $\mu$ satisfies $\lambda_j-\mu\notin\{0,1,2,\dots\}$
for $j=1,\dots,n$, then $P$
defines a linear bijective map
\begin{equation}\label{eq:Pbij}
 P: \mathcal O_c(\mu,m)\ \overset{\sim}{\to}\ \mathcal O_c(\mu,m)
\end{equation}
for any non-negative integer $m$.

Let $\hat {\mathcal O}_c$ be the ring of formal power series
$\sum_{j=0}^\infty a_j(x-c)^j$ $(a_j\in\mathbb C)$ of $x$ at $c$.
For a domain $U$ of $\mathbb C$ we denote by $\mathcal O(U)$ the ring of
holomorphic functions on $U$.
Put
\index{00Brc@$B_r(c)$}
\begin{equation}\label{eq:Brc}
  B_r(c):=\{x\in\mathbb C\,;\,|x-c|<r\}
\end{equation}
for $r>0$ and
\begin{align}
 \hat{\mathcal O}_c(\mu,m) &:=\bigoplus_{\nu=0}^m
 \mathcal (x-c)^\mu\log^\nu(x-c)\hat{\mathcal O}_c,\\
 \mathcal O_{B_r(c)}(\mu,m) &:= \bigoplus_{\nu=0}^m
 \mathcal (x-c)^\mu\log^\nu(x-c)\mathcal O_{B_r(c)} .
\end{align}
Then $\mathcal O_{B_r(c)}(\mu,m)\subset\mathcal O_c(\mu,m)
\subset \hat {\mathcal O}_c(\mu,m)$.

Suppose $a_j(x)\in\mathcal O\bigl(B_r(c)\bigr)$ and $a_n(x)\ne 0$ for 
$x\in B_r(c)\setminus\{c\}$
and moreover $\lambda_j-\mu\notin \{0,1,2,\ldots\}$, we have
\begin{align}
 P&: \mathcal O_{B_r(c)}(\mu,m) \ \overset{\sim}{\to} \ \ 
 \mathcal O_{B_r(c)}(\mu,m),\label{eq:PBr}\\
 P&: \hat{\mathcal O}_c(\mu,m) \ \overset{\sim}{\to} \ \ 
 \hat{\mathcal O}_c(\mu,m).\label{eq:Phat}
\end{align}
The proof of these results are reduced to the case when $\mu=m=c=0$ 
by the translation $x\mapsto x-c$,  the operation $\Ad\bigl(x^{-\mu}\bigr)$, 
and the fact $P(\sum_{j=0}^m f_j(x)\log^jx) = (Pf_m(x))\log^jx+
\sum_{j=0}^{m-1}\phi_j(x)\log^jx$ with suitable $\phi_j(x)$
and moreover we may assume
\begin{align*}
  P &= \prod_{j=0}^n(\vartheta -\lambda_j) - xR(x,\vartheta),\\
  xR(x,\vartheta) &= x\sum_{j=0}^{n-1}r_j(x)\vartheta^j
  \quad(r_j(x)\in\mathcal O\bigl(B_r(c)\bigr)).
\end{align*}
When $\mu=m=0$, \eqref{eq:Phat} is easy and \eqref{eq:PBr} and hence \eqref{eq:Pbij}
are also easily proved by the method of majorant series
(for example, cf.~\cite{O0}).

For the differential operator
\[
  Q = \tfrac{d^n}{dx^n}+b_{n-1}(x)\tfrac{d^{n-1}}{dx^{n-1}}
  + \cdots +b_1(x)\tfrac{d}{dx} + b_0(x)
\]
with $b_j(x)\in\mathcal O\bigl(B_r(c)\bigr)$, we have a bijection
\begin{equation}
  \begin{matrix}
  Q:&\mathcal O\bigl(B_r(c)\bigr) & \overset{\sim}{\to} & 
  \mathcal O\bigl(B_r(c)\bigr)\oplus\mathbb C^n\phantom{ABCDEFG}\\
   & \rotatebox{90}{$\in$} & &\rotatebox{90}{$\in$}\phantom{ABCD}\\
   &u(x)&\mapsto &Pu(x)\oplus \bigl(u^{(j)}(c)\bigr)_{0\le j\le n-1}
  \end{matrix}
\end{equation}
because $Q(x-c)^n$ has a regular singularity at $x=c$ and the characteristic exponents
are $-1,-2,\dots,-n$ and hence \eqref{eq:PBr}
assures that for any $g(x)\in\mathbb C[x]$ and $f(x)\in\mathcal O\bigl(B_r(c)\bigr)$ 
there uniquely exists $v(x)\in\mathcal O\bigl(B_r(c)\bigr)$ such that 
$Q(x-c)^nv(x)= f(x)-Qg(x)$.

If $\lambda_\nu-\lambda_1\notin\mathbb Z_{>0}$, the characteristic exponents of
$R:=\Ad\bigl((x-c)^{-\lambda_1-1}\bigr)P$ at $x=c$ are $\lambda_\nu-\lambda_1-1$
for $\nu=1,\dots,n$ and therefore $R=S(x-c)$ with a differential operator $R$
whose coefficients are in $\mathcal O\bigl(B_r(c)\bigr)$.  
Then there exists $v_1(x)\in\mathcal O\bigl(B_r(c)\bigr)$ such that $-S1=S(x-c) v_1(x)$,
which means $P\bigl((x-c)^{\lambda_1}(1+(x-c)v_1(x))\bigr)=0$.
Hence if $\lambda_i-\lambda_j\notin\mathbb Z$ for $1\le i<j\le n$,
we have solutions $u_\nu(x)$ of $Pu=0$ such that
\begin{equation}
  u_\nu(x)=(x-c)^{\lambda_\nu}\phi_\nu(x)
\end{equation}
with suitable $\phi_\nu\in\mathcal O\bigl(B_r(c)\bigr)$ satisfying $\phi_\nu(c)=1$
for $\nu=1,\dots,n$.

Put $k=\#\{\nu\,;\,\lambda_\nu=\lambda_1\}$ and 
$m=\#\{\nu\,;\,\lambda_\nu-\lambda_1\in\mathbb Z_{\ge0}\}$.
Then we have solutions $u_\nu(x)$ of $Pu=0$ for $\nu=1,\dots,k$ such that
\begin{equation}
  u_\nu(x) - (x-c)^{\lambda_1}\log^{\nu-1}(x-c)
  \in\mathcal O_{B_r(c)}(\lambda_1+1,m-1).
\end{equation}
If $\mathcal O_{B_r(c)}$ is replaced by $\hat{\mathcal O}_c$, 
the solution 
\[u_{\nu}(x) = (x-c)^{\lambda_1}\log^{\nu-1}(x-c)+
 \sum_{i=1}^\infty\sum_{j=0}^{m-1}c_{\nu,i,j}(x-c)^{\lambda_1+i}\log^j(x-c)
 \in \hat{\mathcal O}_c(\lambda_1,m-1)
\]
is constructed by inductively defining $c_{\nu,i,j}\in\mathbb C$.
Since
\[
\begin{split}
 &P\Bigl(\sum_{i=N+1}^\infty\sum_{j=0}^{m-1}c_{\nu,i,j}(x-c)^{\lambda_1+i}\log^j(x-c) \Bigr)=
 -P\Bigl((x-c)^{\lambda_1}\log^{\nu-1}(x-c)\\
 &\qquad +\sum_{i=1}^N c_{\nu,i,j}(x-c)^{\lambda_1+i}\log^j(x-c) \Bigr)
  \in\mathcal O_{B_r(c)}(\lambda_1+N,m-1)
\end{split}
\]
for an integer $N$ satisfying $\RE(\lambda_\ell-\lambda_1)<N$
for $\ell=1,\dots,n$, we have 
\[
 \sum_{i=N+1}^\infty\sum_{j=0}^{m-1}c_{\nu,i,j}(x-c)^{\lambda_1+i}\log^j(x-c)\in
 \mathcal O_{B_r(c)}(\lambda_1+N,m-1)
\]
because of \eqref{eq:PBr} and \eqref{eq:Phat},
which means $u_{\nu}(x)\in\mathcal O_{B_r(c)}(\lambda_1,m)$.

\subsubsection{Fuchsian differential equations}
\index{Fuchsian differential equation/operator}
The regular singularity at $\infty$ is similarly defined by
that at the origin under the coordinate transformation $x\mapsto\frac1x$.
When $P\in W(x)$ and the singular points of $P$ in 
$\overline{\mathbb C}:=\mathbb C\cup\{\infty\}$ are all regular singularities, 
the operator $P$ and the equation $Pu=0$ are called \textsl{Fuchsian}.
\index{Fuchsian differential equation/operator}
Let $\overline{\mathbb C}'$ be the subset of $\overline{\mathbb C}$ deleting
singular points $c_0,\ldots,c_p$ from $\overline{\mathbb C}$.  
Then the solutions of the equation $Pu=0$ defines a map
\begin{equation}\label{eq:Fmap}
  \mathcal F:\ \overline{\mathbb C}'\supset U:\text{(simply connected domain)}
  \mapsto \mathcal F(U)\subset\mathcal O(U)
\end{equation}
by putting $\mathcal F(U):=\{u(x)\in\mathcal O(U)\,;\,Pu(x)=0\}$.
Put
\[
  U_{j,\epsilon,R}=
 \begin{cases}
  \{x=c_j+re^{\sqrt{-1}\theta}\,;\,0<r<\epsilon,\ R<\theta<R+2\pi\}
   &(c_j\ne\infty)\\
  \{x = re^{\sqrt{-1}\theta}\,;\, r>\epsilon^{-1},\  R<\theta<R+2\pi\}
   &(c_j=\infty).
 \end{cases}
\]
For simply connected domains $U$, $V\subset\overline{\mathbb C}'$, 
the map $\mathcal F$ satisfies 
\begin{align}
 &\mathcal F(U) \subset\mathcal O(U)\text{ \ and \ } \dim\mathcal F(U)= n,
 \label{eq:F0}\\
 &V\subset U\ \Rightarrow\ \mathcal F(V)=\mathcal F(U)|_V,\label{eq:F1}\\
 &\begin{cases}
  \exists \epsilon > 0,\ \forall \phi\in\mathcal F(U_{j,\epsilon,R}),\ 
  \exists C>0, \exists m>0\text{\ such that\ }\\
  |\phi(x)|<
     \begin{cases}
        C|x-c_j|^{-m} & (c_j\ne\infty,\ x\in U_{j,\epsilon,R}),\\
        C|x|^m & (c_j=\infty,\ x\in U_{j,\epsilon,R})\\
        & \text{\quad for }j=0,\dots,p,\ \forall R\in\mathbb R.
     \end{cases}
  \end{cases}\label{eq:F2}
\end{align}
Then we have the bijection
\begin{equation}\label{eq:FisP}
 \begin{matrix}
  \bigl\{\p^n+\displaystyle\sum_{j=0}^{n-1}a_j(x)\p^j\in W(x)\,:\,\text{Fuchsian}\bigr\}
   &\overset{\sim}\to&
   \bigl\{\mathcal F\text{ satisfying \eqref{eq:F0}--\eqref{eq:F2}}\bigr\}\\[-8pt]
   \rotatebox{90}{$\in$}&&\rotatebox{90}{$\in$}\\
   P&\mapsto&\bigl\{U\mapsto\{u\in\mathcal O(U)\,;\,Pu=0\}\bigr\}.
 \end{matrix}
\end{equation}
Here if $\mathcal F(U)=\sum_{j=1}^n\mathbb C\phi_j(x)$, 
\begin{equation}\label{eq:F2E}
 a_j(x)=(-1)^{n-j}\frac{\det\Phi_j}{\det\Phi_n}\text{ \ with \ }
 \Phi_j=\begin{pmatrix}
        \phi_1^{(0)}(x) & \cdots & \phi_n^{(0)}(x)\\
        \vdots & \vdots & \vdots\\
        \phi_1^{(j-1)}(x) & \cdots & \phi_n^{(j-1)}(x)\\
        \phi_1^{(j+1)}(x) & \cdots & \phi_n^{(j+1)}(x)\\
        \vdots & \vdots & \vdots\\
        \phi_1^{(n)}(x) & \cdots & \phi_n^{(n)}(x)
        \end{pmatrix}.
\end{equation}
The elements $\mathcal F_1$ and $\mathcal F_2$ of the right hand side of
\eqref{eq:FisP} are naturally identified if there exists a simply connected 
domain $U$ such that $\mathcal F_1(U)=\mathcal F_2(U)$.

Let
\[
  P = \p^n+a_{n-1}(x)\p^{n-1}+\cdots+a_0(x)
\]
be a Fuchsian differential operator with $p+1$ regular singular points 
$c_0=\infty$,$c_1,\dots,c_p$ and let $\lambda_{j,1},\dots,\lambda_{j,n}$ be 
the characteristic exponents of $P$ at $c_j$, respectively.
Since $a_{n-1}(x)$ is holomorphic at $x=\infty$ and $a_{n-1}(\infty)=0$,
there exists $a_{n-1,j}\in\mathbb C$ such that $a_{n-1}(x)=
 -\sum_{j=1}^p\frac{a_{n-1,j}}{x-c_j}$.
For $c\in\mathbb C$ we have $x^n(\p^n - cx^{-1}\p^{n-1}\bigr)=
 \vartheta^n-\bigl(c+\frac{n(n-1)}2\bigr)\vartheta^{n-1}+
 c_{n-2}\vartheta^{n-2}+\cdots+c_0$ with $c_j\in\mathbb C$. 
Hence we have
\[
 \lambda_{j,1}+\cdots+\lambda_{j,n}
   =\begin{cases}
    -\sum_{j=1}^pa_{n-1,j}-\frac{n(n-1)}{2}&(j=0),\\
    a_{n-1,j}+\frac{n(n-1)}{2}&(j=1,\dots,p),
   \end{cases}
\]
and the \textsl{Fuchs relation}\index{Fuchs relation}
\begin{equation}\label{eq:FC0}
 \sum_{j=0}^p\sum_{\nu=1}^n\lambda_{j,\nu} = \frac{(p-1)n(n-1)}{2}.
\end{equation}

Suppose $Pu=0$ is reducible.
Then $P=SR$ with $S,\, R\in W(x)$ so that $n'=\ord R<n$.
Since the solution $v(x)$ of $Rv=0$ satisfies $Pv(x)=0$, $R$ is also Fuchsian.
Note that the set of $m$ characteristic exponents 
$\{\lambda'_{j,\nu}\,;\,\nu=1,\dots,n'\}$ of $Rv=0$ at $c_j$ is a subset of 
$\{\lambda_{j,\nu}\,;\,\nu=1,\dots,n\}$.
The operator $R$ may have other singular points $c_1',\ldots,c_q'$
called \textsl{apparent singular points}\index{apparent singularity}
where any local solutions at
the points is analytic.  Hence the set characteristic exponents
at $x=c'_j$ are $\{\lambda'_{j,\nu}\,\:\,\nu=1,\ldots,n'\}$
such that $0\le \mu_{j,1}<\mu_{j,2}<\cdots<\mu_{j,n'}$
and $\mu_{j,\nu}\in\mathbb Z$ for $\nu=1,\dots,n'$ and $j=1,\dots,q$.
Since $\mu_{j,1}+\cdots+\mu_{j,n'}\ge\frac{n'(n'-1)}2$, the Fuchs
relation for $R$ implies
\begin{equation}\label{eq:FC1}
 \mathbb Z\ni \sum_{j=0}^p\sum_{\nu=1}^{n'} \lambda'_{j,\nu}\le
 \frac{(p-1)n'(n'-1)}2.
\end{equation}

\index{differential equation/operator!irreducible}
Fixing a generic point $q$ and pathes $\gamma_j$ around $c_j$ as in \eqref{fig:mon}
and moreover a base $\{u_1,\dots,u_n\}$ of local solutions of the equation $Pu=0$ 
at $q$, we can define monodromy generators $M_j\in GL(n,\mathbb C)$.
We call the tuple $\mathbf M=(M_0,\dots,M_p)$ the \textsl{monodromy} of the 
equation $Pu=0$.  The monodromy $\mathbf M$ is defined to be \textsl{irreducible}
if there exists no subspace $V$ of $\mathbb C^n$ such that 
$M_j V\subset V_j$ for $j=0,\dots,p$ and $0<\dim V<n$, which is equivalent to 
the condition that $P$ is irreducible.
\index{monodromy}

Suppose $Qv=0$ is another Fuchsian differential equation of order $n$
with the same singular points.
The monodromy $\mathbf N=(N_0,\dots,N_p)$ is 
similarly defined by fixing a base $\{v_1,\dots,v_n\}$ of local solutions 
of $Qv=0$ at $q$.
Then 
\begin{equation}\label{eq:isoWM}
\begin{split}
\mathbf M\sim\mathbf N&\ \overset{\text{def}}\Leftrightarrow \ 
 \exists g\in GL(n,\mathbb C)\text{ such that }
 N_j=gM_jg^{-1}\ (j=0,\dots,p)\\
 & \ \Leftrightarrow\ Qv=0
 \text{ is $W(x)$-isomorphic to }
 Pu=0.
\end{split}
\end{equation}

If $Qv=0$ is $W(x)$-isomorphic to $Pu=0$, the 
isomorphism defines an isomorphism between their solutions and then 
$N_j=M_j$ under the bases corresponding to the isomorphism.

Suppose there exists $g\in GL(n,\mathbb C)$ such that
$N_j=gM_jg^{-1}$ for $j=0,\dots,p$.
The equations $Pu=0$ and $Qu=0$ are $W(x)$-isomorphic
to certain first order systems $U'=A(x)U$ and $V'=B(x)V$ of rank $n$,
respectively.
We can choose bases $\{U_1,\dots,U_n\}$ and $\{V_1,\dots,V_n\}$ 
of local solutions of $PU=0$ and $QV=0$ at $q$, respectively, 
such that their monodromy generators corresponding $\gamma_j$ are
same for each $j$.
Put $\tilde U=(U_1,\dots,U_n)$ and $\tilde V=(V_1,\dots,V_n)$.
Then the element of the matrix $\tilde V\tilde U^{-1}$ is holomorphic at $q$ and
can be extended to a rational function of $x$ and then  $\tilde V\tilde U^{-1}$ 
defines a $W(x)$-isomorphism between the equations $U'=A(x)U$ and $V'=B(x)V$.
\begin{exmp}[apparent singularity]
 The differential equation
\begin{equation}\label{eq:ApSing}
 x(x-1)(x-c)\tfrac{dy^2}{dx} + (x^2-2cx+c)\tfrac{dy}{dx}=0
\end{equation}
is a special case of Heun's equation \eqref{eq:Heun} with 
$\alpha=\beta=\lambda=0$ and $\gamma=\delta=1$.
It has regular singularities at $0$, $1$, $c$ and $\infty$ and its Riemann scheme
equals
\index{Heun's equation}
\begin{equation}
 \begin{Bmatrix}
   x= \infty & 0 & 1 & c\\
      0 & 0 & 0 & 0\\
      0 & 0 & 0 & 2
 \end{Bmatrix}.
\end{equation}

\index{Wronskian}
The local solution at $x=c$ corresponding to the characteristic exponent 0
is holomorphic at the point and therefore $x=c$ is an apparent singularity,
which corresponds to the zero of the Wronskian $\det\Phi_n$ in 
\eqref{eq:F2E}.
Note that the equation \eqref{eq:ApSing} has the solutions $1$ and 
$c\log x+(1-c)\log(x-1)$.
\index{apparent singularity}

The equation \eqref{eq:ApSing} is not $W(x)$-isomorphic to Gauss 
hypergeometric equation if $c\ne 0$ and $c\ne 1$, which follows from 
the fact that $c$ is a modulus of the isomorphic classes of the monodromy.
It is easy to show that any tuple of matrices 
$\mathbf M=(M_0,M_1,M_2)\in GL(2,\mathbb C)$ satisfying $M_2M_1M_0=I_2$ is 
realized as the monodromy of the equation obtained by applying a suitable addition 
$\RAd\bigl(x^{\lambda_0}(1-x)^{\lambda_1}\bigr)$ to a certain Gauss hypergeometric 
equation or the above equation.
\end{exmp}
\subsection{A confluence}
\index{confluence}
The non-trivial equation $(x-a)\tfrac{du}{dx} = \mu u$ 
obtained by the addition $\RAd\bigl((x-a)^\mu\bigr)\p$ 
has a solution $(x-a)^\mu$ and regular singularities at $x=c$ 
and $\infty$.  To consider 
the confluence of the point $x=a$ to $\infty$ we put $a=\frac 1c$.  
Then the equation is \[\bigl((1-cx)\p+c\mu\bigr)u=0\] and it has a 
solution $u(x)=(1-cx)^\mu$.

The substitution $c=0$ for the operator $(1-cx)\p+c\mu\in W[x;c,\mu]$
gives the trivial equation $\tfrac{du}{dx}=0$ with the trivial solution 
$u(x)\equiv1$.
To obtain a nontrivial equation we introduce the parameter 
$\lambda=c\mu$ and we have the equation 
\[\bigl((1-cx)\p+\lambda\bigr)u=0\] with the solution
$(1-cx)^{\frac \lambda c}$. 
The function $(1-cx)^{\frac \lambda c}$ has
the holomorphic parameters $c$ and $\lambda$ and the substitution
$c=0$ gives the equation $(\p+\lambda)u=0$ with the solution $e^{-\lambda x}$.
Here $(1-cx)\p+\lambda=\RAdei\bigl(\frac{\lambda}{1-cx}\bigr)\p
=\RAd\bigl((1-cx)^{\frac{\lambda}{c}}\bigr)\p$.

This is the simplest example of the confluence and we define
a confluence of simultaneous additions in this subsection.

\subsection{Versal additions}\label{sec:VAd}\index{versal!addition}
\index{confluence!addition}
For a function $h(c,x)$ with a holomorphic parameter $c\in\mathbb C$ 
we put
\begin{equation}
 \begin{split}
 h_n(c_1,\dots,c_n,x)
   &:=\frac1{2\pi\sqrt{-1}}\int_{|z|=R}
     \frac{h(z,x)dz}{\prod_{j=1}^n(z-c_j)}\\
   &=\sum_{k=1}^n\frac{h(c_k,x)}{\prod_{1\le i\le n,\, i\ne k}
   (c_k-c_i)}
 \end{split}
\end{equation}
with a sufficiently large $R>0$.
Put
\begin{equation}
 h(c,x):=c^{-1}\log(1-c x)= -x - \frac c2x^2-\frac{c^2}3x^3
-\frac{c^3}4x^4-\cdots.
\end{equation}
Then
\begin{equation}
  (1-c x)h'(c,x)=-1
\end{equation}
and
\begin{equation}\label{eq:hnid}
\begin{split}
 h'_n(c_1,\dots,c_n,x)\prod_{1\le i\le n}(1-c_ix)
  &=-\sum_{k=1}^n\frac{\prod_{1\le i\le n,\,i\ne k}(1-c_ix)}
        {\prod_{1\le i\le n,\,i\ne k}(c_k-c_i)}\\
  &=-x^{n-1}.
\end{split}
\end{equation}
The last equality in the above is obtained as follows.
Since the left hand side of \eqref{eq:hnid} 
is a holomorphic function of $(c_1,\dots,c_n)\in\mathbb C^n$ 
and the coefficient of $x^m$ is homogeneous of degree $m-n+1$, it is zero if
$m<n-1$.
The coefficient of $x^{n-1}$ proved to be $-1$ by putting $c_1=0$.
Thus we have
\begin{align}
  h_n(c_1,\dots,c_n,x)&=-\int_0^x\frac{t^{n-1}dt}{\prod_{1\le i\le n}(1-c_it)},\\
\begin{split}
  e^{\lambda_n h_n(c_1,\dots,c_n,x)}
  &\circ\Bigl(\prod_{1\le i\le n}\bigl(1-c_ix\bigr)\Bigr)\p\circ
  e^{-\lambda_n h_n(c_1,\dots,c_n,x)}\\&=
  \Bigl(\prod_{1\le i\le n}\bigl(1-c_ix\bigr)\Bigr)\p
  +\lambda_n x^{n-1},
\end{split}\\
  e^{\lambda_n h_n(c_1,\dots,c_n,x)}&=
 \prod_{k=1}^n\Bigl(1-c_kx\Bigr)^{\frac{\lambda_n}{c_k\prod_{\substack{1\le i\le n\\i\ne k}}(c_k-c_i)}}.
\end{align}
\begin{defn}[versal addition]
We put
\index{addition!confluence}\index{00AdV@$\AdV$, $\AdV^0$}
\begin{align}
\begin{split}
 \AdV_{(\frac1{c_1},\dots,\frac1{c_p})}(\lambda_1,\dots,\lambda_p)&:=
 \Ad\left(
 \prod_{k=1}^p
   \Bigl(1-c_kx\Bigr)^{
    \sum_{n=k}^p{\frac{\lambda_n}{c_k\prod_{\substack{1\le i\le n\\i\ne k}}
     (c_k-c_i)}}}\right)\\
 &\phantom{:}=\Adei\left(-\sum_{n=1}^p\frac{\lambda_n x^{n-1}}
  {\prod_{i=1}^n(1-c_ix)}\right),
\end{split}
\\
\RAdV_{(\frac1{c_1},\dots,\frac1{c_p})}(\lambda_1,\dots,\lambda_p)&=
 \Red\circ \AdV_{(\frac1{c_1},\dots,\frac1{c_p})}(\lambda_1,\dots,\lambda_p).
\end{align}
We call $\RAdV_{(\frac1{c_1},\dots,\frac1{c_p})}(\lambda_1,\dots,\lambda_p)$ 
a \textsl{versal addition} at the $p$ points
$\frac1{c_1},\dots,\frac1{c_p}$.
\end{defn}
\index{00RAdV@$\RAdV$}
Putting
\begin{align*}
 h(c,x) := \log(x-c),
\end{align*}
we have
\begin{align*}
 h'_n(c_1,\dots,c_n,x)\prod_{1\le i\le n}(x-c_i)
 = \sum _{k=1}^n \frac{\prod_{1\le i\le n,\ i\ne k}(x-c_i)}{\prod_{1\le i\le n,\ i\ne k}(c_k-c_i)}=1
\end{align*}
and the \textsl{conflunence of additions around the origin} is 
defined by
\index{addition!confluence!around the origin}
\begin{align}
\begin{split}
 \AdV_{(a_1,\dots,a_p)}^0(\lambda_1,\dots,\lambda_p):\!&=
 \Ad\left(
 \prod_{k=1}^p
    (x-a_k)^{\sum_{n=k}^p
    {\frac{\lambda_n}{\prod_{\substack{1\le i\le n\\i\ne k}}
     (a_k-a_i)}}}\right)\\
 &=\Adei
 \left(\sum_{n=1}^p\frac{\lambda_n}{\prod_{1\le i\le n}(x-a_i)}\right),
\end{split}
\\
\RAdV_{(a_1,\dots,a_p)}^0(\lambda_1,\dots,\lambda_p)&=
 \Red\circ \AdV_{(a_1,\dots,a_p)}^0(\lambda_1,\dots,\lambda_p).
\end{align}
\begin{rem}
Let $g_k(c,x)$ be meromorphic functions of $x$ with the holomorphic
parameter $c=(c_1,\dots,c_p)\in\mathbb C^p$ for $k=1,\dots,p$ such that
\[
  g_k(c,x) \in
  \sum_{i=1}^p\mathbb C\frac{1}{1-c_ix}
  \text{ \ if \ }0\ne c_i\ne c_j\ne 0\quad(1\le i< j\le p,\ 1\le k\le p).
\]
Suppose $g_1(c,x),\dots,g_p(c,x)$ are linearly independent for any 
fixed $c\in\mathbb C^p$.  Then there exist entire functions $a_{i,j}(c)$ of
$c\in\mathbb C^p$ such that
\[
  g_k(x,c) = \sum_{n=1}^p \frac{a_{k,n}(c)x^{n-1}}{\prod_{i=1}^n(1-c_ix)}
\]
and $\bigl(a_{i,j}(c)\bigr)\in GL(p,\mathbb C)$ for any $c\in\mathbb C^p$
(cf.~\cite[Lemma~6.3]{O1}).
Hence the versal addition is essentially unique.
\end{rem}
\subsection{Versal operators}\label{sec:Versal}\index{versal!operator}
If we apply a middle convolution to a versal addition of the trivial operator 
$\p$, we have a \textsl{versal Jordan-Pochhammer} operator. 
\index{Jordan-Pochhammer!versal}
\begin{align}
 P:\!&=\RAd(\p^{-\mu})\circ
     \RAdV_{(\frac1{c_1},\dots,\frac1{c_p})}(\lambda_1,\dots,\lambda_p)\p
 \\
 &=\RAd(\p^{-\mu})\circ\Red\Bigl(
   \p+\sum_{k=1}^p\frac{\lambda_kx^{k-1}}{\prod_{\nu=1}^k(1-c_\nu x)}
   \Bigr)
 \notag\\
 &=\p^{-\mu+p-1}\Bigl(p_0(x)\p+q(x)\Bigr)\p^{\mu}
  =\sum_{k=0}^p p_k(x)\p^{p-k}\notag
\intertext{with}
 p_0(x) &=\prod_{j=1}^p(1-c_jx),\quad
 q(x)    =\sum_{k=1}^p\lambda_kx^{k-1}\prod_{j=k+1}^p(1-c_jx),
\notag\\
 p_k(x)&=\binom{-\mu+p-1}{k}
  p_0^{(k)}(x) + 
  \binom{-\mu+p-1}{k-1}
  q^{(k-1)}(x)
.\notag
\end{align}
We naturally obtain the integral representation of solutions of
the versal Jordan-Pochhammer equation $Pu=0$,
which we show in the case $p=2$ as follows.
\begin{exmp}\label{ex:VGHG}
We have the \textsl{versal Gauss hypergeometric} operator
\index{hypergeometric equation/function!Gauss!versal}
\index{hypergeometric equation/function!Gauss!confluence}
\begin{align*}
 P_{c_1,c_2;\lambda_1,\lambda_2,\mu}:\!&=\RAd(\p^{-\mu})\circ
   \RAdV_{(\frac1{c_1},\frac1{c_2})}(\lambda_1,\lambda_2)\p\allowdisplaybreaks\\
  &=\RAd(\p^{-\mu})\circ
    \RAd\left((1-c_1x)^{\frac{\lambda_1}{c_1}+\frac{\lambda_2}{c_1(c_1-c_2)}}
(1-c_2x)^{\frac{\lambda_2}{c_2(c_2-c_1)}}\right)\\
  &=\RAd(\p^{-\mu})\circ
    \RAdei\left(-\tfrac{\lambda_1}{1-c_1x}
    -\tfrac{\lambda_2 x}{(1-c_1x)(1-c_2x)} \right)
    \p\allowdisplaybreaks\\
  &=\RAd(\p^{-\mu})\circ
   \Red\left(\p+\tfrac{\lambda_1}{1-c_1x}+\tfrac{\lambda_2x}{(1-c_1x)(1-c_2x)}
    \right)\allowdisplaybreaks\\
  &=\Ad(\p^{-\mu})\left(
     \p(1-c_1x)(1-c_2x)\p+\p(\lambda_1(1-c_2x)+\lambda_2 x)
    \right)\allowdisplaybreaks\\
  &=\bigl((1-c_1x)\p+c_1(\mu-1)\bigr)\bigl((1-c_2x)\p + c_2\mu\bigr)
 \\
  &\quad{}
   +\lambda_1\p + (\lambda_2-\lambda_1c_2)(x\p+1-\mu)\allowdisplaybreaks\\
  &=(1-c_1x)(1-c_2x)\p^2\\
  &\quad{}+\bigl((c_1+c_2)(\mu-1)+\lambda_1
   +(2c_1c_2(1-\mu)+\lambda_2-\lambda_1c_2)x\bigr)\p\\
  &\quad{}+
   (\mu-1)(c_1c_2\mu+\lambda_1c_2-\lambda_2),
\end{align*}
whose solution is obtained by applying $I_c^\mu$ to
\[
K_{c_1,c_2;\lambda_1,\lambda_2}(x)=
 (1-c_1x)^{\frac{\lambda_1}{c_1}+\frac{\lambda_2}{c_1(c_1-c_2)}}(1-c_2x)^{\frac{\lambda_2}{c_2(c_2-c_1)}}
\]
The equation $Pu=0$ has the Riemann scheme
\begin{equation}
  \begin{Bmatrix}
   x = \frac1{c_1} & \frac 1{c_2} & \infty\\
   0               & 0 &1-\mu&\!\!;\,x\\
   \frac{\lambda_1}{c_1}+\frac{\lambda_2}{c_1(c_1-c_2)} + \mu
    & \frac{\lambda_2}{c_2(c_2-c_1)} + \mu
    &-\frac{\lambda_1}{c_1}+\frac{\lambda_2}{c_1c_2}-\mu
  \end{Bmatrix}.
\end{equation}
Thus we have the following well-known confluent equations
\index{Kummer's equation}
\begin{align*}
P_{c_1,0;\lambda_1,\lambda_2,\mu}&=
 (1-c_1x)\p^2+\bigl(c_1(\mu-1)+\lambda_1+\lambda_2x\bigr)\p
 -\lambda_2(\mu-1),&&\text{(Kummer)}\\
K_{c_1,0;\lambda_1,\lambda_2}&=
 (1-c_1x)^{\frac{\lambda_1}{c_1}+\frac{\lambda_2}{c_1^2}}
 \exp({\tfrac{\lambda_2 x}{c_1}}),
\allowdisplaybreaks\\
P_{0,0;0,-1,\mu}
&=\p^2-x\p+(\mu-1),&&\text{(Hermite)}\\
\Ad(e^{\frac{1}4x^2})&P_{0,0;0,1,\mu}=
(\p-\tfrac{1}2x)^2+x(\p-\tfrac{1}2x)
 -(\mu-1)\allowdisplaybreaks\\
&=\p^2+(\tfrac12-\mu-\tfrac{x^2}4),
&&\text{(Weber)}\\
K_{0,0;0,\mp1}&=
 \exp\Bigl(\int_0^x \!\pm tdt\Bigr)=\exp(\pm \tfrac{x^2}{2}).
\end{align*}
The solution
\begin{align*}
 D_{-\mu}(x)&:=(-1)^{-\mu}e^{\frac{x^2}4}I_\infty^\mu(e^{-\frac{x^2}2})
  = \frac{e^{\frac{x^2}4}}{\Gamma(\mu)}
   \int_x^\infty e^{-\frac{t^2}2}(t-x)^{\mu-1}dt\\
 &= \frac{e^{\frac{x^2}4}}{\Gamma(\mu)}
   \int_0^\infty e^{-\frac{(s+x)^2}2}s^{\mu-1}ds
 = \frac{e^{-\frac{x^2}4}}{\Gamma(\mu)}
     \int_0^\infty e^{-xs-\frac{t^2}2}s^{\mu-1}ds\\
 &\sim x^{-\mu}e^{-\frac{x^2}4}
   {}_2F_0(\tfrac\mu2,\tfrac\mu2+\tfrac12;-\tfrac2{x^2})
  =\sum_{k=0}^\infty x^{-\mu}e^{-\frac{x^2}4}
   \frac{(\tfrac\mu2)_k(\tfrac\mu2+\tfrac12)_k}{k!}\bigl(-\tfrac2{x^2}\bigr)^k
\end{align*}
of Weber's equation $\frac{d^2u}{dx^2}=(\frac{x^2}4+\mu-\frac12)u$
is called a parabolic cylinder function
(cf.~\cite[\S16.5]{WW}).
Here the above last line is an asymptotic expansion when $x\to+\infty$.
\index{Weber's equation}\index{parabolic cylinder function}

The normal form of Kummer equation is obtained by the coordinate transformation
$y=x-\frac1{c_1}$ but we also obtain it as follows:
\begin{align*}
 P_{c_1;\lambda_1,\lambda_2,\mu}
  :\!&=\RAd(\p^{-\mu})\circ\Red\circ \Ad(x^{\lambda_2})\circ \AdV_{\frac1{c_1}}
(\lambda_1)\p
\allowdisplaybreaks\\
  &=\RAd(\p^{-\mu})\circ\Red\bigl(
   \p - \tfrac{\lambda_2}{x}+\tfrac{\lambda_1}{1-c_1x}\bigr)
\allowdisplaybreaks\\
  &=\Ad(\p^{-\mu})\bigl(\p x(1-c_1x)\p
   -\p(\lambda_2-(\lambda_1+c_1\lambda_2)x)\bigr)
\allowdisplaybreaks\\
  &=(x\p +1 -\mu)\bigl((1-c_1x)\p+c_1\mu)-\lambda_2\p
   +(\lambda_1+c_1\lambda_2)(x\p +1 -\mu)
\allowdisplaybreaks\\
  &=x(1-c_1x)\p^2+\bigl(1-\lambda_2-\mu
     +(\lambda_1+c_1(\lambda_2+2\mu-2))x\bigr)\p\\
  &\quad{}
     +(\mu-1)\bigl(\lambda_1+c_1(\lambda_2+\mu)\bigr),
\allowdisplaybreaks\\
 P_{0;\lambda_1,\lambda_2,\mu}
 &=x\p^2+(1-\lambda_2-\mu +\lambda_1x)\p
 +\lambda_1(\mu-1),\\
 P_{0;-1,\lambda_2,\mu}
 &=x\p^2+(1-\lambda_2-\mu - x)\p
 +1-\mu\qquad\text{(Kummer)},
\allowdisplaybreaks\\
 K_{c_1;\lambda_1,\lambda_2}(x)&:=x^{\lambda_2}
  (1-c_1x)^{\frac{\lambda_1}{c_1}},\quad
 K_{0;\lambda_1,\lambda_2}(x)=x^{\lambda_2}\exp(-\lambda_1x).
\end{align*}
The Riemann scheme of the equation $P_{c_1;\lambda_1,\lambda_2,\mu}u=0$ is 
\begin{align}
 \begin{Bmatrix}
  x=0 & \frac1{c_1} & \infty \\
    0 & 0           &1-\mu&\!\!;\,x\\
   \lambda_2+\mu & \frac{\lambda_1}{c_1}+\mu&-\frac{\lambda_1}{c_1}-\lambda_2-\mu 
 \end{Bmatrix}
\end{align}
and the local solution at the origin corresponding to the characteristic 
exponent $\lambda_2+\mu$ is given by
\begin{align*}
I_0^\mu(K_{c_1;\lambda_1,\lambda_2})(x)
 = \frac1{\Gamma(\mu)}\int_0^x
  t^{\lambda_2}
  (1-c_1t)^{\frac{\lambda_1}{c_1}}(x-t)^{\mu-1}dt.
\end{align*} 
In particular, 
we have a solution
\begin{align*}
 u(x)&=I_0^\mu(K_{0;-1,\lambda_2})(x)=
 \frac1{\Gamma(\mu)}\int_0^x t^{\lambda_2}e^{t}(x-t)^{\mu-1}dt\\
 &=\frac{x^{\lambda_2+\mu}}{\Gamma(\mu)}\int_0^1
   s^{\lambda_2}(1-s)^{\mu-1}e^{xs}ds\qquad(t=xs)\\
 &=\frac{\Gamma(\lambda_2+1)x^{\lambda_2+\mu}}{\Gamma(\lambda_2+\mu+1)}
   {}_1F_1(\lambda_2+1,\mu+\lambda_2+1;x)
\end{align*}
of the Kummer equation $P_{0;-1,\lambda_2,\mu}u=0$ corresponding to the
exponent $\lambda_2+\mu$ at the origin.
If $c_1\notin(-\infty,0]$ and $x\notin[0,\infty]$ and 
$\lambda_2\notin\mathbb Z_{\ge0}$, 
the local solution at $-\infty$ corresponding to the exponent
$-\lambda_2-\frac{\lambda_1}{c_1}-\mu$ is given by
\begin{align*}
 &\frac1{\Gamma(\mu)}
   \int_{-\infty}^x
  (-t)^{\lambda_2}
  (1-c_1t)^{\frac{\lambda_1}{c_1}}(x-t)^{\mu-1}dt\\
 &=\frac{(-x)^{\lambda_2}}{\Gamma(\mu)}\int_0^\infty 
   \left(1-\frac{s}{x}\right)^{\lambda_2}
   \bigl(1+c_1(s-x)\bigr)^{\frac{\lambda_1}{c_1}}
   s^{\mu-1}ds\qquad(s=x-t)
 \\
  &\hspace{-20pt}\xrightarrow[c_1\to+0]{\lambda_1=-1}\quad\\
 &
   \frac{(-x)^{\lambda_2}}{\Gamma(\mu)}\int_{0}^\infty
   \left(1-\frac{s}{x}\right)^{\lambda_2}e^{x-s}
   s^{\mu-1}ds\\
&=
   \frac{(-x)^{\lambda_2}e^x}{\Gamma(\mu)}\int_{0}^\infty
   s^{\mu-1}e^{-s}\left(1-\frac{s}{x}\right)^{\lambda_2}ds\\
&\sim \sum_{n=0}^\infty \frac{\Gamma(\mu+n)\Gamma(-\lambda_2+n)}{\Gamma(\mu)\Gamma(-\lambda_2)n!x^n}(-x)^{\lambda_2}
 e^x =(-x)^{\lambda_2}e^{x}{}_2F_0(-\lambda_2,\mu;\tfrac1x).
\end{align*}
Here the above last line is an asymptotic expansion of a rapidly decreasing
solution of the Kummer equation when $\mathbb R\ni-x\to+\infty$.
The Riemann scheme of the equation $P_{0;-1,\lambda_2,\mu}u=0$ can be expressed by
\begin{equation}
 \begin{Bmatrix}
  x = 0            & \infty     & (1)\\
      0            & 1-\mu      &  0 \\
     \lambda_2+\mu & -\lambda_2 &  1
 \end{Bmatrix}.
\end{equation}\index{Riemann scheme}%
In general, the expression
$
  \begin{Bmatrix}
      \infty  & (r_1) & \cdots & (r_k)\\
      \lambda & \alpha_1 & \cdots &\alpha_k
  \end{Bmatrix}
$
with $0<r_1<\cdots<r_k$ means the existence of a solution $u(x)$ satisfying 
\begin{equation}
 u(x)\sim x^{-\lambda}
 \exp\Bigl(\sum_{\nu=1}^k\alpha_\nu\frac{x^{r_\nu}}{r_\nu}\Bigr)
\text{ \ for \ }|x|\to\infty
\end{equation}
under a suitable restriction of $\Arg x$.
Here $k\in\mathbb Z_{\ge 0}$ and $\lambda$, $\alpha_\nu\in\mathbb C$.
\end{exmp}
\section{Series expansion}\label{sec:series}
In this section we review the Euler transformation 
and remark on its relation to middle convolutions.

First we note the following which will be frequently used:
\begin{gather}
 \int_0^1 t^{\alpha-1}(1-t)^{\beta-1} dt = 
 \frac{\Gamma(\alpha)\Gamma(\beta)}{\Gamma(\alpha+\beta)},\label{eq:beta}
 \allowdisplaybreaks\\
 \begin{split}
 (1-t)^{-\gamma}&=\sum_{\nu=0}^\infty
     \frac{(-\gamma)(-\gamma-1)\cdots(-\gamma-\nu+1)}{\nu!}(-t)^\nu\\
     &= \sum_{\nu=0}^\infty\frac{\Gamma(\gamma+\nu)}{\Gamma(\gamma)\nu!}t^\nu
     = \sum_{\nu=0}^\infty\frac{(\gamma)_\nu}{\nu!} t^\nu.\noindent
 \end{split}\label{eq:seriesP}
\end{gather}
The integral \eqref{eq:beta} converges if $\RE \alpha>0$ and $\RE \beta>0$
and the right hand side is meromorphically continued to $\alpha\in\mathbb C$
and $\beta\in\mathbb C$.  If the integral in \eqref{eq:beta} is interpreted
in the sense of generalized functions, \eqref{eq:beta} is valid if
$\alpha\notin\{0,-1,-2,\ldots\}$ and $\beta\notin\{0,-1,-2,\ldots\}$.

Euler transformation $I_c^\mu$ is sometimes expressed by
$\p^{-\mu}$ and as is shown in (\cite[\S5.1]{Kh}), we have
\begin{align}
\begin{split}
 I_c^\mu u(x):\!&=\frac{1}{\Gamma(\mu)}\int_c^x (x-t)^{\mu-1} u(t) dt
             \label{eq:Icdef}\\
             &=\frac{(x-c)^{\mu}}{\Gamma(\mu)}\int_0^1
 (1-s)^{\mu-1}u((x-c)s+c)ds,
\end{split}\allowdisplaybreaks\\
 I_c^\mu\circ I_c^{\mu'} &= I_c^{\mu+\mu'},\allowdisplaybreaks\label{eq:Icprod}\\
 I_c^{-n}u(x)&=\frac{d^n}{dx^n}u(x),\allowdisplaybreaks\label{eq:Icdif}\\
\begin{split}
 I_c^{\mu}\sum_{n=0}^\infty c_n(x-c)^{\lambda+n}&=
  \sum_{n=0}^\infty
  \frac{\Gamma(\lambda+n+1)}{\Gamma(\lambda+\mu+n+1)}
  c_n(x-c)^{\lambda+\mu+n}\\
  &=\frac{\Gamma(\lambda+1)}{\Gamma(\lambda+\mu+1)}
  \sum_{n=0}^\infty\frac{(\lambda+1)_nc_n}{(\lambda+\mu+1)_n}(x-c)^{\lambda+\mu+n},
\end{split}\label{eq:IcP}\allowdisplaybreaks\\
 I_\infty^{\mu}\sum_{n=0}^\infty c_n x^{\lambda-n}&=
 e^{\pi\sqrt{-1}\mu}\sum_{n=0}^\infty
     \frac{\Gamma(-\lambda-\mu+n)}{\Gamma(-\lambda+n)}
  c_n x^{\lambda+\mu-n}.\label{eq:IinfP}
\end{align}
Moreover the following equalities which follow from \eqref{eq:Gmid} 
are also useful.
\begin{equation}
\begin{split}
 &I_0^{\mu}\sum_{n=0}^\infty c_nx^{\lambda+n}(1-x)^\beta\\
 &\ =\frac{\Gamma(\lambda+1)}{\Gamma(\lambda+\mu+1)}
   \sum_{m,\,n=0}^\infty\frac{(\lambda+1)_{m+n}(-\beta)_mc_n}
   {(\lambda+\mu+1)_{m+n}m!}x^{\lambda+\mu+m+n}\\
 &\ =\frac{\Gamma(\lambda+1)}{\Gamma(\lambda+\mu+1)}(1-x)^{-\beta}
    \sum_{m,\,n=0}^\infty\frac{(\lambda+1)_n(\mu)_m(-\beta)_mc_n}
   {(\lambda+\mu+1)_{m+n}m!}x^{\lambda+\mu+n}\Bigl(\frac x{x-1}\Bigr)^m.
\end{split}\label{eq:I0H}
\end{equation}

If $\lambda\notin\mathbb Z_{<0}$ (resp.\ $\lambda+\mu\notin\mathbb Z_{\ge0})$  
and moreover the power series $\sum_{n=0}^\infty c_nt^n$ has a positive
radius of convergence, the equalities \eqref{eq:IcP} (resp,\  
\eqref{eq:IinfP}) is valid since $I_c^\mu$ (resp.\ $I_\infty^\mu$) can be 
defined through analytic continuations with respect to the parameters 
$\lambda$ and $\mu$.
Note that $I_c^\mu$ is an invertible map of $\mathcal O_c(x-c)^\lambda$ 
onto $\mathcal O_c(x-c)^{\lambda+\mu}$
if $\lambda\notin\{-1,-2,-3,\ldots\}$ and 
$\lambda+\mu \notin\{-1,-2,-3,\ldots\}$.

\begin{prop}\label{prop:RAdIc}
Let\/ $\lambda$ and\/ $\mu$ be complex numbers satisfying\/ 
$\lambda\notin\mathbb Z_{<0}$. 
Differentiating the equality \eqref{eq:IcP} with respect to\/ $\lambda$,
we have the linear map
\begin{equation}\label{eq:Iclog}
  I_c^\mu: \mathcal O_c(\lambda,m) \to \mathcal O_c(\lambda+\mu,m)
\end{equation}
under the notation \eqref{def:Ocm}, 
which is also defined by \eqref{eq:Icdef} if\/ $\RE\lambda > -1$ and\/
$\RE\mu > 0$.
Here\/ $m$ is a non-negative integer. 
Then we have
\begin{equation}\label{eq:Iclogtop}
 I_c^\mu\bigl(\sum_{j=0}^m \phi_j\log^j(x-c)\bigr)
 - I_c^\mu(\phi_m)\log^m(x-c)\in\mathcal O(\lambda+\mu,m-1)
\end{equation}
for\/ $\phi_j\in\mathcal O_c$ and\/ $I_c^\mu$ satisfies \eqref{eq:Icp}.
The map \eqref{eq:Iclog} is bijective if\/ $\lambda+\mu\notin\mathbb Z_{<0}$.
In particular for\/ $k\in\mathbb Z_{\ge0}$ we have\/
$I_c^\mu\p^k=\p^kI_c^\mu=I_c^{\mu-k}$ on\/ 
$\mathcal O_c(\lambda,m)$ if\/ $\lambda-k\notin\{-1,-2,-3,\dots\}$.

Suppose that\/ $P\in W[x]$ and\/ $\phi\in\mathcal O_c(\lambda,m)$ satisfy\/
$P\phi=0$, $P\ne0$ and\/ $\phi\ne0$.
Let\/ $k$ and\/ $N$ be non-negative integers such that 
\begin{equation}
 \p^k P = \sum_{i=0}^N\sum_{j\ge 0}a_{i,j}\p^i\bigl((x-c)\p\bigr)^j
\end{equation}
with suitable\/ $a_{j,j}\in\mathbb C$ and put\/ 
$Q=\sum_{i=0}^N\sum_{j\ge0}c_{i,j}\p^i\bigl((x-c)\p - \mu\bigr)^j$.
Then if\/ $\lambda\notin\{N-1,N-2,\ldots,0,-1,\ldots\}$, we have
\begin{equation}
 I_c^\mu\p^kPu=QI_c^\mu(u) \text{ \ for \ }u\in\mathcal O_c(\lambda,m)
\end{equation}
and in particular\/ $QI_c^\mu(\phi)=0$.

Fix\/ $\ell\in\mathbb Z$.
For\/ $u(x)=\sum_{i=\ell}^\infty\sum_{j=0}^mc_{i,j}(x-c)^i\log^j(x-c)
\in\mathcal O_c(\ell,m)$ we
put\/ $(\Gamma_Nu)(x)=\sum_{\nu=\max\{\ell,N-1\}}^\infty
\sum_{j=0}^mc_{i,j}(x-c)^i\log^j(x-c)$.
Then
\[
  \Bigl(\prod_{\ell-N\le\nu\le N-1}\bigl((x-c)\p-\nu\bigr)^{m+1}\Bigr)
  \p^kP\bigl(u(x) - (\Gamma_Nu)(x)\bigr)=0
\]
ane therefore
\begin{equation}
 \begin{split}
  &\Bigl(\prod_{\ell-N\le\nu\le N-1}\bigl((x-c)\p-\mu-\nu\bigr)^{m+1}\Bigr)
  QI_c^\mu(\Gamma_Nu)\\
  &\quad=
  I_c^\mu\Bigl(\prod_{\ell-N\le\nu\le N-1}\bigl((x-c)\p-\nu\bigr)^{m+1}\Bigr)
  \p^kPu.
 \end{split}
\end{equation}
In particular, $\prod_{\ell-N\le\nu\le N-1}\bigl((x-c)\p-\mu-\nu\bigr)^{m+1}
\cdot QI_c^\mu\bigl(\Gamma_N(u)\bigr)=0$ if\/ $Pu=0$.

Suppose moreover\/ $\lambda\notin\mathbb Z$ and $\lambda+\mu\notin\mathbb Z$
and\/ $Q=ST$ with\/ $S$, $T\in W[x]$ such that\/ $x=c$ is not a singular point 
of the operator\/ $S$. 
Then\/ $TI_c^\mu(\phi)=0$.  In particular, 
\begin{equation}
\bigl(\RAd(\p^{-\mu})P\bigr)I_c^\mu(\phi)=0.
\end{equation}
Hence if the differential equation\/ $\bigl(\RAd(\p^{-\mu})P\bigr)v=0$
is irreducible, we have
\begin{equation}
  W(x)\bigl(\RAd(\p^{-\mu})P\bigr)=\{T\in W(x)\,;\,
  TI_c^\mu(\phi)=0\}.
\end{equation}
The statements above are also valid even if we replace\/ 
$x-c$,  $I_c^\mu$ by \/$\frac1x$, $I_\infty^\mu$, respectively.
\end{prop}
\begin{proof}
It is clear that \eqref{eq:Iclog} is well-defined and \eqref{eq:Iclogtop}
is valid.  
Then \eqref{eq:Iclog} is bijective because of \eqref{eq:IcP} 
and \eqref{eq:Iclogtop}.
Since \eqref{eq:Icp} is valid when $m=0$, it is also valid when
$m=1,2,\ldots$ by the definition of \eqref{eq:Iclog}.

The equalities \eqref{eq:IcP} and \eqref{eq:IinfP} assure that
$QI_c^\mu(\phi)=0$.
Note that $TI_c^\mu(\phi)\in\mathcal O(\lambda+\mu-N,m)$
with a suitable positive integer $N$.
Since $\lambda+\mu-N\notin\mathbb Z$ and any solution of the equation
$Sv=0$ is holomorphic at $x=c$, the equality 
$S\bigl(TI_c^\mu(\phi)\bigr)=0$ implies $TI_c^\mu(\phi)=0$.

The remaining claims in the theorem are similarly clear.
\end{proof}
\begin{rem} 
{\rm i)} Let $\gamma: [0,1]\to \mathbb C$ be a path such that $\gamma(0)=c$ and
$\gamma(1)=x$.  Suppose $u(x)$ is holomorphic along the path $\gamma(t)$
for $0<t\le 1$ and $u(\gamma(t))=\phi(\gamma(t))$ for $0<t\ll 1$ with
a suitable function $\phi\in\mathcal O_c(\lambda,m)$.
Then $I_c^\mu(u)$ is defined by the integration along the path $\gamma$.
In fact, if the path $\gamma(t)$ with $t\in[0,1]$ splits into the
three paths corresponding to the decomposition $[0,1]=[0,\epsilon]\cup
[\epsilon,1-\epsilon]\cup[1-\epsilon,1]$ with $0<\epsilon\ll1$. 
Let $c_1=c,\dots,c_p$ be points in $\mathbb C^n$ and suppose moreover
$u(x)$ is extended to a multi-valued holomorphic function on 
$\mathbb C\setminus\{c_1,\dots,c_p\}$.
Then $I_c^x(u)$ also defines a multi-valued holomorphic function on 
$\mathbb C\setminus\{c_1,\dots,c_p\}$.

{\rm ii)}
Proposition~\ref{prop:RAdIc} is also valid if we replace
$\mathcal O_c(\lambda,m)$ by the space of functions given in
Remark~\ref{rem:defIc} ii).   
In fact the above proof also works in this case.
\end{rem}
\section{Contiguity relation}\label{sec:contig}
The following proposition is clear from 
Proposition~\ref{prop:RAdIc}. 
\begin{prop}\label{prop:config0}
Let $\phi(x)$ be a non-zero solution of an ordinary differential equation
$Pu=0$ with an operator $P\in W[x]$.
Let $P_j$ and $S_j\in W[x]$ for $j=1,\dots,N$ so that
$\sum_{j=1}^N P_jS_j\in W[x]P$.
Then for a suitable $\ell\in\mathbb Z$ we have
\begin{equation}
  \sum Q_j\bigl(I_c^\mu(\phi_j)\bigr)=0
\end{equation}
by putting
\begin{equation}
  \begin{split}
   \phi_j&=S_j\phi,\\
   Q_j &= \p^{\ell-\mu}\circ P_j\circ\p^{\mu}\in W[x],
  \end{split}
  \quad(j=1,\dots,N),
\end{equation}
if $\phi(x)\in\mathcal O(\lambda,m)$ with a non-negative integer $m$ and 
a complex number $\lambda$ satisfying
$\lambda\notin\mathbb Z$ and $\lambda+\mu\notin\mathbb Z$
or $\phi(x)$ is a function  given in Remark~\ref{rem:defIc} ii).   
If $P_j=\sum_{k\ge 0,\ \ell\ge 0}c_{j,k,\ell}\p^k\vartheta^\ell$ with
$c_{j,k,\ell}\in\mathbb C$, then we can assume $\ell\le 0$ in the above.
Moreover we have
\begin{equation}
  \p\bigl(I_c^{\mu+1}(\phi_1)\bigr) = I_c^\mu(\phi_1).
\end{equation}
\end{prop}
\begin{proof}
Fix an integer $k$ such that 
$\p^k P_j=\tilde P_j(\p,\vartheta)=\sum_{i_1,i_2}
c_{i_1,i_2}\p^{i_1}\vartheta^{i_2}$ with $c_{i_1,i_2}\in\mathbb C$.
Since $0=\sum_{j=1}^N \p^k P_jS_j\phi$, 
Proposition~\ref{prop:RAdIc} proves
$0=\sum_{j=1}^N I_c^\mu(\tilde P_j(\p,\vartheta)S_j\phi)=
\sum_{j=1}^N \tilde P_j(\p,\vartheta-\mu)I_c^\mu(S_j\phi)$,
which implies the first claim of the proposition.

The last claim is clear from \eqref{eq:Icprod} and \eqref{eq:Icdif}.
\end{proof}
\begin{cor}
Let $P(\xi)$ and $K(\xi)$ be non-zero elements of $W[x;\xi]$.
If we substitute $\xi$ and $\mu$ by generic complex numbers, we assume
that there exists a solution $\phi_\xi(x)$ satisfying the assumption in 
the preceding proposition and that $I_c^\mu(\phi_\xi)$
and $I_c^\mu(K(\xi)\phi_\xi)$ satisfy irreducible differential equations
$T_1(\xi,\mu)v_1=0$ and $T_2(\xi,\mu)v_2=0$ with
$T_1(\xi,\mu)$ and $T_2(\xi,\mu)\in W(x;\xi,\mu)$, 
respectively.
Then the differential equation $T_1(\xi,\mu)v_1=0$ is isomorphic to 
$T_2(\xi,\mu)v_2=0$ as $W(x;\xi,\mu)$-modules.
\end{cor}
\begin{proof}
Since $K(\xi)\cdot 1 - 1\cdot K(\xi)=0$, 
we have $Q(\xi,\mu)I_c^\mu(\phi_\xi)=\p^\ell I_c^\mu(K(\xi)\phi_\xi)$
with $Q(\xi,\mu)=\p^{\ell-\mu}\circ K(\xi)\circ\p^{\mu}$.
Since $\p^{\ell}I_c^\mu(\phi_\xi)\ne0$ 
and the equations $T_j(\xi,\mu)v_j=0$ are irreducible for $j=1$ and $2$,
there exist $R_1(\xi,\mu)$ and $R_2(\xi,\mu)\in W(x;\xi,\mu)$ such
that  $I_c^\mu(\phi_\xi)=R_1(\xi,\mu)Q(\xi,\mu)I_c^\mu(\phi_\xi)
=R_1(\xi,\mu)\p^\ell I_c^\mu(K(\xi)\phi_\xi)$ 
and $I_c^\mu(K(\xi)\phi_\xi)=R_2(\xi,\mu)\p^\ell I_c^\mu(K(\xi)\phi_\xi)
=R_2(\xi,\mu)Q(\xi,\mu)I_c^\mu(\phi_\xi)$.
Hence we have the corollary.
\end{proof}
Using the proposition, we get the contiguity relations 
with respect to the parameters corresponding to powers 
of linear functions defining additions and the middle convolutions. 

For example, in the case of Gauss hypergeometric functions, we have
\begin{align*}
u_{\lambda_1,\lambda_2,\mu}(x):\!&=I_0^\mu(x^{\lambda_1}(1-x)^{\lambda_2}),\\
u_{\lambda_1,\lambda_2,\mu-1}(x)&=\p u_{\lambda_1,\lambda_2,\mu}(x),\\
\p u_{\lambda_1+1,\lambda_2,\mu}(x)&=(x\p+1-\mu)u_{\lambda_1,\lambda_2,\mu}(x),\\
\p u_{\lambda_1,\lambda_2+1,\mu}(x)&=((1-x)\p+\mu-1)u_{\lambda_1,\lambda_2,\mu}(x).
\end{align*}
Here Proposition~\ref{prop:config0} with 
$\phi=x^{\lambda_1}(1-x)^{\lambda_2}$, 
$(P_1,S_1,P_2,S_2)=(1,x,-x,1)$ and $\ell=1$ gives the above third identity.

Since $P_{\lambda_1,\lambda_2,\mu}u_{\lambda_1,\lambda_2,\mu}(x)=0$
with
\begin{align*}
P_{\lambda_1,\lambda_2,\mu}&
 =\bigl(x(1-x)\p+(1-\lambda_1-\mu-(2-\lambda_1-\lambda_2-2\mu)x\bigr)\p\\
&\quad{}
 -(\mu-1)(\lambda_1+\lambda_2+\mu)
\intertext{as is given in Example~\ref{ex:midconv}, the inverse of the relation
$u_{\lambda_1,\lambda_2,\mu-1}(x)=\p u_{\lambda_1,\lambda_2,\mu}(x)$ is}
u_{\lambda_1,\lambda_2,\mu}(x)&=
-\frac{x(1-x)\p+(1-\lambda_1-\mu-(2-\lambda_1-\lambda_2-2\mu)x\bigr)}
{(\mu-1)(\lambda_1+\lambda_2+\mu)}
u_{\lambda_1,\lambda_2,\mu-1}(x).
\end{align*}
The equalities 
$u_{\lambda_1,\lambda_2,\mu-1}(x)=\p u_{\lambda_1,\lambda_2,\mu}(x)$
and \eqref{eq:Gmid} mean
\begin{align*}
&\frac{\Gamma(\lambda_1+1)x^{\lambda_1+\mu-1}}{\Gamma(\lambda_1+\mu)}
F(-\lambda_2,\lambda_1+1,\lambda_1+\mu;x)\\
&=\frac{\Gamma(\lambda_1+1)x^{\lambda_1+\mu-1}}{\Gamma(\lambda_1+\mu)}
F(-\lambda_2,\lambda_1+1,\lambda_1+\mu+1;x)\\
&\quad{}+\frac{\Gamma(\lambda_1+1)x^{\lambda_1+\mu}}{\Gamma(\lambda_1+\mu+1)}
 \frac{d}{dx}F(-\lambda_2,\lambda_1+1,\lambda_1+\mu+1;x)
\end{align*}
and therefore 
$u_{\lambda_1,\lambda_2,\mu-1}(x)=\p u_{\lambda_1,\lambda_2,\mu}(x)$
is equivalent to
\[
  (\gamma-1)F(\alpha,\beta,\gamma-1;x)=(\vartheta+\gamma-1)F(\alpha,\beta,\gamma;x).
\]
The contiguity relations are very important for the study of differential
equations.
For example the author's original proof of the connection formula
\eqref{eq:Icon} announced in \cite{O3} is based on the relations
(cf.~\S\ref{sec:C2}).

Some results related to contiguity relations will be given in \S\ref{sec:shift}
but we will not go further in this subject and it will 
be discussed in another paper.
\section{Fuchsian differential equation and generalized Riemann scheme}
\label{sec:index}
\subsection{Generalized characteristic exponents}\label{sec:Gexp}
We examine the Fuchsian differential equations 
\begin{equation}
  P = a_n(x)\tfrac{d^n}{dx^n}
      +a_{n-1}(x)\tfrac{d^{n-1}}{dx^{n-1}}+\cdots+a_0(x)
  \label{eq:P}
\end{equation}
with given local monodromies at regular singular points.
For this purpose
we first study the condition so that monodromy generators of the solutions of 
a Fuchsian differential equation is semisimple even when its exponents
are not free of multiplicity.
\begin{lem}\label{lem:block}
Suppose that the operator \eqref{eq:P} defined in a neighborhood of the origin
has a regular singularity at the origin.
We may assume $a_\nu(x)$ are holomorphic at $0$ and
$a_n(0) = a_n'(0)=\cdots=a_n^{(n-1)}(0)=0$ and $a_n^{(n)}(0)\ne0$.
Then the following conditions are equivalent
for a positive integer $k$.
\begin{align}
 P &= x^kR&&\text{with a suitable holomorphic differential operator $R$}
    \label{C:divblock}\\[-2pt]
  & && \text{at the origin,}\notag
\allowdisplaybreaks\\
 Px^\nu &= o(x^{k-1})&&\text{for \  \ }\nu=0,\ldots,k-1,
 \label{C:applyblock}
\allowdisplaybreaks\\
 Pu &=0&&\text{has a solution $x^\nu+o(x^{k-1})$ for }
   \nu=0,\dots,k-1,\label{C:solblock}
\allowdisplaybreaks\\
 P  &= \sum_{j\ge0}x^jp_j(\vartheta)&&\text{with polynomials }p_j
    \text{\ satisfying\ }p_j(\nu)=0\label{C:formblock}\\[-10pt]
    & &&\text{for \ }0\le\nu< k-j\text{ \ and \ }
        j=0,\dots,k-1.\notag
\end{align}
\end{lem}
\begin{proof}
\eqref{C:divblock} $\Rightarrow$ \eqref{C:applyblock} $\Leftrightarrow$ 
\eqref{C:solblock} is clear.

Assume \eqref{C:applyblock}.
Then $Px^\nu=o(x^{k-1})$ for $\nu=0,\dots,k-1$ implies 
$a_j(x)=x^kb_j(x)$ for $j=0,\dots,k-1$.  
Since $P$ has a regular singularity at the origin, $a_j(x)=x^jc_j(x)$
for $j=0,\dots,n$.  Hence we have \eqref{C:divblock}.

Since $Px^\nu = \sum_{j=0}^\infty x^{\nu+j}p_j(\nu)$, the equivalence
\eqref{C:applyblock} $\Leftrightarrow$ \eqref{C:formblock} is clear.
\end{proof}
\begin{defn}
Suppose $P$ in \eqref{eq:P} has a regular singularity at $x=0$.
Under the notation \eqref{eq:mult}
we define that $P$ has a (\textsl{generalized}) \textsl{characteristic exponent}
\index{characteristic exponent!generalized}
$[\lambda]_{(k)}$ at $x=0$ if 
$x^{n-k}\Ad(x^{-\lambda})(a_n(x)^{-1}P)\in W[x]$.
\end{defn}

Note that Lemma~\ref{lem:block} shows that $P$ has a characteristic
exponent $[\lambda]_{(k)}$ at $x=0$ if and only if
\begin{equation}
 x^na_n(x)^{-1}P=\sum_{j\ge 0}x^jq_j(\vartheta)\prod_{0\le i< k-j}
(\vartheta - \lambda- i) 
\end{equation}
with polynomials $q_j(t)$.
By a coordinate transformation we can define generalized characteristic exponents
for any regular singular point as follows.
\begin{defn}[generalized characteristic exponents]
\label{def:RScheme}
Suppose  $P$ in \eqref{eq:P} has regular singularity at $x=c$.
Let $n=m_1+\cdots+m_q$ be a partition of the positive integer $n$
and let $\lambda_1,\dots,\lambda_q$ be complex numbers.
We define that $P$  has  
the (set of \textsl{generalized}) \textsl{characteristic exponents}
$\{[\lambda_1]_{(m_1)},\dots,[\lambda_q]_{(m_q)}\}$ and the
\textsl{spectral type} $\{m_1,\dots,m_q\}$ at 
$x=c\in\mathbb C\cup\{\infty\}$ if there exist polynomials $q_\ell(s)$
such that
\begin{equation}\label{eq:GEXP}
  (x-c)^na_n(x)^{-1}P=\sum_{\ell\ge 0}(x-c)^\ell
  q_\ell\bigl((x-c)\p\bigr)
  \prod_{\nu=1}^q\,\prod_{0\le i< m_\nu-\ell}
  \bigl((x-c)\p-\lambda_\nu-i\bigr)
\end{equation}
in the case when $c\ne\infty$ and
\begin{equation}\label{eq:gcexp}
  x^{-n}a_n(x)^{-1}P=\sum_{\ell\ge 0}x^{-\ell}
  q_\ell\bigl(\vartheta)
  \prod_{\nu=1}^q\,\prod_{0\le i< m_\nu-\ell}
  \bigl(\vartheta+\lambda_\nu+i\bigr)
\end{equation}
in the case when $c=\infty$.
Here if $m_j=1$,
$[\lambda_j]_{(m_j)}$ may be simply written as $\lambda_j$.
\end{defn}
\begin{rem}\label{rem:GCexp}
{\rm i) }
In Definition~\ref{def:RScheme} we may replace the left hand side of
\eqref{eq:GEXP} by $\phi(x)a_n(x)^{-1}P$ where $\phi$ is analytic
function in a neighborhood of $x=c$ such that
$\phi(c)=\cdots=\phi^{(n-1)}(c)=0$ and $\phi^{(n)}(c)\ne 0$.
In particular when $a_n(c)=\cdots=a_n^{(n)}(c)=0$ and $a_n(c)\ne 0$,
$P$ is said to be \textsl{normalized} at the singular point $x=c$ 
and the left hand side of \eqref{eq:GEXP} can be replaced by $P$.
\index{Fuchsian differential equation/operator!normalized}

In particular when $c=0$ and $P$ is normalized at the regular singular point
$x=0$, the condition \eqref{eq:GEXP} is equivalent to
\begin{equation}
  \prod_{\nu=1}^k\prod_{0\le i<m_\nu-\ell}(s-\lambda_\nu-i)\bigm|p_j(s)
  \qquad(\forall\ell=0,1,\dots,\max\{m_1,\dots,m_k\}-1)
\end{equation}
under the expression
$
  P = \sum_{j=0}^\infty x^jp_j(\vartheta)
$.

{\rm ii) }
In Definition~\ref{def:RScheme} the condition that 
the operator $P$ has a set of 
generalized characteristic exponents $\{\lambda_1,\dots,\lambda_n\}$ 
is equivalent to the condition that it is the set of the usual 
characteristic exponents.

{\rm iii) }
Any one of 
$\{\lambda,\lambda+1,\lambda+2\}$,
$\{[\lambda]_{(2)},\lambda+2\}$ and 
$\{\lambda,[\lambda+1]_{(2)}\}$
is the set of characteristic exponents of
\[
  P= (\vartheta-\lambda)(\vartheta-\lambda-1)(\vartheta - \lambda - 2+x)
    +x^2(\vartheta-\lambda+1)
\]
at $x=0$ but $\{[\lambda]_{(3)}\}$ is not.

{\rm iv) }
Suppose $P$ has a holomorphic parameter $t\in B_1(0)$ (cf.~\eqref{eq:Brc})
and $P$ has regular singularity at $x=c$.
Suppose the set of the corresponding characteristic exponents is
$\{[\lambda_1(t)]_{(m_1)},\ldots,[\lambda_q(t)]_{(m_q)}\}$
for $t\in B_1(0)\setminus\{0\}$ with 
$\lambda_\nu(t)\in\mathcal O\bigl(B_1(0)\bigr)$.
Then this is also valid in the case $t=0$, which
clearly follows from the definition.

When
\[
  P=\sum_{\ell\ge 0}x^{-\ell}
  q_\ell\bigl((x-c)\p\bigr)
  \prod_{\nu=1}^q\,\prod_{0\le i< m_\nu-\ell}
  \bigl((x-c)\p-\lambda_\nu-i\bigr),
\]
we put
\[
  P_t=\sum_{\ell\ge 0}x^{-\ell}
  q_\ell\bigl((x-c)\p\bigr)
  \prod_{\nu=1}^q\,\prod_{0\le i< m_\nu-\ell}
  \bigl((x-c)\p-\lambda_\nu-\nu t-i\bigr).
\]
Here $\lambda_\nu\in\mathbb C$, $q_0\ne 0$ and $\ord P = m_1+\dots+m_q$.
Then the set of the characteristic 
exponents of $P_t$ is $\{[\tilde\lambda_1(t)]_{(m_1)},\ldots,
[\tilde\lambda_q(t)]_{(m_q)}\}$ with $\tilde\lambda_j(t)=\lambda_j+jt$.
Since $\tilde \lambda_i(t)-\tilde\lambda_j(t)\notin\mathbb Z$
for $0<|t|\ll 1$, we can reduce certain claims to the case when 
the values of characteristic exponents are generic.
Note that we can construct local independent solutions which holomorphically 
depend on $t$ (cf.~\cite{Or}).
\end{rem}
\begin{lem}\label{lem:GRS}
{\rm i) }
Let $\lambda$ be a complex number and let 
$p(t)$ be a polynomial such that $p(\lambda)\ne 0$.
Then for non-negative integers $k$ and $m$ we have the exact sequence
\begin{equation*}
  0\longrightarrow 
   \mathcal O_0(\lambda,k-1)
  \longrightarrow
   \mathcal O_0(\lambda,m+k-1)
  \xrightarrow{p(\vartheta)(\vartheta-\lambda)^k}{}
   \mathcal O_0(\lambda,m-1)
  \longrightarrow 0
\end{equation*}
under the notation \eqref{def:Ocm}.

{\rm ii) }  Let $m_1,\dots,m_q$ be non-negative integers.
Let $P$ be a differential operator of order $n$ whose coefficients
are in $\mathcal O_0$ such that
\begin{equation}\label{eq:GRS0}
  P  = \sum_{\ell=0}^\infty x^\ell r_\ell(\vartheta)
       \prod_{\nu=1}^q\,\prod_{0\le k< m_\nu-\ell}\bigl(\vartheta-k\bigr)
\end{equation}
with polynomials $r_\ell$.  
Put $m_{max}=\max\{m_1,\dots,m_q\}$ and
suppose $r_0(\nu)\ne 0$ for $\nu=0,\dots,m_{max}-1$.

Let\/ ${\mathbf m}^\vee=(m^\vee_1,\dots,m^\vee_{m_{max}})$ 
be the dual partition of 
$\mathbf m:=(m_1,\dots,m_q)$, 
namely,\index{dual partition}
\begin{equation}\label{eq:dualP}
   m^\vee_\nu=\#\{j\,;\,m_j\ge \nu\}.
\end{equation}
Then for $i=0,\ldots,m_{max}-1$ and $j=0,\ldots,m^\vee_{i+1}-1$
we have the functions
\begin{equation}
  u_{i,j}(x) = x^i\log^j x+ \sum_{\mu=i+1}^{m_{max}-1}
               \sum_{\nu=0}^{j}
               c_{i,j}^{\mu,\nu}x^\mu \log^\nu x
\end{equation}
such that $c_{i,j}^{\mu,\nu}\in\mathbb C$ 
and $Pu_{i,j}\in\mathcal O_0(m_{max},j)$.  

{\rm iii) } 
Let\/ $m'_1,\dots,m'_{q'}$ be non-negative integers and
let $P'$ be a differential operator of order $n'$ whose coefficients
are in $\mathcal O_0$ such that
\begin{equation}
  P'= \sum_{\ell=0}^\infty x^\ell r'_\ell(\vartheta)
       \prod_{\nu=1}^q\,\prod_{0\le k< m'_\nu-\ell}
        \bigl(\vartheta-m_\nu - k\bigr)
\end{equation}
with polynomials $q'_\ell$.
Then for a differential operator $P$ of the form \eqref{eq:GRS0} we have
\begin{equation}
  P'P= \sum_{\ell=0}^\infty x^\ell\Bigl(\sum_{\nu=0}^\ell
         r'_{\ell-\nu}(\vartheta+\nu) r_\nu(\vartheta)\Bigr)
       \prod_{\nu=1}^q\,\prod_{0\le k< m_\nu+m'_\nu-\ell}
        \bigl(\vartheta- k\bigr).
\end{equation}
\end{lem}
\begin{proof} {\rm i) }
The claim is easy if $(p,k) = (1,1)$ or $(\vartheta-\mu,0)$ with 
$\mu\ne\lambda$.
Then the general case follows from induction on $\deg p(t)+k$.

{\rm ii) }
Put $P=\sum_{\ell\ge 0}x^\ell p_\ell(\vartheta)$ and $m^\vee_\nu=0$ 
if $\nu>m_{max}$.
Then for a non-negative integer $\nu$,
the multiplicity of the root $\nu$ of the equation $p_\ell(t)=0$
is equal or larger than $m^\vee_{\nu+\ell+1}$ for $\ell=1,2,\dots$.
If $0\le \nu\le m_{max}-1$, the multiplicity of the root $\nu$ of the 
equation $p_0(t)=0$ equals $m^\vee_{\nu+1}$.

For non-negative integers $i$ and $j$, we have
\[
 x^\ell p_\ell(\vartheta)x^i\log^j x=x^{i+\ell}\sum_{0\le \nu\le j - m^\vee_{i+\ell+1}}
 c_{i,j,\ell,\nu}\log^\nu x
\]
with suitable $c_{i,j,\ell,\nu}\in\mathbb C$.
In particular, $p_0(\vartheta)x^i\log^j x=0$ if $j<m^\vee_i$.
If $\ell>0$ and $i+\ell<m_{\max}$, there exist functions
\[
  v_{i,j,\ell} =  x^{i+\ell} \sum_{\nu=0}^j a_{i,j,\ell,\nu}
   \log^\nu x
\]
with suitable $a_{i,j,\ell,\nu}\in\mathbb C$ 
such that $p_0(\vartheta)v_{i,j,\ell} = x^\ell p_\ell(\vartheta)x^i\log^j x$
and we define a $\mathbb C$-linear map $Q$ by
\[
  Qx^i\log^j x= -\sum_{\ell=1}^{m_{max}-i-1}v_{i,j,\ell}
  = -\sum_{\ell=1}^{m_{max}-i-1}\sum_{\nu=0}^j
   a_{i,j,\ell,\nu}x^{i+\ell}\log^\nu x,
\]
which implies $p_0(\vartheta)Qx^i\log^jx
=-\sum_{\ell=1}^{m_{max}-i-1}x^\ell p_\ell(\vartheta)x^i\log^j$
and $Q^{m_{max}}=0$.
Putting $Tu:=\sum_{\nu=0}^{m_{max}-1}Q^\nu u$ for 
$u\in\sum_{i=0}^{m_{max}-1}\sum_{j=0}^{q-1}\mathbb Cx^i\log^jx$, 
we have 
\begin{align*}
 P Tu &\equiv
   p_0(\vartheta)Tu+\sum_{\ell=1}^{m_{max}-1}x^\ell p_\ell(\vartheta)Tu
    &&\mod \mathcal O_0(m_{max},j)\\
   &\equiv p_0(\vartheta)(1-Q)Tu &&\mod \mathcal O_0(m_{max},j)\\
   &\equiv p_0(\vartheta)(1-Q)(1+Q+\cdots+Q^{m_{max}-1})u 
    &&\mod \mathcal O_0(m_{max},j)\\
   &= p_0(\vartheta)u.
\end{align*}
Hence if $j<m^\vee_i$, $PTx^i\log^j x\equiv0\mod\mathcal O_0(m_{max},j)$ and  
$u_{i,j}(x):=Tx^i\log^j x$ are required functions.

{\rm iii) } Since
\begin{align*}
&x^{\ell'} r'_{\ell'}(\vartheta)
       \prod_{\nu=1}^q\,\prod_{0\le k'< m'_{\nu}-\ell'}
 (\vartheta - m_\nu - k')\cdot
x^{\ell} r_{\ell}(\vartheta)
       \prod_{\nu=1}^q\,\prod_{0\le k< m_\nu-\ell}
 (\vartheta - k)\\
&=x^{\ell+\ell'}r'_{\ell'}(\vartheta+\ell)r_\ell(\vartheta)
    \prod_{\nu=1}^q
    \prod_{0\le k'< m'_{\nu}-\ell'}
   (\vartheta - m_\nu - k'+\ell)
    \prod_{0\le k< m_\nu-\ell}
   (\vartheta - k)\\
&=x^{\ell+\ell'}r'_{\ell'}(\vartheta+\ell)r_\ell(\vartheta)
    \prod_{\nu=1}^q
    \prod_{0\le k< m_\nu+m_{\nu'}-\ell-\ell'}
   (\vartheta - k),
\end{align*}
we have the claim. 
\end{proof}
\begin{defn}[generalized Riemann scheme]\label{def:GRS}
\index{Riemann scheme!generalized}
Let $P\in W[x]$. Then we call $P$ is \textsl{Fuchsian}
in this paper when $P$ has at most regular singularities
in $\mathbb C\cup\{\infty\}$.
Suppose $P$ is Fuchsian with regular singularities at 
$x=c_0=\infty$, $c_1$,\dots, $c_p$
and the functions $\frac{a_j(x)}{a_n(x)}$ 
are holomorphic on $\mathbb C\setminus\{c_1,\dots,c_p\}$
for $j=0,\dots,n$.
Moreover suppose $P$ has the set of characteristic exponents 
$\{[\lambda_{j,1}]_{(m_{j,1})},\dots,[\lambda_{j,n_j}]_{(m_{j,n_j})}\}$
at $x=c_j$.  Then we define the Riemann scheme of $P$ or the equation $Pu=0$
by
\begin{equation}\label{eq:GRS}
  \begin{Bmatrix}
   x = c_0=\infty & c_1 & \cdots & c_p\\
  [\lambda_{0,1}]_{(m_{0,1})} & [\lambda_{1,1}]_{(m_{1,1})}&\cdots
    &[\lambda_{p,1}]_{(m_{p,1})}\\
  \vdots & \vdots & \vdots & \vdots\\
    [\lambda_{0,n_0}]_{(m_{0,n_0})} & [\lambda_{1,n_1}]_{(m_{1,n_1})}&\cdots
    &[\lambda_{p,n_p}]_{(m_{p,n_p})}
  \end{Bmatrix}.
\end{equation}
\end{defn}
\begin{rem}
The Riemann scheme \eqref{eq:GRS} always satisfies
the Fuchs relation\index{Fuchs relation}
(cf.~\eqref{eq:FC0}):
\begin{equation}\label{eq:Fuchs}
 \sum_{j=0}^p\sum_{\nu=1}^{n_j}\sum_{i=0}^{m_{j,\nu}-1}
  \bigl(\lambda_{j,\nu}+i\bigr)
 =\frac{(p-1)n(n-1)}{2}.
\end{equation}
\end{rem}
\begin{defn}[spectral type]\index{spectral type|see{tuple of partitions}}
In Definition~\ref{def:GRS} we put
\[
 \mathbf m=(m_{0,1},\dots,m_{0,n_0};
 m_{1,1},\dots;m_{p,1},\dots,m_{p,n_p}),
\]
which will be also written as $m_{0,1}m_{0,2}\cdots 
m_{0,n_0},m_{1,1}\cdots,m_{p,1}\cdots m_{p,n_p}$ for simplicity.
Then $\mathbf m$ is a
$(p+1)$-tuple of partitions of $n$ and
we define that $\mathbf m$ is the \textsl{spectral type} of $P$.
\index{spectral type}

If the set of (usual) characteristic exponents 
\begin{equation}\label{eq:Lambdaj}
 \Lambda_j:=\{\lambda_{j,\nu}+i\,;\,0\le i\le m_{j,\nu}-1\text{ and } 
\nu=1,\dots,n_\nu\}
\end{equation}
of the Fuchsian differential operator
$P$ at every regular singular point $x=c_j$ are $n$ different complex 
numbers, $P$ is said to have \textsl{distinct exponents}.
\index{characteristic exponent!distinct}
\end{defn}
\begin{rem}
We remark that the Fuchsian differential equation $\mathcal M: Pu=0$
is irreducible (cf.~Definition~\ref{def:irred})
if and only if the monodromy of the equation is irreducible.

If $P=QR$ with $Q$ and $R\in W(x;\xi)$, 
the solution space of the equation $Qv=0$ is a subspace of that of $\mathcal M$
and closed under the monodromy and therefore the monodromy is reducible.
Suppose the space spanned by certain linearly independent solutions 
$u_1,\dots,u_m$ is invariant under the monodromy.
We have a non-trivial simultaneous solution of the linear relations
$b_mu^{(m)}_j+\cdots+b_1u^{(1)}_j+b_0u_j=0$ for $j=1,\dots,m$.
Then $\frac{b_j}{b_m}$ are single-valued holomorphic functions on
$\mathbb C\cup\{\infty\}$ excluding finite number of singular points.
In view of the local behavior of solutions, the singularities of 
$\frac{b_j}{b_m}$ are at most poles and hence they are rational 
functions. Then we may assume $R=b_m\p^m+\cdots+b_0\in W(x;\xi)$
and $P\in W(x;\xi)R$.

Here we note that $R$ is Fuchsian but $R$ may have a singularity
which is not a singularity of $P$ and is an 
\textsl{apparent singularity}. 
\index{apparent singularity}
For example, we have
\begin{equation}
x(1-x)\p^2+(\gamma-\alpha x)\p+\alpha
=\Bigl(\frac\gamma\alpha -x\Bigr)^{-1}
\Bigl(x(1-x)\p +(\gamma - \alpha x)\Bigr)
\biggl(\Bigl(\frac\gamma\alpha-x\Bigr)\p+1\biggr).
\end{equation}
We also note that the equation $\p^2u=xu$ is irreducible and 
the monodromy of its solutions is reducible.
\end{rem}
\subsection{Tuples of partitions}\index{tuple of partitions}
For our purpose it will be better to allow some $m_{j,\nu}$ equal 0
and we generalize the notation of tuples of partitions as in \cite{O3}.
\begin{defn}\label{def:tuples}
Let $\mathbf m
=\bigl(m_{j,\nu}\bigr)_{\substack{j=0,1,\ldots\\ \nu=1,2,\ldots}}$ 
be an ordered set of infinite number of non-negative integers 
indexed by non-negative integers $j$ and positive integers $\nu$.
Then $\mathbf m$ is called a \textsl{$(p+1)$-tuple of partitions of $n$} 
if the following two conditions are satisfied.
\begin{align}
  \sum_{\nu=1}^\infty m_{j,\nu}&=n\qquad(j=0,1,\ldots),\\
  m_{j,1} &= n\qquad(\forall j > p).
\end{align}
A $(p+1)$-tuple of partition $\mathbf m$ is called 
\textsl{monotone}\index{tuple of partitions!monotone} if\begin{equation}
  m_{j,\nu} \ge m_{j,\nu+1}\quad(j=0,1,\ldots,\ \nu=1,2,\ldots)
\end{equation}
\index{tuple of partitions!trivial}%
and called \textsl{trivial} if $m_{j,\nu}=0$ for $j=0,1,\ldots$ and 
$\nu=2,3,\ldots$.
\index{tuple of partitions!standard}%
Moreover $\mathbf m$ is called \textsl{standard} if $\mathbf m$ is
monotone and $m_{j,2}>0$ for $j=0,\dots,p$.
\index{tuple of partitions!indivisible}%
\index{tuple of partitions!divisible}%
\index{00gcd@$\gcd$}%
The greatest common divisor of $\{m_{j,\nu};j=0,1,\ldots,\ \nu=1,2,\ldots\}$
is denoted by $\gcd\mathbf m$ and $\mathbf m$ is called 
\textsl{divisible} (resp.~\textsl{indivisible})
if $\gcd\mathbf m\ge2$ (resp.~$\gcd\mathbf m=1$).
The totality of $(p+1)$-tuples of partitions of $n$ are denoted by
${\mathcal P}_{p+1}^{(n)}$ and we put
\index{00P@$\mathcal P,\ \mathcal P_{p+1},\ 
\mathcal P^{(n)},\ \mathcal P_{p+1}^{(n)}$}
\index{00ord@$\ord$}%
\index{00Pidx@$\Pidx$}
\begin{align}
  {{\mathcal P}}_{p+1} &:=
    \bigcup_{n=0}^\infty {{\mathcal P}}_{p+1}^{(n)},\quad
  {{\mathcal P}}^{(n)} :=
    \bigcup_{p=0}^\infty {{\mathcal P}}_{p+1}^{(n)},\quad
  {{\mathcal P}}       :=
    \bigcup_{p=0}^\infty {{\mathcal P}}_{p+1},\\
 \ord\mathbf m &:= n\quad\text{if \ }
 \mathbf m\in{{\mathcal P}}^{(n)},
\allowdisplaybreaks\\
 \mathbf 1&:=(1,1,\ldots)=
  \bigl(m_{j,\nu}=\delta_{\nu,1}\bigr)_{\substack{j=0,1,\ldots\\\nu=1,2,\ldots}}\in\mathcal P^{(1)},
\allowdisplaybreaks\\
 \idx(\mathbf m,\mathbf m')&:=
 \sum_{j=0}^p\sum_{\nu=1}^\infty m_{j,\nu}m'_{j,\nu}
 -(p-1)\ord\mathbf m\cdot\ord\mathbf m',
\allowdisplaybreaks\index{00idx@$\idx$}\\
\idx \mathbf m&:=\idx(\mathbf m,\mathbf m) = 
 \sum_{j=0}^p\sum_{\nu=1}^\infty m_{j,\nu}^2-(p-1)\ord\mathbf m^2,
\allowdisplaybreaks\\
\Pidx\mathbf m&:=1-\frac{\idx\mathbf m}2.
\end{align}
\end{defn}
Here $\ord\mathbf m$ is called the \textsl{order} of $\mathbf m$.
\index{tuple of partitions!order}
For $\mathbf m,\,\mathbf m'\in\mathcal P$ and a non-negative integer $k$,
$\mathbf m+k\mathbf m'\in\mathcal P$ is naturally defined.
Note that
\begin{align}
 \idx(\mathbf m+\mathbf m') &= 
 \idx\mathbf m+\idx\mathbf m'+2\idx(\mathbf m,\mathbf m'),\\
 \Pidx(\mathbf m+\mathbf m') &=
 \Pidx\mathbf m+\Pidx\mathbf m'-\idx(\mathbf m,\mathbf m')-1.
\end{align}
\index{00Pidx@$\Pidx$}
For $\mathbf m\in{{\mathcal P}}_{p+1}^{(n)}$ we choose
integers $n_0,\dots,n_k$ so that $m_{j,\nu}=0$
for $\nu>n_j$ and $j=0,\dots,p$ and we will sometimes express
$\mathbf m$ as
\begin{align*}
 \mathbf m&=(\mathbf m_0,\mathbf m_1,\dots,\mathbf m_p)\\
          &=m_{0,1},\dots,m_{0,n_0};\ldots;m_{k,1},\dots,m_{p,n_p}\\
          &=m_{0,1}\cdots m_{0,n_0},m_{1,1}\cdots m_{1,n_1},\dots,
           m_{k,1}\cdots m_{p,n_p}
\end{align*}
if there is no confusion.
Similarly $\mathbf m=(m_{0,1},\dots,m_{0,n_0})$ 
if $\mathbf m\in\mathcal P_1$. Here
\begin{equation*}
  \mathbf m_j = (m_{j,1},\dots,m_{j,n_j}) \text{ \ and \ }
  \ord\mathbf m=m_{j,1}+\cdots+m_{j,n_j}\quad(0\le j\le p).
\end{equation*}
For example $\mathbf m=(m_{j,\nu})\in{{\mathcal P}}_{3}^{(4)}$
with $m_{1,1}=3$ and
$m_{0,\nu}=m_{2,\nu}=m_{1,2}=1$ for $\nu=1,\dots,4$
will be expressed by
\begin{equation*}
 \mathbf m=1,1,1,1;3,1;1,1,1,1=1111,31,1111=1^4,31,1^4.
\end{equation*}
Let $\mathfrak S_\infty$ be the restricted permutation group of
the set of indices $\{0,1,2,3,\ldots\}=\mathbb Z_{\ge 0}$, 
which is generated by the transpositions $(j,j+1)$ with $j\in\mathbb Z_{\ge0}$.
Put $\mathfrak S_\infty'=\{\sigma\in\mathfrak S_\infty\,;\,\sigma(0)=0\}$,
which is isomorphic to $\mathfrak S_\infty$.
\begin{defn}\label{def:Sinfty}
\index{00s@$s$, $s\mathbf m$}%
\index{00Sinfty@$S_\infty,\ S_\infty'$}%
The transformation groups  $S_\infty$ and $S_\infty'$ of
$\mathcal P$ are defined by
\begin{equation}\label{eq:S_infty}
 \begin{split}
 S_\infty :\!&=H\ltimes S_\infty',\\
 S_\infty':\!&=\{(\sigma_i)_{i=0,1,\ldots}\,;\,\sigma_i\in \mathfrak S_\infty',\ 
 \sigma_i=1\ \ (i\gg1)\},\quad
H\simeq\mathfrak S_\infty,\\
 m'_{j,\nu} &= m_{\sigma(j),\sigma_j(\nu)}\qquad(j=0,1,\ldots,\ \nu=1,2,\ldots)
 \end{split}
\end{equation}
for $g = (\sigma,\sigma_1,\ldots) \in S_\infty$, 
$\mathbf m=(m_{j,\nu})\in \mathcal P$ and $\mathbf m'=g\mathbf m$.
A tuple $\mathbf m\in\mathcal P$ is \textsl{isomorphic} to a tuple
$\mathbf m'\in\mathcal P$ if there exists $g\in S_\infty$ such that 
$\mathbf m'=g\mathbf m$.
\index{tuple of partitions!isomorphic}
We denote by $s\mathbf m$ the unique monotone element in 
$S'_\infty\mathbf m$.
\end{defn}
\begin{defn}\label{def:FRLM}
For a tuple of partitions 
$\mathbf m=\Bigl(m_{j,\nu}\Bigr)
_{\substack{1\le\nu\le n_j\\0\le j\le p}}
\in\mathcal P_{p+1}$
and $\lambda=\Bigl(\lambda_{j,\nu}\Bigr)_{\substack{1\le\nu\le n_j\\0\le j\le p}}$ with $\lambda_{j,\nu}\in\mathbb C$, 
we define
\index{000lambda@$\arrowvert$\textbraceleft$\lambda_{\mathbf m}$\textbraceright$\arrowvert$}
\begin{equation}
 \bigl|\{\lambda_{\mathbf m}\}\bigr|
 :=\sum_{j=0}^{p}\sum_{\nu=1}^{n_j}
  m_{j,\nu}\lambda_{j,\nu}-\ord\mathbf m
  +\frac{\idx\mathbf m}2.
\end{equation}
\end{defn}
We note that the Fuchs relation \eqref{eq:Fuchs} is equivalent to
\begin{equation}\label{eq:Fuchidx}\index{Fuchs relation}
 |\{\lambda_{\mathbf m}\}|=0
\end{equation}
because
\begin{align*}
 \sum_{j=0}^p\sum_{\nu=1}^{n_j}\sum_{i=0}^{m_{j,\nu}-1}i
 &=\frac12\sum_{j=0}^p\sum_{\nu=1}^{n_j} m_{j,\nu}(m_{j,\nu}-1)
 =\frac12\sum_{j=0}^p\sum_{\nu=1}^{n_j}m_{j,\nu}^2-\frac12(p+1)n\\
 &=\frac12\Bigl(\idx\mathbf m+(p-1)n^2\Bigr)-\frac12(p+1)n\\
 &=\frac12\idx\mathbf m-n+\frac{(p-1)n(n-1)}2.
\end{align*}
\subsection{Conjugacy classes of matrices}
Now we review on the conjugacy classes of matrices.
For $\mathbf m=(m_1,\dots,m_N)\in\mathcal P^{(n)}_1$ 
and $\lambda=(\lambda_1,\dots,\lambda_N)\in\mathbb C^N$
we define a matrix $L(\mathbf m;\lambda)\in M(n,\mathbb C)$ 
as follows, which is introduced and effectively used by \cite{Os}
and \cite{O3}:

If $\mathbf m$ is monotone, then
\begin{equation}\label{eq:OSNF}
\begin{split}
 L(\mathbf m;\mathbf \lambda) 
  :\!&= \Bigl(A_{ij}\Bigr)_{\substack{1\le i\le N\\1\le j\le N}},\quad
 A_{i,j}\in M(m_i,m_j,\mathbb C),\\
 A_{ij} &= \begin{cases}
          \lambda_i I_{m_i}&(i=j),\\
          I_{m_i,m_j}:=
          \Bigl(\delta_{\mu\nu}\Bigr)
          _{\substack{1\le \mu\le n_i\\1\le \nu\le n_j}}
	=
          \begin{pmatrix}
          I_{m_j} \\ 0
          \end{pmatrix}&(i=j-1),\\
          0            &(i\ne j,\ j-1).
          \end{cases}
\end{split}
\end{equation}
\index{00Lmlambda@$L(\mathbf m;\lambda)$}%
\index{00Mm@$M(n,\mathbb C)$, $M(m_1,m_2,\mathbb C)$}%
Here $I_{m_i}$ denote the identity matrix of size $m_i$ and
$M(m_i,m_j,\mathbb C)$ means the set of matrices of size $m_i\times m_j$
with components in $\mathbb C$ and
$M(m,\mathbb C):=M(m,m,\mathbb C)$.

For example
\begin{equation*}
 L(2,1,1;\lambda_1,\lambda_2,\lambda_3):=
 \begin{pmatrix}
 \lambda_1 & 0       &1\\
 0         &\lambda_1& 0\\
           &         &\lambda_2&1\\
           &         &         &\lambda_3\\
 \end{pmatrix}.
\end{equation*}
If $\mathbf m$ is not monotone, we fix a permutation 
$\sigma$ of $\{1,\dots,N\}$ so that 
$(m_{\sigma(1)},\ldots,m_{\sigma(N)})$ is monotone and put 
$L(\mathbf m;\mathbf \lambda)=L(m_{\sigma(1)},\ldots,
m_{\sigma(N)};\lambda_{\sigma(1)},\ldots,\lambda_{\sigma(N)})$.

When $\lambda_1=\cdots=\lambda_N=\mu$, $L(\mathbf m;\lambda)$ may be
simply denoted by $L(\mathbf m,\mu)$.

We denote $A\sim B$ for $A$, $B\in M(n,\mathbb C)$ if and only if
there exists $g\in GL(n,\mathbb C)$ with $B=gAg^{-1}$.

When $A\sim L(\mathbf m;\lambda)$, 
$\mathbf m$ is called the \textsl{spectral type} of $A$ and 
denoted by $\spc A$ with a monotone $\mathbf m$.

\begin{rem}\label{rm:1}
{\rm i)\ }
If $\mathbf m=(m_1,\dots,m_K)\in\mathcal P_1^{(n)}$ is monotone, we have 
\begin{equation*}
 A\sim L(\mathbf m;\lambda)\ \Leftrightarrow\ 
  \rank\prod_{\nu=1}^j(A-\lambda_\nu)
 = n - (m_1+\cdots+m_j)\quad(j=0,1,\dots,K).
\end{equation*}

{\rm ii)\ } For $\mu\in\mathbb C$, put 
\begin{equation}\label{eq:Msub}
 (\mathbf m;\lambda)_\mu
    =(m_{i_1},\ldots,m_{i_K};\mu)
\text{ \ with \ }\{i_1,\dots,i_K\}=\{i\,;\,\lambda_i=\mu\}.
\end{equation}
Then we have
\begin{equation}\label{eq:Leigen}
  L(\mathbf m;\lambda) \sim\bigoplus_{\mu\in\mathbb C}
  L\bigl((\mathbf m;\lambda)_\mu\bigr).
\end{equation}

{\rm iii)\ } Suppose $\mathbf m$ is monotone.  
Then for $\mu\in\mathbb C$ 
\begin{equation}\label{eq:LJordan}
 \begin{split}
  L(\mathbf m,\mu) &\sim
  \bigoplus_{j=1}^{m_1} J\bigl(\max\{\nu\,;\,m_\nu\ge j\},\mu\bigr),\\
  J(k,\mu) :\!&=L(1^k,\mu)\in M(k,\mathbb C).
 \end{split}
\end{equation}

{\rm iv)\ }  For $A\in M(n,\mathbb C)$, we put $Z(A)=Z_{M(n,\mathbb C)}(A)
:=\{X\in M(n,\mathbb C)\,;\,AX=XA\}$.  
\index{00Z@$Z(A)$, $Z(\mathbf M)$}%
Then 
\begin{equation*}
  \dim Z_{M(n,\mathbb C)}\bigl(L(\mathbf m,\lambda)\bigr)
  = m_1^2+m_2^2+\cdots
\end{equation*} 

{\rm v)\ } (cf.~\cite[Lemma~3.1]{O5}).
Let $\mathbf A(t):\,[0,1)\to M(n,\mathbb C)$ be a continuous function.
Suppose there exist a continuous function 
$\lambda=(\lambda_1,\dots,\lambda_K):\,[0,1)\to\mathbb C^K$
such that $A(t)\sim L(\mathbf m;\lambda(t))$ for $t\in(0,1)$.
Then
\begin{equation}
  A(0)\sim L\bigl(\mathbf m;\lambda(0)\bigr)
  \text{ \ if and only if \ }\dim Z\bigl(A(0)\bigr)=m_1^2+\cdots+m_K^2.
\end{equation}
\end{rem}
Note that the Jordan canonical form of $L(\mathbf m;\lambda)$ is 
easily obtained by \eqref{eq:Leigen} and \eqref{eq:LJordan}.
For example, $L(2,1,1;\mu)\simeq J(3,\mu)\oplus J(1,\mu)$.
\subsection{Realizable tuples of partitions}
\begin{prop}
Let $Pu=0$ be a differential equation of order $n$ 
which has a regular singularity at $0$.
Let $\{[\lambda_1]_{(m_1)},\dots,[\lambda_q]_{(m_q)}\}$
be the corresponding set of the characteristic exponents.
Here $\mathbf m=(m_1,\dots,m_q)$ a partition of $n$.

{\rm i) }
Suppose there exists $k$ such that
\begin{gather*}
 \lambda_1=\lambda_2=\cdots=\lambda_k,\\
 m_1\ge m_2\ge \cdots\ge m_k,\\
 \lambda_j-\lambda_1\notin\mathbb Z\qquad(j=k+1,\dots,q).
\end{gather*}
Let $\mathbf m^\vee=(m_1^\vee,\dots,m^\vee_r)$ be the
dual partition of $(m_1,\dots,m_k)$ {\rm (cf.~\eqref{eq:dualP})}.
Then for $i=0,\ldots,m_1-1$ and $j=0,\ldots,m^\vee_{i+1}-1$
the equation has the solutions
\begin{equation}
  u_{i,j}(x) = \sum_{\nu=0}^j x^{\lambda_1+i}\log^\nu x\cdot\phi_{i,j,\nu}(x).
\end{equation}
Here $\phi_{i,j,\nu}(x)\in\mathcal O_0$ and $\phi_{i,\nu,j}(0)=\delta_{\nu,j}$
for $\nu=0,\dots,j-1$.

{\rm ii) }
Suppose
\begin{equation}
 \lambda_i-\lambda_j\ne\mathbb Z\setminus\{0\}
 \qquad(0\le i<j\le q).
\end{equation}
In this case we say that the set of characteristic exponents
$\{[\lambda_1]_{(m_1)},\dots,[\lambda_q]_{(m_q)}\}$ is 
\textsl{distinguished}.
\index{characteristic exponent!distinguished}
Then the monodromy generator of the solutions of the
equation at $0$ is conjugate to
\[
 L\bigl(\mathbf m;(e^{2\pi\sqrt{-1}\lambda_1},\dots,
  e^{2\pi\sqrt{-1}\lambda_q})\bigr).
\]
\end{prop}
\begin{proof} 
Lemma~\ref{lem:GRS} ii) shows that there exist $u_{i,j}(x)$ of the form
stated in i) which satisfy $Pu_{i,j}(x)\in\mathcal O_0(\lambda_1+m_1,j)$
and then we have $v_{i,j}(x)\in\mathcal O_0(\lambda_1+m_1,j)$ such that 
$Pu_{i,j}(x)=Pv_{i,j}(x)$ because of \eqref{eq:Pbij}.
Thus we have only to replace $u_{i,j}(x)$ by
$u_{i,j}(x)-v_{i,j}(x)$ to get the claim in i).
The claim in ii) follows from that of i).
\end{proof}
\begin{rem} 
{\rm i) } Suppose $P$ is a Fuchsian differential operator with regular 
singularities at $x=c_0=\infty, c_1,\dots,c_p$ and moreover suppose 
$P$ has distinct exponents.
Then the Riemann scheme of $P$ is \eqref{eq:GRS} if and only if 
$Pu=0$ has local solutions $u_{j,\nu,i}(x) $ of the form
\begin{equation}
  u_{j,\nu,i}(x) =
  \begin{cases}
    (x-c_j)^{\lambda_{j,\nu}+i}\bigl(1+o(|x-c_j|^{m_j,{\nu}-i-1})\bigr)\\
    \qquad\qquad(x\to c_j,\ i=0,\dots,m_{j,\nu}-1,\ j=1,\dots,p),\\
    x^{-\lambda_{0,\nu}-i}\bigl(1+o(x^{-m_{0,\nu}+i+1})\bigr)\\
    \qquad\qquad(x\to\infty,\ i=0,\dots,m_{0,\nu}).
  \end{cases}
\end{equation}
Moreover suppose $\lambda_{j,\nu}-\lambda_{j,\nu'}\notin\mathbb Z$
for $1\le\nu<\nu'\le n_j$ and $j=0,\dots,p$.
Then
\begin{equation}
  u_{j,\nu,i}(x)=
  \begin{cases}
     (x-c_j)^{\lambda_{j,\nu}+i}\phi_{j,\nu,i}(x) &(1\le j\le p)\\
     x^{-\lambda_{0,\nu}-i}\phi_{0,\nu,i}(x) & (j=0)
  \end{cases}
\end{equation}
with $\phi_{j,\nu,i}(x)\in\mathcal O_{c_j}$ satisfying
$\phi_{j,\nu,i}(c_j)=1$.
In this case $P$ has the Riemann scheme \eqref{eq:GRS} if and only if
at the each singular point $x=c_j$, 
the set of characteristic exponents of the equation $Pu=0$ equals
$\Lambda_j$ in \eqref{eq:Lambdaj} and the monodromy generator of its 
solutions is semisimple.

{\rm ii) }
Suppose $P$ has the Riemann scheme \eqref{eq:GRS} and
$\lambda_{1,1}=\dots=\lambda_{1,n_1}$.
Then the monodromy generator of the solutions of $Pu=0$ at $x=c_1$ has 
the eigenvalue $e^{2\pi\sqrt{-1}\lambda_{1,1}}$ with multiplicity $n$.
Moreover the monodromy generator is conjugate to
the matrix 
$L\bigl((m_{1,1},\dots,m_{1,n_1}),e^{2\pi\sqrt{-1}\lambda_{1,1}}\bigr)$,
 which is also conjugate to
\begin{align*}
  J(m^\vee_{1,1},e^{2\pi\sqrt{-1}\lambda_{1,1}})\oplus\cdots\oplus 
  J(m^\vee_{1,n'_1},e^{2\pi\sqrt{-1}\lambda_{1,1}}).
\end{align*}
Here $(m^\vee_{1,1},\dots,m^\vee_{1,n^\vee_1})$ is the dual partition of 
$(m_{1,1},\dots,m_{1,n_1})$.
A little weaker condition for $\lambda_{j,\nu}$ assuring the same conclusion
is given in Proposition~\ref{prop:nondeg}.
\end{rem}
\begin{defn}[realizable spectral type]\label{defn:real}
\index{tuple of partitions!realizable}
Let $\mathbf m=(\mathbf m_0,\dots,\mathbf m_p)$ be a $(p+1)$-tuple
of partitions of a positive integer $n$.  Here
$\mathbf m_j = (m_{j,1},\dots,m_{j,n_j})$ and 
$n=m_{j,1}+\dots+m_{j,n_j}$ for $j=0,\dots,p$ and
$m_{j,\nu}$ are non-negative numbers.
Fix $p$ different points $c_j$ ($j=1,\dots,p$) in $\mathbb C$
and put $c_0=\infty$.

Then $\mathbf m$ is a \textsl{realizable spectral type} if
there exists a Fuchsian operator $P$ with the Riemann scheme
\eqref{eq:GRS} for generic $\lambda_{j,\nu}$ satisfying
the Fuchs relation \eqref{eq:Fuchs}.
Moreover in this case if there exists such $P$ so that the equation
$Pu=0$ is irreducible, which is equivalent to say that
the monodromy of the equation is irreducible, then $\mathbf m$ is 
\textsl{irreducibly realizable}.
\index{tuple of partitions!realizable}
\index{tuple of partitions!irreducibly realizable}
\end{defn}
\begin{rem}\label{rem:generic}
 {\rm i)} \ 
In the above definition $\{\lambda_{j,\nu}\}$ are generic if,
for example, $0<m_{0,1}<\ord\mathbf m$ and
$\{\lambda_{j,\nu}\,;\,(j,\nu)\ne (0,1),\ j=0,\dots,p,\ 1\le\nu\le n_j\}
\cup\{1\}$ are linearly independent over $\mathbb Q$.

{\rm ii) }
It follows from the facts (cf.~\eqref{eq:FC1}) in 
\S\ref{sec:reg} that if $\mathbf m\in\mathcal P$ satisfies
\begin{equation}\label{eq:FC12}
 \begin{split}
  &|\{\lambda_{\mathbf m'}\}|\notin\mathbb Z_{\le 0}
  =\{0,-1,-2,\ldots\}
  \text{ for any }\mathbf m',\,\mathbf m''\in\mathcal P\\
  &\quad\text{ satisfying }\mathbf m=\mathbf m'+\mathbf m''
  \text{ and }0<\ord\mathbf m'<\ord\mathbf m,
 \end{split}
\end{equation}
the Fuchsian differential equation with the Riemann scheme 
\eqref{eq:GRS} is irreducible.
Hence if $\mathbf m$ is indivisible and realizable, 
$\mathbf m$  is irreducibly realizable.
\end{rem}
Fix distinct $p$ points $c_1,\dots,c_p$ in $\mathbb C$ and put $c_0=\infty$.
The Fuchsian differential operator $P$ with regular singularities
at $x=c_j$ for $j=1,\dots,n$ has the \textsl{normal form}
\index{Fuchsian differential equation/operator!normal form}
\begin{equation}\label{eq:FNF}
 P = \Bigl(\prod_{j=1}^p (x-c_j)^n\Bigr)\p^n
     + a_{n-1}(x)\p^{n-1}+\cdots+a_1(x)\p + a_0(x),
\end{equation}
where $a_i(x)\in\mathbb C[x]$ satisfy
\begin{align}
 \deg a_i(x)&\le (p-1)n+i,\label{C:RSfin}\\
 (\p^\nu a_i)(c_j) &=0
 \quad(0\le\nu \le i-1)\label{C:RSinfin}
\end{align}
for $i=0,\dots,n-1$.

Note that the condition \eqref{C:RSfin} (resp.~\eqref{C:RSinfin})
corresponds to the fact
that $P$ has regular singularities at $x=c_j$ for $j=1,\dots,p$ 
(resp.~at $x=\infty$).

Since $a_i(x)=b_i(x)\prod_{j=1}^p(x-c_j)^i$ with
$b_i(x)=\sum_{r=0}^{(p-1)(n-i)}b_{i,r}x^r \in W[x]$ satisfying
$\deg b_i(x)\le (p-1)n+i - pi=(p-1)(n-i)$, 
the operator $P$ has the parameters $\{b_{i,r}\}$.  
The numbers of the parameters equals
\[
  \sum_{i=0}^{n-1}\bigl((p-1)(n-i)+1\bigr)
 =\frac{(pn +p-n+1)n}{2},
\]
The condition $(x-c_j)^{-k}P\in W[x]$ implies $(\p^\ell a_i)(c_j)=0$
for $0\le \ell\le k-1$ and $0\le i\le n$, 
which equals $(\p^\ell b_i)(c_j)=0$ for $0\le \ell\le k-1-i$
and $0\le i\le k-1$.  
Therefore the condition
\begin{equation}\label{C:eachblock}
 (x-c_j)^{-m_{j,\nu}}\Ad\bigl((x-c_j)^{-\lambda_{j,\nu}}\bigr)P\in W[x]
\end{equation}
gives $\frac{(m_{j,\nu}+1)m_{j,\nu}}{2}$ independent linear equations 
for $\{b_{\nu,r}\}$ since $\sum_{i=0}^{m_{j,\nu}-1} (m_{j,\nu}-i)=
\frac{(m_{j,\nu}+1)m_{j,\nu}}{2}$.
If all these equations have a simultaneous solution and they are 
independent except for the relation caused by the Fuchs relation,
the number of the parameters of the solution equals
\begin{equation}
 \begin{split}
 &\frac{(pn +p-n+1)n}{2} - 
  \sum_{j=0}^p\sum_{\nu=1}^{n_j}\frac{m_{j,\nu}(m_{j,\nu}+1)}{2}+1\\
 &=\frac{(pn +p-n+1)n}{2} - 
  \sum_{j=0}^p\sum_{\nu=1}^{n_j}\frac{m_{j,\nu}^2}{2}-(p+1)\frac{n}{2}+1\\
 &=\frac12\Bigl((p-1)n^2 - \sum_{j=0}^p\sum_{\nu=1}^{n_j}m_{j,\nu}^2+1
   \Bigr)=\Pidx\mathbf m.
 \end{split}
\end{equation}
\begin{rem}[{cf.~\cite[\S5]{O3}}]
Katz \cite{Kz} introduced the \textsl{index of rigidity}
of an irreducible local system by the number  $\idx\mathbf m$ 
whose spectral type equals 
$\mathbf m=(m_{j,\nu})_{\substack{ j=0,\dots,p\ \,\\ \nu=1,\dots,n_{j}}}$
and proves $\idx\mathbf m\le 2$, if the local system is irreducible.
\index{tuple of partitions!index of rigidity}

Assume the local system is irreducible.
Then Katz \cite{Kz} shows that
the local system is uniquely determined by the local monodromies
if and only if $\idx\mathbf m=2$ and in this case the local system and
the tuple of partition $\mathbf m$ are called \textsl{rigid}.
\index{tuple of partitions!rigid}
If $\idx\mathbf m>2$, the corresponding system of differential 
equations of \textsl{Schleginger normal form}
\begin{equation}
  \frac{du}{dx} = \sum_{j=1}^p\frac{A_j}{x-a_j}u
\end{equation}
has $2\Pidx\mathbf m$ parameters
which are independent from the characteristic exponents and 
local monodromies.
They are called \textsl{accessory parameters}.
\index{accessory parameter}
Here $A_j$ are constant square matrices of size $n$.
The number of accessory parameters of the single Fuchsian
differential operator without apparent singularities will be the
half of this number $2\Pidx\mathbf m$ (cf.~Theorem~\ref{thm:univmodel} and 
\cite{Sz}). 
\end{rem}
Lastly in this subsection we calculate the Riemann scheme of the products and 
the dual of Fuchsian differential operators.
\begin{thm}\label{thm:prod} 
Let $P$ be a Fuchsian differential operator with
the Riemann scheme \eqref{eq:GRS}.
Suppose $P$ has the normal form \eqref {eq:FNF}.

{\rm i)}
Let $P'$ be a Fuchsian differential operator
with regular singularities 
also at $x=c_0=\infty,c_1,\dots,c_p$.
Then if $P'$ has the Riemann scheme 
\begin{equation}\label{eq:GRSP}
   \begin{Bmatrix}
   x = c_0=\infty & c_j\quad(j=1,\dots,p)\\
  [\lambda_{0,1}+m_{0,1}-(p-1)\ord\mathbf m]_{(m'_{0,1})} & 
  [\lambda_{j,1}+m_{j,1}]_{(m'_{j,1})} \\
   \vdots & \vdots\\
    [\lambda_{0,n_0}+m_{0,n_0}-(p-1)\ord\mathbf m]_{(m'_{0,n_0})} 
   & [\lambda_{j,n_j}+m_{j,n_j}]_{(m'_{j,n_j})}
  \end{Bmatrix},
\end{equation}
the Fuchsian operator $P'P$ has the spectral type $\mathbf m+\mathbf m'$ and
the Riemann scheme
\begin{equation}\label{eq:GRSA}
  \begin{Bmatrix}
   x = c_0 =\infty & c_1 & \cdots & c_p\\
  [\lambda_{0,1}]_{(m_{0,1}+m'_{0,1})} 
   & [\lambda_{1,1}]_{(m_{1,1}+m'_{1,1})}&\cdots
    &[\lambda_{p,1}]_{(m_{p,1}+m'_{p,1})}\\
  \vdots & \vdots & \vdots & \vdots\\
    [\lambda_{0,n_0}]_{(m_{0,n_0}+m'_{0,n_0})} & 
     [\lambda_{1,n_1}]_{(m_{1,n_1}+m'_{1,n_1})}&\cdots
    &[\lambda_{p,n_p}]_{(m_{p,n_p}+m'_{1,n_p})}
  \end{Bmatrix}.
\end{equation}
Suppose the Fuchs relation \eqref{eq:Fuchidx} for \eqref{eq:GRS}.
Then the Fuchs relation for \eqref{eq:GRSP} is valid if and only if
so is the Fuchs relation for \eqref{eq:GRSA}.

{\rm ii)} For $Q=\sum_{k\ge0} q_k(x)\p^k \in W(x)$, we define
\begin{equation}
  Q^* := \sum_{k\ge 0}(-\p)^kq_k(x)
\end{equation}
and the dual operator $P^\vee$ of $P$ by
\index{00Pvee@$P^\vee,\ P^*$}
\index{differential equation/operator!dual}
\index{Fuchsian differential equation/operator!dual}
\begin{equation}\label{eq:dualop}
  P^\vee := a_n(x)(a_n(x)^{-1}P)^*
\end{equation}
when $P=\sum_{k=0}^na_k(x)\p^k$.
Then the Riemann scheme of $P^\vee$ equals
\begin{equation}
\begin{Bmatrix}
    x = c_0 =\infty & c_j\quad(j=1,\dots,p)\\
  [2-n-m_{0,1}-\lambda_{0,1}]_{(m_{0,1})} & 
  [n-m_{j,1}-\lambda_{j,1}]_{(m_{j,1})}\\
  \vdots & \vdots\\
  [2-n-m_{0,n_0}-\lambda_{0,n_0}]_{(m_{0,n_0})} & 
  [n-m_{j,n_j}-\lambda_{j,n_j}]_{(m_{j,n_j})}
\end{Bmatrix}.
\end{equation}
\end{thm}
\begin{proof} i) \ 
It is clear that $P'P$ is a Fuchsian differential operator of the
normal form if so is $P'$ and Lemma~\ref{lem:GRS} iii) shows that
the characteristic exponents of $P'P$ at $x=c_j$ for $j=1,\dots,p$ are 
just as given in the Riemann scheme \eqref{eq:GRSA}.
Put $n=\ord\mathbf m$ and $n'=\mathbf m'$.
We can also apply Lemma~\ref{lem:GRS} iii) to $x^{-(p-1)n}P$ and 
$x^{-(p-1)n'}P'$ under the coordinate transformation $x\mapsto\frac1x$,
we have the set of characteristic exponents as is given in \eqref{eq:GRSA}
because $x^{-(p-1)(n+n')}P'P=
\bigl(\Ad(x^{-(p-1)n})x^{-(p-1)n'}P'\bigr)(x^{-(p-1)n})P$.

The Fuchs relation for \eqref{eq:GRSP} equals
\[
 \sum_{j=0}^p\sum_{\nu=1}^{n_j} m'_{j,\nu}\bigl(\lambda_{j,\nu}+m_{j,\nu}
  -\delta_{j,0}(p-1)\ord\mathbf m\bigr) =\ord\mathbf m'-\frac{\idx\mathbf m'}2.
\]
Since
\[
 \sum_{j=0}^p\sum_{\nu=1}^{n_j}
  m'_{j,\nu}\bigl(m_{j,\nu} - \delta_{j,0}(p-1)\ord\mathbf m\bigr)
 =\idx(\mathbf m,\mathbf m'),
\]
the condition is equivalent to
\begin{equation}
  \sum_{j=0}^p\sum_{\nu=1}^{n_j} m'_{j,\nu}\lambda_{j,\nu}
  =\ord\mathbf m' -\frac{\idx\mathbf m}2 - \idx(\mathbf m,\mathbf m')
\end{equation}
and also to
\begin{equation}
  \sum_{j=0}^p\sum_{\nu=1}^{n_j} 
  (m_{j,\nu}+m'_{j,\nu})\lambda_{j,\nu}
  =\ord(\mathbf m + \mathbf m') -\frac{\idx(\mathbf m+\mathbf m')}2
\end{equation}
under the condition \eqref{eq:Fuchidx}.

ii) \ We may suppose $c_1=0$.  Then
\begin{align*}
  a_n(x)^{-1}P &= \sum_{\ell\ge 0}x^{\ell-n}q_\ell(\vartheta)
      \prod_{\substack{1\le\nu\le n_1\\0\le i< m_{1,\nu}-\ell}}
      (\vartheta-\lambda_{1,\nu}-i),\\
  a_n(x)^{-1}P^\vee &=
      \sum_{\ell\ge 0}q_\ell(-\vartheta-1)
      \prod_{\substack{1\le\nu\le n_1\\0\le i< m_{1,\nu}-\ell}}
      (-\vartheta-\lambda_{1,\nu}-i-1)x^{\ell-n}\\
      &=\sum_{\ell\ge 0}x^{\ell-n}s_\ell(\vartheta)
        \prod_{\substack{1\le\nu\le n_1\\0\le i< m_{1,\nu}-\ell}}
      (\vartheta+\lambda_{1,\nu}+i+1+\ell-n)\\
    &=\sum_{\ell\ge 0}x^{\ell-n}s_\ell(\vartheta)
        \prod_{\substack{1\le\nu\le n_1\\0\le j< m_{1,\nu}-\ell}}
      (\vartheta+\lambda_{1,\nu}-j+m_{1,\nu}-n)
\end{align*}
with suitable polynomials $q_\ell$ and $s_\ell$ such that 
$q_0,\,s_0\in\mathbb C^\times$.
Hence the set of characteristic exponents of $P^\vee$ at $c_1$
is $\{[n-m_{1,\nu}-\lambda_{1,\nu}]_{(m_{1,\nu})}\,;\,\nu=1,\dots,n_1\}$.

At infinity we have
\begin{align*}
 a_n(x)^{-1}P&=\sum_{\ell\ge 0}x^{-\ell-n}
               q_\ell(\vartheta)
              \prod_{\substack{1\le\nu\le n_1\\0\le i< m_{0,\nu}-\ell}}
              (\vartheta+\lambda_{0,\nu}+i),\\
(a_n(x)^{-1}P)^*&=\sum_{\ell\ge 0}x^{-\ell-n}
               s_\ell(\vartheta)
               \prod_{\substack{1\le\nu\le n_0\\0\le i< m_{0,\nu}-\ell}}
               (\vartheta-\lambda_{0,\nu}-i+1-\ell-n)\\
              &=\sum_{\ell\ge 0}x^{-\ell-n}
               s_\ell(\vartheta)
               \prod_{\substack{1\le\nu\le n_1\\0\le j< m_{0,\nu}-\ell}}
               (\vartheta-\lambda_{0,\nu}+j+2-n-m_{0,\nu})
\end{align*}
with suitable polynomials $q_\ell$ and $s_\ell$ with 
$q_0,\,s_0\in\mathbb C^\times$ and
the set of characteristic exponents of $P^\vee$ at $c_1$
is $\{[2-n - m_{0,\nu}-\lambda_{0,\nu}]_{(m_{0,\nu})}\,;\,\nu=1,\dots,n_0\}$
\end{proof}
\begin{exmp}
The Riemann scheme of the dual $P_{\lambda_1,\dots,\lambda_p,\mu}^\vee$ 
of Jordan-Pochhammer operator
$P_{\lambda_1,\dots,\lambda_p,\mu}^\vee$ given in Example~\ref{ex:midconv} iii)
is
\[
 \begin{Bmatrix}
  \frac1{c_1} & \cdots & \frac1{c_p} & \infty\\
   [1]_{(p-1)}  &  \cdots & [1]_{(p-1)} & [2-2p+\mu]_{(p-1)}\\
   \lambda_1-\mu+p-1 & \cdots & -\lambda_p-\mu+p-1 & 
   \lambda_1+\cdots+\lambda_p+\mu-p+1
 \end{Bmatrix}.
\]
\end{exmp}
\section{Reduction of Fuchsian differential equations}\label{sec:reduction}
Additions and middle convolutions introduced in \S\ref{sec:frac}
are transformations within Fuchsian differential operators
and we examine how their Riemann schemes change under the
transformations.
\begin{prop}\label{prop:invred}
{\rm i)}
Let $Pu=0$ be a Fuchsian differential equation.
Suppose there exists $c\in\mathbb C$ such that $P\in(\p-c)W[x]$.
Then $c=0$.

{\rm ii)}
For $\phi(x)\in \mathbb C(x)$, $\lambda\in\mathbb C$, $\mu\in\mathbb C$ and 
$P\in W[x]$, we have
\begin{align}
  P&\in\mathbb C[x]\RAdei\bigl(-\phi(x)\bigr)
    \circ\RAdei\bigl(\phi(x)\bigr)P,\\
  P&\in\mathbb C[\p]\RAd\bigl(\p^{-\mu}\bigr)
    \circ\RAd\bigl(\p^\mu\bigr)P.
\end{align}
In particular, if the equation $Pu=0$ is irreducible and $\ord P>1$, 
$\RAd\bigl(\p^{-\mu}\bigr)
    \circ\RAd\bigl(\p^\mu\bigr)P = cP$ with $c\in\mathbb C^\times$.
\end{prop}
\begin{proof}
{\rm i)} Put $P=(\p-c)Q$.  Then there is a function $u(x)$ satisfying
$Qu(x)=e^{cx}$.  Since $Pu=0$ has at most a regular singularity at 
$x=\infty$, there exist $C>0$ and $N>0$ such that $|u(x)|<C|x|^N$ for 
$|x|\gg1$ and $0\le \arg x\le 2\pi$, which implies $c=0$.

{\rm ii)}
This follows from the fact
\begin{align*}
 &\Adei\bigl(-\phi(x)\bigr)\circ\Adei\bigl(\phi(x)\bigr)=\id,\\ 
 &\Adei\bigl(\phi(x)\bigr)f(x)P=f(x)\Adei\bigl(\phi(x)\bigr)P
 \quad(f(x)\in\mathbb C(x))
\end{align*}
and the definition of $\RAdei\bigl(\phi(x)\bigr)$ and $\RAd(\p^\mu)$.
\end{proof}
The addition and the middle convolution
transform the Riemann scheme of the Fuchsian differential equation
as follows.

\begin{thm}\label{thm:GRSmid}
Let $Pu=0$ be a Fuchsian differential equation with 
the Riemann scheme \eqref{eq:GRS}.
We assume that $P$ has the normal form \eqref{eq:FNF}. 

{\rm i) (addition)} \ 
The operator $\Ad\bigl((x-c_j)^\tau \bigr)P$ has the Riemann scheme
\begin{equation*}
   \begin{Bmatrix}
   x = c_0=\infty & c_1 & \cdots &c_j&\cdots& c_p\\
  [\lambda_{0,1}-\tau]_{(m_{0,1})} & [\lambda_{1,1}]_{(m_{1,1})}&\cdots
    & [\lambda_{j,1}+\tau]_{(m_{j,1})}&\cdots &[\lambda_{p,1}]_{(m_{p,1})}\\
  \vdots & \vdots & \vdots & \vdots& \vdots& \vdots\\
    [\lambda_{0,n_0}-\tau]_{(m_{0,n_0})} 
    & [\lambda_{1,n_1}]_{(m_{1,n_1})}&\cdots
    & [\lambda_{j,n_j}+\tau]_{(m_{j,1})}&\cdots&[\lambda_{p,n_p}]_{(m_{p,n_p})}
  \end{Bmatrix}.
\end{equation*}

{\rm ii) (middle convolution)} \ 
Fix $\mu\in\mathbb C$.
By allowing the condition $m_{j,1}=0$, we may assume 
\begin{equation}\label{eq:midgen}
 \mu=\lambda_{0,1}-1\text{ \ and \ }
 \lambda_{j,1}=0\text{ for }j=1,\dots,p
\end{equation}
and $\#\{j\,;\,m_{j,1}<n\}\ge 2$ and $P$ is of the normal form \eqref{eq:FNF}.
Putting \begin{equation}\label{eq:defd}
 d:=\sum_{j=0}^p m_{j,1} - (p-1)n,
\end{equation}
we suppose
\begin{align}
   &\quad m_{j,1}\ge d\text{ \ for \ }j=0,\dots,p,\label{eq:redC1}\\
   &\begin{cases}
    \lambda_{0,\nu}\notin\{0,-1,-2,\ldots,m_{0,1}-m_{0,\nu}-d+2\}\\
    \text{if \ }m_{0,\nu}+\cdots+m_{p,1}-(p-1)n\ge 2,\ 
     m_{1,1}\cdots m_{p,1}\ne0\text{ \ and \ }\nu\ge 1,
   \end{cases}\label{eq:redC2}\\
   &\begin{cases}
    \lambda_{0,1}+\lambda_{j,\nu}\notin\{0,-1,-2,\ldots,m_{j,1}-m_{j,\nu}-d+2\}\\
    \text{if \ }m_{0,1}+\cdots+m_{j-1,1}+m_{j,\nu}+m_{j+1,1}+\cdots+m_{p,1}
    -(p-1)n \ge2,\\
    m_{j,1}\ne0,\ 1\le j\le p\text{ \ and \ }\nu\ge 2.
   \end{cases}\label{eq:redC3}
\end{align}
Then $S:=\p^{-d}\!\Ad(\p^{-\mu})
\prod_{j=1}^p(x-c_j)^{-m_{j,1}}P \in W[x]$
and the Riemann scheme of $S$ equals
\begin{equation}\label{eq:midR}
 \begin{Bmatrix}
   x = c_0=\infty & c_1 & \cdots & c_p\\
   [1-\mu]_{(m_{0,1}-d)}&[0]_{(m_{1,1}-d)}&\cdots
    &[0]_{(m_{p,1}-d)}\\
   [\lambda_{0,2}-\mu]_{(m_{0,2})} & [\lambda_{1,2}+\mu]_{(m_{1,2})}&\cdots
    &[\lambda_{p,2}+\mu]_{(m_{p,2})}\\
  \vdots & \vdots & \vdots & \vdots\\
    [\lambda_{0,n_0}-\mu]_{(m_{0,n_0})} 
    & [\lambda_{1,n_1}+\mu]_{(m_{1,n_1})}&\cdots
    &[\lambda_{p,n_p}+\mu]_{(m_{p,n_p})}
 \end{Bmatrix}.
\end{equation}
More precisely, the condition \eqref{eq:redC1} and the condition 
\eqref{eq:redC2} for $\nu=1$ assure $S\in W[x]$.
In this case the condition \eqref{eq:redC2} {\rm (resp.~\eqref{eq:redC3} for a 
fixed $j$)} assures that the sets of characteristic exponents of 
$P$ at $x=\infty$ {\rm (resp.~$c_j$)} are equal to the sets given in 
\eqref{eq:midR}, respectively.

Here we have $\RAd(\p^{-\mu})\Red P=S$, if
\begin{equation}
   \begin{cases}
   \lambda_{j,1}+m_{j,1}\text{ \ are not 
    characteristic exponents of $P$}\\
   \quad
    \text{ at  $x=c_j$ for }j=0,\dots,p,\text{ respectively},
\end{cases}\label{eq:redC4}
\end{equation}
and moreover
\begin{equation}
 m_{0,1}=d\text{ \ or \ }
 \lambda_{0,1}\notin\{-d,-d-1,\dots,1-m_{0,1}\}.\label{eq:redC5}
\end{equation}

Using the notation in Definition~\ref{def:coord}, we have 
\begin{equation}\label{eq:redcoord}
 \begin{split}
 S &= \Ad\bigl((x-c_1)^{\lambda_{0,1}-2}\bigr)(x-c_1)^{d}T_{\frac1{x-c_1}}^*
    (-\p)^{-d}\Ad(\p^{-\mu})T_{\frac1x+c_1}^*\\
   &\qquad\cdot
    (x-c_1)^{d}
    \prod_{j=1}^p(x-c_j)^{-m_{j,1}}\Ad\bigl((x-c_1)^{\lambda_{0,1}}\bigr)P
\end{split}
\end{equation}
under the conditions \eqref{eq:redC1} and
\begin{equation}
\begin{cases}
    \lambda_{0,\nu}\notin\{0,-1,-2,\ldots,m_{0,1}-m_{0,\nu}-d+2\}\\
    \text{if \ }m_{0,\nu}+\cdots+m_{p,1}-(p-1)n\ge 2,\ 
     m_{1,1}\ne0\text{ \ and \ }\nu\ge 1.
\end{cases}
\end{equation}

{\rm iii)} Suppose $\ord P>1$ and $P$ is irreducible in {\rm ii)}.
Then the conditions \eqref{eq:redC1}, \eqref{eq:redC2}, \eqref{eq:redC3} 
are valid.
The condition \eqref{eq:redC5} is also valid if $d\ge 1$.

All these conditions in {\rm ii)} are valid if $\#\{j\,;\,m_{j,1}<n\}\ge 2$ and 
$\mathbf m$ is realizable and moreover $\lambda_{j,\nu}$ are generic under 
the Fuchs relation with $\lambda_{j,1}=0$ for $j=1,\dots,p$.

{\rm iv)}
Let $\mathbf m=\bigl(m_{j,\nu}\bigr)_{\substack{j=0,\dots,p\\\nu=1,\dots,n_j}}
\in\mathcal P^{(n)}_{p+1}$. Define $d$ by \eqref{eq:defd}.
Suppose $\lambda_{j,\nu}$ are complex numbers satisfying 
\eqref{eq:midgen}.
Suppose moreover $m_{j,1}\ge d$ for $j=1,\dots,p$.
Defining $\mathbf m'\in\mathcal P^{(n)}_{p+1}$ and $\lambda_{j,\nu}'$ by
\begin{align}
 m'_{j,\nu}&=m_{j,\nu}-\delta_{\nu,1}d\quad(j=0,\dots,p,\ \nu=1,\dots,n_j),\\
 \lambda'_{j,\nu}&=
 \begin{cases}
   2-\lambda_{0,1}&(j=0,\ \nu=1),\\
   \lambda_{j,\nu}-\lambda_{0,1}+1&(j=0,\ \nu>1),\\
   0 &(j>0,\ \nu=1),\\
   \lambda_{j,\nu}+\lambda_{0,1}-1&(j>0,\ \nu>1),
 \end{cases}
\end{align}
we have
\begin{equation}\label{eq:midinv}
 \idx\mathbf m=\idx\mathbf m',\quad
 |\{\lambda_{\mathbf m}\}|=|\{\lambda'_{\mathbf m'}\}|.
\end{equation}
\end{thm}
\begin{proof}
The claim i) is clear from the definition of the Riemann scheme.

{\rm ii)}
Suppose \eqref{eq:redC1}, \eqref{eq:redC2} and \eqref{eq:redC3}.
Then
\begin{equation}
  P' := \Bigl(\prod_{j=1}^p (x-c_j)^{-m_{j,1}}\Bigr)P\in W[x].
\end{equation}
Note that $\Red P=P'$ under the condition \eqref{eq:redC4}.
Put $Q:=\p^{(p-1)n - \sum_{j=1}^pm_{j,1}}P'$.
Here we note that \eqref{eq:redC1} assures 
$(p-1)n - \sum_{j=1}^pm_{j,1}\ge0$.

Fix a positive integer $j$ with $j\le p$.
For simplicity suppose $j=1$ and $c_j=0$.
Since $P'=\sum_{j=0}^n a_j(x)\p^j$ with 
$\deg a_j(x)\le (p-1)n+j-\sum_{j=1}^pm_{j,1}$, we have
\begin{align*}
  x^{m_{1,1}}P' &= 
         \sum_{\ell=0}^Nx^{N-\ell}
            r_\ell(\vartheta)
            \prod_{\substack{1\le \nu\le n_0\\
              0\le i< m_{0,\nu}-\ell}}
            (\vartheta+\lambda_{0,\nu}+i)
\end{align*}
and
\[
  N:=(p-1)n-\sum_{j=2}^pm_{j,1}=m_{0,1}+m_{1,1}-d
\]
with suitable polynomials $r_\ell$ such that $r_0\in\mathbb C^\times$.
Suppose 
\begin{equation}\label{eq:redC0}
   \prod_{\substack{1\le \nu\le n_0\\0\le i< m_{0,\nu}-\ell}}
            (\vartheta+\lambda_{0,\nu}+i)
   \notin x W[x]
   \text{ \ if \ }N-m_{1,1}+1\le \ell\le N.
\end{equation}
Since $P'\in W[x]$, we have
\[
    x^{N-\ell}r_\ell(\vartheta) =
    x^{N-\ell}x^{\ell-N+m_{1,1}}\p^{\ell-N+m_{1,1}}s_\ell(\vartheta)
    \text{ \ if \ }N-m_{1,1}+1\le \ell\le N
\]
for suitable polynomials $s_\ell$. 
Putting $s_\ell=r_\ell$ for $0\le\ell\le N-m_{1,1}$, we have
\begin{equation}\begin{split}
    P'   &=\sum_{\ell=0}^{N-m_{1,1}}
           x^{N-m_{1,1}-\ell}
            s_\ell(\vartheta)
           \prod_{\substack{1\le \nu\le n_0\\0\le i< m_{0,\nu}-\ell}}
           (\vartheta+\lambda_{0,\nu}+i)\\
         &\quad+\sum_{\ell=N-m_{1,1}+1}^N\p^{\ell-N+m_{1,1}}
            s_\ell(\vartheta)
           \prod_{\substack{1\le \nu\le n_0\\0\le i< m_{0,\nu}-\ell}}
           (\vartheta+\lambda_{0,\nu}+i).
\end{split}\label{eq:Pdash}
\end{equation}
Note that $s_0\in\mathbb C^\times$ and 
the condition \eqref{eq:redC0} is equivalent to the condition
$
 \lambda_{0,\nu}+i\ne 0
$
for any $\nu$ and $i$ such that there exists an integer $\ell$
with $0\le i\le m_{0,\nu}-\ell-1$ and $N-m_{1,1}+1\le \ell\le N$.
This condition is valid if \eqref{eq:redC2} is valid, namely, $m_{1,1}=0$ or 
\[
  \lambda_{0,\nu}\notin\{0,-1,\dots,m_{0,1}-m_{0,\nu}-d+2\}
\]
for $\nu$ satisfying $m_{0,\nu}\ge m_{0,1}-d+2$.
Under this condition we have
\begin{align*}
  Q &= \sum_{\ell=0}^N\p^\ell
   s_\ell(\vartheta)
   \prod_{1\le i\le N-m_{1,1}-\ell}(\vartheta+i)\cdot
   \prod_{\substack{1\le \nu\le n_0\\0\le i< m_{0,\nu}-\ell}}
   (\vartheta+\lambda_{0,\nu}+i),
   \allowdisplaybreaks\\
 \Ad(\p^{-\mu})Q&=
   \sum_{\ell=0}^N\p^\ell
            s_\ell(\vartheta-\mu)
            \prod_{1\le i\le N-m_{1,1}-\ell}(\vartheta-\mu+i)
             \\
  &\quad\cdot
            \prod_{1\le i\le m_{0,1}-\ell}(\vartheta+i)
            \cdot
            \prod_{\substack{2\le \nu\le n_0\\0\le i< m_{0,\nu}-\ell}}
            (\vartheta-\mu+\lambda_{0,\nu}+i)
\end{align*}
since $\mu=\lambda_{0,1}-1$.
Hence $\p^{-m_{0,1}}\! \Ad(\p^{-\mu})Q$ equals
\[
\begin{split}
  &\sum_{\ell=0}^{m_{0,1}-1} x^{m_{0,1}-\ell}s_\ell(\vartheta-\mu)
  \prod_{1\le i\le N-m_{1,1}-\ell}(\vartheta-\mu+i)
  \prod_{\substack{2\le \nu\le n_0\\ 0\le i< m_{0,\nu}-\ell}}
  (\vartheta-\mu+\lambda_{0,\nu}+i)\\
  &+\sum_{\ell=m_{0,1}}^N
   \p^{\ell-m_{0,1}}
            s_\ell(\vartheta-\mu)
            \prod_{1\le i\le N-m_{1,1}-\ell}(\vartheta-\mu+i)
            \prod_{\substack{2\le \nu\le n_0\\ 0\le i< m_{0,\nu}-\ell}}
            (\vartheta-\mu+\lambda_{0,\nu}+i)
\end{split}
\]
and then the set of characteristic exponents of this operator
at $\infty$ is \[\{[1-\mu]_{(m_{0,1}-d)},[\lambda_{0,2}-\mu]_{(m_{0,2})},\dots,
[\lambda_{0,n_0}-\mu]_{(m_{0,n_0})}\}.\]
Moreover
$\p^{-m_{0,1}-1}\! \Ad(\p^{-\mu})Q\notin W[x]$ if $\lambda_{0,1}+m_{0,1}$
is not a characteristic exponent of $P$ at $\infty$ and
$-\lambda_{0,1}+1+i\ne m_{0,1}+1$ for $1\le i\le N-m_{1,1}=m_{0,1}-d$,
which assures 
$x^{m_{0,1}}s_0
  \prod_{1\le i\le N-m_{1,1}}(\vartheta-\mu+i)
  \prod_{\substack{2\le \nu\le n_0\\ 0\le i< m_{0,\nu}}}
  (\vartheta-\mu+\lambda_{1,\nu}+i)\notin\p W[x]$.

Similarly we have
\begin{align*}
  P' &= \sum_{\ell=0}^{m_{1,1}}\p^{m_{1,1}-\ell}
            q_\ell(\vartheta)
            \prod_{\substack{2\le \nu\le n_1\\0\le i< m_{1,\nu}-\ell}}
            (\vartheta-\lambda_{1,\nu}-i)\\
         &\quad+\sum_{\ell=m_{1,1}+1}^Nx^{\ell-m_{1,1}}
            q_\ell(\vartheta)
            \prod_{\substack{2\le \nu\le n_1\\0\le i< m_{1,\nu}-\ell}}
            (\vartheta-\lambda_{1,\nu}-i),\allowdisplaybreaks\\
  Q &= \sum_{\ell=0}^{m_{1,1}}
   \p^{N-\ell}q_\ell(\vartheta)
   \prod_{\substack{2\le \nu\le n_1\\0\le i< m_{1,\nu}-\ell}}
   (\vartheta+\lambda_{1,\nu}-i)\\
   &\quad+\sum_{\ell=m_{1,1}+1}^N \p^{N-\ell}q_\ell(\vartheta)
            \prod_{i=1}^{\ell-m_{1,1}}(\vartheta+i)
            \prod_{\substack{2\le \nu\le n_1\\0\le i< m_{1,\nu}-\ell}}
            (\vartheta-\lambda_{1,\nu}-i).\allowdisplaybreaks\\
 \Ad(\p^{-\mu})Q&=
   \sum_{\ell=0}^N\p^{N-\ell}
            q_\ell(\vartheta-\mu)
            \prod_{1\le i\le\ell-m_{1,1}}(\vartheta-\mu+i)\\
 &\quad\cdot
           \prod_{\substack{2\le \nu\le n_1\\0\le i< m_{1,\nu}-\ell}}
           (\vartheta-\mu-\lambda_{1,\nu}-i)
\end{align*}
with $q_0\in\mathbb C^\times$.
Then the set of characteristic exponents of $\p^{-m_{0,1}}\!\Ad(p^{-\mu})Q$
equals
\[
 \{[0]_{(m_{1,1}-d)},[\lambda_{1,2}+\mu]_{(m_{1,2})},\dots,
 [\lambda_{1,n_1}+\mu]_{(m_{1,n_1})}\}
\]
if
\begin{equation*}
  \prod_{\substack{2\le \nu\le n_1\\0\le i< m_{1,\nu}-\ell}}
  (\vartheta-\mu-\lambda_{1,\nu}-i)
\notin \p W[x]
\end{equation*}
for any integers $\ell$ satisfying $0\le\ell\le N$ and $N-\ell<m_{0,1}$.
This condition is satisfied if
\eqref{eq:redC3} is valid, namely, $m_{0,1}=0$ or 
\begin{equation*}
 \begin{split}
 &\lambda_{0,1}+\lambda_{1,\nu}\notin\{0,-1,\dots,m_{1,1}-m_{1,\nu}-d+2\}\\
 &\qquad
 \text{ \ for \ }\nu\ge 2\text{ \ satisfying \ }
 m_{1,\nu}\ge m_{1,1}-d+2
 \end{split}
\end{equation*}
because $m_{1,\nu}-\ell-1\le m_{1,\nu}+m_{0,1}-N-2=m_{1,\nu}-m_{1,1}+d-2$
and the condition $\vartheta-\mu-\lambda_{1,\nu}-i\in \p W[x]$ means 
$-1=\mu+\lambda_{1,\nu}+i=\lambda_{0,1}-1+\lambda_{1,\nu}+i$.

Now we will prove \eqref{eq:redcoord}.
Under the conditions, it follows from \eqref{eq:Pdash} that
\begin{align*}
 \tilde P:\!&=x^{m_{0,1}-N}\Ad\bigl(x^{\lambda_{0,1}}\bigr)\prod_{j=2}^p(x-c_j)^{-{m_{j,1}}}P
\allowdisplaybreaks\\
         &=x^{m_{0,1}+m_{1,1}-N}
            \Ad\bigl(x^{\lambda_{0,1}}\bigr)P'
\allowdisplaybreaks\\
         &=\sum_{\ell=0}^N x^{m_{0,1}-\ell}\Ad\bigl(x^{\lambda_{0,1}}\bigr)
           s_\ell(\vartheta)
           \prod_{0\le\nu<\ell-N+m_{1,1}}\!\!\!(\vartheta-\nu)
           \prod_{\substack{1\le\nu\le n_0\\0\le i<m_{0,\nu}-\ell}}
           (\vartheta+\lambda_{0,\nu}+i),
\allowdisplaybreaks\\
 \tilde Q:\!&=(-\p)^{N-m_{0,1}}T^*_{\frac1x}\tilde P
\allowdisplaybreaks\\
         &=(-\p)^{N-m_{0,1}}\sum_{\ell=0}^N x^{\ell-m_{0,1}}
           s_\ell(-\vartheta-\lambda_{0,1})
           \prod_{0\le\nu<\ell-N+m_{1,1}}\!\!\!(-\vartheta-\lambda_{0,1}-\nu)
   \\&\quad\cdot
           \prod_{\substack{2\le\nu\le n_0\\0\le i<m_{0,\nu}-\ell}}
           (-\vartheta+\lambda_{0,\nu}-\lambda_{0,1}+i)
           \prod_{0\le i\le m_{0,1}-\ell}(-\vartheta+i)
\allowdisplaybreaks\\
   &=\sum_{\ell=0}^N
           (-\p)^{N-\ell}
           s_\ell(-\vartheta-\lambda_{0,1})
          \!\!\!\prod_{1\le i\le \ell-m_{0,1}}(-\vartheta-i)
        \\&\quad\cdot
           \prod_{0\le\nu<\ell-N+m_{1,1}}\!\!\!\!\!(-\vartheta-\lambda_{0,1}-\nu)
           \prod_{\substack{2\le\nu\le n_0\\0\le i<m_{0,\nu}-\ell}}\!\!\!
           (-\vartheta+\lambda_{0,\nu}-\lambda_{0,1}+i)
\end{align*}
and therefore
\begin{align*}
 \Ad(\p^{-\mu})\tilde Q
  &=\sum_{\ell=0}^{N}
           (-\p)^{N-\ell}s_\ell(-\vartheta-1)
          \prod_{1\le i\le \ell-m_{0,1}}\!\!\!
          (-\vartheta+\lambda_{0,1}-1-i)
\\
  &
   \cdot\prod_{0\le\nu<\ell-N+m_{1,1}}\!\!\!\!\!(-\vartheta-1-\nu)
           \prod_{\substack{2\le\nu\le n_0\\0\le i<m_{0,\nu}-\ell}}\!\!\!
           (-\vartheta+\lambda_{0,\nu}-1+i).
\end{align*}
Since
\begin{align*}
  (-\p)^{N-\ell-m_{1,1}}\prod_{0\le\nu<\ell-N+m_{1,1}}
  \!\!\!\!\!(-\vartheta-1-\nu)
 &=\begin{cases}
    x^{\ell-N+m_{1,1}}&(N-\ell<m_{1,1}),\\
    (-\p)^{N-\ell-m_{1,1}}&(N-\ell\ge m_{1,1}),
  \end{cases}
\allowdisplaybreaks\\
 &=x^{\ell-N+m_{1,1}}\prod_{0\le\nu< N-\ell-m_{1,1}}
   (-\vartheta+\nu),
\end{align*}
we have
\begin{align*}
 \tilde Q'&:=(-\p)^{-m_{1,1}}\Ad(\p^{-\mu})
 \tilde Q
=\sum_{\ell=0}^N x^{\ell-N+m_{1,1}}\!\!\!\!\!
          \prod_{0\le\nu< N-\ell-m_{1,1}}\!\!\!\!\!\!\!\!\!(-\vartheta+\nu)\\
        &\qquad\cdot
          s_\ell(-\vartheta-1) \prod_{0\le \nu< \ell-m_{0,1}}\!\!\!\!\!
          (-\vartheta+\lambda_{0,1}-2-\nu)
           \prod_{\substack{2\le\nu\le n_0\\0\le i<m_{0,\nu}-\ell}}\!\!\!
           (-\vartheta+\lambda_{0,\nu}-1+i)
\intertext{and}
&x^{m_{0,1}+m_{1,1}-N}\Ad(x^{\lambda_{0,1}-2})T^*_{\frac1x}
\tilde Q'
=
\sum_{\ell=0}^N x^{m_{0,1}-\ell}\!\!\!\!\!
         \prod_{0\le \nu< \ell-m_{0,1}}\!\!\!\!\!
          (\vartheta-\nu)\cdot
         s_\ell(\vartheta-\lambda_{0,1}+1)\\
        &\quad
          \cdot\prod_{0\le\nu< N-m_{1,1}-\ell}\!\!\!\!\!\!\!\!\!
          (\vartheta-\lambda_{0,1}+2+\nu)
           \prod_{\substack{2\le\nu\le n_0\\0\le i<m_{0,\nu}-\ell}}\!\!\!
           (\vartheta+\lambda_{0,\nu}-\lambda_{0,1}+1+i),
\end{align*}
which equals $\p^{-m_{0,1}}\! \Ad(\p^{-\mu})Q$ 
because $\prod_{0\le \nu<k}(\vartheta-\nu)
=x^k\p^k$ for $k\in\mathbb Z_{\ge 0}$. 

iv) (Cf.~Remark~\ref{rem:KacGRS} ii) for another proof.) \ \ 
Since
\begin{align*}
 \idx\mathbf m-\idx\mathbf m'&=\sum_{j=0}^p m_{j,1}^2 - (p-1)n^2 -
                  \sum_{j=0}^p(m_{j,1}-d)^2+(p-1)(n-d)^2
\allowdisplaybreaks\\
                &=2d\sum_{j=0}^pm_{j,1}-(p+1)d^2-2(p-1)nd+(p-1)d^2
\allowdisplaybreaks\\
                &=d\Bigl(2\sum_{j=0}^pm_{j,1}-2d -2(p-1)n\Bigr)=0
\intertext{and}
 \sum_{j=0}^p\sum_{\nu=1}^{n_j}m_{j,\nu}\lambda_{j,\nu}&
 -\sum_{j=0}^p\sum_{\nu=1}^{n_j}m'_{j,\nu}\lambda'_{j,\nu}
\allowdisplaybreaks\\
 &=m_{0,1}(\mu+1)-(m_{0,1}-d)(1-\mu)+\mu(n-m_{0,1}-\sum_{j=1}^p(n-m_{j,1}))
\allowdisplaybreaks\\
 &=\Bigl(\sum_{j=0}^pm_{j,1} - d - (p-1)n\Bigr)\mu
  -m_{0,1}d-(m_{0,1}-d) =d,
\end{align*}
we have the claim.

The claim iii) follows from the following lemma when $P$ is irreducible.

Suppose $\lambda_{j,\nu}$ are generic in the sense of the claim iii).
Put $\mathbf m=\gcd(\mathbf m)\overline{\mathbf m}$.
Then an irreducible subspace of the solutions of $Pu=0$ has the spectral
type $\ell'\overline{\mathbf m}$ with $1\le \ell'\le \gcd(\mathbf m)$
and the same argument as in the proof of the following lemma
shows iii).
\end{proof}
The following lemma is known which follows from Scott's lemma
(cf.~\S\ref{eq:rigididx}).
\begin{lem}\label{lem:irrred}
Let $P$ be a Fuchsian differential operator with the Riemann scheme
\eqref{eq:GRS}.
Suppose $P$ is irreducible.
Then
\begin{equation}\label{eq:rigididx}
 \idx\mathbf m\le 2.
\end{equation}

Fix $\ell =(\ell_0,\dots,\ell_p)\in\mathbb Z_{>0}^{p+1}$
and suppose $\ord P>1$.
Then
\begin{equation}\label{eq:SL0}
 m_{0,\ell_0}+m_{1,\ell_1}+\cdots+m_{p,\ell_p}-(p-1)\ord\mathbf m \le  m_{k,\ell_k}
 \text{ \ for \ }k=0,\dots,p.
\end{equation}
Moreover the condition
\begin{equation}
 \lambda_{0,\ell_0}+\lambda_{1,\ell_1}+\cdots+\lambda_{p,\ell_p}\in\mathbb Z
\end{equation}
implies
\begin{equation}\label{eq:SL1}
 m_{0,\ell_0}+m_{1,\ell_1}+\cdots+m_{p,\ell_p}\le (p-1)\ord\mathbf m.
\end{equation}\end{lem}
\begin{proof}
Let $M_j$ be the monodromy generators of the solutions of $Pu=0$
at $c_j$, respectively.
Then $\dim Z(M_j)\ge \sum_{\nu=1}^{n_j}m_{j,\nu}^2$ and therefore
$\sum_{j=0}^p\codim Z(M_j)\le (p+1)n^2 -\bigl(\idx\mathbf m +(p-1)n^2)
=2n^2-\idx\mathbf m$. 
Hence Corollary \ref{cor:katz} (cf.~\eqref{eq:iridx}) 
proves \eqref{eq:rigididx}.
\index{00Z@$Z(A)$, $Z(\mathbf M)$}%

We may assume $\ell_j=1$ for $j=0,\dots,p$ and $k=0$ to prove the lemma.
By the map $u(x)\mapsto \prod_{j=1}^p(x-c_j)^{-\lambda_{j,1}}u(x)$
we may moreover assume $\lambda_{j,\ell_j}=0$ for $j=1,\dots,p$.
Suppose $\lambda_{0,1}\in\mathbb Z$.
We may assume $M_p\cdots M_1M_0=I_n$.
Since $\dim\ker M_j\ge m_{j,1}$, 
Scott's lemma (Lemma~\ref{lem:Scott}) assures \eqref{eq:SL1}.

The condition \eqref{eq:SL0} is reduced to \eqref{eq:SL1}
by putting $m_{0,\ell_0}=0$ and $\lambda_{0,\ell_0}=-\lambda_{1,\ell_1}-\cdots
-\lambda_{p,\ell_p}$ because we may assume $k=0$ and $\ell_0=n_0+1$.
\end{proof}

\begin{rem}\label{rem:midisom}
{\rm i) }
Retain the notation in Theorem~\ref{thm:GRSmid}.
The operation in Theorem~\ref{thm:GRSmid} i) corresponds to the \textsl{addition}
and the operation in Theorem~\ref{thm:GRSmid} ii)  corresponds to
Katz's \textsl{middle convolution} (cf.~\cite{Kz}), which are 
studied by \cite{DR} for the systems of Schlesinger canonical form.

The operation $c(P):=\Ad(\p^{-\mu})\p^{(p-1)n}P$ is always well-defined
for the Fuchsian differential operator of the normal form which 
has $p+1$ singular points including $\infty$.
This corresponds to the \textsl{convolution} defined by Katz.
Note that the equation $Sv=0$ is a quotient of the equation $c(P)\tilde u=0$.
\index{convolution}

{\rm ii) }
Retain the notation in the previous theorem.
Suppose the equation $Pu=0$ is irreducible and
$\lambda_{j,\nu}$ are generic complex numbers satisfying 
the assumption in Theorem~\ref{thm:GRSmid}.  
Let $u(x)$ be a local solution of the equation $Pu=0$ corresponding 
to the characteristic exponent $\lambda_{i,\nu}$ at $x=c_i$.
Assume $0\le i\le p$ and $1<\nu\le n_i$.
Then the irreducible equations $\bigl(\Ad\bigl((x-c_j)^r\bigr)P\bigr)u_1=0$ and
$\bigl(\RAd(\p^{-\mu})\circ\Red P\bigr)u_2=0$ are characterized by the equations
satisfied by $u_1(x)=(x-c_j)^ru(x)$ and $u_2(x)=I_{c_i}^\mu(u(x))$, respectively.

Moreover for any integers $k_0,k_1,\dots,k_p$ the irreducible equation $Qu_3=0$
satisfied by $u_3(x)=I_{c_i}^{\mu+k_0}\bigl(\prod_{j=1}^p(x-c_j)^{k_j}u(x)\bigr)$ 
is isomorphic to the equation $\bigl(\RAd(\p^{-\mu})\circ\Red P\bigr)u_2=0$ as 
$W(x)$-modules (cf.~\S\ref{sec:ODE} and \S\ref{sec:contig}).
\end{rem}
\begin{exmp}[exceptional parameters]\label{ex:outuniv}
\index{Jordan-Pochhammer!exceptional parameter}
\index{tuple of partitions!rigid!21,21,21,21}

The Fuchsian differential equation
with the Riemann scheme
\begin{equation*}
 \begin{Bmatrix}
  x=\infty & 0 & 1 & c\\
  [\delta]_{(2)} & [0]_{(2)} & [0]_{(2)}  & [0]_{(2)}\\
  2-\alpha-\beta-\gamma-2\delta & \alpha & \beta & \gamma
 \end{Bmatrix}
\end{equation*}
is a Jordan-Pochhammer equation (cf.~Example~\ref{ex:midconv} ii))
if $\delta\ne0$, which is proved by the reduction using the operation
$\RAd(\p^{1-\delta})\Red$ given in Theorem~\ref{thm:GRSmid} ii).

The Riemann scheme of the operator
\begin{align*}
 P_r&=x(x-1)(x-c)\p^3\\
 &\quad
 -\bigl((\alpha+\beta+\gamma-6)x^2-((\alpha+ \beta-4)c+\alpha+\gamma-4)x
  +(\alpha-2)c \bigr)\p^2\\
&\quad
 -\bigl(2(\alpha+\beta+\gamma-3)x+(\alpha+\beta-2)c+\alpha+\gamma-2+r\bigr)\p
\end{align*}
equals
\begin{equation*}
 \begin{Bmatrix}
  x=\infty & 0 & 1 & c\\
  [0]_{(2)} & [0]_{(2)} & [0]_{(2)}  & [0]_{(2)}\\
  2-\alpha-\beta-\gamma & \alpha & \beta & \gamma
 \end{Bmatrix},
\end{equation*}
which corresponds to a Jordan-Pochhammer operator when $r=0$.
If the parameters are generic, $\RAd(\p)P_r$ 
is Heun's operator \eqref{eq:Heun} with the Riemann scheme
\index{Heun's equation}
\[
\begin{Bmatrix}
 x=\infty & 0 & 1 & c\\
 2& 0 & 0 & 0\\
 3-\alpha-\beta-\gamma&\alpha-1&\beta-1&\gamma-1
\end{Bmatrix},
\]
which contains the accessory parameter $r$.
This transformation doesn't satisfy \eqref{eq:redC2} for $\nu=1$.

The operator $\RAd(\p^{1-\alpha-\beta-\gamma})P_r$ has the Riemann scheme
\[
\begin{Bmatrix}
 x=\infty & 0 & 1 & c\\
 \alpha+\beta+\gamma-1& 0 & 0 & 0\\
 \alpha+\beta+\gamma&1-\beta-\gamma&1-\gamma-\alpha&1-\alpha-\beta
\end{Bmatrix}
\]
and the monodromy generator at $\infty$ is semisimple
if and only if $r=0$.
This transformation doesn't satisfy \eqref{eq:redC2} for $\nu=2$.
\end{exmp}
\begin{defn}\label{def:pell}
Let
\[
  P=a_n(x)\p^n+a_{n-1}(x)\p^{n-1}+\cdots+a_0(x)
\]
be a Fuchsian differential operator with the Riemann scheme 
\eqref{eq:GRS}.
Here some $m_{j,\nu}$ may be 0. 
Fix $\ell=(\ell_0,\dots,\ell_p)\in\mathbb Z^{p+1}_{>0}$ with
$1\le\ell_j\le n_j$.  
Suppose
\begin{equation}\label{eq:redok}
 \#\{j\,;\,m_{j,\ell_j}\ne n\text{ and }0\le j\le p\}\ge 2.
\end{equation}
Put
\index{00dm@$d_\ell(\mathbf m)$}
\begin{equation}
 d_{\mathbf\ell}(\mathbf m):=m_{0,\ell_0}+\cdots+m_{p,\ell_p}
   - (p-1)\ord\mathbf m\label{eq:dm}
\end{equation}
and
\begin{equation}
 \begin{split}
   \p_\ell P:=&\,
   \Ad\bigl(\prod_{j=1}^p(x-c_j)^{\lambda_{j,\ell_j}})
   \prod_{j=1}^p(x-c_j)^{m_{j,\ell_j}-d_\ell(\mathbf m)}\p^{-m_{0,\ell_0}}
   \Ad(\p^{1-\lambda_{0,\ell_0}-\cdots-\lambda_{p,\ell_p}})\\
   &\ 
   \cdot\p^{(p-1)n-m_{1,\ell_1}-\cdots-m_{p,\ell_p}}
   a_n^{-1}(x)\prod_{j=1}^n(x-c_j)^{n-m_{j,\ell_j}}
   \Ad\bigl(\prod_{j=1}^p(x-c_j)^{-\lambda_{j,\ell_j}})
   P.
  \end{split}\label{eq:opred}
\end{equation}
\index{000deltaell@$\p,\ \p_\ell,\ \p_{max}$}
If $\lambda_{j,\nu}$ are generic under 
the Fuchs relation or $P$ is irreducible,
$\p_\ell P$ is well-defined as an element of $W[x]$ and
\begin{align}
 \p_\ell^2 P &= P\text{ \ with $P$ of the form \eqref{eq:FNF}},\label{eq:pellP2}\\
 \begin{split}
   \p_\ell P&\in W(x)\RAd\bigl(\prod_{j=1}^p(x-c_j)^{\lambda_{j,\ell_j}})
  \RAd(\p^{1-\lambda_{0,\ell_0}-\cdots-\lambda_{p,\ell_p}})\\
  &\quad
  \cdot\RAd\bigl(\prod_{j=1}^p(x-c_j)^{-\lambda_{j,\ell_j}})P
 \end{split}\label{eq:defpell}
\end{align}
and $\p_\ell$ gives a correspondence between differential operators 
of normal form \eqref{eq:FNF}.
Here the spectral type $\p_\ell\mathbf m$ 
of $\p_\ell P$ is given by
\begin{align}
 \p_\ell\mathbf m
  &:=\bigr(m'_{j,\nu}\bigr)_{\substack{0\le j\le p\\ 1\le\nu\le n_j}}
  \text{ \ and \ }
  m_{j,\nu}'=
 m_{j,\nu}-\delta_{\ell_j,\nu}\cdot d_{\mathbf\ell}(\mathbf m)
 \label{eq:redm}
\end{align}
and the Riemann scheme of $\p_\ell P$ equals
\begin{equation}\label{eq:pellGRS}
 \p_\ell\bigl\{\lambda_{\mathbf m}\bigr\}
 :=\bigl\{\lambda'_{\mathbf m'}\bigr\}\text{ \ with \ }
 \lambda'_{j,\nu}=
 \begin{cases}
   \lambda_{0,\nu} - 2\mu_\ell&(j=0,\ \nu=\ell_0)\\
   \lambda_{0,\nu} -\mu_\ell&(j=0,\ \nu\ne\ell_0)\\
   \lambda_{j,\nu}          &(1\le j\le p,\ \nu=\ell_j)\\
   \lambda_{0,\nu} +\mu_\ell&(1\le j\le p,\ \nu\ne\ell_j)
 \end{cases}
\end{equation}
by putting
\begin{equation}
  \mu_\ell := \sum_{j=0}^p \lambda_{j,\ell_j} -1.
\end{equation}
It follows from Theorem~\ref{thm:GRSmid} that 
the above assumption is satisfied if 
\begin{equation}\label{eq:non-neg}
  m_{j,\ell_j}\ge d_\ell(\mathbf m)\qquad(j=0,\dots,p)
\end{equation}
and 
\begin{equation}\label{eq:dless2}
\begin{split}
&\sum_{j=0}^p\lambda_{j,\ell_j+(\nu-\ell_j)\delta_{j,k}}
  \notin\bigl\{i\in\mathbb Z\,;\,
 (p-1)n-\sum_{j=0}^p m_{j,\ell_j+(\nu-\ell_j)\delta_{j,k}}+2\le i\le 0\bigr\}\\
&\qquad\text{for }k=0,\dots,p\text{ and }\nu=1,\dots,n_k.
\end{split}
\end{equation}

Note that $\p_\ell \mathbf m\in\mathcal P_{p+1}$ is \textsl{well-defined} 
for a given $\mathbf m\in\mathcal P_{p+1}$ if \eqref{eq:non-neg} is valid.
Moreover we define
\begin{align}
 \p\mathbf m&:=\p_{(1,1,\ldots)}\mathbf m,\\
 \begin{split}
 \p_{max}\mathbf m&:=\p_{\ell_{max}(\mathbf m)}\mathbf m\text{ \ with \ }\\
  \ell_{max}(\mathbf m)_j&:=\min\bigl\{\nu\,;\,m_{j,\nu}
 =\max\{m_{j,1},m_{j,2},\ldots\}\bigr\},
 \end{split}\label{eq:pmax}\\
 d_{max}(\mathbf m)
 &:=\sum_{j=0}^p\max\{m_{j,1},m_{j,2},\dots,m_{j,n_j}\}-(p-1)\ord\mathbf m.
\end{align}
\index{000deltaell@$\p,\ \p_\ell,\ \p_{max}$}
\index{00dmax@$d_{max},\ \ell_{max}$}
For a Fuchsian differential operator $P$ with the Riemann 
scheme \eqref{eq:GRS} we define 
\begin{equation}
 \p_{max}P:=\p_{\ell_{max}(\mathbf m)}P
 \text{ \ and \ }\p_{max}\bigl\{\lambda_{\mathbf m}\bigr\}
 =\p_{\ell_{max}(\mathbf m)}\bigl\{\lambda_{\mathbf m}\bigr\}.
\end{equation}
A tuple $\mathbf m\in\mathcal P$ is called \textsl{basic}
if $\mathbf m$ is indivisible and $d_{max}(\mathbf m)\le 0$.
\index{tuple of partitions!basic}
\end{defn}
\begin{prop}[linear fractional transformation]\label{prop:coordf}
\index{linear fractional transformation}
\index{coordinate transformation!linear fractional}
Let $\phi$ be a linear fractional transformation of $\mathbb P^1(\mathbb C)$,
namely there exists 
$\left(\begin{smallmatrix}\alpha&\beta\\ \gamma&\delta\end{smallmatrix}
\right)\in GL(2,\mathbb C)$
such that $\phi(x)=\frac{\alpha x+\beta}{\gamma x +\delta}$.
Let $P$ be a Fuchsian differential operator with the Riemann scheme \eqref{eq:GRS}.
We may assume $-\frac\delta\gamma=c_j$ with a suitable $j$
by putting $c_{p+1}=-\frac\delta\gamma$, $\lambda_{p+1,1}=0$ and 
$m_{p+1,1}=n$ if necessary.
Fix $\ell=(\ell_0,\dotsm\ell_p)\in\mathbb Z_{>0}^{p+1}$.  
If \eqref{eq:non-neg} and \eqref{eq:dless2} are valid, we have
\begin{equation}
 \begin{split}
 \p_\ell P&\in W(x)\Ad\bigl((\gamma x+\delta)^{2\mu}\bigr)
 T^*_{\phi^{-1}}\p_\ell T^*_{\phi}P,\\
 \mu&=\lambda_{0,\ell_0}+\cdots+\lambda_{p,\ell_p}-1.
 \end{split}
\end{equation}
\end{prop}
\begin{proof}
The claim is clear if $\gamma=0$.
Hence we may assume $\phi(x)=\frac1x$ and the claim follows from 
\eqref{eq:redcoord}.
\end{proof}
\begin{rem}\label{rem:idxFuchs}
{\rm i) \ }
Fix $\lambda_{j,\nu}\in\mathbb C$.
If $P$ has the Riemann scheme $\{\lambda_{\mathbf m}\}$
with $d_{max}(\mathbf m)=1$, $\p_\ell P$ is 
well-defined and $\p_{max}P$ has the Riemann scheme 
$\p_{max}\{\lambda_{\mathbf m}\}$.
This follows from the fact that the conditions \eqref{eq:redC1}, 
\eqref{eq:redC2} and \eqref{eq:redC3} are valid when we apply 
Theorem~\ref{thm:GRSmid} to the operation $\p_{max}:P\mapsto \p_{max}P$. 

{\rm ii) \ }
We remark that
\begin{align}
 \idx\mathbf m &= \idx\p_\ell \mathbf m,\\
 \ord\p_{max}\mathbf m&=\ord\mathbf m-d_{max}(\mathbf m).
\end{align}
Moreover if $\idx\mathbf m > 0$, we have
\begin{equation}
 d_{max}(\mathbf m)> 0
\end{equation}
because of the identity
\begin{equation}\label{eq:idxd}
 \Bigl(\sum_{j=0}^km_{j,\ell_j}-(p-1)\ord\mathbf m\Bigr)
 \cdot\ord\mathbf m
 =  \idx\mathbf m +  \sum_{j=0}^p\sum_{\nu=1}^{n_j}
 (m_{j,\ell_j}-m_{j,\nu})\cdot m_{j,\nu}.
\end{equation}
If $\idx\mathbf m=0$, then $d_{max}(\mathbf m)\ge 0$ and
the condition  $d_{max}(\mathbf m)=0$ implies 
$m_{j,\nu}=m_{j,1}$ for $\nu=2,\dots,n_j$ and $j=0,1,\dots,p$
(cf.~Corollary~\ref{cor:idx0}).

{\rm iii)} The set of indices $\ell_{max}(\mathbf m)$ is defined
in \eqref{eq:pmax} so that it is uniquely determined.
It is sufficient to impose only the condition
\begin{equation}\label{eq:lmaxe}
  m_{j,\ell_{max}(\mathbf m)_j}=\max\{m_{j,1},m_{j,2},\ldots\}
 \qquad(j=0,\dots,p)
\end{equation}
on $\ell_{max}(\mathbf m)$ for the arguments in this paper.
\end{rem}

Thus we have the following result.
\begin{thm}\label{thm:realizable}
A tuple $\mathbf m\in\mathcal P$ is realizable if and only if
$s\mathbf m$ is trivial {\rm (cf.~Definitions~\ref{def:tuples} and 
\ref{def:Sinfty})}
or $\p_{max} \mathbf m$ is well-defined and realizable.
\end{thm}
\begin{proof}
We may assume $\mathbf m\in\mathcal P_{p+1}^{(n)}$ is monotone.

Suppose $\#\{j\,;\,m_{j,1}<n\}<2$.
Then $\p_{max}\mathbf m$ is not well-defined.
We may assume $p=0$ and 
the corresponding equation $Pu=0$ has no singularities in $\mathbb C$
by applying a suitable addition to the equation and then $P\in W(x)\p^n$.
Hence $\mathbf m$ is realizable if and only if 
$\#\{j\,;\,m_{j,1}<n\}=0$, namely, $\mathbf m$ is trivial. 

Suppose $\#\{j\,;\,m_{j,1}<n\}\ge 2$.
Then Theorem~\ref{thm:GRSmid} assures that $\p_{max}\mathbf m$ is 
realizable if and only if $\p_{max}\mathbf m$ is realizable.
\end{proof}

In the next section we will prove that $\mathbf m$ is realizable 
if $d_{max}(\mathbf m)\le 0$.  
Thus we will have a criterion whether a given $\mathbf m\in\mathcal P$ 
is realizable or not by successive applications of $\p_{max}$.

\begin{exmp}
There are examples of successive applications of $s\circ\p$ to monotone 
elements of $\mathcal P$:

$\underline411,\underline411,\underline42,\underline33
\overset{15-2\cdot6=3}\longrightarrow\underline111,\underline111,\underline21
\overset{4-3=1}\longrightarrow{\underline1}1,{\underline1}1,{\underline1}1
\overset{3-2=1}\longrightarrow1,1,1$ (rigid)

$\underline211,\underline211,\underline1111\overset{5-4=1}\longrightarrow
\underline111,\underline111,\underline111\overset{3-3=0}\longrightarrow 111,111,111$ (realizable, not rigid)

$\underline211,\underline211,\underline211,\underline31
\overset{9-8=1}\longrightarrow
\underline111,\underline111,\underline111,\underline21
\overset{5-6=-1}\longrightarrow$ (realizable, not rigid)

${\underline2}2,{\underline2}2,{\underline1}111\overset{5-4=1}\longrightarrow
{\underline2}1,{\underline2}1,{\underline1}11
\overset{5-3=2}\longrightarrow\times$ (not realizablej\\
The numbers on the above arrows are $d_{(1,1,\dots)}(\mathbf m)$.
We sometimes delete the trivial partition as above.
\end{exmp}

The transformation of the generalized Riemann scheme of the application
of $\p_{max}^k$ is described in the following definition.

\begin{defn}[Reduction of Riemann schemes]\label{def:redGRS}
Let $\mathbf m=\bigl(m_{j,\nu}\bigr)_{\substack{j=0,\dots,p\\\nu=1,\dots,n_j}}
\in\mathcal P_{p+1}$ and $\lambda_{j,\nu}\in\mathbb C$ for $j=0,\dots,p$ and
$\nu=1,\dots,n_j$.
Suppose $\mathbf m$ is realizable.
Then there exists a positive integer $K$ such that
\begin{equation}
\begin{split}
 &\ord\mathbf m > \ord\p_{max}\mathbf m > \ord\p_{max}^2\mathbf m > \cdots
  > \ord\p_{max}^K\mathbf m\\
 &\qquad\text{ and }
  s\p_{\max}^K\mathbf m\text{ \ is trivial or \ }
d_{max}\bigl(\p_{max}^K\mathbf m\bigr)\le 0.
\end{split}
\end{equation}
Define $\mathbf m(k)\in\mathcal P_{p+1}$, $\ell(k)\in\mathbb Z$, 
$\mu(k)\in\mathbb C$ and $\lambda(k)_{j,\nu\in\mathbb C}$ for 
$k=0,\dots,K$ by
\index{00mk@$\mathbf m(k),\ \lambda(k),\ \mu(k),\ \ell(k)$}
\begin{align}
 \mathbf m(0)&=\mathbf m\text{ \ and \ }
 \mathbf m(k)=\p_{max}\mathbf m(k-1)\quad(k=1,\dots,K),\\
 \ell(k) &= \ell_{max}\bigl(\mathbf m(k)\bigr)
 \text{ \ and \ }
 d(k) = d_{max}\bigl(\mathbf m(k)\bigr)
 ,\\
 \bigl\{\lambda(k)_{\mathbf m(k)}\bigr\}
  &=\p_{max}^k\bigl\{\lambda_{\mathbf m}\bigr\}
 \text{ \ and \ }\mu(k)=\lambda(k+1)_{1,\nu} - \lambda(k)_{1,\nu}
 \quad(\nu\ne\ell(k)_1).
\end{align}
Namely, we have
\begin{align}
 \lambda(0)_{j,\nu}&=\lambda_{j,\nu}\quad(j=0,\dots,p,\ \nu=1,\dots,n_j),
\allowdisplaybreaks\\
 \mu(k) &= \sum_{j=0}^p\lambda(k)_{j,\ell(k)_j}-1,
\allowdisplaybreaks\\
\begin{split}
 \lambda(k+1)_{j,\nu}
 &=\begin{cases}
     \lambda(k)_{0,\nu}-2\mu(k)&(j=0,\ \nu=\ell(k)_0),\\
     \lambda(k)_{0,\nu}-\mu(k)&(j=0,\ 1\le\nu\le n_0,\ \nu\ne \ell(k)_0),\\
     \lambda(k)_{j,\nu}&(1\le j\le p,\ \nu=\ell(k)_j),\\
     \lambda(k)_{j,\nu}+\mu(k)&(1\le j\le p,\ 1\le\nu\le n_j,\ \nu\ne \ell(k)_j)
   \end{cases}\\
 &=\lambda(k)_{j,\nu}+\bigl((-1)^{\delta_{j,0}}-\delta_{\nu,\ell(k)_j}\bigr)
   \mu(k),
\end{split}\\
\bigl\{\lambda_{\mathbf m}\bigr\}\xrightarrow{\p_{\ell(0)}}&\cdots
\longrightarrow
\bigl\{\lambda(k)_{\mathbf m(k)}\bigr\}\xrightarrow{\p_{\ell(k)}}
\bigl\{\lambda(k+1)_{\mathbf m(k+1)}\bigr\}\xrightarrow{\p_{\ell(k+1)}}\cdots.
\end{align}
\end{defn}

\section{Deligne-Simpson problem}\label{sec:DS}
In this section we give an answer for the existence and the construction
of Fuchsian differential equations with given Riemann schemes and examine
the irreducibility for generic spectral parameters.

\subsection{Fundamental lemmas}
First we prepare two lemmas to construct Fuchsian differential operators
with a given spectral type.
\begin{defn}\label{def:Nnu}
For $\mathbf m=\bigl(m_{j,\nu}\bigr)_{\substack{j=0,\dots,p\\ 1\le \nu\le n_j}}
\in\mathcal P^{(n)}_{p+1}$, we put \index{00Nnu@$N_\nu(\mathbf m)$}
\begin{align}
 \begin{split}
   N_\nu(\mathbf m)&:=(p-1)(\nu+1)+1\\
   &\qquad-\#\{(j,i)\in\mathbb Z^2\,;\,
   i\ge 0,\ 0\le j\le p,\ \widetilde m_{j,i}\ge n-\nu\},\label{eq:Nnu}
 \end{split}\\
 \widetilde m_{j,i}&:=\sum_{\nu=1}^{n_j}\max\bigl\{m_{j,\nu}-i,0\bigr\}.
 \label{eq:tildem}
\end{align}
\end{defn}
See the Young diagram in \eqref{eq:young} and its explanation for
an interpretation of the number $\widetilde m_{j,i}$.

\begin{lem}\label{lem:ex1}
We assume that 
$\mathbf m=\bigl(m_{j,\nu}\bigr)_{\substack{j=0,\dots,p\\ 1\le \nu\le n_j}}
\in\mathcal P^{(n)}_{p+1}$
satisfies
\begin{equation}\label{eq:NTP}
 m_{j,1}\ge m_{j,2}\ge \cdots\ge m_{j,n_j}>0\text{ \ and \ }
 n>m_{0,1}\ge m_{1,1}\ge \cdots\ge m_{p,1}
\end{equation}
and 
\begin{equation}\label{eq:NRed}
  m_{0,1}+\cdots+m_{p,1}\le (p-1)n.
\end{equation}
Then
\begin{equation}\label{eq:LDS}
  N_\nu(\mathbf m)\ge 0
  \qquad(\nu=2,3,\dots,n-1)
\end{equation}
if and only if $\mathbf m$ is not any one of
\begin{equation}\label{eq:MAF}
\begin{gathered}
(k,k;k,k;k,k;k,k),\quad
(k,k,k;k,k,k;k,k,k),\\
(2k,2k;k,k,k,k;k,k,k,k)\\
\text{and \ }
(3k,3k;2k,2k,2k;k,k,k,k,k,k)
\text{ \ with \ }k\ge 2.
\end{gathered}
\end{equation}
\end{lem}
\begin{proof}  Put
\begin{align*}
 \phi_j(t)&:=\sum_{\nu=1}^{n_j}\max\{m_{j,\nu}-t,0\}\text{ \ and \ }
 \bar \phi_j(t):=n\Bigl(1-\frac{t}{m_{j,1}}\Bigr)
 \text{ \ for \ }j=0,\dots,p.
\end{align*}
Then $\phi_j(t)$ and $\bar\phi_j(t)$ are strictly decreasing continuous 
functions of $t\in[0,m_{j,1}]$ and
\begin{align*}
\phi_j(0)&=\bar \phi_j(0)=n,\\ 
\phi_j(m_{j,1})&=\bar \phi_j(m_{j,1})=0,\\ 
2\phi_j(\tfrac{t_1+t_2}2)&\le \phi_j(t_1)+\phi_j(t_2)
&(0\le t_1\le t_2\le m_{j,1}),\\
\phi_j'(t)&=-n_j\le -\tfrac{n}{m_{j,1}}=\bar \phi_j'(t) &(0<t<1).
\end{align*}
Hence we have
\begin{align*}
  \phi_j(t) &= \bar\phi_j(t) &&(0<t<m_{j,1},\ n=m_{j,1}n_j),\\
  \phi_j(t) &< \bar\phi_j(t) &&(0<t<m_{j,1},\ n<m_{j,1}n_j)
\end{align*}
and for $\nu=2,\dots,n-1$
\begin{equation*}
 \begin{split}
\sum_{j=0}^p  \#\{i\in\mathbb Z_{\ge 0}\,;\,
  \phi_j(i)\ge n-\nu\}
 &=\sum_{j=0}^p\bigl[\phi_j^{-1}(n-\nu)+1\bigr]\\
 &\le \sum_{j=0}^p\bigl(\phi_j^{-1}(n-\nu)+1\bigr)\\
 &\le \sum_{j=0}^p\bigl(\bar\phi_j^{-1}(n-\nu)+1\bigr)
 =\sum_{j=0}^p \Bigl(\frac{\nu m_{j,1}}n+1\Bigr)\\
 &\le (p-1)\nu+(p+1)=(p-1)(\nu+1)+2.
 \end{split}
\end{equation*}
Here $[r]$ means the largest integer which is not larger than a real
number $r$.

Suppose there exists $\nu$ with $2\le \nu \le n-1$ such that 
\eqref{eq:LDS} doesn't hold.
Then the equality holds in the above each line, which means
\begin{equation}\label{eq:LDSs}
 \begin{aligned}
  \phi_j^{-1}(n-\nu)&\in\mathbb Z&(j=0,\dots,p),\\
  n&=m_{j,1}n_j&(j=0,\dots,p),\\
  (p-1)n&=m_{0,1}+\cdots+m_{p,1}.
 \end{aligned}
\end{equation}
Note that $n=m_{j,1}n_j$ implies $m_{j,1}=\cdots=m_{j,n_j}=\frac{n}{n_j}$
and $p-1=\frac{1}{n_0}+\cdots+\frac{1}{n_p}\le \frac{p+1}2$.
Hence $p=3$ with $n_0=n_1=n_2=n_3=2$ or
$p=2$ with $1=\frac{1}{n_0}+\frac{1}{n_1}+\frac{1}{n_2}$.
If $p=2$, $\{n_0,n_1,n_2\}$ equals
$\{3,3,3\}$ or $\{2,4,4\}$ or $\{2,3,6\}$.
Thus we have \eqref{eq:MAF} with $k=1,2,\ldots$.
Moreover since
\[
   \phi_j^{-1}(n-\nu)=\bar\phi_j^{-1}(n-\nu)
   =\frac{\nu m_{j,1}}{n}=\frac{\nu}{n_j}\in\mathbb Z
   \quad(j=0,\dots,p),
\]
$\nu$ is a common multiple of $n_0,\dots,n_p$ and thus $k\ge 2$.
If $\nu$ is the least common multiple of $n_0,\dots,n_p$ and $k\ge 2$,
then \eqref{eq:LDSs} is valid and the equality holds in the above each line
and hence \eqref{eq:LDS} is not valid.
\end{proof}
\begin{cor}[Kostov \cite{Ko3}]\label{cor:idx0}%
\index{tuple of partitions!basic!index of rigidity $=0$}
Let $\mathbf m\in\mathcal P$ satisfying
$\idx\mathbf m=0$ and $d_{max}(\mathbf m)\le 0$.
Then $\mathbf m$ is isomorphic to one of the tuples in \eqref{eq:MAF} 
with $k=1,2,3,\ldots$.
\end{cor}
\begin{proof}
Remark~\ref{rem:idxFuchs} assures that $d_{max}(\mathbf m)=0$ and $n=m_{j,1}n_j$.
Then the proof of the final part of Lemma~\ref{lem:ex1} shows the corollary.
\end{proof}
\begin{lem}\label{lem:polydef}
Let $c_0,\dots,c_p$ be $p+1$ distinct points in\/ $\mathbb C\cup\{\infty\}$.
Let $n_0,n_1,\dots,n_p$ be non-negative integers and
let $a_{j,\nu}$ be complex numbers for $j=0,\dots,p$
and $\nu=1,\dots,n_j$.
Put $\tilde n:=n_0+\cdots+n_p$.
Then there exists a unique polynomial $f(x)$ of degree
$\tilde n-1$ such that
\begin{equation}\label{eq:PCond}
 \begin{split}
 f(x)&=a_{j,1}+a_{j,2}(x-c_j)+\cdots
      +a_{j,n_j}(x-c_j)^{n_j-1}\\
    &\quad{}+o(|x-c_j|^{n_j-1})
     \qquad(x\to c_j,\ c_j\ne\infty),\\
 x^{1-\tilde n}f(x)&=a_{j,1} + a_{j,2}x^{-1}
            +a_{j,n_j}x^{1-n_j}+o(|x|^{1-n_j})\\
    &\qquad(x\to \infty,\ c_j=\infty).
 \end{split}
\end{equation}
Moreover the coefficients of $f(x)$ are linear functions of
the $\tilde n$ variables $a_{j,\nu}$.
\end{lem}
\begin{proof}
We may assume $c_p=\infty$ with allowing $n_p=0$.
Put $\tilde n_i=n_0+\cdots+n_{i-1}$ and $\tilde n_0=0$.
For $k=0,\dots,\tilde n-1$ we define
\[
  f_k(x):=
   \begin{cases}
    (x-c_i)^{k-\tilde n_i}\prod_{\nu=0}^{i-1}(x-c_\nu)^{n_\nu}
    &(\tilde n_i\le k<\tilde n_{i+1},\ 0\le i< p),\\
    x^{k-\tilde n_p}\prod_{\nu=0}^{n_{p-1}}(x-c_\nu)^{n_\nu}
    &(\tilde n_p\le k< \tilde n).
   \end{cases}
\]
Since $\deg f_k(x)=k$, the polynomials 
$f_0(x),f_1(x),\dots,f_{\tilde n-1}(x)$ are linearly independent over 
$\mathbb C$.
Put $f(x)=\sum_{k=0}^{\tilde n-1} u_k f_k(x)$ with $c_k\in\mathbb C$ and 
\[
  v_k = \begin{cases}
         a_{i,k-\tilde n_i+1}&(\tilde n_i\le k<\tilde n_{i+1},\ 0\le i< p),\\
         a_{p,\tilde n-k}&(\tilde n_p\le k<\tilde n)
        \end{cases}
\]
by \eqref{eq:PCond}. 
The correspondence which maps the column vectors 
$u:=(u_k)_{k=0,\ldots,\tilde n-1}\in\mathbb C^{\tilde n}$ to
the column vectors $v:=(v_k)_{k=0,\ldots,\tilde n-1}
\in\mathbb C^{\tilde n}$ 
is given by $v=Au$ with a square matrix $A$ of size $\tilde n$.
Then $A$ is an upper triangular matrix of size $\tilde n$ with non-zero 
diagonal entries and therefore the lemma is clear.
\end{proof}
\subsection{Existence theorem}
\begin{defn}[top term]
\index{00Top@$\Top$}
Let 
\[
  P = a_n(x)\tfrac{d^n}{dx^n}+a_{n-1}(x)\tfrac{d^{n-1}}{dx^{n-1}}
     + \cdots+a_1(x)\tfrac{d}{dx} + a_0(x)
\]
be a differential operator with polynomial coefficients.
Suppose $a_n\ne0$.
If $a_n(x)$ is a polynomial of degree $k$ with respect to $x$,
we define $\Top P := a_{n,k}x^k\p^n$ with the coefficient $a_{n,k}$ of the
term $x^k$ of $a_n(x)$.
We put $\Top P=0$ when $P=0$.
\end{defn}
\index{00Nnu@$N_\nu(\mathbf m)$}
\begin{thm}\label{thm:ExF}
Suppose $\mathbf m\in\mathcal P^{(n)}_{p+1}$ 
satisfies \eqref{eq:NTP}.
Retain the notation in Definition~\ref{def:Nnu}.

{\rm i) } We have $N_1(\mathbf m) = p-2$ and
\begin{equation}\label{eq:sumN}
 \sum_{\nu=1}^{n-1} N_\nu(\mathbf m) = \Pidx\mathbf m.
\end{equation}

{\rm ii) }
Suppose $p\ge 2$ and $N_\nu(\mathbf m)\ge 0$ for $\nu=2,\dots,n-1$.
Put
\begin{align}
  q^0_\nu&:=\#\{i\,;\,\widetilde m_{0,i}\ge n-\nu,\ i\ge 0\},\\
  I_{\mathbf m}&:=
  \{(j,\nu)\in\mathbb Z^2\,;\,
    q^0_\nu\le j< q^0_\nu+N_\nu(\mathbf m)\text{ and }
    1\le \nu\le n-1\}.
\end{align}
Then there uniquely exists a 
Fuchsian differential operator $P$ of the normal form 
\eqref{eq:FNF} which has the Riemann scheme \eqref{eq:GRS}
with $c_0=\infty$ under the Fuchs relation \eqref{eq:Fuchs} 
and satisfies
\begin{equation}
 \frac{1}{(\deg P-j-\nu)!}
 \frac{d^{\deg P-j-\nu} a_{n-\nu-1}}{dx^{\deg P-j-\nu}}(0) 
  = g_{j,\nu}\qquad(\forall(j,\nu)\in I_{\mathbf m}).
\end{equation}
Here $\bigl(g_{j,\nu}\bigr)_{(j,\nu)\in I_{\mathbf m}}
 \in\mathbb C^{\Pidx\mathbf m}$ is arbitrarily given.
Moreover the coefficients of $P$ are polynomials of
$x$, $\lambda_{j,\nu}$ and $g_{j,\nu}$ and satisfy
\begin{equation}\label{eq:TopAcc}
 x^{j+\nu}\Top
 \Bigl(\frac{\p P}{\p g_{j,\nu}}\Bigr)\p^{\nu+1}=\Top P\text{ \ and \ }
 \frac{\p^2 P}{\p g_{j,\nu}^2}=0.
\end{equation}

Fix the characteristic exponents $\lambda_{j,\nu}\in\mathbb C$ 
satisfying the Fuchs relation.
Then all the Fuchsian differential operators of the normal form 
with the Riemann scheme \eqref{eq:GRS} are parametrized by
$(g_{j,\nu})\in\mathbb C^{\Pidx \mathbf m}$.
Hence the operators are unique if and only if\/ $\Pidx\mathbf m=0$.
\end{thm}
\begin{proof}
i) 
Since $\widetilde m_{j,1}=n-n_j\le n-2$, 
$N_1(\mathbf m)=2(p-1)+1-(p+1)=p-2$ and
\begin{align*}
  \sum_{\nu=1}^{n-1}
  &\#\{(j,i)\in\mathbb Z^2\,;\,i\ge 0,\ 0\le j\le p,\ 
  \widetilde m_{j,i}\ge n-\nu\}
\allowdisplaybreaks\\
 &=\sum_{j=0}^p\Bigl(\sum_{\nu=0}^{n-1}\#\{i\in\mathbb Z_{\ge 0}\,;\,
   \widetilde m_{j,i}\ge n-\nu\}-1\Bigr)
\allowdisplaybreaks\\
 &=\sum_{j=0}^p\Bigl(\sum_{i=0}^{m_{j,1}}\widetilde m_{j,i} -1\Bigr) =
   \sum_{j=0}^p\Bigl(\sum_{i=0}^{m_{j,1}}\sum_{\nu=1}^{n_j}
   \max\{m_{j,\nu}-i,0\}-1\Bigr)
\allowdisplaybreaks\\
 &=\sum_{j=0}^p
 \Bigl(\sum_{\nu=1}^{n_j}\frac{m_{j,\nu}(m_{j,\nu}+1)}{2}-1\Bigr)\\
 &=\frac12\Bigl(\sum_{j=0}^p\sum_{\nu=1}^{n_j}m_{j,\nu}^2 + (p+1)(n-2) \Bigr),
\allowdisplaybreaks\\
 \sum_{\nu=1}^{n-1}N_\nu(\mathbf m) &=
 (p-1)\Bigl(\frac{n(n+1)}2-1\Bigr)+(n-1)
 -\frac12\Bigl(\sum_{j=0}^p\sum_{\nu=1}^{n_j}m_{j,\nu}^2 +(p+1)(n-2)\Bigr)
\allowdisplaybreaks\\
 &= \frac12\Bigl((p-1)n^2+2-\sum_{j=0}^p\sum_{\nu=1}^{n_j}m_{j,\nu}^2\Bigr)
 = \Pidx\mathbf m.
\end{align*}

{\rm ii) }  Put
\begin{align*}
 P &= \sum_{\ell=0}^{pn}x^{pn-\ell}p^P_{0,\ell}(\vartheta)\\
   &= \sum_{\ell=0}^{pn}(x-c_j)^\ell p^P_{j,\ell}\bigl(
     (x-c_j)\p\bigr)\qquad(1\le j\le n),\\
  h_{j,\ell}(t):\!&=
 \begin{cases}
     \prod_{\nu=1}^{n_0}\prod_{0\le i<m_{0,\nu}-\ell}
     \bigl(t+\lambda_{0,\nu}+i\bigr)&(j=0),\\
     \prod_{\nu=1}^{n_j}\prod_{0\le i<m_{j,\nu}-\ell}
     \bigl(t-\lambda_{j,\nu}-i\bigr)
     &(1\le j\le p),
 \end{cases}\\
 p^P_{j,\ell}(t)&=q^P_{j,\ell}(t)h_{j,\ell}(t)+r^P_{j,\ell}(t)
 \qquad(\deg r^P_{j,\ell}(t) < \deg h_{j,\ell}(t)).
\end{align*}
Here $p^P_{j,\ell}(t)$, $q^P_{j,\ell}(t)$, 
$r^P_{j,\ell}(t)$ and $h_{j,\ell}(t)$
are polynomials of $t$ and
\begin{equation}
  \deg h_{j,\ell} = \sum_{\nu=1}^{n_j}\max\{m_{j,\nu}-\ell,0\}.
\end{equation}
The condition that $P$ of the form \eqref{eq:FNF} have the Riemann 
scheme \eqref{eq:GRS}
if and only if $r^P_{j,\ell}=0$ for any $j$ and $\ell$.
Note that $a_{n-k}(x)\in\mathbb C[x]$ should satisfy
\begin{equation}
  \deg a_{n-k}(x)\le pn-k\text{ \ and \ }
  a_{n-k}^{(\nu)}(c_j)=0\quad(0\le\nu\le n-k-1,\ 1\le k\le n),
\end{equation}
which is equivalent to the condition that $P$ is of the Fuchsian type.

Put $P(k):=\bigl(\prod_{j=1}^p(x-c_j)^n\bigr)\frac{d^n}{dx^n}+a_{n-1}(x)
 \frac{d^{n-1}}{dx^{n-1}}+\cdots
 +a_{n-k}(x)\frac{d^{n-k}}{dx^{n-k}}$.

Assume that $a_{n-1}(x),\ldots,a_{n-k+1}(x)$ have already defined so that
$\deg r_{j,\ell}^{P(k-1)}<n-k+1$ and we will define $a_{n-k}(x)$ so that
$\deg r_{j,\ell}^{P(k)}<n-k$.

When $k=1$, we put
\[
  a_{n-1}(x) = -a_n(x)\sum_{j=1}^p(x-c_j)^{-1}
  \sum_{\nu=1}^{n_j}\sum_{i=0}^{m_{j,\nu}-1}(\lambda_{j,\nu}+i)
\]
and then we have $\deg r_{j,\ell}^{P(1)}<n-1$ for $j=1,\dots,p$.
Moreover we have $\deg r_{0,\ell}^{P(1)}<n-1$ because of the Fuchs relation.

Suppose $k\ge 2$ and put
\[
 a_{n-k}(x) =
   \begin{cases}
       \sum_{\ell\ge 0}c_{0,k,\ell}x^{pn-k-\ell},\\
       \sum_{\ell\ge 0}c_{j,k,\ell}(x-c_j)^{n-k+\ell} &(j=1,\dots,p)
  \end{cases}
\]
with $c_{i,j,\ell}\in\mathbb C$.
Note that
\begin{align*}
 a_{n-k}(x)\p^{n-k}
 &=\sum_{\ell\ge 0}c_{0,k,\ell}x^{(p-1)n-\ell}\prod_{i=0}^{n-k-1}
 (\vartheta-i)\\
 &=\sum_{\ell\ge0}c_{j,k,\ell}(x-c_j)^{\ell}\prod_{i=0}^{n-k-1}
 \bigl((x-c_j)\p-i\bigr).
\end{align*}
Then $\deg r_{j,\ell}^{P(k)}<n-k$ if and only if $\deg h_{j,\ell}\le n-k$ or
\begin{equation}\label{eq:cfc1}
  c_{j,k,\ell} = -\frac1{(n-k)!}
   \Bigl(\frac{d^{n-k}}{dt^{n-k}}r_{j,\ell}^{P(k-1)}(t)\Bigr)\Bigl|_{t=0}.
\end{equation}
Namely, we impose the condition \eqref{eq:cfc1} for all $(j,\ell)$
satisfying
\[
  {\widetilde m}_{j,\ell}=\sum_{\nu=1}^{n_j}\max\{m_{j,\nu}-\ell,0\} > n-k.
\]
The number of the pairs $(j,\ell)$ satisfying this condition equals
$(p-1)k+1- N_{k-1}(\mathbf m)$.  
Together with the conditions
$a_{n-k}^{(\nu)}(c_j)=0$ for $j=1,\dots,p$ and $\nu=0,\dots,n-k-1$,
the total number of conditions imposing to 
the polynomial $a_{n-k}(x)$ of degree $pn-k$ equals
\[
  p(n-k)+(p-1)k+1- N_{k-1}(\mathbf m)=(pn-k+1)-N_{k-1}(\mathbf m).
\]
Hence Lemma~\ref{lem:polydef} shows that $a_{n-k}(x)$ is uniquely defined
by giving $c_{0,k,\ell}$ arbitrarily for 
$q_{k-1}^0\le\ell< q_{k-1}^0+N_{k-1}(\mathbf m)$
because $q_{k-1}^0=\#\{\ell\ge0\,;\,\widetilde m_{0,\ell}>n-k\}$.
Thus we have the theorem.
\end{proof}
\begin{rem}  
The numbers $N_\nu(\mathbf m)$ don't change if we replace a 
$(p+1)$-tuple $\mathbf m$ of partitions of $n$  by
the $(p+2)$-tuple of partitions of $n$ defined 
by adding a trivial partition $n=n$ of $n$ to $\mathbf m$.
\end{rem}
\index{00Nnu@$N_\nu(\mathbf m)$}
\begin{exmp}\label{ex:Nj}
We will examine the number $N_\nu(\mathbf m)$ in Theorem~\ref{thm:ExF}.
In the case of the Simpson's list (cf.~\S\ref{sec:rigidEx}) 
we have the following.
\begin{align*}
 \mathbf m&= n-11,1^n,1^n\tag{$H_n$: hypergeometric family}\\
    \widetilde{\mathbf m} &= n,n-2,n-3,\ldots1;n;n
\allowdisplaybreaks\\
 \mathbf m&=mm,mm-11,1^{2m}\tag{$EO_{2m}$: even family}\\
    \widetilde{\mathbf m} &= 2m,2m-2,\dots,2;2m,2m-3,\dots,1;2m
\allowdisplaybreaks\\
 \mathbf m&= m+1m,mm1,1^{2m+1}\tag{$EO_{2m+1}$: odd family}\\
    \widetilde{\mathbf m} &=2m+1,2m-1,\dots,1;2m+1,2m-2,\dots,2;2m+1
\allowdisplaybreaks\\
 \mathbf m&= 42,222,1^6\tag{$X_6$: extra case}\\
    \widetilde{\mathbf m} &=6,4,2,1;6,3;6
\end{align*}
In these cases $p=2$ and we have $N_\nu(\mathbf m)=0$ for $\nu=1,2,\dots,n-1$ 
because
\begin{equation}\label{eq:tildebm}
 \begin{split}
 \widetilde{\mathbf m}
 :\!&=\{\widetilde m_{j,\nu}\,;\,\nu=0,\dots,m_{j,1}-1,\ j=0,\dots,p\bigr\}\\
 &= \{n,n,n,n-2,n-3,n-4,\dots,2,1\}.
 \end{split}
\end{equation}
\index{00Hn@$H_n$}\index{00EOn@$EO_n$}\index{00X6@$X_6$}%
See Proposition~\ref{prop:sred} ii) for the condition that 
$N_\nu(\mathbf m)\ge 0$ for $\nu=1,\dots,\ord\mathbf m-1$.

We give other examples:
\smallskip

\begin{tabular}{|c|c|l|l|}
\hline
$\hfill\mathbf m\hfill$ & $\Pidx$ & 
$\hfill\widetilde{\mathbf m}\hfill$ 
 & $N_1,N_2,\dots,N_{\ord\mathbf m-1}$
\\ \hline\hline
$221,221,221$  &0& $52,52,52$ & $0,1,-1,0$ \\ \hline
$21,21,21,21\,(P_3)$  &0& $31,31,31,31$ & $1,-1$ \\ \hline
$22,22,22$     &$-3$& $42,42,42$ & $0,-2,-1$ \\ \hline
$11,11,11,11\,(\tilde D_4)$  &1& $2,2,2,2$& $1$ \\ \hline
$111,111,111\,(\tilde E_6)$  &1& $3,3,3$ & $0,1$ \\ \hline
$22,1111,1111\,(\tilde E_7)$ &1& $42,4,4$ & $0,0,1$ \\ \hline
$33,222,111111\,(\tilde E_8)$ &1& $642,63,6$ & $0,0,0,0,1$ \\ \hline
$21,21,21,111$ &1& $31,31,31,3$ & $1,0$ \\ \hline
$222,222,222$ &1& $63,63,63$ & $0,1,-1,0,1$ \\ \hline
$11,11,11,11,11$ &2& $2,2,2,2,2$ & $2$ \\ \hline
$55,3331,22222$ & 2 & $10,8,6,4,2;10,6,3;10,5$ &
$0,0,1,0,0,0,0,0,1$ \\ \hline
$22,22,22,211$ &2& $42,42,42,41$ & $1,0,1$ \\ \hline
$22,22,22,22,22$ &5& $42,42,42,42,42$ & $2,0,3$ \\ \hline
$32111,3221,2222$ &8& $831,841,84$ & $0,1,2,1,1,2,1$ \\ \hline
\end{tabular}
\end{exmp}
Note that if $\Pidx \mathbf m=0$, in particular, if $\mathbf m$ is rigid,
then $\mathbf m$ doesn't satisfy \eqref{eq:NRed}.
The tuple $222,222,222$ of partitions is 
the second case in \eqref{eq:MAF} with $k=2$.
\begin{rem}\label{rem:bas0}
Note that 
\cite[Proposition~8.1]{O3} proves that there exit only finite basic tuples
of partitions with a fixed index of rigidity.  

\index{tuple of partitions!basic!index of rigidity $=0$}%
Those with index of rigidity $0$ are of only 4 types, which are $\tilde D_4$, 
$\tilde E_6$, $\tilde E_7$ and $\tilde E_8$ given in the above 
(cf.~Corollary~\ref{cor:idx0}, Kostov \cite{Ko3}).
Namely, those are in the $S_\infty$-orbit of
\begin{equation}\label{eq:basic0}
  \{11,11,11,11\quad111,111,111\quad22,1111,1111\quad33,222,111111\}
\end{equation}
and the operator $P$ in Theorem~\ref{thm:ExF} with any one of this spectral 
type has one accessory parameter in its $0$-th order term.

The equation corresponding to $11,11,11,11$ is called Heun's equation
(cf.~\cite{SW,WW}),
which is given by the operator
\index{Heun's equation}
\begin{equation}\label{eq:Heun}
 \begin{split}
 P_{\alpha,\beta,\gamma,\delta,\lambda} &= x(x-1)(x-c)\p^2
  +\bigl(\gamma(x-1)(x-c)+\delta x(x-c)\\
   &\quad{}+(\alpha+\beta+1-\gamma-\delta)x(x-1)\bigr)\p
  +\alpha\beta x-\lambda
 \end{split}
\end{equation}
with the Riemann scheme
\begin{equation}
 \begin{Bmatrix}
   x = 0 & 1 & c & \infty\\
    0 & 0 & 0 & \alpha &;\,x\\
   1-\gamma & 1-\delta & \gamma+\delta-\alpha-\beta & \beta&;\,\lambda
 \end{Bmatrix}.
\end{equation}
Here $\lambda$ is an accessory parameter.
Our operation cannot decrease the order of 
$P_{\alpha,\beta,\gamma,\delta,\lambda}$ but gives the following 
transformation.
\begin{equation}
 \begin{split}
 &\Ad(\p^{1-\alpha})P_{\alpha,\beta,\gamma,\delta,\lambda}
 =P_{\alpha',\beta',\gamma',\delta',\lambda'},\\
 &\begin{cases}
  \alpha'=2-\alpha,\ \beta'=\beta-\alpha+1,\ \gamma'=\gamma-\alpha+1,
  \ \delta'=\delta-\alpha+1,\\
  \lambda'=\lambda+(1-\alpha)\bigl(\beta-\delta+1+(\gamma+\delta-\alpha)c\bigr).
 \end{cases}
 \end{split}
\end{equation}
\end{rem}
\index{tuple of partitions!basic!index of rigidity $=-2$}
\begin{prop}{\rm(\cite[Proposition~8.4]{O3})}.\label{prop:bas2}
The basic tuples of partitions with index of rigidity $-2$ are in
the $S_\infty$-orbit of the set of the $13$ tuples
\begin{align*}
\bigl\{&
    11,11,11,11,11
\ \ 21,21,111,111
\ \ 31,22,22,1111
\ \ 22,22,22,211
\\&
   211,1111,1111
\ \ \ 221,221,11111
\ \ 32,11111,11111
\ \ 222,222,2211
\\&
 33,2211,111111
\ \ 44,2222,22211
\ \ 44,332,11111111
\ \ 55,3331,22222
\\&
    66,444,2222211
\bigr\}.
\end{align*}
\end{prop}
\begin{proof}
Here we give the proof in \cite{O3}.

Assume that $\mathbf m\in\mathcal P_{p+1}$ is
basic and monotone and $\idx\mathbf m=-2$.
Note that \eqref{eq:idxd} shows
\[
   0\le \sum_{j=0}^p\sum_{\nu=2}^{n_j}
 (m_{j,1}-m_{j,\nu})\cdot m_{j,\nu}\le -\idx\mathbf m=2.
\]
Hence \eqref{eq:idxd} implies
$\sum_{j=0}^p\sum_{\nu=2}^{n_j} (m_{j,1}-m_{j,\nu})m_{j,\nu}=0$ or $2$ 
and we have only to examine the following 5 possibilities.

(A) \ $m_{0,1}\cdots m_{0,n_0} = 2\cdots211$
      and $m_{j,1}=m_{j,n_j}$ for $1\le j\le p$.

(B) \ $m_{0,1}\cdots m_{0,n_0} = 3\cdots31$
    and $m_{j,1}=m_{j,n_j}$ for $1\le j\le p$.

(C) \ $m_{0,1}\cdots m_{0,n_0} = 3\cdots32$
    and $m_{j,1}=m_{j,n_j}$ for $1\le j\le p$.

(D) \ $m_{i,1}\cdots m_{i,n_0} = 2\cdots21$
     and $m_{j,1}=m_{j,n_j}$ 
     for $0\le i\le 1<j\le p$.

(E) \ $m_{j,1}=m_{j,n_j}$ for $0\le j\le p$ and $\ord\mathbf m=2$.

\smallskip
Case (A). \ 
If $2\cdots211$ is replaced by $2\cdots22$, $\mathbf m$
is transformed into $\mathbf m'$ with $\idx\mathbf m'=0$.
If $\mathbf m'$ is indivisible, $\mathbf m'$ is basic and
$\idx\mathbf m'=0$ and therefore
$\mathbf m$ is $211,1^4,1^4$ or $33,2211,1^6$.
If $\mathbf m'$ is not indivisible, $\frac12\mathbf m'$ is basic and
$\idx\frac12\mathbf m'=0$ and hence
$\mathbf m$ is one of the tuples in
\[\{211,22,22,22\ \ 2211,222,222\ \ 
    22211,2222,44\ \ 2222211,444,66\}.\]

Put $m=n_0-1$ and examine the identity
\[
\sum_{j=0}^p\frac{m_{j,1}}{\ord\mathbf m}
=p-1 + (\ord\mathbf m)^{-2}\Bigl(\idx\mathbf m
  +  \sum_{j=0}^p\sum_{\nu=1}^{n_j}(m_{j,1}-m_{j,\nu})m_{j,\nu}\Bigr)
\]

\smallskip
Case (B). \ 
Note that $\ord\mathbf m=3m+1$ and therefore
$\frac3{3m+1}+\frac1{n_1}+\cdots+\frac1{n_p}
=p-1$.
Since $n_j\ge 2$, we have $\frac12p -1\le\frac{3}{3m+1}<1$
and $p\le 3$.

If $p=3$, we have $m=1$, $\ord\mathbf m=4$,
$\frac1{n_1}+\frac1{n_2}+\frac1{n_3}=\frac54$,
$\{n_1,n_2,n_3\}=\{2,2,4\}$ and 
$\mathbf m=31,22,22,1111$.

Assume $p=2$.  Then
$\frac1{n_1}+\frac1{n_2}=1-\frac{3}{3m+1}$.
If $\min\{n_1,n_2\}\ge 3$, $\frac1{n_1}+\frac1{n_2}\le\frac23$
and $m\le2$.
If $\min\{n_1,n_2\}=2$, $\max\{n_1,n_2\}\ge 3$ and $\frac3{3m+1}\ge\frac16$ 
and $m\le 5$.
Note that
$\frac1{n_1}+\frac1{n_2}=\frac{13}{16}$, $\frac{10}{13}$, $\frac{7}{10}$,
$\frac{4}{7}$ and $\frac{1}{4}$
according to $m=5$, $4$, $3$, $2$ and $1$, respectively.
Hence we have $m=3$, $\{n_1,n_2\}=\{2,5\}$ and $\mathbf m=3331,55,22222$.

\smallskip
Case (C). \ 
We have $\frac{3}{3m+2}+\frac1{n_1}+\cdots+\frac1{n_p}=p-1$.
Since $n_j\ge 2$, $\frac12p-1\le \frac{3}{3m+2}<1$
and $p\le 3$.
If $p=3$, then $m=1$, $\ord\mathbf m=5$ and 
$\frac1{n_1}+\frac1{n_2}+\frac1{n_3}=\frac75$, which never occurs.

Thus we have $p=2$, 
$\frac1{n_1}+\frac1{n_2}=1-\frac{3}{3m+2}$ and 
hence $m\le 5$ as in Case (B).
Then $\frac1{n_1}+\frac1{n_2}=\frac{14}{17}$, 
$\frac{11}{14}$, $\frac{8}{11}$, $\frac{5}{8}$ 
and $\frac{2}{5}$ according to $m=5$, $4$, $3$, $2$ and $1$, respectively.
Hence we have $m=1$ and $n_1=n_2=5$ and
$\mathbf m=32,11111,11111$
or $m=2$ and $n_1=2$ and $n_2=8$ and $\mathbf m=332,44,11111111$.

\smallskip
Case (D). \ 
We have $\frac2{2m+1}+\frac2{2m+1}+\frac1{n_2}+\cdots+\frac1{n_p}=p-1$.
Since $n_j\ge3$ for $j\ge2$, we have
$p-1\le \frac32\frac{4}{2m+1}=\frac{6}{2m+1}$ and $m\le2$.
If $m=1$, then $p=3$ and $\frac1{n_2}+\frac1{n_3}=2-\frac43=\frac23$ and 
we have $\mathbf m=21,21,111,111$.
If $m=2$, then $p=2$, $\frac1{n_2}=1-\frac{4}{5}$ and
$\mathbf m=221,221,11111$.

\smallskip
Case (E). \ 
Since $m_{j,1}=1$ and \eqref{eq:idxd} means $-2=\sum_{j=0}^p 2m_{j,1}-4(p-1)$, 
we have $p=4$ and $\mathbf m=11,11,11,11,11$.
\end{proof}
\subsection{Divisible spectral types}
\begin{prop}\label{prop:dividx0}\index{tuple of partitions!divisible}
Let $\mathbf m$ be any one of the partition of
type $\tilde D_4$, $\tilde E_6$, $\tilde E_7$ or $\tilde E_8$
in Example~\ref{ex:Nj} and put $n=\ord\mathbf m$.
Then $k\mathbf m$ is realizable but it isn't 
irreducibly realizable for $k=2,3,\ldots$.
Moreover we have the operator $P$ of order $k\ord\mathbf m$
satisfying the properties in\/ {\rm Theorem~\ref{thm:ExF} ii)} for the 
tuple $k\mathbf m$.
\end{prop}
\begin{proof}
Let $P(k,c)$ be the operator of the normal form with the Riemann scheme
\[
  \begin{Bmatrix}
     x=c_0=\infty &              x=c_j\ (j=1,\dots,p)\\
     [\lambda_{0,1}-k(p-1)n+km_{0,1}]_{(m_{0,1})} 
       & [\lambda_{j,1}+km_{j,1}]_{(m_{j,1})}\\
     \vdots&\vdots\\
     [\lambda_{0,n_1}-k(p-1)n+km_{0,1}]_{(m_{0,n_1})} 
       & [\lambda_{j,n_j}+km_{j,n_j}]_{(m_{j,n_j})}
  \end{Bmatrix}
\]
of type $\mathbf m$.
Here $\mathbf m=\bigl(m_{j,\nu}\bigr)_{\substack{j=0,\dots,p\\ \nu=1,\dots,n_j}}$, 
$n=\ord\mathbf m$ and 
$c$ is the accessory parameter contained in the coefficient of the
0-th order term of $P(k,c)$.
Since $\Pidx\mathbf m=0$ means
\begin{align*}
  \sum_{j=0}^{p}\sum_{\nu=1}^{n_j} m_{j,\nu}^2
  =(p-1)n^2=\sum_{\nu=0}^{n_0}(p-1)nm_{0,\nu},
\end{align*}
the Fuchs relation \eqref{eq:Fuchs} is valid for any $k$.
Then it follows from Lemma~\ref{lem:block} that the  Riemann scheme of 
the operator $P_k(c_1,\dots,c_k)=P(k-1,c_k)P(k-2,c_{k-1})\cdots P(0,c_1)$
equals
\begin{equation}\label{eq:MAFGRS}
  \begin{Bmatrix}
     x=c_0=\infty &              x=c_j\ (j=1,\dots,p)\\
     [\lambda_{0,1}]_{(km_{0,1})} 
       & [\lambda_{j,1}]_{(km_{j,1})}\\
     \vdots&\vdots\\
     [\lambda_{0,n_1}]_{(km_{0,n_1})} 
       & [\lambda_{j,n_j}]_{(km_{j,n_j})}
  \end{Bmatrix}
\end{equation}
and it contain an independent accessory parameters in the coefficient
of $\nu n$-th order term of $P_k(c_1,\dots,c_k)$ for $\nu=0,\dots,k-1$
because for the proof of this statement we may assume $\lambda_{j,\nu}$ 
are generic under the Fuchs relation.

Note that
\[
  N_\nu(k\mathbf m) = \begin{cases}
          1 & (\nu\equiv n-1 \mod n),\\
         -1 & (\nu\equiv 0 \mod n),\\
          0 & (\nu\not\equiv 0,\ n-1\mod n)
        \end{cases}
\]
for $\nu=1,\dots,kn-1$ because
\[
  \widetilde{k\mathbf m}=
 \begin{cases}
  \{2i,2i,2i,2i\,;\,i=1,2,\ldots,k\}\quad\text{if }\mathbf m\text{ is of type }
    \tilde D_4,\\
  \{ni,ni,ni,ni-2,ni-3,\ldots,ni-n+1\,;\,i=1,2,\ldots,k\}\\
  \qquad\qquad\qquad\qquad\qquad\qquad\quad\ 
    \text{if }\mathbf m \text{ is of type }
    \tilde E_6,\,\tilde E_7\text{ or }\tilde E_8
 \end{cases}
\]
under the notation \eqref{eq:tildem} and \eqref{eq:tildebm}.
Then the operator
$P_k(c_1,\dots,c_k)$ shows that when we inductively 
determine the coefficients of the operator with the Riemann scheme 
\eqref{eq:MAFGRS} as in the proof of Theorem~\ref{thm:ExF}, 
we have a new accessory parameter in the coefficient 
of the $\bigl((k-j)n\bigr)$-th order term and then the conditions for the 
coefficients of the
$\bigl((k-j)n-1\bigr)$-th order term are overdetermined but 
they are automatically compatible for $j=1,\dots,k-1$.

Thus we can conclude that the operators of the normal form 
with the Riemann scheme \eqref{eq:MAFGRS} are $P_k(c_1,\dots,c_k)$,
which are always reducible.
\end{proof}
\begin{prop}\label{prop:irred}
Let $k$ be a positive integer and let\/ $\mathbf m$ be an indivisible 
$(p+1)$-tuple of partitions of $n$.  
Suppose $k\mathbf m$ is realizable and\/
$\idx\mathbf m<0$.
Then any Fuchsian differential equation with the Riemann scheme 
\eqref{eq:MAFGRS} is always irreducible if $\lambda_{j,\nu}$ is 
generic under the Fuchs relation
\begin{equation}\label{eq:FCdiv}
  \sum_{j=0}^p\sum_{\nu=1}^{n_j} m_{j,\nu}\lambda_{j,\nu}
  = \ord \mathbf m - k\frac{\idx\mathbf m}2.
\end{equation}
\end{prop}
\begin{proof}
The above Fuchs relation follows from \eqref{eq:Fuchidx}
with the identities $\ord k\mathbf m=k\ord\mathbf m$ and
$\idx k\mathbf m=k^2\idx\mathbf m$.

Suppose $Pu=0$ is reducible.
Then Remark~\ref{rem:generic} ii) says that there exist 
$\mathbf m'$, $\mathbf m''\in\mathcal P$ such that 
$k\mathbf m=\mathbf m'+\mathbf m''$ and 
$0<\ord\mathbf m'<k\ord\mathbf m$ and 
$|\{\lambda_{\mathbf m'}\}|\in\{0,-1,-2,\ldots\}$.
Suppose $\lambda_{j,\nu}$ are generic under \eqref{eq:FCdiv}.
Then the condition $|\{\lambda_{\mathbf m'}\}|\in\mathbb Z$ implies
$\mathbf m'=\ell\mathbf m$ with a positive integer satisfying $\ell<k$
and
\begin{align*}
 |\{\lambda_{\ell\mathbf m}\}|
  &=\sum_{j=0}^p\sum_{\nu=1}^{n_j}\ell m_{j,\nu}\lambda_{j,\nu}-\ord\ell\mathbf m
    +\ell^2\idx\mathbf m\\
 &= \ell\Bigl(\ord\mathbf m - k\frac{\idx\mathbf m}2\Bigr)
    - \ell \ord\mathbf m +\ell^2\idx\mathbf m\\
 &= \ell(\ell-k)\idx\mathbf m>0.
\end{align*}
Hence $|\{\lambda_{\mathbf m'}\}|>0$. 
\end{proof}
\subsection{Universal model}\index{universal model}
\index{Fuchsian differential equation/operator!universal operator}
Now we have a main result in \S\ref{sec:DS} which assures the existence
of Fuchsian differential operators with given spectral types.
\begin{thm}\label{thm:univmodel} 
Fix a tuple\/ $\mathbf m=\bigl(m_{j,\nu}\bigr)
_{\substack{0\le j\le p\\ 1\le \nu\le n_j}}\in\mathcal P^{(n)}_{p+1}$.

{\rm i) }
Under the notation in Definitions~\ref{def:tuples}, \ref{defn:real}
and \ref{def:pell}, the tuple\/ $\mathbf m$ is realizable if and 
only if there exists a non-negative integer\/ $K$ such that\/
$\p_{max}^i \mathbf m$ are well-defined for $i=1,\dots,K$ and
\begin{equation}\label{eq:decord}
 \begin{split}
 &\ord\mathbf m>\ord\p_{max}\mathbf m>\ord\p_{max}^2\mathbf m>\cdots
  >\ord\p_{max}^K\mathbf m,\\
 &d_{max}(\p_{max}^K\mathbf m)=2\ord \p_{max}^K\mathbf m\text{ \ or \ }
  d_{max}(\p_{max}^K\mathbf m)\le0.
 \end{split}
\end{equation}

{\rm ii) }
Fix complex numbers $\lambda_{j,\nu}$.
If there exists an irreducible Fuchsian operator with the 
Riemann scheme \eqref{eq:GRS} such that it is locally non-degenerate
{\rm (cf.~Definition~\ref{def:locnondeg}),} 
then\/ $\mathbf m$ is irreducibly realizable.

Here we note that if $P$ is irreducible and\/ $\mathbf m$ is rigid, 
$P$ is locally non-degenerate {\rm (cf.~Definition~\ref{def:locnondeg}).} 

Hereafter in this theorem we assume\/ $\mathbf m$ is realizable.

{\rm iii) }
$\mathbf m$ is irreducibly realizable if and only if\/
$\mathbf m$ is indivisible or $\idx\mathbf m<0$.

{\rm iv) }
There exists a \textsl{universal model} $P_{\mathbf m}u=0$ 
associated with\/ $\mathbf m$ which has the following property.

Namely, $P_{\mathbf m}$ 
is the Fuchsian differential operator of the form
\begin{equation}\label{eq:uinvPm}
\begin{split}
 P_{\mathbf m}&=\Bigl(\prod_{j=1}^p(x-c_j)^n\Bigr)\frac{d^n}{dx^n}
 + a_{n-1}(x)\frac{d^{n-1}}{dx^{n-1}}+\cdots+a_1(x)\frac{d}{dx}+a_0(x),\\
 &a_j(x)\in\mathbb C[\lambda_{j,\nu},g_1,\dots,g_N]
 \end{split}
\end{equation}
such that $P_{\mathbf m}$ has 
regular singularities at $p+1$ fixed points $x=c_0=\infty, c_1,\dots,c_p$
and the Riemann scheme of $P_{\mathbf m}$ equals \eqref{eq:GRS}
for any $g_i\in\mathbb C$  and $\lambda_{j,\nu}\in\mathbb C$
under the Fuchs relation \eqref{eq:Fuchs}.
Moreover the coefficients 
$a_j(x)$ are polynomials of $x$, $\lambda_{j,\nu}$ and 
$g_i$ with the degree at most $(p-1)n+j$ for $j=0,\dots,n$, respectively.
Here $g_i$ are called accessory parameters and
we call $P_{\mathbf m}$ the \textsl{universal operator} of type\/ $\mathbf m$.

The non-negative integer $N$ will be denoted by\/ $\Ridx\mathbf m$
and given by
\begin{equation}
 N=\Ridx\mathbf m:=\begin{cases}
   0&(\idx\mathbf m > 0),\\
   \gcd \mathbf m&(\idx\mathbf m=0),\\
   \Pidx\mathbf m&(\idx\mathbf m<0).
   \end{cases}
\end{equation}\index{00Ridx@$\Ridx$}
Put\/ $\overline{\mathbf m}=\bigl({\overline m}_{j,\nu}\bigr)_
{\substack{0\le j\le p\\ 1\le \nu\le n_j}}:=\p_{max}^K\mathbf m$
with the non-negative integer $K$ given in\/ {\rm i)}.

When\/ $\idx\mathbf m\le 0$, we define
\begin{align*}
  q^0_\ell &:= \#\{i\,;\,\sum_{\nu=1}^{\bar n_0}
   \max\{{\overline m}_{0,\nu}-i,0\}\ge\ord\overline{\mathbf m}-\ell,\ i\ge 0\},\\
  I_{\mathbf m}&:=
  \{(j,\nu)\in\mathbb Z^2\,;\,
    q^0_\nu\le j\le q^0_\nu+N_\nu-1,\ 
  1\le \nu\le\ord\overline{\mathbf m}-1\}.
\end{align*}
When\/ $\idx\mathbf m> 0$, we put $I_{\mathbf m}=\emptyset$.

Then\/ $\#I_{\mathbf m}=\Ridx\mathbf m$ and
we can define\/ $I_i$ such that $I_{\mathbf m}=\{I_i\,;\,i=1,\dots,N\}$
and $g_i$ satisfy \eqref{eq:TopAcc} by putting $g_{I_i}=g_i$
for $i=1,\dots,N$.

{\rm v) } Retain the notation in Definition~\ref{def:redGRS}.  If
$\lambda_{j,\nu}\in\mathbb C$ satisfy
\begin{equation}\label{eq:inUniv}
 \begin{cases}
 \sum_{j=0}^p
 \lambda(k)_{j,\ell(k)_j+\delta_{j,j_o}(\nu_o-\ell(k)_j)}\\
 \ 
 \notin\{0,-1,-2,-3,\ldots,m(k)_{j_o,\ell(k)_{j_o}}-m(k)_{j_o,\nu_o}-d(k)+2\}\\
 \quad\text{for any }k=0,\dots,K-1\text{ and }(j_0,\nu_o)\text{ satisfying}\\
 \quad m(k)_{j_o,\nu_o}\ge m(k)_{j_o,\ell(k)_{j_o}} - d(k)+2,
 \end{cases}
\end{equation}
any Fuchsian differential operator $P$ of the normal form 
which has the Riemann scheme \eqref{eq:GRS} 
belongs to $P_{\mathbf m}$ with a suitable 
$(g_1,\dots,g_N)\in\mathbb C^N$.
\begin{align}
&\begin{cases}
\text{If\/ $\mathbf m$ is a scalar multiple of a fundamental tuple
or simply reducible, }\\
\text{\eqref{eq:inUniv} is always valid for any $\lambda_{j,\nu}$.}
\end{cases}\allowdisplaybreaks\label{eq:CSR}\\
&\begin{cases}
\text{Fix $\lambda_{j,\nu}\in\mathbb C$.
Suppose there is an irreducible Fuchsian differential}\\
\text{operator with the
Riemann scheme \eqref{eq:GRS} such that the operator is}\\
\text{locally non-degenerate or $K\le 1$, 
then \eqref{eq:inUniv} is valid.}
\end{cases}\label{eq:irinuniv}
\end{align}

Suppose\/ $\mathbf m$ is monotone.
Under the notation in {\rm\S\ref{sec:KM}}, 
the condition \eqref{eq:inUniv} is equivalent to
\begin{equation}\label{eq:KinUniv}
 \begin{split}
  &(\Lambda(\lambda)|\alpha)+1\notin
  \{0,-1,\ldots,2-(\alpha|\alpha_{\mathbf m})\}\\
  &\qquad
  \text{ for any }\alpha\in\Delta(\mathbf m)
  \text{ satisfying }(\alpha|\alpha_{\mathbf m})>1.
 \end{split}
\end{equation}
\end{thm}
Example~\ref{ex:outuniv} gives a Fuchsian differential operator with 
the rigid spectral type $21,21,21,21$ which doesn't belong to the corresponding
universal operator.

The fundamental tuple and the simply reducible tuple are defined as follows.
\begin{defn}\label{def:fund}
{\rm i) }(fundamental tuple)
\index{tuple of partitions!fundamental}\index{00fm@$f\mathbf m$}
An irreducibly realizable tuple $\mathbf m\in\mathcal P$ 
is called \textsl{fundamental} if $\ord\mathbf m=1$ 
or $d_{\max}(\mathbf m)\le 0$.

For an irreducibly realizable tuple 
$\mathbf m\in\mathcal P$, there exists a non-negative 
integer $K$ such that $\p^K_{max}\mathbf m$ is fundamental
and satisfies \eqref{eq:decord}.
Then we call $\p^K_{max}\mathbf m$ is 
a fundamental tuple corresponding to $\mathbf m$ and define 
$f\mathbf m:=\p^K_{max}\mathbf m$.

{\rm ii) }(simply reducible tuple)
\index{tuple of partitions!simply reducible}
A tuple $\mathbf m$ is \textsl{simply reducible} if 
there exists a positive integer 
$K$ satisfying \eqref{eq:decord} and
$\ord\p_{max}^K\mathbf m=\ord\mathbf m-K$.
\end{defn}
\begin{proof}[Proof of Theorem~\ref{thm:univmodel}]
{\rm i)}
We have proved that $\mathbf m$ is realizable 
if $d_{max}(\mathbf m)\le 0$.
Note that the condition $d_{max}(\mathbf m)=2\ord\mathbf m$ is equivalent
to the fact that $s\mathbf m$ is trivial.
Hence Theorem~\ref{thm:realizable} proves the claim.

{\rm iv) }
Now we use the notation in Definition~\ref{def:redGRS}.
The existence of the universal operator
is clear if $s\mathbf m$ is trivial.
If $d_{max}(\mathbf m)\le 0$, Theorem~\ref{thm:ExF} and 
Proposition~\ref{prop:dividx0} with Corollary~\ref{cor:idx0}
assure the existence of the universal operator $P_{\mathbf m}$
claimed in iii).
Hence iii) is valid for the tuple 
$\mathbf m(K)$ and we have a universal operator $P_K$
with the Riemann scheme $\{\lambda(K)_{\mathbf m(K)}\}$.

The universal operator $P_k$ with the Riemann scheme 
$\{\lambda(k)_{\mathbf m(k)}\}$ are inductively obtained by 
applying $\p_{\ell(k)}$ to the universal operator $P_{k+1}$
with the Riemann scheme $\{\lambda(k+1)_{\mathbf m(k+1)}\}$ for 
$k=K-1,K-2,\dots,0$.
Since the claims in iii) such as \eqref{eq:TopAcc} 
are kept by the operation $\p_{\ell(k)}$, we have iv).

{\rm iii) }
Note that $\mathbf m$ is irreducibly realizable if $\mathbf m$ is indivisible
(cf.~Remark~\ref{rem:generic} ii)).
Hence suppose $\mathbf m$ is not indivisible.  Put $k=\gcd\mathbf m$ and
$\mathbf m=k\mathbf m'$.  Then $\idx\mathbf m
=k^2\idx\mathbf m'$.

If $\idx\mathbf m>0$, then $\idx\mathbf m>2$ and the inequality 
\eqref{eq:rigididx} in Lemma~\ref {lem:irrred} implies that 
$\mathbf m$ is not irreducibly realizable.
If $\idx\mathbf m<0$, Proposition~\ref{prop:irred} assures that $\mathbf m$ 
is irreducibly realizable.

Suppose $\idx\mathbf m=0$.
Then the universal operator $P_{\mathbf m}$ has $k$ accessory parameters.
Using the argument in the first part of the proof of 
Proposition~\ref{prop:dividx0}, we can construct a Fuchsian differential 
operator $\tilde P_{\mathbf m}$ with the Riemann scheme 
$\bigl\{\lambda_{\mathbf m}\bigr\}$.
Since $\tilde P_{\mathbf m}$ is a product of $k$ copies of
the universal operator $P_{\overline{\mathbf m}}$ and it has $k$
accessory parameters, the operator $P_{\mathbf m}$ coincides with 
the reducible operator $\tilde P_{\mathbf m}$ and hence $\mathbf m$ is not 
irreducibly realizable.

{\rm v) } 
Fix $\lambda_{j,\nu}\in\mathbb C$.
Let $P$ be a Fuchsian differential operator with the Riemann
scheme $\{\lambda_{\mathbf m}\}$.
Suppose $P$ is of the normal form.

Theorem~\ref{thm:ExF} and Proposition~\ref{prop:dividx0} assure that 
$P$ belongs to $P_{\mathbf m}$ if $K=0$.

Theorem~\ref{thm:GRSmid} proves that if $\p_{max}^kP$ has the Riemann
scheme $\{\lambda(k)_{\mathbf m(k)}\}$ and \eqref{eq:inUniv} is valid, 
then $\p_{max}^{k+1}P=\p_{\ell(k)}\p_{max}^kP$ is well-defined and has
the Riemann scheme $\{\lambda(k+1)_{\mathbf m(k+1)}\}$ for $k=0,\dots,K-1$
and hence it follows from \eqref{eq:pellP2} that 
$P$ belongs to the universal operator $P_{\mathbf m}$
because $\p_{max}^K P$ belongs to the universal operator $P_{\mathbf m(K)}$.

If $\mathbf m$ is simply reducible, $d(k)=1$ and therefore 
\eqref{eq:inUniv} is valid because 
$m(k)_{j,\nu}\le m(k)_{j,\ell(k)_\nu}<m(k)_{j,\ell(k)_\nu}-d(k)+2$
for $j=0,\dots,p$ and $\nu=1,\dots,n_j$ and $k=0,\dots,K-1$.

The equivalence of the conditions \eqref{eq:inUniv} and \eqref{eq:KinUniv}
follows from the argument in \S\ref{sec:KM},
Proposition~\ref{prop:wm} and Theorem~\ref{thm:irrKac}.

{\rm ii) }
Suppose there exists an irreducible operator 
$P$ with the Riemann scheme \eqref{eq:GRS}.
Let $\mathbf M=(M_0,\dots,M_p)$ be the tuple of monodromy generators of 
the equation $Pu=0$
and put $\mathbf M(0)=\mathbf M$.
Let $\mathbf M(k+1)$ be the tuple of matrices applying the operations
in \S\ref{sec:MM} to $\mathbf M(k)$ corresponding to the operations
$\p_{\ell(k)}$ for $k=0,1,2,\ldots$.

Comparing the operations on $\mathbf M(k)$ and $\p_{\ell(k)}$,
we can conclude that there exists a non-negative integer $K$ 
satisfying the claim in i).
In fact Theorem~\ref{thm:Mmid} proves that $\mathbf M(k)$ are irreducible, 
which assures that the conditions \eqref{eq:redC2} and \eqref{eq:redC3} 
corresponding to the operations $\p_{\ell(k)}$ are always valid
(cf.~Corollary~\ref{cor:irred}).
Therefore $\mathbf m$ is realizable and moreover we can conclude
that \eqref{eq:irinuniv} implies \eqref{eq:inUniv}.
If $\idx\mathbf m$ is divisible and $\idx\mathbf m=0$, then $P_{\mathbf m}$
is reducible for any fixed parameters $\lambda_{j,\nu}$ and $g_i$.
Hence $\mathbf m$ is irreducibly realizable.
\end{proof}
\begin{rem}\label{rem:inuniv}
{\rm i) }
The uniqueness of the universal operator in Theorem~\ref{thm:univmodel}
is obvious.  But it is not valid in the case of systems of Schlesinger 
canonical form (cf.~Example~\ref{ex:univSch}).

{\rm ii)}
The assumption that $Pu=0$ is locally non-degenerate seems to be not
necessary in Theorem~\ref{thm:univmodel} ii) and \eqref{eq:irinuniv}.  
When $K=1$, this is clear from the proof of the theorem.
For example, the rigid irreducible operator with the spectral
type $31,31,31,31,31$ belongs to the universal operator
of type $211,31,31,31,31$.
\end{rem}
\subsection{Simply reducible spectral type}\label{sec:simpred}
\index{tuple of partitions!simply reducible}
In this subsection we characterize the tuples of the simply reducible 
spectral type.
\index{00Nnu@$N_\nu(\mathbf m)$}
\begin{prop}\label{prop:sred}
{\rm i)}
A realizable tuple\/ $\mathbf m\in\mathcal P^{(n)}$ satisfying
$m_{0,\nu}=1$ for $\nu=1,\ldots,n$ is simply reducible
if\/ $\mathbf m$ is not fundamental.

{\rm ii)}
The simply reducible rigid tuple corresponds to the tuple in 
Simpson's list {\rm(cf.~\S\ref{sec:rigidEx})} or it is isomorphic to\/
$21111,222,33$.
\index{tuple of partitions!rigid!21111,222,33}

{\rm iii)}
Suppose\/ $\mathbf m\in\mathcal P_{p+1}$ is not fundamental.
Then\/ $\mathbf m$ satisfies the condition $N_\nu(\mathbf m)\ge 0$ 
for $\nu=2,\dots,\ord\mathbf m-1$ in Definition~\ref{def:Nnu}
if and only if\/ $\mathbf m$ is realizable and simply reducible.

{\rm iv)}
Let $\mathbf m\in\mathcal P_{p+1}$ be a realizable 
monotone tuple.  Suppose\/ $\mathbf m$ is not fundamental.
Then under the notation in {\rm\S\ref{sec:KM}},
$\mathbf m$ is simply reducible if and only if
\begin{equation}
 (\alpha|\alpha_{\mathbf m})=1\quad(\forall\alpha\in\Delta(\mathbf m)),
\end{equation}
namely\/ $[\Delta(\mathbf m)]=1^{\#\Delta(\mathbf m)}$ {\rm(cf.~Remark~\ref{rem:length}~ii)).}
\end{prop}
\begin{proof}
{\rm i)}
The claim is obvious from the definition.

{\rm ii)}
Let $\mathbf m'$ be a simply reducible rigid tuple.
We have only to prove that $\mathbf m=\p_{max}\mathbf m'$ is in the
Simpson's list or $21111,222,33$ and $\ord\mathbf m'=\ord\mathbf m+1$
and $d_{max}(\mathbf m)=1$, then $\mathbf m'$ is in Simpson's list
or $21111,222,33$.
The condition $\ord\mathbf m'=\ord\mathbf m+1$ implies 
$\mathbf m\in\mathcal P_3$.
We may assume $\mathbf m$ is monotone and 
$\mathbf m'=\p_{\ell_0,\ell_1,\ell_2}\mathbf m$.
The condition $\ord\mathbf m'=\ord\mathbf m+1$ also implies
\[
 (m_{0,1}-m_{0,\ell_0})+(m_{1,1}-m_{1,\ell_0})+(m_{2,1}-m_{2,\ell_0})=2.
\]
Since $\p_{max}\mathbf m'=\mathbf m$, we have $m_{j,\ell_j}\ge m_{j,1}-1$
for $j=0,1,2$.  Hence there exists an integer $k$ with $0\le k\le 2$ such that 
$m_{j,\ell_j} = m_{j,1}-1+\delta_{j,k}$ for $j=0,1,2$.
Then the following claims are easy, which assures the proposition.

If $\mathbf m=11,11,11$, $\mathbf m'$ is isomorphic to $1^3,1^3,21$.

If $\mathbf m=1^3,1^3,21$, $\mathbf m'$ is isomorphic to
$1^4,1^4,31$ or $1^4,211,22$.

If $\mathbf m=1^n,1^n,n-11$ with $n\ge 4$, $\mathbf m'=1^{n+1},1^{n+1},n1$.

If $\mathbf m=1^{2n},nn-11,nn$ with $n\ge 2$, $\mathbf m'=1^{2n+1},nn1,n+1n$.

If $\mathbf m=1^5,221,32$, then $\mathbf m'=1^6,33,321$ or $1^6,222,42$ or 
$21111,222,33$.

If $\mathbf m=1^{2n+1},n+1n,nn1$ with $n\ge 3$, $\mathbf m'=1^{2n+2},n+1n+1,n+1n1$.

If $\mathbf m=1^6,222,42$ or $\mathbf m=21111,222,33$, $\mathbf m'$ 
doesn't exists.

{\rm iii)}
Note that Theorem~\ref{thm:ExF} assures that the condition 
$N_\nu(\mathbf m)\ge 0$ for $\nu=1,\dots,\ord\mathbf m-1$ implies
that $\mathbf m$ is realizable.

We may assume $\mathbf m\in\mathcal P_{p+1}^{(n)}$ is standard.
Put $d=m_{0,1}+\cdots+m_{p,1}-(p-1)n>0$ and $\mathbf m'=\p_{max}\mathbf m$.
Then $m'_{j,\nu}=m_{j,\nu}-\delta_{\nu,1}d$ for $j=0,\dots,p$ and $\nu\ge 1$.
Under the notation in Definition~\ref{def:Nnu}
the operation $\p_{max}$ transforms the sets
\begin{equation*}
 \mathfrak m_j:=\{\widetilde m_{j,k}\,;\,
  k=0,1,2,\ldots\text{ and }\widetilde m_{j,k}>0\}
\end{equation*}
into 
\begin{equation*}
  \mathfrak m_j'
    =\bigl\{\widetilde m_{j,k} - \min\{d,m_{j,1}-k\}\,;\,k=0,\dots,
    \max\{m_{j,1}-d,m_{j,2}-1\}\bigr\},
\end{equation*}
respectively because $\widetilde m_{j,i}=\sum_\nu\max\{m_{j,\nu}-i,0\}$.
Therefore $N_\nu(\mathbf m')\le N_\nu(\mathbf m)$ for 
$\nu=1,\dots,n-d-1=\ord\mathbf m'-1$.
Here we note that
\begin{equation*}
 \sum_{\nu=1}^{n-1}N_\nu(\mathbf m) 
 = \sum_{\nu=1}^{n-d-1}N_\nu(\mathbf m') =\Pidx\mathbf m.
\end{equation*}
Hence $N_\nu(\mathbf m)\ge 0$ for $\nu=1,\dots,n-1$ if and only if 
$N_\nu(\mathbf m')=N_\nu(\mathbf m)$
for $\nu=1,\dots,(n-d)-1$ and moreover $N_\nu(\mathbf m)=0$ for 
$\nu=n-d,\dots,n-1$.
Note that the condition that $N_\nu(\mathbf m')= N_\nu(\mathbf m)$ 
for $\nu=1,\dots,(n-d)-1$ equals
\begin{equation}\label{eq:young}
 m_{j,1}-d\ge m_{j,2}-1\text{ \  for \ }j=0,\dots,p.\qquad\qquad
\raisebox{-15pt}{\scalebox{.5}
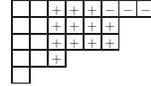
{
\begin{Young}
 & &+&+&+&$-$&$-$&$-$\cr
 & &+&+&+&+\cr
 & &+&+&+&+\cr
 & &+\cr
 \cr
\end{Young}}}
\end{equation}
This is easy to see by using a Young diagram.
For example, when $\{8,6,6,3,1\}=
\{m_{0,1},m_{0,2},m_{0,3},m_{0,4},m_{0,5}\}$ is 
a partition of $n=24$, the corresponding Young diagram 
is as above and then $\widetilde m_{0,2}$ equals 15,
namely, the number of boxes with the sign + or $-$.
Moreover when $d=3$, the boxes with the sign $-$ are deleted by $\p_{max}$
and the number $\widetilde m_{0,2}$ changes into 12.
In this case $m_0=\{24,19,15,11,8,5,2,1\}$ and 
$m_0'=\{21,16,12,8,5,2\}$.

If $d\ge2$, then $1\in\mathfrak m_j$ for $j=0,\dots,p$ and therefore
$N_{n-2}(\mathbf m)-N_{n-1}(\mathbf m)=2$, 
which means $N_{n-1}(\mathbf m)\ne 0$ or $N_{n-2}(\mathbf m)\ne 0$.
When $d=1$, we have $N_\nu(\mathbf m)=N_\nu(\mathbf m')$ for 
$\nu=1,\dots,n-2$ and $N_{n-1}(\mathbf m)=0$.
Thus we have the claim.

iv) The claim follows from Proposition~\ref{prop:wm}.
\end{proof}
\begin{exmp}\label{ex:SR0}
\index{tuple of partitions!simply reducible}
We show the simply reducible tuples with index 0
whose fundamental tuple is of type $\tilde D_4$, $\tilde E_6$,
$\tilde E_7$ or $\tilde E_8$ (cf.~Example~\ref{ex:Nj}).

$\tilde D_4$: $21,21,21,111\quad 22,22,31,211\quad 22,31,31,1111\quad$

$\tilde E_6$: $211,211,1111\ \  221,221,2111\ \  221,311,11111\ \  222,222,3111
\ \ 222,321,2211$

$\qquad\!222,411,111111 \ \  322,331,2221 \ \  332,431,2222 \ \ 333,441,3222$

$\tilde E_7$: $11111,2111,32 \ \ 111111,2211,42 \ \ 21111,2211,33 \ \ 
111111,3111,33$

$\qquad\!22111,2221,43 \ \ 1111111,2221,52 \ \ 22211,2222,53 \ \ 
11111111,2222,62$

$\qquad\!32111,2222,44 \ \ 22211,3221,53$

$\tilde E_8$: $1111111,322,43 \ \ 11111111,332,53 \ \ 2111111, 332,44 \ \ 
11111111,422,44$

$\qquad\!2211111,333,54\ \ 111111111,333,63 \ \ 2221111,433,55 \ \ 
2222111,443,65$

$\qquad\!3222111,444,66 \ \ 2222211,444,75  \ \ 2222211,543,66 \ \ 2222221,553,76$

$\qquad\!2222222,653,77$
\end{exmp}
In general, we have the following proposition.
\begin{prop}
There exist only a finite number of standard and simply reducible tuples
with any fixed non-positive index of rigidity.
\end{prop}
\begin{proof}
First note that $\mathbf m\in\mathcal P_{p+1}$ if 
$d_{max}(\mathbf m)=1$ and $\ord\mathbf m>3$
and $\p_{max}\mathbf m\in\mathcal P_{p+1}$.
Since there exist only finite basic tuples with any fixed index 
of rigidity (cf.~Remark~\ref{rem:Fbasic}), 
we have only to prove the non-existence of the infinite sequence
\[
 \mathbf m(0)\xleftarrow{\p_{max}}\mathbf m(1)\xleftarrow{\p_{max}}\cdots\cdots
 \xleftarrow{\p_{max}}\mathbf m(k)\xleftarrow{\p_{max}}\mathbf m(k+1)
 \xleftarrow{\p_{max}}\cdots
\]
such that $d_{max}(\mathbf m(k))=1$ for $k\ge 1$ and $\idx\mathbf m(0)\le 0$.

Put
\begin{align*}
  \bar m(k)_j&=\max_{\nu}\{m(k)_{j,\nu}\},\\
  a(k)_j&=\#\{\nu\,;\,m(k)_{j,\nu}=\bar m(k)_j\},\\ 
  b(k)_j&=\begin{cases} 
       \#\{\nu\,;\,m(k)_{j,\nu}=\bar m(k)_j-1\} &(\bar m(k)_j>1),\\
       \infty&(\bar m(k)_j=1).
      \end{cases}
\end{align*}
The assumption $d_{max}(\mathbf m(k))=d_{max}(\mathbf m(k+1))=1$
implies that there exist indices $0\le j_k<j'_k$ such that
\begin{equation}\label{eq:SRP1}
 (a(k+1)_j,b(k+1)_j)
 =\begin{cases}
   (a(k)_j+1,b(k)_j-1) &(j=j_k\text{ or }j'_k),\\
   (1,a(k)_j-1)        &(j\ne j_k\text{ and }j'_k)
  \end{cases}
\end{equation}
and
\begin{equation}\label{eq:RSM1}
 \bar m(k)_0+\cdots+\bar m(k)_p=(p-1)\ord\mathbf m(k)+1\qquad(p\gg 1)
\end{equation}
for $k=1,2,\dots$.
Since $a(k+1)_j+b(k+1)_j\le a(k)_j+b(k)_j$, there exists a positive integer
$N$ such that $a(k+1)_j+b(k+1)_j = a(k)_j+b(k)_j$ for $k\ge N$, which means
\begin{equation}\label{eq:SRP2}
  b(k)_j\begin{cases}
          >0&(j=j_k\text{ or }j'_k),\\
          =0&(j\ne j_k\text{ and }j'_k).
         \end{cases}
\end{equation}
Putting $(a_j,b_j)=(a(N)_j,b(N)_j)$,
we may assume $b_0\ge b_1>b_2=b_3=\cdots=0$ and $a_2\ge a_3\ge\cdots$.
Moreover we may assume $j'_{N+1}\le 3$, which means $a_j=1$ for $j\ge 4$.
Then the relations \eqref{eq:SRP1} and \eqref{eq:SRP2} for 
$k=N, N+1, N+2$ and $N+3$ prove that $\bigl((a_0,b_0),\cdots,(a_3,b_3)\bigr)$
is one of the followings:
\begin{align}
&((a_0,\infty),(a_1,\infty),(1,0),(1,0)),\label{eq:RSC1}\\
&((a_0,\infty),(1,1),(2,0),(1,0)),\label{eq:RSC2}\\
&((2,2),(1,1),(4,0),(1,0)),\ ((1,3),(3,1),(2,0),(1,0)),\label{eq:RSC3}\\
&((1,2),(2,1),(3,0),(1,0)),\label{eq:RSC4}\\
&((1,1),(1,1),(2,0),(2,0)).\label{eq:RSC5}
\end{align}
In fact if $b_1>1$, $a_2=a_3=1$ and we have \eqref{eq:RSC1}.
Thus we may assume $b_1=1$.
If $b_0=\infty$, $a_3=1$ and we have \eqref{eq:RSC2}.
If $b_0=b_1=1$, we have easily \eqref{eq:RSC5}.
Thus we may moreover assume $b_1=1<b_0<\infty$ and $a_3=1$.
In this case the integers $j''_k$ satisfying
$b(k)_{j''_k}=0$ and $0\le j''_k\le 2$ for $k\ge N$ are
uniquely determined and we have easily \eqref{eq:RSC3} or 
\eqref{eq:RSC4}.

Put $n=\ord\mathbf m(N)$.
We may suppose $\mathbf m(N)$ is standard.
Let $p$ be an integer such that $m_{j,0}<n$ if and only 
if $j\le p$.  Note that $p\ge 2$.
Then if $\mathbf m(N)$ satisfies \eqref{eq:RSC1} 
(resp.~\eqref{eq:RSC2}), \eqref{eq:RSM1} implies 
$\mathbf m(N)=1^n,1^n,n-11$ (resp.~$1^n,mm-11,mm$ or $1^n,m+1m,mm1$) 
and $\mathbf m(N)$ is rigid.

Suppose one of \eqref{eq:RSC3}--\eqref{eq:RSC5}.
Then it is easy to check that $\mathbf m(N)$ doesn't satisfy 
\eqref{eq:RSM1}.
For example, suppose \eqref{eq:RSC4}.
Then $3m_{0,1}-2\le n$, $3m_{1,1}-1\le n$ and $3m_{2,1}\le n$
and we have $m_{0,1}+m_{1,1}+m_{2,1}\le 
[\frac {n+2}3]+[\frac{n+1}3]+[\frac{n}3]=n$, 
which contradicts to \eqref{eq:RSM1}.
The relations
$[\frac{n+2}{4}]+[\frac n2]+[\frac{n}4]\le n$
and
$2[\frac {n+1}2]+2[\frac n2]=2n$ 
assure the same conclusion in the other cases.
\end{proof}
%
%
%
\section{A Kac-Moody root system}\label{sec:KacM}
\subsection{Correspondence with a Kac-Moody root system}\label{sec:KM}
\index{Kac-Moody root system}
We review a Kac-Moody root system to describe the combinatorial
structure of middle convolutions on the spectral types.
Its relation to Deligne-Simpson problem is first clarified by \cite{CB}.

Let
\begin{equation}
  I:=\{0,\,(j,\nu)\,;\,j=0,1,\ldots,\ \nu=1,2,\ldots\}.
\end{equation}
\index{00I@$I,\ I'$}%
be a set of indices and 
let $\mathfrak h$ be an infinite dimensional real vector space with 
the set of basis $\Pi$, where
\begin{equation}
 \Pi=\{\alpha_i\,;\,i\in I\}
    =\{\alpha_0,\ \alpha_{j,\nu}\,;\,j=0,1,2,\ldots,\ \nu=1,2,\ldots\}.
\end{equation}
\index{000Pi@$\Pi,\ \Pi',\ \Pi(\mathbf m)$}
Put
\begin{align}
 I' &:= I\setminus\{0\},\qquad \Pi':=\Pi\setminus\{\alpha_0\},\\
 Q  &:=\sum_{\alpha\in\Pi}\mathbb Z\alpha\ \supset\ 
 Q_+:=\sum_{\alpha\in\Pi}\mathbb Z_{\ge0}\alpha.\index{00Qp@$Q_+$}
\end{align}
We define an indefinite symmetric bilinear form on $\mathfrak h$ by
\begin{equation}\label{eq:PIKac}%
\index{000alpha0@$\alpha_0$, $\alpha_\ell$, $\alpha_{\mathbf m}$}%
 \begin{split}
 (\alpha|\alpha)
 &= 2\qquad\ \,(\alpha\in\Pi),\\
 (\alpha_0|\alpha_{j,\nu})
 &=-\delta_{\nu,1},\\
 (\alpha_{i,\mu}|\alpha_{j,\nu})
 &=\begin{cases}
    0 &(i\ne j\text{ \ \ or \ \ }|\mu-\nu|>1),\\
    -1&(i=j\text{ \ and \ }|\mu-\nu|=1).
  \end{cases}
 \end{split}\quad\ \ \raisebox{10pt}{\text{\small
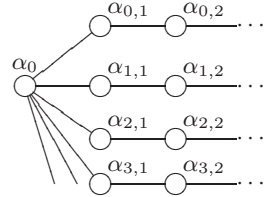
\begin{xy}
\ar@{-}               *++!D{\text{$\alpha_0$}}  *\cir<4pt>{}="O";
             (10,0)   *+!L!D{\text{$\alpha_{1,1}$}} *\cir<4pt>{}="A",
\ar@{-} "A"; (20,0)   *+!L!D{\text{$\alpha_{1,2}$}} *\cir<4pt>{}="B",
\ar@{-} "B"; (30,0)   *{\cdots}, 
\ar@{-} "O"; (10,-7)  *+!L!D{\text{$\alpha_{2,1}$}} *\cir<4pt>{}="C",
\ar@{-} "C"; (20,-7)  *+!L!D{\text{$\alpha_{2,2}$}} *\cir<4pt>{}="E",
\ar@{-} "E"; (30,-7)  *{\cdots}
\ar@{-} "O"; (10,8)   *+!L!D{\text{$\alpha_{0,1}$}} *\cir<4pt>{}="D",
\ar@{-} "D"; (20,8)   *+!L!D{\text{$\alpha_{0,2}$}} *\cir<4pt>{}="F",
\ar@{-} "F"; (30,8)   *{\cdots}
\ar@{-} "O"; (10,-13) *+!L!D{\text{$\alpha_{3,1}$}} *\cir<4pt>{}="G",
\ar@{-} "G"; (20,-13) *+!L!D{\text{$\alpha_{3,2}$}} *\cir<4pt>{}="H",
\ar@{-} "H"; (30,-13) *{\cdots},
\ar@{-} "O"; (7, -13),
\ar@{-} "O"; (4, -13),
\end{xy}}}
\end{equation}

\index{Weyl group}
\index{simple root}
\index{simple reflection}
The element of $\Pi$ is called the \textsl{simple root} of a 
Kac-Moody root system and the \textsl{Weyl group} 
$W_{\!\infty}$ of this Kac-Moody root system is generated by the 
\textsl{simple reflections} $s_i$ with $i\in I$.  
Here the \textsl{reflection} with respect to an element 
$\alpha\in\mathfrak h$ satisfying $(\alpha|\alpha)\ne0$ is 
the linear transformation
\index{00salpha@$s_\alpha,\ s_i$}\index{reflection}
\begin{equation}
 s_\alpha\,:\,\mathfrak h\ni x\mapsto
 x-2\frac{(x|\alpha)}{(\alpha|\alpha)}\alpha\in\mathfrak h
\end{equation}
and
\begin{equation}\label{eq:Kzri}
  s_i=s_{\alpha_i} \text{ \ for \ } i\in I.
\end{equation}
In particular $s_i(x)=x-(\alpha_i|x)\alpha_i$ for $i\in I$
and the subgroup of $W_{\!\infty}$ generated by $s_i$ for 
$i\in I\setminus\{0\}$ is denoted by $W'_{\!\infty}$.

The Kac-Moody root system is determined by the set of 
simple roots $\Pi$ and its Weyl group $W_{\!\infty}$ and 
it is denoted by $(\Pi,W_{\!\infty})$.
\index{00Winfty@$W_{\!\infty},\ W'_{\!\infty},\ \widetilde W_{\!\infty}$}

Denoting $\sigma(\alpha_0)=\alpha_0$ and $\sigma(\alpha_{j,\nu})=
\alpha_{\sigma(j),\nu}$ for $\sigma\in\mathfrak S_\infty$, we put
\begin{equation}\label{eq:Kzouter}
  \widetilde W_{\!\infty}:=\mathfrak S_\infty\ltimes W_{\!\infty},
\end{equation}
which is an automorphism group of the root system.
\begin{rem}[\cite{Kc}]\label{rem:Kac}
\index{000Delta@$\Delta,\ \Delta_+,\ \Delta_-$}
\index{000Delta0@$\Delta^{re},\ \Delta^{re}_+,\ \Delta^{re}_-,\ \Delta^{im}, \Delta^{im}_+,\ \Delta^{im}_-$}
\index{00B@$B$}
The set $\Delta^{re}$ of \textsl{real roots} equals the $W_{\!\infty}$-orbit
of $\Pi$, which also equals $W_{\!\infty}\alpha_0$.
Denoting 
\begin{equation}
 B:=\{\beta\in Q_+\,;\,\supp\beta\text{ is connected and }
 (\beta,\alpha)\le 0\quad(\forall\alpha\in\Pi)\},
\end{equation}
the set of \textsl{positive imaginary roots} 
$\Delta^{im}_+$ equals $W_{\!\infty} B$.
Here 
\begin{equation}
\supp\beta:=\{\alpha\in\Pi\,;\,n_\alpha\ne0\}\text{ \  if  \ }
\beta=\sum_{\alpha\in\Pi} n_\alpha\alpha.
\end{equation} 

The set $\Delta$ of roots equals 
$\Delta^{re}\cup\Delta^{im}$ by denoting
$\Delta_-^{im}=-\Delta_+^{im}$ and 
$\Delta^{im}=\Delta_+^{im}\cup\Delta_-^{im}$.
Put $\Delta_+=\Delta\cap Q_+$, $\Delta_-=-\Delta_+$, 
$\Delta^{re}_+=\Delta^{re}\cap Q_+$ and $\Delta^{re}_-=-\Delta^{re}_+$.
Then
$\Delta=\Delta_+\cup\Delta_-$,
$\Delta^{im}_+\subset\Delta_+$ and $\Delta^{re}=\Delta^{re}_+\cup\Delta^{re}_-$.
The root in $\Delta$ is called \textsl{positive} if and only if $\alpha\in Q_+$.
\index{real root}\index{imaginary root}\index{positive root}

A subset $L\subset\Pi$ is called \textsl{connected}
if the decomposition $L_1\cup L_2= L$ with 
$L_1\ne\emptyset$ and $L_2\ne\emptyset$ always implies the 
existence of $v_j\in L_j$ satisfying $(v_1|v_2)\ne0$.
Note that $\supp\alpha\ni\alpha_0$ for $\alpha\in\Delta^{im}$.

The subset $L$ is called classical if it corresponds to the classical Dynkin
diagram, which is equivalent to the condition that the group generated by
the reflections with respect to the elements in $L$ is a finite group.

The connected subset $L$ is called \textsl{affine} if it corresponds to
affine Dynkin diagram and in our case it corresponds to $\tilde D_4$
or $\tilde E_6$ or $\tilde E_7$ or $\tilde E_8$ with the following 
Dynkin diagram, respectively.

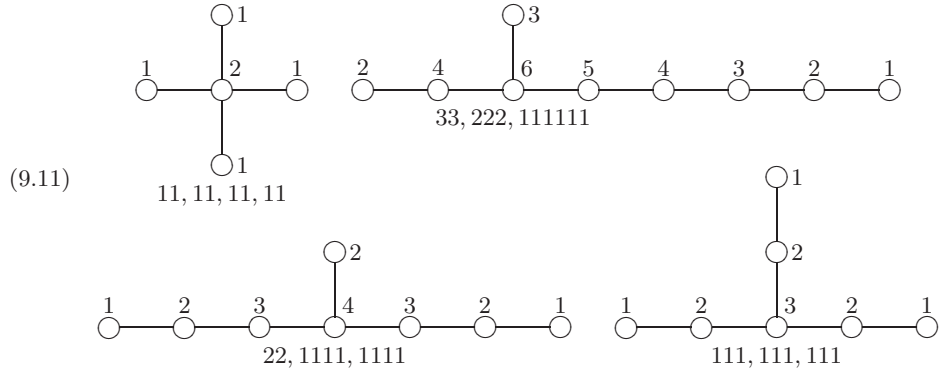
\index{Dynkin diagram}
\index{affine}
\index{00D4@$\tilde D_4,\ \tilde E_6,\ \tilde E_7,\ \tilde E_8$}
{\small
\begin{equation}\index{Dynkin diagram}\label{eq:Dynkinidx0}
\begin{gathered}
\begin{xy}
\ar@{-}               *++!D{1}  *\cir<4pt>{};
             (10,0)   *+!L+!D{2}*\cir<4pt>{}="A",
\ar@{-} "A"; (20,0)   *++!D{1}  *\cir<4pt>{},
\ar@{-} "A"; (10,-10) *++!L{1}  *\cir<4pt>{},
\ar@{-} "A"; (10,10)  *++!L{1}  *\cir<4pt>{},
\ar@{} (10,-14) *{11,11,11,11}
\end{xy}
\quad
\begin{xy}
\ar@{-}               *++!D{2}  *\cir<4pt>{};
             (10,0)   *++!D{4}  *\cir<4pt>{}="A",
\ar@{-} "A"; (20,0)   *+!L+!D{6}*\cir<4pt>{}="B",
\ar@{-} "B"; (30,0)   *++!D{5}  *\cir<4pt>{}="C",
\ar@{-} "C"; (40,0)   *++!D{4}  *\cir<4pt>{}="D",
\ar@{-} "D"; (50,0)   *++!D{3}  *\cir<4pt>{}="E",
\ar@{-} "E"; (60,0)   *++!D{2}  *\cir<4pt>{}="F",
\ar@{-} "F"; (70,0)   *++!D{1}  *\cir<4pt>{},
\ar@{-} "B"; (20,10)  *++!L{3}  *\cir<4pt>{}
\ar@{} (20,-4) *{33,222,111111}
\end{xy}
\allowdisplaybreaks\\[-.8cm]
\begin{xy}
\ar@{-}               *++!D{1}  *\cir<4pt>{};
             (10,0)   *++!D{2}  *\cir<4pt>{}="A",
\ar@{-} "A"; (20,0)   *++!D{3}  *\cir<4pt>{}="B",
\ar@{-} "B"; (30,0)   *+!L+!D{4}*\cir<4pt>{}="C",
\ar@{-} "C"; (40,0)   *++!D{3}  *\cir<4pt>{}="D",
\ar@{-} "D"; (50,0)  *++!D{2}   *\cir<4pt>{}="E",
\ar@{-} "E"; (60,0)  *++!D{1}   *\cir<4pt>{},
\ar@{-} "C"; (30,10)  *++!L{2}  *\cir<4pt>{},
\ar@{} (30,-4) *{22,1111,1111}
\end{xy}
\quad
\begin{xy}
\ar@{-}               *++!D{1}  *\cir<4pt>{};
             (10,0)   *++!D{2}  *\cir<4pt>{}="O",
\ar@{-} "O"; (20,0)   *+!L+!D{3}*\cir<4pt>{}="A",
\ar@{-} "A"; (30,0)   *++!D{2}  *\cir<4pt>{}="B",
\ar@{-} "B"; (40,0)   *++!D{1}  *\cir<4pt>{}
\ar@{-} "A"; (20,10)  *++!L{2}  *\cir<4pt>{}="C",
\ar@{-} "C"; (20,20)  *++!L{1}  *\cir<4pt>{},
\ar@{} (20,-4) *{111,111,111}
\end{xy}
\end{gathered}\end{equation}}
\end{rem}
\smallskip

\noindent
Here the circle correspond to simple roots and the numbers attached
to simple roots are the coefficients $n$ and $n_{j,\nu}$
in the expression \eqref{eq:root2part} of a root $\alpha$.
\medskip

For a tuple of partitions $\mathbf m
 =\bigl(m_{j,\nu}\bigr)_{j\ge 0,\ \nu\ge 1}
\in\mathcal P^{(n)}$, we define
\index{000alpha0@$\alpha_0$, $\alpha_\ell$, $\alpha_{\mathbf m}$}%
\begin{equation}\label{eq:Kazpart}
 \begin{split}
  n_{j,\nu}&:=m_{j,\nu+1}+m_{j,\nu+2}+\cdots,\\
  \alpha_{\mathbf m}&:=n\alpha_0
   + \sum_{j=0}^\infty\sum_{\nu=1}^\infty n_{j,\nu}\alpha_{j,\nu}
 \in Q_+,\\
 \kappa(\alpha_{\mathbf m})&:=\mathbf m.
 \end{split}
\end{equation}
As is given in \cite[Proposition~2.22]{O3} we have
\begin{prop}\label{prop:Kac}
{\rm i)} \ 
\hfill$\idx(\mathbf m,\mathbf m')=
 (\alpha_{\mathbf m}|\alpha_{\mathbf m'})$.\hfill\phantom{ABCDEFG}

{\rm ii)} \ 
Given $i\in I$, we have
$\alpha_{\mathbf m'} = s_i(\alpha_\mathbf m)$ with
\begin{equation*}\index{tuple of partitions!index}
 \mathbf m'=
 \begin{cases}
   \p\mathbf m&(i=0),\\
   (m_{0,1}\dots,m_{j,1}\dots
   \overset{\underset{\smallsmile}\nu}{m_{j,\nu+1}}
   \overset{\underset{\smallsmile}{\nu+1}}{m_{j,\nu}}\dots,\dots)
   &\bigl(i=(j,\nu)\bigr).
 \end{cases}
\end{equation*}
Moreover for $\ell=(\ell_0,\ell_1,\ldots)\in\mathbb Z_{>0}^\infty$
satisfying $\ell_\nu=1$ for $\nu\gg1$ we have
\index{000alpha0@$\alpha_0$, $\alpha_\ell$, $\alpha_{\mathbf m}$}%
\begin{align}
 \alpha_\ell:=\alpha_{\mathbf 1_\ell} &=\alpha_0
     +\sum_{j=0}^\infty\sum_{\nu=1}^{\ell_j-1}\alpha_{j,\nu}
 =\biggl(\prod_{j\ge 0}
   s_{j,\ell_j-1}\cdots s_{j,2}s_{j,1}\biggr)(\alpha_0),\label{eq:alpl}\\
 \alpha_{\p_\ell(\mathbf m)} &=s_{\alpha_\ell}(\alpha_{\mathbf m})
  = \alpha_{\mathbf m} 
  - 2\frac{(\alpha_{\mathbf m}|\alpha_\ell)}
     {(\alpha_\ell|\alpha_\ell)}\alpha_\ell
  = \alpha_{\mathbf m}-(\alpha_{\mathbf m}|\alpha_\ell)\alpha_\ell.\label{eq:M2K}
\end{align}
\end{prop}

Note that
\begin{equation}
 \begin{split}\alpha&=n\alpha_0+\sum_{j\ge 0}
 \sum_{\nu\ge 1}n_{j,\nu}\alpha_{j,\nu}\in\Delta^+
\text{ with }n>0\\
 &\quad
 \Rightarrow \ 
 n\ge n_{j,1}\ge n_{j,2}\ge \cdots\qquad(j=0,1,\ldots).
 \end{split}\label{eq:root2part}
\end{equation}
In fact, for a sufficiently large $K\in\mathbb Z_{>0}$,
we have $n_{j,\mu}=0$ for $\mu\ge K$ and
\[
  s_{\alpha_{j,\nu}+\alpha_{j,\nu+1}+\cdots+\alpha_{j,K}}\alpha
  = \alpha + (n_{j,\nu-1}-n_{j,\nu})
    (\alpha_{j,\nu}+\alpha_{j,\nu+1}+\cdots+\alpha_{j,K})\in\Delta^+
\]
for $\alpha\in\Delta_+$ in \eqref{eq:root2part}, 
which means $n_{j,\nu-1}\ge n_{j,\nu}$ for $\nu\ge1$.
Here we put $n_{j,0}=n$ and $\alpha_{j,0}=\alpha_0$.
Hence for $\alpha\in\Delta_+$ with $\supp\alpha\ni\alpha_0$,
there uniquely exists $\mathbf m\in\mathcal P$ satisfying
$\alpha=\alpha_{\mathbf m}$.

It follows from \eqref{eq:M2K} that
under the identification $\mathcal P\subset Q_+$
with \eqref{eq:Kazpart}, our operation $\p_\ell$ corresponds
to the reflection with respect to the root $\alpha_\ell$.
Moreover the rigid (resp.\ indivisible realizable) tuple 
of partitions corresponds to the positive real root (resp.\ 
indivisible positive root) whose support contains $\alpha_0$,
which were first established by \cite{CB} in the case of Fuchsian
systems of Schlesinger canonical form (cf.~\cite{O3}).

The corresponding objects with this identification are as follows,
which will be clear in this subsection.
Some of them are also explained in \cite{O3}.
\medskip

\noindent
\begin{longtable}{|c|c|}\hline
$\mathcal P$
 & Kac-Moody root system\rule{0pt}{11.5pt}\\ \hline\hline
$\mathbf m$ 
 & $\alpha_{\mathbf m}$ \ (cf.~\eqref{eq:Kazpart})\rule{0pt}{11.5pt}\\ \hline
{$\mathbf m$ : monotone}
 & $\alpha\in Q_+$\,:\, $(\alpha|\beta)\le 0
  \ \ (\forall\beta\in\Pi')$\rule{0pt}{11.5pt}\\ \hline
$\mathbf m$ : realizable
 & $\alpha\in\overline\Delta_+$
\rule{0pt}{11.5pt}\\ \hline
$\mathbf m$ : rigid
 & $\alpha\in\Delta_+^{re}\,:\,\supp\alpha\ni\alpha_0$
\rule{0pt}{11.5pt}\\ \hline
$\mathbf m$ : monotone and fundamental
 & $\alpha\in Q_+$\,:$\,\alpha=\alpha_0\text{ or }
  (\alpha|\beta)\le 0
  \ \ (\forall\beta\in\Pi)$\rule{0pt}{11.5pt}\\ \hline
\raisebox{-6.5pt}{$\mathbf m$ : irreducibly realizable}
 & $\alpha\in\Delta_+,\ 
\,\,\supp\alpha\ni\alpha_0$\\[-4pt] 
 & indivisible or $(\alpha|\alpha)<0$\\ \hline
\raisebox{-6.5pt}{$\mathbf m$ : basic and monotone}
 & $\alpha\in Q_+$\,:\, $(\alpha|\beta)\le 0
  \ \ (\forall\beta\in\Pi)$\rule{0pt}{11.5pt}\\[-2pt]
 & indivisible 
\\ \hline
\raisebox{-6.5pt}{$\mathbf m$\,:\! simply reducible and monotone}
 & $\alpha\in\Delta_+$\,:\,$(\alpha|\alpha_{\mathbf m})=1\quad(\forall\alpha\in
\Delta(\mathbf m))$\rule{0pt}{11.5pt}\\[-2pt]
& $\alpha_0\in\Delta(\mathbf m),\ \ (\alpha|\beta)\le 0\ \ (\forall\beta\in\Pi')$
\\ \hline
$\ord\mathbf m$ & $n_0\,:\,\alpha=n_0\alpha_0+\sum_{i,\nu}n_{i,\nu}\alpha_{i,\nu}$
\rule{0pt}{11.5pt}\\ \hline
$\idx(\mathbf m,\mathbf m')$ 
& $(\alpha_{\mathbf m}|\alpha_{\mathbf m'})$
\rule{0pt}{11.5pt}\\ \hline
$\idx\mathbf m$ 
& $(\alpha_{\mathbf m}|\alpha_{\mathbf m})$
\rule{0pt}{11.5pt}\\ \hline
$d_{\ell}(\mathbf m)$ \ (cf.~\eqref{eq:dm})
& $(\alpha_\ell|\alpha_{\mathbf m})$ \ (cf.~\eqref{eq:alpl})
\rule{0pt}{11.5pt}\\ \hline
$\Pidx\mathbf m+\Pidx\mathbf m'=\Pidx(\mathbf m+\mathbf m')$
 &$(\alpha_{\mathbf m}|\alpha_{\mathbf m'})=-1$
\rule{0pt}{11.5pt}\\ \hline
$(\nu,\nu+1)\in G_j\subset S_\infty'$ \ (cf.~\eqref{eq:S_infty})
 & 
 $s_{j,\nu}\in W_{\!\infty}'$ \ (cf.~\eqref{eq:Kzri})
\rule[-4.5pt]{0pt}{16pt}\\ \hline
$H\simeq \mathfrak S_\infty$ \ (cf.~\eqref{eq:S_infty})
 & $\mathfrak S_\infty$ in \eqref{eq:Kzouter}
\rule{0pt}{11.5pt}\\ \hline 
$\p_{\bf 1}$  & $s_0$\rule{0pt}{11.5pt}\\ \hline
$\p_{\ell}$  & $s_{\alpha_\ell}$  \ (cf.~\eqref{eq:alpl})
\rule{0pt}{11.5pt}\\ \hline
$\langle\p_{\bf 1},\,S_\infty\rangle$
 & $\widetilde W_{\!\infty}\rule{0pt}{11.5pt}$ \ (cf.~\eqref{eq:Kzouter})
\rule{0pt}{11.5pt}\\ \hline
$\{\lambda_{\mathbf m}\}$
 &
$(\Lambda(\lambda),\alpha_{\mathbf m})$ \ (cf.~\eqref{eq:defKacGRS})
\rule{0pt}{11.5pt}\\ \hline
$|\{\lambda_{\mathbf m}\}|$
 &
$(\Lambda(\lambda)+\frac12\alpha_{\mathbf m}|\alpha_{\mathbf m})$
\rule[-4.5pt]{0pt}{16pt}\\ \hline
$\Ad\bigl((x-c_j)^\tau\bigr)$
 & $+\tau\Lambda^0_{0,j}$ \ (cf.~\eqref{eq:defKacGRS})
\rule[-4.5pt]{0pt}{16pt}\\ \hline
\end{longtable}\index{000lambda@$\arrowvert$\textbraceleft$\lambda_{\mathbf m}$\textbraceright$\arrowvert$}
\medskip

Here\index{000Delta00@$\overline\Delta_+$}
\begin{equation}
\overline\Delta_+:=\{k\alpha\,;\,\alpha\in\Delta_+,\ k\in\mathbb Z_{>0},
\ \supp\alpha\ni\alpha_0\},
\end{equation}
$\Delta(\mathbf m)\subset \Delta_+^{re}$ is given in \eqref{def:Deltam}
and $\Lambda(\lambda)\in{\overline{\mathfrak h}}_p$ is defined as follows.

\begin{defn}
Fix a positive integer $p$ which may be $\infty$.  
Put
\begin{equation} 
 I_p:=\{0,\ (j,\nu)\,;\,j=0,1,\dots,p,\ \nu=1,2,\ldots\}\subset I
\index{00Ip@$I_p$}
\end{equation}
for a positive integer $p$ and $I_\infty=I$. 

Let $\mathfrak h_p$ be the $\mathbb R$-vector space of finite linear combinations 
the elements of $\Pi_p:=\{\alpha_i\,;\,i\in\Pi_p\}$ and
let $\mathfrak h_p^\vee$ be the $\mathbb C$-vector space whose elements
are linear combinations of infinite or finite elements of $\Pi_p$,
which is identified with $\Pi_{i\in I_p}\mathbb C\alpha_i$ and contains
${\mathfrak h}_p$.%
\index{00hp@$\mathfrak h_p$, $\mathfrak h_p^\vee$, $\overline{\mathfrak h}_p$}

The element $\Lambda\in\mathfrak h_p^\vee$
naturally defines a linear form of $\mathfrak h_p$ by 
$(\Lambda|\ \cdot\ )$ and the group
$\widetilde W_{\!\infty}$ acts on $\mathfrak h_p^\vee$.
If $p=\infty$, we assume that the element $\Lambda=\xi_0\alpha_0+
\sum \xi_{j,\nu}\alpha_{j,\nu}\in\mathfrak h_\infty^\vee$ always satisfies
$\xi_{j,1}=0$ for sufficiently large $j\in\mathbb Z_{\ge0}$.
Hence we have naturally $\mathfrak h_p^\vee\subset\mathfrak h_{p+1}^\vee$
and $\mathfrak h_\infty^\vee=\bigcup_{j\ge0}\mathfrak h_j^\vee$.

Define the elements of $\mathfrak h_p^\vee$:
\begin{equation}
 \begin{split}
 \Lambda_0&:=\frac12\alpha_0+
   \frac12\sum_{j=0}^p\sum_{\nu=1}^\infty(1-\nu)\alpha_{j,\nu},\\
 \Lambda_{j,\nu}&:=\sum_{i=\nu+1}^\infty(i-\nu)\alpha_{j,i}
   \quad(j=0,\dots,p,\ \nu=0,1,2,\ldots)\\
 \Lambda^0&:=2\Lambda_0-2\Lambda_{0,0}
  =\alpha_0+\sum_{\nu=1}^\infty(1+\nu)\alpha_{0,\nu}+
   \sum_{j=1}^p\sum_{\nu=1}^\infty(1-\nu)\alpha_{j,\nu},\\
 \Lambda^0_{j,k}&:=\Lambda_{j,0}-\Lambda_{k,0}=
  \sum_{\nu=1}^\infty\nu(\alpha_{k,\nu}-\alpha_{j,\nu})
  \quad(0\le j<k\le p),\\
 \Lambda(\lambda)&:=-\Lambda_0-\sum_{j=0}^p\sum_{\nu=1}^\infty
  \Bigl(\sum_{i=1}^\nu\lambda_{j,i}
  \Bigr)\alpha_{j,\nu}\\
   &\;=-\Lambda_0+\sum_{j=0}^p\sum_{\nu=1}^\infty\lambda_{j,\nu}
    (\Lambda_{j,\nu-1}-\Lambda_{j,\nu}).
 \end{split}\label{eq:defKacGRS}
\end{equation}
\end{defn}
Under the above definition we have
\begin{align}\index{000lambda0@$\Lambda_0,\ \Lambda_{j,\nu},\ \Lambda^0,\ \Lambda^0_{j,\nu},\ \Lambda(\lambda)$}
 (\Lambda^0|\alpha)&=(\Lambda^0_{j,k}|\alpha)=0
  \quad(\forall\alpha\in\Pi_p),\\
 (\Lambda_{j,\nu}|\alpha_{j',\nu'})&=\delta_{j,j'}\delta_{\nu,\nu'}
 \quad(j,\,j'=0,1,\ldots,\ \nu,\,\nu'=1,2,\ldots)\allowdisplaybreaks\\
 (\Lambda_0|\alpha_i)&=(\Lambda_{j,0}|\alpha_i)=
   \delta_{i,0} \ \quad(\forall i\in \Pi_p),
 \allowdisplaybreaks\\
 |\{\lambda_{\mathbf m}\}|&=
 (\Lambda(\lambda)+\tfrac12 \alpha_{\mathbf m}|\alpha_{\mathbf m}),
 \allowdisplaybreaks\index{Fuchs relation}\\
 \begin{split}
 s_0(\Lambda(\lambda))&=-\Bigl(\sum_{j=0}^p\lambda_{j,1}-1\Bigr)\alpha_0 
  +\Lambda(\lambda)\\
  &=-\mu\Lambda^0-\Lambda_0-
   \sum_{\nu=1}^\infty\Bigl(\sum_{i=1}^\nu(\lambda_{0,i}-(1+\delta_{i,0})\mu
  \Bigr)\alpha_{0,\nu}\\
  &\quad{}-
  \sum_{j=1}^p\sum_{\nu=1}^\infty
   \Bigl(\sum_{i=1}^\nu(\lambda_{j,i}+(1-\delta_{i,0})\mu)
   \Bigr)\alpha_{j,\nu}
 \end{split}\label{eq:s0Kac}
\end{align}
with $\mu=\sum_{j=0}^p\lambda_{j,1}-1$.

We identify the elements of $\mathfrak h_p^\vee$ if their difference are 
in $\mathbb C\Lambda^0$, namely, consider them in 
$\overline{\mathfrak h}_p:=\mathfrak h_p^\vee/\mathbb C\Lambda^0$.  
Then the elements have the unique representatives in $\mathfrak h_p^\vee$
whose coefficients of $\alpha_0$ equal $-\frac12$.
\index{00hp@$\mathfrak h_p$, $\mathfrak h_p^\vee$, $\overline{\mathfrak h}_p$}
\begin{rem}\label{rem:KacGRS}
i) \ 
If $p<\infty$, we have
\begin{equation}
 \{\Lambda\in\mathfrak h_p^\vee\,;\,(\Lambda|\alpha)=0\quad(\forall \alpha\in\Pi_p)\}
 =\mathbb C\Lambda^0+\sum_{j=1}^p\mathbb C\Lambda^0_{0,j}.
\end{equation}

ii) The invariance of the bilinear form $(\ |\ )$ under the Weyl 
group $W_{\!\infty}$ proves \eqref{eq:midinv}.

iii) The addition given in Theorem~\ref{thm:GRSmid} i) 
corresponds to the map $\Lambda(\lambda)\mapsto\Lambda(\lambda)
+\tau\Lambda^0_{0,j}$ with $\tau\in\mathbb C$ and $1\le j\le p$.

iii) 
Combining the action of $s_{j,\nu}$ on $\mathfrak h_p^\vee$ with 
that of $s_0$, we have
\begin{equation}\label{eq:KacGRS}
  \Lambda(\lambda') - s_{\alpha_\ell}
  \Lambda(\lambda)\in\mathbb C\Lambda^0\text{ \ and \ }
  \alpha_{\mathbf m'}=s_{\alpha_\ell}\alpha_{\mathbf m}
  \quad\text{when \ }\{\lambda'_{\mathbf m'}\}=\p_\ell\{\lambda_\mathbf m\}
\end{equation}
because of \eqref{eq:pellGRS} and \eqref{eq:s0Kac}.
\end{rem}
Thus we have the following theorem.
\begin{thm}\label{thm:KatzKac}
Under the above notation we have
the commutative diagram
\begin{equation*}
 \begin{matrix}
\bigl\{P_{\mathbf m}:\,\text{Fuchsian differential operators with }
\{\lambda_{\mathbf m}\}\bigr\}
  &\rightarrow
   &\bigl\{(\Lambda(\lambda),\alpha_{\mathbf m})\,;\,\alpha_{\mathbf m}\in\overline\Delta_+\bigr\}
 \\
 \\
  \downarrow \text{fractional operations}
  &\circlearrowright&\quad \downarrow {W_{\!\infty}\text{-action},\ 
  +\tau\Lambda^0_{0,j}}\\
 \\
\bigl\{P_{\mathbf m}:\,\text{Fuchsian differential operators with }
\{\lambda_{\mathbf m}\}\bigr\}
   & \rightarrow
   & \bigl\{(\Lambda(\lambda), \alpha_{\mathbf m})\,;\,\alpha_{\mathbf m}
 \in\overline\Delta_+\bigr\}.
 \end{matrix}
\end{equation*}
Here the defining domain of $w\in W_{\!\infty}$ is
$\{\alpha\in\overline\Delta_+\,;\,w\alpha\in\overline\Delta_+\}$.
\end{thm}
\begin{proof}
Let $T_i$ denote the corresponding operation on 
$\{(P_{\mathbf m},\{\lambda_{\mathbf m}\})\}$ for $s_i\in W_\infty$ 
with $i\in I$.
Then $T_0$ corresponds to $\p_{\mathbf 1}$ and when $i\in I'$, $T_i$ is naturally defined 
and it doesn't change  $P_{\mathbf m}$.
The fractional transformation of the Fuchsian operators and their Riemann schemes
corresponding to an element $w\in W_{\!\infty}$ is defined through the 
expression of $w$ by the product of simple reflections.
It is clear that the transformation of their Riemann schemes do not depend on
the expression. 

Let $i\in I$ and $j\in I$.
We want to prove that $(T_iT_j)^k=id$ if $(s_is_j)^k=id$ for a non-negative
integer $k$.  
Note that $T_i^2=id$ and the addition commutes with $T_i$.
Since $T_i=id$ if $i\in I'$, we have only to prove that
$(T_{j,1}T_0)^3=id$.
Moreover Proposition~\ref{prop:coordf} assures that we may assume $j=0$.

Let $P$ be a Fuchsian differential operator with the Riemann scheme 
\eqref{eq:GRS}.
Applying suitable additions to $P$, 
we may assume $\lambda_{j,1}=0$ for $j\ge 1$ to prove 
$(T_{0,1}T_0)^3P=P$ and then this easily follows from the definition
of $\p_{\mathbf 1}$ (cf.~\eqref{eq:opred})
and the relation
\begin{align*}
&\begin{Bmatrix}
\infty & c_j\ (1\le j\le p)\\
[\lambda_{0,1}]_{(m_{0,1})}&[0]_{(m_{j,1})}\\
[\lambda_{0,2}]_{(m_{0,2})}&[\lambda_{j,2}]_{(m_{j,2})}\\
[\lambda_{0,\nu}]_{(m_{0,\nu})}&[\lambda_{j,\nu}]_{(m_{j,\nu})}
\end{Bmatrix}\quad(d=m_{0,1}+\cdots+m_{p,1}-\ord\mathbf m)
 \allowdisplaybreaks\\
&\quad\xrightarrow[\p^{1-\lambda_{0,1}}]{T_{0,1}T_0}
\begin{Bmatrix}
\infty & c_j\ (1\le j\le p)\\
[\lambda_{0,2}-\lambda_{0,1}+1]_{(m_{0,1})}&[0]_{(m_{j,1}-d)}\\
[-\lambda_{0,1}+2]_{(m_{0,2}-d)}&[\lambda_{j,2}+\lambda_{0,1}-1]_{(m_{j,2})}\\
[\lambda_{0,\nu}-\lambda_{0,1}+1]_{(m_{0,\nu})}&[\lambda_{j,\nu}+\lambda_{0,1}-1]_{(m_{j,\nu})}
\end{Bmatrix}
 \allowdisplaybreaks\\
&\quad
\xrightarrow[\p^{\lambda_{0,1}-\lambda_{0,2}}]{T_{0,1}T_0}
\begin{Bmatrix}
\infty & c_j\ (1\le j\le p)\\
[-\lambda_{0,2}+2]_{(m_{0,1}-d)}&[0]_{(m_{j,1}+m_{0,1}-m_{0,2}-d)}\\
[\lambda_{0,1}-\lambda_{0,2}+1]_{(m_{0,1})}&[\lambda_{j,2}+\lambda_{0,2}-1]_{(m_{j,2})}\\
[\lambda_{0,\nu}-\lambda_{0,2}+1]_{(m_{0,\nu})}&[\lambda_{j,\nu}+\lambda_{0,2}-1]_{(m_{j,\nu})}
\end{Bmatrix}
 \allowdisplaybreaks\\
&\quad
\xrightarrow[\p^{\lambda_{0,2}-1}]{T_{0,1}T_0}
\begin{Bmatrix}
\infty & c_j\ (1\le j\le p)\\
[\lambda_{0,1}]_{(m_{0,1})}&[0]_{(m_{j,1})}\\
[\lambda_{0,2}]_{(m_{0,2})}&[\lambda_{j,2}]_{(m_{j,2})}\\
[\lambda_{0,\nu}]_{(m_{0,\nu})}&[\lambda_{j,\nu}]_{(m_{j,\nu})}
\end{Bmatrix}
\end{align*}
and $(T_{0,1}T_0)^3P
 \in \mathbb C[x]
 \Ad(\p^{\lambda_{0,2}-1})\circ\Ad(\p^{\lambda_{0,2}-\lambda_{0,1}})\circ
 \Ad(\p^{1-\ \lambda_{0,1}})\Red P=\mathbb C[x]\Red P$.
\end{proof}
\begin{defn}
For an element $w$ of the Weyl group $W_{\!\infty}$ we put
\begin{equation}
  \Delta(w):=\Delta^{re}_+\cap w^{-1}\Delta^{re}_-.
\end{equation}
\index{000Delta0s@$\Delta(\mathbf m)$, $\Delta(w)$}
If $w=s_{i_1}s_{i_2}\cdots s_{i_k}$ with $i_\nu\in I$ 
is the \textsl{minimal expression}
\index{Weyl group!minimal expression} of $w$ as the products of simple 
reflections which means $k$ is minimal by definition, we have
\index{minimal expression}
\begin{equation}\label{eq:Lw}
 \Delta(w)=\bigl\{
   \alpha_{i_k},s_{i_k}(\alpha_{i_{k-1}}),s_{i_k}s_{i_{k-1}}(\alpha_{i_{k-2}}),
   \ldots,s_{i_k}\cdots s_{i_2}(\alpha_{i_1})\bigr\}.
\end{equation}
\end{defn}
The number of the elements of $\Delta(w)$ equals the number of the simple 
reflections in the minimal expression of $w$, which is called the \textsl{length} 
of $w$ and denoted by $L(w)$.\index{00Lw@$L(w)$}
\index{Weyl group!length}
The equality \eqref{eq:Lw} follows from the following lemma.
\begin{lem}\label{lem:minrep}
Fix $w\in W_{\!\infty}$ and $i\in I$. 
If $\alpha_i\in\Delta(w)$, there exists 
a minimal expression $w=s_{i'_1}s_{i'_2}\cdots s_{i'_k}$ 
with $s_{i'_k}=s_i$ and $L(ws_i)=L(w)-1$ and $\Delta(ws_i)
 =s_i\bigl(\Delta(w)\setminus\{\alpha_i\}\bigr)$.
If $\alpha_i\notin\Delta(w)$, $L(ws_i)=L(w)+1$ and
$\Delta(ws_i)=s_i\Delta(w)\cup\{\alpha_i\}$.
Moreover if $v\in W_{\!\infty}$ satisfies $\Delta(v)=\Delta(w)$,
then $v=w$.
\end{lem}
\begin{proof}
The proof is standard as in the case of classical root system,
which follows from the fact that the condition
$\alpha_i=s_{i_k}\cdots s_{i_{\ell+1}}(\alpha_{i_\ell})$ implies
\begin{equation}
 s_i=s_{i_k}\cdots s_{i_{\ell+1}}s_{i_\ell}s_{i_{\ell+1}}\cdots s_{i_k}
\end{equation}
and then $w=ws_is_i=s_{i_1}\cdots s_{i_{\ell-1}}s_{i_{\ell+1}}\cdots s_{i_k}s_i$.
\end{proof}

\begin{defn}\label{def:wm}
For $\alpha\in Q$, put
\begin{equation}
 h(\alpha):=n_0+\sum_{j\ge 0}\sum_{\nu\ge 1}n_{j,\nu}
 \text{ \ if \ }
 \alpha=n_0\alpha_0+\sum_{j\ge 0}\sum_{\nu\ge 1}n_{j,\nu}\alpha_{j,\nu}\in Q.
\end{equation}\index{00halpha@$h(\alpha)$}%
Suppose $\mathbf m\in\mathcal P_{p+1}$ is irreducibly realizable.
Note that $sf \mathbf m$ is the monotone fundamental element
determined by $\mathbf m$, namely,
$\alpha_{sf\mathbf m}$ is the unique element of
$W\alpha_{\mathbf m}\cap \bigl(B\cup\{\alpha_0\}\bigr)$.
We inductively define $w_{\mathbf m}\in W_{\!\infty}$ satisfying
$w_{\mathbf m}\alpha_{\mathbf m}=\alpha_{sf{\mathbf m}}$.
We may assume $w_{\mathbf m'}$ has already defined if $h(\alpha_{\mathbf m'})
<h(\alpha_{\mathbf m})$.
If $\mathbf m$ is not monotone, there exists $i\in I\setminus\{0\}$
such that $(\alpha_{\mathbf m}|\alpha_i)>0$ 
and then $w_{\mathbf m}=w_{\mathbf m'}s_i$
with $\alpha_{\mathbf m'}=s_i\alpha_{\mathbf m}$.
If $\mathbf m$ is monotone and $\mathbf m\ne f\mathbf m$,
$w_{\mathbf m}=w_{\p\mathbf m}s_0$.
\index{00wm@$w_{\mathbf m}$}

We moreover define
\begin{align}\label{def:Deltam}
 \Delta(\mathbf m)&:= \Delta(w_{\mathbf m}).
\end{align}
\end{defn}
\index{000Delta0s@$\Delta(\mathbf m)$, $\Delta(w)$}
Suppose $\mathbf m$ is monotone, irreducibly realizable
and $\mathbf m\ne sf{\mathbf m}$.
We define $w_{\mathbf m}$ so that there exists 
$K\in\mathbb Z_{>0}$ and $v_1,\dots,v_K\in W'_{\!\infty}$
satisfying
\begin{equation}\label{eq:vmrp}
 \begin{split}
  w_{\mathbf m}&=v_Ks_0\cdots v_2s_0v_1s_0,\\
  (v_k s_0\cdots v_1s_0\alpha_{\mathbf m}|\alpha)&\le 0
  \quad(\forall\alpha\in\Pi\setminus\{0\},\ k=1,\dots,K),
 \end{split}
\end{equation}
which uniquely characterizes $w_{\mathbf m}$.
Note that
\begin{equation}\label{eq:vmrps}
 v_k s_0\cdots v_1s_0\alpha_{\mathbf m}=\alpha_{(s\p)^k\mathbf m}
 \quad(k=1,\dots,K).
\end{equation}

The following proposition gives the correspondence between the reduction of 
realizable tuples of partitions and the minimal expressions of the elements
of the Weyl group.

\begin{prop}\label{prop:wm}
Definition~\ref{def:wm} naturally gives the product expression
$w_{\mathbf m}=s_{i_1}\cdots s_{i_k}$ with $i_\nu\in I\ \ (1\le\nu\le k)$.

{\rm i) } We have
\begin{align}
 L(w_{\mathbf m})&=k,\\
 (\alpha|\alpha_{\mathbf m})&>0\quad(\forall\alpha\in\Delta(\mathbf m)),
 \label{eq:amplus}\\
 h(\alpha_{\mathbf m})&=h(\alpha_{sf\mathbf m})+
 \sum_{\alpha\in\Delta(\mathbf m)}(\alpha|\alpha_{\mathbf m}).
 \label{eq:dht}
\end{align}
Moreover $\alpha_0\in\supp\alpha$ for $\alpha\in\Delta(\mathbf m)$ 
if\/ $\mathbf m$ is monotone.

{\rm ii)} Suppose\/ $\mathbf m$ is monotone and 
$f\mathbf m\ne\mathbf m$. 
Fix maximal integers $\nu_j$ 
such that $m_{j,1}-d_{max}(\mathbf m)<m_{j,\nu_j+1}$ for $j=0,1,\dots$
Then 
\begin{gather}
 \begin{aligned}
 \Delta(\mathbf m)&=s_0\Bigl(\prod_{\substack{j\ge 0\\ \nu_j>0}}
 s_{j,1}\cdots s_{j,\nu_j}\Bigr)
 \Delta(s\p\mathbf m)\cup\{\alpha_0\}\label{eq:Ddelta}\\
 &\qquad\cup\{\alpha_0+\alpha_{j,1}+\cdots+\alpha_{j,\nu}\,;\,
   1\le \nu\le\nu_j\text{ and } j=0,1,\ldots\},
 \end{aligned}\\
 (\alpha_0+\alpha_{j,1}+\cdots+\alpha_{j,\nu}|\alpha_{\mathbf m})=
 d_{max}(\mathbf m)+m_{j,\nu+1}-m_{j,1}\quad(\nu\ge0).
 \label{eq:Dd2}
\end{gather}

{\rm iii)}
Suppose\/ $\mathbf m$ is not rigid. Then 
$\Delta(\mathbf m)=\{\alpha\in\Delta^{re}_+\,;\,(\alpha|\alpha_{\mathbf m})>0\}$.

{\rm iv)}
Suppose\/ $\mathbf m$ is rigid.
Let $\alpha\in\Delta^{re}_+$ satisfying $(\alpha|\alpha_{\mathbf m})>0$
and $s_{\alpha}(\alpha_{\mathbf m})\in\Delta_+$.
Then
\begin{equation}
\begin{cases}
\alpha\in\Delta(\mathbf m)
 &\text{if \ }(\alpha|\alpha_{\mathbf m})>1,\\
\#\Bigl(\{\alpha,\,\alpha_{\mathbf m}-\alpha\}\cap\Delta(\mathbf m)\Bigr)
 = 1
 &\text{if \ }(\alpha|\alpha_{\mathbf m})=1.
\end{cases}
\end{equation}
Moreover if a root $\gamma\in\Delta(\mathbf m)$ satisfies 
$(\gamma|\alpha_{\mathbf m})=1$, then 
$\alpha_{\mathbf m}-\gamma\in\Delta^{re}_+$ and $\alpha_0\in
\supp(\alpha_{\mathbf m}-\gamma)$.

{\rm v)}
$w_{\mathbf m}$ is the unique element with the minimal length satisfying
$w_{\mathbf m}\alpha_{\mathbf m}=\alpha_{sf{\mathbf m}}$.
\end{prop}
\begin{proof}
Since $h(s_{i'}\alpha)-h(\alpha)=-(\alpha_{i'}|\alpha)
 =(s_{i'}\alpha_{i'}|\alpha)$,
we have
\[
 \begin{split}
 h(s_{i'_\ell}\cdots s_{i'_1}\alpha)-h(\alpha)
  &=\sum_{\nu=1}^\ell\Bigl(h(s_{i'_\nu}\cdots s_{i'_1}\alpha)
  -h(s_{i'_{\nu-1}}\cdots s_{i'_1}\alpha)\Bigr)
 \\
  &=\sum_{\nu=1}^\ell
   (\alpha_{i'_\nu}|s_{i'_\nu}\cdots s_{i'_1}\alpha)
  = \sum_{\nu=1}^\ell
  (s_{i'_\ell}\cdots s_{i'_{\nu+1}}\alpha_{i'_\nu}|
  s_{i'_\ell}\cdots s_{i'_1}\alpha)
\end{split}
\]
for $i',\,i'_{\nu}\in I$ and $\alpha\in\Delta$.

i)
We show by the induction on $k$.
We may assume $k\ge 1$.
Put $w'=s_{i_1}\cdots s_{i_{k-1}}$ and 
$\alpha_{\mathbf m'}=s_{i_k}\alpha_{\mathbf m}$
and $\alpha(\nu)=s_{i_{k-1}}\cdots s_{i_{\nu+1}}\alpha_{i_\nu}$
for $\nu=1,\dots,k-1$.
The hypothesis of the induction assures
$L(w')=k-1$, $\Delta(\mathbf m')=\{\alpha(1),\dots,\alpha(k-1)\}$
and $(\alpha(\nu)|\alpha_{\mathbf m'})>0$ for $\nu=1,\dots,k-1$.
If $L(w_{\mathbf m})\ne k$, there exists $\ell$ such that 
$\alpha_{i_k}=\alpha(\ell)$ and 
$w_{\mathbf m}=s_{i_1}\cdots s_{i_{\ell-1}}s_{i_{\ell+1}}\cdots s_{i_{k-1}}$
is a minimal expression.
Then $h(\alpha_{\mathbf m})-h(\alpha_{\mathbf m'})
=-(\alpha_{i_k}|\alpha_{\mathbf m'})=-(\alpha(\ell)|\alpha_{\mathbf m'})<0$, which contradicts to the definition
of $w_{\mathbf m}$. Hence we have i).
Note that \eqref{eq:amplus} implies $\supp\alpha\ni\alpha_0$
if $\alpha\in\Delta(\mathbf m)$ and $\mathbf m$ is monotone.

ii) The equality \eqref{eq:Ddelta} follows from
\[
  \Delta(\p\mathbf m)\cap\!\!
  \sum_{\alpha\in\Pi\setminus\{0\}}\!\!\!\mathbb Z\alpha
 =\{\alpha_{j,1}+\cdots+\alpha_{j,\nu_j}\,;\,\nu=1,\dots,\nu_j,\ \nu_j>0
  \text{ and }j=0,1,\ldots\}
\]
because $\Delta(\mathbf m)=s_0\Delta(\p\mathbf m)\cup\{\alpha_0\}$
and $\Bigl(\prod_{\substack{j\ge 0\\ \nu_j>0}}
 s_{j,\nu_j}\cdots s_{j,1}\Bigr)
 \alpha_{\p\mathbf m}=\alpha_{s\p\mathbf m}$.

The equality \eqref{eq:Dd2} is clear because
$(\alpha_0|\alpha_{\mathbf m})=d_{\mathbf 1}(\mathbf m)=d_{max}(\mathbf m)$
and $(\alpha_{j,\nu}|\alpha_{\mathbf m})=m_{j,\nu+1}-m_{j,\nu}$.

iii) 
Note that $\gamma\in\Delta(\mathbf m)$ satisfies 
$(\gamma|\alpha_{\mathbf m})>0$.

Put $w_\nu=s_{i_{\nu+1}}\cdots s_{i_{k-1}}s_{i_k}$ for 
$\nu=0,\dots,k$.
Then $w_{\mathbf m}=w_0$ and 
$\Delta(\mathbf m)=\{w_\nu^{-1}\alpha_{i_\nu}\,;\,\nu=1,\dots,k\}$.
Moreover
$w_{\nu'}w_\nu^{-1}\alpha_{i_\nu}\in\Delta^{re}_-$ if and only if
$0\le \nu'<\nu$.

Suppose $\mathbf m$ is not rigid.
Let $\alpha\in\Delta^{re}_+$ with $(\alpha|\alpha_{\mathbf m})>0$.
Since $(w_{\mathbf m}\alpha|\alpha_{\overline{\mathbf m}})>0$, 
$w_{\mathbf m}\alpha\in\Delta^{re}_-$.
Hence there exists $\nu$ such that $w_{\nu}\alpha\in\Delta_+$ and
$w_{\nu-1}\alpha\in\Delta_-$, which implies
$w_{\nu}\alpha=\alpha_{i_\nu}$ and the claim.

iv)
Suppose $\mathbf m$ is rigid.
Let $\alpha\in\Delta^{re}_+$.
Put $\ell=(\alpha|\alpha_{\mathbf m})$.
Suppose $\ell>0$ and $\beta:=s_{\alpha}\alpha_{\mathbf m}\in\Delta_+$.
Then $\alpha_{\mathbf m}=\ell\alpha+\beta$,
$\alpha_0=\ell w_{\mathbf m}\alpha+w_{\mathbf m}\beta$ and
$(\beta|\alpha_{\mathbf m})
=(\alpha_{\mathbf m}-\ell\alpha|\alpha_{\mathbf m})=2-\ell^2$.
Hence if $\ell\ge 2$, $\mathbb R\beta\cap\Delta(\mathbf m)=\emptyset$ 
and the same argument as in the proof of iii) 
assures $\alpha\in\Delta(\mathbf m)$.

Suppose $\ell=1$. There exists $\nu$ such that 
$w_{\nu}\alpha$ or $w_{\nu}\beta$
equals $\alpha_{i_\nu}$.
We may assume $w_{\nu}^{-1}\alpha=\alpha_{i_\nu}$.
Then $\alpha\in\Delta(\mathbf m)$.

Suppose there exists $w_{\nu'}\beta=\alpha_{i_{\nu'}}$.
We may assume $\nu'< \nu$.  Then $w_{\nu'}\alpha_{\mathbf m}=
w_{\nu'-1}\alpha+w_{\nu'-1}\beta\in\Delta^{re}_-$,
which contradicts to the definition of $w_\nu$.
Hence $w_{\nu'}\beta=\alpha_{i_{\nu'}}$ for $\nu'=1,\dots,k$ and
therefore $\beta\notin \Delta(\mathbf m)$.

Let $\gamma=w_\nu^{-1}\alpha_{i_\nu}\in\Delta(\mathbf m)$ and 
$(\gamma|\alpha_{\mathbf m})=1$.
Put $\beta=\alpha_{\mathbf m}-\alpha=s_{\alpha}\alpha_{\mathbf m}$.
Then $w_{\nu-1}\alpha_{\mathbf m}=w_{\nu}\beta\in\Delta^{re}_+$.
Since $\beta\notin\Delta(\mathbf m)$, we have $\beta\in\Delta^{re}_+$.

Replacing $\mathbf m$ by $s\mathbf m$, 
we may assume $\mathbf m$ is monotone to prove $\alpha_0\in\supp\beta$.
Since $(\beta|\alpha_{\mathbf m})=1$ and $(\alpha_i|\alpha_{\mathbf m})\le 0$
for $i\in I\setminus\{0\}$, we have $\alpha_0\in\supp\beta$.

v)
The uniqueness of $w_{\mathbf m}$ follows from iii) when $\mathbf m$ is not rigid.
It follows from \eqref{eq:amplus}, Theorem~\ref{thm:Nuida} and 
Corollary~\ref{cor:Nuida} when $\mathbf m$ is rigid.
\end{proof}
\begin{cor}
Let\/ $\mathbf m$, $\mathbf m'$, $\mathbf m''\in\mathcal P$ and
$k\in\mathbb Z_{>0}$ such that
\begin{equation}
 \mathbf m=k\mathbf m'+\mathbf m'',\ \idx\mathbf m=\idx\mathbf m''
 \text{ and\/ $\mathbf m'$ is rigid}.
\end{equation}
Then\/ $\mathbf m$ is irreducibly realizable if and only if
so is $\mathbf m''$.

Suppose\/ $\mathbf m$ is irreducibly realizable.
If\/ $\idx\mathbf m\le 0$ or $k>1$, then\/ 
$\mathbf m'\in\Delta(\mathbf m)$.
If $\idx\mathbf m=2$, then\/
$\{\alpha_{\mathbf m'},\,\alpha_{\mathbf m''}\}\cap\Delta(\mathbf m)
=\{\alpha_{\mathbf m'}\}$ or $\{\alpha_{\mathbf m''}\}$.
\end{cor}
\begin{proof}
The assumption implies $(\alpha_{\mathbf m}|\alpha_{\mathbf m})=
2k^2+2k(\alpha_{\mathbf m'}|\alpha_{\mathbf m''})+
(\alpha_{\mathbf m''}|\alpha_{\mathbf m''})$ and hence
$(\alpha_{\mathbf m'}|\alpha_{\mathbf m''})=-k$ and
$s_{\alpha_{\mathbf m'}}\alpha_{\mathbf m''}=\alpha_{\mathbf m}$.
Thus we have the first claim (cf.~Theorem~\ref{thm:KatzKac}).
The remaining claims follow from Proposition~\ref{prop:wm}.
\end{proof}
\begin{rem}\label{rem:length}
i) \ 
In general, $\gamma\in\Delta(\mathbf m)$ does not always imply
$s_\gamma\alpha_{\mathbf m}\in\Delta_+$.

Put $\mathbf m=32,32,32,32$, $\mathbf m'=10,10,10,10$
and $\mathbf m''=01,01,01,01$.
Putting $v=s_{0,1}s_{1,1}s_{2,1}s_{3,1}$, we have
$\alpha_{\mathbf m'}=\alpha_0$,
$\alpha_{\mathbf m''}=v\alpha_0$,
$(\alpha_{\mathbf m'}|\alpha_{\mathbf m''})=-2$,
$s_0\alpha_{\mathbf m''}=2\alpha_{\mathbf m'}+\alpha_{\mathbf m''}$,
$vs_0\alpha_{\mathbf m''}=\alpha_0+2\alpha_{\mathbf m''}$ and
$s_0vs_0v\alpha_0=s_0vs_0\alpha_{\mathbf m''}
=3\alpha_{\mathbf m'}+2\alpha_{\mathbf m''}=\alpha_{\mathbf m}$.

Then $\gamma:=
s_0v\alpha_0=2\alpha_{\mathbf m'}+\alpha_{\mathbf m''}\in\Delta(\mathbf m)$,
$(\gamma|\alpha_{\mathbf m})=
(s_0v\alpha_{\mathbf m'}|s_0vs_0v\alpha_{\mathbf m'})
=(\alpha_{\mathbf m'}|s_0v\alpha_{\mathbf m'})
=(\alpha_{\mathbf m'}|2\alpha_{\mathbf m'}+\alpha_{\mathbf m''})
=2$ and $s_\gamma(\alpha_{\mathbf m})=
(3\alpha_{\mathbf m'}+2\alpha_{\mathbf m''})
-2(2\alpha_{\mathbf m'}+\alpha_{\mathbf m''})=
-\alpha_{\mathbf m'}\in\Delta_-$.

{\rm ii)} 
Define
\begin{equation}\label{eq:tDelta}
 \index{000Delta1@$[\Delta(\mathbf m)]$}%
 [\Delta(\mathbf m)]:=
 \bigl\{(\alpha|\alpha_{\mathbf m})\,;\,\alpha\in\Delta(\mathbf m)\bigr\}.
\end{equation}
Then $[\Delta(\mathbf m)]$ gives a partition of the non-negative integer
$h(\alpha_{\mathbf m})-h(sf\mathbf m)$, which we call \textsl{the type of}\/ 
$\Delta(\mathbf m)$.
It follows from \eqref{eq:dht} that
\begin{equation}
 \#\Delta(\mathbf m)\le h(\alpha_{\mathbf m})-h(sf\mathbf m)
\end{equation}
for a realizable tuple $\mathbf m$
and the equality holds in the above if $\mathbf m$ is monotone and simply 
reducible.  Moreover we have
\begin{align}
 [\Delta(\mathbf m)]&=[\Delta(s\p\mathbf m)]\cup\{d(\mathbf m)\}
 \cup\bigcup_{j=0}^p\{m_{j,\nu}-m_{j,1}-d(\mathbf m)\in\mathbb Z_{>0}\,;\,
  \nu>1\},\label{eq:DmInd}\\
 \#\Delta(\mathbf m)&=
 \#\Delta(s\p\mathbf m)
  +\sum_{j=0}^p\Bigl(\min\bigl\{\nu\,;\,m_{j,\nu}>m_{j,1}-d(\mathbf m)\bigr\}
  -1\Bigr)+1
 ,\\
 h(\mathbf m)&=h(sf\mathbf m)+\sum_{i\in[\Delta(\mathbf m)]}i\label{eq:sumDelta}
\end{align}
if $\mathbf m\in\mathcal P_{p+1}$ is monotone, irreducibly realizable and 
not fundamental.
Here we use the notation in Definitions~\ref{def:Sinfty}, \ref{def:pell} and 
\ref{def:fund}.  
For example,
\[
\begin{tabular}{|c|c|c|c|}\hline
type&$\mathbf m$ & $h(\alpha_{\mathbf m})$ & $\#\Delta(\mathbf m)$\\
\hline\hline
$H_n\rule[-2pt]{0pt}{12pt}$ & $1^n,1^n,n-11$ & $n^2+1$ & $n^2$ \\ \hline
$EO_{2m}\rule[-5pt]{0pt}{15pt}$ 
 & $1^{2m},mm,mm-11$ & {\small $2m^2+3m+1$} & $\binom{2m}{2}+4m$
 \\ \hline
$\!EO_{2m+1}\rule[-5pt]{0pt}{15pt}\!$
 & $1^{2m+1},m+1m,mm1$ & {\small $2m^2+5m+3$} & $\binom{2m+1}{2}+4m+2$\\ \hline
$X_6\rule[-2pt]{0pt}{12pt}$ & $111111,222,42$ & 29 & 28 \\ \hline
\rule[-2pt]{0pt}{12pt}&$21111,222,33$ & 25 & 24\\ \hline
$\!P_n\rule[-2pt]{0pt}{12pt}\!$
 & {\small $n-11,n-11,\ldots\in\mathcal P_{n+1}^{(n)}$} & $2n+1$ & 
 {\!\small$[\Delta(\mathbf m)]:$\,}$1^{n+1}\cdot${\small$(n-1)$}\!\\ \hline
 $\!P_{4,2m+1}\!\rule[-2pt]{0pt}{12pt}$
 & {\small $m+1m,m+1m,m+1m,m+1m$} & $6m+1$ & 
 {\small$[\Delta(\mathbf m)]:\,$}$1^{4m}\cdot2^m$\\ \hline
\end{tabular}
\]
\end{rem}

Suppose $\mathbf m\in\mathcal P_{p+1}$ is basic.  
We may assume \eqref{eq:NTP}.  
Suppose $(\alpha_{\mathbf m}|\alpha_0)=0$, which is equivalent
to $\sum_{j=0}^p m_{j,1}=(p-1)\ord\mathbf m$.
Let $k_j$ be positive integers such that
\begin{equation}\label{eq:classic}
(\alpha_{\mathbf m}|\alpha_{j,\nu})=0\text{ \  for  \ }
1\le \nu<k_j \text{ \ and \ }
(\alpha_{\mathbf m}|\alpha_{j,k_j})<0,
\end{equation}
which is equivalent to
$m_{j,1}=m_{j,2}=\cdots=m_{j,k_j}>m_{j,k_j+1}$
for $j=0,\dots,p$. Then
\begin{equation}\label{eq:R2ineq}
 \sum_{j=0}^p\frac1{k_j}\ge 
 \sum_{j=0}^p\frac{m_{j,1}}{\ord\mathbf m}=p-1.
\end{equation}
If the equality holds in the above, we have $k_j\ge 2$ and $m_{j,k_j+1}=0$
and therefore $\mathbf m$ is of one of the types 
$\tilde D_4$ or $\tilde E_6$ or $\tilde E_7$ or $\tilde E_8$.
Hence if $\idx\mathbf m<0$, the set $\{k_j\,;\,0\le j\le p,\ k_j> 1\}$
equals one of the set $\emptyset$, $\{2\}$, $\{2,\nu\}$ with $2\le \nu\le 5$, 
$\{3,\nu\}$ with $3\le\nu\le 5$,
$\{2,2,\nu\}$ with $2\le\nu\le 5$ and
$\{2,3,\nu\}$ with $3\le\nu\le 5$.
In this case the corresponding Dynkin diagram of 
$\{\alpha_0,\alpha_{j,\nu}\,;\,1\le\nu<k_j,\ j=0,\dots,p\}$ is one of 
the types $A_\nu$ with $1\le \nu\le 6$, $D_\nu$ with $4\le\nu\le 7$
and $E_\nu$ with $6\le\nu\le 8$.
Thus we have the following remark.
\begin{rem}\label{rem:classinbas}
Suppose a tuple $\mathbf m\in\mathcal P_{p+1}^{(n)}$ is basic and monotone.
The subgroup of $W_{\!\infty}$ generated by reflections with
respect to $\alpha_\ell$ (cf.~\eqref{eq:alpl}) 
which satisfy $(\alpha_{\mathbf m}|\alpha_\ell)=0$ 
is infinite if and only if $\idx\mathbf m=0$.

For a realizable monotone tuple $\mathbf m\in\mathcal P$, we define%
\index{000Pi@$\Pi,\ \Pi',\ \Pi(\mathbf m)$}
\begin{equation}\label{eq:Pim}
 \Pi(\mathbf m):=
 \{\alpha_{j,\nu}\in\supp\alpha_{\mathbf m}\,;\,
  m_{j,\nu}=m_{j,\nu+1}\}\cup
 \begin{cases}
  \{\alpha_0\}&(d_{\mathbf 1}(\mathbf m)=0),\\
  \emptyset&(d_{\mathbf 1}(\mathbf m)\ne 0).
 \end{cases}
\end{equation}
Note that the condition $(\alpha_{\mathbf m}|\alpha_\ell)=0$,
which is equivalent to say that $\alpha_\ell$ is a root of the root space
with the fundamental system $\Pi(\mathbf m)$,
means that the corresponding middle convolution $\p_\ell$
keeps the spectral type invariant.
\end{rem}

\subsection{Fundamental tuples}\label{sec:basic}
We will prove some inequalities \eqref{eq:bineq} and \eqref{eq:b4ineq} 
for fundamental tuples which are announced in \cite{O3}.
\begin{prop}\label{prop:Bineq}
Let\/ $\mathbf m\in\mathcal P_{p+1}\setminus\mathcal P_p$ 
be a fundamental tuple.
Then
\begin{align}\label{eq:bineq}
 \ord\mathbf m &\le 3|\idx\mathbf m| + 6,\allowdisplaybreaks\\
 \label{eq:b4ineq}
 \ord\mathbf  m&\le |\idx\mathbf m|+2\quad\text{if \ }
 p\ge 3,\allowdisplaybreaks\\
 p&\le\tfrac12|\idx\mathbf m|+3.\label{eq:pineq}
\end{align}
\end{prop}
\begin{exmp}\label{ex:special}
For a positive integer $m$ we have special 4 elements
\begin{equation}\index{00D4z@$D_4^{(m)},\ E_6^{(m)},\ E_7^{(m)},\ E_8^{(m)}$}
\begin{aligned}\label{eq:Qsp}
 &D_4^{(m)}: m^2,m^2,m^2,m(m-1)1&& 
 \quad E_6^{(m)}: m^3,m^3,m^2(m-1)1\\ 
 &E_7^{(m)}: (2m)^2,m^4,m^3(m-1)1&& 
 \quad E_8^{(m)}: (3m)^2,(2m)^3,m^5(m-1)1
\end{aligned}
\end{equation}
with orders $2m$, $3m$, $4m$ and $6m$, 
respectively, and index of rigidity $2-2m$.

Note that $E_8^{(m)}$, $D_4^{(m)}$ and 
$11,11,11,\cdots\in\mathcal P_{p+1}^{(2)}$ attain the equalities
\eqref{eq:bineq}, \eqref{eq:b4ineq} and \eqref{eq:pineq}, respectively.
\end{exmp}
\begin{rem}\label{rem:Fbasic}
It follows from the Proposition~\ref{prop:Bineq} that
there exist only finite basic tuples $\mathbf m\in\mathcal P$ 
with a fixed index of rigidity under the normalization \eqref{eq:NTP}.
This result is given in \cite[Proposition~8.1]{O3}.

\index{Fuchsian differential equation/operator!universal operator!fundamental}
Hence there exist only finite \textsl{fundamental universal Fuchsian differential 
operators} with a fixed number of accessory parameters.
Here a fundamental universal Fuchsian differential operator means
a universal operator given in Theorem~\ref{thm:univmodel} 
whose spectral type is fundamental (cf.~Definition~\ref{def:fund}). 
\end{rem}
Now we prepare a lemma.
\begin{lem}\label{lem:abck}
Let $a\ge 0$, $b>0$ and $c>0$ be integers such that 
$a+c-b>0$.
Then
\[
  \frac{b+kc-6}{(a+c-b)b}
  \begin{cases}
   < k+1&(0\le k\le 5),\\
   \le 7&(0\le k\le 6).
  \end{cases}
\]
\end{lem}
\begin{proof}
Suppose $b\ge c$.  Then
\begin{align*}
\frac{b+kc-6}{(a+c-b)b}
 &\le \frac{b+kb-6}{b} < k+1.
\end{align*}
Next suppose $b<c$.  Then
\begin{align*}
(k+1)(a+c-b)b - (b+kc-6)
 &\ge (k+1)(c-b)b - b-kc+6\\
 &\ge (k+1)b - b - k(b+1)+6 = 6-k.
\end{align*}
Thus we have the lemma.
\end{proof}
\begin{proof}[Proof of Proposition~\ref{prop:Bineq}]
Since $\idx k\mathbf m=k^2\idx\mathbf m$ for a basic tuple $\mathbf m$
and $k\in\mathbb Z_{>0}$, we may assume that $\mathbf m$ is basic
and $\idx\mathbf m\le -2$  to prove the proposition.

Fix a basic monotone tuple $\mathbf m$.
Put $\alpha=\alpha_{\mathbf m}$ under the notation
\eqref{eq:Kazpart} and $n=\ord\mathbf m$.
Note that 
\begin{equation}\label{eq:DynNorm}
 (\alpha|\alpha)=n(\alpha|\alpha_0)+\sum_{j=0}^p\sum_{\nu=1}^{n_j}
 n_{j,\nu}(\alpha|\alpha_{j,\nu}),\quad
 (\alpha|\alpha_0)\le 0,\quad(\alpha|\alpha_{j,\nu})\le0.
\end{equation}
We first assume that \eqref{eq:bineq} is not valid, namely,
\begin{equation}\label{eq:beq}
  3|(\alpha|\alpha)|+6 < n.
\end{equation}
In view of \eqref{eq:basic0}, we have $(\alpha|\alpha)<0$ 
and the assumption implies $|(\alpha|\alpha_0)|=0$
because $|(\alpha|\alpha)|\ge n|(\alpha|\alpha_0)|$.

Let $\Pi_0$ be the connected component of
$\{\alpha_i\in\Pi\,;\,(\alpha|\alpha_i)=0
\text{ and }\alpha_i\in\supp\alpha\}$ containing $\alpha_0$.
Note that $\supp\alpha$ generates a root system
which is neither classical nor affine but 
$\Pi_0$ generates a root system of finite type.

Put $J=\{j\,;\,\exists\alpha_{j,\nu}\in\supp\alpha_{\mathbf m}
\text{ such that }(\alpha|\alpha_{j,\nu})<0\}\ne\emptyset$
and for each $j\in J$ define $k_j$ with the condition \eqref{eq:classic}.
Then we note that
\[
 (\alpha|\alpha_{j,\nu})
 =\begin{cases}
  0      & (1\le \nu < k_j),\\
  2n_{j,k_j}-n_{j,k_{j-1}}-n_{j,k_{j+1}}\le -1& (\nu=k_j).
 \end{cases}
\]
Applying the above lemma to $\mathbf m$ by putting $n=b+k_jc$ and 
$n_{j,\nu}=b+(k_j-\nu)c$ $(1\le \nu \le k_j)$
and $n_{j,k_j+1}=a$, we have
\begin{equation}\label{eq:Dynk1}
  \frac{n-6}{(n_{j,k_{j-1}}+n_{j,k_{j+1}}-2n_{j,k_j})n_{j,k_j}}
  \begin{cases}
  < k_j+1 &(1\le k_j\le 5),\\
  \le 7   &(1\le k_j\le 6).
  \end{cases}
\end{equation}
Here $(\alpha|\alpha_{j,k_j})=b-c-a\le -1$ and 
we have $|(\alpha|\alpha)| \ge |(\alpha|\alpha_{j,\nu})| 
> \frac{n-6}{k_j+1}$ if $k_j<6$
and therefore $k_j\ge 3$.

It follows from the condition $k_j\ge 3$
that $\mathbf m\in\mathcal P_3$ because $\Pi_0$ is of finite type
and moreover that $\Pi_0$ is of exceptional type, namely, of type $E_6$ or $E_7$ 
or $E_8$ because $\supp\alpha$ is not of finite type.

Suppose $\#J\ge 2$. We may assume $\{0,1\}\subset J$
and $k_0\le k_1$.
Since $\Pi_0$ is of exceptional type and $\supp\alpha$ is not of finite type,
we may assume $k_0=3$ and $k_1\le 5$.
Owing to \eqref{eq:DynNorm} and \eqref{eq:Dynk1}, we have
\begin{align*}
 |(\alpha|\alpha)|&\ge
  n_{0,3}(n_{0,2}+n_{0,4}-2n_{0,3})
 +n_{1,k_1}(n_{1,k_1-1}+n_{1,k_1+1}-2n_{1,k_1})\\
  &>\tfrac{n-6}{3+1}+\tfrac{n-6}{5+1}
  >\tfrac{n-6}3,
\end{align*}
which contradicts to the assumption.

Thus we may assume $J=\{0\}$.
For $j=1$ and $2$ let $n_j$ be the positive integer
such that $\alpha_{j,n_j}\in\supp\alpha$ and
$\alpha_{j,n_j+1}\notin\supp\alpha$.
We may assume $n_1\ge n_2$. 

Fist suppose $k_0=3$.
Then $(n_1,n_2)=(2,1)$, $(3,1)$ or $(4,1)$
and the Dynkin diagram of $\supp\alpha$ with the numbers $m_{j,\nu}$
is one of the diagrams:
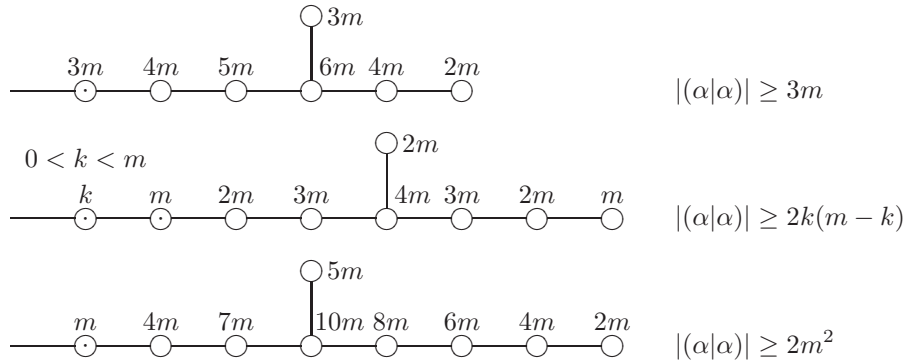
\begin{align*}
&\begin{xy}
\ar@{-}  , (10,0)     *++!D{3m}  *{\cdot}*\cir<4pt>{}="F";
\ar@{-} "F"; (20,0)   *++!D{4m}  *\cir<4pt>{}="A",
\ar@{-} "A"; (30,0)   *++!D{5m}  *\cir<4pt>{}="B",
\ar@{-} "B"; (40,0)   *+!L+!D{6m}*\cir<4pt>{}="C",
\ar@{-} "C"; (50,0)   *++!D{4m}  *\cir<4pt>{}="D",
\ar@{-} "D"; (60,0)  *++!D{2m}   *\cir<4pt>{}="E",
\ar@{-} "C"; (40,10)  *++!L{3m}  *\cir<4pt>{}
\end{xy}&&|(\alpha|\alpha)|\ge 3m\allowdisplaybreaks\\
&\begin{xy}
(10,8) *{0<k<m};
\ar@{-} (0,0);(10,0) *{\cdot}*++!D{k} *\cir<4pt>{}="H";
\ar@{-} "H";
             (20,0)   *++!D{m}   *{\cdot}*\cir<4pt>{}="F";
\ar@{-} "F"; (30,0)   *++!D{2m}  *\cir<4pt>{}="A",
\ar@{-} "A"; (40,0)   *++!D{3m}  *\cir<4pt>{}="B",
\ar@{-} "B"; (50,0)   *+!L+!D{4m}*\cir<4pt>{}="C",
\ar@{-} "C"; (60,0)   *++!D{3m}  *\cir<4pt>{}="D",
\ar@{-} "D"; (70,0)  *++!D{2m}   *\cir<4pt>{}="E",
\ar@{-} "E"; (80,0)  *++!D{m}   *\cir<4pt>{},
\ar@{-} "C"; (50,10)  *++!L{2m}  *\cir<4pt>{}
\end{xy}&&|(\alpha|\alpha)|\ge 2k(m-k)\allowdisplaybreaks\\
&\begin{xy}
\ar@{-}    , (10,0)   *++!D{m}  *{\cdot}*\cir<4pt>{}="F";
\ar@{-} "F"; (20,0)   *++!D{4m}  *\cir<4pt>{}="A",
\ar@{-} "A"; (30,0)   *++!D{7m}  *\cir<4pt>{}="B",
\ar@{-} "B"; (40,0)   *+!L+!D{\!10m}*\cir<4pt>{}="C",
\ar@{-} "C"; (50,0)   *++!D{\ 8m}  *\cir<4pt>{}="D",
\ar@{-} "D"; (60,0)  *++!D{6m}   *\cir<4pt>{}="E",
\ar@{-} "E"; (70,0)  *++!D{4m}    *\cir<4pt>{}="G",
\ar@{-} "G"; (80,0)  *++!D{2m}    *\cir<4pt>{},
\ar@{-} "C"; (40,10)  *++!L{5m}  *\cir<4pt>{}\ar@{-}
\end{xy}
&&|(\alpha|\alpha)|\ge 2m^2
\end{align*}
For example, when $(n_1,n_2)=(3,1)$, then $k:=m_{0,4}\ge 1$
because $(\alpha|\alpha_{0,3})\ne 0$ and therefore
$0<k<m$ and $|(\alpha|\alpha)|\ge k(m-2k)+m(2m+k-2m)
=2k(m-k)\ge 2m-2$ and $3|(\alpha|\alpha)|+6 - 4m\ge 3(2m-2)+6-4m > 0$.
Hence \eqref{eq:beq} doesn't hold.

Other cases don't happen because of the inequalities
$3\cdot 3m +6 - 6m > 0$ and
$3\cdot 2m^2 + 6 - 10m > 0$.

Lastly suppose $k_0>3$.
Then $(k_0,n_1,n_2)=(4,2,1)$ or $(5,2,1)$.
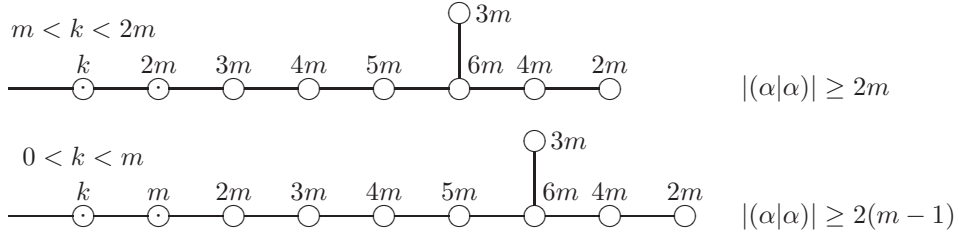
\begin{align*}
&\begin{xy}
(10,8) *{m<k<2m};
\ar@{-} (0,0); (10,0)   *++!D{k}   *{\cdot}*\cir<4pt>{}="F";
\ar@{-} "F"; (20,0)   *++!D{2m}  *{\cdot}*\cir<4pt>{}="G",
\ar@{-} "G"; (30,0)   *++!D{3m}  *\cir<4pt>{}="H",
\ar@{-} "H"; (40,0)   *++!D{4m}  *\cir<4pt>{}="A",
\ar@{-} "A"; (50,0)   *++!D{5m}  *\cir<4pt>{}="B",
\ar@{-} "B"; (60,0)   *+!L+!D{6m}*\cir<4pt>{}="C",
\ar@{-} "C"; (70,0)   *++!D{4m}  *\cir<4pt>{}="D",
\ar@{-} "D"; (80,0)   *++!D{2m}   *\cir<4pt>{},
\ar@{-} "C"; (60,10)  *++!L{3m}  *\cir<4pt>{},
\end{xy}&&
\!|(\alpha|\alpha)|\ge 2m\allowdisplaybreaks\\
&\begin{xy}
(10,8) *{0<k<m};
\ar@{-} (0,0);(10,0)   *++!D{k}   *{\cdot}*\cir<4pt>{}="F";
\ar@{-} "F"; (20,0)   *++!D{m}  *{\cdot}*\cir<4pt>{}="G",
\ar@{-} "G"; (30,0)   *++!D{2m}  *\cir<4pt>{}="H",
\ar@{-} "H"; (40,0)   *++!D{3m}  *\cir<4pt>{}="I",
\ar@{-} "I"; (50,0)   *++!D{4m}  *\cir<4pt>{}="A",
\ar@{-} "A"; (60,0)   *++!D{5m}  *\cir<4pt>{}="B",
\ar@{-} "B"; (70,0)   *+!L+!D{6m}*\cir<4pt>{}="C",
\ar@{-} "C"; (80,0)   *++!D{4m}  *\cir<4pt>{}="D",
\ar@{-} "D"; (90,0)   *++!D{2m}   *\cir<4pt>{},
\ar@{-} "C"; (70,10)  *++!L{3m}  *\cir<4pt>{}
\end{xy}
&&\!|(\alpha|\alpha)|\ge 2(m-1)
\end{align*}

In the above first case we have $(\alpha|\alpha)|\ge 2m$, which contradicts 
to \eqref{eq:beq}.
Note that $(|\alpha|\alpha)|\ge
k\cdot(m-2k)+m\cdot k=2k(m-k)\ge 2(m-1)$ in the above
last case, which also  contradicts to \eqref{eq:beq} because 
$3\cdot 2(m-1)+ 6 = 6m$.

Thus we have proved \eqref{eq:bineq}.

Assume $\mathbf m\notin\mathcal P_3$
to prove a different inequality \eqref{eq:b4ineq}.
In this case, we may assume $(\alpha|\alpha_0)=0$, $|(\alpha|\alpha)|\ge 2$
and $n>4$.
Note that
\begin{equation}\label{eq:P40}
 2n=n_{0,1}+n_{1,1}+\cdots+n_{p,1}\text{ \ with \ }
 p\ge 3\text{ and }n_{j,1}\ge 1\text{ for }j=0,\dots,p.
\end{equation}
If there exists $j$ with $1\le n_{j,1}\le\frac n2-1$, 
\eqref{eq:b4ineq} follows from \eqref{eq:DynNorm} and
$
 |(\alpha|\alpha_{j,1})|=n_{j,1}(n+n_{j,2}-2 n_{j,1})\ge 
 2n_{j,1}(\tfrac n2- n_{j,1})\ge n-2
$.

Hence we may assume $n_{j,1}\ge \frac{n-1}2$ for $j=0,\dots,p$.
Suppose there exists $j$ with $n_{j,1}=\frac{n-1}2$.
Then $n$ is odd and \eqref{eq:P40} means that there also exists $j'$ with
$j\ne j'$ and $n_{j',1}=\frac{n-1}2$.
In this case we have \eqref{eq:b4ineq} since
\[
 |(\alpha|\alpha_{j,1})|+|(\alpha|\alpha_{j',1})|=
 n_{j,1}(n+n_{j,2}-2 n_{j,1})+n_{j',1}(n+n_{j',2}-2 n_{j,1})
 \ge \tfrac{n-1}2+\tfrac{n-1}2.
\]

Now we may assume $n_{j,1}\ge \frac n2$ for $j=0,\dots,p$.
Then \eqref{eq:P40} implies that $p=3$ and $n_{j,1}=\frac n2$
for $j=0,\dots,3$.
Since $(\alpha|\alpha)<0$, there exists $j$ with $n_{j,2}\ge1$ and
\[
\begin{split}
 |(\alpha|\alpha_{j,1})|+|(\alpha|\alpha_{j,2})|&=n_{j,1}(n+n_{j,2}-2n_{j,1})
 + n_{j,2}(n_{j,1}+n_{j,3}-2n_{j,2})\\
 &=
  \tfrac n2 n_{j,2}+n_{j,2}(\tfrac n2+n_{j,3}-2n_{j,2})\\
 &
  \begin{cases}
   \ge n&(n_{j,2}\ge1),\\
   = n-2&(n_{j,2}=1\text{ and }n_{j,3}=0).
 \end{cases}
\end{split}
\]
Thus we have completed the proof of \eqref{eq:b4ineq}.

There are 4 basic tuples with the index of the rigidity $0$
and 13 basic tuples with the index of the rigidity $-2$,
which are given in \eqref{eq:basic0} and Proposition~\ref{prop:bas2}.
They satisfy \eqref{eq:pineq}.

Suppose that \eqref{eq:pineq} is not valid.
We may assume that $p$ is minimal under this assumption.
Then $\idx\mathbf m<-2$, $p\ge 5$ and $n=\ord\mathbf m>2$.
We may assume $n>n_{0,1}\ge n_{1,1}\ge\cdots\ge n_{p,1}>0$.
Since $(\alpha|\alpha_0)\le 0$, we have
\begin{equation}\label{eq:basp}
  n_{0,1}+n_{1,1}+\cdots+n_{p,1}\ge 2n>n_{0,1}+\cdots+n_{p-1,1}.
\end{equation}
In fact, if $n_{0,1}+\cdots+n_{p-1,1}\ge 2n$, the tuple 
$\mathbf m'=(\mathbf m_0,\dots,\mathbf m_{p-1})$
is also basic and $|(\alpha|\alpha)|-
|(\alpha_{\mathbf m'},\alpha_{\mathbf m'})|=n^2-\sum_{\nu\ge1}n_{p,\nu}^2\ge 2$, 
which contradicts to the minimality.

Thus we have $2n_{j,1}<n$ for $j=3,\dots,p$.
If $n$ is even, $|\idx\mathbf m|\ge\sum_{j=3}^p|(\alpha|\alpha_{j,1})|
=\sum_{j=3}^p(n+n_{j,2}-2n_{j,1})\ge 2(p-2)$, which contradicts to the assumption.
If $n=3$, \eqref{eq:basp} assures $p=5$ and $n_{0,1}=\cdots=n_{5,0}=1$ and 
therefore $\idx\mathbf m=-4$, which also contradicts to the assumption.
Thus $n=2m+1$ with $m\ge 2$.  
Choose $k$ so that $n_{k-1,1}\ge m>n_{k,1}$.
Then $|\idx\mathbf m|\ge\sum_{j=k}^p(\alpha|\alpha_{j,1})|=\sum_{j=k}^p
(n+n_{j,2}-2n_{j,1})\ge 3(p-k+1)$.  Owing to \eqref{eq:basp}, we have
$2(2m+1)>km+(p-k)$ and $k< \frac{4m+2-p}{m-1}\le \frac{4m-3}{m-1}\le 5$, 
which means $k\le 4$, $|\idx\mathbf m|\ge 3(p-3)\ge 2p-4$ and a contradiction to 
the assumption.
\end{proof}
\section{Expression of local solutions}\label{sec:exp}
Fix $\mathbf m=\bigl(m_{j,\nu}\bigr)_{\substack{j=0,\dots,p\\1\le\nu\le n_j}}
\in\mathcal P_{p+1}$.
Suppose $\mathbf m$ is monotone and irreducibly realizable.
Let $P_{\mathbf m}$ be the universal operator with the Riemann scheme
\eqref{eq:GRS}, which is given in Theorem~\ref{thm:univmodel}.
Suppose $c_1=0$ and $m_{1,n_1}=1$.
We give expressions of the local 
solution of $P_{\mathbf m}u=0$ at $x=0$ corresponding to the
characteristic exponent $\lambda_{1,n_1}$.
\begin{thm}\label{thm:expsol}
Retain the notation above and in Definition~\ref{def:redGRS}.
Suppose $\lambda_{j,\nu}$ are generic.
Let
\begin{equation}
  v(x) = \sum_{\nu=0}^\infty C_\nu x^{\lambda(K)_{1,n_1}+\nu}
\end{equation}
be the local solution of $\bigl(\p_{\max}^KP_{\mathbf m}\bigr)v=0$
at $x=0$ with the condition $C_0=1$.
Put
\begin{equation}
 \lambda(k)_{j,max}=\lambda(k)_{j,\ell(k)_j}.
\end{equation}
Note that if\/ $\mathbf m$ is rigid, then
\begin{equation}
 v(x) = x^{\lambda(K)_{1,n_1}}\prod_{j=2}^p
  \Bigl(1-\frac x{c_j}\Bigr)^{\lambda(K)_{j,max}}.
\end{equation}
The function
\begin{equation}\label{eq:intexp}
 \begin{split}
 u(x)&:=\prod_{k=0}^{K-1}
 \frac{\Gamma\bigl(\lambda(k)_{1,n_1}-\lambda(k)_{1,max}+1\bigr)}
  {\Gamma\bigl(\lambda(k)_{1,n_1}-\lambda(k)_{1,max}+\mu(k)+1\bigr)
   \Gamma\bigl(-\mu(k)\bigr)}
\\
 &\int_0^{s_0}\cdots \int_0^{s_{K-1}}
  \prod_{k=0}^{K-1}(s_k-s_{k+1})^{-\mu(k)-1}\\
 &\quad\cdot\prod_{k=0}^{K-1}\biggl(
 \Bigl(\frac{s_k}{s_{k+1}}\Bigr)^{\lambda(k)_{1,max}}
 \prod_{j=2}^p\Bigl(\frac{1-c_j^{-1}s_k}{1-c_j^{-1}s_{k+1}}\Bigr)
  ^{\lambda(k)_{j,{max}}}
 \biggr)\\
 &\quad\cdot v(s_K)ds_K\cdots ds_1\Bigl|_{s_0=x}
 \end{split}
\end{equation}
is the solution of $P_{\mathbf m}u=0$ so normalized that
$u(x)\equiv x^{\lambda_{1,n_1}}\mod x^{\lambda_{1,n_1}+1}\mathcal O_0$.

Here we note that
\begin{equation}\label{eq:solrepsub}
 \begin{split}
 &\prod_{k=0}^{K-1}\biggl(
 \Bigl(\frac{s_k}{s_{k+1}}\Bigr)^{\lambda(k)_{1,max}}
 \prod_{j=2}^p\Bigl(\frac{1-c_j^{-1}s_k}{1-c_j^{-1}s_{k+1}}\Bigr)
  ^{\lambda(k)_{j,max}}
 \biggr)\\
 &\quad=\frac{s_0^{\lambda(0)_{1,max}}}
  {s_K^{\lambda(K-1)_{1,max}}}
     \prod_{j=1}^p
        \frac{(1-c_j^{-1}s_0)^{\lambda(0)_{j,{max}}}}
        {(1-c_j^{-1}s_K)^{\lambda(K-1)_{j,{max}}}}\\
    &\qquad\cdot\prod_{k=1}^{K-1}
     \Bigl(s_k^{\lambda(k)_{1,{max}}-\lambda(k-1)_{1,{max}}
       }
    \prod_{j=2}^{p}
    (1-c_j^{-1}s_k)
    ^{\lambda(k)_{j,{max}}-\lambda(k-1)_{j,{max}}}\Bigr).
 \end{split}
\end{equation}
When\/ $\mathbf m$ is rigid,
\begin{equation}\label{eq:serexpr}
\begin{split}
  u(x)&= x^{\lambda_{1,n_1}}
  \biggl(\prod_{j=2}^p\Bigl(1-\frac x{c_j}\Bigr)^{\lambda(0)_{j,{max}}}\biggr)
  \sum_{\bigl(\nu_{j,k}\bigr)
   _{\substack{2\le j\le p\\ 1\le k\le K}}\in\mathbb Z_{\ge 0}^{(p-1)K}}
 \\&\quad
  \prod_{i=0}^{K-1}
  \frac
  {\bigl(\lambda(i)_{1,n_1}-\lambda(i)_{1,{max}}+1\bigr)
  _{\sum_{s=2}^p\sum_{t=i+1}^{K}\nu_{s,t}}}
  {\bigl(\lambda(i)_{1,n_1}-\lambda(i)_{1,{max}}+\mu(i)+1\bigr)
  _{\sum_{s=2}^p\sum_{t=i+1}^{K}\nu_{s,t}}}
 \\&
  \quad
  \cdot\prod_{i=1}^K\prod_{s=2}^p\frac{
   \bigl(\lambda(i-1)_{s,{max}}-\lambda(i)_{s,{max}}\bigr)_{\nu_{s,i}}}
  {\nu_{s,i}!}\cdot
  \prod_{s=2}^p\Bigl(\frac{x}{c_s}\Bigr)^{\sum_{i=1}^K\nu_{s,i}}.
\end{split}
\end{equation}
When\/ $\mathbf m$ is not rigid
\begin{equation}\label{eq:serexp}
\begin{split}
 u(x)&= x^{\lambda_{1,n_1}}
  \biggl(\prod_{j=2}^p\Bigl(1-\frac x{c_j}\Bigr)^{\lambda(0)_{j,{max}}}\biggr)
   \sum_{\nu_0=0}^\infty\sum_{\bigl(\nu_{j,k}\bigr)
   _{\substack{2\le j\le p\\ 1\le k\le K}}\in\mathbb Z_{\ge 0}^{(p-1)K}}
 \\&\quad
  \prod_{i=0}^{K-1}
  \frac
  {\bigl(\lambda(i)_{1,n_1}-\lambda(i)_{1,{max}}+1\bigr)
  _{\nu_0+\sum_{s=2}^p\sum_{t=i+1}^{K}\nu_{s,t}}}
  {\bigl(\lambda(i)_{1,n_1}-\lambda(i)_{1,{max}}+\mu(i)+1\bigr)
  _{\nu_0+\sum_{s=2}^p\sum_{t=i+1}^{K}\nu_{s,t}}}\\
 &
  \quad\cdot
  \prod_{s=2}^p\frac{
   \bigl(\lambda(K-1)_{s,{max}}\bigr)_{\nu_{s,K}}
  }{\nu_{s,K}!}\cdot
  \prod_{i=1}^{K-1}\prod_{s=2}^p\frac{
   \bigl(\lambda(i-1)_{s,{max}}
     -\lambda(i)_{s,{max}}\bigr)_{\nu_{s,i}}
  }{\nu_{s,i}!}
  \\&\quad\cdot
  C_{\nu_0}x^{\nu_0}\prod_{s=2}^p\Bigl(\frac{x}{c_s}\Bigr)^{\sum_{i=1}^K\nu_{s,i}}.
\end{split}
\end{equation}

Fix $j$ and $k$ and suppose
\begin{equation}
 \begin{cases}
   \ell(k-1)_j=\ell(k)_\nu&\text{when }\mathbf m\text{ is rigid or }k<K,\\
   \ell(k-1)_j=0&\text{when }\mathbf m\text{ is not rigid and }k=K.
 \end{cases}
\end{equation}
Then the terms satisfying $\nu_{j,k}>0$ vanish because $(0)_{\nu_{j,k}}=
\delta_{0,\nu_{j,k}}$ for $\nu_{j,k}=0,1,2,\ldots$.
\end{thm}
\begin{proof}
The theorem follows from \eqref{eq:opred}, \eqref{eq:pellP2}, \eqref{eq:defpell},
\eqref{eq:seriesP} and \eqref{eq:IcP} 
by the induction on $K$.
Note that the integral representation of the normalized solution of 
$\bigr(\p_{max}P\bigr)v=0$ corresponding to the exponent $\lambda(1)_{n_1}$
equals
\[
 \begin{split}
 v(x)&:=\prod_{k=1}^{K-1}
 \frac{\Gamma\bigl(\lambda(k)_{1,n_1}-\lambda(k)_{1,max}+1\bigr)}
 {\Gamma\bigl(\lambda(k)_{1,n_1}-\lambda(k)_{1,max}+\mu(k)+1\bigr)
  \Gamma\bigl(-\mu(k)\bigr)}\\
 &\quad\cdot\int_0^{s_1}\cdots \int_0^{s_{K-1}}
  \prod_{k=0}^{K-1}(s_k-s_{k+1})^{-\mu(k)-1}\\
 &\quad\cdot\prod_{k=0}^{K-1}\biggl(
 \Bigl(\frac{s_k}{s_{k+1}}\Bigr)^{\lambda(k)_{1,max}}
 \prod_{j=2}^p\Bigl(\frac{1-c_j^{-1}s_k}{1-c_j^{-1}s_{k+1}}\Bigr)
  ^{\lambda(k)_{j,{max}}}
 \biggr)\\
 &\quad\cdot v(s_K)ds_K\cdots ds_1\Bigl|_{s_1=x}\\
 &\equiv x^{\lambda(1)_{1,n_1}}\mod x^{\lambda(1)_{1,n_1}+1}\mathcal O_0
 \end{split}
\]
by the induction hypothesis and the normalized solution of $Pu=0$ corresponding
to the exponent $\lambda_{1,n_1}$ equals
\[
 \begin{split}
  &\frac{\Gamma\bigl(\lambda(0)_{1,n_1}-\lambda(0)_{1,max}+1\bigr)}
   {\Gamma\bigl(\lambda(0)_{1,n_1}-\lambda(0)_{1,max}+\mu(0)+1\bigr)
    \Gamma\bigl(-\mu(0)\bigr)}\\
  &\quad{}\cdot\int_0^x(x-s_0)^{-\mu(0)-1}\frac{x^{-\lambda(0)_{1,\max}}}
   {s_0^{-\lambda(0)_{1,\max}}}\prod_{j=2}^p
  \Bigl(\frac{1-c_j^{-1}x}{1-c_j^{-1}s_0}\Bigr)^{-\lambda(0)_{j,\max}}
  v(s_0)ds_0
 \end{split}
\]
and hence we have \eqref{eq:intexp}.
Then the integral expression \eqref{eq:intexp} with \eqref{eq:solrepsub},
\eqref{eq:seriesP} and \eqref{eq:IcP} inductively proves \eqref{eq:serexpr} 
and \eqref{eq:serexp}.
\end{proof}
\begin{exmp}[Gauss hypergeometric equation]\label{ex:localGauss}
\index{hypergeometric equation/function!Gauss}
The reduction \eqref{eq:redGG} shows
\begin{align*}
 \lambda(0)_{j,\nu}&=\lambda_{j,\nu},\ 
 m(0)_{j,\nu} = 1 \quad(0\le j\le 2,\ 1\le\nu\le 2),\ 
  \mu(0) =-\lambda_{0,2}-\lambda_{1,2}-\lambda_{2,2},\\
 m(1)_{j,1}&=0,\ m(1)_{j,2}=1\quad(j=0,1,2),\\
 \lambda(1)_{0,1} &=\lambda_{0,1}+2\lambda_{0,2}+2\lambda_{1,2}+2\lambda_{2,2},\ 
 \lambda(1)_{1,1}=\lambda_{1,1},\ \lambda(1)_{2,1}=\lambda_{2,1},\\
 \lambda(1)_{0,2} &=2\lambda_{0,2}+\lambda_{1,2}+\lambda_{2,2},\ 
 \lambda(1)_{1,2}  =-\lambda_{0,2}-\lambda_{2,2},\ 
 \lambda(1)_{2,2}  =-\lambda_{0,2}-\lambda_{1,2}
\end{align*}
and therefore
\begin{align*}
 \lambda(0)_{1,n_1}-\lambda(0)_{1,max}+\mu(0)+1
 &=\lambda_{1,2}-\lambda_{1,1}-(\lambda_{0,2}+\lambda_{1,2}+\lambda_{2,2})+1\\
 &=\lambda_{0,1}+\lambda_{1,2}+\lambda_{2,1},
 \allowdisplaybreaks\\
 \lambda(0)_{2,max}-\lambda(1)_{2,max}&=\lambda(0)_{2,1}-\lambda(1)_{2,2}
  =\lambda_{2,1}+\lambda_{0,2}+\lambda_{1,2}.
\end{align*}
Hence \eqref{eq:intexp} says that 
the normalized local solution corresponding to the characteristic 
exponent $\lambda_{1,2}$ with $c_1=0$ and $c_2=1$ equals
\begin{equation}\label{eq:gaussIrep}
\begin{split}
 u(x)
   &=\frac{\Gamma\bigl(\lambda_{1,2}-\lambda_{1,1}+1\bigr)x^{\lambda_{1,1}}
     (1-x)^{\lambda_{2,1}}}
      {\Gamma\bigl(\lambda_{0,1}+\lambda_{1,2}+\lambda_{2,1}\bigr)
       \Gamma\bigl(\lambda_{0,2}+\lambda_{1,2}+\lambda_{2,2}\bigr)}\\
   &\quad\int_0^x(x-s)^{\lambda_{0,2}+\lambda_{1,2}+\lambda_{2,2}-1}
   s^{-\lambda_{0,2}-\lambda_{1,1}-\lambda_{2,2}}
  (1-s)^{-\lambda_{0,2}-\lambda_{1,2}  - \lambda_{2,1}}ds
\end{split}\end{equation}
and moreover \eqref{eq:serexpr} says
\begin{equation}\begin{split}
  u(x)&=x^{\lambda_{1,2}}(1-x)^{\lambda_{2,1}}
        \sum_{\nu=0}^\infty\frac{(\lambda_{0,1}+\lambda_{1,2}+\lambda_{2,1})_\nu
        (\lambda_{0,2}+\lambda_{1,2}+\lambda_{2,1})_\nu}
        {(\lambda_{1,2}-\lambda_{1,1}+1)_\nu\nu!}x^\nu.
\end{split}\end{equation}
Note that $u(x)=F(a,b,c;x)$ when
\begin{equation}\label{eq:RSgauss}
 \begin{Bmatrix}
  x=\infty & 0 & 1\\
  \lambda_{0,1} &  \lambda_{1,1} &  \lambda_{2,1}\\
  \lambda_{0,2} &  \lambda_{1,2} &  \lambda_{2,2}
 \end{Bmatrix}=
 \begin{Bmatrix}
  x=\infty & 0 & 1\\
   a &  1-c &  0\\
   b &  0   &  c-a-b
 \end{Bmatrix}.
\end{equation}
The integral expression \eqref{eq:gaussIrep}
is based on the minimal expression $w=s_{0,1}s_{1,1}s_{1,2}s_0$ satisfying
$w\alpha_{\mathbf m}=\alpha_0$.
Here $\alpha_{\mathbf m}=2\alpha_0+\sum_{j=0}^2\alpha_{j,1}$.
When we replace $w$ and its minimal expression by
$w'=s_{0,1}s_{1,1}s_{1,2}s_0s_{0,1}$ or $w''=s_{0,1}s_{1,1}s_{1,2}s_0s_{2,1}$,
we get the different integral expressions
\begin{equation}\begin{split}
 u(x)
   &=\frac{\Gamma\bigl(\lambda_{1,2}-\lambda_{1,1}+1\bigr)x^{\lambda_{1,1}}
     (1-x)^{\lambda_{2,1}}}
      {\Gamma\bigl(\lambda_{0,2}+\lambda_{1,2}+\lambda_{2,1}\bigr)
       \Gamma\bigl(\lambda_{0,1}+\lambda_{1,2}+\lambda_{2,2}\bigr)}\\
   &\quad\int_0^x(x-s)^{\lambda_{0,1}+\lambda_{1,2}+\lambda_{2,2}-1}
   s^{-\lambda_{0,1}-\lambda_{1,1}-\lambda_{2,2}}
  (1-s)^{-\lambda_{0,1}-\lambda_{1,2}  - \lambda_{2,1}}ds
\allowdisplaybreaks\\
   &=\frac{\Gamma\bigl(\lambda_{1,2}-\lambda_{1,1}+1\bigr)x^{\lambda_{1,1}}
     (1-x)^{\lambda_{2,2}}}
      {\Gamma\bigl(\lambda_{0,1}+\lambda_{1,2}+\lambda_{2,2}\bigr)
       \Gamma\bigl(\lambda_{0,2}+\lambda_{1,2}+\lambda_{2,1}\bigr)}\\
   &\quad\int_0^x(x-s)^{\lambda_{0,2}+\lambda_{1,2}+\lambda_{2,1}-1}
   s^{-\lambda_{0,2}-\lambda_{1,1}-\lambda_{2,1}}
  (1-s)^{-\lambda_{0,2}-\lambda_{1,2}  - \lambda_{2,2}}ds.
\end{split}\end{equation}
\index{hypergeometric equation/function!Gauss!integral expression}
These give different integral expressions of $F(a,b,c;x)$ under
\eqref{eq:RSgauss}.

Since $s_{\alpha_0+\alpha_{0,1}+\alpha_{0,2}}\alpha_{\mathbf m}
=\alpha_{\mathbf m}$, we have
\begin{align*}
&\begin{Bmatrix}
 x=\infty & 0 & 1\\
 a & 1-c & 0\\
 b & 0 & c-a-b
\end{Bmatrix}
 \xrightarrow{x^{c-1}}
\begin{Bmatrix}
 x=\infty & 0 & 1\\
 a-c+1 & 0   &0\\
 b-c+1 & c-1 &c-a-b
\end{Bmatrix}\\
&\quad \xrightarrow{\p^{c-d}}
\begin{Bmatrix}
 x=\infty & 0 & 1\\
 a-d+1 & 0   &0\\
 b-d+1 & d-1 &d-a-b
\end{Bmatrix}
  \xrightarrow{x^{1-d}}
\begin{Bmatrix}
 x=\infty & 0 & 1\\
  a & 1-d & 0\\
  b & 0 & d-a-b
\end{Bmatrix}
\end{align*}
and hence (cf.~\eqref{eq:IcP})
\index{hypergeometric equation/function!Gauss!Euler transformation}
\begin{equation}
F(a,b,d;x)=\frac{\Gamma(d)x^{1-d}}{\Gamma(c)\Gamma(d-c)}
 \int_0^x (x-s)^{d-c-1}s^{c-1}F(a,b,c;s)ds.
\end{equation}
\end{exmp}
\begin{rem}\label{rem:Irep}
The integral expression of the local solution $u(x)$ as is given
in Theorem~\ref{thm:expsol} is obtained from the expression of the
element $w$ of $W_{\!\infty}$ satisfying 
$w\alpha_{\mathbf m}\in B\cup\{\alpha_0\}$ as a product of simple 
reflections and therefore the integral expression depends on such
element $w$ and the expression of $w$ as such product.
The dependence on $w$ seems non-trivial as in the preceding example
but the dependence on the expression of $w$ as a product of simple 
reflections is understood as follows.

First note that the integral expression doesn't depend on the coordinate 
transformations $x\mapsto ax$ and $x\mapsto x+b$ with $a\in\mathbb C^\times$
and $b\in\mathbb C$.  Since
\begin{align*}
   \int_c^x(x-t)^{\mu-1}\phi(t)dt
 &=-\int_{\frac1c}^{\frac1x}(x-\tfrac1s)^{\mu-1}\phi(\tfrac1s)s^{-2}ds\\
 &=-(-1)^{\mu-1}x^{\mu-1}\int_{\frac1c}^{\frac1x}(\tfrac1x-s)^{\mu-1}
    (\tfrac1s)^{\mu+1}\phi(\tfrac1s)ds,
\end{align*}
we have
\begin{equation}\label{eq:Iinv}
 I_c^\mu(\phi)=-(-1)^{\mu-1}x^{\mu-1}\left.
     \left(I_{\frac1c}^x\left.
     \bigl(x^{\mu+1}\phi(x)\bigr)\right|_{x\mapsto \frac1x}\right)
   \right|_{x\to\frac1x},
\end{equation}
which corresponds to \eqref{eq:redcoord}.
Here the value $(-1)^{\mu-1}$ depends on the branch of the value of 
$(x-\frac1s)^{\mu-1}$ and that of $x^{\mu-1}x^{1-\mu}(\frac1x -s)^{\mu-1}$.

Hence the argument as in the proof of Theorem~\ref{thm:KatzKac} shows 
that the dependence on the expression of $w$ by a product of simple reflections 
can be understood by the identities
\eqref{eq:Iinv} and $I_c^{\mu_1}I_c^{\mu_2}=I_c^{\mu_1+\mu_2}$
(cf.~\eqref{eq:Icprod}) etc.
\end{rem}
\section{Monodromy}
\label{sec:monodromy}
The transformation of monodromy generators for irreducible Fuchsian systems 
of Schlesinger canonical form under the middle convolution or the addition is 
studied by \cite{Kz} and \cite{DR, DR2} etc.
A non-zero homomorphism of an irreducible single Fuchsian
differential equation to an irreducible system of Schlesinger canonical 
form induces the isomorphism of their monodromies of the 
solutions (cf.~Remark~\ref{rem:SCFmc}).  
In particular since any rigid local system is 
realized by a single Fuchsian differential equation, their
monodromies naturally coincide with each other through the 
correspondence of their monodromy generators.
The correspondence between the local monodromies and the global monodromies
is described by \cite{DR2}, which we will review.

\subsection{Middle convolution of monodromies}\label{sec:MM}
For given matrices $A_j\in M(n,\mathbb C)$ for $j=1,\dots,p$ the Fuchsian
system
\begin{equation}\label{eq:MSCF}
  \frac{dv}{dx}=\sum_{j=1}^p\frac{A_j}{x-c_j}v
\end{equation}
of Schlesinger canonical form (SCF) is defined.  
Put $A_0=-A_1-\dots-A_p$ and $\mathbf A=(A_0,A_1,\dots,A_p)$
which is an element of
\begin{equation}\label{eq:Msub0}
 M(n,\mathbb C)^{p+1}_0:=\{(C_0,\dots,C_p)\in 
 M(n,\mathbb C)^{p+1}\,;\,C_0+\cdots+C_p=0\}, 
\end{equation}
The Riemann scheme of \eqref{eq:MSCF} is defined by
\begin{equation}
\begin{Bmatrix}
 x = c_0=\infty & c_1&\cdots&c_p\\
 [\lambda_{0,1}]_{m_{0,1}} &  [\lambda_{1,1}]_{m_{1,1}} &\cdots 
    &[\lambda_{p,1}]_{m_{p,1}}\\
 \vdots & \vdots & \vdots & \vdots\\
 [\lambda_{0,n_0}]_{m_{0,n_0}} &  [\lambda_{1,n_1}]_{m_{1,n_1}} &\cdots 
    &[\lambda_{p,n_p}]_{m_{p,1}}
\end{Bmatrix},\quad
[\lambda]_k:=\begin{pmatrix}
              \lambda\\
              \vdots\\
              \lambda
             \end{pmatrix}\in M(1,k,\mathbb C)
\end{equation}
if
\[
  A_j \sim L(m_{j,1},\ldots,m_{j,n_j};\lambda_{j,1},\dots,\lambda_{j,n_j})
  \quad(j=0,\dots,p)
\]
under the notation \eqref{eq:OSNF}.
Here the Fuchs relation equals
\begin{equation}
  \sum _{j=0}^p\sum_{\nu=1}^{n_j} m_{j,\nu}\lambda_{j,\nu}=0.
\end{equation}

We define that  $\mathbf A$ is \textsl{irreducible} if a subspace $V$ of 
$\mathbb C^n$ satisfies $A_jV\subset A_j$ for $j=0,\dots,p$, then $V=\{0\}$ or 
$V=\mathbb C^n$.
In general, $\mathbf A=(A_0,\dots,A_p)$, $\mathbf A'=(A'_0,\dots,A'_p)\in 
M(n,\mathbb C)^{p+1}$, we denote by $\mathbf A\sim\mathbf A'$ if there exists
$U\in GL(n,\mathbb C)$ such that $A'_j=U A_jU^{-1}$ for $j=0,\dots,p$.
\index{monodromy!irreducible}

For $(\mu_0,\dots,\mu_p)\in \mathbb C^{p+1}$ with $\mu_0+\cdots+\mu_p=0$,
the addition $\mathbf A'=(A'_0,\dots,A'_p)\in M(n,\mathbb C)^{p+1}_0$
of $\mathbf A$ with respect to $(\mu_0,\dots,\mu_p)$ is defined by 
$A'_j=A_j+\mu_j$ for $j=0,\dots,p$.

For a complex number $\mu$ the middle convolution 
$\bar{\mathbf A}:=mc_\mu(\mathbf A)$ of $\mathbf A$ is defined by 
$\bar A_j=\bar A_j(\mu)$ for $j=1,\dots,p$ and $\bar A_0=-\bar A_1-\dots-\bar A_p$
under the notation in \S\ref{sec:DR}. 
Then we have the following theorem.

\begin{thm}[\cite{DR, DR2}]\label{thm:schmid}
Suppose that $\mathbf A$ satisfies the conditions
\begin{align}
  \bigcap_{\substack{1\le j\le p\\ j\ne i}}\ker A_j
  \cap \ker(A_0-\tau)&=\{0\}&(i=1,\dots,p,\ \forall\tau\in\mathbb C),
  \label{eq:star}\\
    \bigcap_{\substack{1\le j\le p\\ j\ne i}}\ker {}^t\!A_j
  \cap \ker({}^t\!A_0-\tau)&=\{0\}
  &(i=1,\dots,p,\ \forall\tau\in\mathbb C).
  \label{eq:starstar}
\end{align}

{\rm i)} 
The tuple $mc_\mu(\mathbf A)=(\bar A_0,\dots,\bar A_p)$ also satisfies the same 
conditions as above with replacing $A_\nu$ by $\bar A_\nu$ for $\nu=0,\dots,p$, 
respectively.
Moreover we have
\begin{align}
 mc_\mu(\mathbf A)&\sim mc_\mu(\mathbf A')\text{ \ if \ }\mathbf A\sim\mathbf A',\\
 mc_{\mu'}\circ mc_{\mu}(\mathbf A)&\sim mc_{\mu+\mu'}(\mathbf A),\\
 mc_0(\mathbf A)&\sim \mathbf A
\end{align}
and $\mathbf A$ is irreducible if and only if $\mathbf A'$ is irreducible.

{\rm ii) (cf.~\cite[Theorem~5.2]{O3})}  Assume 
\begin{equation}
 \mu=\lambda_{0,1}\ne0\text{ \ and \ } \lambda_{j,1}=0\text{ \ for \ }j=1,\dots,p
\end{equation}
and
\begin{align}
 \lambda_{j,\nu}=\lambda_{j,1}\text{ \ implies \ }m_{j,\nu}\le m_{j,1}
\end{align}
for $j=0,\dots,p$ and $\nu=2,\dots,n_j$.
Then the Riemann scheme of $mc_{\mu}(\mathbf A)$ equals
\begin{gather}
\begin{Bmatrix} x = \infty & c_1&\cdots&c_p\\
 [-\mu]_{m_{0,1}-d} &  [0]_{m_{1,1}-d} &\cdots 
    &[0]_{m_{p,1}-d}\\
 [\lambda_{0,2}-\mu]_{m_{0,2}} &  [\lambda_{1,2}+\mu]_{m_{1,2}} &\cdots 
    &[\lambda_{p,2}+\mu]_{m_{p,2}}\\
 \vdots & \vdots & \vdots & \vdots\\
 [\lambda_{0,n_0}-\mu]_{m_{0,n_0}} &  [\lambda_{1,n_1}+\mu]_{m_{1,n_1}} &\cdots 
    &[\lambda_{p,n_p}+\mu]_{m_{p,1}}
\end{Bmatrix}
\intertext{with}
d:=m_{0,1}+\cdots+m_{p,1}-(p-1)\ord\mathbf m.\label{eq:defd2}
\end{gather}
\end{thm}
\begin{exmp}\label{ex:univSch}
The addition of 
\[
 mc_{-\lambda_{0,1}-\lambda_{1,2}-\lambda_{2,2}}(\{\lambda_{0,2}-\lambda_{0,1}, 
 \lambda_{0,1}+\lambda_{1,1}+\lambda_{2,2},
  \lambda_{0,1}+\lambda_{1,2}+\lambda_{2,1}\})
\]
with respect to $(-\lambda_{1,2}-\lambda_{2,2},\lambda_{1,2},\lambda_{2,2})$
give the Fuchsian system of Schlesinger canonical form
\[
 \begin{gathered}
 \frac{du}{dx}=\frac{A_1}{x}u+\frac{A_2}{x-1}u,\\ 
 A_1=\begin{pmatrix}
      \lambda_{1,1} & \lambda_{0,1}+\lambda_{1,2}+\lambda_{2,1}\\
      & \lambda_{1,2}
     \end{pmatrix}\text{ \ and \ }
 A_2=\begin{pmatrix}
      \lambda_{2,2} & \\
      \lambda_{0,1}+\lambda_{1,1}+\lambda_{2,2} & \lambda_{2,1}\\
     \end{pmatrix}.
 \end{gathered}
\]
with the Riemann scheme
\[
 \begin{Bmatrix}
 x=\infty & 0 & 1\\
  \lambda_{0,1} & \lambda_{1,1} & \lambda_{2,1}\\
  \lambda_{0,2} & \lambda_{1,2} & \lambda_{2,2}\\
 \end{Bmatrix} \qquad
 (\lambda_{0,1}+\lambda_{0,2}+\lambda_{1,1}+\lambda_{1,2}
   +\lambda_{2,1}+\lambda_{2,2}=0).
\]
The system is invariant as $W(x;\lambda_{j,\nu})$-modules 
under the transformation $\lambda_{j,\nu}\mapsto\lambda_{j,3-\nu}$ 
for $j=0,1,2$ and $\nu=1,2$.

Suppose $\lambda_{j,\nu}$ are generic complex numbers under the condition
$
  \lambda_{0,1}+\lambda_{1,2}+\lambda_{2,1}
 =\lambda_{0,2}+\lambda_{1,1}+\lambda_{2,2}=0.
$
Then $A_1$ and $A_2$ have a unique simultaneous eigenspace.
In fact,
$A_1\binom 01=\lambda_{1,2}\binom 01$ and 
$A_2\binom 01=\lambda_{2,1}\binom 01$.
Hence the system is not invariant as $W(x)$-modules under 
the transformation above and $\mathbf A$ is not irreducible
in this case.
\end{exmp}

To describe the monodromies, we review the multiplicative version of
these operations.

Let $\mathbf M=(M_0,\dots,M_p)$ be an element of
\begin{equation}\label{eq:Gsub1}
 GL(n,\mathbb C)^{p+1}_1:=\{(G_0,\dots,G_p)\in GL(n,\mathbb C)^{p+1}\,;\,
  G_p\cdots G_0=I_n\}.
\end{equation}
For $(\rho_0,\dots,\rho_p)\in\mathbb C^{p+1}$ satisfying $\rho_0\cdots \rho_p=1$,
the \textsl{multiplication\/} of $\mathbf M$ with respect to $\rho$ is defined by
$(\rho_0M_0,\dots,\rho_pM_p)$.

For a given $\rho\in\mathbb C^\times$, we define 
$\tilde M_j=\bigl(M_{j,\nu,\nu'}\bigr)_{
  {\substack{1\le\nu\le n\\1\le \nu'\le p}}}\in GL(pn,\mathbb C)$ by
\begin{align*}
 \tilde M_{j,\nu,\nu'} =
  \begin{cases}
  \delta_{\nu,\nu'}I_n &(\nu\ne j),\\
  M_{\nu'}-1 &(\nu=j,\ 1\le\nu' \le j-1),\\
  \rho M_j&(\nu=\nu'=j),\\
  \rho(M_{\nu'}-1)&(\nu=j,\ j+1\le\nu'\le p).
  \end{cases}
\end{align*}
Let $\bar M_j$ denote the quotient $\tilde M_j|_{\mathbb C^{pn}/V}$ of
\begin{equation}\label{eq:conM}
 \tilde M_j=\begin{pmatrix}
     I_n\\
       &\ddots\\
     M_1-1 & \cdots &\rho M_j & \cdots &\rho(M_p-1)\\
           &        &        & \ddots \\
           &        &        &        & I_n
     \end{pmatrix}\in GL(pn,\mathbb C)
\end{equation}
for $j=1,\dots,p$ and $M_0=(M_p\dots M_1)^{-1}$.
The tuple $\MC_\rho(\mathbf M)=(\bar M_0,\dots,\bar M_p)$ is called (the 
multiplicative version of) the middle convolution of $\mathbf M$ with 
respect to $\rho$.
Here $V:= \ker(\tilde M-1)+\bigcap_{j=1}^p \ker(\tilde M_j-1)$
with
\[
  \tilde M := \begin{pmatrix}
             M_1\\
                &\ddots\\
                & & M_p
             \end{pmatrix}.
\]
Then we have the following theorem.
\begin{thm}[\cite{DR,DR2}]\label{thm:Mmid}
Let $\mathbf M=(M_0,\dots,M_p)\in GL(n,\mathbb C)^{p+1}_1$. 
Suppose 
\begin{align}
 \bigcap_{\substack{1\le\nu\le p\\ \nu\le i}}\ker (M_\nu-1)\cap\ker(M_i-\tau)
  &=\{0\} &&(1\le i\le p,\ \forall\tau\in\mathbb C^\times),\label{eq:mulstar}\\
 \bigcap_{\substack{1\le\nu\le p\\ \nu\le i}}\ker ({}^t\!M_\nu-1)\cap\ker({}^t\!M_i-\tau)
  &=\{0\} 
 &&(1\le i\le p,\ \forall\tau\in\mathbb C^\times).\label{eq:mulss}
\end{align}

{\rm i)}
The tuple $\MC_\rho(\mathbf M)=(\bar M_0,\dots,\bar M_p)$ also satisfies the same 
conditions as above with replacing $M_\nu$ by $\bar M_\nu$ for $\nu=0,\dots,p$, 
respectively.
Moreover we have
\begin{align}
 \MC_\rho(\mathbf M)&\sim \MC_\rho(\mathbf M')
    \text{ \ if \ }\mathbf M\sim\mathbf M',\\
 \MC_{\rho'}\circ \MC_{\rho}(\mathbf M)&\sim \MC_{\rho\rho'}(\mathbf M),\\
 \MC_1(\mathbf M)&\sim \mathbf M
\end{align}
and $\MC_\rho(\mathbf M)$ is irreducible if and only if $\mathbf M$ is irreducible.

{\rm ii)} Assume 
\begin{align}
 M_j&\sim L(m_{j,1},\dots,m_{j,n_j};\rho_{j,1},\dots,\rho_{j,n_j})
 \text{ \ for \ }j=0,\dots,p,\label{eq:Lisom}\\
 \rho&=\rho_{0,1}\ne1\text{ \ and \ } \rho_{j,1}=1\text{ \ for \ }j=1,\dots,p
\end{align}
and
\begin{align}
 \rho_{j,\nu}=\rho_{j,1}\text{ \ implies \ }m_{j,\nu}\le m_{j,1}
\end{align}
for $j=0,\dots,p$ and $\nu=2,\dots,n_j$.
In this case, we say that $\mathbf M$ has a spectral type 
$\mathbf m:=(\mathbf m_0,\dots,\mathbf m_p)$ with 
$\mathbf m_j=(m_{j,1},\ldots,m_{j,n_j})$.

Putting $(\bar M_0,\dots,\bar M_p)=\MC_\rho(M_0,\dots,M_p)$, we have
\begin{equation}
 \bar M_j\sim
 \begin{cases}
  L(m_{0,1}-d,m_{0,2},\dots,m_{0,n_0};\rho^{-1},\rho^{-1}\rho_{0,2},\dots
  \rho^{-1}\rho_{0,n_0}) & (j=0),\\
  L(m_{j,1}-d,m_{j,2},\dots,m_{j,n_j};1,\rho\rho_{j,2},\dots
  \rho\rho_{j,n_j}) & (j=1,\dots,p).
 \end{cases}
\end{equation}
Here $d$ is given by \eqref{eq:defd2}.
\end{thm}

\begin{rem}\label{rem:monred}
i) \ 
We note that some $m_{j,1}$ may be zero in Theorem~\ref{thm:schmid} 
and Theorem~\ref{thm:Mmid}.

ii) 
It follows from Theorem~\ref{thm:schmid} (resp.~Theorem~\ref{thm:Mmid}) 
and Scott's lemma
that any irreducible tuple $\mathbf A\in M(n,\mathbb C)^{p+1}_0$ 
(resp.~$\mathbf M\in GL(n,\mathbb C)^{p+1}_1$) can be connected by 
successive applications of middle convolutions and additions 
(resp.~ multiplications) to an tuple whose spectral type is fundamental 
(cf.~Definition~\ref{def:fund}).
In particular, the spectral type of $\mathbf M$ is an irreducibly realizable
tuple if $\mathbf M$ is irreducible.
\end{rem}

\begin{defn}
Let $\mathbf M=(M_0,\dots,M_p)\in GL(n,\mathbb C)^{p+1}_1$.
Suppose \eqref{eq:Lisom}.
Fix $\ell=(\ell_0,\dots,\ell_p)\in\mathbb Z_{\ge 1}^{p+1}$ and define 
$\p_\ell\mathbf M$ as follows.
\begin{align*}
  \rho_j&:=\begin{cases}
             \rho_{j,\ell_j} &(0\le j\le p,\ 1\le \ell_j\le n_j),\\
             \text{any complex number} &(0\le j\le p,\ n_j<\ell_j),
           \end{cases}\\
  \rho&:=\rho_0\rho_1\dots\rho_p,\\ 
  (M_0',\dots,M_p')&
   :=\MC_\rho(\rho_1\cdots\rho_pM_0,\rho_1^{-1}M_1,\rho_2^{-1}M_2,\dots,\rho_p^{-1}M_p),\\
  \p_\ell\mathbf M&:=(\rho_1^{-1}\cdots\rho_p^{-1}M_0',\rho_1M_1',\rho_2M_2,'\dots,\rho_pM_p').
\end{align*}
Here we note that if $\ell=(1,\dots,1)$ and $\rho_{j,1}=1$ for $j=2,\dots,p$, 
$\p_\ell\mathbf M=\MC_\rho(\mathbf M)$.
\end{defn}

Let $u(1),\dots,u(n)$ be independent solutions of \eqref{eq:MSCF} at a 
generic point $q$. \index{monodromy}
Let $\gamma_j$ be a closed path around $c_j$ as in the following figure.
Denoting the result of the analytic continuation of 
$\tilde u:=(u(1),\dots,u(n))$ along $\gamma_j$ by $\gamma_j(\tilde u)$,
we have a \textsl{monodromy generator} $M_j\in GL(n,\mathbb C)$ such that
$\gamma_j(\tilde u)=\tilde u M_j$.
\index{monodromy!generator}
We call the tuple $\mathbf M=(M_0,\ldots,M_p)$ the \textsl{monodromy} of 
\eqref{eq:MSCF} with respect to $\tilde u$ and $\gamma_0,\dots,\gamma_p$.
The connecting path first going along $\gamma_i$ and then going along $\gamma_j$ 
is denoted by $\gamma_i\circ\gamma_j$.

\quad
\begin{equation}\label{fig:mon}
 \begin{xy}
 (90,20) *{\gamma_i\circ\gamma_j(\tilde u)=\gamma_j(\tilde u M_i)};
 (90,14) *{\phantom{\gamma_i\circ\gamma_j(\tilde u)}=\gamma_j(\tilde u) M_i};
 (90,8) *{\phantom{\gamma_i\circ\gamma_j(\tilde u)}=\tilde u M_jM_i,};
 (90,2) *{M_pM_{p-1}\cdots M_1M_0=I_n.};
 (15,0) *+[Fo]{{\tiny \phantom{a}\times\!c_0}}="A";
 (35,0) *+[Fo]{{\tiny c_1\!\times\phantom{c}}}="B";
 (55,0) *+[Fo]{{\tiny \phantom{a_1}\!\times\phantom{c}}}="C";
 (33,25) *+{q};
 (55,2.3) *+{\tiny c_2};
 \ar@/_2pt/ @{->} (9.3,3);(9.3,-3)
 \ar@/_2pt/ @{->} (20.7,-3);(20.7,3)
 \ar@/_2pt/ @{->} (29.3,3);(29.3,-3)
 \ar@/_2pt/ @{->} (40.7,-3);(40.7,3)
 \ar@/_2pt/ @{->} (49.3,3);(49.3,-3)
 \ar@/_2pt/ @{->} (60.7,-3);(60.7,3)
 \ar@{{*}-}_{\gamma_0} (30,25);(17.4,4,3)
 \ar@{->} (19.1,10.2);(16.5,6)
 \ar@{<-} (22.6,10.2);(20,6)
 \ar@{-}^{\gamma_1} (30,25);(33.8,5)
 \ar@{-}^{\gamma_2} (30,25);(53,3.9)
 \ar@{-}^{\gamma_3} (30,25);(53,15)
\end{xy}
\end{equation}
\medskip

The following theorem says that the monodromy of solutions of the 
system obtained by a middle convolution of the system \eqref{eq:MSCF} 
is a multiplicative middle convolution of that of the original system 
\eqref{eq:MSCF}.

\begin{thm}[\cite{DR2}]\label{thm:mcMC}
Let\/ $\Mon(\mathbf A)$ denote the monodromy of the equation \eqref{eq:MSCF}.
Put\/ $\mathbf M=\Mon(\mathbf A)$.
Suppose\/ $\mathbf M$ satisfies  \eqref{eq:mulstar} and \eqref{eq:mulss} and
\begin{align}
 \rank(A_0-\mu) &= \rank(M_0 - e^{2\pi\sqrt{-1}\mu}),\\
  \rank(A_j) &= \rank(M_j-1)
\end{align}
for $j=1,\dots,p$, then
\begin{equation}
 \Mon\bigl(mc_\mu(\mathbf A)\bigr) \sim 
 \MC_{e^{2\pi\sqrt{-1}\mu}}\bigl(\Mon(\mathbf A)\bigr).
\end{equation}
\end{thm}

Let $\mathcal F$  be a space of (multi-valued) holomorphic 
functions on $\mathbb C\setminus\{c_1,\dots,c_p\}$ valued in 
$\mathbb C^n$ such that $\mathcal F$ satisfies \eqref{eq:Fmap}, \eqref{eq:F0} 
and \eqref{eq:F1}.
For example the solutions of the equation \eqref{eq:MSCF} defines $\mathcal F$.
Fixing a base $u=\bigl(u(1),\dots,u(n)\bigr)$ of $\mathcal F(U)$ with $U\ni q$,  
we can define monodromy generators $(M_0,\dots,M_p)$.  
Fix $\mu\in\mathbb C$ and put $\rho = e^{2\pi\sqrt{-1}\mu}$ and
\begin{align*}
 v_j(x) = 
 \begin{pmatrix}
  \int^{(x+,c_j+,x-,c_j-)}\frac{u(t)(x-t)^{\mu-1}}{t-c_1}dt
 \\
 \vdots\\
  \int^{(x+,c_j+,x-,c_j-)}\frac{u(t)(x-t)^{\mu-1}}{t-c_p}dt
 \end{pmatrix}\text{ \ and \ }
  v(x)=\bigl(v_1(x),\dots,v_p(x)\bigr).
\end{align*}
Then $v(x)$ is a holomorphic function valued in $M(pn,\mathbb C)$ and 
the $pn$ column vectors of $v(x)$ define a \textsl{convolution} 
$\tilde{\mathcal F}$ of $\mathcal F$ and the following facts are 
shown by \cite{DR2}.

The monodromy generators of $\tilde{\mathcal F}$ with respect to the base 
$v(x)$ equals the \textsl{convolution} $\tilde{\mathbf M}=
(\tilde M_0,\dots,\tilde M_1)$ of $\mathbf M$ given by \eqref{eq:conM}
and if $\mathcal F$ corresponds to the space of solutions of 
\eqref{eq:SCF}, $\tilde{\mathcal F}$ corresponds to that of the system of 
Schlesinger canonical form defined by 
$\bigl(\tilde A_0(\mu),\dots,\tilde A_p(\mu)\bigr)$ in \eqref{eq:conSch},
which we denote by $\mathcal M_{\tilde{\mathbf A}}$.

The middle convolution $\MC_\rho(\mathbf M)$ of $\mathbf M$ is the induced
monodromy generators on the quotient space of $\mathbb C^{pn}/V$ where
$V$ is the maximal invariant subspace such the restriction of $\tilde{\mathbf M}$
on $V$ is a direct sum of finite copies of $1$-dimensional spaces with the actions 
$(\rho^{-1},1,\dots,1,\overset{\underset{\smallsmile}j}\rho,1,\dots,1)
\in GL(1,\mathbb C)^{p+1}_1$ $(j=1,\dots,p)$ and $(1,1,\dots,1)$. 
The system defined by the middle convolution $mc_\mu(\mathbf A)$ is 
the quotient of the system $\mathcal M_{\tilde{\mathbf A}}$
by the maximal submodule such that the submodule is a direct sum of 
finite copies of the equations $(x-c_j)\frac{dw}{dx}=\mu w$ $(j=1,\dots,p)$ 
and $\frac{dw}{dx}=0$.

Suppose $\mathbf M$ and $\MC_\rho(\mathbf M)$ are irreducible and $\rho\ne 1$.
Assume $\phi(x)$  is a function belonging to $\mathcal F$ 
such that it is defined around $x=c_j$ and corresponds to the eigenvector
of the monodromy matrix $M_j$ with the eigenvalue different from 1.
Then the holomorphic continuation of
$\Phi(x)=\int^{(x+,c_j+,x-,c_j-)}\frac{\phi(t)(t-x)^\mu}{t-c_j}dt$
defines the monodromy isomorphic to $\MC_\rho(\mathbf M)$.

\begin{rem}\label{rem:Mon}
We can define the monodromy $\mathbf M=(M_0,\dots,M_p)$ of 
the universal model $P_{\mathbf m}u=0$ (cf.~Theorem~\ref{thm:univmodel})  
so that $\mathbf M$ is entire holomorphic with  respect to the 
spectral parameters $\lambda_{j,\nu}$ 
and the accessory parameters $ g_i$ under the normalization $u(j)^{(\nu-1)}(q)=
\delta_{j,\nu}$ for $j,\ \nu=1,\dots,n$ and 
$q\in\mathbb C\setminus\{c_1,\dots,c_p\}$.
Here $u(1),\dots,u(n)$ are solutions of $P_{\mathbf m}u=0$.
\end{rem}

\begin{defn}\label{def:locnondeg}\index{locally non-degenerate}
Let $P$ be a Fuchsian differential operator 
with the Riemann scheme \eqref{eq:GRS} and the spectral type 
$\mathbf m=\bigl(m_{j,\nu}\bigr)_{\substack{0\le j\le p\\1\le\nu\le n_j}}$.
We define that $P$ is \textsl{locally non-degenerate}
\index{Fuchsian differential equation/operator!locally non-degenerate}
if the tuple of the monodromy generators $\mathbf M:=(M_0,\dots,M_p)$ satisfies
\begin{equation}\label{eq:nondeg1}
  M_j\sim L(m_{j,1},\dots,m_{j,n_j};e^{2\pi\sqrt{-1}\lambda_{j,1}},\dots,
  e^{2\pi\sqrt{-1}\lambda_{j,n_j}})\quad(j=0,\dots,p),
\end{equation}
which is equivalent to the condition that
\index{00Z@$Z(A)$, $Z(\mathbf M)$}%
\begin{equation}\label{eq:nondeg0}
\dim Z(M_j) = m_{j,1}^2+\cdots+m_{j,n_j}^2\quad(j=0,\dots,p).
\end{equation}
Suppose $\mathbf m$ is irreducibly realizable.
Let $P_{\mathbf m}$ be the universal operator with the 
Riemann scheme \eqref{eq:GRS}.
We say that the parameters $\lambda_{j,\nu}$ and $g_i$ are
\textsl{locally non-degenerate} if the corresponding operator
is locally non-degenerate.
\end{defn}
Note that the parameters are locally non-degenerate if
\[
  \lambda_{j,\nu}-\lambda_{j,\nu'}\notin\mathbb Z
  \quad(j=0,\dots,p,\ \nu=1,\dots,n_j,\ \nu'=1,\dots,n_j).
\]
Define $P_t$ as in Remark~\ref{rem:GCexp} iv).
Then we can define monodromy generator $M_t$ of $P_t$ at $x=c_j$ 
so that $M_t$ holomorphically depend on $t$ (cf.~Remark~\ref{rem:Mon}).
Then Remark~\ref{rm:1} v) proves that 
\eqref{eq:nondeg0} implies \eqref{eq:nondeg1} for every $j$.

The following proposition gives a sufficient condition such that an operator
is locally non-degenerate.
\begin{prop}\label{prop:nondeg}
Let $P$ be a Fuchsian differential operator with the Riemann scheme \eqref{eq:GRS}
and let $M_j$ be the monodromy generator at $x=c_j$. 
Fix an integer $j$ with $0\le j\le p$.
Then the condition
\begin{equation}\label{eq:nondegA}
 \begin{split}
  &\lambda_{j,\nu}-\lambda_{j,\nu'}\notin\mathbb Z\text{ \ or \ }
  (\lambda_{j,\nu}-\lambda_{j,\nu'})
  (\lambda_{j,\nu}+m_{j,\nu}-\lambda_{j,\nu'}-m_{j,\nu'})\le 0\\
  &\qquad\text{ \ for \ }1\le \nu\le n_j\text{ \ and \ }1\le \nu'\le n_j
 \end{split}
\end{equation}
implies $\dim Z(M_j)=m_{j,1}^2+\cdots+m_{j,n_j}^2$.
In particular, $P$ is locally non-degenerate
if \eqref{eq:nondegA} is valid for $j=0,\dots,p$.
 
Here we remark that the following condition implies \eqref{eq:nondegA}.
\begin{equation}\label{eq:nondegB}
 \lambda_{j,\nu}-\lambda_{j,\nu'}\notin\mathbb Z\setminus\{0\}
 \quad\text{for \ }1\le \nu\le n_j\text{ \ and \ }1\le \nu'\le n_j.
\end{equation}
\end{prop}
\begin{proof}
For $\mu\in\mathbb C$ we put 
\[
 N_\mu=\bigl\{\nu\,;\,1\le\nu\le n_j,\ \mu\in\{\lambda_{j,\nu},\lambda_{j,\nu}+1,
 \dots,\lambda_{j,\nu}+m_{j,\nu}-1\}\bigr\}.
\]
If $N_\mu > 0$, we have a local solution $u_{\mu,\nu}(x)$ 
of the equation $Pu=0$ such that
\begin{equation}\label{eq:nondegs1}
  u_{\mu,\nu}(x)=(x-c_j)^{\mu}\log^{\nu}(x-c_j)+\mathcal O_{c_j}(\mu+1,L_\nu)
 \text{ \ for \ }\nu=0,\dots,N_\mu-1.
\end{equation}
Here $L_\nu$ are positive integers and if $j=0$, then $x$ and $x-c_j$ should 
be replaced by $y=\frac1x$ and $y$, respectively.

Suppose \eqref{eq:nondegA}.
Put $\rho=e^{2\pi \mu i}$, $\mathbf m'_\rho=\{m_{j,\nu}\,;\,
\lambda_{j,\nu}-\mu\in\mathbb Z\}$ and $\mathbf m'_\rho=\{m'_{\rho,1},\dots,
m'_{\rho,n_\rho}\}$ with $m'_{\rho,1}\ge m'_{\rho,2}\ge\cdots\ge m'_{\rho,n_\rho}
\ge 1$.
Then \eqref{eq:nondegA} implies
\begin{equation}\label{eq:nondegs2}
 n-\rank(M_j-\rho)^k \le
                     \begin{cases}
                       m'_{\rho,1}+\cdots+m'_{\rho,k}&(1\le k\le n_\rho),\\
                       m'_{\rho,1}+\cdots+m'_{\rho,n_\rho}&(n_\rho< k).
                      \end{cases}
\end{equation}
The above argument proving \eqref{eq:nondeg1} under the condition
\eqref{eq:nondeg0} shows that the left hand side of \eqref{eq:nondegs2} 
is not smaller than the right hand side of \eqref{eq:nondegs2}.
Hence we have the equality in \eqref{eq:nondegs2}.
Thus we have \eqref{eq:nondeg0} and we can assume that 
$L_\nu=\nu$ in \eqref{eq:nondegs1}.
\end{proof}
Theorem~\ref{thm:Mmid}, Theorem~\ref{thm:mcMC} and Proposition~\ref{prop:RAdIc} 
show the following corollary.
One can also prove it by the same way as in the proof of
\cite[Theorem~4.7]{DR2}. 
\begin{cor}\label{cor:irredred}
Let\/ $P$ be a Fuchsian differential operator 
with the Riemann scheme \eqref{eq:GRS}.
Let\/ $\Mon(P)$ denote the monodromy of the equation $Pu=0$.
Put\/ $\Mon(P)=(M_0,\dots,M_p)$.
Suppose 
\begin{equation}
 M_j\sim L(m_{j,1},\dots,m_{j,n_j};e^{2\pi\sqrt{-1}\lambda_{j,1}},\dots,
 e^{2\pi\sqrt{-1}\lambda_{j,n_j}}) \text{ \ for \ }j=0,\dots,p.
\end{equation}
In this case, $P$ is said to be locally non-degenerate.
Under the notation in Definition~\ref{def:pell}, we fix 
$\ell\in\mathbb Z_{\ge 1}^{p+1}$ and suppose \eqref{eq:redok}.
Assume moreover
\begin{gather}
 \mu_\ell\notin\mathbb Z,\\
 m_{j,\nu}\le m_{j,\ell_j}\text{ \ or \ }
 \lambda_{j,\ell_j}-\lambda_{j,\nu}\notin\mathbb Z\quad
 (j=0,\dots,p,\ \nu=1,\dots,n_j).
\end{gather}
Then we have
\begin{equation}
  \Mon(\p_\ell P)\sim\p_\ell\Mon(P).
\end{equation}
In particular, $\Mon(P)$ is irreducible if and only if
$\Mon(\p_\ell P)$ is irreducible.
\end{cor}
\subsection{Scott's lemma and Katz's rigidity}\label{sec:Scott}
The results in this subsection are known but we will review them with
their proof for the completeness of this paper.
\begin{lem}[Scott \cite{Sc}]\index{Scott's lemma}
\label{lem:Scott}
Let\/ $\mathbf M\in GL(n,\mathbb C)^{p+1}_1$ and\/ 
$\mathbf A\in M(n,\mathbb C)^{p+1}_0$
under the notation \eqref{eq:Msub0} and \eqref{eq:Gsub1}.
Then 
\begin{align}
   \sum_{j=0}^p \codim\ker(M_j-1)&\ge
   \codim\bigcap_{j=0}^p\ker(M_j-1)+\codim\bigcap_{j=0}^p\ker({}^t\!M_j-1),\\
   \sum_{j=0}^p \codim\ker A_j&\ge
   \codim\bigcap_{j=0}^p\ker A_j +\codim\bigcap_{j=0}^p\ker{}^t\!A_j.
\end{align}
In particular, if\/ $\mathbf M$ and $\mathbf A$ are irreducible, then
\begin{align}
 \sum_{j=0}^p\dim\ker(M_j-1)&\le (p-1)n,\\
 \sum_{j=0}^p\dim\ker A_j&\le (p-1)n.
\end{align}
\end{lem}
\begin{proof}
Consider the following linear maps:
\begin{align*}
 V&=\IM(M_0-1)\times\cdots\times \IM(M_p-1)\subset\mathbb C^{n(p+1)},\\
 \beta&:\ \mathbb C^n\to V,\quad v\mapsto ((M_0-1)v,\dots,(M_p-1)v),\\
 \delta&:\ V\to \mathbb C^n,\quad (v_0,\dots,v_p)
  \mapsto M_p\cdots M_1v_0+M_p\cdots M_2v_1+\cdots+M_pv_{p-1}+v_p.
\end{align*}
Since
$M_p\cdots M_1(M_0-1)+\cdots+M_p(M_{p-1}-1)+(M_p-1)
=M_p\cdots M_1M_0-1=0$, we have $\delta\circ\beta=0$.
Moreover we have
\begin{align*}
 &\sum_{j=0}^pM_p\cdots M_{j+1}(M_j-1)v_j
 =\sum_{j=0}^p\Bigl(1+\sum_{\nu=j+1}^{p}(M_\nu-1)M_{\nu-1}
 \cdots M_{j+1}\Bigr)(M_j-1)v_j\\
 &\quad=\sum_{j=0}^p(M_j-1)v_j+\sum_{\nu=1}^{p}\sum_{i=0}^{\nu-1}
 (M_\nu-1)M_{\nu-1}\cdots M_{i+1}(M_i-1)v_i\\
 &\quad=\sum_{j=0}^p(M_j-1)\Bigl(v_j +
  \sum_{i=0}^{j-1} M_{j+1}\cdots M_{i-1}(M_i-1)v_i\Bigr)
\end{align*}
and therefore $\IM\delta=\sum_{j=0}^p\IM(M_j-1)$.
Hence
\begin{align*}
 \dim\IM\delta&=\rank (M_0-1,\dots,M_p-1)
              =\rank \begin{pmatrix}{}^t\!M_0-1\\\vdots\\ {}^t\!M_p-1\end{pmatrix}
\\[-.3cm]
\intertext{and}
 \sum_{j=0}^p\codim\ker(M_j-1)&=\dim V=\dim\ker\delta+\dim\IM\delta\\
                 &\ge\dim\IM\beta+\dim\IM\delta\\
                 &=\codim\bigcap_{j=0}^p\ker(M_j-1)
                  +\codim\bigcap_{j=0}^p\ker({}^t\!M_j-1).
\end{align*}

Putting
\begin{align*}
 V&=\IM A_0\times\cdots\times \IM A_p\subset\mathbb C^{n(p+1)},
\allowdisplaybreaks\\
 \beta&:\ \mathbb C^n\to V,\quad v\mapsto (A_0 v,\dots, A_p v),
\allowdisplaybreaks\\
 \delta&:\ V\to \mathbb C^n,\quad (v_0,\dots,v_p)
  \mapsto v_0+v_1+\cdots+v_p,
\end{align*}
we have the claims for $\mathbf A\in M(n,\mathbb C)^{p+1}$
in the same way as in the proof for $\mathbf M\in GL(n,\mathbb C)^{p+1}_1$.
\end{proof}
\begin{cor}[Katz \cite{Kz} and \cite{SV}]\label{cor:katz}
\index{Katz's rigidity}
Let $\mathbf M\in GL(n,\mathbb C)^{p+1}_1$.
The dimensions of the manifolds
\begin{align}
  V_1 &:=\{\mathbf H\in GL(n,\mathbb C)^{p+1}_1\,;\,\mathbf H\sim\mathbf M\}
\intertext{and}
  V_2 &:=\{\mathbf H\in GL(n,\mathbb C)^{p+1}_1\,;\,
  H_j\sim M_j\quad(j=0,\dots,p)\}
\end{align}
are give by
\index{00Z@$Z(A)$, $Z(\mathbf M)$}%
\begin{align}
 \dim V_1 &= \codim Z(\mathbf M),\\
 \dim V_2 &=\sum_{j=0}^p\codim Z(M_j)-\codim Z(\mathbf M).
\end{align}
Here $Z(\mathbf M):=\bigcap_{j=0}^p Z(M_j)$ and 
$Z(M_i)=\{X\in M(n,\mathbb C)\,;\,XM_j=M_jX\}$.

Suppose\/ $\mathbf M$ is irreducible.
Then\/ $\codim Z(\mathbf M)=n^2-1$ and 
\begin{align}\label{eq:iridx}
 \sum_{j=0}^p\codim Z(M_j)\ge 2n^2-2.
\end{align}
Moreover\/ $\mathbf M$ is \textsl{rigid}, namely, $V_1=V_2$
if and only if\/ $\displaystyle\sum_{j=0}^p\codim Z(M_j)=2n^2-2$.
\end{cor}
\begin{proof}
The group $GL(n,\mathbb C)$ transitively acts on $V_1$ as simultaneous 
conjugations and the isotropy group with respect to $\mathbf M$ equals 
$Z(\mathbf M)$ and hence $\dim V_1=\codim Z(\mathbf M)$.

The group $GL(n,\mathbb C)^{p+1}$ naturally acts on 
$GL(n,\mathbb C)^{p+1}$ by conjugations.
Putting $L=\{(g_j)\in GL(n,\mathbb C)^{p+1}\,;\,g_pM_pg_p^{-1}
\cdots g_0M_0g_0^{-1}=M_p\cdots M_0\}$, $V_2$ is identified with
$L/Z(M_0)\times\cdots\times Z(M_p)$.
Denoting $g_j=\exp(t X_j)$ with $X_j\in M(n,\mathbb C)$ and $t\in\mathbb R$
with $|t|\ll1$ and defining  $A_j\in\End\bigl(M(n,\mathbb C)\bigr)$ by 
$A_j X = M_jXM_j^{-1}$, we can prove that the dimension of $L$ equals the 
dimension of the kernel of the map
\[
 \gamma: M(n,\mathbb C)^{p+1}\ni
 (X_0,\dots,X_p) \mapsto
 \sum_{j=0}^p A_p\cdots A_{j+1}(A_j-1)X_j
\]
by looking at the tangent space of $L$ at the identity element because
\begin{align*}
 &\exp(t X_p)M_p\exp(-t X_p)\cdots \exp(tX_0)M_0(-t X_0) - M_p\cdots M_0\\
 &\quad= t\Bigl(\sum_{j=0}^p A_p\cdots A_{j+1}(A_j-1)X_j\Bigr) M_p\cdots M_0+ o(t).
\end{align*}
We have obtained in the proof of Lemma~\ref{lem:Scott} that 
$\codim\ker\gamma=\dim\IM\gamma=\dim\sum_{j=0}^p\IM(A_j-1)
=\codim\bigcap_{j=0}^p\ker({}^t\!A_j-1)$.
We will see that $\bigcap_{j=0}^p\ker({}^t\!A_j-1)$ is identified with 
$Z(\mathbf M)$ and hence
$\codim\ker\gamma=\codim Z(\mathbf M)$
and
\[\dim V_2=\dim \ker\gamma-\sum_{j=0}^p\dim Z(M_j)
=\sum_{j=0}^p\codim Z(M_j)-\codim Z(\mathbf M).
\]

In general, fix $\mathbf H\in V_1$ and define 
$A_j\in \End\bigl(M(n,\mathbb C)\bigr)$ by $X\mapsto M_jXH_j^{-1}$ 
for $j=0,\dots,p$.
Note that $A_pA_{p-1}\cdots A_0$ is the identity map.
If we identify $M(n,\mathbb C)$ with its dual by the inner product 
$\trace XY$ for $X$, $Y\in M(n,\mathbb C)$, ${}^t\!A_j$ are identified with the
map $Y\mapsto H_j^{-1}YM_j$, respectively.

Fix $P_j\in GL(n,\mathbb C)$ such that $H_j=P_jM_jP_j^{-1}$.
Then
\[
 \begin{split}
A_j(X)=X \Leftrightarrow&\ M_jXH_j^{-1}=X
 \Leftrightarrow M_jX=XP_jM_jP_j^{-1}
 \Leftrightarrow M_jXP_j=XP_jM_j,\\
 {}^t\!A_j(X)=X \Leftrightarrow&\ H_j^{-1}XM_j=X
 \Leftrightarrow XM_j=P_jM_jP_j^{-1}X
 \Leftrightarrow P_j^{-1}XM_j=M_jP_j^{-1}X
\end{split}
\]
and $\codim\ker(A_j-1)=\codim Z(M_j)$ and 
$\bigcap_{j=0}^p\ker({}^t\!A_j-1)\simeq Z(\mathbf M)$.

Suppose
$\mathbf M$ is irreducible.
Then $\codim Z(\mathbf M)=n^2-1$ and the inequality \eqref{eq:iridx} follows 
from $V_1\subset V_2$.
Moreover suppose $\sum_{j=0}^p\codim Z(M_i)=2n^2-2$. Then
Scott's lemma proves
\[
\begin{split}
  2n^2-2&=\sum_{j=0}^p\codim\ker(A_j-1)\\
  &\ge n^2- \dim \bigcap_{j=0}^p 
  \{X\in M(n,\mathbb C)\,;\,M_jX=XH_j\}\\
  &\quad+
  n^2- \dim \bigcap_{j=0}^p 
  \{X\in M(n,\mathbb C)\,;\,H_jX=XM_j\}.
 \end{split}
\]
Hence there exists a non-zero matrix $X$
such that $M_jX=XH_j$ ($j=0,\dots,p$) or
$H_jX=XM_j$ ($j=0,\dots,p$). 
If $M_jX=XH_j$ (resp.~$H_jX=XM_j$) for $j=0,\dots,p$, 
$\ker X$ (resp.~$\IM X$) is $M_j$-stable for $j=0,\dots,p$ and hence
$X\in GL(n,\mathbb C)$ because $\mathbf M$ is irreducible,
Thus we have $V_1=V_2$
and we get all the claims in the corollary.
\end{proof}
\section{Reducibility}
\subsection{Direct decompositions}\label{sec:reddirect}
For a realizable $(p+1)$-tuple $\mathbf m\in\mathcal P^{(n)}_{p+1}$,
Theorem~\ref{thm:univmodel} gives the universal 
Fuchsian differential operator $P_{\mathbf m}(\lambda_{j,\nu},g_i)$ 
with the Riemann scheme
\eqref{eq:GRS}.
Here $g_1,\dots,g_N$ are accessory parameters and $N=\Ridx\mathbf m$.

First suppose $\mathbf m$ is basic. 
Choose positive numbers $n'$, $n''$, $m'_{j,1}$ and $m''_{j,1}$ 
such that
\begin{equation}
 \begin{gathered}
  n= n'+n'',\quad 0<m'_{j,1}\le n',\quad 0<m''_{j,1}\le n'',\\
  m'_{0,1}+\cdots+m'_{p,1}\le (p-1)n',\quad
  m''_{0,1}+\cdots+m''_{p,1}\le (p-1)n''.
 \end{gathered}
\end{equation}
We choose other positive integers $m'_{j,\nu}$ and $m''_{j,\nu}$
so that $\mathbf m'=\bigl(m'_{j,\nu}\bigr)$ and 
$\mathbf m''=\bigl(m''_{j,\nu}\bigr)$ are monotone tuples of 
partitions of $n'$ and $n''$, respectively, and moreover
\begin{equation}\label{eq:tsum}
 \mathbf m=\mathbf m'+\mathbf m''.
\end{equation}
Theorem~\ref{thm:ExF} shows that $\mathbf m'$ and $\mathbf m''$ are realizable.
If $\{\lambda_{j,\nu}\}$ satisfies the Fuchs relation 
\begin{equation}\label{eq:FuchsProd}
 \sum_{j=0}^p\sum_{\nu=1}^{n_j} m'_{j,\nu}\lambda_{j,\nu}
 = n' - \frac{\idx\mathbf m'}2
\end{equation}
for the Riemann scheme $\bigl\{[\lambda_{j,\nu}]_{(m'_{j,\nu})}\bigr\}$,
Theorem~\ref{thm:prod} shows that the operators
\begin{equation}\label{eq:2prod}
  P_{\mathbf m''}(\lambda_{j,\nu}+m'_{j,\nu}-\delta_{j,0}(p-1)n',g''_i)
  \cdot
  P_{\mathbf m'}(\lambda_{j,\nu},g'_i)
\end{equation}
has the Riemann scheme $\{[\lambda_{j,\nu}]_{(m_{j,\nu})}\}$.
This shows that the equation 
$P_{\mathbf m}(\lambda_{j,\nu},g_i)u=0$ is not irreducible when 
the parameters take the values corresponding to \eqref{eq:2prod}.

In this subsection, we study the condition
\begin{equation}\label{eq:Rsum}
 \Ridx\mathbf m=\Ridx\mathbf m'+\Ridx\mathbf m''
\end{equation}
for realizable tuples $\mathbf m'$ and $\mathbf m''$ with
$\mathbf m=\mathbf m'+\mathbf m''$.
Under this condition the Fuchs relation \eqref{eq:FuchsProd} assures that 
the universal operator is reducible for any values of accessory parameters.
\begin{defn}[direct decomposition]
\index{direct decomposition}
If realizable tuples $\mathbf m$, $\mathbf m'$ and $\mathbf m''$
satisfy \eqref{eq:tsum} and \eqref{eq:Rsum}, we define that $\mathbf m$
is the \textsl{direct sum} of $\mathbf m'$ and $\mathbf m''$ and call 
$\mathbf m=\mathbf m'+\mathbf m''$ a \textsl{direct decomposition}
of $\mathbf m$ and express it as follows.
\index{tuple of partitions!direct decomposition}
\begin{equation}\label{eq:dsum}
 \mathbf m=\mathbf m'\oplus\mathbf m''.
\end{equation}
\end{defn}
\begin{thm}\label{thm:ddecmp}
Let \eqref{eq:dsum} be a direct decomposition
of a realizable tuple $\mathbf m$.

{\rm i) }
%
Suppose $\mathbf m$ is irreducibly realizable and\/ $\idx \mathbf m''>0$. 
Put\/ ${\overline{\mathbf m}'}=\gcd(\mathbf m')^{-1}\mathbf m'$.
If\/ $\mathbf m'$ is indivisible or\/ $\idx\mathbf m\le 0$, then
\begin{align}\label{eq:mref}
 \alpha_{\mathbf m}&=\alpha_{\mathbf m'}-2
         \frac{(\alpha_{{\overline{\mathbf m}}''}|\alpha_{\mathbf m'})}
      {(\alpha_{{\overline{\mathbf m}''}}|\alpha_{{\overline{\mathbf m}}''})}
       \alpha_{{\overline{\mathbf m}}''}
\end{align}
or $\mathbf m=\mathbf m'\oplus\mathbf m''$ is isomorphic to one of
the decompositions
\begin{equation}
 \begin{aligned}
  32,32,32,221&=22,22,22,220\oplus10,10,10,10,001\\
  322,322,2221&=222,222,2220\oplus100,100,0001\\
  54,3222,22221&=44,2222,22220\oplus10,1000,00001\\
  76,544,2222221&=66,444,2222220\oplus10,100,0000001\\
 \end{aligned}
\end{equation}
under the action of $\widetilde W_{\!\infty}$.

{\rm ii) }
Suppose $\idx\mathbf m \le 0$ and $\idx\mathbf m'\le 0$ and 
$\idx\mathbf m''\le 0$.
Then $\mathbf m=\mathbf m'\oplus\mathbf m''$ or
$\mathbf m=\mathbf m''\oplus\mathbf m'$ is 
transformed into one of the decompositions
\begin{equation}
\begin{aligned}
  \Sigma=11,11,11,11\ \ \ &111,111,111\ \ \ 22,1^4,1^4\ \ \ 33,222,1^6\\
m\Sigma&=k\Sigma\oplus \ell\Sigma\\ 
mm,mm,mm,m(m-1)1&=
  kk,kk,kk,k(k-1)1\oplus\ell\ell,\ell\ell,\ell\ell,\ell\ell0\\
mmm,mmm,mm(m-1)1&=kkk,kkk,kkk,kk(k-1)1
 \oplus\ell\ell\ell,\ell\ell\ell,\ell\ell\ell0\\
(2m)^2,m^4,mmm(m-1)1&=(2k)^2,k^4,k^4,kkk(k-1)1
 \oplus(2\ell)^2,\ell^4,\ell^40\\
 (3m)^2,(2m)^3,m^5(m-1)1&=(3k)^2,(2k)^3,k^5(k-1)1
 \oplus (3\ell)^2,(2\ell)^3,\ell^60
\end{aligned}
\end{equation}
under the action of\/ $\widetilde W_{\!\infty}$.
Here $m$, $k$ and $\ell$ are positive integers satisfying
$m=k+\ell$.
\index{00D4z@$D_4^{(m)},\ E_6^{(m)},\ E_7^{(m)},\ E_8^{(m)}$}
These are expressed by
\begin{equation}\begin{aligned}
  m\tilde D_4&=k\tilde D_4\oplus\ell\tilde D_4,&
  m\tilde E_j&=k\tilde E_j\oplus\ell\tilde E_j\quad(j=6,7,8),\\
  D_4^{(m)}&=D_4^{(k)}\oplus\ell\tilde D_4,&
  E_j^{(m)}&=E_j^{(k)}\oplus\ell\tilde E_j\quad(j=6,7,8).
\end{aligned}\end{equation}
\end{thm}
\begin{proof}
Put $\mathbf m'=k{\overline{\mathbf m}}'$ and $\mathbf m''
 =\ell{\overline{\mathbf m}}''$ with indivisible 
${\overline{\mathbf m}}'$ and ${\overline{\mathbf m}}''$.
First note that
\begin{equation}\label{eq:squaresum}
(\alpha_{\mathbf m}|\alpha_{\mathbf m})=
 (\alpha_{\mathbf m'}|\alpha_{\mathbf m'})
 +2(\alpha_{\mathbf m'}|\alpha_{\mathbf m''})
 +(\alpha_{\mathbf m''}|\alpha_{\mathbf m''}).
\end{equation} 

{\rm ii) } Using Lemma~\ref{lem:sumlem}, we will prove the theorem.
If $\idx\mathbf m=0$, then 
\eqref{eq:squaresum} and \eqref{eq:imneg} show
$0=(\alpha_{\mathbf m'}|\alpha_{\mathbf m''})=k\ell
(\alpha_{{\overline{\mathbf m}}'}|\alpha_{{\overline{\mathbf m}}''})$,
Lemma~\ref{lem:sumlem} proves 
$\idx\mathbf m'=0$ and ${\overline{\mathbf m}}'={\overline{\mathbf m}}''$ and 
we have the theorem.

Suppose $\idx\mathbf m<0$.

If $\idx\mathbf m'<0$ and $\idx\mathbf m''<0$, we have 
$\Pidx\mathbf m=\Pidx\mathbf m'+\Pidx\mathbf m''$, 
which implies 
$(\alpha_{\mathbf m'}|\alpha_{\mathbf m''})=-1$ and contradicts
to Lemma~\ref{lem:sumlem}.

Hence we may assume $\idx\mathbf m''=0$.

\underline{Case: $\idx\mathbf m'<0$}.
It follows from  \eqref{eq:squaresum} that
$2-2\Ridx\mathbf m=2-2\Ridx\mathbf m'+2\ell(\mathbf m,{\overline{\mathbf m}})$.
Since $\Ridx\mathbf m=\Ridx\mathbf m'+\ell$, 
we have $(\alpha_{\mathbf m}|\alpha_{\overline{\mathbf m}'})=-1$ 
and the theorem follows from Lemma~\ref{lem:sumlem}.

\underline{Case: $\idx\mathbf m'=0$}.
It follows from \eqref{eq:squaresum} that $2-2\Ridx\mathbf m=2k\ell
(\alpha_{{\overline{\mathbf m}}'}|\alpha_{{\overline{\mathbf m}}''})$.
Since the condition $\Ridx\mathbf m=k+\ell$ shows
$(\alpha_{{\overline{\mathbf m}}'}|\alpha_{{\overline{\mathbf m}}''})
=\frac1{k\ell}-\frac1k-\frac1\ell$ and we have 
$(\alpha_{{\overline{\mathbf m}}'}|\alpha_{{\overline{\mathbf m}}''})=-1$.
Hence the theorem also follows from Lemma~\ref{lem:sumlem}.

{\rm i) } 
First suppose $\idx\mathbf m'\ne0$.
Note that $\mathbf m$ and $\mathbf m'$ are rigid 
if $\idx\mathbf m'>0$.
We have $\idx\mathbf m=\idx\mathbf m'$ and
$\idx\mathbf m
 =(\alpha_{\mathbf m'}+\ell\alpha_{{\overline{\mathbf m}}''}|
   \alpha_{\mathbf m'}+\ell\alpha_{{\overline{\mathbf m}}''})
 = \idx\mathbf m'+2\ell(\alpha_{\mathbf m}|\alpha_{{\overline{\mathbf  m}}''})
  +2\ell^2$, 
which implies \eqref{eq:mref}.

Thus we may assume $\idx\mathbf m<0$ and $\idx\mathbf m'=0$.
If $k=1$, $\idx\mathbf m=\idx\mathbf m'=0$ and we have \eqref{eq:mref}
as above.
Hence we may moreover assume $k\ge 2$.
Then \eqref{eq:squaresum} and the assumption imply
$2-2k=2k\ell(\alpha_{{\overline{\mathbf m}}'}|
 \alpha_{{\overline{\mathbf m}}''})+2\ell^2$, which means
\[
 -(\alpha_{{\overline{\mathbf m}}'}|
  \alpha_{{\overline{\mathbf m}}''})=\frac{k-1+\ell^2}{k\ell}.
\]
Here $k$ and $\ell$ are mutually prime and hence there exists
a positive integer $m$ with $k=m\ell+1$ and
\[
   -(\alpha_{{\overline{\mathbf m}}'}|
  \alpha_{{\overline{\mathbf m}}''}) = \frac{m+\ell}{m\ell+1}
  = \frac{1}{\ell+\frac1m}+\frac{1}{m+\frac1\ell}<2.
\] 
Thus we have $m=\ell=1$, $k=2$ and 
$(\alpha_{{\overline{\mathbf m}}'}|
  \alpha_{{\overline{\mathbf m}}''})=-1$.
By the transformation of an element of $\widetilde W_{\!\infty}$,
we may assume $\overline{\mathbf m}'\in\mathcal P_{p+1}$ 
is a tuple in \eqref{eq:aftmp}.
Since
$(\alpha_{\overline{\mathbf m}'}|\alpha_{\overline{\mathbf m}''})=-1$ and
$\alpha_{\overline{\mathbf m}''}$ is a positive real root, we have 
the theorem by a similar argument as in the proof of Lemma~\ref{lem:sumlem}.
Namely, $m'_{p,n'_p}=2$ and $m'_{p,n'_p+1}=0$ and we may assume
$m''_{j,n'_j+1}=0$ for $j=0,\dots,p-1$ and 
$m''_{p,n'_p+1}+m''_{p,n'_p+2}+\cdots=1$,
which proves the theorem in view of $\alpha_{\mathbf m''}\in\Delta^{re}_+$.
\end{proof}
\begin{lem}\label{lem:sumlem}
Suppose $\mathbf m$ and $\mathbf m'$ are realizable and 
$\idx\mathbf m\le 0$ and $\idx\mathbf m'\le0$.
Then
\begin{align}\label{eq:imneg}
 (\alpha_{\mathbf m}|\alpha_{\mathbf m'})\le0.
\end{align}

If\/ $\mathbf m$ and\/ $\mathbf m'$ are basic and monotone,
\begin{equation}\label{eq:maxidx}
  (\alpha_{\mathbf m}|w\alpha_{\mathbf m'})\le 
   (\alpha_{\mathbf m}|\alpha_{\mathbf m'})
  \qquad(\forall w\in W_{\!\infty}).
\end{equation}

If\/ $(\alpha_{\mathbf m}|\alpha_{\mathbf m'})=0$ and 
$\mathbf m$ and $\mathbf m'$ are indivisible,
then\/ $\idx\mathbf m=0$ and $\mathbf m=\mathbf m'$.

If\/ $(\alpha_{\mathbf m}|\alpha_{\mathbf m'})=-1$, then the pair
is isomorphic to one of the pairs 
\begin{equation}\label{eq:idx-1}
\begin{aligned}
 (D_4^{(k)},\tilde D_4):&\ \bigl((kk,kk,kk,k(k-1)1),&&(11,11,11,110)\bigr)\\
 (E_6^{(k)},\tilde E_6):&\ \bigl((kkk,kkk,k(k-1)1), &&(111,111,1110)\bigr)\\
 (E_7^{(k)},\tilde E_7):&\ \bigl(((2k)^2,kkkk,kkk(k-1)1), &&(22,1111,11110)\bigr)\\
 (E_8^{(k)},\tilde E_8):&\ \bigl(((3k)^2,(2k)^3,kkkkk(k-1)1),&&(33,222,1111110)\bigr)
\end{aligned}
\end{equation}
under the action of\/ $\widetilde W_{\!\infty}$.
\end{lem}
\begin{proof}
We may assume that $\mathbf m$ and $\mathbf m'$ are indivisible.
Under the transformation of the Weyl group, we may assume that
$\mathbf m$ is a basic monotone tuple in $\mathcal P_{p+1}$, 
namely, $(\alpha_{\mathbf m}|\alpha_0)\le 0$ and 
$(\alpha_{\mathbf m}|\alpha_{j,\nu})\le0$.

If $\mathbf m'$ is basic and monotone, 
$w\alpha_{\mathbf m'}-\alpha_{\mathbf m'}$ 
is a sum of positive real roots, which proves \eqref{eq:maxidx}.

Put $\alpha_{\mathbf m}=n\alpha_0+\sum n_{j,\nu}\alpha_{j,\nu}$
and $\mathbf m'=n'_0\alpha_0+\sum n'_{j,\nu}\alpha_{j,\nu}$.
Then
\begin{equation}
\begin{split}
(\alpha_{\mathbf m}|\alpha_{\mathbf m'})&=
 n'_0(\alpha_{\mathbf m}|\alpha_0)+
 \sum n'_{j,\nu}(\alpha_{\mathbf m}|\alpha_{j,\nu}),\\
(\alpha_{\mathbf m}|\alpha)&\le 0\quad(\forall\alpha\in\supp\alpha_{\mathbf m}).
\end{split}
\end{equation}
Let 
$k_j$ be the maximal positive integer satisfying $m_{j,k_j}=m_{j,1}$
and put $\Pi_0=\{\alpha_0,\alpha_{j,\nu}\,;\,1\le\nu<k_j,\ j=0,\dots,p\}$.
Note that $\Pi_0$ defines a classical root system if $\idx\mathbf m<0$
(cf.~Remark~\ref{rem:classinbas}).

Suppose $(\alpha_{\mathbf m}|\alpha_{\mathbf m'})=0$ 
and $\mathbf m\in\mathcal P_{p+1}$.
Then $m_{0,1}+\cdots+m_{p,1}=(p-1)\ord\mathbf m$
and $\supp\alpha_{\mathbf m'}\subset\Pi_0$
because $(\alpha_{\mathbf m}|\alpha)=0$ for $\alpha\in\supp\alpha_{\mathbf m'}$.
Hence it follows from $\idx\mathbf m'\le0$ that $\idx\mathbf m=0$ 
and we may assume that $\mathbf m$ is one of the tuples \eqref{eq:aftmp}. 
Since $\supp\alpha_{\mathbf m'}\subset\supp\alpha_{\mathbf m}$ and
$\idx\mathbf m'\le 0$, we conclude that $\mathbf m'=\mathbf m$.

Lastly suppose $(\alpha_{\mathbf m}|\alpha_{\mathbf m'})=-1$.

\underline{Case: $\idx\mathbf m=\idx\mathbf m'=0$}.
If $\mathbf m'$ is basic and monotone and $\mathbf m'\ne\mathbf m$,
then it is easy to see that $(\alpha_{\mathbf m}|\alpha_{\mathbf m'})<-1$ 
(cf.~Remark~\ref{rem:Kac}).  
Hence \eqref{eq:maxidx} assures $\mathbf m'=w\mathbf m$ 
with a certain $w\in W_{\!\infty}$ and
therefore $\supp\mathbf m\subsetneq\supp\mathbf m'$.
Moreover there exists $j_0$ and $L\ge k_{j_0}$ such that
$\supp m'=\supp m\cup\{\alpha_{j_0,k_{j_0}},\alpha_{j_0,k_j+1},\dots,
\alpha_{j_0,L}\}$ and $m_{j_0,k_{j_0}}=1$ and $m'_{j_0,k_{j_0}+1}=1$.
Then by a transformation of an element of the Weyl group, we may assume 
$L=k_{j_0}$ and $\mathbf m'=r_{i_N}\cdots r_{i_1}
r_{(j_0,k_{j_0})}\mathbf m$ with suitable $i_\nu$ satisfying
$\alpha_{i_\nu}\in\supp\mathbf m$ for $\nu=1,\dots,N$.
Applying $r_{i_1}\cdots r_{i_N}$ to the pair 
$(\mathbf m,\mathbf m')$, we may assume $\mathbf m'=r_{(j_0,k_{j_0})}\mathbf m$.
Hence the pair $(\mathbf m,\mathbf m')$ is isomorphic to one of
the pairs in the list \eqref{eq:idx-1} with $k=1$.

\underline{Case: $\idx\mathbf m<0$ and $\idx\mathbf m'\le 0$}.
There exists $j_0$ such that 
$\supp\alpha_{\mathbf m'}\ni\alpha_{j_0,k_j}$.
Then the fact $\idx(\mathbf m,\mathbf m')=-1$ implies 
$n'_{j_0,k_0}=1$ and $n'_{j,k_j}=0$ for $j\ne j_0$. 
Let $L$ be the maximal positive integer with $n'_{j_0,L}\ne0$.
Since $(\alpha_{\mathbf m}|\alpha_{j_0,\nu})=0$ for $k_0+1\le \nu\le L$,
we may assume $L=k_0$ by the transformation 
$r_{(j_0,k_0+1)}\circ \cdots\circ r_{(j_0,L)}$ if $L>k_0$.
Since the Dynkin diagram corresponding to 
$\Pi_0\cup\{\alpha_{j_0,k_{j_0}}\}$ is classical or affine
and $\supp\mathbf m'$ is contained in this set,
$\idx\mathbf m'=0$ and $\mathbf m'$ is basic
and we may assume that $\mathbf m'$ is one of the tuples
\begin{equation}\label{eq:aftmp}
  11,11,11,11\quad 111,111,111 \quad 22,1111,1111 \quad 33,222,111111
\end{equation}
and $j_0=p$.
In particular $m'_{p,1}=\cdots=m'_{p,k_p}=1$ and $m'_{p,k_p+1}=0$.
It follows from  $(\alpha_{\mathbf m}|\alpha_{p,k_p})=-1$
that there exists an integer $L'\ge k_p+1$ satisfying 
$\supp\mathbf m=\supp\mathbf m'\cup\{\alpha_{p,\nu}\,;\,k_p\le \nu< L'\}$
and $m_{p,k_p}=m_{p,k_p-1}-1$.
In particular, $m_{j,\nu}=m_{j,1}$ for $\nu=1,\dots,k_j-\delta_{j,p}$
and $j=0,\dots,p$.
Since $\sum_{j=0}^pm_{j,1}=(p-1)\ord\mathbf m$, there exists a positive
integer $k$ such that
\[
  m_{j,\nu}=
  \begin{cases}
   km'_{j,1}&(j=0,\dots,p,\ \nu=1,\dots,k_j-\delta_{j,p}),\\
   km'_{p,1}-1&(j=p,\ \nu=k_p).
  \end{cases}
\]
Hence $m_{p,k_p+1}=1$ and $L'=k_p+1$ and the pair $(\mathbf m,\mathbf m')$
is one of the pairs in the list \eqref{eq:idx-1} with $k>1$.
\end{proof}
\begin{rem}
Let $k$ be an integer with $k\ge2$ and let $P$ be a differential operator 
with the spectral type $D_4^{(k)}$, $E_6^{(k)}$, $E_7^{(k)}$ or $E_8^{(k)}$.
It follows from Theorem~\ref{thm:prod} and Theorem~\ref{thm:univmodel} that $P$ 
is reducible for any values of accessory parameters when the characteristic 
exponents satisfy Fuchs relation with respect to the subtuple given 
in \eqref{eq:idx-1}.
For example, the Fuchsian differential operator $P$ with the Riemann scheme
\begin{align*}
  &\begin{Bmatrix}
    [\lambda_{0,1}]_{(k)} &  [\lambda_{1,1}]_{(k)} & [\lambda_{2,1}]_{(k)}
    & [\lambda_{3,1}]_{(k)}\\
    [\lambda_{0,2}]_{(k)} &  [\lambda_{1,2}]_{(k)} & [\lambda_{2,2}]_{(k)}
    & [\lambda_{3,2}]_{(k-1)}\\
    &&& \lambda_{3,2}+2k-2
  \end{Bmatrix}
\end{align*}
is reducible.
\end{rem}
\begin{exmp}\label{ex:JPH}
i) \ (generalized Jordan-Pochhammer)\index{Jordan-Pochhammer!generalized}
If $\mathbf m=k\mathbf m'\oplus\ell\mathbf m''$ with a rigid
tuples $\mathbf m$, $\mathbf m'$ and $\mathbf m''$ and 
positive integers $k$ and $\ell$ satisfying $1\le k\le \ell$, we have
\begin{equation}
  (\alpha_{\mathbf m'}|\alpha_{\mathbf m''})=-\frac{k^2+\ell^2-1}{k\ell}
  \in\mathbb Z.
\end{equation}
For positive integers $k$ and $\ell$ satisfying $1\le k\le \ell$ and
\begin{equation}\label{eq:pkl}
 p:=\frac{k^2+\ell^2-1}{k\ell}+1\in\mathbb Z,
\end{equation}
we have an example of direct decompositions
\begin{equation}\label{eq:kl}
\begin{split}
 \overbrace{\ell k,\ell k,\ldots,\ell k}^{p+1\text{ partitions}} 
 &= 0k,0k,\ldots,0k\oplus\ell0,\ell0,\ldots,\ell0\\
 &=((p-1)k-\ell)k,((p-1)k-\ell)k,\ldots,((p-1)k-\ell)k\\
 &\quad
  \oplus(2\ell-(p-1)k)0,(2\ell-(p-1)k)0,\ldots,(2\ell-(p-1)k)0.
\end{split}\end{equation}
Here $p=3+\frac{(k-\ell)^2-1}{k\ell}\ge 2$
and the condition $p=2$ implies $k=\ell=1$ and the condition $p=3$ implies
$\ell=k+1$.
If $k=1$, then $(\alpha_{\mathbf m'}|\alpha_{\mathbf m''})=-\ell$ and 
we have an example corresponding to Jordan-Pochhammer equation:
\index{Jordan-Pochhammer}
\begin{equation}
 \overbrace{\ell1,\cdots,\ell1}^{\ell+2\text{ partitions}}
 =01,\cdots,01\oplus \ell0,\cdots,\ell0.\\
\end{equation}
When $\ell=k+1$, we have 
$(\alpha_{\mathbf m'}|\alpha_{\mathbf m''})=-2k$
and an example
\begin{equation}\begin{split}
 &(k+1)k,(k+1)k,(k+1)k,(k+1)k\\
  &\quad=0k,0k,0k,0k\oplus(k+1)0,(k+1)0,(k+1)0,(k+1)0\\
  &\quad=(k-1)k,(k-1)k,(k-1)k,(k-1)k\oplus 20,20,20,20.
\end{split}\end{equation}
We have another example
\begin{equation}\begin{split}
  83,83,83,83,83&=03,03,03,03,03\oplus 80,80,80,80,80\\
                &=13,13,13,13,13\oplus 70,70,70,70,70
\end{split}\end{equation}
in the case $(k,\ell)=(3,8)$, which is a special case
where $\ell=k^2-1$, $p=k+1$ and 
$(\alpha_{\mathbf m'}|\alpha_{\mathbf m''})=-k$.

When $p$ is odd, the equation \eqref{eq:pkl} is equal to the Pell equation
\begin{equation}
  y^2 - (m^2-1)x^2=1
\end{equation}
by putting $p-1=2m$, $x=\ell$ and $y=m\ell-k$ and hence the reduction of the tuple of partition
\eqref{eq:kl} by $\p_{\max}$ and 
its inverse give all the integer solutions of this Pell equation.
\index{Pell equation}

The tuple of partitions 
$\ell k,\ell k,\ldots,\ell k\in\mathcal P_{p+1}^{(\ell+k)}$ with 
\eqref{eq:pkl} is called a \textsl{generalized Jordan-Pochhammer} tuple and 
denoted by $P_{p+1,\ell+k}$. 
In particular, $P_{n+1,n}$ is simply denoted by $P_n$.
\index{00P@$P_{p+1,n},\ P_n$}

ii)
We give an example of direct decompositions of a rigid tuple:
\begin{align*}
3322,532,532&=0022,202,202\oplus 3300,330,330:1\\
            &=1122,312,312\oplus 2200,220,220:1\\
            &=0322,232,232\oplus 3000,300,300:2\\
            &=3302,332,332\oplus 0020,200,200:2\\
            &=1212,321,321\oplus 2110,211,211:4\\
            &=2211,321,312\oplus 1111,211,220:2\\
            &=2212,421,322\oplus 1110,111,210:4\\
            &=2222,431,422\oplus 1100,101,110:2\\
            &=2312,422,422\oplus 1010,110,110:4\\
            &=2322,522,432\oplus 1000,010,100:4.
\end{align*}
They are all the direct decompositions of the tuple
$3322,532,532$ modulo obvious symmetries.
Here we indicate the number of the decompositions of 
the same type.

\end{exmp}
\begin{cor}
Let\/ $\mathbf m\in\mathcal P$ be realizable.
Put\/ $\mathbf m=\gcd(\mathbf m)\overline{\mathbf m}$.
Then\/ $\mathbf m$ has no direct decomposition\/ \eqref{eq:dsum} 
if and only if 
\begin{align}
&\ord\mathbf m=1\\
\intertext{or}
&\idx\mathbf m=0\text{ and basic}\\
\intertext{or}
\begin{split}
&\idx\mathbf m<0\text{ and\/ $\overline{\mathbf m}$ is basic and\/ 
$\mathbf m$ is not isomorphic to any one of tuples}\\
&\text{in\/ {\rm Example~\ref{ex:special}} with\/ $m>1$}.
\end{split}
\end{align}
\end{cor}
Moreover we have the following result.
\begin{prop}\label{prop:rigdecomp}
The direct decomposition\/ $\mathbf m=\mathbf m'\oplus\mathbf m''$
is called \textsl{rigid decomposition}
\index{tuple of partitions!rigid decomposition}
if\/ $\mathbf m$, $\mathbf m'$ and $\mathbf m''$ are rigid.
If\/ $\mathbf m\in\mathcal P$ is rigid
and\/ $\ord\mathbf m>1$, there exists a rigid decomposition.
\end{prop}
\begin{proof}
We may assume that $\mathbf m$ is monotone and
there exist a non-negative integer $p$ such that
$m_{j,2}\ne 0$ if and only if $0\le j<p+1$.
If $\ord\p\mathbf m=1$, then we may assume
$\mathbf m=(p-1)1,(p-1)1,\dots,(p-1)1\in\mathcal P^{(p)}_{p+1}$
and there exists a decomposition
\[
 (p-1)1,(p-1)1,\dots,(p-1)1=01,10,\dots,10\oplus
 (p-1)0,(p-2)1,\dots,(p-2)1.
\]
Suppose $\ord{\p \mathbf m}>1$.
Put 
$d=\idx(\mathbf m,\mathbf 1)=m_{0,1}+\cdots+m_{p,1}
-(p-1)\cdot\ord\mathbf m>0$.

The induction hypothesis assures the existence of a 
decomposition $\p \mathbf m=\bar{\mathbf  m}'\oplus\bar{\mathbf m}''$
such that $\bar{\mathbf m}'$ and $\bar{\mathbf m}''$ are rigid.
If $\p\bar{\mathbf m}'$ and $\p\bar{\mathbf m}''$ are well-defined,
we have the decomposition 
$\mathbf m=\p^2\mathbf m=\p\bar{\mathbf m}'\oplus \p\bar{\mathbf m}''$ and 
the proposition.

If $\ord \bar{\mathbf m}'>1$, $\p\bar{\mathbf m}'$ is well-defined.
Suppose
$\bar{\mathbf m}'
=\bigl(\delta_{\nu,\ell_j}\bigr)_{\substack{j=0,\dots,p\\ \nu=1,2,\dots}}$.
Then 
\begin{align*}
 \idx(\p \mathbf m,\mathbf 1)- 
 \idx(\p \mathbf m,\bar{\mathbf m}')
  &=\sum_{j=0}^p\bigl((m_{j,1}-d -(m_{j,\ell_j}-d\delta_{\ell_j,1})\bigr)\\
  &\ge-d\#\{j\,;\,\ell_j>1,\ 0\le j\le p\}.
\end{align*}
Since  $\idx(\p \mathbf m,\mathbf 1)=-d$  and
$\idx(\p \mathbf m,\bar{\mathbf m}')=1$,
we have $d\#\{j\,;\,\ell_j>1,\ 0\le j\le p\}\ge d+1$ and
therefore $\#\{j\,;\,\ell_j>1,\ 0\le j\le p\}\ge 2$.
Hence $\p\bar{\mathbf m}'$ is well-defined.
\end{proof}
\begin{rem}
The author's original construction of a differential operator with a given
rigid Riemann scheme doesn't use the middle convolutions and additions but 
uses Proposition~\ref{prop:rigdecomp}.
\end{rem}
\begin{exmp}
We give direct decompositions of a rigid tuple:
\begin{equation}
 \begin{split}
 721,3331,22222
  &=200,2000,20000\oplus521,1331,02222:15\\
  &=210,1110,11100\oplus511,2221,11122:10\\
  &=310,1111,11110\oplus411,2220,11112:\phantom{1}5
 \end{split}
\end{equation}

The following irreducibly realizable tuple 
has only two direct decompositions:
\begin{equation}
\begin{split}
 44,311111,311111
 &=20,200000,200000\oplus 24,111111,111111\\
 &=02,200000,200000\oplus 42,111111,111111
\end{split}
\end{equation}
But it cannot be a direct sum of two irreducibly
realizable tuples.
\end{exmp}
\index{direct decomposition!example}
\subsection{Reduction of reducibility}\label{sec:redred}
We give a necessary and sufficient condition so that
a Fuchsian differential equation is irreducible, which follows
from \cite{Kz} and \cite{DR,DR2}.  Note that a Fuchsian differential equation
is irreducible if and only if its monodromy is irreducible.

\begin{thm}\label{thm:irred}
Retain the notation in\/ {\rm\S\ref{sec:reddirect}}.
Suppose\/ $\mathbf m$ is monotone, realizable and
$\p_{max}\mathbf m$ is well-defined and
\begin{align}
   d &:= m_{0,1}+\cdots+m_{p,1} - (p-1)\ord\mathbf m\ge0.\label{eq:mc1}\\
\intertext{Put $P=P_{\mathbf m}$ {\rm(cf.~\eqref{eq:uinvPm})} and}
 \mu&:=\lambda_{0,1}+\lambda_{1,1}+\cdots+\lambda_{p,1}-1,\\
 Q&:=\p_{max}P,\\
 P^o&:=P|_{\lambda_{j,\nu}=\lambda_{j,\nu}^o,\ g_i=g_i^o},\quad
 Q^o:=Q|_{\lambda_{j,\nu}=\lambda_{j,\nu}^o,\ g_i=g_i^o}
\end{align}
with some complex numbers $\lambda_{j,\nu}^o$ and $g_i^o$ satisfying
the Fuchs relation $|\{\lambda^o_{\mathbf m}\}|=0$.

{\rm i)} The Riemann scheme $\{\tilde\lambda_{\tilde{\mathbf m}}\}$
of Q is given by
\begin{equation}
  \begin{cases}
   \tilde m_{j,\nu}=m_{j,\nu}-d\delta_{\nu,1},\\
   \tilde\lambda_{j,\nu}=
    \lambda_{j,\nu}+\bigl((-1)^{\delta_{j,0}}-\delta_{\nu,1}\bigr)\mu.
   \end{cases}
\label{eq:mc2}
\end{equation}

{\rm ii)} 
Assume that the equation $P^ou=0$ is irreducible.
If $d>0$, then $\mu\notin\mathbb Z$. 
If the parameters given by $\lambda_{j,\nu}^o$ and 
$g_i^o$ are locally non-degenerate, the equation $Q^ov=0$ is irreducible and the 
parameters are locally non-degenerate.

{\rm iii)}
Assume that the equation $Q^ov=0$ is irreducible and the parameters given by 
$\lambda_{j,\nu}^o$ and $g_i^o$ are locally non-degenerate. 
Then the equation $P^ov=0$ is irreducible if and only if
\begin{align}\label{eq:irrmax}
  \sum_{j=0}^p\lambda^o_{j,1+\delta_{j,j_o}(\nu_o-1)}\notin\mathbb Z
  \text{ \ for any \ }(j_o,\nu_o)\text{ \ satisfying \ }
  m_{j_o,\nu_o}> m_{j_o,1}-d.
\end{align}
If the equation $P^ov=0$ is irreducible, the parameters are locally non-degenerate.

{\rm iv)}
Put\/ $\mathbf m(k):=\p_{max}^k\mathbf m$ and $P(k)=\p_{max}^kP$.
Let $K$ be a non-negative integer such that
$\ord \mathbf m(0)>\ord\mathbf m(1)>\cdots>\ord\mathbf m(K)$ and\/
$\mathbf m(K)$ is fundamental.
The operator $P(k)$ is essentially the universal operator of type\/ 
$\mathbf m(k)$ but parametrized by\/ $\lambda_{j,\nu}$ and $g_i$.
Put $P(k)^o=P(k)|_{\lambda_{j,\nu}=\lambda_{j,\nu}^o}$.

If the equation\ $P^ou=0$ is irreducible and 
the parameters are locally non-degenerate, so are $P(k)^ou=0$ for $k=1,\dots,K$.

If the equation $P^ou=0$ is irreducible and locally non-degenerate, so is the 
equation $P(K)^ou=0$.

Suppose the equation $P(K)^ou=0$ is irreducible and locally non-degenerate,
which is always valid when $\mathbf m$ is rigid.
Then the equation $P^ou=0$ is irreducible if and only if the equation $P(k)^ou=0$
satisfy the condition \eqref{eq:irrmax} for $k=0,\dots,K-1$.
If the equation $P^ou=0$ is irreducible, it is locally non-degenerate.
\end{thm}
\begin{proof}
The claim i) follows from Theorem~\ref{thm:GRSmid} and the claims ii) and iii)
follow from Lemma~\ref{lem:irrred} and Corollary~\ref{cor:irredred}, which implies the claim iv).
\end{proof}

\begin{rem}\label{rem:irredrigid}
{\rm i)}
In the preceding theorem the equation $P^ou=0$ may not be locally
non-degenerate even if it is irreducible.
For example the equation satisfied by ${}_3F_2$ is contained in
the universal operator of type $111,111,111$.

{\rm ii)} It is also proved as follows that the irreducible 
differential equation with a rigid spectral type is locally non-degenerate.

The monodromy generators $M_j$ of the equation with the 
Riemann scheme 
at $x=c_j$ satisfy
\[
  \rank(M'_j-e^{2\pi\sqrt{-1}\lambda_{j,1}})
  \cdots(M'_j-e^{2\pi\sqrt{-1}\lambda_{j,k}})
  \le m_{j,k+1}+\cdots+m_{j,n_j}\quad(k=1,\dots,n_j)
\]
for $j=0,\dots,p$.
The equality in the above is clear 
when $\lambda_{j,\nu}-\lambda_{j,\nu'}\notin\mathbb Z$ for $1\le\nu<\nu'\le n_j$
and hence the above is proved by the continuity for general $\lambda_{j,\nu}$.
The rigidity index of $\mathbf M$ is calculated by the dimension of the 
centralizer of $M_j$ and it should be 2 if $\mathbf M$ is irreducible and rigid,
the equality in the above is valid (cf.~\cite{Kz}, \cite{O3}), which means
the equation is locally non-degenerate.

{\rm iii)} \ 
The same results as in Theorem~\ref{thm:irred} are also valid in the case of the
Fuchsian system of Schlesinger canonical form \eqref{eq:MSCF} since 
the same proof works.
A similar result is given by a different proof (cf.~\cite{CB}).

{\rm iv)} \ 
Let $(M_0,\dots,M_p)$ be a tuple of matrices in $GL(n,\mathbb C)$
with $M_pM_{p-1}\cdots M_0=I_n$.
Then $(M_0,\dots,M_p)$ is called \textsl{rigid} if for any 
$g_0,\dots,g_p\in GL(n,\mathbb C)$
satisfying $g_pM_pg_p^{-1}\cdot g_{p-1}M_{p-1}g_{p-1}^{-1}\cdots 
g_0M_0g_0^{-1}=I_n$, 
there exists $g\in GL(n,\mathbb C)$ such that $g_iM_ig_i^{-1}=gM_ig^{-1}$ for 
$i=0,\dots,p$.
The tuple $(M_0,\dots,M_p)$ is called \textsl{irreducible} if no subspace $V$ of 
$\mathbb C^n$ 
satisfies $\{0\}\subsetneqq V \subsetneqq \mathbb C^n$ and $M_iV\subset V$ for 
$i=0,\dots,p$.
Choose $\mathbf m\in\mathcal P^{(n)}_{p+1}$ and $\{\mu_{j,\nu}\}$ such that 
$L(\mathbf m;\mu_{j,1},\dots,\mu_{j,n_j})$ are in the conjugacy classes 
containing $M_j$, respectively.
Suppose $(M_0,\dots,M_p)$ is irreducible and rigid.
Then Katz \cite{Kz} shows that $\mathbf m$ is rigid and gives a construction of 
irreducible and rigid $(M_0,\dots,M_p)$ for any rigid $\mathbf m$
(cf.~Remark \ref{rem:monred} ii)).
It is an open problem given by Katz \cite{Kz} whether the monodromy generators 
$M_j$ are realized by solutions of a single Fuchsian differential equations without
an apparent singularity, whose affirmative answer is given by the following
corollary.
\end{rem}
\begin{cor}\label{cor:irred}
Let\/ $\mathbf m=\bigl(m_{j,\nu}\bigr)_{\substack{0\le j\le p\\1\le\nu\le n_j}}$ 
be a 
rigid monotone $(p+1)$-tuple of partitions with\/ 
$\ord\mathbf m>1$.
Retain the notation in Definition~\ref{def:redGRS}.

{\rm i)} 
Fix complex numbers $\lambda_{j,\nu}$ for $0\le j\le p$ and $1\le \nu_j$
such that it satisfies the Fuchs relation
\begin{equation}\label{eq:Fuchsirr}
 \sum_{j=0}^p\sum_{\nu=1}^{n_j}m_{j,\nu}\lambda_{j,\nu}=\ord\mathbf m-1
\end{equation}
The universal operator $P_{\mathbf m}(\lambda)u=0$ with the Riemann scheme
\eqref{eq:IGRS} is irreducible if and only if 
the condition
\begin{multline}\label{Cond:irr}
  \sum_{j=0}^p\lambda(k)_{j,\ell(k)_j+\delta_{j,j_o}(\nu_o-\ell(k)_j)}
  \notin\mathbb Z
  \\\text{ \ for any \ }(j_o,\nu_o)\text{ \ satisfying \ }
  m(k)_{j_o,\nu_o}> m(k)_{j_o,\ell(k)_{j_o}}-d(k)
\end{multline}
is satisfied for $k=0,\dots,K-1$.

{\rm ii)}  Define $\tilde \mu(k)$ and $\mu(k)_{j,\nu}$ for $k=0,\dots,K$ by
\begin{align}
 \mu(0)_{j,\nu}&=\mu_{j,\nu}\quad(j=0,\dots,p,\ \nu=1,\dots,n_j),\\
 \tilde\mu(k) &= \prod_{j=0}^p\mu(k)_{j,\ell(k)_j},\\
 \mu(k+1)_{j,\nu}
 &=\mu(k)_{j,\nu}\cdot
       \tilde\mu(k)^{(-1)^{\delta_{j,0}}-\delta_{\nu,1}}.
\end{align}
Then there exists an irreducible tuple $(M_0,\dots,M_p)$ of matrices satisfying
\begin{equation}\label{eq:monDS}
\begin{gathered}
 M_p\cdots M_0=I_n,\\
 M_j\sim L(m_{j,1},\dots,m_{j,n_j};\mu_{j,1},\dots,\mu_{j,n_j})\quad(j=0,\dots,p)
\end{gathered}
\end{equation}
under the notation \eqref{eq:OSNF} if and only if
\begin{equation}
\prod_{j=0}^p\prod_{\nu=1}^{n_j}\mu_{j,\nu}^{m_{j,\nu}}=1
\end{equation}
and the condition
\begin{multline}
  \prod_{j=0}^p\mu(k)_{j,\ell(k)_j+\delta_{j,j_o}(\nu_o-\ell(k)_j)}
  \ne1
  \\\text{ \ for any \ }(j_o,\nu_o)\text{ \ satisfying \ }
  m(k)_{j_o,\nu_o}> m(k)_{j_o,\ell(k)_{j_o}}-d(k)
\end{multline}
is satisfied for $k=0,\dots,K-1$.

{\rm iii)}  Let $(M_0,\dots,M_p)$ be an irreducible tuple of matrices satisfying 
\eqref{eq:monDS}.
Then there uniquely exists a Fuchsian differential equation $Pu=0$ with $p+1$ 
singular points $c_0,\dots,c_p$ and its local independent solutions 
$u_1,\dots,u_{\ord \mathbf m}$ in a neighborhood of a non-singular point $q$ 
such that the monodromy generators around the points $c_j$ with respect to the 
solutions equal $M_j$, respectively, for $j=0,\dots,p$ {\rm (cf.~\eqref{fig:mon})}.
\end{cor}
\begin{proof}
The clam i) is a direct consequence of Theorem~\ref{thm:irred}
and the claim ii) is proved by Theorem~\ref{thm:Mmid} and 
Lemma~\ref{lem:Scott}
as in the case of the proof of Theorem~\ref{thm:irred}
(cf.~Remark~\ref{rem:monred} ii)).

iii)
Since $\gcd\mathbf m=1$,
we can choose $\lambda_{j,\nu}\in\mathbb C$ such that
$e^{2\pi\sqrt{-1}\lambda_{j,\nu}}=\mu_{j,\nu}$ and 
$\sum_{j,\nu}m_{j,\nu}\lambda_{j,\nu}=\ord\mathbf m-1$.
Then we have a universal operator $P_{\mathbf m}(\lambda_{j,\nu})u=0$
with the Riemann scheme \eqref{eq:IGRS}.
The irreducibility of $(M_p,\dots,M_0)$ and Theorem~\ref{thm:mcMC}
assure the claim.
\end{proof}

Now we state the condition \eqref{Cond:irr} using the terminology of 
the Kac-Moody root system.
Suppose $\mathbf m\in\mathcal P$ is monotone and irreducibly realizable.
Let $\{\lambda_{\mathbf m}\}$ be the Riemann scheme of the universal
operator $P_{\mathbf m}$.
According to Remark~\ref{rem:idxFuchs} iii) we may relax the definition
of $\ell_{max}(\mathbf m)$ as is given by \eqref{eq:lmaxe} and then 
we may assume
\begin{equation}
  v_k s_0\cdots v_1s_0\Lambda(\lambda)
 \in W'_{\infty}\Lambda\bigl(\lambda(k)\bigr)\qquad(k=1,\dots,K)
\end{equation}
under the notation in Definition~\ref{def:redGRS} and \eqref{eq:vmrp}.
Then we have the following theorem.
\begin{thm}\label{thm:irrKac} 
Let\/ $\mathbf m=\bigl(m_{j,\nu}\bigr)_{\substack{0\le j\le p\\1\le \nu\le n_j}}$
be an irreducibly realizable monotone tuple of partition in $\mathcal P$.
Under the notation in\/ {\rm Corollary~\ref{cor:irred}} and\/ {\rm \S\ref{sec:KM}},
there uniquely exists a bijection
\index{000pix@$\varpi$}
\begin{equation}
 \begin{split}
  \varpi:\Delta(\mathbf m) \xrightarrow\sim\ &
  \bigl\{(k,j_0,\nu_0)\,;\,
  0\le k<K,\ 0\le j_0\le p,\ 1\le\nu_0\le n_{j_0},
  \\
  &\qquad \nu_0\ne\ell(k)_{j_0}\text{ \ and \ }
  m(k)_{j_0,\nu_0}>m(k)_{j_0,\ell(k)_{j_0}}-d(k)
  \bigr\}\\
  &\quad\cup\bigl\{(k,0,\ell(k)_0)\,;\,0\le k<K\bigr\}
 \end{split}
\end{equation}
such that
\begin{equation}
 (\Lambda(\lambda)|\alpha)
 =\sum_{j=0}^p\lambda(k)_{j,\ell(k)_j+\delta_{j,j_o}(\nu_o-\ell(k)_j)}
 \text{ \ when \ }\varpi(\alpha)=(k,j_0,\nu_0).
\end{equation}
Moreover we have
\begin{equation}
 \begin{split}
  (\alpha|\alpha_{\mathbf m})&=
  m(k)_{j_0,\nu_0}-m(k)_{j_0,\ell(k)_{j_0}}+d(k)\\
  &\qquad\qquad
  (\alpha\in\Delta(\mathbf m),\ (k,j_0,\nu_0)=\varpi(\alpha))
 \end{split}
\end{equation}
and if the universal equation $P_{\mathbf m}(\lambda)u=0$ is irreducible,
we have
\begin{equation}\label{eq:airred}
    (\Lambda(\lambda)|\alpha)\notin\mathbb Z
  \quad\text{ for any \ }\alpha\in\Delta(\mathbf m).
\end{equation}
In particular, if $\mathbf m$ is rigid and \eqref{eq:airred} is valid,
the universal equation 
is irreducible.
\end{thm}
\begin{proof}
Assume $\ord \mathbf m>1$ and use the notation in Theorem~\ref{thm:irred}.
Since $\tilde{\mathbf m}$ may not be monotone, we consider the monotone
tuple $\mathbf m'=s\tilde{\mathbf m}$ in $S'_\infty\tilde{\mathbf m}$
(cf.~Definition~\ref{def:Sinfty}). 
First note that
\[
  d-m_{j,1}+m_{j,\nu} 
  = (\alpha_0+\alpha_{j,1}+\cdots+\alpha_{j,\nu-1}|\alpha_{\mathbf m}).
\]
Let $\bar \nu_j$ be the positive integers defined by
\[
    m_{j,\bar\nu_j+1}\le m_{j,1}-d<m_{j,\bar\nu_j}
\]
for $j=0,\dots,p$.
Then
\[
  \alpha_{\mathbf m'}=v^{-1}\alpha_{\tilde{\mathbf m}}
  \text{ \ with \ }
  v:=\Bigl(\prod_{j=0}^p s_{j,1}\cdots s_{j,\bar\nu_j-1}
  \Bigr)
\]
and $w(\mathbf m)=s_0vs_{\alpha_{\tilde{\mathbf m}}}$ and
\[
\begin{split}
 \Delta(\mathbf m)
 &= \Xi\cup s_0v\Delta(\mathbf m'),\\
 \Xi&:=\{\alpha_0\}\cup
   \bigcup_{\substack{0\le j\le p\\\nu_j\ne 1}}
  \{\alpha_0+\alpha_{j,1}+\cdots+\alpha_{j,\nu}\,;\,
    \nu=1,\ldots,\bar\nu_j-1\}.
\end{split}
\]

Note that $\ell(0)=(1,\dots,1)$ and the condition
$m_{j_0,\nu_0}>m_{j_0,1}-d(0)$ is valid if and only if 
$\nu_0\in\{1,\dots,\bar\nu_{j_0}\}$.
Since
\[
  \sum_{j=0}^p\lambda(0)_{j,1+\delta_{j,j_0}(\nu_0-1)}
 =(\Lambda(\lambda)|\alpha_0+\alpha_{j_0,1}+\cdots+\alpha_{j_0,\nu_0-1})+1,
\]
we have
\[
 L(0)=\bigl\{(\Lambda(\lambda)|\alpha)+1\,;\,\alpha\in\Xi\bigr\}
\]
by denoting
\[
  L(k):=\bigl\{\sum_{j=0}^p\lambda(k)_{j,\ell(k)_j+\delta_{j,j_o}(\nu_o-\ell(k)_j)}\,;\,
  m(k)_{j_o,\nu_o}> m(k)_{j_o,\ell(k)_{j_o}}-d(k)\bigr\}.
\]
Applying $v^{-1}s_0$ to $\mathbf m$ and $\{\lambda_{\mathbf m}\}$,
they changes into $\mathbf m'$ and $\{\lambda'_{\mathbf m'}\}$, respectively,
such that $\Lambda(\lambda') - v^{-1}s_0\Lambda(\lambda)\in\mathbb C\Lambda_0$.
Hence we obtain the corollary by the induction as in the proof of 
Corollary~\ref{cor:irred}.
\end{proof}
\begin{rem} 
Let $\mathbf m$ be an irreducibly realizable monotone tuple in $\mathcal P$.
Fix $\alpha\in\Delta(\mathbf m)$.
We have $\alpha=\alpha_{\mathbf m'}$ with a rigid tuple $\mathbf m'\in\mathcal P$
and
\begin{equation}
 |\{\lambda_{\mathbf m'}\}|=(\Lambda(\lambda)|\alpha).
\end{equation}
\end{rem}
\index{00indexa@$\idx_{\mathbf m}$}
\begin{defn}\label{def:linidx}
Define an \textsl{index}
$\idx_{\mathbf m}\bigl(\ell(\lambda)\bigr)$ 
of 
the non-zero linear form $\ell(\lambda)
=\sum_{j=0}^p\sum_{\nu=1}^{n_j}k_{j,\nu}\lambda_{j,\nu}$
of with $k_{j,\nu}\in\mathbb Z_{\ge 0}$ 
as the positive integer $d_i$ such that
\begin{equation}
 \Bigl\{\sum_{j=0}^p\sum_{\nu=1}^{n_j} k_{j,\nu}\epsilon_{j,\nu}\,;\,
 \epsilon_{j,\nu}\in\mathbb Z\text{ and }\sum_{j=0}^p\sum_{\nu=1}^{n_j}
 m_{j,\nu}\epsilon_{j,\nu}=0\Bigr\} = \mathbb Z d_i.
\end{equation}
\end{defn}
\begin{prop}\label{prop:subrep}
For a rigid tuple $\mathbf m$ in\/ {\rm Corollary~\ref{cor:irred},} define 
rigid tuples $\mathbf m^{(1)},\dots,\mathbf m^{(N)}$ with a non-negative integer 
$N$ so that
$\Delta(\mathbf m)=\{\mathbf m^{(1)},\dots,\mathbf m^{(N)}\}$
and put
\begin{equation}
  \ell_i(\lambda):=\sum_{j=0}^p\sum_{\nu=1}^{n_j} 
  m^{(i)}_{j,\nu}\lambda_{j,\nu}
  \qquad(i=1,\dots,N).
\end{equation}
Here we note that\/ {\rm Theorem~\ref{thm:irrKac}} implies that
$P_{\mathbf m}(\lambda)$ is irreducible if and only if
$\ell_i(\lambda)\notin\mathbb Z$ for $i=1,\dots,n$.

Fix a function $\ell(\lambda)$ of $\lambda_{j,\nu}$ such that 
$\ell(\lambda)=\ell_i(\lambda)-r$ with $i\in\{1,\dots,N\}$ and $r\in\mathbb Z$.
Moreover fix generic complex numbers $\lambda_{j,\nu}\in\mathbb C$ under 
the condition $\ell(\lambda)=|\{\lambda_{\mathbf m}\}|=0$
and a decomposition $P_{\mathbf m}(\lambda)=P''P'$ such that 
$P',\,P''\in W(x)$,  $0<n':=\ord P'<n$ and the differential equation $P'v=0$
is irreducible.
Then there exists an irreducibly realizable
subtuple\/ $\mathbf m'$ of\/ $\mathbf m$ compatible to 
$\ell(\lambda)$ such that 
the monodromy generators $M'_j$ of the equation $P'u=0$
satisfies
\[
  \rank(M_j-e^{2\pi\sqrt{-1}\lambda_{j,1}})
  \cdots(M_j-e^{2\pi\sqrt{-1}\lambda_{j,k}})
  \le m'_{j,k+1}+\cdots+m'_{j,n_j}\quad(k=1,\dots,n_j)
\]
for $j=0,\dots,p$.
Here we define that the decomposition
\begin{equation}\label{eq:subrep}
   \mathbf m=\mathbf m'+\mathbf m''\quad( \mathbf m'\in\mathcal P^{(n')}_{p+1},\ 
   \mathbf m''\in\mathcal P^{(n'')}_{p+1},\ 0<n'<n)
\end{equation}
is \textsl{compatible} to $\ell(\lambda)$ and that $\mathbf m'$ is 
a subtuple of $\mathbf m$
\textsl{compatible} to $\ell(\lambda)$ if the following conditions are valid
\begin{align}
&|\{\lambda_{\mathbf m'}\}|\in\mathbb Z_{\le 0} \text{ \ and \ }
|\{\lambda_{\mathbf m''}\}|\in\mathbb Z,\\
&\text{$\mathbf m'$ is realizable if there exists $(j,\nu)$
such that $m''_{j,\nu}=m_{j,\nu}>0$,}\\
&\text{$\mathbf m''$ is realizable if there exists $(j,\nu)$
such that $m'_{j,\nu}=m_{j,\nu}>0$}.
\end{align}
Here we note $|\{\lambda_{\mathbf m'}\}|+|\{\lambda_{\mathbf m''}\}|=1$
if\/ $\mathbf m'$ and\/ $\mathbf m''$ are rigid.
\end{prop}
\begin{proof}
The equation $P_{\mathbf m}(\lambda)u=0$ is reducible since $\ell(\lambda)=0$.
We may assume $\lambda_{j,\nu}-\lambda_{j,\nu'}\ne 0$
for $1\le \nu<\nu'\le n_j$ and $j=0,\dots,p$.
The solutions of the equation define the map $\mathcal F$ given by \eqref{eq:Fmap}
and the reducibility implies the existence of an irreducible 
submap $\mathcal F'$ such that $\mathcal F'(U)\subset \mathcal F(U)$ and 
$0<n':=\dim\mathcal F'(U)<n$.
Then $\mathcal F'$ defines a irreducible Fuchsian differential equation 
$P'v=0$ which has regular singularities at $x=c_0=\infty,c_1,\dots,c_p$ 
and may have other apparent singularities $c'_1,\dots,c'_q$.
Then the characteristic exponents of $P'$ at the singular points are as
follows.

There exists a decomposition $\mathbf m=\mathbf m'+\mathbf m''$
such that $\mathbf m'\in\mathcal P^{(n')}$ and 
$\mathbf m''\in\mathcal P^{(n'')}$ with $n'':=n-n'$.
The sets of characteristic exponents of $P'$ at $x=c_j$ are
$\{\lambda'_{j,\nu,i}\,;\,i=1,\dots,m'_{j,\nu},\ 
\nu=1,\dots,n\}$ which satisfy
\[
  \lambda'_{j,\nu,i}-\lambda_{j,\nu}\in\{0,1,\dots,m_{j,\nu}-1\}\text{ \ and \ }
  \lambda'_{j,\nu,1}<\lambda'_{j,\nu,2}<\cdots<\lambda'_{j,\nu,m'_{j,\nu}}
\]
for $j=0,\dots,p$.  
The sets of characteristic exponents at $x=c'_j$
are $\{\mu_{j,1},\dots,\mu_{j,n'}\}$, which satisfy
$\mu_{j,i}\in\mathbb Z$ and $0\le \mu_{j,1}<\cdots<\mu_{j,n'}$
for $j=1,\dots,q$. Then Remark~\ref{rem:generic} ii) says that
the Fuchs relation of the equation $P'v=0$ implies
$|\{\lambda_{\mathbf m'}\}|\in\mathbb Z_{\le 0}$.

Note that there exists a Fuchsian differential operator $P''\in W(x)$ such 
that $P=P''P'$.
If there exists $j_o$ and $\nu_o$ such that $m'_{j_o,n_o}=0$, namely,
$m''_{j_o,\nu_o}=m_{j_o,\nu_o}>0$, the exponents of the monodromy generators
of the solution $P'v=0$ are generic and hence $\mathbf m'$ should be 
realizable.
The same claim is also true for the tuple $\mathbf m''$.
Hence we have the proposition.
\end{proof}
\begin{exmp}\label{exmp:irred}
{\rm i)}\index{hypergeometric equation/function!Gauss!reducibility}
The reduction of the universal operator with the spectral type
$11,11,11$ which is given by Theorem~\ref{thm:irred} is
\begin{equation}\begin{split}
 &\begin{Bmatrix}
  x=\infty & 0 & 1\\
  \lambda_{0,1} & \lambda_{1,1} & \lambda_{2,1}\\
  \lambda_{0,2} & \lambda_{1,2} & \lambda_{2,2}
 \end{Bmatrix}\quad(\sum \lambda_{j,\nu}=1)\\
 &\longrightarrow
 \begin{Bmatrix}
  x=\infty & 0 & 1\\
  2\lambda_{0,2}+\lambda_{1,1}+\lambda_{2,1} 
  & -\lambda_{0,2}-\lambda_{2,2} & -\lambda_{0,2}-\lambda_{1,2}
 \end{Bmatrix}
\end{split}\label{eq:redGG}
\end{equation}
because $\mu=\lambda_{0,1} + \lambda_{1,1} + \lambda_{2,1} -1
 =-\lambda_{0,2} - \lambda_{1,2} - \lambda_{2,2}$.
Hence the necessary and sufficient condition for the irreducibility of 
the universal operator given by \eqref{eq:irrmax} is
\begin{equation*}
 \begin{cases}
 \lambda_{0,1}+\lambda_{1,1}+\lambda_{2,1}\notin\mathbb Z,\\
 \lambda_{0,2}+\lambda_{1,1}+\lambda_{2,1}\notin\mathbb Z,\\
 \lambda_{0,1}+\lambda_{1,2}+\lambda_{2,1}\notin\mathbb Z,\\
 \lambda_{0,1}+\lambda_{1,1}+\lambda_{2,2}\notin\mathbb Z,
 \end{cases}
\end{equation*}
which is equivalent to
\begin{equation}
 \lambda_{0,i}+\lambda_{1,1}+\lambda_{2,j}\notin\mathbb Z
 \quad\text{for \ }i=1,2 \text{ and }j=1,2.
\end{equation}
The rigid tuple $\mathbf m=11,11,11$ corresponds to the real root
$\alpha_{\mathbf m}=
2\alpha_0+\alpha_{0,1}+\alpha_{1,1}+\alpha_{2,1}$ under the notation
in \S\ref{sec:KM}. 
Then $\Delta(\mathbf m)
 =\{\alpha_0,\alpha_0+\alpha_{j,1}\,;\,j=0,1,2\}$
and $(\Lambda|\alpha_0)=\lambda_{0,1}+\lambda_{1,1}+\lambda_{2,1}$
and $(\Lambda|\alpha_0+\alpha_{0,1})=\lambda_{0,2}+\lambda_{1,1}+\lambda_{2,1}$, 
etc. under the notation in Theorem~\ref{thm:irrKac}.

The Riemann scheme for the Gauss hypergeometric series 
${}_2F_1(a,b,c;z)$ is given by
$\begin{Bmatrix}
  x=\infty & 0 & 1\\
  a & 0 & 0\\
  b & 1-c & c-a-b
\end{Bmatrix}$
and therefore the condition for the irreducibility is
\begin{equation}
 a\notin\mathbb Z,\ b\notin\mathbb Z,\ c-b\notin\mathbb Z\text{ \ and \ } c-a\notin\mathbb Z.
\end{equation}

{\rm ii)}  The reduction of the Riemann scheme for the equation
corresponding to ${}_3F_2(\alpha_1,\alpha_2,\alpha_3,\beta_1,\beta_2;x)$ is
\begin{equation}\begin{split}
 &\begin{Bmatrix}
  x = \infty & 0 & 1\\
  \alpha_1 & 0 & [0]_{(2)} \\
  \alpha_2 & 1-\beta_1 & -\beta_3\\
  \alpha_3 & 1-\beta_2 
 \end{Bmatrix}\qquad(\sum_{i=1}^3\alpha_i=\sum_{i=1}^3\beta_i)\\
&\longrightarrow
 \begin{Bmatrix}
  x = \infty & 0 & 1\\
  \alpha_2-\alpha_1+1 & \alpha_1-\beta_1 & 0\\
  \alpha_3-\alpha_1+1 & \alpha_1-\beta_2 & \alpha_1-\beta_3-1
 \end{Bmatrix}
\end{split}\end{equation}
with $\mu=\alpha_1-1$.
Hence Theorem~\ref{thm:irred} says that the condition for the irreducibility equals
\begin{align*}
 &\begin{cases}
 \alpha_i\notin\mathbb Z&(i=1,2,3),\\
 \alpha_1-\beta_j\notin\mathbb Z&(j=1,2)
 \end{cases}
\intertext{together with}
 &\alpha_i-\beta_j\notin\mathbb Z\qquad(i=2,3,\ j=1,2).
\end{align*}
Here the second condition follows from {\rm i)}.
Hence the condition for the irreducibility is
\begin{equation}
 \alpha_i\notin\mathbb Z\text{ \ and \ }
 \alpha_i-\beta_j\notin\mathbb Z\quad(i=1,2,3,\ j=1,2).
\end{equation}

{\rm iii)}  The reduction of the even family is as follows:
\begin{align*}
 \begin{Bmatrix}
  x = \infty & 0 & 1\\
  \alpha_1 & [0]_{(2)} & [0]_{(2)} \\
  \alpha_2 & 1-\beta_1 & [-\beta_3]_{(2)}\\
  \alpha_3 & 1-\beta_2\\
  \alpha_4
 \end{Bmatrix}
 \longrightarrow&
 \begin{Bmatrix}
   x = \infty & 0 & 1\\
  \alpha_2-\alpha_1+1 & 0 &  0 \\
  \alpha_3-\alpha_1+1 & \alpha_1-\beta_1 & [\alpha_1-\beta_3-1]_{(2)}\\
  \alpha_4-\alpha_1+1 & \alpha_1-\beta_2
 \end{Bmatrix}\\
 \xrightarrow{(x-1)^{-\alpha_1+\beta_3+1}}&
  \begin{Bmatrix}
   x = \infty & 0 & 1\\
  \alpha_2-\beta_3 & 0 &  -\alpha_1+\beta_3+1 \\
  \alpha_3-\beta_3 & \alpha_1-\beta_1 & [0]_{(2)}\\
  \alpha_4-\beta_3 & \alpha_1-\beta_2
 \end{Bmatrix}.
\end{align*}
Hence the condition for the irreducibility is
\begin{align*}
 &\begin{cases}
   \alpha_i\notin\mathbb Z&(i=1,2,3,4),\\
   \alpha_1-\beta_3\notin\mathbb Z
  \end{cases}
\intertext{together with}
  &\begin{cases}
  \alpha_i-\beta_3\notin\mathbb Z&(i=2,3,4).\\
  \alpha_1+\alpha_i-\beta_j-\beta_3\notin\mathbb Z&(i=2,3,4,\ j=1,2)
  \end{cases}
\end{align*}
by the result in ii).
Thus the condition is
\begin{equation}
 \begin{split}
 &\alpha_i\notin\mathbb Z,\ \alpha_i-\beta_3\notin\mathbb Z\text{ and } 
 \alpha_1+\alpha_k-\beta_j-\beta_3\notin\mathbb Z\\
 &\qquad\qquad(i=1,2,3,4,\ j=1,2,\ k=2,3,4).
 \end{split}
\end{equation}
Hence the condition for the irreducibility
for the equation with the Riemann scheme
\begin{equation}
 \begin{Bmatrix}
  \lambda_{0,1} & [\lambda_{1,1}]_{(2)} & [\lambda_{2,1}]_{(2)}\\
  \lambda_{0,2} & \lambda_{1,2} & [\lambda_{2,2}]_{(2)}\\
  \lambda_{0,3} & \lambda_{1,3}\\
  \lambda_{0,4}
 \end{Bmatrix}
\end{equation}
 of type $1111,211,22$ is
\begin{equation}\label{eq:e4irred}
 \begin{cases}
 \lambda_{0,\nu} + \lambda_{1,1}+\lambda_{2,k}\notin\mathbb Z
  &(\nu=1,2,3,4,\ k=1,2)\\
 \lambda_{0,\nu}+\lambda_{0,\nu'} + \lambda_{1,1}+ \lambda_{1,2} + \lambda_{2,1}
 +\lambda_{2,2}\notin\mathbb Z&(1\le\nu<\nu'\le 4).
 \end{cases}
\end{equation}
This condition corresponds to the rigid decompositions
\begin{equation}\label{eq:e4ddec}
 1^4,21^2,2^2 = 1,10,1\oplus1^3,11^2,21 = 1^2,11,1^2\oplus1^2,11,1^2,
\end{equation}
which are also important in the connection formula.

{\rm iv)} (generalized Jordan-Pochhammer)\index{Jordan-Pochhammer!generalized}
The reduction of the universal operator of the rigid spectral type
$32,32,32,32$ is as follows:
\begin{align*}
 &\begin{Bmatrix}
  [\lambda_{0,1}]_{(3)} & [\lambda_{1,1}]_{(3)}
  &[\lambda_{2,1}]_{(3)} & [\lambda_{3,1}]_{(3)}\\
  [\lambda_{0,2}]_{(2)} & [\lambda_{1,2}]_{(2)}  & [\lambda_{2,2}]_{(2)} 
  & [\lambda_{3,2}]_{(2)}
 \end{Bmatrix}
 \quad(3\sum_{j=0}^3\lambda_{j,1}+2\sum_{j=0}^3\lambda_{j,2}=4)\\
 &\longrightarrow
 \begin{Bmatrix}
  \lambda_{0,1}-2\mu & \lambda_{1,1} &\lambda_{2,1} & \lambda_{3,1}\\
  [\lambda_{0,2}-\mu]_{(2)}& [\lambda_{1,2}+\mu]_{(2)} & [\lambda_{2,2}+\mu]_{(2)} 
  & [\lambda_{3,2}+\mu]_{(2)}
 \end{Bmatrix}
\end{align*}
with $\mu=\lambda_{0,1}+\lambda_{1,1}+\lambda_{2,1}+\lambda_{3,1}-1$.
Hence the condition for the irreducibility is
\begin{equation}\label{eq:red32}
 \begin{cases}
  \sum_{j=0}^3\lambda_{j,1+\delta_{j,k}}\notin\mathbb Z&(k=0,1,2,3,4),\\
  \sum_{j=0}^3(1+\delta_{j,k})\lambda_{j,1}
  +\sum_{j=0}^3(1-\delta_{j,k})\lambda_{j,2}\notin\mathbb Z  &(k=0,1,2,3,4).
 \end{cases}
\end{equation}
Note that under the notation defined by Definition~\ref{def:linidx} we have
\begin{equation}\label{eq:32idx}
 \idx_{\mathbf m}
 \bigl(\lambda_{0,1}+\lambda_{1,1}+\lambda_{2,1}+\lambda_{3,1}\bigr)=2
\end{equation}
and the index of any other linear form in \eqref{eq:red32} is 1.

In general, the universal operator with the Riemann scheme
\begin{align}
 \begin{split}
  &\begin{Bmatrix}
  [\lambda_{0,1}]_{(k)} & [\lambda_{1,1}]_{(k)}
  &[\lambda_{2,1}]_{(k)} & [\lambda_{3,1}]_{(k)}\\
  [\lambda_{0,2}]_{(k-1)} & [\lambda_{1,2}]_{(k-1)}  & [\lambda_{2,2}]_{(k-1)} 
  & [\lambda_{3,2}]_{(k-1)}
 \end{Bmatrix}\\
 &\qquad(k\sum_{j=0}^3\lambda_{j,1}+(k-1)\sum_{j=0}^3\lambda_{j,2}=2k)
\end{split}
\end{align}
is irreducible if and only if
\begin{equation}\label{eq:redgen}
 \begin{cases}
  \sum_{j=0}^3(\nu-\delta_{j,k})\lambda_{j,1}
  +\sum_{j=0}^3(\nu-1+\delta_{j,k})\lambda_{j,1} \notin\mathbb Z&(k=0,1,2,3,4),\\
  \sum_{j=0}^3(\nu'+\delta_{j,k})\lambda_{j,1}
  +\sum_{j=0}^3(\nu'-\delta_{j,k})\lambda_{j,2}\notin\mathbb Z  &(k=0,1,2,3,4),
 \end{cases}
\end{equation}
for any integers $\nu$ and $\nu'$ satisfying 
$1\le 2\nu\le k$ and $1\le 2\nu'\le k-1$.

The rigid decomposition
\begin{equation}
  65,65,65,65=12,21,21,21\oplus 53,44,44,44
\end{equation}
gives an example of the decomposition $\mathbf m=\mathbf m'\oplus\mathbf m''$
with $\supp\alpha_{\mathbf m}=\supp\alpha_{\mathbf m'}=\supp\alpha_{\mathbf m''}$.

{\rm v)} The rigid Fuchsian differential equation with the Riemann scheme
\index{tuple of partitions!rigid!831,93,93,93}
\[
 \begin{Bmatrix}
 x=0 & 1& c_3& c_4&\infty\\
[0]_{(9)} & [0]_{(9)} & [0]_{(9)} & [0]_{(9)} & [e_0]_{(8)}\\
[a]_{(3)} & [b]_{(3)} & [c]_{(3)} & [d]_{(3)} & [e_1]_{(3)}\\
 &  &  &  & e_2
\end{Bmatrix}
\]
is reducible when
\[
 a+b+c+d+3e_0+e_1\in\mathbb Z,
\]
which is equivalent to $\frac13(e_0-e_2-1)\in\mathbb Z$
under the Fuchs relation.
At the generic point of this reducible condition, 
the spectral types of the decomposition
in the Grothendieck group of the monodromy is
\[
  93,93,93,93,831=31,31,31,31,211+31,31,31,31,310+31,31,31,31,310.
\]
Note that the following reduction of the spectral types
\[
\begin{matrix}
 93,93,93,93,831&\to&13,13,13,13,031&\to&10,10,10,10,001\\ 
 31,31,31,31,211&\to&11,11,11,11,011\\
 31,31,31,31,310&\to&01,01,01,01,010
\end{matrix}
\]
and $\idx(31,31,31,31,211)=-2$.
\end{exmp}
\section{Shift operators}\label{sec:shift}
In this section we study an integer shift of spectral parameters $\lambda_{j,\nu}$ 
of the Fuchsian equation $P_{\mathbf m}(\lambda)u=0$.
Here $P_{\mathbf m}(\lambda)$  is the universal operator
(cf.~Theorem~\ref{thm:univmodel})
corresponding to the spectral type $\mathbf m
=(m_{j,\nu}\bigr)_{\substack{j=0,\dots,p\\\nu=1,\dots,n_j}}$.
For simplicity, we assume that $\mathbf m$ is rigid in this section
unless otherwise stated.
\subsection{Construction of shift operators and recurrence relations}
\label{sec:shift1}
First we construct shift operators for general shifts.
\begin{defn}
For $\mathbf m=\bigl(m_{j,\nu}\bigr)_{\substack{j=0,\dots,p\\\nu=1,\dots,n_j}}
\in\mathcal P^{(n)}_{p+1}$, a set of integers 
$\bigl(\epsilon_{j,\nu}\bigr)_{\substack{j=0,\dots,p\\\nu=1,\dots,n_j}}$
parametrized by $j$ and $\nu$ is called a \textsl{shift compatible to} 
$\mathbf m$ if 
\index{characteristic exponent!(compatible) shift}
\begin{equation}
 \sum_{j=0}^p\sum_{\nu=1}^{n_j}\epsilon_{j,\nu}m_{j,\nu}=0.
\end{equation}
\end{defn}
\index{shift operator}
\index{00Rme@$R_{\mathbf m}(\epsilon,\lambda)$}
\begin{thm}[shift operator]\label{thm:irredrigid}
Fix a shift\/ $(\epsilon_{j,\nu})$ compatible to 
$\mathbf m\in\mathcal P^{(n)}_{p+1}$.
Then there is a \textsl{shift operator}\index{shift operator}
$R_{\mathbf m}(\epsilon,\lambda)\in W[x]
\otimes\mathbb C[\lambda_{j,\nu}]$ which gives 
a homomorphism of the equation $P_{\mathbf m}(\lambda')v=0$ to 
$P_{\mathbf m}(\lambda)u=0$ defined by $v=R_{\mathbf m}(\epsilon,\lambda)u$.
Here the Riemann scheme of $P_{\mathbf m}(\lambda)$ is
$\{\lambda_{\mathbf m}\}=\bigl\{[\lambda_{j,\nu}]_{(m_{j,\nu})}\bigr\}
_{\substack{j=0,\dots,p\\\nu=1,\dots,n_j}}$
and that of $P_{\mathbf m}(\lambda')$ is  $\{\lambda'_{\mathbf m}\}$ 
defined by $\lambda'_{j,\nu}=\lambda_{j,\nu}+\epsilon_{j,\nu}$.
Moreover we may assume $\ord R_{\mathbf m}(\epsilon,\lambda) < \ord\mathbf m$
and $R_{\mathbf m}(\epsilon,\lambda)$ never vanishes as a function of $\lambda$
and then $R_{\mathbf m}(\epsilon,\lambda)$ is uniquely determined up to a constant multiple. 

Putting
\begin{equation}\label{eq:sfttau}
 \tau=\bigl(\tau_{j,\nu}\bigr)
  _{\substack{0\le j\le p\\1\le\nu\le n_j}} \text{ \ with \ }
 \tau_{j,\nu}:=\bigl(2+(p-1)n\bigr)\delta_{j,0}-m_{j,\nu}
\end{equation}
and $d=\ord R_{\mathbf m}(\epsilon,\lambda)$, 
we have
\begin{equation}\label{eq:Sid}
  P_{\mathbf m}(\lambda+\epsilon)R_{\mathbf m}(\epsilon,\lambda)
  =(-1)^dR_{\mathbf m}(\epsilon,\tau-\lambda-\epsilon)^*
    P_{\mathbf m}(\lambda) 
\end{equation}
under the notation in\/ {\rm Theorem~\ref{thm:prod} ii).}
\end{thm}
\begin{proof}
We will prove the theorem by the induction on $\ord\mathbf m$.
The theorem is clear if $\ord\mathbf m=1$.

We may assume that $\mathbf m$ is monotone.
Then the reduction $\{\tilde\lambda_{\tilde{\mathbf m}}\}$ of 
the Riemann scheme is defined by \eqref{eq:mc2}. 
Hence putting
\begin{equation}
 \begin{cases}
  {\tilde \epsilon}_1    = \epsilon_{0,1}+\cdots+\epsilon_{p,1},\\
  {\tilde \epsilon}_{j,\nu} =\epsilon_{j,\nu} +
    \bigl((-1)^{\delta_{j,0}}-\delta_{\nu,1}\bigr){\tilde \epsilon}_1
   &(j=0,\dots,p,\ \nu=1,\dots,n_j),
  \end{cases}
\end{equation}
there is a shift operator $R(\tilde\epsilon, \tilde\lambda)$ of the
equation $P_{\tilde{\mathbf m}}(\tilde\lambda')\tilde v=0$
to $P_{\tilde{\mathbf m}}(\tilde\lambda)\tilde u=0$ defined
by $\tilde v=R(\tilde\epsilon, \tilde\lambda)\tilde u$.
Note that
\begin{align*}
 P_{\tilde{\mathbf m}}(\tilde\lambda) 
  &=\p_{max}P_{\mathbf m}(\lambda)
  = \Ad\bigl(\prod_{j=1}^p(x-c_j)^{\lambda_{j,1}}\bigr)
    \prod_{j=1}^p(x-c_j)^{m_{j,1}-d}\p^{-d}\!\Ad(\p^{-\mu})\\
  &\quad
    \prod_{j=1}^p(x-c_j)^{-m_{j,1}}
    \Ad\bigl(\prod_{j=1}^p(x-c_j)^{-\lambda_{j,1}}\bigr)
    P_{\mathbf m}(\lambda),
\allowdisplaybreaks\\
 P_{\tilde{\mathbf m}}(\tilde\lambda') 
 &=\p_{max}P_{\mathbf m}(\lambda')
 = \Ad\bigl(\prod_{j=1}^p(x-c_j)^{\lambda_{j,1}}\bigr)
   \prod_{j=1}^p(x-c_j)^{m_{j,1}-d}\p^{-d}\!\Ad(\p^{-\mu'})\\
   &\quad
    \prod_{j=1}^p(x-c_j)^{-m_{j,1}}
    \Ad\bigl(\prod_{j=1}^p(x-c_j)^{-\lambda'_{j,1}}\bigr)
    P_{\mathbf m}(\lambda').
\end{align*}

Suppose $\lambda_{j,\nu}$ are generic.
Let $u(x)$ be a local solution of $P_{\mathbf m}(\lambda)u=0$
at $x=c_1$ corresponding to a characteristic exponent different
from $\lambda_{1,1}$.
Then
\[
 \tilde u(x):=\prod_{j=1}^p(x-c_j)^{\lambda_{j,1}}
  \p^{-\mu}\prod_{j=1}^p(x-c_j)^{-\lambda_{j,1}}u(x)
\]
satisfies $P_{\tilde{\mathbf m}}(\tilde\lambda)\tilde u(x)=0$.
Putting
\[
 \begin{split}
  \tilde v(x)&:=R(\tilde\epsilon,\tilde\lambda)\tilde u(x),
  \allowdisplaybreaks\\
  v(x)&:=\prod_{j=1}^p(x-c_j)^{\lambda'_{j,1}}
  \p^{\mu'}\prod_{j=1}^p(x-c_j)^{\lambda'_{j,1}}\tilde v(x),
  \allowdisplaybreaks\\
 \tilde R(\tilde\epsilon,\tilde\lambda)&
  :=\Ad(\prod_{j=1}^p(x-c_j)^{\lambda_{j,1}})R(\tilde\epsilon,\tilde\lambda)
 \end{split}
\]
we have $P_{\tilde{\mathbf m}}(\tilde\lambda')\tilde u(x)=0$,
$P_{\mathbf m}(\lambda')v(x)=0$ and
\[
  \prod_{j=1}^p(x-c_j)^{\epsilon_{j,1}}
  \p^{-\mu'}\prod_{j=1}^p(x-c_j)^{-\lambda_{j,1}'}
  v(x)=\tilde R(\tilde\epsilon,\tilde\lambda)\p^{-\mu}
  \prod_{j=1}^p(x-c_j)^{-\lambda_{j,1}}u(x).
\]

In general, if
\begin{equation}\label{eq:sht1}
  S_2\prod_{j=1}^p(x-c_j)^{\epsilon_{j,1}}
  \p^{-\mu'}\prod_{j=1}^p(x-c_j)^{-\lambda_{j,1}'}
  v(x)=S_1\p^{-\mu}
  \prod_{j=1}^p(x-c_j)^{-\lambda_{j,1}}u(x)
\end{equation}
with $S_1$, $S_2\in W[x]$, we have
\begin{equation}\label{eq:sht2}
     R_2 v(x) =  R_1 u(x)
\end{equation}
by putting
\begin{equation}\label{eq:sht3}\begin{split}
 R_1 &= 
    \prod_{j=1}^p(x-c_j)^{\lambda_{j,\nu}+k_{1,j}}
    \p^{\mu+\ell}\prod_{j=1}^p(x-c_j)^{k_{2,j}} S_1\prod_{j=1}^{\epsilon_{j,1}}\p^{-\mu}
    \prod_{j=1}^p(x-c_j)^{-\lambda_{j,\nu}},\\
 R_2 &= 
    \prod_{j=1}^p(x-c_j)^{\lambda_{j,\nu}+k_{1,j}}
    \p^{\mu+\ell}\prod_{j=1}^p(x-c_j)^{k_{2,j}} 
    S_2\prod_{j=1}^{\epsilon_{j,1}}\p^{-\mu'}
    \prod_{j=1}^p(x-c_j)^{-\lambda_{j,\nu}'}
\end{split}\end{equation}
with suitable integers $k_{1,j}$, $k_{2,j}$ and $\ell$ so 
that $R_1,\ R_2\in W[x;\lambda]$.

We choose a non-zero polynomial $S_2\in\mathbb C[x]$ so that 
$S_1=S_2\tilde R(\tilde\epsilon,\tilde\lambda)\in W[x]$.
Since $P_{\mathbf m}(\lambda')$ is irreducible in $W(x;\lambda)$
and $R_2v(x)$ is not zero, there exists $R_3 \in W(x;\xi)$ such that 
$R_3 R_2-1\in W(x;\lambda)P_{\mathbf m}(\lambda')$.
Then $v(x)=R u(x)$ with the operator $R=R_3 R_1\in W(x;\lambda)$.

Since the equations $P_{\mathbf m}(\lambda)u=0$ and $P_{\mathbf m}(\lambda')v=0$
are irreducible $W(x;\lambda)$-modules, the correspondence $v=Ru$ gives
an isomorphism between these two modules.
Since any solutions of these equations are holomorphically continued along the 
path contained in $\mathbb C\setminus\{c_1,\dots,c_p\}$,
the coefficients of the operator $R$ are holomorphic in 
$\mathbb C\setminus\{c_1,\dots,c_p\}$.
Multiplying $R$ by a suitable element of $\mathbb C(\lambda)$, we may assume 
$R\in W(x)\otimes \mathbb C[\lambda]$ and $R$ does not vanish at any
$\lambda_{j,\nu}\in\mathbb C$.

Put $f(x)=\prod_{j=1}^p(x-c_j)^n$.
Since $R_{\mathbf m}(\epsilon,\lambda)$ is a shift operator, there exists
$S_{\mathbf m}(\epsilon,\lambda)\in W(x;\lambda)$ such that
\begin{equation}\label{eq:sft01}
  f^{-1}P_{\mathbf m}(\lambda+\epsilon)R_{\mathbf m}(\epsilon,\lambda)
 = S_{\mathbf m}(\epsilon,\lambda)f^{-1}P_{\mathbf m}(\lambda).
\end{equation}
Then Theorem~\ref{thm:prod} ii) shows
\begin{align}
   R_{\mathbf m}(\epsilon,\lambda)^*
   \bigl(f^{-1}P_{\mathbf m}(\lambda+\epsilon)\bigr)^*
 &=\bigl(f^{-1}P_{\mathbf m}(\lambda)\bigr)^* 
    S_{\mathbf m}(\epsilon,\lambda)^*,\notag\allowdisplaybreaks\\
   R_{\mathbf m}(\epsilon,\lambda)^*\cdot f^{-1}
   P_{\mathbf m}(\lambda+\epsilon)^\vee
 &=f^{-1}P_{\mathbf m}(\lambda)^\vee\cdot
   S_{\mathbf m}(\epsilon,\lambda)^*,\notag\allowdisplaybreaks\\
   R_{\mathbf m}(\epsilon,\lambda)^*f^{-1}P_{\mathbf m}(\rho-\lambda-\epsilon)
 &=f^{-1}P_{\mathbf m}(\rho-\lambda)S_{\mathbf m}(\epsilon,\lambda)^*,
 \notag\allowdisplaybreaks\\
 R_{\mathbf m}(\epsilon,\rho-\mu-\epsilon)^*f^{-1}P_{\mathbf m}(\mu)
 &=f^{-1}P_{\mathbf m}(\mu+\epsilon)S_{\mathbf m}(\epsilon,\rho-\mu-\epsilon)^*.
 \label{eq:sht02}
\end{align}
Here we use the notation \eqref{eq:dualop}
and put $\rho_{j,\nu}=2(1-n)\delta_{j,0}+n-m_{j,\nu}$ and 
$\mu=\rho-\lambda-\epsilon$. 
Comparing \eqref{eq:sht02} with \eqref{eq:sft01}, we see that
$S_{\mathbf m}(\epsilon,\lambda)$ is a constant multiple of the operator 
$R_{\mathbf m}(\epsilon,\rho-\lambda-\epsilon)^*$ and 
$fR_{\mathbf m}(\epsilon,\rho-\lambda-\epsilon)^*f^{-1}
=\bigl(f^{-1}R_{\mathbf m}(\epsilon,\rho-\lambda-\epsilon)f\bigr)^*
=R_{\mathbf m}(\epsilon,\tau-\lambda-\epsilon)^*$ and we have \eqref{eq:Sid}.
\end{proof}

Note that the operator $R_{\mathbf m}(\epsilon,\lambda)$ is uniquely 
defined up to a constant multiple.

The following theorem gives a recurrence relation among specific local solutions 
with a rigid spectral type and a relation between the shift operator
$R_{\mathbf m}(\epsilon,\lambda)$ and the universal operator $P_{\mathbf m}(\lambda)$.
\begin{thm}\label{thm:shifm1}
Retain the notation in {\rm Corollary~\ref{cor:irred}} and\/ 
{\rm Theorem~\ref{thm:shiftC}} with a rigid tuple 
$\mathbf m$.
Assume $m_{j,n_j}=1$ for $j=0$, $1$ and $2$.
Put $\epsilon=(\epsilon_{j,\nu})$, $\epsilon'=(\epsilon'_{j,\nu})$,
\begin{equation}
 \epsilon_{j,\nu} = \delta_{j,1}\delta_{\nu,n_1}-\delta_{j,2}\delta_{\nu,n_2}
 \text{ \ and \ }
 \epsilon'_{j,\nu} = \delta_{j,0}\delta_{\nu,n_0}-\delta_{j,2}\delta_{\nu,n_2} 
\end{equation}
for $j=0,\dots,p$ and $\nu=1,\dots,n_j$.

{\rm i)}
Define $Q_{\mathbf m}(\lambda)\in W(x;\lambda)$ so that
$Q_{\mathbf m}(\lambda) P_{\mathbf m}(\lambda+\epsilon') - 1
\in W(x;\lambda)P_{\mathbf m}(\lambda+\epsilon)$.
Then 
\begin{equation}\label{eq:shtUniv}
 R_{\mathbf m}(\epsilon,\lambda) - C(\lambda)Q_{\mathbf m}(\lambda)
 P_{\mathbf m}(\lambda+\epsilon')
\in W(x;\lambda)P_{\mathbf m}(\lambda)
\end{equation} with a rational function
$C(\lambda)$ of $\lambda_{j,\nu}$.

{\rm ii)} 
Let $u_{\lambda}(x)$ be the local solution of $P_{\mathbf m}(\lambda)u=0$
such that $u_{\lambda}(x)\equiv (x-c_1)^{\lambda_{1,n_1}}\mod  
(x-c_1)^{\lambda_{1,n_1}+1}O_{c_1}$
for generic $\lambda_{j,\nu}$.
Then we have the recurrence relation
\begin{align}\label{eq:recF}
 u_{\lambda}(x) &= 
  u_{\lambda+\epsilon'}(x)+(c_1-c_2)\prod_{\nu=0}^{K-1}
 \frac{\lambda(\nu+1)_{1,n_1}-\lambda(\nu)_{1,\ell(\nu)_1}+1}
 {\lambda(\nu)_{1,n_1}-\lambda(\nu)_{1,\ell(\nu)_1}+1}\cdot
 u_{\lambda+\epsilon}(x).
\end{align}
\end{thm}
\begin{proof}
Under the notation in Corollary~\ref{cor:irred}, $\ell(k)_j\ne n_j$ for 
$j=0,1,2$ and $k=0,\dots,K-1$ and 
therefore the operation $\p_{max}^K$ on $P_{\mathbf m}(\lambda)$ 
is equals to $\p_{max}^K$ on $P_{\mathbf m}(\lambda+\epsilon)$ if they are
realized by the product of the operators of the form \eqref{eq:opred}.
Hence by the induction on $K$,
the proof of Theorem~\ref{thm:irredrigid} 
(cf.~\eqref{eq:sht1}, \eqref{eq:sht2} and \eqref{eq:sht3}) 
shows
\begin{equation}\label{eq:shifm1}
 P_{\mathbf m}(\lambda+\epsilon')u(x)=P_{\mathbf m}(\lambda+\epsilon')v(x)
\end{equation}
for suitable functions $u(x)$ and $v(x)$ satisfying 
$P_{\mathbf m}(\lambda)u(x)=P_{\mathbf m}(\lambda+\epsilon)v(x)=0$
and moreover \eqref{eq:recF} is calculated by \eqref{eq:IcP}.
Note that the identities
\begin{align*}
 (c_1-c_2)\prod_{j=1}^p(x-c_j)^{\lambda_j+\epsilon'_j}
 &=\prod_{j=1}^p(x-c_j)^{\lambda_j}-\prod_{j=1}^p(x-c_j)^{\lambda_j+\epsilon_j},
 \allowdisplaybreaks\\
 \Bigl(\p - \sum_{j=1}^p\frac{\lambda_j+\epsilon'_j}
 {x-c_j}\Bigr)\prod_{j=1}^p(x-c_j)^{\lambda_j}
 &= \Bigl(\p - \sum_{j=1}^p\frac{\lambda_j+\epsilon'_j}{x-c_j}\Bigr)
  \prod_{j=1}^p(x-c_j)^{\lambda_j+\epsilon_j}
\end{align*}
correspond to \eqref{eq:recF}  and \eqref{eq:shifm1}, respectively, when $K=0$.

Note that \eqref{eq:shifm1} may be proved by \eqref{eq:recF}.
The claim i) in this theorem follows from the fact 
 $v(x)=Q_{\mathbf m}(\lambda)P_{\mathbf m}(\lambda+\epsilon')v(x)
=Q_{\mathbf m}(\lambda)P_{\mathbf m}(\lambda+\epsilon')u(x)$.
\end{proof}
In general, we have the following theorem for the recurrence relation.
\begin{thm}[recurrence relations]
Let\/ $\mathbf m\in\mathcal P^{(n)}$ be a rigid tuple with\/ $m_{1,n_1}=1$
and let\/ $u_1(\lambda,x)$ be the normalized solution of 
the equation\/ $P_{\mathbf m}(\lambda)u=0$ with respect to the exponent\/
$\lambda_{1,n_1}$ at $x=c_1$.
Let\/ $\epsilon^{(i)}$ be shifts compatible to\/ $\mathbf m$ for\/ $i=0,\dots,n$.
Then there exists polynomial functions\/ 
$r_i(x,\lambda)\in\mathbb C[x,\lambda]$ such that $(r_0,\dots,r_n)\ne 0$ and
\begin{equation}
  \sum_{i=0}^{n} r_i(x,\lambda)u_1(\lambda+\epsilon^{(i)},x) = 0.
\end{equation}
\end{thm}
\begin{proof}
There exist $R_i\in \mathbb C(\lambda)R_{\mathbf m}(\epsilon^{(i)},\lambda)$
satisfying $u_1(\lambda+\epsilon^{(i)},x)=R_iu_1(\lambda,x)$ and $\ord R_i< n$.
We have $r_i(x,\lambda)$ with
$\sum_{i=0}^n r_i(x,\lambda)R_i=0$ and the claim.
\end{proof}
\begin{exmp}[Gauss hypergeometric equation]%
\index{hypergeometric equation/function!Gauss}
Let $P_\lambda u=0$ and $P_{\lambda'}v=0$ be Fuchsian differential
equations with the Riemann Scheme
\[
 \begin{Bmatrix}
  x=\infty & 0 & 1\\
  \lambda_{0,1} & \lambda_{1,1} & \lambda_{2,1}\\
  \lambda_{0,2} & \lambda_{1,2} & \lambda_{2,2}
 \end{Bmatrix}\text{ and }
 \begin{Bmatrix}
  x=\infty & 0 & 1\\
  \lambda'_{0,1}=\lambda_{0,1} & 
    \lambda'_{1,1}=\lambda_{1,1} & \lambda'_{2,1}=\lambda_{2,1}\\
  \lambda'_{0,2}=\lambda_{0,2} & 
    \lambda'_{1,2}=\lambda_{1,2}+1 & \lambda_{2,2}=\lambda_{2,2}-1
 \end{Bmatrix},
\]
respectively.
Here the operators $P_\lambda=P_{\lambda_{0,1},\lambda_{0,2},\lambda_{1,1},
\lambda_{1,2},\lambda_{2,1},\lambda_{2,1}}$ and $P_{\lambda'}$ are given in 
\eqref{eq:GH2}.
The normalized local solution $u_{\lambda}(x)$ of $P_\lambda u=0$ corresponding to 
the exponent $\lambda_{1,2}$ at $x=0$ is
\begin{equation}
x^{\lambda_{1,2}}(1-x)^{\lambda_{2,1}}
    F(\lambda_{0,1}+\lambda_{1,2}+\lambda_{2,1},
    \lambda_{0,2}+\lambda_{1,2}+\lambda_{2,1},
    1-\lambda_{1,1}+\lambda_{1,2};x).
\end{equation}
By the reduction  $\begin{Bmatrix}
  x=\infty & 0 & 1\\
  \lambda_{0,1} & \lambda_{1,1} & \lambda_{2,1}\\
  \lambda_{0,2} & \lambda_{1,2} & \lambda_{2,2}
 \end{Bmatrix}\to
 \begin{Bmatrix}
  x=\infty & 0 & 1\\
  \lambda_{0,2} -\mu& 
  \lambda_{1,2} +\mu&
  \lambda_{2,2} +\mu
 \end{Bmatrix}
$
with 
$\mu=\lambda_{0,1}+\lambda_{1,1}+\lambda_{2,1}-1$, the recurrence relation 
\eqref{eq:recF} means
\begin{align*}
  &x^{\lambda_{1,2}}(1-x)^{\lambda_{2,1}}
    F(\lambda_{0,1}+\lambda_{1,2}+\lambda_{2,1},
    \lambda_{0,2}+\lambda_{1,2}+\lambda_{2,1},
    1-\lambda_{1,1}+\lambda_{1,2};x)\\
  &=x^{\lambda_{1,2}}(1-x)^{\lambda_{2,1}}
    F(\lambda_{0,1}+\lambda_{1,2}+\lambda_{2,1},
    \lambda_{0,2}+\lambda_{1,2}+\lambda_{2,1}+1,
    1-\lambda_{1,1}+\lambda_{1,2};x)\\
  &\quad{}-\frac{\lambda_{0,1}+\lambda_{1,2}+\lambda_{2,1}}
    {1-\lambda_{1,1}+\lambda_{1,2}}
    x^{\lambda_{1,2}+1}(1-x)^{\lambda_{2,1}}\\
  &\qquad \cdot F(\lambda_{0,1}+\lambda_{1,2}+\lambda_{2,1}+1,
    \lambda_{0,2}+\lambda_{1,2}+\lambda_{2,1}+1,
    2-\lambda_{1,1}+\lambda_{1,2};x),
\end{align*}
which is equivalent to the recurrence relation
\begin{equation}
 F(\alpha,\beta,\gamma,x)=F(\alpha,\beta+1,\gamma;x)
 -\frac{\alpha}{\gamma}x F(\alpha+1,\beta+1,\gamma+1;x).
\end{equation}

Using the expression \eqref{eq:GH2}, we have
\begin{align*}
 P_{\lambda+\epsilon'} - P_{\lambda}&
  =x^2(x-1)\p+\lambda_{0,1}x^2-(\lambda_{0,1}+\lambda_{2,1})x,\\
 P_{\lambda+\epsilon'} - P_{\lambda+\epsilon}&=x(x-1)^2\p+\lambda_{0,1}x^2
  -(\lambda_{0,1}+\lambda_{1,1})x-\lambda_{1,1},\\
 (x-1)P_{\lambda+\epsilon}&= \bigl(x(x-1)\p +(\lambda_{0,2}-2)x+\lambda_{1,2}
  +1\bigr)
 \bigl(P_{\lambda+\epsilon'} - P_{\lambda+\epsilon}\bigr)\\
 &{}\quad-(\lambda_{0,1}+\lambda_{1,1}+\lambda_{2,1})(\lambda_{0,2}
  +\lambda_{1,2}+\lambda_{2,1})x(x-1),\\
 x^{-1}(x-1)^{-1}&\bigl(x(x-1)\p +(\lambda_{0,2}-2)x+\lambda_{1,2}
  +1\bigr)(P_{\lambda+\epsilon'} - P_{\lambda})-(x-1)^{-1}P_\lambda\\
 &=-(\lambda_{0,1}+\lambda_{1,1}+\lambda_{2,1})
   \bigl(x\p-\lambda_{1,2}-\frac{\lambda_{2,1} x}{x-1}\bigr)
\end{align*}
and hence \eqref{eq:shtUniv} says
\begin{equation}
 R_{\mathbf m}(\epsilon,\lambda)=x\p-\lambda_{1,2}-
 \lambda_{2,1}\frac{x}{x-1}.
\end{equation}
In the same way we have
\begin{equation}
 R_{\mathbf m}(-\epsilon,\lambda+\epsilon)=
  (x-1)\p-\lambda_{2,2}+1-\lambda_{1,1}\frac{x-1}{x}.
\end{equation}
Then
\begin{equation}
 \begin{split}
  R_{\mathbf m}(-\epsilon,\lambda+\epsilon)R_{\mathbf m}(\epsilon,\lambda)
   &-x^{-1}(x-1)^{-1}P_\lambda\\
   &=-(\lambda_{0,1}+\lambda_{1,2}+\lambda_{2,1})
   (\lambda_{0,2}+\lambda_{1,2}+\lambda_{2,1})
 \end{split}
\end{equation}
and since
$-R_{\mathbf m}(\epsilon,\tau-\lambda-\epsilon)^*
=-\bigl(x\p+(\lambda_{1,2}+2)+(\lambda_{2,1}+1)\frac x{x-1}\bigr)^*
=x\p - \lambda_{1,2}-1-(\lambda_{2,1}+1)\frac x{x-1}$ 
with $\tau$ given by \eqref{eq:sfttau},  the identity
\eqref{eq:Sid} means 
\begin{equation}
 \begin{split}
 P_\lambda R_{\mathbf m}(\epsilon,\lambda)
 = \Bigl(x\p-(\lambda_{1,2}+1) -(\lambda_{2,1}+1)\frac{x}{x-1}\Bigr)
  P_{\lambda+\epsilon}.
 \end{split}
\end{equation}
\end{exmp}
\begin{rem}
Suppose $\mathbf m$ is irreducibly realizable but it is not rigid.
If the reductions of $\{\lambda_{\mathbf m}\}$ and $\{\lambda'_{\mathbf m}\}$
to Riemann schemes with a fundamental tuple of partitions are 
transformed into each other by suitable additions, we can construct
a shift operator as in Theorem~\ref{thm:irredrigid}.
If they are not so, we need a shift operator for equations whose spectral type
are fundamental and such an operator is called a 
\textsl{Schlesinger transformation}.
\end{rem}
Now we examine the condition that a universal operator defines a shift operator.
\begin{thm}[universal operator and shift operator]\label{thm:sftUniv}\index{shift operator}
Let\/ $\mathbf m=\bigl(m_{j,\nu}\bigr)_{\substack{0\le j\le p\\1\le\nu\le n_j}}$ 
and\/ $\mathbf m'=\bigl(m'_{j,\nu}\bigr)_{\substack{0\le j\le p\\1\le\nu\le n_j}}
\in \mathcal P_{p+1}$ be irreducibly realizable and monotone.
They may not be rigid.
Suppose\/ $\ord\mathbf m>\ord\mathbf m'$.
Fix $j_0$ with $0\le j_0\le p$.
Let $n'_{j_0}$ be a positive integer such that $m'_{j_0,n'_{j_0}}>m'_{j_0,n'_{j_0+1}}=0$
and let $P_{\mathbf m}(\lambda)$ be the universal operator corresponding to 
$\{\lambda_{\mathbf m}\}$.
Putting $\lambda'_{j,\nu}=\lambda_{j,\nu}$ when $(j,\nu)\ne(j_0,n'_{j_0})$, we define
the universal operator $P^{j_0}_{\mathbf m'}(\lambda):=P_{\mathbf m'}(\lambda')$ 
with the Riemann scheme $\{\lambda'_{\mathbf m'}\}$.  
Here $\lambda'_{j_0,n'_{j_0}}$ is determined by the 
Fuchs condition. 
Suppose
\begin{equation}\label{eq:sftUinv}
  (\alpha_{\mathbf m}|\alpha_{\mathbf m'})\Bigl(=\sum_{j=0}^p\sum_{\nu=1}^{n_j}m_{j,\nu}m'_{j,\nu}
  -(p-1)\ord\mathbf m\cdot\ord\mathbf m'\Bigr)
  =m_{j_0,n'_{j_0}}m'_{j_0,n'_{j_0}}.
\end{equation}
Then\/ $\mathbf m'$ is rigid and
the universal operator $P^{j_0}_{\mathbf m'}(\lambda)$ is the shift operator 
$R_{\mathbf m}(\epsilon,\lambda)$: 
\begin{equation}
 \begin{split}
 &\Bigl\{
  \bigl[\lambda_{j,\nu}\bigr]_{(m_{j,\nu})}
 \Bigr\}_{\substack{0\le j\le p\\1\le\nu\le n_j}}
 \xrightarrow{R_{\mathbf m}(\epsilon,\lambda)=P^{j_0}_{\mathbf m'}(\lambda)}
 \Bigl\{
  \bigl[\lambda_{j,\nu}+\epsilon_{j,\nu}\bigr]_{(m_{j,\nu})}
 \Bigr\}_{\substack{0\le j\le p\\1\le\nu\le n_j}}\\
 &\qquad\text{with \ }\epsilon_{j,\nu}=
   \bigl(1-\delta_{j,j_0}\delta_{\nu,n'_{j_0}}\bigr)m'_{j,\nu}
  -\delta_{j,0}\cdot(p-1)\ord\mathbf m'.
\end{split}
\end{equation}
\end{thm}
\begin{proof}
We may assume $\lambda$ is generic.
Let $u(x)$ be the solution of the irreducible differential equation 
$P_{\mathbf m}(\lambda)u=0$.
Then 
\begin{align*}
 P_{\mathbf m'}(\lambda')(x-c_j)^{\lambda_{j,\nu}}\mathcal O_{c_j}&\subset
   (x-c_j)^{\lambda_{j,\nu}+(1-\delta_{j,j_0}\delta_{\nu,n'_{j_0}})m'_{j,\nu}}
   \mathcal O_{c_j},\\
 P_{\mathbf m'}(\lambda')x^{-\lambda_{0,\nu}}\mathcal O_{\!\infty}&\subset
    x^{-\lambda_{0,\nu}-(1-\delta_{0,j_0}\delta_{\nu,n'_{j_0}})
    m'_{0,\nu}+(p-1)\ord\mathbf m'}
   \mathcal O_{\!\infty}
\end{align*}
and $P_{\mathbf m'}(\lambda')u(x)$ satisfies 
a Fuchsian differential equation.
Hence the fact $R_{\mathbf m}(\epsilon,\lambda)=P_{\mathbf m'}(\lambda')$ 
is clear from the characteristic exponents of the equation
at each singular points.
Note that the left hand side of \eqref{eq:sftUinv} is never larger than the right
hand side and if they are not equal, $P_{\mathbf m'}(\lambda')u(x)$ satisfies 
a Fuchsian differential equation with apparent singularities for the solutions
$u(x)$ of $P_{\mathbf m}(\lambda)u=0$.

It follows from Lemma~\ref{lem:sumlem} that the condition \eqref{eq:sftUinv} means 
that at least one of the irreducibly realizable tuples $\mathbf m$ and $\mathbf m'$
is rigid and therefore if $\mathbf m$ is rigid, so is $\mathbf m'$ because 
$R_{\mathbf m}(\epsilon,\lambda)$ is unique up to constant multiple.
\end{proof}
If $\ord\mathbf m'=1$, the condition \eqref{eq:sftUinv} means that $\mathbf m$ is 
of Okubo type, which will be examined in the next subsection.
It will be interesting to examine other cases.
When $\mathbf m=\mathbf m'\oplus\mathbf m''$ is a rigid decomposition
or $\alpha_{\mathbf m'}\in\Delta(\mathbf m)$, we easily have many examples 
satisfying \eqref{eq:sftUinv}.

Here we give examples of the pairs $(\mathbf m\,;\mathbf m')$ with $\ord\mathbf m'>1$:
\begin{equation}
 \begin{aligned}
 &(1^n,1^n,n-11\,;1^{n-1},1^{n-1},n-21)&&
  (221,32,32,41\,;\,110,11,11,20)\\
 &(1^{2m},mm-11,m^2\,;1^2,110,1^2) &&
 (1^{2m+1},m^21,m+1m\,;1^2,1^20,11)\\
 &(221,221,221\,;110,110,110)\ &&
 (211,221,221\,;110,110,110).
 \end{aligned}
\end{equation}

\subsection{Relation to reducibility}
In this subsection, we will examine whether the shift operator defines a 
$W(x)$-isomorphism or doesn't.
\begin{thm}\label{thm:shiftC}
\index{00cme@$c_{\mathbf m}(\epsilon;\lambda)$}
Retain the notation in\/ {\rm Theorem~\ref{thm:irredrigid}} and define a polynomial 
function $c_{\mathbf m}(\epsilon;\lambda)$ of $\lambda_{j,\nu}$ by 
\begin{equation}
 R_{\mathbf m}(-\epsilon,\lambda+\epsilon)R_{\mathbf m}(\epsilon,\lambda)
 -c_{\mathbf m}(\epsilon;\lambda)
 \in \bigl(W[x]\otimes\mathbb C[\lambda]\bigr)P_{\mathbf m}(\lambda).
\end{equation}

{\rm i)}
Fix $\lambda_{j,\nu}^o\in\mathbb C$.
If $c_{\mathbf m}(\epsilon;\lambda^o)\ne0$, 
the equation $P_{\mathbf m}(\lambda^o)u=0$ is isomorphic to 
the equation $P_{\mathbf m}(\lambda^o+\epsilon)v=0$.
If $c_{\mathbf m}(\epsilon;\lambda^o)=0$, then the equations
$P_{\mathbf m}(\lambda^o)u=0$ and $P_{\mathbf m}(\lambda^o+\epsilon)v=0$
are not irreducible.

{\rm ii)}
Under the notation in\/ {\rm Proposition~\ref{prop:subrep},} there exists a set
$\Lambda$ whose elements $(i,k)$ are in $\{1,\dots,N\}\times \mathbb Z$ such 
that
\begin{equation}
 c_{\mathbf m}(\epsilon;\lambda)=C\prod_{(i,k)\in\Lambda}
 \bigl(\ell_i(\lambda)-k\bigr)
\end{equation}
with a constant $C\in\mathbb C^\times$.
Here $\Lambda$ may contain some elements $(i,k)$
with multiplicities.
\end{thm}
\begin{proof}
Since $u\mapsto 
R_{\mathbf m}(-\epsilon,\lambda+\epsilon)R_{\mathbf m}(\epsilon,\lambda)u$
defined an endomorphism of the irreducible 
equation $P_{\mathbf m}(\lambda)u=0$, the existence of
$c_{\mathbf m}(\epsilon;\lambda)$ is clear.

If $c_{\mathbf m}(\epsilon;\lambda^o)=0$, the non-zero homomorphism of 
$P_{\mathbf m}(\lambda^o)u=0$ to $P_{\mathbf m}(\lambda^o+\epsilon)v=0$
defined by $u=R_{\mathbf m}(\epsilon;\lambda^o)v$ is not surjective nor injective.
Hence the equations are not irreducible.
If $c_{\mathbf m}(\epsilon;\lambda^o)\ne0$, then the homomorphism is an isomorphism
and the equations are isomorphic to each other.

The claim ii) follows from Proposition~\ref{prop:subrep}.
\end{proof}
\begin{thm}\label{thm:isom}
Retain the notation in\/ {\rm Theorem~\ref{thm:shiftC}} with a rigid tuple\/ 
$\mathbf m$.
Fix a linear function $\ell(\lambda)$ of $\lambda$ such that the 
condition $\ell(\lambda)=0$ implies the reducibility of the 
universal equation $P_{\mathbf m}(\lambda)u=0$. 

{\rm i) }
If there is no irreducible realizable subtuple\/ $\mathbf m'$ 
of\/ $\mathbf m$ which is compatible to $\ell(\lambda)$ and 
$\ell(\lambda+\epsilon)$, $\ell(\lambda)$ is a factor of 
$c_{\mathbf m}(\epsilon;\lambda)$.

If there is no dual decomposition of\/ $\mathbf m$ with respect to the
pair $\ell(\lambda)$ and $\ell(\lambda+\epsilon)$, 
$\ell(\lambda)$ is not a factor of 
$c_{\mathbf m}(\epsilon;\lambda)$.
Here we define that the decomposition \eqref{eq:subrep} 
is \textsl{dual} with respect to 
the pair $\ell(\lambda)$ and $\ell(\lambda+\epsilon)$ 
if the following conditions are valid.
\begin{align}
&\text{$\mathbf m'$ is an irreducibly realizable 
subtuple of\/ $\mathbf m$ compatible to $\ell(\lambda)$},\\
&\text{$\mathbf m''$ is a 
subtuple of\/ $\mathbf m$ compatible to $\ell(\lambda+\epsilon)$}.
\end{align}

{\rm ii)}
Suppose there exists a decomposition\/
$\mathbf m=\mathbf m'\oplus\mathbf m''$ with rigid tuples\/
$\mathbf m'$ and\/ $\mathbf m''$ such that
$\ell(\lambda)=|\{\lambda_{\mathbf m}\}|+k$ with $k\in\mathbb Z$
and $\ell(\lambda+\epsilon)=\ell(\lambda)+1$.
Then $\ell(\lambda)$ is a factor of $c_{\mathbf m}(\epsilon;\lambda)$
if and only if $k=0$.
\end{thm}
\begin{proof}
Fix generic complex numbers $\lambda_{j,\nu}\in\mathbb C$ satisfying
$\ell(\lambda)=|\{\lambda_{\mathbf m}\}|=0$.
Then we may assume $\lambda_{j,\nu}-\lambda_{j,\nu'}\notin\mathbb Z$
for $1\le\nu<\nu'\le n_j$ and $j=0,\dots,p$.

i) The shift operator 
$R:=R_{\mathbf m}(-\epsilon,\lambda+\epsilon)$ gives a 
non-zero $W(x)$-homomorphism of the equation 
$P_{\mathbf m}(\lambda+\epsilon)v=0$ to $P_{\mathbf m}(\lambda)u=0$ 
by the correspondence $v=Ru$.
Since the equation $P_{\mathbf m}(\lambda)u=0$ is reducible, 
we examine the decompositions of $\mathbf m$ described in 
Proposition~\ref{prop:subrep}.
Note that the genericity of $\lambda_{j,\nu}\in\mathbb C$ assures that
the subtuple $\mathbf m'$ of $\mathbf m$ corresponding to a decomposition 
$P_{\mathbf m}(\lambda)=P''P'$ is uniquely determined, namely, $\mathbf m'$
corresponds to the spectral type of the monodromy of the equation $P'u=0$.

If the shift operator $R$ is bijective, 
there exists a subtuple $\mathbf m'$ of $\mathbf m$ compatible to 
$\ell(\lambda)$ and $\ell(\lambda+\epsilon)$ because $R$ indices
an isomorphism of monodromy.

Suppose $\ell(\lambda)$ is a factor of $c_{\mathbf m}(\epsilon;\lambda)$.
Then $R$ is not bijective.
We assume that the image of $R$ is the equation $P''\bar u=0$ and the kernel of
$R$ is the equation $P'_\epsilon \bar v=0$.
Then $P_{\mathbf m}(\lambda)=P''P'$ and 
$P_{\mathbf m}(\lambda+\epsilon)=P'_\epsilon P''_\epsilon$ 
with suitable Fuchsian differential operators $P'$ and $P''_\epsilon$.
Note that the spectral type of the monodromy of $P'u=0$ and 
$P''_\epsilon v=0$ corresponds to $\mathbf m'$ and $\mathbf m''$ with 
$\mathbf m=\mathbf m'+\mathbf m''$.
Applying Proposition~\ref{prop:subrep} to the decompositions
$P_{\mathbf m}(\lambda)=P''P'$ and $P_{\mathbf m}(\lambda+\epsilon)
=P'_\epsilon P''_\epsilon$, 
we have a dual decomposition \eqref{eq:subrep} of 
$\mathbf m$ with respect to the pair $\ell(\lambda)$ and 
$\ell(\lambda+\epsilon)$.

ii)  Since $P_{\mathbf m}(\lambda)u=0$ is reducible, we have a decomposition
$P_{\mathbf m}(\lambda)=P''P'$ with $0<\ord P'<\ord P_{\mathbf m}(\lambda)$.
We may assume $P'u=0$ and let $\tilde{\mathbf m}'$ be the spectral type of
the monodromy of the equation $P'u=0$.
Then $\tilde{\mathbf m}'=\ell_1\mathbf m'+\ell_2\mathbf m''$ with integers 
$\ell_1$ and $\ell_2$ because 
$|\{\lambda_{\tilde{\mathbf m}'}\}|\in\mathbb Z_{\le 0}$.
Since $P'u=0$ is irreducible, 
$2\ge \idx\tilde{\mathbf m}'=2(\ell_1^2-\ell_1\ell_2+\ell_2^2)$ and therefore
$(\ell_1,\ell_2)=(1,0)$ or $(0,1)$.
Hence the claim follows from i) and the identity
$|\{\lambda_{\mathbf m'}\}|+|\{\lambda_{\mathbf m''}\}|=1$
\end{proof}
\begin{rem} i) \ 
The reducibility of $P_{\mathbf m}(\lambda)$ implies
that of the dual of $P_{\mathbf m}(\lambda)$.

ii)
When $\mathbf m$ is simply reducible (cf.~Definition \ref{def:fund}), 
each linear form of $\lambda_{j,\nu}$ describing the reducibility 
uniquely corresponds to a rigid decomposition of $\mathbf m$ and 
therefore Theorem~\ref{thm:isom} gives the necessary and sufficient condition
for the bijectivity of the shift operator $R_{\mathbf m}(\epsilon,\lambda)$.
\end{rem}
\begin{exmp}[$EO_4$]
\index{even family!$EO_4$}
\index{tuple of partitions!rigid!1111,211,22}
Let $P(\lambda)u=0$ and $P(\lambda')v=0$ be the Fuchsian differential equation 
with the Riemann schemes
\[
\begin{Bmatrix}
  \lambda_{0,1} & [\lambda_{1,1}]_{(2)} & [\lambda_{2,1}]_{(2)}\\
  \lambda_{0,2} & \lambda_{1,2} & [\lambda_{2,2}]_{(2)}\\
  \lambda_{0,3} & \lambda_{1,3}\\
  \lambda_{0,4}
 \end{Bmatrix}
\text{ \ and \ } 
\begin{Bmatrix}
  \lambda_{0,1} & [\lambda_{1,1}]_{(2)} & [\lambda_{2,1}]_{(2)}\\
  \lambda_{0,2} & \lambda_{1,2} & [\lambda_{2,2}]_{(2)}\\
  \lambda_{0,3} & \lambda_{1,3}+1\\
  \lambda_{0,4}-1
  \end{Bmatrix},
\]
respectively.
Since the condition of the reducibility of the equation corresponds to 
rigid decompositions \eqref{eq:e4ddec}, it easily follows from 
Theorem~\ref{thm:isom} that the shift operator between $P(\lambda)u=0$ and 
$P(\lambda')v=0$ is bijective if and only if
\begin{equation*}
 \begin{cases}
  \lambda_{0,4}+\lambda_{1,2}+\lambda_{2,\mu}-1\ne 0
   &(1\le \mu\le2),\\
  \lambda_{0,\nu}+\lambda_{0,\nu'}
  +\lambda_{1,1}+\lambda_{1,3}+\lambda_{2,1}+\lambda_{2,2}-1\ne 0
   &(1\le \nu<\nu'\le 3).
 \end{cases}
\end{equation*}

In general, for a shift $\epsilon=(\epsilon_{j,\nu})$ compatible to the
spectral type $1111,211,22$, the shift operator between $P(\lambda)u=0$ 
and $P(\lambda+\epsilon)v=0$ is bijective
if and only if the values of each function in the list
\begin{align}
  &\lambda_{0,\nu}+\lambda_{1,1}+\lambda_{2,\mu}
    &&(1\le\nu\le 4,\ 1\le\mu\le 2),\\
  &\lambda_{0,\nu}+\lambda_{0,\nu'}+\lambda_{1,1}+\lambda_{1,3}+
     \lambda_{2,1}+\lambda_{2,2}-1
    &&(1\le\nu<\nu'\le4)
\end{align}
are
\begin{equation}\label{eq:non-pos}
 \begin{cases}
   \text{\hspace{12.5pt}not integers for $\lambda$ and $\lambda+\epsilon$}\\
   \text{or positive integers for $\lambda$ and $\lambda+\epsilon$}\\
   \text{or non-positive integers for $\lambda$ and $\lambda+\epsilon$}.
 \end{cases}
\end{equation}
Note that the shift operator gives a homomorphism between monodromies
(cf.~\eqref{eq:isoWM}).
\end{exmp}
The following conjecture gives $c_{\mathbf m}(\epsilon;\lambda)$ under certain
conditions.
\begin{conj}\label{conj:shift}
Retain the assumption that 
$\mathbf m=\bigl(\lambda_{j,\nu}\bigr)_{\substack{0\le j\le p\\1\le\nu\le n_j}}
\in\mathcal P^{(n)}_{p+1}$ is rigid.

{\rm i) }
If $\ell(\lambda)=\ell(\lambda+\epsilon)$ in Theorem~\ref{thm:isom}, then
$\ell(\lambda)$ is not a factor of $c_{\mathbf m}(\epsilon;\lambda)$,

{\rm ii) }
Assume $m_{1,n_1}=m_{2,n_2}=1$ and
\begin{equation}
  \epsilon:=\bigl(\epsilon_{j,\nu}\bigr)_{\substack{0\le j\le p\\1\le\nu\le n_j}},
  \quad \epsilon_{j,\nu}=\delta_{j,1}\delta_{\nu,n_1}
  -\delta_{j,2}\delta_{\nu,n_2},
\end{equation}
Then we have
\begin{equation}\label{eq:cef}
  c_{\mathbf m}(\epsilon;\lambda)=C\prod_{\substack
   {\mathbf m=\mathbf m'\oplus\mathbf m''\\m'_{1,n_1}=m''_{2,n_2}=1}}
   |\{\lambda_{\mathbf m'}\}|
\end{equation}
with $C\in\mathbb C^\times$.
\end{conj}
Suppose the spectral type $\mathbf m$ is of \textsl{Okubo type}, namely, 
\index{Okubo type}
\begin{equation}\label{eq:OkuboT}
 m_{1,1}+\cdots+m_{p,1}=(p-1)\ord\mathbf m.
\end{equation}
Then some shift operators are easily obtained as follows.
By a suitable addition we may assume that the Riemann scheme is
\begin{equation}\label{eq:GRSshift}
  \begin{Bmatrix}
   x = \infty & x=c_1 & \cdots & x=c_p\\
   [\lambda_{0,1}]_{(m_0,1)} & [0]_{(m_{1,1})} & \cdots & [0]_{(m_{p,1})}\\
      [\lambda_{0,2}]_{(m_{0,2})} & [\lambda_{1,2}]_{(m_{1,2})} & \cdots 
      & [\lambda_{p,2}]_{(m_{p,2})}\\
  \vdots & \vdots & \vdots & \vdots\\
      [\lambda_{0,n_0}]_{(m_{0,n_0})} & [\lambda_{1,n_1}]_{(m_{1,n_1})} & \cdots 
      & [\lambda_{p,n_p}]_{(m_{p,n_p})}
  \end{Bmatrix}
\end{equation}
and the corresponding differential equation $Pu=0$ is of the form
\[
  P_{\mathbf m}(\lambda)=\prod_{j=1}^p(x-c_j)^{n-m_{j,1}}\frac{d^n}{dx^n}
 + \sum_{k=0}^{n-1}
 \prod_{j=1}^p
 (x-c_j)^{\max\{k-m_{j,1},0\}}a_{k}(x)\frac{d^k}{dx^k}.
\]
Here $a_k(x)$ is a polynomial of $x$ whose degree is not larger than
$k - \sum_{j=1}^n\max\{k-m_{j,1},0\}$.
Moreover we have
\begin{equation}
 a_0(x)=\prod_{\nu=1}^{n_0}\prod_{i=0}^{m_{0,\nu}-1}(\lambda_{0,\nu}+i).
\end{equation}
Define the differential operators $R_1$ and $R_{\mathbf m}(\lambda)\in W[x]\otimes\mathbb C[\lambda]$ by
\begin{equation}\label{eq:SftOku}
R_1 = \tfrac{d}{dx}
\text{ \ and \ }P_{\mathbf m}(\lambda) = -R_{\mathbf m}(\lambda)R_1+a_0(x).
\end{equation}
Let $P_{\mathbf m}(\lambda')v=0$ be the differential equation with 
the Riemann scheme
\begin{equation}
  \begin{Bmatrix}
   x = \infty & x=c_1 & \cdots & x=c_p\\
   [\lambda_{0,1}+1]_{(m_0,1)} & [0]_{(m_{1,1})} & \cdots & [0]_{(m_{p,1})}\\
      [\lambda_{0,2}+1]_{(m_{0,2})} & [\lambda_{1,2}-1]_{(m_{1,2})} & \cdots 
      & [\lambda_{p,2}-1]_{(m_{p,2})}\\
  \vdots & \vdots & \vdots & \vdots\\
      [\lambda_{0,n_0}+1]_{(m_{0,n_0})}
      & [\lambda_{1,n_1}-1]_{(m_{1,n_1})} & \cdots 
      & [\lambda_{p,n_p}-1]_{(m_{p,n_p})}
  \end{Bmatrix}.
\end{equation}
Then the correspondences $u=R_{\mathbf m}(\lambda)v$ and $v=R_1u$ give 
$W(x)$-homomorphisms between the differential equations.
\begin{prop}\label{prop:shift}
Let\/ $\mathbf m=\{m_{j,\nu}\}_{\substack{0\le j\le p\\1\le \nu\le n_j}}$ 
be a rigid tuple of partitions satisfying \eqref{eq:OkuboT}.
Putting
\begin{equation}\label{eq:shifte}
  \epsilon_{j,\nu} = 
   \begin{cases}
    1 & (j=0,\ 1\le \nu\le n_0),\\
    \delta_{\nu,0}-1 & (1\le j\le p,\ 1\le \nu\le n_j),
   \end{cases}
\end{equation}
we have 
\begin{equation}
c_{\mathbf m}(\epsilon;\lambda) = 
 \prod_{\nu=1}^{n_0}\prod_{i=0}^{m_{0,\nu}-1}
  (\lambda_{0,\nu}+\lambda_{1,1}+\cdots+\lambda_{p,1}+i).
\end{equation}
\end{prop}
\begin{proof}
By suitable additions the proposition follows from the result assuming
$\lambda_{j,1}=0$ for $j=1,\dots,p$, which has been shown.
\end{proof}
\begin{exmp}
\index{hypergeometric equation/function!generalized!shift operator}
The generalized hypergeometric equations with the Riemann schemes
\begin{align}
 \begin{Bmatrix}
   \lambda_{0,1} & \lambda_{1,1} & [\lambda_{2,1}]_{(n-1)}\\
   \vdots & \vdots\\
   \lambda_{0,\nu} & \lambda_{1,\nu_o}  \\
   \vdots & \vdots \\
   \lambda_{0,n} & \lambda_{1,n} & \lambda_{2,2}
 \end{Bmatrix}
 \text{ \ and \ }
 \begin{Bmatrix}
   \lambda_{0,1} & \lambda_{1,1} & [\lambda_{2,1}]_{(n-1)}\\
   \vdots & \vdots\\
   \lambda_{0,\nu} & \lambda_{1,\nu_o} +1 \\
   \vdots & \vdots \\
   \lambda_{0,n} & \lambda_{1,n} & \lambda_{2,2} -1
 \end{Bmatrix},
\end{align}
respectively, 
whose spectral type is $\mathbf m=1^n,1^n,(n-1)1$ 
are isomorphic to each other by the shift operator if and only if
\begin{equation}
  \lambda_{0,\nu}+\lambda_{1,\nu_o}+\lambda_{2,1}\ne 0\quad(\nu=1,\dots,n).
\end{equation}
This statement follows from Proposition~\ref{prop:shift} with suitable additions.

Theorem~\ref{thm:isom} shows that 
in general $P(\lambda)u=0$ with the Riemann scheme $\{\lambda_{\mathbf m}\}$
is $W(x)$-isomorphic to $P(\lambda+\epsilon)v=0$ by the shift operator if 
and only if the values of the function 
$\lambda_{0,\nu}+\lambda_{1,\mu}+\lambda_{2,1}$
satisfy \eqref{eq:non-pos} for $1\le \nu\le n$ and $1\le\mu\le n$.
Here $\epsilon$ is any shift compatible to $\mathbf m$.

The shift operator between
\begin{align}
 \begin{Bmatrix}
   \lambda_{0,1} & \lambda_{1,1} & [\lambda_{2,1}]_{(n-1)}\\
   \lambda_{0,2} & \lambda_{1,2} & \lambda_{2,2}\\
   \vdots & \vdots \\
   \lambda_{0,n} & \lambda_{1,n}
 \end{Bmatrix}
 \text{ \ and \ }
 \begin{Bmatrix}
   \lambda_{0,1} & \lambda_{1,1}+1 & [\lambda_{2,1}]_{(n-1)}\\
   \lambda_{0,2} & \lambda_{1,2}-1  & \lambda_{2,2}\\
   \vdots & \vdots \\
   \lambda_{0,n} & \lambda_{1,n}
 \end{Bmatrix}
\end{align}
is bijective if and only if
\[
  \lambda_{0,\nu}+\lambda_{1,1}+\lambda_{2,1}\ne 0
  \text{\ \ and\ \ }
  \lambda_{0,\nu}+\lambda_{1,2}+\lambda_{2,1}\ne 1
  \text{ \ for \ }
  \nu=1,\dots,n.
\]
Hence if $\lambda_{1,1}=0$ and $\lambda_{1,2}=1$ and 
$\lambda_{0,1}+\lambda_{2,1}=0$, 
the shift operator defines a non-zero endomorphism which is not bijective and 
therefore the monodromy 
of the space of the solutions are decomposed into a direct sum of 
the spaces of solutions of two Fuchsian differential equations.
The other parameters are generic in this case, the decomposition is unique
and the dimension of the smaller space equals 1. When $n=2$ and 
$(c_0,c_1,c_2)=(\infty,1,0)$ and $\lambda_{2,1}$ and $\lambda_{2,2}$ are 
generic, the space equals 
$\mathbb Cx^{\lambda_{2,1}}\oplus\mathbb C x^{\lambda_{2,2}}$
\end{exmp}
\subsection{Polynomial solutions}\label{sec:polyn}
We characterize some polynomial solutions of a differential equation
of Okubo type.
\index{Okubo type!polynomial solution}
\index{polynomial solution}
\begin{prop}\label{prop:polynomial}
Retain the notation in $\S\ref{sec:shift1}$.
Let $P_{\mathbf m}(\lambda)u=0$ be the differential equation with 
the Riemann scheme \eqref{eq:GRSshift}.
Suppose that\/ $\mathbf m$ is rigid and satisfies \eqref{eq:OkuboT}.
Suppose moreover that there exists $j_o$ satisfying $m_{j_o,1}=1$ and
$0\le j_o\le p$. Fix a complex number $C$.
Suppose $\lambda_{0,1}=-C$ and $\lambda_{j,\nu}\notin\mathbb Z$ for $j=0,\dots,p$
and $\nu=2,\dots,n_j$.
Then the equation has a polynomial solution of degree $k$ if and only if\/
$C=k$.

We denote the polynomial solution by $p_\lambda$.
Then $p'_\lambda$ is a polynomial solution of $P_{\mathbf m}(\lambda+\epsilon)v=0$
under the notation \eqref{eq:shifte}.
Moreover
\begin{equation}\label{eq:getpol}
  R_{\mathbf m}(\lambda)\circ R_{\mathbf m}(\lambda+\epsilon)\circ
 \cdots\circ R_{\mathbf m}(\lambda+(k-1)\epsilon)1
\end{equation}
is a non-zero constant multiple of $p_\lambda$ under the notation 
\eqref{eq:SftOku}.
\end{prop}
\begin{proof}
Since $\mathbf m=\bigl(\delta_{1,\nu}\bigr)
  _{\substack{0\le j\le p\\1\le\nu\le n_j}}
 \oplus \bigl(m_{j,\nu}-\delta_{1,\nu}\bigr)
  _{\substack{0\le j\le p\\1\le\nu\le n_j}}$
is a rigid decomposition of $\mathbf m$, we have $P_{\mathbf m}(\lambda) = P_1\p$
with suitable $P_1\in W(x)$ when $C=0$.  
Note that $R_{\mathbf m}(\lambda+\ell\epsilon)$ defines an isomorphism
of the equation $P_{\mathbf m}(\lambda+(\ell+1)\epsilon)u_{k+1}=0$
to the equation  $P_{\mathbf m}(\lambda+\ell\epsilon)u_k=0$
by $u_k=R_{\mathbf m}(\lambda+\ell\epsilon)u_{k+1}$
if $C\ne\ell$, the function \eqref{eq:getpol}
is a polynomial solution of $P_{\mathbf m}(\lambda)u=0$.
The remaining part of the proposition is clear.
\end{proof}
\begin{rem}
We have not used the assumption that $\mathbf m$ 
is rigid in Proposition~\ref{prop:shift} and Proposition~\ref{prop:polynomial} 
and hence the propositions are valid without this assumption.
\end{rem}
\section{Connection problem}\label{sec:C}
\subsection{Connection formula}\label{sec:C1}
For a realizable tuple $\mathbf m\in\mathcal P_{p+1}$,
let $P_{\mathbf m}u=0$ be a universal Fuchsian differential equation 
with the Riemann scheme
\begin{equation}\label{eq:GRSC}
   \begin{Bmatrix}
   x =0 &  {c_1}=1 & \cdots &{c_j} &\cdots& {c_p}=\infty\\
  [\lambda_{0,1}]_{(m_{0,1})} & [\lambda_{1,1}]_{(m_{1,1})}&\cdots
    &[\lambda_{j,1}]_{(m_{j,1})}&\cdots&[\lambda_{p,1}]_{(m_{p,1})}\\
  \vdots & \vdots & \vdots & \vdots &\vdots &\vdots\\
    [\lambda_{0,n_0}]_{(m_{0,n_0})} & [\lambda_{1,n_1}]_{(m_{1,n_1})}&\cdots
    &[\lambda_{j,n_j}]_{(m_{j,n_j})}&\cdots&[\lambda_{p,n_p}]_{(m_{p,n_p})}
  \end{Bmatrix}.
\end{equation}
The singular points of the equation are 
${c_j}$ for $j=0,\dots,p$.
In this subsection we always assume ${c_0}=0$, ${c_1}=1$ and 
${c_p}=\infty$ and $c_j\notin[0,1]$ for $j=2,\dots,p-1$.
We also assume that $\lambda_{j,\nu}$ are generic.
\index{connection coefficient}%
\index{00c@$c(\lambda_{j,\nu}\rightsquigarrow\lambda_{j',\nu'})$}
\begin{defn}[connection coefficients]\label{def:cc}
Suppose $\lambda_{j,\nu}$ are generic under the Fuchs relation.
Let $u_0^{\lambda_{0,\nu_0}}$ and $u_1^{\lambda_{1,\nu_1}}$ be normalized 
local solutions of $P_\mathbf m=0$ at $x=0$ and $x=1$ corresponding 
to the exponents $\lambda_{0,\nu_0}$ and 
$\lambda_{1,\nu_1}$, respectively, so that
$u_0^{\lambda_{0,\nu_0}}\equiv x^{\lambda_{0,\nu_0}}\mod 
x^{\lambda_{0,\nu_0}+1}\mathcal O_0$
and 
$u_1^{\lambda_{1,\nu_1}}\equiv (1-x)^{\lambda_{1,\nu_1}}\mod 
(1-x)^{\lambda_{1,\nu_1}+1}\mathcal O_1$.
Here $1\le\nu_0\le n_0$ and $1\le\nu_1\le n_1$.
If $m_{0,\nu_0}=1$, $u_0^{\lambda_{0,\nu_0}}$ is uniquely determined
and then the analytic continuation of 
$u_0^{\lambda_{0,\nu_0}}$ to $x=1$ along $(0,1)\subset\mathbb R$ 
defines a \textsl{connection coefficient} with respect to 
$u_1^{\lambda_{1,\nu_1}}$, which is denoted 
by $c(0\!:\!\lambda_{0,\nu_0}\!\rightsquigarrow\!1\!:\!\lambda_{1,\nu_1})$ 
or simply by
$c(\lambda_{0,\nu_0}\!\rightsquigarrow\!\lambda_{1,\nu_1})$.
The connection coefficient 
$c(1\!:\!\lambda_{1,\nu_1}\!\rightsquigarrow\!0\!:\!\lambda_{0.\nu_0})$ or
$c(\lambda_{1,\nu_1}\!\rightsquigarrow\!\lambda_{0.\nu_0})$ of
$u_1^{\lambda_{1,\nu_1}}$ with respect to $u_0^{\lambda_{0,\nu_0}}$
are similarly defined if $m_{1,\nu_1}=1$.
\index{connection coefficient}

Moreover we define $c({c_i}:\lambda_{i,\nu_i}\!\rightsquigarrow\!
{c_j}:\lambda_{j,\nu_j})$ by using a suitable linear fractional 
transformation $T$ of $\mathbb C\cup\{\infty\}$ which transforms 
$\{{c_i},{c_j}\}$ to $\{0,1\}$ so that $T({c_\nu})\notin(0,1)$
for $\nu=0,\dots,p$.
If $p=2$, we define the map $T$ so that $T({c_k})=\infty$ for the other
singular point ${c_k}$.
For example if $c_j\notin[0,1]$ for $j=2,\dots,p-1$, we put
$T(x)=\frac{x}{x-1}$ to define 
$c(0:\lambda_{0,\nu_0}\!\rightsquigarrow\!\infty:\lambda_{p,\nu_p})$
or $c(\infty:\lambda_{p,\nu_p}\!\rightsquigarrow\!0:\lambda_{0,\nu_0})$.
\end{defn}
In the definition $u_0^{\lambda_{0,\nu_0}}(x)=x^{\lambda_{0,\nu_0}}
\phi(x)$ with 
analytic function $\phi(x)$ at $0$ which satisfies $\phi(0)=1$ and 
if $\RE\lambda_{1,\nu_1}<\RE\lambda_{1,\nu}$ for $\nu\ne \nu_1$, we have
\begin{equation}
 c(\lambda_{0,\nu_0}\!\rightsquigarrow\!\lambda_{1,\nu_1})
 = \lim_{x\to 1-0}(1-x)^{-\lambda_{1,\nu_1}}u_0^{\lambda_{0,\nu_0}}(x)
 \qquad(x\in[0,1))
\end{equation}
by the analytic continuation.
The connection coefficient  
$c(\lambda_{0.\nu_0}\!\rightsquigarrow\!\lambda_{1,\nu_1})$ meromorphically 
depends on spectral parameters $\lambda_{j,\nu}$.
It also holomorphically depends on accessory parameters $g_i$ 
and singular points $\frac1{c_j}$ $(j=2,\dots,p-1)$ in a neighborhood of
given values of parameters.

The main purpose in this subsection is to get the explicit expression of 
the connection coefficients in terms of gamma functions 
when $\mathbf m$ is rigid and $m_{0,\nu}=m_{1,\nu'}=1$.

Fist we prove the following key lemma which describes the effect of a middle 
convolution on connection coefficients.
\begin{lem}\label{lem:conn}
Using the integral transformation \eqref{eq:fracdif}, we put
\begin{align}
(T_{a,b}^\mu u)(x)&:=x^{-a-\mu}(1-x)^{-b-\mu}I_0^\mu x^{a}(1-x)^{b}u(x),\\
(S_{a,b}^\mu u)(x)&:=x^{-a-\mu}I_0^\mu x^{a}(1-x)^bu(x)
\end{align}
for a continuous function $u(x)$ on $[0,1]$.
Suppose\/ $\RE a \ge 0$ and\/ $\RE \mu > 0$.
Under the condition\/ $\RE b+\RE\mu<0$ or\/ $\RE b+\RE\mu>0$, 
$(T_{a,b}^\mu u)(x)$ or $S_{a,b}^\mu (u)(x)$ 
defines a continuous function on\/ $[0,1]$, respectively, and
we have
\begin{align}
T_{a,b}^\mu(u)(0)&
 =S_{a,b}^\mu(u)(0)=\frac{\Gamma(a+1)}{\Gamma(a+\mu+1)}u(0),
\allowdisplaybreaks\\
\frac{T_{a,b}^\mu(u)(1)}{T_{a,b}^\mu(u)(0)}
 &=\frac{u(1)}{u(0)}C_{a,b}^\mu,\quad
 C_{a,b}^\mu:=\frac{\Gamma(a+\mu+1)\Gamma(-\mu-b)}
{\Gamma(a+1)\Gamma(-b)},\label{eq:CR1}
\allowdisplaybreaks\\
\frac{S_{a,b}^\mu(u)(1)}{S_{a,b}^\mu(u)(0)}
&=\frac1{u(0)}\frac
 {\Gamma(a+\mu+1)}{\Gamma(\mu)\Gamma(a+1)}
\int_0^1t^a(1-t)^{b+\mu-1}u(t)dt.
\end{align}
\end{lem}
\begin{proof}
Suppose $\RE a\ge 0$ and $0<\RE\mu<-\RE b$.  Then
\begin{align*}
\Gamma(\mu)&T_{a,b}^\mu(u)(x)\\
&={x^{-a-\mu}(1-x)^{-b-\mu}}
  \int_0^x t^a(1-t)^b (x-t)^{\mu-1}u(t)dt&\hspace{-1.2cm}(t=xs_1,\ 0\le x<1)
 \allowdisplaybreaks\\
&={(1-x)^{-b-\mu}}
  \int_0^1 s_1^{a}(1-s_1)^{\mu-1}(1-xs_1)^{b}u(xs_1)ds_1
 \allowdisplaybreaks\\
&=\int_0^1 s_1^{a}\Bigl(\frac{1-s_1}{1-x}\Bigr)^\mu
  \Bigl(\frac{1-xs_1}{1-x}\Bigr)^{b}u(xs_1)\frac{ds}{1-s_1}
 \allowdisplaybreaks\\
&=\int_0^1 (1-s_2)^{a}\Bigl(\frac{s_2}{1-x}\Bigr)^\mu
  \Bigl(1+\frac {xs_2}{1-x}\Bigr)^b
  u(x-xs_2)\frac{ds_2}{s_2}&(s_1=1-s_2)\allowdisplaybreaks\\
&=\int_0^\frac1{1-x} \bigl(1-s(1-x)\bigr)^{a}s^\mu
  (1+xs)^{b}u\bigl(x-x(1-x)s\bigr)\frac{ds}{s}
  &(s_2=(1-x)s).
\end{align*}
Since
\[
 \left|s_1^{a}(1-s_1)^{\mu-1}(1-xs_1)^{b}u(xs_1)\right|
 \le \max\{(1-s_1)^{\RE\mu-1},1\}3^{-\RE b}\max_{0\le t\le 1}|u(t)|
\]
for $0\le s_1< 1$ and $0\le x\le \frac23$,
$T^\mu_{a,b}(u)(x)$ is continuous for $x\in[0,\tfrac23)$.
We have
\begin{align*}
\left|\bigl(1-s(1-x)\bigr)^{a}s^{\mu-1}
  (1+xs)^{b}u\bigl(x-x(1-x)s)\bigr)\right|
\le s^{\RE\mu-1}(1+\tfrac s2)^{\RE b}\max_{0\le t\le 1}|u(t)|
\end{align*}
for $\frac12\le x\le 1$ and $0<s\le \frac1{1-x}$
and therefore $T_{a,b}^\mu(u)(x)$ is continuous for
$x\in(\frac12,1]$.
Hence $T_{a,b}^\mu(x)$ defines a continuous function on $[0,1]$ and
\begin{align*}
T_{a,b}^\mu(u)(0)&
 =\frac{1}{\Gamma(\mu)}\int_0^1(1-s_2)^{a}s_2^\mu u(0)\frac{ds_2}{s_2}
 =\frac{\Gamma(a+1)}{\Gamma(a+\mu+1)}u(0),\\
T_{a,b}^\mu(u)(1)&
 =\frac{1}{\Gamma(\mu)}\int_0^\infty s^\mu (1+s)^{b}u(1)\frac{ds}s
\intertext{$(t=\frac s{1+s}=1-\frac1{1+s},\ \frac1{1+s}=1-t,\ 
1+s=\frac1{1-t},\ s=\frac1{1-t}-1=\frac{t}{1-t},\ 
\frac{ds}{dt}=-\frac1{(1-t)^2})$}
 &=\frac{1}{\Gamma(\mu)}\int_0^1\Bigl(\frac{t}{1-t}\Bigr)^{\mu-1}(1-t)^{-b-2}u(1)dt
  =\frac{\Gamma(-\mu-b)}{\Gamma(-b)}u(1).
\end{align*}
The claims for $S_{a,b}^\mu$ are clear from
\[
 \Gamma(\mu)S^\mu_{a,b}(u)(x)=
 \int_0^1 s_1^a(1-s_1)^{\mu-1}(1-xs_1)^b u(xs_1)ds_1.
\]\\[-1cm]
\end{proof}

This lemma is useful for the middle convolution $mc_\mu$ not
only when it gives a reduction but also when it doesn't change the spectral 
type.
\begin{exmp}\index{Jordan-Pochhammer}
Applying Lemma~\ref{lem:conn} to the solution
\[
  u_0^{\lambda_0+\mu}(x)=\int_0^x t^{\lambda_0}(1-t)^{\lambda_1}
  \biggl(\prod_{j=2}^{p-1}
  \Bigl(1-\frac t{c_j}\Bigr)^{\lambda_j}\biggr)(x-t)^{\mu-1}dt
\]
of the Jordan-Pochhammer equation (cf.\ Example~\ref{ex:midconv} iii)) 
with the Riemann scheme
\[
   \begin{Bmatrix}
   x =0 & {c_1}=1 & \cdots &{c_j} &\cdots& {c_p}=\infty\\
  [0]_{(p-1)} & [0]_{(p-1)}&\cdots
    &[0]_{(p-1)}&\cdots&[1-\mu]_{(p-1)}\\
  \lambda_0+\mu&\lambda_1+\mu&\cdots&\lambda_j+\mu&\cdots&
  -\sum_{\nu=0}^{p-1}\lambda_\nu-\mu
  \end{Bmatrix},
\]
we have
\begin{align*}
c(0\!:\!\lambda_0+\mu\!\rightsquigarrow\!1\!:\!\lambda_1+\mu)
  &=\frac{\Gamma(\lambda_0+\mu+1)\Gamma(-\lambda_1-\mu)}
   {\Gamma(\lambda_0+1)\Gamma(-\lambda_1)}
   \prod_{j=2}^{p-1}\Bigl(1-\frac1{c_j}\Bigr)^{\lambda_j},
 \allowdisplaybreaks\\
c(0\!:\!\lambda_0+\mu\!\rightsquigarrow\!1:\!0)&=
 \frac{\Gamma(\lambda_0+\mu+1)}{\Gamma(\mu)\Gamma(\lambda_0+1)}\int_0^1
 t^{\lambda_0}(1-t)^{\lambda_1+\mu-1}\prod_{j=1}^{p-1}
  \Bigl(1-\frac t{c_j}\Bigr)^{\lambda_j}dt.
\end{align*}
Moreover the equation $Pu=0$ with
\[
 P:=\RAd(\p^{-\mu'})\RAd(x^{\lambda'})
 \RAd(\p^{-\mu})\RAd(x^{\lambda_0}
 (1-x)^{\lambda_1})\p
\]
is satisfied by the generalized hypergeometric
function ${}_3F_2$ with the Riemann scheme
\[
 \begin{Bmatrix}
  x = 0 &1& \infty\\
  0 &[0]_{(2)}&1-\mu'\\
  \lambda'+\mu'& &1-\lambda'-\mu-\mu'\\
  \lambda_0+\lambda'+\mu+\mu'&\lambda_1+\mu+\mu'&
  -\lambda_0-\lambda_1-\lambda'-\mu-\mu'
 \end{Bmatrix}
\]
corresponding to $111,21,111$ and therefore
\begin{align*}
 &c(\lambda_0+\lambda'+\mu+\mu'\!\rightsquigarrow\!\lambda_1+\mu+\mu')=
 C_{\lambda_0,\lambda_1}^\mu\cdot 
  C_{\lambda_0+\lambda'+\mu,\lambda_1+\mu}^{\mu'}\\
 &\quad=
 \frac{\Gamma(\lambda_0+\mu+1)\Gamma(-\lambda_1-\mu)}
   {\Gamma(\lambda_0+1)\Gamma(-\lambda_1)}\cdot
  \frac{\Gamma(\lambda_0+\lambda'+\mu+\mu'+1)\Gamma(-\lambda_1-\mu-\mu')}
   {\Gamma(\lambda_0+\lambda'+\mu+1)\Gamma(-\lambda_1-\mu)}\\
 &\quad=
  \frac{\Gamma(\lambda_0+\mu+1)
    \Gamma(\lambda_0+\lambda'+\mu+\mu'+1)\Gamma(-\lambda_1-\mu-\mu')}
   {\Gamma(\lambda_0+1)\Gamma(-\lambda_1)\Gamma(\lambda_0+\lambda'+\mu+1)}.
\end{align*}
\end{exmp}
We further examine the connection coefficient.

In general, putting $c_0=0$ and $c_1=1$ and
$\lambda_1=\sum_{k=0}^p\lambda_{k,1}-1$, we have 
{
\begin{align*}
&\begin{Bmatrix}
x={c_j}\quad(j=0,\dots,p-1)&\infty \\
[\lambda_{j,\nu}-(\delta_{j,0}+\delta_{j,1})
  \lambda_{j,n_j}]_{(m_{j,\nu})}
 &[\lambda_{p,\nu}+\lambda_{0,n_0}+\lambda_{1,n_1}]_{(m_{0,\nu})} 
\end{Bmatrix}\allowdisplaybreaks\\
&\xrightarrow{x^{\lambda_{0,n_0}}(1-x)^{\lambda_{1,n_1}}}{}
\begin{Bmatrix}
x= {c_j} &\infty\\
 [\lambda_{j,\nu}]_{(m_{j,\nu})}
 &[\lambda_{p,\nu}]_{(m_{p,\nu})} 
\end{Bmatrix}\allowdisplaybreaks\\
\\
&\xrightarrow{x^{-\lambda_{0,1}}\prod_{j=1}^{p-1}
  (1-c_j^{-1}x)^{-\lambda_{j,1}}}{}
\begin{Bmatrix}
 [0]_{(m_{j,1})}
  & [\lambda_{p,1}+\sum_{k=0}^{p-1} \lambda_{k,1}]_{(m_{p,1})} 
\\
 [\lambda_{j,\nu}-\lambda_{j,1}]_{(m_{j,\nu})}
  &[\lambda_{p,\nu}+\sum_{k=0}^{p-1} \lambda_{k,1}]_{(m_{p,\nu})} 
\end{Bmatrix}\allowdisplaybreaks\\
&\xrightarrow{\p^{1-\sum_{k=0}^p \lambda_{k,1}}}{}
\begin{Bmatrix}
 [0]_{(m_{j,1}-d)}
  &[\lambda_{p,1}+\sum_{k=0}^{p-1} \lambda_{k,1}-2\lambda_1]_{(m_{p,1}-d)}
\\
 [\lambda_{j,\nu}-\lambda_{j,1}+\lambda_1]_{(m_{j,\nu})}
 & [\lambda_{p,\nu}+\sum_{k=0}^{p-1} \lambda_{k,1}-\lambda_1]_{(m_{p,\nu})} 
\end{Bmatrix}\\
&\qquad(d=\sum_{k=0}^pm_{k,1}-(p-1)n)
\allowdisplaybreaks\\
&\xrightarrow{x^{\lambda_{0,1}}
  \prod_{j=1}^{p-1} (1-c_j^{-1}x)^{\lambda_{j,1}}}{}
\begin{Bmatrix}
x=\frac1{c_j} & \infty\\
  [\lambda_{j,1}]_{(m_{j,1}-d)}
   &[\lambda_{p,1}-2\lambda_1]_{(m_{p,1}-d)} 
\\
 [\lambda_{j,\nu}+\lambda_1]_{(m_{j,\nu})}
  &[\lambda_{p,\nu}-\lambda_1]_{(m_{p,\nu})}
\end{Bmatrix},\allowdisplaybreaks\\
&C_{\lambda_{0,n_1}-\lambda_{0,1},\lambda_{1,n_1}-\lambda_{1,1}}
 ^{\lambda_1}=
\frac{\Gamma(\lambda_{0,n_0}+\lambda_1-\lambda_{0,1}+1)
  \Gamma(\lambda_{1,1}-\lambda_{1,n_1}-\lambda_1)}
  {\Gamma(\lambda_{0,n_0}-\lambda_{0,1}+1)
   \Gamma(\lambda_{1,1}-\lambda_{1,n_1})
  }.
\end{align*}}

In general, the following theorem is a direct consequence of 
Definition~\ref{def:pell} and Lemma~\ref{lem:conn}. 
\begin{thm}\label{thm:GC}
Put $c_0=\infty$, $c_1=1$ and $c_j\in\mathbb C\setminus\{0\}$ 
for $j=3,\dots,p-1$.
By the transformation 
\[
\begin{split}
 &\RAd\Bigl(x^{\lambda_{0,1}}\!
  \prod_{j=1}^{p-1}\bigl(1-\frac x{c_j}\bigr)^{\lambda_{j,1}}\Bigr)
 \circ
 \RAd\Bigl(\p^{1-\sum_{k=0}^p\lambda_{k,1}}\Bigr)\circ
 \RAd\Bigl(x^{-\lambda_{0,1}}\!\prod_{j=1}^{p-1}
  \bigl(1-\frac x{c_j}\bigr)^{-\lambda_{j,1}}\Bigr)
\end{split}
\]
the Riemann scheme of a Fuchsian ordinary differential equation
and its connection coefficient change as follows:
{
\begin{align*}
&\left\{\lambda_{\bf m}\right\}
=\left\{[\lambda_{j,\nu}]_{(m_{j,\nu})}\right\}
 _{\substack{0\le j\le p\\1\le\nu\le n_j}}
=\begin{Bmatrix}
x={c_j}&\infty\\
[\lambda_{j,1}]_{(m_{j,1})} & [\lambda_{p,1}]_{(m_{p,1})}\\
[\lambda_{j,\nu}]_{(m_{j,\nu})} &
  [\lambda_{p,\nu}]_{(m_{p,\nu})}
\end{Bmatrix}
\allowdisplaybreaks\\&
\ \mapsto
\{\lambda'_{\mathbf m'}\}
=
\left\{[\lambda'_{j,\nu}]_{(m'_{j,\nu})}\right\}
 _{\substack{0\le j\le p\\1\le\nu\le n_j}}
\\&\qquad
=\begin{Bmatrix}
x={c_j} & \infty\\
  [\lambda_{j,1}]_{(m_{j,1}-d)}
 &[\lambda_{p,1}-2\sum_{k=0}^p\lambda_{k,1}+2]_{(m_{p,1}-d)} 
\\
 [\lambda_{j,\nu}+\sum_{k=0}^p\lambda_{k,1}-1]_{(m_{j,\nu})}
  &[\lambda_{p,\nu}-\sum_{k=0}^p\lambda_{k,1}+1]_{(m_{p,\nu})} 
\end{Bmatrix}
\intertext{with}
&\qquad d=m_{0,1}+\cdots+m_{p,1}-(p-1)\ord\mathbf m,\\
&\qquad m_{j,\nu}'=m_{j,\nu}-d\delta_{\nu,1}\quad(j=0,\dots,p,\ \nu=1,\dots,n_j),\\
&\qquad \lambda'_{j,1}=\lambda_{j,1}\quad(j=0,\dots,p-1),\ 
 \lambda'_{p,1}=-2\lambda_{0,1}-\cdots-2\lambda_{p-1,1}-\lambda_{p,1}+2,\\
&\qquad \lambda'_{j,\nu}=\lambda_{j,\nu}+\lambda_{0,1}+\lambda_{1,1}
 +\cdots+\lambda_{p,1}-1\quad(j=0,\dots,p-1,\ \nu=2,\dots,n_j),\\
&\qquad \lambda'_{p,\nu}=\lambda_{p,\nu}-\lambda_{0,1}-\cdots-\lambda_{p,1}+1
\end{align*}}
and if $m_{0,n_0}=1$ and $n_0>1$ and $n_1>1$, then
\begin{equation}\label{eq:cid}
\frac{c'(\lambda'_{0,n_0}\!\rightsquigarrow\!\lambda'_{1,n_1})}
  {\Gamma(\lambda'_{0,n_0}-\lambda'_{0,1}+1)
  \Gamma(\lambda'_{1,1}-\lambda'_{1,n_1})}
=
\frac{c(\lambda_{0,n_0}\!\rightsquigarrow\!\lambda_{1,n_1})}
  {\Gamma(\lambda_{0,n_0}-\lambda_{0,1}+1)
  \Gamma(\lambda_{1,1}-\lambda_{1,n_1})}.
\end{equation}
\end{thm}

Applying the successive reduction by $\p_{max}$ to the above theorem, 
we obtain the following theorem.
\begin{thm}\label{thm:conG}
Suppose that a tuple\/ $\mathbf m\in\mathcal P$ is irreducibly realizable and 
$m_{0,n_0}=m_{1,n_1}=1$ in the Riemann scheme \eqref{eq:GRSC}.
Then the connection coefficient satisfies
\begin{align*}
 &\frac{c(\lambda_{0,n_0}\!\rightsquigarrow\!\lambda_{1,n_1})}
 {\bar c\bigl(\lambda(K)_{0,n_0}\!\rightsquigarrow\!\lambda(K)_{1,n_1}\bigr)}\\
 &\quad= \prod_{k=0}^{K-1}
 \frac{\Gamma\bigl(\lambda(k)_{0,n_0}-\lambda(k)_{0,\ell(k)_0}+1\bigr)
  \cdot\Gamma\bigl(\lambda(k)_{1,\ell(k)_1}-\lambda(k)_{1,n_1}\bigr)}
 {\Gamma\bigl(\lambda(k+1)_{0,n_0}-\lambda(k+1)_{0,\ell(k)_0}+1\bigr)
  \cdot\Gamma\bigl(\lambda(k+1)_{1,\ell(k)_1}-\lambda(k+1)_{1,n_1}\bigr)}
\end{align*}
under the notation in Definitions~\ref{def:redGRS}.
Here 
$\bar c\bigl(\lambda(K)_{0,n_0}\!\rightsquigarrow\!\lambda(K)_{1,n_1}\bigr)$
is a corresponding connection coefficient for the 
equation $(\p_{max}^K P_{\mathbf m})v=0$ with the fundamental spectral type
$f\mathbf m$.
We note that
\begin{equation}\label{eq:cdeqn}
\begin{split}
 &\bigl(\lambda(k+1)_{0,n_0}-\lambda(k+1)_{0,\ell(k)_0}+1\bigr)
  +\bigl(\lambda(k+1)_{1,\ell(k)_1}-\lambda(k+1)_{1,n_1}\bigr)\\
 &\quad=\bigl(\lambda(k)_{0,n_0}-\lambda(k)_{0,\ell(k)_0}+1\bigr)+
 \bigl(\lambda(k)_{1,\ell(k)_1}-\lambda(k)_{1,n_1}\bigr)
\end{split}
\end{equation}
for $k=0,\dots,K-1$.
\end{thm}
When $\mathbf m$ is rigid in the theorem above, we note that 
$\bar c(\lambda_{0,n_0}(K)\!\rightsquigarrow\!\lambda_{1,n_1}(K))=1$
and we have the following more explicit result.
\begin{thm}\label{thm:c}
Let\/ $\mathbf m\in\mathcal P$ be a rigid tuple.
Assume $m_{0,n_0}=m_{1,n_1}=1$, $n_0>1$ and $n_1>1$
in the Riemann scheme \eqref{eq:GRSC}.
Then
\index{000lambda@$\arrowvert$\textbraceleft$\lambda_{\mathbf m}$\textbraceright$\arrowvert$}%
\begin{gather}\label{eq:connection}%
\begin{split}
 c(\lambda_{0,n_0}\!\rightsquigarrow\!\lambda_{1,n_1})
 =\frac
  {\displaystyle\prod_{\nu=1}^{n_0-1} 
    \Gamma\bigl(\lambda_{0,n_0}-\lambda_{0,\nu}+1\bigr)
   \cdot\prod_{\nu=1}^{n_1-1}
    \Gamma\bigl(\lambda_{1,\nu}-\lambda_{1,n_1}\bigr)
  }
  {\displaystyle\prod_{\substack{\mathbf m'\oplus\mathbf m''=\mathbf m\\
                    m'_{0,n_0}=m''_{1,n_1}=1}}
    \Gamma\bigl(|\{\lambda_{\mathbf m'}\}|\bigr)
    \cdot\prod_{j=2}^{p-1}\Bigl(1-\frac1{c_j}\Bigr)^{-\lambda(K)_{j,\ell(K)_j}}
  },\\
\end{split}\allowdisplaybreaks\\
  \sum_{\substack{\mathbf m'\oplus\mathbf m''=\mathbf m\\
    m'_{0,n_0}=m''_{1,n_1}=1}}
   \!\!\!\!\!\! m'_{j,\nu}
  = (n_1-1)m_{j,\nu}-\delta_{j,0}(1-n_0\delta_{\nu,_{n_0}})
   +\delta_{j,1}(1-n_1\delta_{\nu,_{n_1}})\label{eq:concob}\\[-.5cm]
  \hspace{4.5cm}(1\le\nu\le n_j,\ 0\le j\le p)\notag
\end{gather}
under the notation in Definitions~\ref{def:FRLM} and \ref{def:redGRS}.
\end{thm}
\begin{proof}
We may assume $\mathbf m$ is monotone and $\ord\mathbf m>1$.

We will prove this theorem by the induction on $\ord\mathbf m$.
Suppose
\begin{equation}\label{eq:DecC}
\mathbf m=\mathbf m'\oplus\mathbf m''\text{ \ with \ }
m'_{0,n_0}=m''_{1,n_1}=1.
\end{equation}

If $\p_{\mathbf 1}\mathbf m'$ is not well-defined, then
\begin{equation}\label{eq:DecC0}
  \ord\mathbf m'=1\text{ \ and \ }
  m'_{j,1}=1\text{ \ for \ }j=1,2,\dots,p
\end{equation}
and $1+m_{1,1}+\cdots+m_{p,1}-(p-1)\ord\mathbf m=1$
because $\idx(\mathbf m,\mathbf m')=1$ 
and therefore
\begin{equation}
d_{\mathbf 1}(\mathbf m)=m_{0,1}.
\end{equation}
If $\p_{\mathbf 1}\mathbf m''$ is not well-defined, 
\begin{equation}\label{eq:DecC1}
\begin{split}
  \ord\mathbf m''&=1\text{ \ and \ }
    m''_{j,1}=1\text{ \ for \ }j=0,2,\dots,p,\\
  d_{\mathbf 1}(\mathbf m)&=m_{1,1}.
\end{split}
\end{equation}

Hence if $d_{\mathbf 1}(\mathbf m)<m_{0,1}$ and 
$d_{\mathbf 1}(\mathbf m)<m_{1,1}$,
$\p_{\mathbf 1}\mathbf m'$ and $\p_{\mathbf 1}\mathbf m''$
are always well-defined and 
$\p_{\mathbf 1}\mathbf m=\p_{\mathbf 1}\mathbf m'\oplus
\p_{\mathbf 1}\mathbf m''$ and the direct decompositions
\eqref{eq:DecC} of $\mathbf m$ correspond to those of 
$\p_{\mathbf 1}\mathbf m$ 
and therefore Theorem~\ref{thm:GC} shows \eqref{eq:connection}
by the induction because we may assume 
$d_{\mathbf 1}(\mathbf m)>0$.
In fact, it follows from \eqref{eq:midinv} that 
the gamma factors in the denominator of 
the fraction in the right hand side of 
\eqref{eq:connection} don't change by the reduction 
and the change of the numerator just corresponds to the formula in 
Theorem~\ref{thm:GC}.

If $d_{\mathbf 1}(\mathbf m)=m_{0,1}$, there exists the direct
decomposition \eqref{eq:DecC} with \eqref{eq:DecC0}
which doesn't correspond to a direct decomposition of 
$\p_{\mathbf 1}\mathbf m$ but corresponds to the term
$\Gamma(|\{\lambda_{\mathbf m'}\}|)
=\Gamma(\lambda_{0,n_1}+\lambda_{1,1}+\cdots+\lambda_{p,1})
=\Gamma(\lambda'_{0,n_1}-\lambda'_{0,1}+1)$ in \eqref{eq:cid}.
Similarly if $d_{\mathbf 1}(\mathbf m)=m_{1,1}$,
there exists the direct
decomposition \eqref{eq:DecC} with \eqref{eq:DecC1}
and it corresponds to the term
$\Gamma(|\{\lambda_{\mathbf m'}\}|) = \Gamma(1-|\{\lambda_{\mathbf m''}\}|)
=\Gamma(1-\lambda_{0,1}-\lambda_{1,n_1}-\lambda_{2,1}-\cdots-\lambda_{p,1})
=\Gamma(\lambda'_{1,1}-\lambda'_{1,n_1})$
(cf.~\eqref{eq:sum1}).
Thus Theorem~\ref{thm:GC} assures \eqref{eq:connection} by the induction
on $\ord\mathbf m$.

Note that the above proof with \eqref{eq:cdeqn} shows \eqref{eq:csum}.
Hence
\begin{align*}
 \sum_{\substack{\mathbf m'\oplus\mathbf m''=\mathbf m\\
    m'_{0,n_0}=m''_{1,n_1}=1}}|\{\lambda_{\mathbf m}\}|
 &=\sum_{\nu=1}^{n_0-1}(\lambda_{0,n_0}-\lambda_{0,\nu}+1)
+\sum_{\nu=1}^{n_1-1}(\lambda_{1,\nu}-\lambda_{1,n_1})\\[-8pt]
&
=(n_0-1)+(n_0-1)\lambda_{0,n_0}
 -\sum_{\nu=1}^{n_0-1}
  \lambda_{0,\nu}+\sum_{\nu=1}^{n_1-1}\lambda_{1,\nu}\\
&\quad +(n_1-1)\Bigl(\sum_{j=0}^p\sum_{\nu=1}^{n_j-\delta_{j,1}}
  m_{j,\nu}\lambda_{j,\nu}-n+1\Bigr)\allowdisplaybreaks\\
&=
 (n_0+n_1-2)\lambda_{0,n_0}
 +
 \sum_{\nu=1}^{n_0-1}\bigl((n_1-1)m_{0,\nu}-1\bigr)\lambda_{0,\nu}\\
 &\quad+\sum_{\nu=1}^{n_1-1}\bigl((n_1-1)m_{1,\nu}+1\bigr)\lambda_{1,\nu}
 +\sum_{j=2}^p\sum_{\nu=1}^{n_2}(n_1-1)m_{j,\nu}\lambda_{j,\nu}\\
 &\quad+(n_0+n_1-2)-(n_1-1)\ord\mathbf m.
\end{align*}
The left hand side of the above first equation
and the right hand side of the above last equation don't
contain the term $\lambda_{1,n_1}$ and therefore the coefficients of
$\lambda_{j,\nu}$ in the both sides are equal, 
which implies \eqref{eq:concob}.
\end{proof}
\begin{cor}\label{cor:C}
Retain the notation in\/ {\rm Theorem~\ref{thm:c}.} We have
\begin{align}
 \#\{\mathbf m'\,;\,\mathbf m'\oplus\mathbf m''&=\mathbf m\text{ \ with \ }
m'_{0,n_0}=m''_{1,n_1}=1\} = n_0+n_1-2,\label{eq:numdec}
\allowdisplaybreaks\\
\sum_{\substack{\mathbf m'\oplus\mathbf m''=\mathbf m\\
    m'_{0,n_0}=m''_{1,n_1}=1}}
   \!\!\!\!\!\! \ord\mathbf m'&=(n_1-1)\ord\mathbf m,\label{eq:ordsum}
\allowdisplaybreaks\\
\sum_{\substack{\mathbf m'\oplus\mathbf m''=\mathbf m\\
    m'_{0,n_0}=m''_{1,n_1}=1}}
 |\{\lambda_m'\}|&= \sum_{\nu=1}^{n_0-1}(\lambda_{0,n_0} - \lambda_{0,\nu} + 1)
+\sum_{\nu=1}^{n_1-1}(\lambda_{1,\nu}-\lambda_{1,n_1}).
\label{eq:csum}
\end{align}
Let $c(\lambda_{0,n_0}+t\!\rightsquigarrow\!\lambda_{1,n_1}-t)$ be
the connection coefficient for the Riemann scheme
$\bigl\{[\lambda_{j,\nu}+t(\delta_{j,0}\delta_{\nu,n_0}
 -\delta_{j,1}\delta_{\nu,n_1})]_{(m_{j,\nu})}\bigr\}$. Then
\begin{equation}\label{eq:clim}
 \lim_{t\to+\infty}
 c(0\!:\!\lambda_{0,n_0}+t\rightsquigarrow1\!:\!\lambda_{1,n_1}-t)=
 \prod_{j=2}^{p-1}\bigl(1-c_j\bigr)^{\lambda(K)_{j,\ell(K)_j}}.
\end{equation}
Under the notation in\/ {\rm Theorem~\ref{thm:irrKac}}
\begin{equation}\label{eq:dncKac}
\begin{split}
 &\{\mathbf m'\,;\,\mathbf m'\oplus\mathbf m''=\mathbf m\text{ \ with \ }
 m'_{0,n_0}=m''_{1,n_1}=1\}\\
 &=\{\mathbf m'\in\mathcal P\,;\,m'_{0,n_0}=1,\ m'_{1,n_1}=0,\ 
  \alpha_{\mathbf m'}\text{ or }\alpha_{\mathbf m-\mathbf m'}
  \in\Delta(\mathbf m)\}.
\end{split}
\end{equation}
\end{cor}
\begin{proof}
We have \eqref{eq:csum} in the proof of Theorem~\ref{thm:GC}
and then  Stirling's formula and \eqref{eq:csum} prove \eqref{eq:clim}. 
Putting $(j,\nu)=(0,n_0)$ in \eqref{eq:concob} and
considering the sum $\sum_\nu$ for \eqref{eq:concob} with $j=1$, we have
\eqref{eq:numdec} and \eqref{eq:ordsum}, respectively.

Comparing the proof of Theorem~\ref{thm:c} with that of 
Theorem~\ref{thm:irrKac}, we have \eqref{eq:dncKac}.
Proposition~\ref{prop:wm} also proves \eqref{eq:dncKac}.
\end{proof}
\begin{rem}\label{rem:conn}
{\rm i)\ } 
When we calculate a connection coefficient for a given rigid 
partition $\mathbf m$ by \eqref{eq:connection}, it is necessary
to get all the direct decompositions $\mathbf m=\mathbf m'\oplus
\mathbf m''$ satisfying $m'_{0,n_0}=m''_{1,n_1}=1$.  In this case
the equality \eqref{eq:numdec} is useful because we know that the number of 
such decompositions equals $n_0+n_1-2$, namely, the number of gamma 
functions appearing in the numerator equals
that appearing in the denominator in \eqref{eq:connection}.

{\rm ii) }
A direct decomposition $\mathbf m=\mathbf m'\oplus\mathbf m''$ for
a rigid tuple $\mathbf m$ means that 
$\{\alpha_{\mathbf m'},\alpha_{\mathbf m''}\}$ is a fundamental system
of a root system of type $A_2$ in $\mathbb R\alpha_{\mathbf m'}
+\mathbb R\alpha_{\mathbf m''}$ such that $\alpha_{\mathbf m}=
\alpha_{\mathbf m'}+\alpha_{\mathbf m''}$ and\\[2pt]
\qquad$\begin{cases}
 (\alpha_{\mathbf m'}|\alpha_{\mathbf m'})
 =(\alpha_{\mathbf m''}|\alpha_{\mathbf m''})=2,\\
  (\alpha_{\mathbf m'}|\alpha_{\mathbf m''})=-1.
 \end{cases}
$

\vspace{-1.4cm}\hspace{7cm}
\begin{xy}{\ar (10,0)*+!L{\alpha_{\mathbf m'}}},
(.5,\halfrootthree): {\ar(0,0);(10,0)*+!D{\alpha_{\mathbf m}}},
(0,0),(.5,\halfrootthree): {\ar(0,0);(10,0)*+!D{\alpha_{\mathbf m''}}},
\end{xy}

{\rm iii)\ } In view of Definition~\ref{def:FRLM}, 
the condition $\mathbf m=\mathbf m'\oplus\mathbf m''$ 
in \eqref{eq:connection} means
\begin{equation}\label{eq:sum1}
 \bigl|\{\lambda_{\mathbf m'}\}\bigr|+
 \bigl|\{\lambda_{\mathbf m''}\}\bigr|=1.
\end{equation}
Hence we have
\begin{equation}\label{eq:cprod}
\begin{split}
 &c(\lambda_{0,n_0}\!\rightsquigarrow\!\lambda_{1,n_1})\cdot
 c(\lambda_{1,n_1}\!\rightsquigarrow\!\lambda_{0,n_0})\\
 &\qquad=\frac
  {\displaystyle\prod_{\substack{\mathbf m'\oplus\mathbf m''=\mathbf m\\
                    m'_{0,n_0}=m''_{1,n_1}=1}}
    \sin\bigl(|\{\lambda_{\mathbf m'}\}|\pi\bigr)
  }
  {\displaystyle\prod_{\nu=1}^{n_0-1} 
    \sin\bigl(\lambda_{0,\nu}-\lambda_{1,\nu}\bigr)\pi
   \cdot\prod_{\nu=1}^{n_1-1}
    \sin\bigl(\lambda_{1,\nu}-\lambda_{1,n_1}\bigr)\pi
  }.
\end{split}
\end{equation}

{\rm iv)\ } By the aid of a computer, the author obtained the table of the 
concrete connection coefficients \eqref{eq:connection} for the rigid triplets
$\mathbf m$ satisfying $\ord\mathbf m\le 40$ 
together with checking \eqref{eq:concob}, 
which contains 4,111,704 independent cases
(cf.~\S\ref{sec:okubo}).
\end{rem}
\subsection{An estimate for large exponents}\label{sec:estimate}
The Gauss hypergeometric series
\[
  F(\alpha,\beta,\gamma;x) := \sum_{k=0}^\infty\frac{\alpha(\alpha+1)
   \cdots(\alpha+k-1)\cdot
   \beta(\beta+1)\cdots(\beta+k-1)}
   {\gamma(\gamma+1)\cdots(\gamma+k-1)\cdot k!}x^k
\]
uniformly and absolutely converges for 
\begin{equation}
 x\in \overline D:=\{x\in\mathbb C\,;\,|x|\le 1\}
\end{equation}
if $\RE\gamma>\RE(\alpha+\beta)$
and defines a continuous function on $\overline D$.  
The continuous function $F(\alpha,\beta,\gamma+n;x)$
on $\overline D$ uniformly converges to the constant function 
$1$ when $n\to+\infty$, which obviously implies 
\begin{equation}\label{eq:gammainf}
\lim_{n\to\infty} F(\alpha,\beta,\gamma+n;1)=1
\end{equation}
and proves Gauss's summation formula \eqref{eq:Gausssum} by using 
the recurrence relation 
\begin{equation}\label{eq:GCratio}
 \frac{F(\alpha,\beta,\gamma;1)}{F(\alpha,\beta,\gamma+1;1)}
  =\frac{(\gamma-\alpha)(\gamma-\beta)}{\gamma(\gamma-\alpha-\beta)}.
\end{equation}
We will generalize such convergence in a general system of
ordinary differential equations of
Schlesinger canonical form.

Under the condition
\[
  a>0,\ b>0\text{ and }c>a+b,
\]
the function $F(a,b,c;x)=\sum_{k=0}^\infty\frac{(a)_k(b)_k}{(c)_kk!}x^k$ is strictly increasing continuous 
function of $x\in [0,1]$ satisfying
\[
  1\le F(a,b,c;x)\le F(a,b,c;1)=
  \frac{\Gamma(c)\Gamma(c-a-b)}{\Gamma(c-a)\Gamma(c-b)}
\]
and it increases if $a$ or $b$ or $-c$ increases.
In particular, if 
\[
  0\le a\le N,\ 0\le b\le N \text{ and } c> 2N
\]
with a positive integer $N$, we have
\begin{align*}
 0&\le F(a,b,c;x) - 1\\
 &\le \frac{\Gamma(c)\Gamma(c-2N)}{\Gamma(c-N)\Gamma(c-N)}-1
 =\frac{(c-N)_N}{(c-2N)_N}-1
 =\prod_{\nu=1}^N\frac{c-\nu}{c-N-\nu}-1\\
 &\le\left(\frac{c-N}{c-2N}\right)^N-1
 =\left(1+\frac{N}{c-2N}\right)^N-1\\
 &\le N\left(1+\frac{N}{c-2N}\right)^{N-1}\frac{N}{c-2N}.
\end{align*}
Thus we have the following lemma.
\begin{lem}\label{lem:GHestim}
For a positive integer $N$ we have
\begin{equation}
 |F(\alpha,\beta,\gamma;x) - 1|\le \left(1+\frac{N}{\RE\gamma-2N}\right)^N-1
\end{equation}
if
\begin{equation}
 x\in\overline D,\ |\alpha|\le N,\ |\beta|\le N\text{\quad and\quad}\RE\gamma>2N.
\end{equation}
\end{lem}
\begin{proof} The lemma is clear because
\begin{align*}
 \Bigl|\sum_{k=1}^\infty\frac{(\alpha)_k(\beta)_k}{(\gamma)_kk!}x^k\Bigr|
 &\le \sum_{k=1}^\infty\frac{(|\alpha|)_k(|\beta|)_k}{(\RE\gamma)_kk!}|x|^k
 =F(|\alpha|,|\beta|,\RE\gamma-2N;|x|)-1
\end{align*}\\[-9mm]
\end{proof}
For the Gauss hypergeometric equation
\[x(1-x)u''+\bigl(\gamma-(\alpha+\beta+1)x\bigr)u'-\alpha\beta u=0\]
we have 
\begin{align*}
  (xu')'&=
   u'+xu''=\frac{xu'}x
      +\frac{((\alpha+\beta+1)x-\gamma)u'+\alpha\beta u}{1-x}\\
   &=\frac{\alpha\beta}{1-x}u+\left(\frac1{x}-\frac{\gamma}{x(1-x)}
     +\frac{\alpha+\beta+1}{1-x}\right)xu'\\
   &=\frac{\alpha\beta}{1-x}u
     +\left(\frac{1-\gamma}x+\frac{\alpha+\beta-\gamma+1}{1-x}\right)xu'.
\end{align*}
Putting
\begin{equation}
   \tilde u=\binom{u_0}{u_1}:=\binom u{\frac{xu'}\alpha}
\end{equation}
we have
\begin{equation}
 \begin{split}
   \tilde u'&=
    \frac{\begin{pmatrix}
      0 & \alpha\\ 0 & 1-\gamma
    \end{pmatrix}}x\tilde u
     +\frac{\begin{pmatrix}
      0 & 0\\ \beta& \alpha+\beta-\gamma+1
    \end{pmatrix}}{1-x}\tilde u.
 \end{split}
\end{equation}
In general, for
\begin{align*}
    v'&=\frac Axv+\frac B{1-x}v\allowdisplaybreaks
\intertext{we have}\allowdisplaybreaks
    xv'&=Av + \frac x{1-x}Bv\\
       &= Av + x\bigl(xv'+(B-A)v\bigr).
\end{align*}
Thus 
\begin{equation}\label{eq:majgauss}
 \begin{cases}
    xu_0'=\alpha u_1,\\
    xu_1'=(1-\gamma)u_1+x\bigl(xu_1'+\beta u_0+(\alpha+\beta)u_1\bigr)
  \end{cases}
\end{equation}
and the functions
\begin{equation}\label{eq:majGsol}
  \begin{cases}
   u_0 = F(\alpha,\beta,\gamma;x),\\
   u_1 =\displaystyle\frac{\beta x }\gamma F(\alpha+1,\beta+1,\gamma+1;x)
  \end{cases}
\end{equation}
satisfies \eqref{eq:majgauss}.
\begin{thm}\label{thm:paralimits}
Let $n$, $n_0$ and $n_1$ be positive integers satisfying $n=n_0+n_1$
and let $A=\begin{pmatrix}0 & A_0\\ 0 & A_1\end{pmatrix}$, $B=
\begin{pmatrix}0 & 0 \\ B_0 & B_1\end{pmatrix}\in M(n,\mathbb C)$ 
such that $A_1$, $B_1\in M(n_1,\mathbb C)$, $A_0\in M(n_0,n_1,\mathbb C)$
and $B_0\in M(n_1,n_0,\mathbb C)$.
Let $D({\mathbf 0},{\mathbf m})=D({\mathbf 0},m_1,\ldots,m_{n_1})$
be the diagonal matrix of size $n$ whose $k$-th diagonal element
is\/ $m_{k-n_0}$ if\/ $k>n_0$ and\/ $0$ otherwise.
Let\/ $u^{\mathbf m}$ be the local holomorphic solution of
\begin{equation}
  u =\frac {A-D({\mathbf 0},{\mathbf m})}xu 
     + \frac{B-D({\mathbf 0},{\mathbf m})}{1-x}u
\end{equation}
at the origin.
Then if\/ $\RE m_\nu$ are sufficiently large for $\nu=1,\dots,n_1$,
the Taylor series of $u^{\mathbf m}$ at the origin 
uniformly converge on $\overline D=\{x\in\mathbb C\,;\,|x|\le 1\}$ 
and for a positive number $C$,
the function $u^{\mathbf m}$ and their derivatives 
uniformly converge to constants on $\overline D$ when\/ 
$\min\{\RE m_1,\ldots,\RE m_{n_1}\}\to+\infty$ with\/ $|A_{ij}|+|B_{ij}|\le C$.
In particular, for $x\in\overline D$ and an integer $N$ satisfying
\begin{equation}
  \sum_{\nu=1}^{n_1}|(A_0)_{i\nu}|\le N,\ \sum_{\nu=1}^{n_1}|(A_1)_{i\nu}|\le N,\ 
  \sum_{\nu=1}^{n_0}|(B_0)_{i\nu}|\le N,\ \sum_{\nu=1}^{n_1}|(B_1)_{i\nu}|\le N
\end{equation}
we have
\begin{equation}
 \max_{1\le\nu\le n}
  \bigl|u_\nu^{\mathbf m}(x)-u_\nu^{\mathbf m}(0)\bigr|
 \le
 \max_{1\le\nu\le n_0}
 |u^{\mathbf m}_\nu(0)|
 \cdot\frac{2^N(N+1)^2}{\displaystyle\min_{1\le \nu\le n_1}\RE m_\nu-4N-1}
\end{equation}
if\/ $\RE m_\nu>5N+4$ for $\nu=1,\dots,n_1$.
\end{thm}
\begin{proof}
Use the method of majorant series and 
compare to the case of Gauss hypergeometric
series (cf.~\eqref{eq:majgauss} and \eqref{eq:majGsol}), 
namely, $\lim_{c\to+\infty}F(a,b,c;x)=1$
on $\overline D$ with a solution of the Fuchsian system
\begin{align*}
  u' &=\frac Axu + \frac B{1-x}u,\allowdisplaybreaks\\
  A &=\begin{pmatrix}
       0  & A_0\\ 0 & A_1
      \end{pmatrix},\quad
  B =\begin{pmatrix}
       0  & 0\\ B_0 & B_1
      \end{pmatrix},\quad
  u =\binom{v_0}{v_1},\allowdisplaybreaks\\
  xv_0' &= A_0v_1,\\
  xv_1' &=x^2v_1'+(1-x)A_1v_1+xB_0v_0+xB_1v_1\\
        &= A_1v_1+x\bigl(xv_1'+B_0v_0+(B_1-A_1)v_1\bigr)
\end{align*}
or the system obtained by the substitution
$A_1\mapsto A_1-D(\mathbf m)$ and $B_1\mapsto B_1-D(\mathbf m)$.
Fix positive real numbers $\alpha$, $\beta$ and $\gamma$ satisfying
\begin{align*}
  \alpha&\ge \sum_{\nu=1}^{n_1}|(A_0)_{i\nu}|\quad(1\le i\le n_0),
  \quad
  \beta\ge\sum_{\nu=1}^{n_0} |(B_0)_{i\nu}|\quad(1\le i\le n_1),\\
  \alpha+\beta
   &\ge \sum_{\nu=1}^{n_1}|(B_1-A_1)_{i\nu}|\quad(1\le i\le n_0),\\
  \gamma&=
   \min\{\RE m_1,\dots,\RE m_{n_1}\}
   -2\max_{1\le i\le n_1}\sum_{\nu=1}^{n_1}|(A_1)_{i\nu}|-1
   >\alpha+\beta.
\end{align*}
Then the method of majorant series with Lemma~\ref{lem:RSestim},
\eqref{eq:majgauss} and \eqref{eq:majGsol} imply
\[
  u^{\mathbf m}_i \ll 
  \begin{cases}
   \max_{1\le \nu\le n_0}
   |u^{\mathbf m}_\nu(0)|
   \cdot F(\alpha,\beta,\gamma;x)&(1\le i\le n_0),\\
   \frac\beta\gamma\cdot\max_{1\le\nu\le n_0}
   |u^{\mathbf m}_\nu(0)|
   \cdot F(\alpha+1,\beta+1,\gamma+1;x)&(n_0<i\le n),
  \end{cases}
\]
which proves the theorem because of Lemma~\ref{lem:GHestim}
with $\alpha=\beta=N$ as follows.
Here $\sum_{\nu=0}^\infty a_\nu x^\nu \ll \sum_{\nu=0}^\infty b_\nu x^\nu$
for formal power series means $|a_\nu|\le b_\nu$ for 
$\nu\in\mathbb Z_{\ge 0}$.

Put $\bar m=\min\{\RE m_1,\dots,\RE m_{n_1}\}-2N-1$
and $L=\max_{1\le \nu\le n_0}|u^{\mathbf m}_\nu(0)|$.  Then
$\gamma\ge \bar m-2N-1$ and 
if $0\le i\le n_0$ and $x\le\overline D$,
\begin{align*}
 |u^{\mathbf m}_i(x) - u^{\mathbf m}_i(0)|
 &\le L
      \cdot\bigl(F(\alpha,\beta,\gamma;|x|)-1\bigr)\\
 &\le L\biggl(\Bigl(1+\frac{N}{\bar m-4N-1}\Bigr)^N-1\biggr)\\
 &\le L\Bigl(1 + \frac{N}{\bar m-4N-1}\Bigr)^{N-1}\frac{N^2}{\bar m-4N-1}
  \le \frac{L2^{N-1}N^2}{\bar m - 4N-1}.
\end{align*}
If $n_0<i\le n$ and $x\in\overline D$, 
\begin{align*}
 |u^{\mathbf m}_i(x)|
  &\le \frac{\beta}{\gamma}\cdot LF(\alpha+1,\beta+1,\gamma+1;|x|)\\
  &\le \frac{LN}{\bar m-2N-1}
   \biggl(\Bigl(1+\frac{N+1}{\bar m-4N-3}\Bigr)^{N+1}+1\biggr)
  \le \frac{LN(2^{N+1}+1)}{\bar m-2N-1}.
\end{align*}\\[-.8cm]
\end{proof}

\begin{lem}\label{lem:RSestim}
Let $A\in M(n,\mathbb L)$ and put
\begin{equation}
  |A|:=\max_{1\le i\le n}
  \sum_{\nu=1}^n|A_{i\nu}|.
\end{equation}
If positive real numbers $m_1,\ldots,m_n$ satisfy
\begin{equation}
  m_{\text{min}}:=\min\{m_1,\dots,m_n\} > 2|A|,
\end{equation}
we have
\begin{equation}
  |\bigl(kI_n+D(\mathbf m)-A\bigr)^{-1}|
  \le (k+m_{\text{min}}-2|A|)^{-1}\qquad(\forall k\ge0).
\end{equation}
\end{lem}
\begin{proof}Since
\begin{align*}
  \bigl|\bigl(D(\mathbf m)-A\bigr)^{-1}\bigr|
 &= \bigl|D(\mathbf m)^{-1}(I_n-D(\mathbf m)^{-1}A)^{-1}\bigr|\\
 &=\Bigl|D(\mathbf m)^{-1}\sum_{k=0}^\infty\bigl(D(\mathbf m)^{-1}A\bigr)^k\Bigr|\\
 & \le m_{\text{min}}^{-1}\cdot\Bigl(1+\frac{2|A|}{m_{\text{min}}}\Bigr)
  \le (m_{\min}-2|A|)^{-1},
\end{align*}
we have the lemma by replacing $m_\nu$ by $m_\nu+k$ for $\nu=1,\dots,n$.
\end{proof}
\subsection{Zeros and poles of connection coefficients}\label{sec:C2}
In this subsection we examine the connection coefficients to calculate
them in a different way from the one given in \S\ref{sec:C1}.

First review the connection coefficient 
$c(0\!:\!\lambda_{0,2}\!\rightsquigarrow\!1\!:\!\lambda_{1,2})$ 
for the solution of Fuchsian differential equation with the Riemann scheme
$\begin{Bmatrix}
 x=0 & 1 & \infty\\
 \lambda_{0,1} & \lambda_{1,1} & \lambda_{2,1}\\
 \lambda_{0,2} & \lambda_{1,2} & \lambda_{2,2}
\end{Bmatrix}$.
Denoting the connection coefficient 
$c(0\!:\!\lambda_{0,2}\!\rightsquigarrow\!1\!:\!\lambda_{1,2})$
by
$
  c(\left\{\begin{smallmatrix}
  \lambda_{0,1}&&\lambda_{1,1}&\lambda_{2,1}\\
   \lambda_{0,2}&\rightsquigarrow&\lambda_{1,2}&\lambda_{2,2}
    \end{smallmatrix}\right\}),
$
we have
\begin{equation}
 u_0^{\lambda_{0,2}}=  c(\left\{\begin{smallmatrix}
  \lambda_{0,1}&&\lambda_{1,1}&\lambda_{2,1}\\
   \lambda_{0,2} &\rightsquigarrow& \lambda_{1,2}&\lambda_{2,2}
    \end{smallmatrix}\right\})u_1^{\lambda_{1,2}}
 + c(\left\{\begin{smallmatrix}
  \lambda_{0,1} && \lambda_{1,2}&\lambda_{2,1}\\
   \lambda_{0,2} &\rightsquigarrow& \lambda_{1,1}&\lambda_{2,2} 
    \end{smallmatrix}\right\})u_1^{\lambda_{1,1}}.
\end{equation}
\begin{equation}
 \begin{split}
  &c(\left\{\begin{smallmatrix}
  \lambda_{0,1} && \lambda_{1,1} &\lambda_{2,1}\\
   \lambda_{0,2} &\rightsquigarrow& \lambda_{1,2} &\lambda_{2,2} 
    \end{smallmatrix}\right\})
  =c(\left\{\begin{smallmatrix}
   \lambda_{0,1}-\lambda_{0,2} && 
    \lambda_{1,1}-\lambda_{1,2}&
     \lambda_{0,2}+\lambda_{1,2}+\lambda_{2,1}\\
   0&\rightsquigarrow&
     0&\lambda_{0,2}+\lambda_{1,2}+\lambda_{2,2}
   \end{smallmatrix}\right\})\\
  &\quad= F(\lambda_{0,2}+\lambda_{1,2}+\lambda_{2,1},
     \lambda_{0,2}+\lambda_{1,2}+\lambda_{2,2}, 
     \lambda_{0,2}-\lambda_{0,1}+1;1)
 \end{split}\label{eq:CGsum}
\end{equation}
under the notation in Definition~\ref{def:cc}.
As was explained in the first part of \S\ref{sec:estimate},
the connection coefficient is calculated from
\begin{gather}
 \lim_{n\to\infty} c(\left\{\begin{smallmatrix}
  \lambda_{0,1}-n \phantom{\rightsquigarrow} \lambda_{1,1}+n &\lambda_{2,1}\\
   \lambda_{0,2}\ \  \rightsquigarrow \ \ \lambda_{1,2} & \lambda_{2,2}
    \end{smallmatrix}\right\})=1\label{eq:GGlimit}
\intertext{and}
  \frac{c(\left\{\begin{smallmatrix}
   \lambda_{0,1} \phantom{\rightsquigarrow} \lambda_{1,1} & \lambda_{2,1}\\
   \lambda_{0,2} \rightsquigarrow \lambda_{1,2} & \lambda_{2,2}
    \end{smallmatrix}\right\})}
  {c(\left\{\begin{smallmatrix}
   \lambda_{0,1}-1 \phantom{\rightsquigarrow} \lambda_{1,1}+1 & \lambda_{2,1}\\
   \lambda_{0,2}\  \rightsquigarrow \ \lambda_{1,2} & \lambda_{2,2}
    \end{smallmatrix}\right\})}
  =\frac{(\lambda_{0,2}+\lambda_{1,1}+\lambda_{2,2})
          (\lambda_{0,2}+\lambda_{1,1}+\lambda_{2,1})}
         {(\lambda_{0,2}-\lambda_{0,1}+1)(\lambda_{1,1}-\lambda_{1,2})}.
  \label{eq:GGratio}
\end{gather}
The relation \eqref{eq:GGlimit} is easily obtained from \eqref{eq:CGsum}
and \eqref{eq:gammainf} or can be reduced to Theorem~\ref{thm:paralimits}.

We will examine \eqref{eq:GGratio}. 
For example, the relation \eqref{eq:GGratio} follows from the relation 
\eqref{eq:GCratio} which is obtained from
\begin{multline*}
\gamma\bigl(\gamma-1-(2\gamma-\alpha-\beta-1)x\bigr)
 F(\alpha,\beta,\gamma;x)+(\gamma-\alpha)(\gamma-\beta)
xF(\alpha,\beta,\gamma+1;x)\\
=\gamma(\gamma-1)(1-x)F(\alpha,\beta,\gamma-1;x)
\end{multline*}
by putting $x=1$ (cf.~\cite[\S14.1]{WW}).
We may use a shift operator as follows.
Since
\begin{align*}
 &\frac{d}{dx}F(\alpha,\beta,\gamma;x)=
 \frac{\alpha\beta}{\gamma}F(\alpha+1,\beta+1,\gamma+1;x)\\\
 &\quad=c(\left\{\begin{smallmatrix}
   1-\gamma&& \gamma-\alpha-\beta&\alpha\\
   0&\rightsquigarrow&0&\beta
    \end{smallmatrix}\right\})\tfrac{d}{dx}u_1^0
 + c(\left\{\begin{smallmatrix}
   1-\gamma&& 0&\alpha\\
   0&\rightsquigarrow&\gamma-\alpha-\beta&\beta
    \end{smallmatrix}\right\})\tfrac{d}{dx}u_1^{\gamma-\alpha-\beta}
\end{align*}
and
\begin{equation*}
 \tfrac{d}{dx}u_1^{\gamma-\alpha-\beta}\equiv
 (\alpha+\beta-\gamma)(1-x)^{\gamma-\alpha-\beta-1}
  \mod(1-x)^{\gamma-\alpha-\beta}\mathcal O_1,
\end{equation*}
we have
\begin{equation*}
\frac{\alpha\beta}{\gamma}c(\left\{\begin{smallmatrix}
   -\gamma&& 0&\alpha+1\\
   0&\rightsquigarrow&\gamma-\alpha-\beta-1&\beta+1
    \end{smallmatrix}\right\})
=(\alpha+\beta-\gamma)
 c(\left\{\begin{smallmatrix}
   1-\gamma&& 0&\alpha\\
   0&\rightsquigarrow&\gamma-\alpha-\beta&\beta
    \end{smallmatrix}\right\}),
\end{equation*}
which also proves \eqref{eq:GGratio} because
\begin{equation*}
  \frac{c(\left\{\begin{smallmatrix}
   \lambda_{0,1} && \lambda_{1,1} & \lambda_{2,1}\\
   \lambda_{0,2} &\rightsquigarrow& \lambda_{1,2} &\lambda_{2,2}
    \end{smallmatrix}\right\})}
  {c(\left\{\begin{smallmatrix}
   \lambda_{0,1}-1&& \lambda_{1,1}+1 & \lambda_{2,1}\\
   \lambda_{0,2} & \rightsquigarrow & \lambda_{1,2} & \lambda_{2,2}
    \end{smallmatrix}\right\})}
=
  \frac{c(\left\{\begin{smallmatrix}
   \lambda_{0,1} -\lambda_{0,2}&&
     0&\lambda_{0,2}+\lambda_{1,1}+\lambda_{2,1}\\
   0&\rightsquigarrow&\lambda_{1,2}-\lambda_{1,1}
     &\lambda_{0,2}+\lambda_{1,1}+\lambda_{2,2}
    \end{smallmatrix}\right\})}
  {c(\left\{\begin{smallmatrix}
   \lambda_{0,1} -\lambda_{0,2}-1&&
   0&\lambda_{0,2}+\lambda_{1,2}+\lambda_{2,1}+1\\
   0&\rightsquigarrow
   &\lambda_{1,2}-\lambda_{1,1}-1& \lambda_{0,2}+\lambda_{1,2}+\lambda_{2,2}+1
    \end{smallmatrix}\right\})}.
\end{equation*}
Furthermore each linear term appeared in the right hand side of \eqref{eq:GGratio}
has own meaning, which is as follows.

Examine the zeros and poles of the connection coefficient
$c(\left\{\begin{smallmatrix}
  \lambda_{0,1}&&\lambda_{1,1} & \lambda_{2,1}\\
   \lambda_{0,2}&\rightsquigarrow&\lambda_{1,2} & \lambda_{2,2}
    \end{smallmatrix}\right\})
$.
We may assume that the parameters $\lambda_{j,\nu}$ are generic in
the zeros or the poles.

Consider the linear form 
$\lambda_{0,2}+\lambda_{1,1}+\lambda_{2,2}$.
The local solution $u_0^{\lambda_{0,2}}$
corresponding to the characteristic exponent $\lambda_{0,2}$ at 
$0$ satisfies a Fuchsian differential equation of order 1 which
has the characteristic exponents $\lambda_{2,2}$ and 
$\lambda_{1,1}$ at $\infty$ and $1$, respectively, 
if and only if the value of the linear form is $0$ or a negative integer.
In  this case $c(\left\{\begin{smallmatrix}
  \lambda_{0,1}&&\lambda_{1,1}&\lambda_{2,1}\\
   \lambda_{0,2}&\rightsquigarrow&\lambda_{1,2}&\lambda_{2,2}
    \end{smallmatrix}\right\})$ vanishes.
This explains the term $\lambda_{0,2}+\lambda_{1,1}+\lambda_{2,2}$ in the numerator
of the right hand side of \eqref{eq:GGratio}.
The term $\lambda_{0,2}+\lambda_{1,2}+\lambda_{2,2}$ is similarly explained.

The normalized local solution $u_0^{\lambda_{0,2}}$ has poles 
where $\lambda_{0,1}-\lambda_{0,2}$ is a positive integer.  
The residue at the pole is a local solution corresponding to the exponent 
$\lambda_{0,2}$.
This means that $c(\left\{\begin{smallmatrix}
  \lambda_{0,1}&&\lambda_{1,1} & \lambda_{2,1}\\
   \lambda_{0,2}&\rightsquigarrow&\lambda_{1,2} & \lambda_{2,2}
    \end{smallmatrix}\right\})$
has poles where  $\lambda_{0,1}-\lambda_{0,2}$ is a positive integer,
which explains the term $\lambda_{0,2}-\lambda_{0,1}+1$ in the denominator
of the right hand side of \eqref{eq:GGratio}.

There exists a local solution 
$a(\lambda)u_1^{\lambda_{1,1}}+b(\lambda)u_1^{\lambda_{1,2}}$ such that it
is holomorphic for $\lambda_{j,\nu}$ and $b(\lambda)$ has a pole if
the value of $\lambda_{1,1}-\lambda_{1,2}$ is a non-negative integer,
which means $c(\left\{\begin{smallmatrix}
  \lambda_{0,1}&&\lambda_{1,1}&\lambda_{2,1}\\
   \lambda_{0,2}&\rightsquigarrow&\lambda_{1,2}&\lambda_{2,2}
    \end{smallmatrix}\right\})$ has poles where $\lambda_{1,2}-\lambda_{1,1}$
is non-negative integer.
This explains the term $\lambda_{1,1}-\lambda_{1,2}$ in the denominator
of the right hand side of \eqref{eq:GGratio}.
These arguments can be generalized, which will be explained
in this subsection.

Fist we examine the possible poles of connection coefficients.
\begin{prop}\label{prop:paradep}
Let $Pu=0$ be a differential equation of order $n$ with a regular singularity
at $x=0$ such that $P$ contains a holomorphic parameter 
$\lambda=(\lambda_1,\dots,\lambda_N)$ defined in a neighborhood of 
$\lambda^o=(\lambda_1^o,\dots,\lambda_N^o)$ in $\mathbb C^N$.
Suppose that the set of characteristic exponents of $P$ at $x=0$ equals
$\{[\lambda_1]_{(m_1)},\dots,[\lambda_N]_{(m_N)}\}$ with $n=m_1+\dots+m_N$ and
\begin{equation}
 \lambda_{2,1}^o:=\lambda_2^o-\lambda_1^o\in\mathbb Z_{\ge0}\text{ and\/ }
 \lambda_i^o-\lambda_j^o\notin\mathbb Z\text{ if\/ }
 1\le i< j\le N\text{ and\/ }j\ne 2.
\end{equation}
Let $u_{j,\nu}$ be local solutions of $Pu=0$ uniquely defined by
\begin{equation}\label{eq:SolGen}
 u_{j,\nu}\equiv x^{\lambda_j+\nu}\mod x^{\lambda_j+m_j}\mathcal O_0
 \quad(j=1,\dots,m_j\text{ and }\nu=0,\dots,m_j-1).
\end{equation}
Note that $u_{j,\nu}=\sum_{k\ge0}a_{k,j,\nu}(\lambda)x^{\lambda_j+\nu+k}$ 
with meromorphic functions $a_{k,j,\nu}(\lambda)$ of $\lambda$ which are 
holomorphic in a neighborhood of $\lambda^o$ if 
$\lambda_2-\lambda_1\ne \lambda_{2,1}^o$.
Then there exist solutions $v_{j,\nu}$ with holomorphic parameter
$\lambda$ in a neighborhood of $\lambda^o$ which satisfy the following
relations.  Namely
\begin{equation}
 v_{j,\nu}=u_{j,\nu}\quad(3\le j\le N\text{ and }\nu=0,\dots,m_j-1)
\end{equation}
and when $\lambda_1^o+m_1\ge \lambda_2^o+m_2$,
\begin{equation}
\begin{split}
 v_{1,\nu}&=u_{1,\nu}\qquad(0\le\nu < m_1),\\
 v_{2,\nu}&=\frac{u_{2,\nu}-u_{1,\nu+\lambda_{2,1}^o}}
            {\lambda_1-\lambda_2+\lambda_{2,1}^o}
   - \sum_{m_2+\lambda_{2,1}^o\le i < m_1}
              \frac{b_{\nu,i}u_{1,i}}
              {\lambda_1-\lambda_2+\lambda_{2,1}^o}
 \quad(0\le\nu < m_2)
\end{split}
\end{equation}

\noindent
with the diagram \begin{xy}
 \ar@{-}  *++!D{\lambda_1^o}  *{\circ} ;
  (10,0)  *++!D{\lambda_1^o+1}  *{\circ}="B"
 \ar@{-} "B";(18,0) \ar@{.} (18,0);(28,0)^*!U{\cdots}
 \ar@{-} (28,0);(36,0) *++!D{\lambda_1^o+\lambda_{2,1}^o} *{\circ}="H"
 \ar@{-} "H";(44,0) \ar@{.} (44,0);(54,0)
 \ar@{-} (54,0);(62,0) *++!D{\lambda_1^o+\lambda_{2,1}^o+m_2-1} *{\circ}="D"
 \ar@{}  (36,-8) *++!D{\lambda_2^o} *{\circ}="C"
 \ar@{-} "C";(44,-8) \ar@{.} (44,-8);(54,-8)^*!U{\cdots}
 \ar@{-} (54,-8);(62,-8) *++!D{\lambda_2^o+m_2-1} *{\circ};
 \ar@{-} "D";(70,0) \ar@{.} (70,0);(80,0)
 \ar@{-} (80,0);(88,0) *++!D{\lambda_1^o+m_1-1} *{\circ}
\end{xy}\\

\noindent
which illustrates some exponents and when $\lambda_1^o+m_1< \lambda_2^o+m_2$, 
\begin{equation}
\begin{split}
 v_{2,\nu}&=u_{2,\nu}\qquad(0\le\nu<m_2),\\
 v_{1,\nu}&=u_{1,\nu}-\!\!\!\!\!\!
            \sum_{\max\{0,m_1-\lambda_{2,1}^o\}\le i< m_2} 
              \frac{b_{\nu,i} u_{2,i}}
              {\lambda_1-\lambda_2+\lambda_{2,1}^o}
            \qquad(0\le\nu<\min\{m_1,\lambda_{2,1}^o\}),\\
 v_{1,\nu}&=\frac{u_{1,\nu}-u_{2,\nu-\lambda_{2,1}^o}}
           {\lambda_1-\lambda_2+\lambda_{2,1}^o}-\!\!\!\!
            \sum_{\max\{0,m_1-\lambda_{2,1}^o\}\le i< m_2} 
              \frac{b_{\nu,i} u_{2,i}}
              {\lambda_1-\lambda_2+\lambda_{2,1}^o}
  \quad (\lambda_{2,1}^o\le \nu< m_1)
\end{split}
\end{equation}
with \quad\begin{xy}
 \ar@{-}  *++!D{\lambda_1^o}  *{\circ} ;
  (10,0)  *++!D{\lambda_1^o+1}  *{\circ}="B"
 \ar@{-} "B";(18,0) \ar@{.} (18,0);(28,0)^*!U{\cdots}
 \ar@{-} (28,0);(36,0) *++!D{\lambda_1^o+\lambda_{2,1}^o} *{\circ}="H"
 \ar@{-} "H";(44,0) \ar@{.} (44,0);(54,0)^*!U{\cdots}
 \ar@{-} (54,0);(62,0) *++!D{\lambda_1^o+m_1-1} *{\circ}
 \ar@{}  (36,-8) *++!D{\lambda_2^o} *{\circ}="C"
 \ar@{-} "C";(44,-8) \ar@{.} (44,-8);(54,-8)
 \ar@{-} (54,-8);(62,-8) *++!D{\lambda_2^o-\lambda_{2,1}^o+m_1-1} *{\circ}="D";
 \ar@{-} "D";(70,-8) \ar@{.} (70,-8);(80,-8)
 \ar@{-} (80,-8);(88,-8) *++!D{\lambda_2^o+m_2-1} *{\circ}
\end{xy}\\[5pt]

\noindent
and here\/ $b_{\nu,i}\in\mathbb C$.
Note that\/ $v_{j,\nu}$ $(1\le j\le N,\ 0\le \nu< m_j)$ are linearly independent
for any fixed\/ $\lambda$ in a neighborhood of\/ $\lambda^o$.
\end{prop}
\begin{proof}
See \S\ref{sec:reg} and the proof of Lemma~\ref{lem:GRS} 
(and \cite[Theorem~6.5]{O1} in a more general setting) for the construction 
of local solutions of $Pu=0$.

Note that $u_{j,\nu}$ for $j\ge 3$ are 
holomorphic with respect to  $\lambda$ in a neighborhood of $\lambda=\lambda^o$.
Moreover note that the local monodromy generator $M_0$ of the solutions $Pu=0$ 
at $x=0$ satisfies $\prod_{j=1}^N(M_0-e^{2\pi\sqrt{-1}\lambda_j})=0$ and therefore 
the functions $(\lambda_1-\lambda_2-\lambda_{2,1}^o)u_{j,\nu}$ of $\lambda$
are holomorphically extended to the point $\lambda=\lambda^o$ for $j=1$ and $2$, 
and the values of the functions at $\lambda=\lambda^o$ are solutions of 
the equation $Pu=0$ with $\lambda=\lambda^o$.

Suppose $\lambda_1^o+m_1\ge \lambda_2^o+m_2$.
Then $u_{j,\nu}$ $(j=1,2)$ are holomorphic with respect to 
$\lambda$ at $\lambda=\lambda^o$ and there exist $b_{j,\nu}\in\mathbb C$ 
such that
\[
  u_{2,\nu}|_{\lambda=\lambda^o}=u_{1,\nu+\lambda_{2,1}^o}|_{\lambda=\lambda^o}
  + \sum_{m_2+\lambda_{2,1}^o\le\nu < m_1}b_{\nu,i}
   \bigl(u_{1,i}|_{\lambda=\lambda^o}\bigr)
\]
and we have the proposition.
Here
\[
 u_{2,\nu}|_{\lambda=\lambda^o}\equiv x^{\lambda_2^o}+
 \sum_{m_2+\lambda_{2,1}^o\le\nu < m_1}b_{\nu,i}x^{\lambda_1^o+\nu}
 \mod x^{\lambda_1^o+m_1}\mathcal O_0.
\]

Next suppose $\lambda_1^o+m_1< \lambda_2^o+m_2$.
Then there exist $b_{j,\nu}\in\mathbb C$ such that
\begin{align*}
 \bigl((\lambda_1-\lambda_2+\lambda_{2,1}^o)u_{1,\nu}\bigr)|_{\lambda=\lambda^o}
    &=\!\sum_{\max\{0,m_1-\lambda_{2,1}^o\}\le i< m_2} \!\!\!
            {b_{\nu,i} \bigl(u_{2,i}|_{\lambda=\lambda^o}\bigr)}\\
    &\qquad\qquad\qquad\qquad\qquad\qquad(0\le\nu<\min\{m_1,\lambda_{2,1}^o\}),\\
  u_{1,\nu}|_{\lambda=\lambda^o}&=
            \!\sum_{\max\{0,m_1-\lambda_{2,1}^o\}\le i< m_2}\!\!\!
            {b_{\nu,i} \bigl(u_{2,i}|_{\lambda=\lambda^o}\bigr)}
  \quad (\lambda_{2,1}^o\le \nu< m_1)
\end{align*}
and we have the proposition.
\end{proof}
The proposition implies the following corollaries.
\index{Wronskian}
\begin{cor}\label{cor:zeropole}
Retain the notation and the assumption in\/ {\rm Proposition~\ref{prop:paradep}.}

{\rm i)} \ 
Let\/ $W_j(\lambda,x)$ be the Wronskian of\/ $u_{j,1},\dots,u_{j,m_j}$
for $j=1,\dots,N$.
Then\/ $(\lambda_1-\lambda_2+\lambda_{2,1}^o)^{\ell_1}W_1(\lambda)$
and\/ $W_j(\lambda)$ with $2\le j\le N$
are holomorphic with respect to $\lambda$ in a neighborhood
of\/ $\lambda^o$ by putting
\begin{align}
 \ell_1&=
\max\bigl\{0,\min\{m_1,m_2,\lambda_{2,1}^o,\lambda_{2,1}^o+m_2-m_1\}\bigr\}.
\end{align}

{\rm ii)} \ Let
\[
  w_k=\sum_{j=1}^{N}\sum_{\nu=1}^{m_j}a_{j,\nu,k}(\lambda)u_{j,\nu,k}
\]
be a local solution defined in a neighborhood of\/ $0$ with a holomorphic
$\lambda$ in a neighborhood of $\lambda^o$.
Then
\[
  (\lambda_1-\lambda_2+\lambda_{2,1}^o)^{\ell_{2,j}}
   \det\Bigl(a_{j,\nu,k}(\lambda)
   \Bigr)_{\substack{1\le \nu\le m_j\\1\le k\le m_j}}
\]
with
\[
 \begin{cases} 
 \ell_{2,1}=\max\bigl\{0,\min\{m_1-\lambda_{2,1}^o,m_2\}\bigr\},\\
 \ell_{2,2}=\min\{m_1,m_2\},\\
 \ell_{2,j}=0\qquad(3\le j\le N)
 \end{cases}
\]
are holomorphic with respect to $\lambda$ in a neighborhood of $\lambda^o$.
\end{cor}
\begin{proof} i)
Proposition~\ref{prop:paradep} shows that $u_{j,\nu}$ 
($2\le j \le N,\ 0\le \nu<m_j$) are holomorphic with respect to $\lambda$ at
$\lambda^o$.
The functions $u_{1,\nu}$ for $\min\{m_1,\lambda_{2,1}^o\}\le \nu\le m_1$ are same.
The functions $u_{1,\nu}$ for $0\le\nu<\min\{m_1,\lambda_{2,1}^o\}$ may have poles
of order 1 along $\lambda_2-\lambda_1=\lambda_{2,1}^o$ and their residues
are linear combinations of $u_{2,i}|_{\lambda_2=\lambda_1+\lambda_{2,1}^o}$
with $\max\{0,m_1-\lambda_{2,1}^o\}\le i< m_2$.
Since
\begin{align*}
&\min\bigl
 \{\#\{\nu\,;\,0\le\nu<\min\{m_1,\lambda_{2,1}^o\}\},\ 
 \#\{i\,;\, \max\{0,m_1-\lambda_{2,1}^o\}\le i< m_2\}\bigr\}\\
&\quad=\max\bigl\{0,\min\{m_1,\lambda_{2,1}^o,m_2,m_2-m_1+\lambda_{2,1}^o
  \}\bigr\},
\end{align*}
we have the claim.

ii) A linear combination of $v_{j,\nu}$ ($1\le j\le N,\ 0\le\nu\le m_j$)
may have a pole of order 1 along $\lambda_1-\lambda_2+\lambda_{2,1}^o$ and
its residue is a linear combination of
\begin{align*}
 &\bigl(u_{1,\nu}+\sum_{m_2+\lambda_{2,1}^o\le i < m_1}
    b_{\nu+\lambda_{2,1}^o,i}u_{1,i}\bigr)|_{\lambda_2=\lambda_1+\lambda_{2,1}^o}
   \quad(\lambda_{2,1}^o\le \nu< \min\{m_1,m_2+\lambda_{2,1}^o\}),\\
 &\bigl(u_{2,\nu}
  +\sum_{\max\{0,m_1-\lambda_{2,1}^o\}\le i< m_2}b_{\nu+\lambda_{2,1}^o,i}
   u_{2,i}\bigr)
  |_{\lambda_2=\lambda_1+\lambda_{2,1}^o}
   \quad(0\le\nu<m_1-\lambda_{2,1}^o),\\
 &\sum_{\max\{0,m_1-\lambda_{2,1}^o\}\le i< m_2}b_{\nu,i}u_{2,i}|
   _{\lambda_2=\lambda_1+\lambda_{2,1}^o}
   \quad(0\le \nu<\min\{m_1,\lambda_{2,1}^o\}).
\end{align*}
Since
\begin{align*}
 &\#\bigl\{\nu\,;\,\lambda_{2,1}^o\le \nu< \min\{m_1,m_2+\lambda_{2,1}^o\}\bigr\}=
 \max\bigl\{0,\min\{m_1-\lambda_{2,1}^o,m_2\}\bigr\},\\
 &\#\{\nu\,;\,0\le\nu<m_1-\lambda_{2,1}^o\}\\
 &\quad+\min\bigl\{
 \#\{i\,;\,\max\{0,m_1-\lambda_{2,1}^o\}\le i<m_2\},
 \#\{\nu\,;\,0\le\nu<\min\{m_1,\lambda_{2,1}^o\}\}
 \bigr\}\\
 &=\min\{m_1,m_2\},
\end{align*}
we have the claim.
\end{proof}
\begin{rem}
If the local monodromy of the solutions of $Pu=0$ at $x=0$ is locally 
non-degenerate, the value of 
$(\lambda_1-\lambda_2+\lambda_{2,1}^o)^{\ell_1}W_1(\lambda)$ at 
$\lambda=\lambda^o$ does not vanish. 
\end{rem}
\index{Wronskian}
\begin{cor}\label{cor:parapole}
Let $Pu=0$ be a differential equation of order $n$ with a regular singularity
at $x=0$ such that $P$ contains a holomorphic parameter 
$\lambda=(\lambda_1,\dots,\lambda_N)$ defined on $\mathbb C^N$.
Suppose that the set of characteristic exponents of $P$ at $x=0$ equals
$\bigl\{[\lambda_1]_{(m_1)},\dots,[\lambda_N]_{(m_N)}\bigr\}$ with 
$n=m_1+\dots+m_N$.
Let $u_{j,\nu}$ be the solutions of $Pu=0$ defined by \eqref{eq:SolGen}.

{\rm i)} \ 
Let $W_1(x,\lambda)$ denote the Wronskian of  $u_{1,1},\dots,u_{1,m_1}$.
Then 
\begin{equation}\label{eq:W1hol}
  \frac{W_1(x,\lambda)}
  {\prod_{j=2}^N\prod_{0\le \nu< \min\{m_1,m_j\}}
    \Gamma(\lambda_1-\lambda_j+m_1-\nu)}
\end{equation}
is holomorphic for $\lambda\in\mathbb C^N$.

{\rm ii)} \ 
Let
\begin{equation}
 v_k(\lambda)=\sum_{j=1}^N\sum_{\nu=1}^{m_j} a_{j,\nu,k}(\lambda)u_{j,\nu}
 \quad(1\le k\le m_1)
\end{equation}
be local solutions of\/ $Pu=0$ defined in a neighborhood of\/ $0$ which have
a holomorphic parameter $\lambda\in\mathbb C^N$.
Then 
\begin{equation}\label{eq:solchol}
 \frac{\det\Bigl(a_{1,\nu,k}(\lambda)
   \Bigr)_{\substack{1\le \nu\le m_1\\1\le k\le m_1}}}
 {\prod_{j=2}^N\prod_{1\le\nu\le\min\{m_1,m_j\}}
    \Gamma(\lambda_j-\lambda_1-m_1+\nu)}
\end{equation}
is a holomorphic function of $\lambda\in\mathbb C^N$.
\end{cor}
\begin{proof}
Let $\lambda_{j,1}^o\in\mathbb Z$.
The order of poles of \eqref{eq:W1hol} and that of \eqref{eq:solchol}
along $\lambda_j-\lambda_1=\lambda_{j,1}^o$ are
\begin{align*}
&\#\{\nu\,;\,0\le\nu<\min\{m_1,m_j\}\text{ \ and \ }m_1-\lambda_{j,1}^o-\nu\le 0\}
 \\
&\quad=\#\{\nu\,;\,\max\{0,m_1-\lambda_{j,1}^o\}\le\nu<\min\{m_1,m_j\}\}\\
&\quad=\max\bigl\{0,\min\{m_1,m_j,\lambda_{j,1}^o,\lambda_{j,1}^o+m_j-m_1\}\bigr\}
\intertext{and}
&\#\{\nu\,;\, 1\le\nu\le\min\{m_1,m_j\}\text{ and } 
 \lambda_{j,1}^o-m_1+\nu\le 0\}\\
&\quad=\max\bigl\{0,\min\{m_1,m_j,m_1-\lambda_{j,1}^o\}\bigr\},
\end{align*}
respectively. Hence Corollary~\ref{cor:zeropole} assures this
corollary.
\end{proof}
\begin{rem}
The product of denominator of \eqref{eq:W1hol} and that of \eqref{eq:solchol}
equals the periodic function
\[
 \prod_{j=2}^N(-1)^{[\frac{\min\{m_1,m_j\}}2]+1}
  \Bigl(\frac\pi{\sin(\lambda_1-\lambda_j)\pi}\Bigr)
  ^{\min\{m_1,m_j\}}.
\]
\end{rem}
\index{connection coefficient!generalized}
\begin{defn}[generalized connection coefficient]\label{def:GC}
Let $P_{\mathbf m}u=0$ be the Fuchsian differential equation
with the Riemann scheme 
\begin{equation}
  \begin{Bmatrix}
   x = c_0=0 & c_1=1 & c_2& \cdots & c_p=\infty\\
  [\lambda_{0,1}]_{(m_{0,1})} & [\lambda_{1,1}]_{(m_{1,1})}
   &[\lambda_{2,1}]_{(m_{2,1})}&\cdots
    &[\lambda_{p,1}]_{(m_{p,1})}\\
  \vdots & \vdots & \vdots &\vdots & \vdots\\
    [\lambda_{0,n_0}]_{(m_{0,n_0})} & [\lambda_{1,n_1}]_{(m_{1,n_1})}&
     [\lambda_{2,n_2}]_{(m_{2,n_2})}&\cdots
      &[\lambda_{p,n_p}]_{(m_{p,n_p})}
  \end{Bmatrix}.
\end{equation}
We assume $c_2,\dots,c_{p-1}\notin[0,1]$.
Let $u_{0,\nu}^{\lambda_{0,\nu}+k}$ $(1\le\nu\le n_0,\ 0\le k<m_{0,\nu})$ 
and $u_{1,\nu}^{\lambda_{1,\nu}+k}$ $(1\le\nu\le n_1,\ 0\le k<m_{1,\nu})$ 
be local solutions of $P_{\mathbf m}u=0$ such that
\begin{equation}
 \begin{cases}
  u_{0,\nu}^{\lambda_{0,\nu}+k}\equiv
     x^{\lambda_{0,\nu}+k} &\mod x^{\lambda_{0,\nu}+m_{0,\nu}}\mathcal O_0,\\
  u_{1,\nu}^{\lambda_{1,\nu}+k}\equiv
     (1-x)^{\lambda_{1,\nu}+k} &\mod (1-x)^{\lambda_{1,\nu}+m_{1,\nu}}\mathcal O_1.
 \end{cases}
\end{equation}
They are uniquely defined on $(0,1)\subset\mathbb R$ 
when $\lambda_{j,\nu}-\lambda_{j,\nu'}\notin\mathbb Z$
for $j=0,\,1$ and $1\le\nu<\nu'\le n_j$.
Then the connection coefficients $c^{\nu',k'}_{\nu,k}(\lambda)$ are defined by
\begin{equation}
  u_{0,\nu}^{\lambda_{0,\nu}+k}
  =\sum_{\nu',\,k'}c^{\nu',k'}_{\nu,k}(\lambda)
  u_{1,\nu'}^{\lambda_{1,\nu'}+k'}.
\end{equation}
Note that $c^{\nu',k'}_{\nu,k}(\lambda)$ is a
meromorphic function of $\lambda$ when $\mathbf m$ is rigid.

Fix a positive integer $n'$ and the integer sequences
$1\le \nu^0_1<\nu^0_2<\dots<\nu^0_L\le n_0$
and $1\le \nu^1_1<\nu^1_2<\dots<\nu^1_{L'}\le n_1$
such that
\begin{equation}
  n'=m_{0,\nu^0_1}+\cdots+m_{0,\nu^0_L}=m_{1,\nu^1_1}+\cdots+m_{1,\nu^1_{L'}}.
\end{equation}
Then a \textsl{generalized connection coefficient} is defined by
\begin{equation}\label{eq:GC}
\begin{split}
 &c\bigl(0:[\lambda_{0,\nu^0_1}]_{(m_{0,\nu^0_1})},\dots,
     [\lambda_{0,\nu^0_L}]_{(m_{0,\nu^0_L})}\rightsquigarrow
   1:[\lambda_{1,\nu^1_1}]_{(m_{1,\nu^1_1})},\dots,
     [\lambda_{1,\nu^1_{L'}}]_{(m_{1,\nu^1_{L'}})}\bigr)\\
 &\quad:= \det\Bigl(c^{\nu',k'}_{\nu,k}(\lambda)\Bigr)_{\substack{
     \nu\in\{\nu^0_1,\dots,\nu^0_L\},\ 0\le k<m_{0,\nu}\\
     \nu'\in\{\nu^1_1,\dots,\nu^1_{L'}\},\ 0\le k'<m_{1,\nu'}\\
   }}.
\end{split}
\end{equation}
The connection coefficient defined in \S\ref{sec:C1} corresponds to the case
when $n'=1$.
\end{defn}
\begin{rem} i)
When $m_{0,1}=m_{1,1}$, Corollary~\ref{cor:parapole} assures that
\[
  \frac{c\bigl(0:[\lambda_{0,1}]_{(m_{0,1})}
     \rightsquigarrow 1:[\lambda_{1,1}]_{(m_{1,1})}\bigr)}{
    {\displaystyle
    \prod_{\substack{2\le j\le n_0\\0\le k< \min\{m_{0,1},\,m_{0,j}\}}}
    \!\!\!\!\!\!\!\!\!\!\!\!
    \Gamma(\lambda_{0,1}-\lambda_{0,j}+m_{0,1}-k)}\cdot\!\!
    \displaystyle
    \prod_{\substack{2\le j\le n_1\\0< k\le \min\{m_{1,1},\,m_{1,j}\}}}
    \!\!\!\!\!\!\!\!\!\!\!\!
    \Gamma(\lambda_{1,j}-\lambda_{1,1}-m_{1,1}+k)}
\]
is holomorphic for $\lambda_{j,\nu}\in\mathbb C$.

ii)
Let $v_1,\dots,v_{n'}$ be generic solutions of $P_{\mathbf m}u=0$.
Then the generalized connection coefficient in Definition~\ref{def:GC}
corresponds to a usual connection coefficient of the Fuchsian differential
equation satisfied
by the Wronskian of the $n'$ functions $v_1,\dots,v_{n'}$.
The differential equation is of order $\binom n{n'}$.
In particular, when $n'=n-1$, the differential equation is isomorphic to
the dual of the equation $P_{\mathbf m}=0$ (cf.~Theorem~\ref{thm:prod}) 
and therefore the result in \S\ref{sec:C1} can be applied to the 
connection coefficient.
The precise result will be explained in another paper.
\end{rem}

\begin{rem}\label{rem:Cproc} 
The following procedure has not been completed in general.
But we give a procedure to calculate the generalized connection coefficient 
\eqref{eq:GC}, which we put $c(\lambda)$ here for simplicity when 
$\mathbf m$ is rigid.
\begin{enumerate}
\item
Let $\bar\epsilon=\bigl(\bar\epsilon_{j,\nu}\bigr)$ be the shift of 
the Riemann scheme $\{\lambda_{\mathbf m}\}$ such that
\begin{equation}
\begin{cases}
 \bar\epsilon_{0,\nu}
  =-1&(\nu\in\{1,2,\dots,n_0\}\setminus\{\nu^0_1,\dots,\nu^0_L\}),\\
 \bar\epsilon_{1,\nu}
  =1&(\nu\in\{1,2,\dots,n_1\}\setminus\{\nu^1_1,\dots,\nu^1_{L'}\}),\\
 \bar\epsilon_{j,\nu}
  =0&(\text{otherwise}).
\end{cases}\end{equation}
Then for generic $\lambda$ we show that
the connection coefficient \eqref{eq:GC} converges to a non-zero
meromorphic function $\bar c(\lambda)$ of $\lambda$ by the shift 
$\{\lambda_{\mathbf m}\}\mapsto \{(\lambda+k\bar\epsilon)_{\mathbf m}\}$ 
when $\mathbb Z_{>0}\ni k\to \infty$.

\item
Choose suitable linear functions $b_i(\lambda)$ of $\lambda$ by applying
Proposition~\ref{prop:paradep} or 
Corollary~\ref{cor:parapole} to $c(\lambda)$ so that
$e(\lambda):=\prod_{i=1}^N\Gamma\bigl(b_i(\lambda)\bigr)^{-1}\cdot c(\lambda)
\bar c(\lambda)^{-1}$
is holomorphic for any $\lambda$.

In particular, when $L=L'=1$ and $\nu_1^0=\nu_1^1=1$, we may put
\begin{equation*}
\begin{split}
 \{b_i\}&=\bigcup_{j=2}^{n_0}
 \bigl\{\lambda_{0,1}-\lambda_{0,j}+m_{0,1}-\nu\,;\,
  0\le\nu<\min\{m_{0,1},m_{0,j}\}\bigr\}\\
  &\quad\cup\bigcup_{j=2}^{n_1}
  \bigl\{\lambda_{1,j}-\lambda_{1,1}-m_{1,1}+\nu\,;\,
    1\le\nu\le\min\{m_{1,1},m_{1,j}\}\bigr\}.
\end{split}
\end{equation*}

\item
Find the zeros of $e(\lambda)$ some of which are explained by the 
reducibility or the shift operator of the equation $P_{\mathbf m}u=0$ and 
choose linear functions $c_i(\lambda)$ of $\lambda$ so that 
$f(\lambda):=\prod_{i=1}^{N'}\Gamma\bigl(c_i(\lambda)\bigr)\cdot e(\lambda)$ 
is still holomorphic for any $\lambda$.

\item
If $N=N'$ and $\sum_i{d_i(\lambda)}=\sum_i{c_i(\lambda)}$,
Lemma~\ref{lem:GammaRatio} assures $f(\lambda)=\bar c(\lambda)$ and
\begin{equation}\label{eq:Ccj}
 c(\lambda)=\frac{\prod_{i=1}^N\Gamma\bigl(b_i(\lambda)\bigr)}
 {\prod_{i=1}^N\Gamma\bigl(c_i(\lambda)\bigr)}\cdot\bar c(\lambda)
\end{equation}
because $\frac{f(\lambda)}{f(\lambda+\epsilon)}$ is a rational function of 
$\lambda$, which follows from the existence of a shift operator assured by 
Theorem~\ref{thm:irredrigid}.
\end{enumerate}
\end{rem}
\begin{lem}\label{lem:GammaRatio}
Let $f(t)$ be a meromorphic function of\/ $t\in\mathbb C$ such that
$r(t)=\frac {f(t)}{f(t+1)}$ is a rational function and
\begin{equation}
 \lim_{\mathbb Z_{>0}\ni k\to\infty}f(t+k)=1.
\end{equation}
Then there exists $N\in\mathbb Z_{\ge 0}$ and $b_i$, 
$c_i\in\mathbb C$ for $i=1,\dots,n$ such that
\begin{align}
 b_1+\cdots+b_N&=c_1+\cdots+c_N,\label{eq:sumeq}\\
 f(t)&=\frac{\prod_{i=1}^N\Gamma(t+b_i)}{\prod_{i=1}^N\Gamma(t+c_i)}.
 \label{eq:Gammaq}
\end{align}
Moreover, if $f(t)$ is an entire function, then $f(t)$ is the constant function $1$.
\end{lem}
\begin{proof}
Since $\lim_{k\to\infty}r(t+k)=1$, we may assume
\[
  r(t)=\frac{\prod_{i=1}^N(t+c_i)}{\prod_{i=1}^N(t+b_i)}
\]
and then
\[
 f(t)=\frac{\prod_{i=1}^N\prod_{\nu=0}^{n-1}(t+c_i+\nu)}{\prod_{i=1}^N
     \prod_{\nu=0}^{n-1}(t+b_i+\nu)}f(t+n).
\]
Since
\[
 \lim_{n\to\infty}\frac{n!n^{x-1}}{\prod_{\nu=0}^{n-1}(x+\nu)}=\Gamma(x),
\]
the assumption implies \eqref{eq:sumeq} and \eqref{eq:Gammaq}.

We may assume $b_i\ne c_j$ for $1\le i\le N$ and $1\le j\le N$.
Then the function \eqref{eq:Gammaq} with \eqref{eq:sumeq} has a pole if 
$N>0$.
\end{proof}
We have the following proposition for zeros of $c(\lambda)$.
\begin{prop}\label{prop:Czero}
Retain the notation in Remark~\ref{rem:Cproc} and fix $\lambda$
so that
\begin{equation}\label{eq:Cgeneric}
 \lambda_{j,\nu}-\lambda_{j,\nu'}\notin\mathbb Z\quad
 (j=0,\ 1\text{ \ and \ }0\le\nu<\nu'\le n_j).
\end{equation}

{\rm i)} \ The relation $c(\lambda)=0$ is valid if and only if
there exists a non-zero function
\begin{equation*}
 v =
\sum_{\substack{\nu\in\{\nu^0_1,\dots,\nu^0_L\}\\ 0\le k<m_{0,\nu}}}
  C_{\nu,k}u_0^{\lambda_{0,\nu}+k}
   =
\sum_{\substack{\nu\in\{1,\dots,n_1\}
  \setminus\{\nu^1_1,\dots,\nu^1_{L'}\}\\0\le k<m_{1,\nu}}}
  C'_{\nu,k}u_1^{\lambda_{1,\nu}+k}
\end{equation*}
on $(0,1)$ with $C_{\nu,k},\,C'_{\nu,k}\in\mathbb C$.

{\rm ii)} \ Fix a shift $\epsilon=(\epsilon_{j,\nu})$ 
compatible to $\mathbf m$ and let
$R_{\mathbf m}(\epsilon,\lambda)$ be the shift operator in\/ 
{\rm Theorem~\ref{thm:irredrigid}.}
Suppose $R_{\mathbf m}(\epsilon,\lambda)$ is bijective, 
namely, $c_{\mathbf m}(\epsilon;\lambda)\ne 0$
{\rm (cf.~Theorem~\ref{thm:shiftC}).} 
Then $c(\lambda+\epsilon)=0$ if and only if $c(\lambda)=0$
\end{prop}
\begin{proof}
Assumption \eqref{eq:Cgeneric} implies that
$\{u_0^{\lambda_{0,\nu}+k}\}$ and $\{u_1^{\lambda_{1,\nu}+k}\}$
define sets of basis of local solutions of the equation 
$P_{\mathbf m}u=0$.
Hence the claim i) is clear from the definition of $c(\lambda)$.

Suppose $c(\lambda)=0$ and $R_{\mathbf m}(\epsilon,\lambda)$ is bijective.
Then applying the claim i) to $R_{\mathbf m}(\epsilon,\lambda)v$, we have
$c(\lambda+\epsilon)=0$. 
If $R_{\mathbf m}(\epsilon,\lambda)$ is bijective, 
so is $R_{\mathbf m}(-\epsilon,\lambda+\epsilon)$
and $c(\lambda+\epsilon)=0$ implies $c(\lambda)=0$.
\end{proof}
\begin{cor}\label{cor:zeroC}
Let\/ $\mathbf m=\mathbf m'\oplus\mathbf m''$ be a rigid decomposition of\/ 
$\mathbf m$ such that
\begin{equation}
  \sum_{\nu\in\{\nu^0_1,\dots,\nu^0_L\}}m'_{0,\nu}>
  \sum_{\nu\in\{\nu^1_1,\dots,\nu^1_{L'}\}}m'_{1,\nu}.
\end{equation}
Then\/ 
$\Gamma(|\{\lambda_{\mathbf m'}\}|)\cdot c(\lambda)$ is holomorphic
under the condition \eqref{eq:Cgeneric}.
\end{cor}
\begin{proof}
When $|\{\lambda_{\mathbf m'}\}|$=0, we have
the decomposition $P_{\mathbf m}=P_{\mathbf m''}P_{\mathbf m'}$ 
and hence $c(\lambda)=0$.
There exists a shift $\epsilon$ compatible to $\mathbf m$
such that $\sum_{j=0}^p\sum_{\nu=1}^{n_j}m'_{j,\nu}\epsilon_{j,\nu}=1$.
Let $\lambda$ be generic under $|\{\lambda_{\mathbf m}\}|=0$ and
$|\{\lambda_{\mathbf m'}\}|\in\mathbb Z\setminus\{0\}$.
Then Theorem~\ref{thm:isom}) ii) assures
$c_{\mathbf m}(\epsilon;\lambda)\ne0$ and
Proposition~\ref{prop:Czero} proves the corollary.
\end{proof}
\begin{rem}\label{rem:Cgamma}
Suppose that Remark~\ref{rem:Cproc} (1) is established.
Then Proposition~\ref{prop:paradep} and Proposition~\ref{prop:Czero} with
Theorem~\ref{thm:shiftC} assure that the denominator and the numerator
of the rational function which equals $\frac{c(\lambda)}{c(\lambda+\bar\epsilon)}$
are products of certain linear functions of $\lambda$
and therefore \eqref{eq:Ccj} is valid with suitable linear functions
$b_i(\lambda)$ and $c_i(\lambda)$ of $\lambda$ satisfying
$\sum_{i=1}^Nb_i(\lambda)=\sum_{i=1}^Nc_i(\lambda)$.
\end{rem}
\index{hypergeometric equation/function!generalized!connection coefficient}
\begin{exmp}[generalized hypergeometric function]
The generalized hypergeometric series \eqref{eq:IGHG} satisfies
the equation $P_n(\alpha;\beta)u=0$ given by \eqref{eq:GHP} and 
\cite[\S4.1.2 Example~9]{Kh} shows that the equation
is isomorphic to the Okubo system
\begin{equation}
 \begin{split}
  &\Bigl(x-\left(\begin{smallmatrix}
    1\\
    & 0\ \\
    & &\ \ddots\\
    & & &\  \ \ddots\\
    & & & & 0
    \end{smallmatrix}\right)\Bigr)
   \frac{d\tilde u}{dx}
  = \left(\begin{smallmatrix}
    -\beta_n & 1 \\
    \alpha_{2,1} & 0 &1\\
    \alpha_{3,1} &   & 1 & 1\\[-3pt]
    {\small \vdots}       &   & & \ {\small \ddots} &\ \ 
     {\small\ddots}\!\!\!\!\!\!\!\\
    \alpha_{n-1,1}&  & & & n-3 & 1\\
    \alpha_{n,1} & -c_{n-1} & -c_{n-2}&\cdots&-c_2 & -c_1+(n-2)
   \end{smallmatrix}\right)\tilde u
 \end{split}
\end{equation}
with
\[
 u=\begin{pmatrix}u_1\\ \vdots\\ u_n\end{pmatrix},\ u=u_1\text{ \ and \ }
 \sum_{\nu=1}^n\alpha_\nu=\sum_{\nu=1}^n\beta_\nu.
\]
Let us calculate the connection coefficient
\[
 c(0\!:\!0\rightsquigarrow 1\!:\!-\beta_n)
 =\lim_{x\to1-0}(1-x)^{\beta_n}
  {}_nF_{n-1}(\alpha_1,\dots,\alpha_n;\beta_1,\dots,\beta_{n-1};x)
  \quad(\RE\beta_n>0).
\]
Applying Theorem~\ref{thm:paralimits} to the system of Schlesinger canonical
form satisfied by $\Ad\bigl((1-x)^{\beta_n}\bigr)$, the connection
coefficient satisfies Remark~\ref{rem:Cproc} i) with 
$\bar c(\lambda)=1$, namely,
\begin{equation}\label{eq:CGHGL}
 \lim_{k\to+\infty}
  c(0\!:\!0\rightsquigarrow 1\!:\!-\beta_n)|
   _{\alpha_j\mapsto \alpha_j+k,\ \beta_j\mapsto \beta_j+k
   \ \ (1\le j\le n)}=1.
\end{equation}
Then Remark~\ref{rem:Cproc} ii) shows that 
$\prod_{j=1}^n\Gamma(\beta_j)^{-1}\cdot 
c(0\!:\!0\rightsquigarrow 1\!:\!-\beta_n)$ 
is a holomorphic function of $(\alpha,\beta)\in\mathbb C^{n+(n-1)}$.

Corresponding to the Riemann scheme \eqref{eq:GRSGHG}, 
the existence of rigid decompositions
\[
  \overbrace{1\cdots1}^n;n-11;\overbrace{1\cdots1}^n=
  \overbrace{0\cdots0}^{n-1}1;10;0\cdots\overset{\overset{i}\smallsmile}{1}
  \cdots0
  \oplus
  \overbrace{1\cdots1}^{n-1}0;n-11;1\cdots\overset{\overset{i}\smallsmile}0\cdots1
\]
for $i=1,\dots,n$ proves that $\prod_{i=1}^n\Gamma(\alpha_i)\cdot 
\prod_{j=1}^n\Gamma(\beta_j)^{-1}\cdot 
c(0\!:\!0\rightsquigarrow 1\!:\!-\beta_n)$ 
is also entire holomorphic.
Then the procedure given in Remark~\ref{rem:Cproc} assures
\begin{equation}\label{eq:CGHG}
 c(0\!:\!0\rightsquigarrow 1\!:\!-\beta_n)
 = \frac{\prod_{i=1}^n\Gamma(\beta_i)}
{\prod_{i=1}^n\Gamma(\alpha_i)}.
\end{equation}

We can also prove \eqref{eq:CGHG} as in the following way.
Since
\[\frac{d}{dx}F(\alpha;\beta;x)
 =\frac{\alpha_1\cdots\alpha_n}{\beta_1\cdots\beta_{n-1}}
  F(\alpha_1+1,\dots,\alpha_n+1;\beta_1+1,\dots,\beta_{n-1}+1;x)
\]
and 
\[\frac{d}{dx}\bigl(1-x\bigr)^{-\beta_n}
  \bigl(1+(1-x)\mathcal O_1\bigr)=\beta_n
  \bigl(1-x\bigr)^{-\beta_n-1}\bigl(1+(1-x)\mathcal O_1\bigr),
\]
we have
\[
\frac{c(0\!:\!0\rightsquigarrow 1\!:\!-\beta_n)}
 {c(0\!:\!0\rightsquigarrow 1\!:\!-\beta_n)|
  _{\alpha_j\mapsto \alpha_j+1,\ \beta_j\mapsto \beta_j+1}}
 =\frac{\alpha_1\dots\alpha_n}{\beta_1\dots\beta_n},
\]
which proves \eqref{eq:CGHG} because of \eqref{eq:CGHGL}.
\end{exmp}

A further study of generalized connection coefficients will be developed
in another paper.
In this paper we will only give some examples in \S\ref{sec:EOEx}
and \S\ref{sec:eq211}.
\section{Examples}\label{sec:ex}
When we classify tuples of partitions in this section, we identify the tuples 
which are isomorphic to each other.
For example, $21,111,111$ is isomorphic to any one of
$12,111,111$ and $111,21,111$ and $21,3,111,111$.

Most of our results in this note are constructible and can be implemented
in computer programs. 
Several reductions and constructions and 
decompositions of tuples of partitions and connections coefficients 
associated with Riemann schemes etc.\ can be computed by a program
\texttt{okubo} written by the author (cf.~\S\ref{sec:okubo}).

In \S\ref{sec:basicEx} and \S\ref{sec:rigidEx} we list
fundamental and rigid tuples respectively, most of which are
obtained by the program \texttt{okubo}.

In \S\ref{sec:PoEx} and \S\ref{sec:GHG} we apply our fractional calculus
to Jordan-Pochhammer equations and a hypergeometric family 
(generalized hypergeometric equations), respectively.
Most of the results in these sections are known but it will be useful to 
understand our unifying interpretation and apply it to general Fuchsian 
equations.

In \S\ref{sec:EOEx} we study an even family and an odd family
corresponding to Simpson's list \cite{Si}.
The differential equations of an even family appear in suitable
restrictions of Heckman-Opdam hypergeometric systems and
in particular the explicit calculation of a connection coefficient
for an even family was the original motivation for the 
study of Fuchsian differential equations developed in this note
(cf.~\cite{OS}).
We also calculate a generalized connection coefficient for an even family
of order 4.

In \S\ref{sec:4Ex}, \S\ref{sec:ord6Ex} and \S\ref{sec:Rob} we study the
rigid Fuchsian differential equations of order not larger than 4
and those of order 5 or 6 and the equations 
belonging to 12 maximal series and some minimal series classified by \cite{Ro}
which include the equations in Yokoyama's list \cite{Yo}.
We list sufficient data from which we get some connection coefficients and
the necessary and sufficient conditions for the irreducibility of the equations
as is explained in \S\ref{sec:RobEx}.

In \S\ref{sec:TriEx} we give some interesting identities of 
trigonometric functions as a consequence of the explicit value of 
connection coefficients.

We examine Appell hypergeometric equations in \S\ref{sec:ApEx}, which
will be further discussed in another paper.

In \S\ref{sec:okubo} we explain computer programs \texttt{okubo} and a library of
\texttt{risa/asir} which calculate
the results described in this paper.
\subsection{Basic tuples}\label{sec:basicEx}\label{sec:Exbasic}
\index{tuple of partitions!basic}
The number of basic tuples and fundamental tuples 
(cf.~Definition~\ref{def:fund}) with a given $\Pidx$ are as follows.
\smallskip

\noindent
{
\begin{tabular}{|c|r|r|r|r|r|r|r|r|r|r|r|r|}\hline
$\Pidx$          &\ \,0\ & \ \,1\ & $2$ & $3$ & $4$ & $5$ & $6$ & $7$ & $8$ & $9$ & $10$ & $11$
\\ \hline\hline
$\#$ fund.\ tuples & \ \,1\ & \ \,4\ & 13 & 36 & 67 & 103 & 162 & 243 & 305 & 456 & 578 & 720
\\ \hline
$\#$ basic tuples & \ \,0\ & \ \,4\ & 13 & 36 & 67 & 90 & 162 & 243 & 305 & 420 & 565 & 720
\\ \hline
$\#$ basic triplets & \ \,0\ & \ \,3\ & \ \,9&24& 44 & 56 & \ \,97& 144 & 163 & 223 & 291 & 342
\\ \hline
$\#$ basic 4-tuples& \ \,0\ & \ \,1\ & \ \,3& \ \,9& 17 & \,24 & \ \,45& \ \,68 & \ \,95 & 128 & 169 & 239
\\ \hline
maximal order & \ \,6\ & 12 & 18 & 24 & 30 & 36  & 42 & 48 & 54 & 60 & 66 & 72
\\ \hline
\end{tabular}}
\medskip

Note that if $\mathbf m$ is a basic tuple with $\idx\mathbf m<0$, then
\begin{equation}
 \Pidx k\mathbf m= 1+k^2(\Pidx\mathbf m -1)\qquad(k=1,2,\ldots).
\end{equation}
Hence the non-trivial fundamental tuple $\mathbf m$ with $\Pidx\mathbf m\le 4$ 
or equivalently $\idx\mathbf m\ge-6$ is always basic.

The tuple $2\mathbf m$ with a basic tuple $\mathbf m$ satisfying 
$\Pidx\mathbf m=2$ is a fundamental tuple and $\Pidx2\mathbf m=5$.
The tuple $422,44,44,44$ is this example.

\subsubsection{$\Pidx\mathbf m=1$}
\index{tuple of partitions!index of rigidity!$=0$}
There exist 4 basic tuples:
(cf.~\cite{Ko3}, Corollary~\ref{cor:idx0})
\index{00D4@$\tilde D_4,\ \tilde E_6,\ \tilde E_7,\ \tilde E_8$}

$\tilde D_4$:\ 11,11,11,11\ \ \ \ $\tilde E_6$:\ 111,111,111\ \ \ \ 
$\tilde E_7$:\ 22,1111,1111\ \ \ \ 
$\tilde E_8$:\ 33,222,111111
\smallskip

They are not of Okubo type.
The tuples of partitions of Okubo type with minimal order which are reduced 
to the above basic tuples are as follows.

$\tilde D_4$:\ 21,21,21,111\ \ \ \ $\tilde E_6$:\ 211,211,1111\ \ \ \ 
$\tilde E_7$:\ 32,2111,11111\ \ \ \ 
$\tilde E_8$:\ 43,322,1111111
\medskip

The list of simply reducible tuples of partitions whose indices of rigidity equal 
$0$ is given in Example~\ref{ex:SR0}.

We list the number of realizable tuples of partitions whose indices of rigidity
equal 0 according to their orders and the corresponding fundamental tuple.
\medskip

\begin{tabular}{|r|r|r|r|r|r|}\hline
ord & {\small 11,11,11,11} & {\small 111,111,111} & 
{\small 22,1111,1111} & {\small 33,222,111111} & total\\
\hline\hline
2 &     1 &       &      &     &    1  \\ \hline
3 &     1 &     1 &      &     &    2 \\ \hline
4 &     4 &     1 &    1 &     &    6 \\ \hline
5 &     6 &     3 &    1 &     &   10 \\ \hline
6 &    21 &     8 &    5 &   1 &   35 \\ \hline
7 &    28 &    15 &    6 &   1 &   50 \\ \hline
8 &    74 &    31 &   21 &   4 &  130 \\ \hline
9 &   107 &    65 &   26 &   5 &  203 \\ \hline
10&   223 &   113 &   69 &  12 &  417 \\ \hline
11&   315 &   204 &   90 &  14 &  623 \\ \hline
12&   616 &   361 &  205 &  37 & 1219 \\ \hline
13&   808 &   588 &  256 &  36 & 1688 \\ \hline
14&  1432 &   948 &  517 &  80 & 2977 \\ \hline
15&  1951 &  1508 &  659 & 100 & 4218 \\ \hline
16&  3148 &  2324 & 1214 & 179 & 6865 \\ \hline
17&  4064 &  3482 & 1531 & 194 & 9271 \\ \hline
18&  6425 &  5205 & 2641 & 389 &14660 \\ \hline
19&  8067 &  7503 & 3246 & 395 &19211 \\ \hline
20& 12233 & 10794 & 5400 & 715 &29142 \\ \hline
\end{tabular}{}

\subsubsection{$\Pidx\mathbf m=2$}\label{sec:idx-2}
\index{tuple of partitions!basic!index of rigidity$\ \ge-6$}
There are 13 basic tuples
(cf.~Proposition~\ref{prop:bas2}, \cite[Proposition~8.4]{O3}):\\[-10pt]
\begin{verbatim}
 +2:11,11,11,11,11     3:111,111,21,21      *4:211,22,22,22
  4:1111,22,22,31      4:1111,1111,211       5:11111,11111,32
  5:11111,221,221      6:111111,2211,33     *6:2211,222,222
 *8:22211,2222,44      8:11111111,332,44    10:22222,3331,55   
*12:2222211,444,66
\end{verbatim}
Here the number preceding to a tuple is the order of the tuple and
the sign \texttt{*} means that the tuple is the one given in 
Example~\ref{eq:Qsp} ($D_4^{(m)}$, $E_6^{(m)}$, $E_7^{(m)}$ and
$E_8^{(m)}$) and the sign $+$ means $d(\mathbf m)<0$.

\subsubsection{$\Pidx\mathbf m=3$} There are 36 basic tuples\\[-10pt]
{
\begin{verbatim}
 +2:11,11,11,11,11,11  3:111,21,21,21,21     4:22,22,22,31,31
 +3:111,111,111,21    +4:1111,22,22,22       4:1111,1111,31,31
  4:211,211,22,22      4:1111,211,22,31     *6:321,33,33,33
  6:222,222,33,51     +4:1111,1111,1111      5:11111,11111,311
  5:11111,2111,221     6:111111,222,321      6:111111,21111,33
  6:21111,222,222      6:111111,111111,42    6:222,33,33,42
  6:111111,33,33,51    6:2211,2211,222       7:1111111,2221,43
  7:1111111,331,331    7:2221,2221,331       8:11111111,3311,44
  8:221111,2222,44     8:22211,22211,44     *9:3321,333,333
  9:111111111,333,54   9:22221,333,441      10:1111111111,442,55 
 10:22222,3322,55     10:222211,3331,55     12:22221111,444,66
*12:33321,3333,66     14:2222222,554,77    *18:3333321,666,99
\end{verbatim}}

\subsubsection{$\Pidx\mathbf m=4$} There are 67 basic tuples\\[-10pt]
{\small \begin{verbatim}
 +2:11,11,11,11,11,11,11    3:21,21,21,21,21,21       +3:111,111,21,21,21
 +4:22,22,22,22,31          4:211,22,22,31,31          4:1111,22,31,31,31
 +3:111,111,111,111        +4:1111,1111,22,31          4:1111,211,22,22
  4:211,211,211,22          4:1111,211,211,31          5:11111,11111,41,41
  5:11111,221,32,41         5:221,221,221,41           5:11111,32,32,32
  5:221,221,32,32           6:3111,33,33,33            6:2211,2211,2211
 +6:222,33,33,33            6:222,33,33,411            6:2211,222,33,51
 *8:431,44,44,44            8:11111111,44,44,71        5:11111,11111,221
  5:11111,2111,2111        +6:111111,111111,33        +6:111111,222,222
  6:111111,111111,411       6:111111,222,3111          6:21111,2211,222
  6:111111,2211,321         6:2211,33,33,42            7:1111111,1111111,52
  7:1111111,322,331         7:2221,2221,322            7:1111111,22111,43
  7:22111,2221,331          8:11111111,3221,44         8:11111111,2222,53
  8:2222,2222,431           8:2111111,2222,44          8:221111,22211,44
  9:33111,333,333           9:3222,333,333             9:22221,22221,54
  9:222111,333,441          9:111111111,441,441       10:22222,33211,55
 10:1111111111,433,55      10:1111111111,4411,55      10:2221111,3331,55
 10:222211,3322,55         12:222111111,444,66        12:333111,3333,66
 12:33222,3333,66          12:222222,4431,66         *12:4431,444,444 
 12:111111111111,552,66    12:3333,444,552            14:33332,4442,77
 14:22222211,554,77        15:33333,555,771          *16:44431,4444,88
 16:333331,5551,88         18:33333111,666,99         18:3333222,666,99
*24:4444431,888,cc
\end{verbatim}}

Here $\verb|a|,\verb|b|,\verb|c|,\ldots$ represent 10,11,12,\ldots, respectively.

\subsubsection{Dynkin diagrams of basic tuples whose indices of rigidity equals
$-2$}
We express the basic root $\alpha_{\mathbf m}$ for
$\Pidx\mathbf m=2$ using the Dynkin diagram 
(See \eqref{eq:Dynkinidx0} for $\Pidx\mathbf m=1$).
The circles in the diagram represent the simple roots in 
$\supp\alpha_{\mathbf m}$ and two circles are connected by a 
line if the inner product of the corresponding simple roots is not zero.
The number attached to a circle is the corresponding coefficient 
$n$ or $n_{j,\nu}$ in the expression \eqref{eq:Kazpart}.

For example, if $\mathbf m=22,22,22,211$, then 
$\alpha_{\mathbf m}=4\alpha_0+2\alpha_{0,1}+2\alpha_{1,1}
+2\alpha_{2,1}+2\alpha_{3,1}+\alpha_{3,2}$, 
which corresponds to the second diagram in the following.

The circle with a dot at the center means a simple root whose inner product
with $\alpha_{\mathbf m}$ does not vanish.
Moreover the type of the root system $\Pi(\mathbf m)$ 
(cf.~\eqref{eq:Pim}) corresponding to the simple roots without a dot is given.

\vspace{-.3cm}
\index{tuple of partitions!basic!Dynkin diagram}\index{Dynkin diagram}
{\small\begin{gather*}
\begin{xy}
\ar@{-}      (10,0)       *++!D{1}  *\cir<4pt>{};
             (0,0)        *+!D+!L{2}*{\cdot}*\cir<4pt>{}="A",
\ar@{-} "A"; (3.09,9.51)  *++!L{1}  *\cir<4pt>{};
\ar@{-} "A"; (-8.09,5.88) *++!D{1}  *\cir<4pt>{};
\ar@{-} "A"; (3.09,-9.51) *++!L{1}  *\cir<4pt>{};
\ar@{-} "A"; (-8.09,-5.88)*++!D{1}  *\cir<4pt>{};
\ar@{}  (0,-14) *{11,11,11,11,11\quad 5A_1}
\end{xy}\qquad
\begin{xy}
\ar@{-}               *++!D{2}  *\cir<4pt>{};
             (10,0)   *+!L+!D{4}*\cir<4pt>{}="A",
\ar@{-} "A"; (20,0)   *++!D{2}  *{\cdot} *\cir<4pt>{}="B",
\ar@{-} "B"; (30,0)   *++!D{1}  *\cir<4pt>{},
\ar@{-} "A"; (10,-10) *++!L{2}  *\cir<4pt>{},
\ar@{-} "A"; (10,10)  *++!L{2}  *\cir<4pt>{};
\ar@{}  (15,-14) *{22,22,22,211\quad D_4+A_1}
\end{xy}
\\[.1cm]
\begin{xy}
\ar@{-}               *++!D{1}  *\cir<4pt>{};
             (10,0)   *++!D{2}  *\cir<4pt>{}="A",
\ar@{-} "A"; (20,0)   *+!L+!D{3}*\cir<4pt>{}="B",
\ar@{-} "B"; (30,0)   *++!D{2}  *\cir<4pt>{}="C",
\ar@{-} "C"; (40,0)   *++!D{1}  *\cir<4pt>{},
\ar@{-} "B"; (20,-10) *++!L{1}  *{\cdot}*\cir<4pt>{}="C",
\ar@{-} "B"; (20,10)  *++!L{1}  *{\cdot}*\cir<4pt>{};
\ar@{}  (22,-14) *{21,21,111,111\quad A_5}
\end{xy}
\qquad
\begin{xy}
\ar@{-}               *++!D{1}  *{\cdot}*\cir<4pt>{};
             (10,0)   *+!L+!D{4}*\cir<4pt>{}="A",
\ar@{-} "A"; (20,0)   *++!D{3}  *\cir<4pt>{}="B",
\ar@{-} "B"; (30,0)   *++!D{2}  *\cir<4pt>{}="C",
\ar@{-} "C"; (40,0)   *++!D{1}  *\cir<4pt>{},
\ar@{-} "A"; (10,-10) *++!L{2}  *\cir<4pt>{},
\ar@{-} "A"; (10,10)  *++!L{2}  *\cir<4pt>{},
\ar@{}  (15,-14) *{31,22,22,1111\quad D_6}
\end{xy}
\\
\begin{xy}
\ar@{-}               *++!D{2}  *\cir<4pt>{};
             (10,0)   *++!D{4}  *\cir<4pt>{}="A",
\ar@{-} "A"; (20,0)   *+!L+!D{6}*\cir<4pt>{}="B",
\ar@{-} "B"; (30,0)   *++!D{4}  *\cir<4pt>{}="C",
\ar@{-} "C"; (40,0)   *++!D{2}  *{\cdot}*\cir<4pt>{}="D",
\ar@{-} "D"; (50,0)   *++!D{1}  *\cir<4pt>{},
\ar@{-} "B"; (20,10)  *++!L{4}  *\cir<4pt>{}="F",
\ar@{-} "F"; (20,20)  *++!L{2}  *\cir<4pt>{},
\ar@{}  (25,-5) *{222,222,2211\quad E_6+A_1}
\end{xy}\quad
\begin{xy}
\ar@{-}               *++!D{1}  *\cir<4pt>{};
             (10,0)   *++!D{2}  *\cir<4pt>{}="A",
\ar@{-} "A"; (20,0)   *++!D{3}  *\cir<4pt>{}="B",
\ar@{-} "B"; (30,0)   *+!L+!D{4}*\cir<4pt>{}="C",
\ar@{-} "C"; (40,0)   *++!D{3}  *\cir<4pt>{}="D",
\ar@{-} "D"; (50,0)   *++!D{2}  *\cir<4pt>{}="E",
\ar@{-} "E"; (60,0)   *++!D{1}  *\cir<4pt>{},
\ar@{-} "C"; (30,10)  *++!L{2}  *{\cdot}*\cir<4pt>{}="F",
\ar@{-} "F"; (30,20)  *++!L{1}  *\cir<4pt>{},
\ar@{}  (35,-4) *{211,1111,1111\quad A_7+A_1}
\end{xy}
\allowdisplaybreaks\\[-.1cm]
\begin{xy}
\ar@{-}               *++!D{1}  *{\cdot}*\cir<4pt>{};
             (10,0)   *++!D{3}  *\cir<4pt>{}="A",
\ar@{-} "A"; (20,0)   *+!L+!D{5}*\cir<4pt>{}="B",
\ar@{-} "B"; (30,0)   *++!D{4}  *\cir<4pt>{}="C",
\ar@{-} "C"; (40,0)   *++!D{3}  *\cir<4pt>{}="D",
\ar@{-} "D"; (50,0)   *++!D{2}  *\cir<4pt>{}="E",
\ar@{-} "E"; (60,0)   *++!D{1}  *\cir<4pt>{},
\ar@{-} "B"; (20,10)  *++!L{3}  *\cir<4pt>{}="F",
\ar@{-} "F"; (20,20)  *++!L{1}  *{\cdot}*\cir<4pt>{},
\ar@{}  (25,-5) *{221,221,11111\quad D_7}
\end{xy}\allowdisplaybreaks\\
\begin{xy}
\ar@{-}               *++!D{2}  *\cir<4pt>{};
             (10,0)   *++!D{4}  *\cir<4pt>{}="A",
\ar@{-} "A"; (20,0)   *++!D{6}  *\cir<4pt>{}="B",
\ar@{-} "B"; (30,0)   *+!L+!D{8}*\cir<4pt>{}="C",
\ar@{-} "C"; (40,0)   *++!D{6}  *\cir<4pt>{}="D",
\ar@{-} "D"; (50,0)   *++!D{4}  *\cir<4pt>{}="E",
\ar@{-} "E"; (60,0)   *++!D{2}  *{\cdot}*\cir<4pt>{}="F",
\ar@{-} "F"; (70,0)   *++!D{1}  *\cir<4pt>{}.
\ar@{-} "C"; (30,10)  *++!L{4}  *\cir<4pt>{},
\ar@{}  (35,-4) *{44,2222,22211\quad E_7+A_1}
\end{xy}\allowdisplaybreaks\\
\begin{xy}
\ar@{-}               *++!D{1}   *{\cdot}*\cir<4pt>{};
             (10,0)   *++!D{4}   *\cir<4pt>{}="A",
\ar@{-} "A"; (20,0)   *++!D{7}   *\cir<4pt>{}="B",
\ar@{-} "B"; (30,0)   *+!L+!D{10}*\cir<4pt>{}="C",
\ar@{-} "C"; (40,0)   *++!D{8}   *\cir<4pt>{}="D",
\ar@{-} "D"; (50,0)   *++!D{6}   *\cir<4pt>{}="E",
\ar@{-} "E"; (60,0)   *++!D{4}   *\cir<4pt>{}="F",
\ar@{-} "F"; (70,0)   *++!D{2}   *\cir<4pt>{}.
\ar@{-} "C"; (30,10)  *++!L{5}   *\cir<4pt>{},
\ar@{}  (35,-4) *{55,3331,22222\quad E_8}
\end{xy}\allowdisplaybreaks\\
\begin{xy}
\ar@{-}               *++!D{1}  *\cir<4pt>{};
             (10,0)   *++!D{2}  *\cir<4pt>{}="A",
\ar@{-} "A"; (20,0)   *++!D{3}  *\cir<4pt>{}="B",
\ar@{-} "B"; (30,0)   *++!D{4}  *\cir<4pt>{}="C",
\ar@{-} "C"; (40,0)   *+!L+!D{5}*\cir<4pt>{}="D",
\ar@{-} "D"; (50,0)   *++!D{4}  *\cir<4pt>{}="E",
\ar@{-} "E"; (60,0)   *++!D{3}  *\cir<4pt>{}="F",
\ar@{-} "F"; (70,0)   *++!D{2}  *\cir<4pt>{}="G".
\ar@{-} "G"; (80,0)   *++!D{1}  *\cir<4pt>{}.
\ar@{-} "D"; (40,10)  *++!L{2}  *{\cdot}*\cir<4pt>{},
\ar@{}  (45,-4) *{32,11111,111111\quad A_9}
\end{xy}\allowdisplaybreaks\\
\begin{xy}
\ar@{-}               *++!D{1}  *\cir<4pt>{};
             (10,0)   *++!D{2}  *{\cdot}*\cir<4pt>{}="A",
\ar@{-} "A"; (20,0)   *++!D{4}  *\cir<4pt>{}="B",
\ar@{-} "B"; (30,0)   *+!L+!D{6}*\cir<4pt>{}="C",
\ar@{-} "C"; (40,0)   *++!D{5}  *\cir<4pt>{}="D",
\ar@{-} "D"; (50,0)   *++!D{4}  *\cir<4pt>{}="E",
\ar@{-} "E"; (60,0)   *++!D{3}  *\cir<4pt>{}="F",
\ar@{-} "F"; (70,0)   *++!D{2}  *\cir<4pt>{}="G".
\ar@{-} "G"; (80,0)   *++!D{1}  *\cir<4pt>{}.
\ar@{-} "C"; (30,10)  *++!L{3}  *\cir<4pt>{},
\ar@{}  (35,-4) *{33,2211,111111\quad D_8+A_1}
\end{xy}\allowdisplaybreaks\\
\begin{xy}
\ar@{-}               *++!D{4}   *\cir<4pt>{};
             (10,0)   *++!D{8}   *\cir<4pt>{}="A",
\ar@{-} "A"; (20,0)   *+!L+!D{12}*\cir<4pt>{}="B",
\ar@{-} "B"; (30,0)   *++!D{10}  *\cir<4pt>{}="C",
\ar@{-} "C"; (40,0)   *++!D{8}   *\cir<4pt>{}="D",
\ar@{-} "D"; (50,0)   *++!D{6}   *\cir<4pt>{}="E",
\ar@{-} "E"; (60,0)   *++!D{4}   *\cir<4pt>{}="F",
\ar@{-} "F"; (70,0)   *++!D{2}   *{\cdot}*\cir<4pt>{}="G".
\ar@{-} "G"; (80,0)   *++!D{1}   *\cir<4pt>{}.
\ar@{-} "B"; (20,10)  *++!L{6}   *\cir<4pt>{},
\ar@{}  (25,-4) *{66,444,2222211\quad E_8+A_1}
\end{xy}\allowdisplaybreaks\\
\begin{xy}
\ar@{-}               *++!D{2}  *{\cdot}*\cir<4pt>{};
             (10,0)   *++!D{5}  *\cir<4pt>{}="A",
\ar@{-} "A"; (20,0)   *+!L+!D{8}*\cir<4pt>{}="B",
\ar@{-} "B"; (30,0)   *++!D{7}  *\cir<4pt>{}="C",
\ar@{-} "C"; (40,0)   *++!D{6}  *\cir<4pt>{}="D",
\ar@{-} "D"; (50,0)   *++!D{5}  *\cir<4pt>{}="E",
\ar@{-} "E"; (60,0)   *++!D{4}  *\cir<4pt>{}="F",
\ar@{-} "F"; (70,0)   *++!D{3}  *\cir<4pt>{}="G".
\ar@{-} "G"; (80,0)   *++!D{2}  *\cir<4pt>{}="H".
\ar@{-} "H"; (90,0)   *++!D{1}  *\cir<4pt>{}.
\ar@{-} "B"; (20,10)  *++!L{4}  *\cir<4pt>{},
\ar@{}  (25,-4) *{44,332,11111111\quad D_{10}}
\end{xy}
\end{gather*}}

\subsection{Rigid tuples}\label{sec:rigidEx}
\index{tuple of partitions!rigid}
\subsubsection{Simpson's list}
Simpson \cite{Si} classified the rigid tuples containing the partition 
$11\cdots1$ into 4 types (Simpson's list), which follows from 
Proposition~\ref{prop:sred}.
They are $H_n$, $EO_{2m}$, $EO_{2m+1}$ and $X_6$ in the following table.
\index{Simpson's list}%
\index{tuple of partitions!rigid!Simpson's list}

See Remark~\ref{rem:length} ii) for $[\Delta(\mathbf m)]$
with these rigid tuples $\mathbf m$.

\index{tuple of partitions!rigid!21111,222,33}
The simply reducible rigid tuple (cf.~\S\ref{sec:simpred})
which is not in Simpson's list is isomorphic
to $21111,222,33$.
\smallskip

\begin{tabular}{|c|c|c|c|}\hline
order&type &name &partitions\\ \hline\hline
$n$&$H_n$&hypergeometric family&$1^n,1^n,n-11$\\ \hline
$2m$&$EO_{2m}$&even family& $1^{2m},mm-11,mm$\\ \hline
$2m+1$&$EO_{2m+1}$&odd family& $1^{2m+1},mm1,m+1m$\\ \hline
$6$&$X_6=\gamma_{6,2}$&extra case&$111111,222,42$\\ \hline
$6$&$\gamma_{6,6}$&&$21111,222,33$\\ \hline
$n$&$P_n$&Jordan Pochhammer&$n-11,n-11,\ldots\in\mathcal P_{n+1}^{(n)}$\\ \hline
\end{tabular}
\smallskip

$H_1=EO_1$, $H_2=EO_2=P_2$, $H_3=EO_3$.

\subsubsection{Isomorphic classes of rigid tuples}
Let $\mathcal R_{p+1}^{(n)}$ be the set of rigid tuples in 
$\mathcal P^{(n)}_{p+1}$.
Put $\mathcal R_{p+1}=\bigcup_{n=1}^\infty \mathcal R_{p+1}^{(n)}$,
$\mathcal R^{(n)}=\bigcup_{p=2}^\infty \mathcal R_{p+1}^{(n)}$ and
$\mathcal R=\bigcup_{n=1}^\infty\mathcal R^{(n)}$.
The sets of isomorphic classes of the elements of $\mathcal R_{p+1}^{(n)}$
(resp.~$\mathcal R_{p+1}$, $\mathcal R^{(n)}$ and $\mathcal R$)
are denoted $\bar{\mathcal R}_{p+1}^{(n)}$
(resp.~$\bar{\mathcal R}_{p+1}$, $\bar{\mathcal R}^{(n)}$ and 
$\bar{\mathcal R}$).
Then the number of the elements of $\bar{\mathcal R}^{(n)}$ are as follows.
\nopagebreak
\medskip

{
\begin{tabular}{|r|r|r||r|r|r||r|r|r|}\hline
{$n$} & ${\#\bar{\mathcal R}}_3^{(n)}$
 & ${\#\bar{\mathcal R}^{(n)}}$ &
{$n$} & ${\#\bar{\mathcal R}^{(n)}_3}$ & $\#\bar{\mathcal R}^{(n)}$ &
{$n$} & ${\#\bar{\mathcal R}^{(n)}_3}$ & $\#\bar{\mathcal R}^{(n)}$ 
 \\ \hline
2 & 1 & 1 & 15&1481&2841 &28&114600&190465\\ \hline
3 & 1 & 2 & 16&2388&4644 &29&143075&230110\\ \hline
4 & 3 & 6 & 17&3276&6128 &30&190766&310804\\ \hline
5 & 5 & 11& 18&5186&9790 &31&235543&371773\\ \hline
6 & 13& 28& 19&6954&12595&32&309156&493620\\ \hline
7 & 20& 44& 20&10517&19269&33&378063&588359\\ \hline
8 & 45& 96& 21&14040&24748&34&487081&763126\\ \hline
9 & 74&157& 22&20210&36078&35&591733&903597\\ \hline
10&142&306& 23&26432&45391&36&756752&1170966\\ \hline
11&212&441& 24&37815&65814&37&907150&1365027\\ \hline
12&421&857& 25&48103&80690&38&1143180&1734857\\ \hline
13&588&1177& 26&66409&112636&39&1365511&2031018\\ \hline
14&1004&2032& 27&84644&139350&40&1704287&2554015\\ \hline
\end{tabular}}

\subsubsection{Rigid tuples of order at most 8}%
\index{tuple of partitions!rigid!order$\ \le8$}
We show all the rigid tuples whose orders are not larger than 8.
\noindent{\begin{longtable}{ll}
\texttt{2:11,11,11} \ ($H_2$: Gauss) \\
\\
\texttt{3:111,111,21} \ ($H_3: {}_3F_2$) &
\texttt{3:21,21,21,21} \ ($P_3$) \\
\\
\texttt{4:1111,1111,31} \ ($H_4: {}_4F_3)$ &
\texttt{4:1111,211,22} \ ($EO_4$: even)\\
\texttt{4:211,211,211} \ 
($B_4$, $\text{II}_2$, $\alpha_4$)&
\texttt{4:211,22,31,31} \ ($I_4,\, \text{II}^*_2$)\\
\texttt{4:\underline{22,22,22,31}} \ ($P_{4,4}$) &
\texttt{4:31,31,31,31,31} \ ($P_4$) \\
\\
\texttt{5:11111,11111,41} \ ($H_5: {}_5F_4$) & 
\texttt{5:11111,221,32} \ ($EO_5$: odd)  \\
\texttt{5:2111,2111,32} \ ($C_5$) &
\texttt{5:2111,221,311} \ ($B_5$, $\text{III}_2$) \\
\texttt{5:\underline{221,221,221}} \ ($\alpha_5$) &
\texttt{5:221,221,41,41} \ ($J_5$)  \\
\texttt{5:221,32,32,41} &
\texttt{5:311,311,32,41} \ ($I_5,\, \text{III}^*_2$) \\
\texttt{5:\underline{32,32,32,32}} \ ($P_{4,5}$) &
\texttt{5:\underline{32,32,41,41,41}} \ ($M_5$) \\
\texttt{5:41,41,41,41,41,41} \ ($P_5$)\\
\\
\texttt{6:111111,111111,51} \ ($H_6: {}_6F_5$)\phantom{A}  
&\texttt{6:111111,222,42} ($D_6=X_6$: extra)\\
\texttt{6:111111,321,33} \ ($EO_6$: even)&\texttt{6:21111,2211,42} \ ($E_6$)\\
\texttt{6:\underline{21111,222,33}} \ ($\gamma_{6,6}$)&\texttt{6:21111,222,411} \ ($F_6$, $\text{IV}$)\\
\texttt{6:21111,3111,33} \ ($C_6$) &\texttt{6:\underline{2211,2211,33}} \ ($\beta_6$)\\
\texttt{6:2211,2211,411} \ ($G_6$) &\texttt{6:2211,321,321}\\
\texttt{6:\underline{222,222,321}} \ ($\alpha_6$)&\texttt{6:222,3111,321}\\
\texttt{6:3111,3111,321} \ ($B_6$, $\text{II}_3$) &\texttt{6:2211,222,51,51} \ ($J_6$)\\
\texttt{6:2211,33,42,51}       &\texttt{6:\underline{222,33,33,51}}\\
\texttt{6:222,33,411,51}       &\texttt{6:3111,33,411,51}\ ($I_6,\,\text{II}^*_3$)\\
\texttt{6:321,321,42,51}       &\texttt{6:321,42,42,42}\\
\texttt{6:\underline{33,33,33,42}} \ ($P_{4,6}$)
                          &\texttt{6:\underline{33,33,411,42}}\\
\texttt{6:33,411,411,42}  &\texttt{6:411,411,411,42} \ ($N_6,\,\text{IV}^*$)\\
\texttt{6:33,42,42,51,51} \ ($M_6$)&\texttt{6:321,33,51,51,51} \ ($K_6$)\\
\texttt{6:411,42,42,51,51}     &\texttt{6:51,51,51,51,51,51,51} \ ($P_6$)\\
\\
\texttt{7:1111111,1111111,61} \ ($H_7$)&  \texttt{7:1111111,331,43} \ ($EO_7$)\\
\texttt{7:211111,2221,52} \ ($D_7$) & \texttt{7:211111,322,43} \ ($\gamma_7$)\\
\texttt{7:22111,22111,52} \ ($E_7$) & \texttt{7:22111,2221,511} \ ($F_7$)\\
\texttt{7:22111,3211,43}  & \texttt{7:22111,331,421}\\
\texttt{7:\underline{2221,2221,43}} \ ($\beta_7$) &
\texttt{7:2221,31111,43} \\
\texttt{7:2221,322,421} & \texttt{7:\underline{2221,331,331}}\\
\texttt{7:2221,331,4111} &  \texttt{7:31111,31111,43} \ ($C_7$)\\
\texttt{7:31111,322,421} & \texttt{7:31111,331,4111} \ ($B_7,\ \text{III}_3$)\\
\texttt{7:3211,3211,421} & \texttt{7:\underline{3211,322,331}}\\
\texttt{7:3211,322,4111} & \texttt{7:\underline{322,322,322}} \ ($\alpha_7$)\\
\texttt{7:2221,2221,61,61} \ ($J_7$) & \texttt{7:2221,43,43,61} \\
\texttt{7:3211,331,52,61} &  \texttt{7:322,322,52,61}\\
\texttt{7:322,331,511,61} & \texttt{7:322,421,43,61}\\
\texttt{7:322,43,52,52} & \texttt{7:\underline{331,331,43,61}}\\
\texttt{7:331,43,511,52} & \texttt{7:4111,4111,43,61} \ ($I_7,\,\text{III}_3^*$) \\
\texttt{7:4111,43,511,52} & \texttt{7:421,421,421,61} \\
\texttt{7:421,421,52,52} & \texttt{7:\underline{421,43,43,52}} \\
\texttt{7:\underline{43,43,43,43}} \ ($P_{4,7}$) & \texttt{7:421,43,511,511} \\
\texttt{7:331,331,61,61,61} \ ($L_7$) & \texttt{7:421,43,52,61,61} \\
\texttt{7:\underline{43,43,43,61,61}} & \texttt{7:43,52,52,52,61} \\
 \texttt{7:511,511,52,52,61} \ ($N_7$) &  \texttt{7:43,43,61,61,61,61} \ ($K_7$)\\
\texttt{7:52,52,52,61,61,61} \ ($M_7$) & \texttt{7:61,61,61,61,61,61,61,61} \ ($P_7$)\\
\\
\texttt{8:11111111,11111111,71} \ ($H_8$) &
\texttt{8:11111111,431,44} \ ($EO_8$)\\
\texttt{8:2111111,2222,62} \ ($D_8$) &
\texttt{8:2111111,332,53} \\ \texttt{8:2111111,422,44}& 
\texttt{8:221111,22211,62} \ ($E_8$) \\  
\texttt{8:221111,2222,611} \ ($F_8$)&
\texttt{8:221111,3311,53} \\ 
\texttt{8:\underline{221111,332,44}} \ ($\gamma_8$)&
\texttt{8:221111,4211,44} \\ 
\texttt{8:22211,22211,611} \ ($G_8$)&
\texttt{8:22211,3221,53}\\
\texttt{8:\underline{22211,3311,44}} & 
\texttt{8:22211,332,521}\\
\texttt{8:22211,41111,44}&
\texttt{8:22211,431,431} \\ 
\texttt{8:22211,44,53,71}&
\texttt{8:\underline{2222,2222,53}} \ ($\beta_{8,2}$) \\ 
\texttt{8:2222,32111,53}&
\texttt{8:\underline{2222,3221,44}} \ ($\beta_{8,4}$) \\  
\texttt{8:2222,3311,521}&
\texttt{8:2222,332,5111}\\ 
\texttt{8:2222,422,431}&
\texttt{8:311111,3221,53} \\
\texttt{8:311111,332,521} & 
\texttt{8:311111,41111,44} \ ($C_8$)\\
\texttt{8:32111,32111,53} & 
\texttt{8:\underline{32111,3221,44}}\\
\texttt{8:32111,3311,521} & 
\texttt{8:32111,332,5111}\\
\texttt{8:32111,422,431} & 
\texttt{8:3221,3221,521}\\
\texttt{8:3221,3311,5111} & 
\texttt{8:\underline{3221,332,431}}\\
\texttt{8:\underline{332,332,332}} \ ($\alpha_8$)& 
\texttt{8:\underline{332,332,4211}}\\
\texttt{8:332,41111,422} &
\texttt{8:332,4211,4211} \\ 
\texttt{8:3221,4211,431} &
\texttt{8:\underline{3311,3311,431}} \\
\texttt{8:\underline{3311,332,422}} &
\texttt{8:3221,422,422} \\
\texttt{8:3311,4211,422} &
\texttt{8:41111,41111,431} \ ($B_8,\,II_4$) \\
\texttt{8:41111,4211,422} &
\texttt{8:4211,4211,4211} \\
\texttt{8:22211,2222,71,71} \ ($J_8$) &
\texttt{8:\underline{2222,44,44,71}} \\
\texttt{8:3221,332,62,71} &
\texttt{8:3221,44,521,71} \\
\texttt{8:3221,44,62,62}  &
\texttt{8:3311,3311,62,71} \\
\texttt{8:3311,332,611,71} & 
\texttt{8:3311,431,53,71} \\
\texttt{8:3311,44,611,62} &
\texttt{8:332,422,53,71} \\
\texttt{8:\underline{332,431,44,71}} & 
\texttt{8:332,44,611,611}\\
\texttt{8:332,53,53,62} &
\texttt{8:41111,44,5111,71} \ ($I_8,\,II_4^*$) \\ 
\texttt{8:41111,44,611,62} &
\texttt{8:4211,422,53,71} \\
\texttt{8:4211,44,611,611} &
\texttt{8:4211,53,53,62} \\
\texttt{8:\underline{422,422,44,71}} &
\texttt{8:422,431,521,71} \\
\texttt{8:422,431,62,62} &
\texttt{8:\underline{422,44,53,62}} \\
\texttt{8:\underline{431,44,44,62}} & 
\texttt{8:\underline{431,44,53,611}}\\
\texttt{8:422,53,53,611} & 
\texttt{8:431,431,611,62}\\
\texttt{8:431,521,53,62} &
\texttt{8:\underline{44,44,44,53}} \ ($P_{4,8}$) \\
\texttt{8:44,5111,521,62} &
\texttt{8:44,521,521,611} \\ 
\texttt{8:\underline{44,521,53,53}} &
\texttt{8:5111,5111,53,62} \\
\texttt{8:5111,521,53,611} & 
\texttt{8:521,521,521,62} \\
\texttt{8:332,332,71,71,71}  &
\texttt{8:332,44,62,71,71}  \\
\texttt{8:4211,44,62,71,71} &
\texttt{8:422,44,611,71,71} \\
\texttt{8:431,53,53,71,71}  &
\texttt{8:\underline{44,44,62,62,71}} \\
\texttt{8:44,53,611,62,71} &
\texttt{8:521,521,53,71,71} \\ 
\texttt{8:521,53,62,62,71} &
\texttt{8:53,53,611,611,71} \\ 
\texttt{8:53,62,62,62,62} &
\texttt{8:611,611,611,62,62} \ ($N_8$) \\ 
\texttt{8:53,53,62,71,71,71} &
\texttt{8:431,44,71,71,71,71} \ ($K_8$) \\
\texttt{8:611,62,62,62,71,71} \ ($M_8$)& 
\texttt{8:71,71,71,71,71,71,71,71,71} \ ($P_8$)
\end{longtable}}

\vspace*{-10pt}
Here the underlined tuples are not of Okubo type (cf.~\eqref{eq:OkuboT}).

The tuples $H_n$, $EO_n$ and $X_6$ are tuples in Simpson's list.
The series $A_n=EO_n$, $B_n$, $C_n$, $D_n$, $E_n$, $F_n$, $G_{2m}$, $I_n$, $J_n$, 
$K_n$, $L_{2m+1}$, $M_n$ and $N_n$ are given in \cite{Ro} 
and called submaximal series.   
The Jordan-Pochhammer tuples are denoted by $P_n$
and the series $H_n$ and $P_n$ are called maximal series by \cite{Ro}.
The series $\alpha_n,\beta_n,\gamma_n$ and $\delta_n$ are given in \cite{Ro} and
called minimal series.  See \S\ref{sec:Rob} for these series introduced by \cite{Ro}.
Then $\delta_n=P_{4,n}$ and they are generalized Jordan-Pochhammer tuples
(cf.~Example~\ref{ex:JPH} and \S\ref{sec:minseries}). 
Moreover $\text{II}_n$, $\text{II}^*_n$, $\text{III}_n$, $\text{III}^*_n$, IV and
$\text{IV}^*$ are in Yokoyama's list in \cite{Yo}
(cf.~\S\ref{sec:Bseries}).
\medskip

\centerline{\bf Hierarchy of rigid triplets}
\nopagebreak
\noindent
{\small\xymatrix{
{1^2,1^2,1^2}\ar[r]\ar[ddrr]&{21,1^3,1^3}\ar[r]\ar[dr]\ar[ddr]
 &{31,1^4,1^4}\ar[r]\ar[ddr]&{41,1^5,1^5}\ar[r]
 &{51,1^6,1^6}\\
{1,1,1}\ar[u]&&{2^2,21^2,1^4}\ar[r]\ar[dr]\ar[ddr]
 &{32,2^21,1^5}\ar[r]\ar[dr]\ar[ddr]&{3^2,321,1^6}\\
&&21^2,21^2,21^2\ar[dr]\ar[ddr]\ar[drr]&32,21^3,21^3\ar[dr]
 &{42,2^3,1^6}\\
&&              &31^2,2^21,21^3\ar[dr]&321,31^3,2^3\\
&&              &\underline{2^21,2^21,2^21}\ar[r]&321,321,2^21^2
}
\\[-.2cm]
\rightline{$\vdots$\phantom{ABCDEFGH}}
}

Here the arrows represent certain operations $\p_\ell$ of tuples
given by Definition~\ref{def:pell}.

\subsection{Jordan-Pochhammer family}
\label{sec:PoEx} $P_n$\index{Jordan-Pochhammer}%
\index{000Delta1@$[\Delta(\mathbf m)]$}%
\index{00P@$P_{p+1,n},\ P_n$}

We have studied the the Riemann scheme of this family in
Example~\ref{ex:midconv} iii). 

$\mathbf m=(p-11,p-11,\dots,p-11)\in\mathcal P_{p+1}^{(p)}$
\begin{align*}
&\begin{Bmatrix}
     x=0         &1=\frac1{c_1} & \cdots & \frac1{c_{p-1}}&\infty\\
     [0]_{(p-1)}  & [0]_{(p-1)}   & \cdots & [0]_{(p-1)}&[1-\mu]_{(p-1)}\\
     \lambda_0+\mu& \lambda_1+\mu & \cdots & \lambda_{p-1}+\mu&-\lambda_0-\dots-\lambda_{p-1}-\mu
    \end{Bmatrix}\\
 \Delta(\mathbf m)&=\{\alpha_0,\,\alpha_0+\alpha_{j,1}\,;\,j=0,\dots,p\}\\
[\Delta(\mathbf m)]&=1^{p+1}\cdot(p-1)\\
P_p&=H_1\oplus P_{p-1}:p+1=(p-1)H_1\oplus H_1:1
\end{align*}
Here the number of the decompositions of a given type is shown after
the decompositions.  For example, $P_p=H_1\oplus P_{p-1}:p+1=(p-1)H_1\oplus H_1:1$ 
represents the decompositions
\[
 \begin{split}
  \mathbf m&=10,\dots,\overset{\underset{\smallsmile}\nu}{01},\dots,10
  \oplus p-21,\dots,\overset{\underset{\smallsmile}\nu}{p-10},\ldots, p-21
  \qquad(\nu=0,\dots,p)\\
   &=(p-1)(10,\dots,10)\oplus 01,\dots,01.
 \end{split}
\]

The differential equation $P_{P_p}(\lambda,\mu)u=0$ 
with this Riemann scheme is given by
\[
P_{P_{p}}(\lambda,\mu):=\RAd\bigl(\p^{-\mu}\bigr)\circ
 \RAd\Bigl(x^{\lambda_0}\prod_{j=1}^{p-1}(1-c_jx)^{\lambda_j}\Bigr)\p
\]
and then
\begin{equation}\label{eq:Poch}
 \begin{split}
 P_{P_{p}}(\lambda,\mu)&=\sum_{k=0}^p p_k(x)\p^{p-k},\\
 p_k(x):\!&=\binom{-\mu+p-1}{k}
  p_0^{(k)}(x) + 
  \binom{-\mu+p-1}{k-1}
  q^{(k-1)}(x)
\end{split}\end{equation}
with
\begin{equation}
p_0(x)=x\prod_{j=1}^{p-1}(1-c_jx),\quad
q(x)=p_0(x)\Bigl(-\frac{\lambda_0}{x}+\sum_{j=1}^{p-1}
 \frac{c_j\lambda_j}{1-c_jx}\Bigr).
\end{equation}
It follows from Theorem~\ref{thm:irred} that the equation is 
irreducible if and only if
\begin{equation}
 \lambda_j\notin\mathbb Z\ \ (j=0,\dots,p-1),\ 
 \mu\notin\mathbb Z\text{ \ and \ } \lambda_0+\cdots+\lambda_{p-1}
 +\mu\notin\mathbb Z.
\end{equation}
It follows from Proposition~\ref{prop:shift} that the shift operator defined by the map 
$u\mapsto \p u$ is bijective if and only if
\begin{equation}
 \mu\notin\{1,2,\dots,p-1\}\text{ \ and \ }
 \lambda_0+\cdots+\lambda_{p-1}+\mu\ne0.
\end{equation}
The normalized solution at $0$ corresponding to the exponent 
$\lambda_0+\mu$ is
\begin{align*}
u_0^{\lambda_0+\mu}(x)&=\frac{\Gamma(\lambda_0+\mu+1)}
    {\Gamma(\lambda_0+1)\Gamma(\mu)}\int_0^x\Bigl(t^{\lambda_0}
  \prod_{j=1}^{p-1}(1-c_jt)^{\lambda_j}\Bigr)(x-t)^{\mu-1}dt
\allowdisplaybreaks\\
&=\frac{\Gamma(\lambda_0+\mu+1)}
    {\Gamma(\lambda_0+1)\Gamma(\mu)}\int_0^x\sum_{m_1=0}^\infty
 \cdots\sum_{m_{p-1}=0}^\infty
 \frac{(-\lambda_1)_{m_1}\cdots(-\lambda_{p-1})_{m_{p-1}}}
 {m_1!\cdots m_{p-1}!}\\
 &\qquad{}c_2^{m_2}\cdots c_{p-1}^{m_{p-1}}
    t^{\lambda_0+m_1+\cdots+m_{p-1}}(x-t)^{\mu-1}dt
\allowdisplaybreaks\\
 &=\sum_{m_1=0}^\infty
 \cdots\sum_{m_{p-1}=0}^\infty
 \frac{(\lambda_0+1)_{m_1+\cdots+m_{p-1}}
  (-\lambda_1)_{m_1}\cdots(-\lambda_{p-1})_{m_{p-1}}}
 {(\lambda_0+\mu+1)_{m_1+\cdots+m_{p-1}}m_1!\cdots m_{p-1}!}\\
 &\qquad
 c_2^{m_2}\cdots c_{p-1}^{m_{p-1}}
 x^{\lambda_0+\mu+m_1+\cdots+m_{p-1}}\\
 &=x^{\lambda_0+\mu}\Bigl
 (1-\frac{(\lambda_0+1)(\lambda_1c_1+\cdots+\lambda_{p-1}c_{p-1})}
  {\lambda_0+\mu+1}x+\cdots\Bigr).
\end{align*}
This series expansion of the solution is easily obtained from the
formula in \S\ref{sec:series} (cf.~Theorem~\ref{thm:expsol})
and Theorem~\ref{thm:shifm1} gives
the recurrence relation
\begin{equation}
 u_0^{\lambda_0+\mu}(x)= u_0^{\lambda_0+\mu}(x)
  \!\bigm|_{\lambda_1\mapsto\lambda_1-1}
  - \Bigl(\frac{\lambda_0}{\lambda_0+\mu}u_0^{\lambda_0+\mu}(x)\Bigr)
    \!\Bigm|_{\substack{\lambda_0\mapsto\lambda_0+1\\\lambda_1\mapsto\lambda_1-1}}.
\end{equation}
Lemma~\ref{lem:conn} with $a=\lambda_0$, $b=\lambda_1$
and $u(x)=\prod_{j=2}^{p-1}(1-c_jx)^{\lambda_j}$ gives
the following connection coefficients
\begin{align*}
c(0:\lambda_0+\mu\!\rightsquigarrow\!1:\lambda_1+\mu)&=
 \frac{\Gamma(\lambda_0+\mu+1)\Gamma(-\lambda_1-\mu)}
 {\Gamma(\lambda_0+1)\Gamma(-\lambda_1)}\prod_{j=2}^{p-1}(1-c_j)^{\lambda_j},
 \allowdisplaybreaks\\
c(0:\lambda_0+\mu\!\rightsquigarrow\!1:0)&=
 \frac{\Gamma(\lambda_0+\mu+1)}{\Gamma(\mu)\Gamma(\lambda_0+1)}\int_0^1
 t^{\lambda_0}(1-t)^{\lambda_1+\mu-1}\prod_{j=2}^{p-1}(1-c_jt)^{\lambda_j}dt\\
 &\hspace{-80pt}=\frac{\Gamma(\lambda_0+\mu+1)\Gamma(\lambda_1+\mu)}
   {\Gamma(\mu)\Gamma(\lambda_0+\lambda_1+\mu+1)}
   F(\lambda_0+1,-\lambda_2,\lambda_0+\lambda_1+\mu+1;c_2)\qquad(p=3).
\end{align*}
Here we have
\begin{equation}
 u_0^{\lambda_0+\mu}(x)=\sum_{k=0}^\infty
   C_k(x-1)^k+ \sum_{k=0}^\infty C'_k(x-1)^{\lambda_1+\mu+k}
\end{equation}
for $0< x< 1$ with $C_0=c(0:\lambda_0+\mu\!\rightsquigarrow\!1:0)$ and 
$C'_0=c(0:\lambda_0+\mu\!\rightsquigarrow\!1:\lambda_1+\mu)$.
Since  $\frac {d^ku_0^{\lambda_0+\mu}}{dx^k}$ is a solution of the 
equation $P_{P_p}(\lambda,\mu-k)u=0$, we have
\begin{equation}
 C_k =  \frac{\Gamma(\lambda_0+\mu+1)}{\Gamma(\mu-k)\Gamma(\lambda_0+1)k!}
 \int_0^1
 t^{\lambda_0}(1-t)^{\lambda_1+\mu-k-1}\prod_{j=2}^{p-1}(1-c_jt)^{\lambda_j}dt.
\end{equation}
When $p=3$, 
\[C_k=\frac{\Gamma(\lambda_0+\mu+1)\Gamma(\lambda_1+\mu-k)}
        {\Gamma(\mu-k)\Gamma(\lambda_0+\lambda_1+\mu+1-k)k!}
   F(\lambda_0+1,-\lambda_2,\lambda_0+\lambda_1+\mu+1-k;c_2).
\]

Put
\[
\begin{split}
  u_{\lambda,\mu}(x)&=\frac1{\Gamma(\mu)}
  \int_0^x\Bigl(t^{\lambda_0}
  \prod_{j=1}^{p-1}(1-c_jt)^{\lambda_j}\Bigr)(x-t)^{\mu-1}dt
  =\p^{-\mu}v_{\lambda},\\
  v_{\lambda}(x):\!&= x^{\lambda_0}\prod_{j=1}^{p-1}(1-c_jx)^{\lambda_j}.
\end{split}
\]
We have
\begin{equation}\label{eq:PoCon0}
\begin{split}
 u_{\lambda,\mu+1}&=\p^{-\mu-1}v_{\lambda}=\p^{-1}\p^{-\mu}v_{\lambda}
 =\p^{-1}u_{\lambda,\mu},\\
 u_{\lambda_0+1,\lambda_1,\ldots,\mu}&
 =\p^{-\mu}v_{\lambda_0+1,\lambda_1,\ldots}=\p^{-\mu}xv_\lambda
 =-\mu\p^{-\mu-1}v_{\lambda}+x\p^{-\mu}v_\lambda\\
 &=-\mu\p^{-1}u_{\lambda,\mu}+xu_{\lambda,\mu},\\
 u_{\ldots,\lambda_j+1,\ldots}&=\p^{-\mu}(1-c_jx)v_{\lambda}=
 \p^{-\mu}v_\lambda +c_j\mu\p^{-\mu-1}v_\lambda - c_jx\p^{-\mu}v_\lambda\\
 &=(1-c_jx)u_{\lambda,\mu}+c_j\mu\p^{-1}u_{\lambda,\mu}.
\end{split}\end{equation}
From these relations with $P_{P_p}u_{\lambda,\mu}=0$ 
we have all the contiguity relations.
For example
\begin{align}
 \p u_{\lambda_0,\dots,\lambda_{p-1},\mu+1}&= u_{\lambda,\mu},\\
 \p u_{\lambda_0+1,\ldots,\lambda_{p-1},\mu}&=
     (x\p +1-\mu) u_{\lambda,\mu},\notag\\
 \p u_{\ldots,\lambda_j+1,\ldots,\mu}&=
    \bigl((1-c_jx)\p - c_j(1-\mu)\bigr)u_{\lambda,\mu}\notag
\end{align}
and
\begin{align*}
 P_{P_p}(\lambda,\mu+1)&=\sum_{j=0}^{p-1}p_j(x)\p^{p-j} + p_n\\[-5pt]
 p_n &= (-1)^{p-1}c_1\dots c_{p-1}\Bigl((-\mu-1)_p+
    (-\mu)_{p-1} \sum_{j=0}^{p-1}\lambda_j\Bigr)\\
    &=c_1\cdots c_{p-1}(\mu+2-p)_{p-1}
        (\lambda_0+\cdots+\lambda_{p-1}-\mu-1)
\end{align*}
and hence
\begin{align*}
 \Bigl(\sum_{j=0}^{p-1}p_j(x)\p^{p-j-1}\Bigr)u_{\lambda,\mu}&=
 -p_n u_{\lambda,\mu+1} = -p_n \p^{-1}u_{\lambda,\mu}.
\end{align*}
Substituting this equation to \eqref{eq:PoCon0}, we have
$Q_j\in W(x;\lambda,\mu)$ such that $Q_ju_{\lambda,\mu}$
equals  
$u_{(\lambda_\nu+\delta_{\nu,j})_{\nu=0,\dots,p-1},\mu}$ for $j=0,\dots,p-1$,
respectively.
The operators $R_j\in W(x;\lambda,\mu)$ satisfying 
$R_jQ_ju_{\lambda,\mu}=u_{\lambda,\mu}$ are calculated by the Euclid 
algorithm, namely, we find $S_j\in W(x;\lambda,\mu)$ so that 
$R_jQ_j+S_jP_{P_p}=1$.
Thus we also have
$T_j\in W(x;\lambda,\mu)$ such that $T_ju_{\lambda,\mu}$
equals  
$u_{(\lambda_\nu-\delta_{\nu,j})_{\nu=0,\dots,p-1},\mu}$ for $j=0,\dots,p-1$,
respectively.

As is shown in \S\ref{sec:Versal} the 
\textsl{Versal Jordan-Pochhammer operator} $\tilde P_{P_p}$ is given by 
\eqref{eq:Poch} with 
\index{Jordan-Pochhammer!versal}
\begin{equation}
 p_0(x)=\prod_{j=1}^{p}(1-c_jx),\quad
 q(x)=\sum_{k=1}^p\lambda_kx^{k-1}\prod_{j=k+1}^p(1-c_jx).
\end{equation}
If $c_1,\dots,c_p$ are different to each other, the Riemann scheme
of $\tilde P_{P_p}$ is
\[
  \begin{Bmatrix} 
    x=\frac1{c_j}\ (j=1,\dots,p) & \infty\\
    [0]_{(p-1)} & [1-\mu]_{(p-1)}\\
    \displaystyle\sum_{k=j}^p
     \frac{\lambda_k}{c_j\prod_{\substack{1\le\nu\le k\\\nu\ne j}}
     (c_j-c_\nu)}+\mu& 
    \displaystyle\sum_{k=1}^p\frac{(-1)^k\lambda_k}{c_1\dots c_k}-\mu
  \end{Bmatrix}.
\]
The solution of $\tilde P_{P_p}u=0$ is given by
\[
 u_C(x) = \int_C\Bigl(\exp
        \int_0^t\sum_{j=1}^p\frac{-\lambda_j s^{j-1}}{\prod_{1\le\nu\le j}
         (1-c_\nu s)}ds\Bigr)(x-t)^{\mu-1}dt.
\]
Here the path $C$ starting from a singular point and ending at
a singular point is chosen so that the integration has a meaning.
In particular when $c_1=\cdots=c_p=0$, we have
\[
  u_C(x)=\int_C\exp\Bigl(-\sum_{j=1}^p\frac{\lambda_jt^j}{j!}\Bigr)
         (x-t)^{\mu-1}dt
\]
and if $\lambda_p\ne0$, the path $C$ starts from $\infty$ to one of the 
$p$ independent directions $\lambda_p^{-1}e^{\frac{2\pi\nu\sqrt{-1}}{p}+t}$ 
$(t\gg1,\ \nu=0,1,\dots,p-1)$ and ends at $x$.

Suppose $n=2$.  
The corresponding Riemann scheme for the generic characteristic exponents 
and its construction from the Riemann scheme of the trivial equation 
$u'=0$ is as follows:
\begin{align*}\begin{split}
 &\begin{Bmatrix}
x=0 & 1& \infty\\
 b_0 & c_0 & a_0\\
 b_1 & c_1 & a_1\\
\end{Bmatrix}\qquad(\text{Fuchs relation: } a_0+a_1+b_0+b_1+c_0+c_1=1)\\
&\quad\xleftarrow{x^{b_0}(1-x)^{c_0}\p^{-a_1-b_1-c_1}}
 \begin{Bmatrix}
x=0 & 1& \infty\\
 -a_1-b_0-c_1 & -a_1-b_1-c_0 & -a_0+a_1+1\\
\end{Bmatrix}\\
&\quad\xleftarrow{x^{-a_1-b_0-c_1}(1-x)^{-a_1-b_1-c_0}}
 \begin{Bmatrix}
x=0 & 1& \infty\\
 0 & 0 & 0\\
\end{Bmatrix}.
\end{split}\end{align*}
Then our fractional calculus gives the corresponding equation
\begin{align}\begin{split}
 &x^2(1-x)^2u''-x(1-x)\bigl((a_0+a_1+1)x+b_0+b_1-1\bigr)u'\\
 &\quad{}+\bigl(a_0a_1x^2-(a_0a_1+b_0b_1-c_0c_1)x+b_0b_1\bigr)u=0,
\end{split}\end{align}
the connection formula
\begin{align}\begin{split}
 &c(0\!:\!b_1 \rightsquigarrow 1\!:\!c_1)=\frac{
  \Gamma(c_0-c_1)
  \Gamma(b_1-b_0+1)
 }{
  \Gamma(a_0+b_1+c_0)
  \Gamma(a_1+b_1+c_0)
 }
\end{split}\end{align}
and expressions of its solution by the integral representation
\begin{align}\label{eq:GGI}\begin{split}
 &\int_0^{x}x^{b_0}(1-x)^{c_0}(x-s)^{a_1+b_1+c_1-1}
 s^{-a_1-c_1-b_0}(1-s)^{-a_1-b_1-c_0-}ds\\
 &\quad=\frac{
  \Gamma(a_0+b_1+c_0)
  \Gamma(a_1+b_1+c_1)
 }{
  \Gamma(b_1-b_0+1)
 }x^{b_1}\phi_{b_1}(x)
\end{split}\end{align}
and the series expansion
\begin{align}\label{eq:GGS}\begin{split}
 &\sum_{n\ge0}\frac{(a_0+b_1+c_0)_n(a_1+b_1+c_0)_n}{(b_1-b_0+1)_nn!}(1-x)^{c_0}x^{b_1+n}\\
 &\quad=(1-x)^{c_0}x^{b_1}F(a_0+b_1+c_0,a_1+b_1+c_0,b_1-b_0-1;x).
\end{split}\end{align}
Here $\phi_{b_1}(x)$ is a holomorphic function in a neighborhood of 
$0$ satisfying $\phi_{b_1}(0)=1$ for generic spectral parameters.
We note that the transposition of $c_0$ and $c_1$ in \eqref{eq:GGS}
gives a nontrivial equality, which corresponds to Kummer's 
relation of Gauss hypergeometric function and the similar statement
is true for \eqref{eq:GGI}.
In general, different procedures of reduction of a equation give
different expressions of its solution.
\subsection{Hypergeometric family} $H_n$\label{sec:GHG}
\index{hypergeometric equation/function!generalized}%
\index{00Hn@$H_n$}

We examine the hypergeometric family which corresponds to the 
equations satisfied by the generalized hypergeometric series 
\eqref{eq:IGHG}.  Its spectral type is in the Simpson's list
(cf.~\S\ref{sec:rigidEx}).

$\mathbf m=(1^n,n-11,1^n)$ : ${}_nF_{n-1}(\alpha,\beta;z)$
\begin{align*}
 1^n,n-11,1^n&=1,10,1\oplus 1^{n-1},n-21,1^{n-1}\\
 \Delta(\mathbf m)&=\{\alpha_0+\alpha_{0,1}
  +\cdots+\alpha_{0,\nu}+\alpha_{2,1}+\cdots+\alpha_{2,\nu'}\,;\,\\
 &\qquad
  0\le\nu<n,\ 0\le\nu'< n\}\\
 [\Delta(\mathbf m)]&=1^{n^2}\\
 H_n &= H_1 \oplus H_{n-1}:n^2\\
 H_n &\overset{1}{\underset{R2E0}\longrightarrow} H_{n-1}
\end{align*}

Since $\mathbf m$ is of Okubo type, we have a system 
of Okubo normal form with the spectral type $\mathbf m$.
Then the above $R2E0$ represents the reduction of systems of
equations of Okubo normal form due to Yokoyama \cite{Yo2}.
The number $1$ on the arrow represents a reduction by a middle
convolution and the number shows the difference of the orders.
\begin{gather}\label{eq:HnR}
\begin{Bmatrix}
     x=0 & 1 & \infty\\
     \lambda_{0,1} & [\lambda_{1,1}]_{(n-1)}   & \lambda_{2,1}\\
     \vdots        &              & \vdots \\
     \lambda_{0,n-1}&  &\lambda_{2,n-1}\\
     \lambda_{0,n}  & \lambda_{1,2} &\lambda_{2,n}
    \end{Bmatrix}
,\quad \begin{Bmatrix}
    x = 0 & 1 & \infty\\
     1-\beta_1    & [0]_{(n-1)}   & \alpha_1\\
     \vdots       &     & \vdots \\
     1-\beta_{n-1}&     & \alpha_{n-1}\\
     0 &   -\beta_n    & \alpha_n
    \end{Bmatrix}\\
  \sum_{\nu=1}^n(\lambda_{0,\nu}+\lambda_{2,\nu})
  +(n-1)\lambda_{1,1}+\lambda_{1,2}=n-1,\notag\\
  \alpha_1+\cdots+\alpha_n=\beta_1+\cdots+\beta_n.\notag
\end{gather}

It follows from Theorem~\ref{thm:sftUniv} that the universal operators
\[
 P_{H_1}^0(\lambda),\ P_{H_1}^{2}(\lambda),\ P_{H_{n-1}}^0(\lambda),\ 
 P_{H_{n-1}}^1(\lambda),\ P_{H_{n-1}}^2(\lambda).
\]
are shift operators for the universal model $P_{H_n}(\lambda)u=0$.

The Riemann scheme of the operator
\[
P=\RAd(\p^{-\mu_{n-1}})\circ\RAd(x^{\gamma_{n-1}})\circ\cdots
   \circ\RAd(\p^{-\mu_1})\circ\RAd(x^{\gamma_1}(1-x)^{\gamma'})\p
\]
equals
\begin{equation}\label{eq:HGRS0}
{
 \begin{Bmatrix}
  x=0 & 1 & \infty\\
  0  & [0]_{(n-1)} & 1-\mu_{n-1}\\
  (\gamma_{n-1}+\mu_{n-1}) & & 1-(\gamma_{n-1} + \mu_{n-1}) - \mu_{n-2}\\
  \displaystyle\sum_{j={n-2}}^{n-1}(\gamma_j+\mu_j)&&
  1 - \displaystyle\sum_{j={n-2}}^{n-1}(\gamma_j+\mu_j) - \mu_{n-3}\\
  \vdots && \vdots\\
 \displaystyle\sum_{j=2}^{n-1}(\gamma_j+\mu_j)&&
  1 - \displaystyle\sum_{j=2}^{n-1}(\gamma_j+\mu_j) -\mu_1\\
 \displaystyle\sum_{j=1}^{n-1}(\gamma_j+\mu_j)&
   \gamma'+\displaystyle\sum_{j=1}^{n-1}\mu_j&
    -\gamma'-\displaystyle\sum_{j=1}^{n-1}(\gamma_j+\mu_j)
 \end{Bmatrix},}
\end{equation}
which is obtained by the induction on $n$ with 
Theorem~\ref{thm:GRSmid} and corresponds to the second Riemann scheme
in \eqref{eq:HnR} by putting
\begin{equation}\begin{aligned}
 \gamma_j&=\alpha_{j+1}-\beta_j&&(j=1,\dots,n-2),&
 \gamma'&=-\alpha_1+\beta_1-1,\\
 \mu_j &= -\alpha_{j+1}+\beta_{j+1}&&(j=1,\dots,n-1),&
 \mu_{n-1} &= 1-\alpha_n.
\end{aligned}\end{equation}
The integral representation of the local solutions at $x=0$ (resp.\ 1 and 
$\infty$) corresponding
to the exponents $\sum_{j=1}^{n-1}(\gamma_j+\mu_j)$ 
(resp.\ $\gamma'+\sum_{j=1}^{n-1}\mu_j$ and $-\gamma'-\sum_{j=1}^{n-1}(\gamma_j+\mu_j)$
are given by
\begin{equation}
   I_c^{\mu_{n-1}}x^{\gamma_{n-1}}I_c^{\mu_{n-2}}\cdots
   I_c^{\mu_1}x^{\gamma_1}(1-x)^{\gamma'}
\end{equation}
by putting $c=0$ (resp.\ 1 and $\infty$).

For simplicity we express this construction using additions and 
middle convolutions by
\begin{equation}
  u=\p^{-\mu_{n-1}}x^{\gamma_{n-1}}\cdots\p^{-\mu_2}x^{\gamma_2}
    \p^{-\mu_2}x^{\gamma_1}(1-x)^{\gamma'}.
\end{equation}

For example, when $n=3$, we have the solution
\[
  \int_c^x t^{\alpha_3-\beta_2}(x-t)^{1-\alpha_3}dt
  \int_c^t s^{\alpha_2-\beta_1}(1-s)^{-\alpha_1+\beta_1-1}
(t-s)^{-\alpha_2-\beta_2}ds.
\]

The operator corresponding to the second Riemann scheme is
\begin{equation}\label{eq:GHP}
  P_n(\alpha;\beta):=\prod_{j=1}^{n-1}(\vartheta-\beta_j)\cdot \p - \prod_{j=1}^{n}(\vartheta-\alpha_j).
\end{equation}
This is clear when $n=1$. In general, we have
\begin{align*}
 &\RAd(\p^{-\mu})\circ \RAd(x^\gamma)P_n(\alpha,\beta)\\
 &=
 \RAd(\p^{-\mu})\circ \Ad(x^\gamma)\Bigl(\prod_{j=1}^{n-1}x(\vartheta+\beta_j)\cdot \p - \prod_{j=1}^{n}x(\vartheta+\alpha_j)\Bigr)\\
 &= \RAd(\p^{-\mu})\Bigl(\prod_{j=1}^{n-1}(\vartheta+\beta_j-1-\gamma)
  (\vartheta-\gamma) - \prod_{j=1}^{n}x(\vartheta+\alpha_j-\gamma)\Bigr)\\
 &=\Ad(\p^{-\mu})\Bigl(\prod_{j=1}^{n-1}(\vartheta+\beta_j-\gamma)\cdot
  (\vartheta-\gamma+1)\p - \prod_{j=1}^{n}(\vartheta+1)(\vartheta+\alpha_j-\gamma)\Bigr)\\
 &=\prod_{j=1}^{n-1}(\vartheta+\beta_j-\gamma-\mu)\cdot
  (\vartheta-\gamma-\mu+1)\p
 -\prod_{j=1}^{n}(\vartheta+1-\mu)\cdot(\vartheta+\alpha_j-\gamma-\mu)
\end{align*}
and therefore we have \eqref{eq:GHP} by the correspondence of the Riemann 
schemes with $\gamma=\gamma_n$ and $\mu=\mu_n$.

Suppose $\lambda_{1,1}=0$.  We will show that 
\index{hypergeometric equation/function!generalized!local solution}
\begin{equation}\label{eq:HGS}
\begin{split}
  &\sum_{k=0}^\infty
  \frac{\prod_{j=1}^n(\lambda_{2,j}-\lambda_{0,n})_k}
       {\prod_{j=1}^{n-1}(\lambda_{0,n}-\lambda_{0,j}+1)_k k!}
  x^{\lambda_{0,n}+k}\\
  &\quad= x^{\lambda_{0,n}}{}_{n}F_{n-1}\bigl(
   (\lambda_{2,j}-\lambda_{0,n})_{j=1,\dots,n},
   (\lambda_{0,n}-\lambda_{0,j}+1)_{j=1,\dots,n-1};x\bigr)
\end{split}
\end{equation}
is the local solution at the origin corresponding to the exponent
$\lambda_{0,n}$.
Here
\begin{equation}
 {}_nF_{n-1}(\alpha_1,\dots,\alpha_n,\beta_1,\dots,\beta_{n-1};x)
 =\sum_{k=0}^\infty
  \frac{(\alpha_1)_k\cdots(\alpha_{n-1})_k(\alpha_n)_k}{
        (\beta_1)_k\cdots(\beta_{n-1})_kk!}x^k.
\end{equation}

We may assume $\lambda_{0,1}=0$ for the proof of \eqref{eq:HGS}.
When $n=1$, the corresponding solution equals $(1-x)^{-\lambda_{2,1}}$
and we have \eqref{eq:HGS}.  
Note that
\begin{align*}
 &I_0^\mu x^\gamma   \sum_{k=0}^\infty
  \frac{\prod_{j=1}^n(\lambda_{2,j}-\lambda_{0,n})_k}
       {\prod_{j=1}^{n-1}(\lambda_{0,n}-\lambda_{0,j}+1)_k k!}
  x^{\lambda_{0,n}+k}\\
 &=\sum_{k=0}^\infty
  \frac{\prod_{j=1}^n(\lambda_{2,j}-\lambda_{0,n})_k}
       {\prod_{j=1}^{n-1}(\lambda_{0,n}-\lambda_{0,j}+1)_k k!}
  \frac{\Gamma(\lambda_{0,n}+\gamma+k+1)}
    {\Gamma(\lambda_{0,n}+\gamma+\mu+k+1)}x^{\lambda_{0,n}+\gamma+\mu+k}\\
 &=\frac{\Gamma(\lambda_{0,n}+\gamma+1)}
   {\Gamma(\lambda_{0,n}+\gamma+\mu+1)}\sum_{k=0}^\infty
  \frac{\prod_{j=1}^n(\lambda_{2,j}-\lambda_{0,n})_k
        \cdot(\lambda_{0,n}+\gamma+1)_k
        \cdot x^{\lambda_{0,n}+\gamma+\mu+k}}
       {\prod_{j=1}^{n-1}(\lambda_{0,n}-\lambda_{0,j}+1)_k
        \cdot(\lambda_{0,n}+\gamma+\mu+1)_k k!}.
\end{align*}
Comparing \eqref{eq:HGRS0} with 
the first Riemann scheme under $\lambda_{0,1}=\lambda_{1,1}=0$
and $\gamma=\gamma_n$ and $\mu=\mu_n$,
we have the solution \eqref{eq:HGS} by the induction on $n$.
The recurrence relation in Theorem~\ref{thm:shifm1} corresponds to the identity
\index{hypergeometric equation/function!generalized!recurrence relation}
\begin{equation}
\begin{split}
&{}_nF_{n-1}(\alpha_1,\dots,\alpha_{n-1},\alpha_n+1;\beta_1,\dots,\beta_{n-1};x)\\
&\quad={}_nF_{n-1}(\alpha_1,\dots,\alpha_n;\beta_1,\dots,\beta_{n-1};x)\\
&\quad\quad{}+\frac{\alpha_1\cdots\alpha_{n-1}}{\beta_1\cdots\beta_{n-1}}
x\cdot{}_nF_{n-1}(\alpha_1+1,\dots,\alpha_n+1;\beta_1+1,\dots,\beta_{n-1}+1;x).
\end{split}
\end{equation}

The series expansion of the local solution at $x=1$ 
corresponding to the exponent $\gamma'+\mu_1+\cdots+\mu_{n-1}$ is 
a little more complicated.  

For the Riemann scheme
\begin{align*}
&\qquad  \begin{Bmatrix}
   x=\infty&0&1\\
   -\mu_2+1 & [0]_{(2)} & 0\\
   1-\gamma_2-\mu_1-\mu_2 & &\gamma_2+\mu_2\\
   -\gamma'-\gamma_1-\gamma_2-\mu_1-\mu_2
   &\underline{\gamma'+\mu_1+\mu_2}&\gamma_1+\gamma_2+\mu_1+\mu_2
 \end{Bmatrix}\allowdisplaybreaks,
\intertext{we have the local solution at $x=0$}
 &I_0^{\mu_2}(1-x)^{\gamma_2}I_0^{\mu_1}x^{\gamma'}(1-x)^{\gamma_1}
  =I_0^{\mu_2}(1-x)^{\gamma_2}\sum_{n=0}^\infty
  \frac{(-\gamma_1)_n}{n!}x^n
 \\&=
 I_0^{\mu_2}\sum_{n=0}^\infty \frac{\Gamma(\gamma'+1+n)
     (-\gamma_1)_n}{\Gamma(\gamma'+\mu_1+1+n)n!}
   x^{\gamma'+\mu_1+n}(1-x)^{\gamma_2}
\allowdisplaybreaks\\
 &=I_0^{\mu_2}\sum_{m,n=0}^\infty
   \frac{\Gamma(\gamma'+1+n)(-\gamma_1)_n(-\gamma_2)_m}
  {\Gamma(\gamma'+\mu_1+1+n)m!n!}
   x^{\gamma'+\mu_1+m+n}
\allowdisplaybreaks\\
 &=\sum_{m,n=0}^\infty
   \frac{\Gamma(\gamma'+\mu_1+1+m+n)\Gamma(\gamma'+1+n)
   (-\gamma_1)_n(-\gamma_2)_m x^{\gamma'+\mu_1+\mu_2+m+n}}
  {\Gamma(\gamma'+\mu_1+\mu_2+1+m+n)\Gamma(\gamma'+\mu_1+1+n)
   m!n!}
\allowdisplaybreaks\\
 &=\frac{\Gamma(\gamma'+1)x^{\gamma'+\mu_1+\mu_2}}
  {\Gamma(\gamma'+\mu_1+\mu_2+1)}
  \sum_{m,n=0}^\infty
  \frac{(\gamma'+\mu_1+1)_{m+n}(\gamma'+1)_n(-\gamma_1)_n
     (-\gamma_2)_m x^{m+n}}
  {(\gamma'+\mu_1+\mu_2+1)_{m+n}(\gamma'+\mu_1+1)_nm!n!}.
\intertext{Applying the last equality in \eqref{eq:I0H} to 
the above second equality, we have}
 &I_0^{\mu_2}(1-x)^{\gamma_2}I_0^{\mu_1}x^{\gamma'}(1-x)^{\gamma_1}
 \\
 &=\sum_{n=0}^\infty \frac{\Gamma(\gamma'+1+n)(-\gamma_1)_n}
  {\Gamma(\gamma'+\mu_1+1+n)n!}
  x^{\gamma'+\mu_1+\mu_2+n}(1-x)^{-\gamma_2}
\allowdisplaybreaks\\
 &\quad\cdot\sum_{m=0}^\infty
  \frac{\Gamma(\gamma'+\mu_1+1+n)}
  {\Gamma(\gamma'+\mu_1+\mu_2+1+n)}
  \frac{(\mu_2)_m(-\gamma_2)_m}
  {(\gamma'+\mu_1+n+\mu_2+1)_mm!}
  \Bigl(\frac{x}{x-1}\Bigr)^m
\allowdisplaybreaks\\
 &=\frac{\Gamma(\gamma'+1)
  x^{\gamma'+\mu_1+\mu_2}(1-x)^{-\gamma_2}}
  {\Gamma(\gamma'+\mu_1+\mu_2+1)}\!
  \sum_{m,n=0}^\infty\!\!
  \frac{(\gamma'+1)_n(-\gamma_1)_n(-\gamma_2)_m(\mu_2)_m}
  {(\gamma'+\mu_1+\mu_2+1)_{m+n}m!n!}x^n
  \Bigl(\frac x{x-1}\Bigr)^m
\allowdisplaybreaks\\
 &=\frac{\Gamma(\gamma'+1)}
  {\Gamma(\gamma'+\mu_1+\mu_2+1)}\\
 &\quad\cdot
  x^{\gamma'+\mu_1+\mu_2}(1-x)^{-\gamma_2}
  F_3\bigl(-\gamma_2,-\gamma_1,\mu_2,\gamma'+1;\gamma'+\mu_1+\mu_2+1;
   x,\frac{x}{x-1}\bigr),
\end{align*}
where 
$F_3$ is Appell's hypergeometric function \eqref{eq:F3}.

Let $u_1^{-\beta_n}(\alpha_1,\dots,\alpha_n;\beta_1,\dots,\beta_{n-1};x)$ be 
the local solution of $P_n(\alpha,\beta)u=0$ at $x=1$ such that
$u_1^{-\beta_n}
(\alpha;\beta;x)\equiv (x-1)^{-\beta_n}\mod (x-1)^{1-\beta_n}\mathcal O_1$
for generic $\alpha$ and $\beta$.
Since the reduction
\[
 \begin{Bmatrix}
  \lambda_{0,1} & [0]_{(n-1)} & \lambda_{2,1}\\
   \vdots &      & \vdots\\
 \lambda_{0,n} &  \lambda_{1,2} & \lambda_{2,n}
 \end{Bmatrix}
 \xrightarrow{\p_{max}}
 \begin{Bmatrix}
    \lambda_{0,1}' & [0]_{(n-2)} & \lambda_{2,1}'\\
   \vdots &      & \vdots\\
 \lambda_{0,n-1}' &  \lambda_{1,2}' & \lambda_{2,n-1}'
 \end{Bmatrix}
\]
satisfies $\lambda_{1,2}'=\lambda_{1,2}+\lambda_{0,1}+\lambda_{0,2}-1$
and $\lambda_{0,j}'+\lambda_{2,j}'=\lambda_{0,j+1}+\lambda_{2,j+1}$
for $j=1,\dots,n-1$, Theorem~\ref{thm:shifm1} proves
\index{hypergeometric equation/function!generalized!recurrence relation}
\begin{equation}
\begin{split}
 u_1^{-\beta_n}(\alpha;\beta;x)&=
 u_1^{-\beta_n}(\alpha_1,\dots,\alpha_n+1;\beta_1,\dots,\beta_{n-1}+1;x)\\
 &\quad{}+ \frac{\beta_{n-1}-\alpha_n}{1-\beta_n}
 u_1^{1-\beta_n}(\alpha;\beta_1,\dots,\beta_{n-1}+1;x).
\end{split}
\end{equation}

The condition for the irreducibility of the equation equals
\begin{equation}
  \lambda_{0,\nu}+\lambda_{1,1}+\lambda_{2,\nu'}\notin\mathbb Z
  \qquad(1\le\nu\le n,\ 1\le  \nu'\le n),
\end{equation}
which is easily proved by the induction on $n$ (cf.~Example~\ref{exmp:irred} ii)).
The shift operator under a compatible shift $(\epsilon_{j,\nu})$ 
is bijective if and only if 
\begin{equation}
 \lambda_{0,\nu}+\lambda_{1,1}+\lambda_{2,\nu'}
\text{ \ and \ }
\lambda_{0,\nu}+\epsilon_{0,\nu}+\lambda_{1,1}+\epsilon_{1,1}+\lambda_{2,\nu'}
+\epsilon_{2,\nu'}
\end{equation}
are simultaneously not integers or positive integers or non-positive integers 
for each $\nu\in\{1,\dots,n\}$ and $\nu'\in\{1,\dots,n\}$.

Connection coefficients in this example are 
calculated by \cite{Le} and \cite{OTY} etc.
In this paper we get them by Theorem~\ref{thm:c}.

\smallskip
There are the following direct decompositions
$(\nu=1,\dots,n)$.
\begin{gather*}
\begin{split}
1\dots1\overline{1};n-1\underline{1};1\dots1 &=
0\dots0\overline{1};\ \ 1\ \ \ \,\underline{0};0\dots0\overset{\underset{\smallsmile}\nu}10\dots0\\[-5pt]
      &\,\oplus 1\dots1\overline{0};n-2\underline{1};1\dots101\dots1.
\end{split}
\end{gather*}
These $n$ decompositions $\mathbf m=\mathbf m'\oplus\mathbf m''$
satisfy the condition $m'_{0,n_0}=m''_{1,n_1}=1$ in 
\eqref{eq:connection}, where $n_0=n$ and $n_1=2$.
Since $n_0+n_1-2=n$, Remark~\ref{rem:conn} i) shows that 
these decompositions give all the 
decompositions appearing in \eqref{eq:connection}.
Thus we have
\index{hypergeometric equation/function!generalized!connection coefficient}
\begin{gather*}
c(\lambda_{0,n}\rightsquigarrow\lambda_{1,2})
=\frac{\displaystyle\prod_{\nu=1}^{n-1}
  \Gamma({\lambda_{0,n}}-\lambda_{0,\nu}+1)
   \cdot\Gamma(\lambda_{1,1}-{\lambda_{1,2}})}
  {\displaystyle\prod_{\nu=1}^n\Gamma({\lambda_{0,n}}+\lambda_{1,1}+\lambda_{2,\nu})}
=\displaystyle\prod_{\nu=1}^n\frac{\Gamma(\beta_\nu)}
  {\Gamma(\alpha_\nu)}\\
\qquad\qquad\ \ =\lim_{x\to 1-0}(1-x)^{\beta_n}
  {}_nF_{n-1}(\alpha,\beta;x)
  \qquad(\RE \beta_n > 0).
\end{gather*}

Other connection coefficients are obtained by the similar way.
\begin{align*}
c(\lambda_{0,n}\rightsquigarrow\lambda_{2,n}&):\quad \text{When }n=3,
\text{ we have}
\\
 11\overline{1},21,11\underline{1}
 &\!\!=\!00{1},10,10{0}\!
     \quad\,00{1},10,01{0}\!
     \quad\,10{1},11,11{0}\!
     \quad\,01{1},11,11{0}\\[-5pt]
 &\!\oplus\!11{0},11,01{1}\!
  =\!11{0},11,10{1}\!
  =\!01{0},10,00{1}\!
  =\!10{0},10,00{1}
\end{align*}
In general, by the rigid decompositions
\begin{align*}
 1\cdots1\overline{1}\,,\,n-11\,,\,1\cdots1\underline{1}
 &\!=0\cdots0\overline{1}\,,\,\ \ \ 1\,\ \ 0\,,\,0\ldots0\overset{\underset{\smallsmile}i}10\cdots0\underline{0}
\\[-8pt]
 &\oplus1\cdots1\overline{0}\,,\,n-21\,,\,1\cdots101\cdots1\underline{1}\\
 &\!=1\cdots1\overset{\underset\smallsmile i}01\cdots1\overline{1}\,,\,n-21\,,\,1\cdots1\underline{0}
\\[-8pt]
 &\oplus0\ldots010\cdots0\overline{0}\,,\,\ \ \ 1\ \ \,0\,,\,0\cdots0\underline{1}
\end{align*}
for $i=1,\dots,n-1$ we have
\begin{align*}
c(\lambda_{0,n}\rightsquigarrow\lambda_{2,n})
&=\prod_{k=1}^{n-1}\frac{
    \Gamma(\lambda_{2,k}-\lambda_{2,n})}
  {\Gamma\bigl(
   \left|\begin{Bmatrix}
     \lambda_{0,n}
   & \lambda_{1,1}
   & \lambda_{2,k}
   \end{Bmatrix}\right|
   \bigr)}\\
 &\quad\cdot\prod_{k=1}^{n-1}\frac{
    \Gamma(\lambda_{0,n}-\lambda_{0,k}+1)}
  {\Gamma\bigl(
   \left|\begin{Bmatrix}
    (\lambda_{0,\nu})_{\substack{1\le\nu\le n\\ \nu\ne k}}
   & [\lambda_{1,1}]_{(n-2)} & 
   & \!\!\!\!(\lambda_{2,\nu})_{1\le\nu\le n-1}\\
      & \lambda_{1,2}
   \end{Bmatrix}\right|
   \bigr)}\\
 &=\prod_{k=1}^{n-1}\frac{\Gamma(\beta_k)\Gamma(\alpha_k-\alpha_n)}
   {\Gamma(\alpha_k)\Gamma(\beta_k-\alpha_n)}.
\intertext{Moreover we have}
  c(\lambda_{1,2}\!\rightsquigarrow\!\lambda_{0,n})
 &=\frac{\Gamma\bigl(\lambda_{1,2}-\lambda_{1,1}+1\bigr)\cdot
   \prod_{\nu=1}^{n-1}\Gamma\bigl(\lambda_{0,\nu}-\lambda_{0,n}
  \bigr)}
  {\displaystyle\prod_{j=1}^n
  \Gamma\bigl(
   \left|\begin{Bmatrix}
    (\lambda_{0,\nu})_{1\le\nu\le n-1}
   & [\lambda_{1,1}]_{(n-2)} & 
   & \!\!\!\!(\lambda_{2,\nu})_{1\le\nu\le n,\ \nu\ne j}\\
   & \lambda_{1,2}
   \end{Bmatrix}\right|
  \bigr)}\\
 &=\prod_{\nu=1}^n\frac{\Gamma(1-\beta_\nu)}{\Gamma(1-\alpha_\nu)}.
\end{align*}
Here we denote
\[
  (\mu_\nu)_{1\le\nu\le n}
  =\left(\begin{smallmatrix}\mu_1\\ \mu_2\\ \vdots\\ \mu_n\end{smallmatrix}\right)\in\mathbb C^n
  \text{\ \ and \ \ }
  (\mu_\nu)_{\substack{1\le\nu\le n\\ \nu\ne i}}
  =\left(\begin{smallmatrix}\mu_1\\ \vdots\\ \mu_{i-1}\\ \mu_{i+1}\\ \vdots\\ \mu_n\
   \end{smallmatrix}\right)\in\mathbb C^{n-1}
\]
for complex numbers $\mu_1,\dots,\mu_n$.

These connection coefficients were obtained by \cite{Le} and \cite{Yos} etc.

We have
\begin{equation}
 \begin{split}
 {}_nF_{n-1}(\alpha,\beta;x)&=\sum_{k=0}^\infty C_k(1-x)^k +
                   \sum_{k=0}^\infty C'_k(1-x)^{k-\beta_n},\\
 C_0 &= {}_nF_{n-1}(\alpha,\beta;1)\quad(\RE\beta_n < 0),\\
 C'_0 &= \prod_{\nu=1}^n\frac{\Gamma(\beta_\nu)}{\Gamma(\alpha_\nu)}
 \end{split}
\end{equation}
for $0<x<1$ if $\alpha$ and $\beta$ are generic.
Since 
\begin{multline*}
\frac{d^k}{dx^k}{}_nF_{n-1}(\alpha,\beta;x)\\=
 \frac{(\alpha_1)_k\cdots(\alpha_n)_k}
{(\beta_1)_k\cdots(\beta_{n-1})_k}
{}_nF_{n-1}(\alpha_1+k,\dots,\alpha_n+k,\beta_1+k,\dots,\beta_{n-1}+k;x),
\end{multline*}
we have
\begin{equation}
 C_k = 
 \frac{(\alpha_1)_k\cdots(\alpha_n)_k}
{(\beta_1)_k\cdots(\beta_{n-1})_kk!}
 {}_nF_{n-1}(\alpha_1+k,\dots,\alpha_n+k,\beta_1+k,\dots,\beta_{n-1}+k;1).
\end{equation}

We examine the monodromy generators for the solutions of the generalized 
hypergeometric equation.
For simplicity we assume $\beta_i\notin\mathbb Z$ and 
$\beta_i-\beta_j\notin\mathbb Z$ for $i\ne j$.
Then $u=(u_0^{\lambda_{0,1}},\ldots,u_0^{\lambda_{0,n}})$
is a base of local solution at $0$ and the corresponding monodromy generator
around $0$ with respect to this base equals
\[
  M_0=\begin{pmatrix}
       e^{2\pi\sqrt{-1}\lambda_{0,1}}\\
       &\ddots\\
       &&e^{2\pi\sqrt{-1}\lambda_{0,n}}\\
      \end{pmatrix}
\]
and that around $\infty$ equals
\begin{align*}
  M_\infty &=\biggl(\sum_{k=1}^n
  e^{2\pi\sqrt{-1}\lambda_{2,\nu}}
  c(\lambda_{0,i}\rightsquigarrow\lambda_{2,k})
  c(\lambda_{2,k}\rightsquigarrow\lambda_{k,j})
  \biggr)_{\substack{1\le i\le n\\1\le j\le n}}\\
  &=\biggl(\sum_{k=1}^n
   e^{2\pi{\sqrt{-1}\lambda_{2,\nu}}}
  \prod_{\nu\in\{1,\dots,n\}\setminus\{k\}}
  \frac{\sin2\pi
   (\lambda_{0,i}+\lambda_{1,1}+\lambda_{2,\nu})}
   {\sin2\pi(\lambda_{0,k}-\lambda_{0,\nu})}\\
  &\quad\cdot
 \prod_{\nu\in\{1,\dots,n\}\setminus\{j\}}\!\!
 \frac{\sin2\pi(\lambda_{0,i}+\lambda_{1,1}+\lambda_{2,\nu})}
   {\sin2\pi(\lambda_{2,j}-\lambda_{2,\nu})}
    \biggr)_{\substack{1\le i\le n\\1\le j\le n}}.
\end{align*}

Lastly we remark that the versal generalized hypergeometric 
operator is
\index{hypergeometric equation/function!generalized!versal}
\[
\begin{split}
\tilde P&=\RAd(\p^{-\mu_{n-1}})\circ
   \RAd\bigl((1-c_1x)^{\frac{\gamma_{n-1}}{c_1}}\bigr)
   \circ\cdots
   \circ\RAd(\p^{-\mu_1})\\
    &\quad\circ
   \RAd\left((1-c_1x)^{\frac{\gamma_1}{c_1}+\frac{\gamma'}{c_1(c_1-c_2)}}
   (1-c_2x)^{\frac{\gamma'}{c_2(c_2-c_1)}}\right)\p\\
 &=\RAd(\p^{-\mu_{n-1}})\circ
   \RAdei\bigl(\frac{\gamma_{n-1}}{1-c_1x}\bigr)
   \circ\cdots
   \circ\RAd(\p^{-\mu_1})\\
    &\quad\circ
   \RAdei\left(\frac{\gamma_1}{1-c_1x}
   + \frac{\gamma'x}{(1-c_1x)(1-c_2x)}\right)\p
\end{split}
\] 
and when $n=3$, we have the integral representation of the solutions
\[
\begin{split}
  \int_c^x\int_c^t\exp\Bigl(
  -\int_c^s\frac{\gamma_1(1-c_2u)+\gamma'u}{(1-c_1u)(1-c_2u)}du\Bigr)
   (t-s)^{\mu_1-1}\bigl(1-c_1t\bigr)^{\frac{\gamma_2}{c_1}}
   (x-t)^{\mu_2-1} ds\,dt.
\end{split}
\]
Here $c$ equals $\frac1{c_1}$ or $\frac1{c_2}$ or $\infty$.
\subsection{Even/Odd family} $EO_n$%
\label{sec:EOEx}%
\index{00EOn@$EO_n$}

The system of differential equations of Schlesinger canonical form
belonging to an even or odd family is concretely given by \cite{Gl}.
We will examine concrete connection coefficients of solutions of
the single differential equation  belonging to an even or odd family.
The corresponding tuples of partitions and their reductions and
decompositions are as follows.
\begin{align*}
 m+1m,m^21,1^{2m+1}&=10,10,1\oplus m^2,mm-11,1^{2m}\\
                &=1^2,1^20,1^2\oplus mm-1,(m-2)^21,1^{2m-1}\\
 m^2,mm-11,1^{2m}&=1,100,1\oplus mm-1,(m-1)^21,1^{2m-1}\\
                 &=1^2,110,1^2\oplus (m-1)^2,m-1m-21,1^{2m-2}\\
 EO_n &= H_1\oplus EO_{n-1}:2n = H_2\oplus EO_{n-2}:\binom n2\\
[\Delta(\mathbf m)]&=1^{\binom n2+2n}\\
 EO_n &\overset{1}{\underset{R1E0R0E0}\longrightarrow}EO_{n-1}\\
 EO_2&=H_2,\quad EO_3=H_3
\end{align*}
The following operators are shift operators of the universal model $P_{EO_n}(\lambda)u=0$:
\[
 P_{H_1}^2(\lambda),\ P_{EO_{n-1}}^1(\lambda),\ P_{EO_{n-1}}^2(\lambda),\ 
 P_{H_2}^2(\lambda),\ P_{EO_{n-2}}^1(\lambda),\ P_{EO_{n-2}}^2(\lambda).
\]
\noindent
$EO_{2m}$ ($\mathbf m=(1^{2m},mm-11,mm)$ :\ even family)
\index{even family}
\begin{gather*}
 \begin{Bmatrix}
     x=\infty & 0 & 1\\
     \lambda_{0,1} & [\lambda_{1,1}]_{(m)}   & [\lambda_{2,1}]_{(m)}\\
     \vdots        & [\lambda_{1,2}]_{(m-1)} & [\lambda_{2,2}]_{(m)}\\
     \lambda_{0,2m} & \lambda_{1,3}
    \end{Bmatrix},\\
\sum_{\nu=1}^{2m}\lambda_{0,\nu}
+m(\lambda_{1,1}+\lambda_{2,1}+\lambda_{2,2})
+(m-1)\lambda_{1,2}+\lambda_{1,3}=2m-1.
\end{gather*}
The rigid decompositions
\begin{align*}
 &1\cdots1\overline1\,,\,mm-1\underline1\,,\,mm\\
  &=0\cdots0\overline1\,,\,10\underline0\,,\,
    \overset{\underset{\smallsmile} i}10\oplus
    1\cdots1\overline0\,,\,m-1m-1\underline1\,,
    \,\overset{\underset{\smallsmile} i}01\\[-3pt]
  &=0\cdots\overset{\underset{\smallsmile} j}1
    \overline1\,,\,11\underline0\,,\,11\oplus
    1\cdots\overset{\underset{\smallsmile} j}0
    \overline0\,,\,m-1m-2\underline1\,,\,m-1m-1,
\end{align*}
which are expressed by 
$EO_{2m}=H_1\oplus EO_{2m-1}=H_2\oplus EO_{2m-2}$,
give
\index{even family!connection coefficient}
\begin{align*}
 c(\lambda_{0,2m}\rightsquigarrow\lambda_{1,3})
 &=\prod_{i=1}^2\frac
 {\Gamma\bigl(\lambda_{1,i}-\lambda_{1,3}\bigr)
  }{\Gamma\bigl(
    \left|\begin{Bmatrix}
   \lambda_{0,2m} &\ \lambda_{1,1}\ &\lambda_{2,i}
    \end{Bmatrix}\right|
  \bigr)}
  \cdot
   \prod_{j=1}^{2m-1}\frac
  {\Gamma\bigl(\lambda_{0,2m}-\lambda_{0,j}+1)}
  {\Gamma\bigl(
    \left|\begin{Bmatrix}
    \lambda_{0,j} & \lambda_{1,1} & \lambda_{2,1}\\
    \lambda_{0,2m} &\ \lambda_{1,2}\ & \lambda_{2,2}
  \end{Bmatrix}\right|
   \bigr)},\allowdisplaybreaks\\
 c(\lambda_{1,3}\rightsquigarrow\lambda_{0,2m})
 &=\displaystyle\prod_{i=1}^2\frac
 {\Gamma\bigl(\lambda_{1,3}-\lambda_{1,i}+1\bigr)
  }{\Gamma\bigl(
   \left|\begin{Bmatrix}
    &[\lambda_{1,1}]_{(m-1)}&[\lambda_{2,\nu}]_{(m)}\\
    (\lambda_{0,\nu})_{1\le\nu\le 2m-1}
   &\ [\lambda_{1,2}]_{(m-1)}\ &[\lambda_{2,3-i}]_{(m-1)}
     \\
    &\lambda_{1,3}
  \end{Bmatrix}\right|
  \bigr)}
  \\
 &\quad
  \cdot
   \prod_{j=1}^{2m-1}\frac
  {\Gamma\bigl(\lambda_{0,j}-\lambda_{0,2m})}
  {\Gamma\bigl(
     \left|\begin{Bmatrix}
   & [\lambda_{1,1}]_{(m-1)}&[\lambda_{2,1}]_{(m-1)}\\
   (\lambda_{0,\nu})_{\substack{1\le\nu\le 2m-1\\ \nu\ne j}}
  &\ [\lambda_{1,2}]_{(m-2)}\ &[\lambda_{2,2}]_{(m-1)}\\[-4pt]
  &\lambda_{1,3}&
  \end{Bmatrix}\right|
   \bigr)}.
\end{align*}
These formulas were obtained by the author in 2007 (cf.~\cite{O3}), 
which is a main motivation for the study in this paper.
The condition for the irreducibility is
\index{even family!reducibility}
\begin{equation*}
 \begin{cases}
  \lambda_{0,\nu}+\lambda_{1,1}+\lambda_{2,k}\notin\mathbb Z
  &(1\le\nu\le 2m,\ k=1,2),\\
  \lambda_{0,\nu}+\lambda_{0,\nu'}+\lambda_{1,1}+\lambda_{1,2}+\lambda_{2,1}
  +\lambda_{2,2}-1\notin\mathbb Z
    &(1\le\nu<\nu'\le 2m,\ k=1,2).
 \end{cases}
\end{equation*}
The shift operator for a compatible shift $(\epsilon_{j,\mu})$ is bijective
if and only if the values of each linear function in the above 
satisfy \eqref{eq:non-pos}.

For the Fuchsian equation $\tilde Pu=0$ of type $EO_4$ 
with the Riemann scheme\index{even family!$EO_4$}
\begin{align}\label{eq:EO4GRS}
\begin{split}
 &\begin{Bmatrix}
x=\infty & 0 & 1\\
[a_1]_{(2)} & b_1 & [0]_{(2)}&\!\!;x\\
[a_2]_{(2)} & b_2 & c_1\\
 & b_3 & c_2\\
 & 0 & \\
\end{Bmatrix}
\end{split}\end{align}
and the Fuchs relation
\begin{align}\label{eq:e4Fuchs}\begin{split}
 &2a_1+2a_2+b_1+b_2+b_3+c_1+c_2=3
\end{split}\end{align}
we have the connection formula
\begin{align}\begin{split}
 &c(0\!:\!0 \rightsquigarrow 1\!:\!c_2)
 =\frac{
  \Gamma(c_1-c_2)
  \Gamma(-c_2)
  \prod_{\nu=1}^3\Gamma(1-b_\nu)
 }{
  \Gamma(a_1)
  \Gamma(a_2)
  \prod_{\nu=1}^3\Gamma(a_1+a_2+b_\nu+c_1-1)
 }.
\end{split}\label{eq:Ceo4}\end{align}

Let $\tilde Q$ be the Gauss hypergeometric operator with the Riemann scheme
\[
 \begin{Bmatrix}
 x=\infty & 0 & 1\\
 a_1 &  1-a_1-a_2-c_1  & 0\\
 a_2 &  0 & c_1
 \end{Bmatrix}.
\]
We may normalize the operators by
\[
 \tilde P=x^3(1-x)\p^4+\cdots\text{ \ and \ }
 \tilde Q=x(1-x)\p^2+\cdots.
\]
Then
\[
 \begin{split}
 \tilde P&=\tilde S\tilde Q -\prod_{\nu=1}^3(a_1+a_2+b_\nu+c_1-1)\cdot\p\\
 \tilde Q&=\bigl(x(1-x)\p+(a_1+a_2+c_1-(a_1+a_2+1)x)\bigr)\p-a_1a_2
 \end{split}
\]
with a suitable $\tilde S,\ \tilde T\in W[x]$ and $e\in\mathbb C$ 
and as is mentioned in Theorem~\ref{thm:sftUniv}, $\tilde Q$ is a 
shift operator satisfying\index{even family!$EO_4$!shift operator}%
\index{shift operator}
\begin{equation}\label{eq:EO4sht}
\begin{Bmatrix}
x=\infty & 0 & 1\\
[a_1]_{(2)} & b_1 & [0]_{(2)}&\!\!;x\\
[a_2]_{(2)} & b_2 & c_1\\
 & b_3 & c_2\\
 & 0 & \\
\end{Bmatrix} \ \xrightarrow{\tilde Q} \ 
\begin{Bmatrix}
x=\infty & 0 & 1\\
[a_1+1]_{(2)} & b_1-1 & [0]_{(2)}&\!\!;x\\
[a_2+1]_{(2)} & b_2-1 & c_1\\
 & b_3-1 & c_2-1\\
 & 0 & \\
\end{Bmatrix}.
\end{equation}
Let $u_0^0=1+\cdots$ and $u_1^{c_2}=(1-x)^{c_2}+\cdots$ be 
the normalized local solutions of
$\tilde Pu=0$ corresponding to the characteristic exponents
$0$ at $0$ and $c_2$ at $1$, respectively.
Then the direct calculation shows
\[
 \begin{split}
 \tilde Q u_0^0&=\frac{a_1a_2\prod_{\nu=1}^3(a_1+a_2+b_\nu+c_1-1)}
  {\prod_{\nu=1}^3(1-b_\nu)}+\cdots,\\
 \tilde Qu_1^{c_2}&=c_2(c_2-c_1)(1-x)^{c_2-1}+\cdots.
 \end{split}
\]
Denoting by $c(a_1,a_2,b_1,b_2,b_3,c_1,c_2)$ 
the connection coefficient $c(0\!:\!0 \rightsquigarrow 1\!:\!c_2)$ 
for the equation with the Riemann scheme \eqref{eq:EO4GRS}, we have
\[
 \frac{c(a_1,a_2,b_1,b_2,b_3,c_1,c_2)}
      {c(a_1+1,a_2+1,b_1-1,b_2-1,b_3-1,c_1,c_2-1)}
 =\frac{a_1a_2\displaystyle\prod_{\nu=1}^3(a_1+a_2+b_\nu+c_1-1)}
       {(c_1-c_2)(-c_2)\displaystyle\prod_{\nu=1}^3(1-b_\nu)},
\]
which proves \eqref{eq:Ceo4} since
$\lim_{k\to\infty}c(a_1+k,a_2+k,b_1-k,b_2-k,b_3-k,c_1,c_2-k)=1$.
Note that the shift operator \eqref{eq:EO4sht} is not bijective if
and only if
\[
 \tilde Qu=\prod_{\nu=1}^3(a_1+a_2+b_\nu+c_1-1)\cdot\p u=0
\]
has a non-zero solution, which is equivalent to
\[
 a_1a_2\prod_{\nu=1}^3(a_1+a_2+b_\nu+c_1-1)=0.
\]

By the transformation $x\mapsto\frac{x}{x-1}$ we have
\begin{align*}
&\begin{Bmatrix}
  x = \infty & 0 & 1\\
  [0]_{(2)} & 0 & [a_1]_{(2)}\\
  c_1 & b_1 & [a_2]_{(2)}\\
  c_2 & b_2\\
      & b_3
 \end{Bmatrix}\\[-.3cm]
&\qquad \xrightarrow{(1-x)^{a_1}\p^{1-a_1}(1-x)^{-a_1}}
 \begin{Bmatrix}
 x = \infty & 0 & 1\\
 2-2a_1 &   & a_1\\
 1+c_1-a_1 & a_1+b_1-1  & [a_1+a_2-1]_{(2)}\\
 1+c_2-a_1 & a_1+b_2-1 \\
           & a_1+b_3-1
 \end{Bmatrix}\\
&\qquad \xrightarrow{x^{1-a_1-b_1}(1-x)^{1-a_1-a_2}}
\begin{Bmatrix}
 x = \infty & 0 & 1\\
 a_2+b_1 &   & 1-a_2\\
 a_1+a_2+b_1+c_1-1 & 0  & [0]_{(2)}\\
 a_1+a_2+b_1+c_2-1 & b_2-b_1 \\
           & b_3-b_1
 \end{Bmatrix}
\end{align*}
and therefore Theorem~\ref{thm:GC} gives the following connection formula for 
\eqref{eq:EO4GRS}:
\begin{align*}
 c(0\!:\!b_1\rightsquigarrow \infty\!:\!a_2)&=
 \frac{\Gamma(b_1+1)\Gamma(a_1-a_2)}{\Gamma(a_1+b_1)\Gamma(1-a_2)}\cdot
 {}_3F_2(a_2+b_1, a_1+a_2+b_1+c_1-1,
 \\&\qquad\qquad a_1+a_2+b_1+c_2-1; b_1-b_2-1, b_1-b_3-1;1).
\intertext{In the same way, we have}
  c(1\!:\!c_1\rightsquigarrow \infty\!:\!a_2)&=
 \frac{\Gamma(c_1+1)\Gamma(a_1-a_2)}{\Gamma(a_1+c_1)\Gamma(1-a_2)}\cdot
 {}_3F_2(b_1-c_1,b_2-c_1,b_3-c_1;
 \\&\qquad\qquad a_1+c_1, c_1-c_2+1;1).
\end{align*} 

\index{Wronskian}%
\index{connection coefficient!generalized}%
\index{even family!connection coefficient!generalized}%
We will calculate generalized connection coefficients
defined in Definition~\ref{def:GC}.
In fact, we get
\begin{align}
 c(1\!:\![0]_{(2)}\rightsquigarrow \infty\!:\![a_2]_{(2)})&=
 \frac{\prod_{\nu=1}^2\Gamma(2-c_\nu)\cdot\prod_{i=1}^2\Gamma(a_1-a_2+i)}
 {\Gamma(a_1)\prod_{\nu=1}^3\Gamma(a_1+b_\nu)},\label{eq:e4dc1}\\
 c(\infty\!:\![a_2]_{(2)}\rightsquigarrow1\!:\![0]_{(2)})&=
 \frac{\prod_{\nu=1}^2\Gamma(c_\nu-1)\cdot\prod_{i=0}^1\Gamma(a_2-a_1-i)}
 {\Gamma(1-a_1)\prod_{\nu=1}^3\Gamma(1-a_1-b_\nu)}\label{eq:e4dc2}
\end{align}
according to the procedure given in Remark~\ref{rem:Cproc},
which we will explain.

The differential equation with the Riemann scheme
$\begin{Bmatrix}
 x=\infty & 0 & 1\\
 \alpha_1 &[0]_{(2)}&[0]_{(2)}\\
 \alpha_2 &[\beta]_{(2)}&\gamma_1\\
 \alpha_3 &&\gamma_2\\
 \alpha_4 \\
\end{Bmatrix}$ is  $Pu=0$ with
\index{tuple of partitions!rigid!1111,211,22}
\begin{equation}\begin{split}
 P=&\prod_{j=1}^4\bigl(\vartheta+\alpha_j\bigr)
 +\p\bigl(\vartheta-\beta\bigr)\bigl((\p-2\vartheta+\gamma_1+\gamma_2-1)(\vartheta-\beta)\\
 &+\sum_{1\le i<j\le 3}\!\!\alpha_i\alpha_j
 -(\beta-2\gamma_1-2\gamma_2-4)(\beta-1)
 -\gamma_1\gamma_2+1\bigr).
\end{split}\end{equation}
The equation $Pu=0$ is isomorphic to the
system
\begin{equation}
\begin{gathered}
 \frac{d\tilde u}{dx}=\frac{A}x\tilde u+\frac{B}{x-1}\tilde u,\\
 A=\begin{pmatrix}
    0 & 0 & 1 & 0\\
    0 & 0 & 0 & 1\\
    0 & 0 & c & 0\\
    0 & 0 & 0 & c
    \end{pmatrix},\ 
 B=\begin{pmatrix}
    0 & 0 & 0 & 0\\
    0 & 0 & 0 & 0\\
    s & 1 & a & 0\\
    r & t & 0 & b
    \end{pmatrix},\ 
 \tilde u=\begin{pmatrix}
    u_1\\u_2\\u_3\\u_4
 \end{pmatrix}
\end{gathered}
\end{equation}
by the correspondence
\begin{align*}
 &\begin{cases}u_1=u,\\
 u_2=(x-1)xu''+\bigl((1-a-c)x+a-1\bigr)u'-su,\\
 u_3=xu',\\
 u_4=x^2(x-1)u'''+\bigl((3-a-c)x^2+(a-2)x\bigr)u''+(1-a-c-s)xu',
 \end{cases}
\end{align*}
where we may assume $\RE\gamma_1\ge\RE\gamma_2$ and
\[
 \begin{split}
 \beta&=c,\ 
 \gamma_1=a+1,\ 
 \gamma_2=b+2,\\
 \prod_{\nu=1}^4(\xi-\alpha_\nu)&=\xi^4+(a+b+2c)\xi^3
  +\bigl((a+c)(b+c)-s-t\bigr)\xi^2\\
 &\quad{}-\bigl((b+c)s+(a+c)t\bigr)\xi+st-r.
 \end{split}
\]
Here $s$, $t$ and $r$ are uniquely determined from
$\alpha_1,\alpha_2,\alpha_3,\alpha_4,\beta,\gamma_1,\gamma_2$ because $b+c\ne a+c$.
We remark that $\Ad(x^{-c})\tilde u$ satisfies a system of Okubo normal form.

Note that the shift of parameters 
$(\alpha_1,\dots,\alpha_4,\beta,\gamma_1,\gamma_2)
\mapsto(\alpha_1,\dots,\alpha_4,\beta-1,\gamma_1+1,\gamma_2+1)$ corresponds to
the shift $(a,b,c,s,t,r)\mapsto(a+1,b+1,c-1,s,t,r)$.

Let $u^j_{\alpha_1,\dots,\alpha_4,\beta,\gamma_1,\gamma_2}(x)$ 
be local holomorphic 
solutions of $Pu=0$ in a neighborhood of $x=0$ determined by
\begin{align*} 
 u^j_{\alpha_1,\dots,\alpha_4,\beta,\gamma_1,\gamma_2}(0)&=\delta_{j,0},\\
 \bigl(\tfrac{d}{dx}
 u^j_{\alpha_1,\dots,\alpha_4,\beta,\gamma_1,\gamma_2}\bigr)
 (0)&=\delta_{j,1}
\end{align*}
for $j=0$ and $1$.
Then Theorem~\ref{thm:paralimits} proves
\begin{align*}
 \lim_{k\to\infty}\tfrac{d^\nu}{dx^\nu}
 u^0_{\alpha,\beta-k,\gamma_1+k,\gamma_1+k}(x)&=\delta_{0,\nu}
 \quad(\nu=0,1,2,\ldots)
\end{align*}
uniformly on $\overline D=\{x\in\mathbb C\,;\,|x|\le 1\}$.

Put $u=v_{\alpha,\beta,\gamma_1,\gamma_2}=
(\gamma_1-2)^{-1}u^1_{\alpha,\beta,\gamma}$.
Then Theorem~\ref{thm:paralimits} proves
\begin{align*}
  &\lim_{k\to\infty}\tfrac{d^\nu}{dx^\nu}
   v_{\alpha,\beta-k,\gamma_1+k,\gamma_2+k}(x)=0\quad(\nu=0,1,2,\ldots),\\
  &\lim_{k\to\infty}\Bigl(
(x-1)x\tfrac{d^2}{dx^2}+\bigl((2-\beta-\gamma_1)x+\gamma_1+k-2\bigr)\tfrac{d}{dx}
 -s\Bigr)
   v_{\alpha,\beta-k,\gamma_1+k,\gamma_2+k}(x)=1
\end{align*}
uniformly on $\overline D$. Hence
\begin{align*}
 \lim_{k\to\infty}\bigl(\tfrac{d}{dx}u^1_{\alpha,\beta-k,\gamma+1,\gamma_1+k}\bigr)
  (x)=1
\end{align*}
uniformly on $\overline D$.
Thus we obtain
\[
 \lim_{k\to\infty}
 c(\infty\!:\![a_2]_{(2)}\rightsquigarrow1\!:\![0]_{(2)})
 |_{a_1\mapsto a_1-k,\ c_1\mapsto c_1+k,\ c_2\mapsto c_2+k}=1
\]
for the connection coefficient in \eqref{eq:e4dc2}.
Then the procedure given in Remark~\ref{rem:Cproc} and
Corollary~\ref{cor:zeroC} with the rigid decompositions
\begin{align*}
2\underline{2},1111,\overline{2}11
           &=1\underline{2},0111,\overline{1}11
              \oplus 1\underline{0},1000,100
           =1\underline{2},1011,\overline{1}11
              \oplus 1\underline{0},0100,\overline{1}00\\
           &=1\underline{2},1101,\overline{1}11
               \oplus 1\underline{0},0010,\overline{1}00
           =1\underline{2},1101,\overline{1}11
                \oplus 1\underline{0},0010,\overline{1}00
\end{align*}
prove \eqref{eq:e4dc2}.
Corresponding to Remark~\ref{rem:Cproc} (4), we note
\[ \sum_{\nu=1}^2(c_\nu-1)+\sum_{i=0}^1(a_2-a_1-i)=
 (1-a_1)+\sum_{\nu=1}^3(1-a_1-b_\nu)
\]
because of the Fuchs relation \eqref{eq:e4Fuchs}.
We can similarly obtain \eqref{eq:e4dc1}.

The holomorphic solution of $\tilde Pu=0$ at the origin is given by
\begin{align*}\begin{split}
 &u_0(x)=\sum_{m\ge0,\ n\ge0}\\
 &\frac{(a_1+a_2+b_3+c_2-1)_{n}
 \prod_{\nu=1}^2\bigl((a_\nu)_{m+n}(a_1+a_2+b_\nu+c_1-1)_{m}\bigr)
}
 {(1-b_1)_{m+n}(1-b_2)_{m+n}(1-b_3)_{m}m!n!}x^{m+n}
\end{split}\end{align*}
and it has the integral representation
\begin{align*}\begin{split}
 u_0(x)&=\frac{
  \prod_{\nu=1}^3\Gamma(1-b_\nu)
 }{
  \prod_{\nu=1}^2\bigl(\Gamma(a_\nu)\Gamma(1-a_\nu-b_\nu)
  \Gamma(b_\nu+c_\nu+a_1+a_2-1)\bigr)}\\
 &\quad\int_0^{x}\int_0^{s_0}\int_0^{s_1}x^{b_1}(x-s_0)^{-b_1-a_1}
 s_0^{b_2+a_1-1}(s_0-s_1)^{-b_2-a_2}\\
 &\quad\cdot s_1^{b_3+a_2-1}(1-s_1)^{-b_3-c_1-a_2-a_1+1}
  (s_1-s_2)^{c_1+b_1+a_2+a_1-2}\\
 &\quad\cdot s_2^{b_2+c_2+a_2+a_1-2}(1-s_2)^{-c_2-b_1-a_2-a_1+1}ds_2ds_1ds_0.
\end{split}\end{align*}
The equation is irreducible if and only if 
any value of the following linear functions is not an integer.
\index{even family!reducibility}
\begin{align*}\begin{split}
 &a_1\quad a_2\\%
 &a_1+b_1\quad a_1+b_2\quad a_1+b_3\quad a_2+b_1\quad a_2+b_2\quad a_2+b_3\\%
 &a_1+a_2+b_1+c_1-1\quad a_1+a_2+b_1+c_2-1\quad a_1+a_2+b_2+c_1-1\\
 &a_1+a_2+b_2+c_2-1\quad a_1+a_2+b_3+c_1-1\quad a_1+a_2+b_2+c_2-1.
\end{split}\end{align*}

In the same way we have the connection coefficients for
odd family.

\noindent
\index{odd family!connection coefficient}
$EO_{2m+1}$ ($\mathbf m=(1^{2m+1},mm1,m+1m)$ :\ odd family)
\index{odd family}
\[
 \begin{Bmatrix}
     x=\infty & 0 & 1\\
     \lambda_{0,1} & [\lambda_{1,1}]_{(m)}   & [\lambda_{2,1}]_{(m+1)}\\
     \vdots        & [\lambda_{1,2}]_{(m)} & [\lambda_{2,2}]_{(m)}\\
     \lambda_{0,2m+1} & \lambda_{1,3}
    \end{Bmatrix}
\]
\hfill$\sum_{\nu=1}^{2m+1}\lambda_{0,\nu}
+m(\lambda_{1,1}+\lambda_{1,2}+\lambda_{2,2})
+(m+1)\lambda_{2,1}+\lambda_{1,3}=2m.$\hfill\phantom{.}
\begin{align*}
 c(\lambda_{0,2m+1}\rightsquigarrow\lambda_{1,3})
 &=\prod_{k=1}^2
   \frac{\Gamma\bigl(\lambda_{1,k}-\lambda_{1,3}\bigr)
  }{\Gamma\bigl(
    \left|\begin{Bmatrix}
   \lambda_{0,2m+1} & \lambda_{1,k} &\lambda_{2,1}
    \end{Bmatrix}\right|
  \bigr)}
  \\
 &\quad
  \cdot
   \prod_{k=1}^{2m}\frac
  {\Gamma\bigl(\lambda_{0,2m+1}-\lambda_{0,k}+1)}
  {\Gamma\bigl(
    \left|\begin{Bmatrix}
    \lambda_{0,k} & \lambda_{1,1} & \lambda_{2,1}\\
    \lambda_{0,2m+1} & \lambda_{1,2} & \lambda_{2,2}
  \end{Bmatrix}\right|
   \bigr)},\allowdisplaybreaks\\
 c(\lambda_{1,3}\rightsquigarrow\lambda_{0,2m+1})
 &=\displaystyle\prod_{k=1}^2\frac
 {\Gamma\bigl(\lambda_{1,3}-\lambda_{1,k}+1\bigr)
  }{\Gamma\bigl(
   \left|\begin{Bmatrix}
    &[\lambda_{1,k}]_{(m)}&[\lambda_{2,1}]_{(m)}\\
    (\lambda_{0,\nu})_{1\le \nu\le 2m}
   &[\lambda_{1,3-k}]_{(m-1)}&[\lambda_{2,2}]_{(m)}
     \\
    &\lambda_{1,3}
  \end{Bmatrix}\right|
  \bigr)}
  \\
 &\quad
  \cdot
   \prod_{k=1}^{2m}\frac
  {\Gamma\bigl(\lambda_{0,k}-\lambda_{0,2m+1})}
  {\Gamma\bigl(
     \left|\begin{Bmatrix}
   & [\lambda_{1,1}]_{(m-1)}&[\lambda_{2,1}]_{(m)}\\
   (\lambda_{0,\nu})_{\substack{1\le \nu\le 2m\\ \nu\ne k}}
  &[\lambda_{1,2}]_{(m-1)}&[\lambda_{2,2}]_{(m-1)}\\
  &\lambda_{1,3}&
  \end{Bmatrix}\right|
   \bigr)}.
\end{align*}

The condition for the irreducibility is
\index{odd family!reducibility}
\begin{equation*}
 \begin{cases}
  \lambda_{0,\nu}+\lambda_{1,k}+\lambda_{2,1}\notin\mathbb Z
  &(1\le\nu\le 2m+1,\ k=1,2),\\
  \lambda_{0,\nu}+\lambda_{0,\nu'}+\lambda_{1,1}+\lambda_{1,2}+\lambda_{2,1}
  +\lambda_{2,2}-1\notin\mathbb Z
    &(1\le\nu<\nu'\le 2m+1,\ k=1,2).
 \end{cases}
\end{equation*}
The same statement using the above linear functions
as in the case of even family is valid for the bijectivity of 
the shift operator with respect to compatible shift $(\epsilon_{j,\nu})$.

We note that the operation $\RAd(\p^{-\mu})\circ
\RAd\bigl(x^{-\lambda_{1,2}}(1-x)^{-\lambda_{2,2}}\bigr)$ transforms 
the operator and solutions 
with the above Riemann scheme of type $EO_n$ into those of type $EO_{n+1}$:
\begin{align*}
 &\begin{Bmatrix}
  \lambda_{0,1} & [\lambda_{1,1}]_{([\frac n2])}
   &[\lambda_{2,1}]_{([\frac{n+1}2])}\\
  \vdots & [\lambda_{1,2}]_{([\frac {n-1}2])}
   &[\lambda_{2,2}]_{([\frac{n}2])}\\
   \lambda_{0,n}&\lambda_{1,3}\\
 \end{Bmatrix}
\allowdisplaybreaks\\ 
&\xrightarrow{x^{-\lambda_{1,2}}(1-x)^{-\lambda_{2,2}}}
\begin{Bmatrix}
  \lambda_{0,1}+\lambda_{1,2}+\lambda_{2,2} & [\lambda_{1,1}-\lambda_{1,2}]_{([\frac n2])}
   &[\lambda_{2,1}-\lambda_{2,2}]_{([\frac{n+1}2])}\\
  \vdots & [0]_{([\frac {n-1}2])}
   &[0]_{([\frac{n}2])}\\
   \lambda_{0,n}+\lambda_{1,2}+\lambda_{2,2}&\lambda_{1,3}-\lambda_{1,2}\\
 \end{Bmatrix}
\allowdisplaybreaks\\
&\xrightarrow{\p^{-\mu}}
\begin{Bmatrix}
  \lambda_{0,1}+\lambda_{1,2}+\lambda_{2,2}-\mu 
   & [\lambda_{1,1}-\lambda_{1,2}+\mu]_{([\frac n2])}
    &[\lambda_{2,1}-\lambda_{2,2}+\mu]_{([\frac{n+1}2])}\\
  \vdots & [\mu]_{([\frac {n+1}2])}
   &[\mu]_{([\frac{n+2}2])}\\
  \lambda_{0,n}+\lambda_{1,2}+\lambda_{2,2}-\mu
   &\lambda_{1,3}-\lambda_{1,2}+\mu\\
  1-\mu
 \end{Bmatrix}.
\end{align*}
\subsection{Trigonometric identities}\label{sec:TriEx}
The connection coefficients corresponding to the Riemann scheme of 
the hypergeometric family in \S\ref{sec:GHG} satisfy
\[
\begin{split}
\sum_{\nu=1}^nc(1:\lambda_{1,2}\!\rightsquigarrow\!0:\lambda_{0,\nu})
\cdot c(0:\lambda_{1,\nu}\!\rightsquigarrow\!1:\lambda_{1,2})=1,\\
\sum_{\nu=1}^nc(\infty:\lambda_{2,i}\!\rightsquigarrow\!0:\lambda_{0,\nu})
\cdot c(0:\lambda_{0,\nu}\!\rightsquigarrow\!\infty:\lambda_{2,j})=\delta_{ij}.
\end{split}
\]
These equations with Remark~\ref{rem:conn} iii) give the identities
\begin{gather*}
\sum_{k=1}^n\frac{\prod_{\nu\in\{1,\dots,n\}}\sin(x_k-y_\nu)}
  {\prod_{\nu\in\{1,\ldots,n\}\setminus\{k\}}\sin(x_k-x_\nu)}
  = \sin\Bigl(\sum_{\nu=1}^n x_\nu-\sum_{\nu=1}^n y_\nu\Bigr),\\
\sum_{k=1}^n\prod_{\nu\in\{1,\dots,n\}\setminus\{k\}}
 \frac{\sin(y_i-x_\nu)}{\sin(x_k-x_\nu)}
 \prod_{\nu\in\{1,\dots,n\}\setminus\{j\}}
 \frac{\sin(x_k-y_\nu)}{\sin(y_j-y_\nu)}
  = \delta_{ij}\quad(1\le i,\,j\le n).
\end{gather*}
We have the following identity from the connection coefficients
of even/odd families.
\begin{align*}
 &\sum_{k=1}^n\sin(x_k+s)\cdot\sin(x_k+t)\cdot
     \prod_{\nu\in\{1,\ldots,n\}\setminus\{k\}}
     \frac{\sin(x_k+x_\nu+2u)}{\sin(x_k-x_\nu)}\\
 &=\begin{cases}
   \sin\Bigl(nu+\displaystyle\sum_{\nu=1}^nx_\nu\Bigr)\cdot
   \sin\Bigl(s+t+(n-2)u+\sum_{\nu=1}^nx_\nu\Bigr)
   \qquad\qquad\ \,\text{if \ }n=2m,\\[5pt]
   \sin\Bigl(s+(n-1)u+\displaystyle\sum_{\nu=1}^nx_\nu\Bigr)\cdot
   \sin\Bigl(t+(n-1)u+\sum_{\nu=1}^nx_\nu\Bigr)
   \qquad\text{if \ }n=2m+1.
   \end{cases}\notag
\end{align*}
The direct proof of these identities using residue calculus is given by 
\cite{Oc}.
It is interesting that similar identities of rational functions are given 
in \cite[Appendix]{Gl}
which studies the systems of Schlesinger canonical form corresponding to
Simpson's list (cf.~\S\ref{sec:rigidEx}).
\subsection{Rigid examples of order at most 4}\label{sec:4Ex}
\index{tuple of partitions!rigid!order$\ \le4$}%
\index{000Delta1@$[\Delta(\mathbf m)]$}%
\subsubsection{order $1$}
\underline{$1,1,1$}
\begin{equation*}
     u(x)=x^{\lambda_1}(1-x)^{\lambda_2}
    \qquad
    \begin{Bmatrix}-\lambda_1-\lambda_2&\lambda_1&\lambda_2\end{Bmatrix}
\end{equation*}
\subsubsection{order $2$}
\underline{$11,11,11$} : $H_2$ (Gauss)\qquad$[\Delta(\mathbf m)]=1^4$
\begin{equation*}
  u_{H_2} = \p^{-\mu_1}u(x)
    \qquad
    \begin{Bmatrix}
     -\mu_1+1& 0& 0\\
     -\lambda_1-\lambda_2-\mu_1&\lambda_1+\mu_1&\lambda_2+\mu_1
    \end{Bmatrix}
\end{equation*}
\subsubsection{order $3$} There are two types.

$\underline{111,21,111} : H_3\ ({}_3F_2)$ \qquad$[\Delta(\mathbf m)]=1^9$
\begin{align*}
&u_{H_3} = \p^{-\mu_2}x^{\lambda_3}u_{H_2}
\\
&  \begin{Bmatrix}
      1-\mu_2                &0 & [0]_{(2)}\\
      -\lambda_3-\mu_1-\mu_2+1& \lambda_3+\mu_2\\
      -\lambda_1-\lambda_2-\lambda_3-\mu_1-\mu_2
        &\lambda_1+\lambda_3+\mu_1+\mu_2&\lambda_2+\mu_1+\mu_2
    \end{Bmatrix}
\end{align*}

\underline{$21,21,21,21$} : $P_3$ (Jordan-Pochhammer)
\qquad$[\Delta(\mathbf m)]=1^4\cdot2$
\begin{align*}
  &u_{P_3} = \p^{-\mu}x^{\lambda_0}(1-x)^{\lambda_1}(c_2-x)^{\lambda_2}\\
  &
       \begin{Bmatrix}
         [1-\mu]_{(2)}& [0]_{(2)} & [0]_{(2)} & [0]_{(2)}\\
         -\lambda_0-\lambda_1-\lambda_2-\mu & 
         \lambda_0+\mu &\lambda_1+\mu &\lambda_2+\mu 
       \end{Bmatrix}
\end{align*}
\subsubsection{order 4} There are 6 types.

\index{tuple of partitions!rigid!211,211,211}
\underline{$211,211,211$}: $\alpha_2$
\qquad$[\Delta(\mathbf m)]=1^{10}\cdot2$
\begin{align*}
   &\p^{-\mu_2}x^{\lambda_3}(1-x)^{\lambda_4}u_{H_2}\\
   &\begin{Bmatrix}
     [-\mu_2+1]_{(2)} & [0]_{(2)}& [0]_{(2)}\\
     -\mu_1-\lambda_3-\lambda_4-\mu_2+1& \lambda_3+\mu_2& \lambda_4+\mu_2\\
     -\lambda_1-\lambda_2-\lambda_3-\lambda_4-\mu_1-\mu_2
       &\lambda_1+\lambda_3+\mu_1+\mu_2
         &\lambda_2+\lambda_4+\mu_1+\mu_2
    \end{Bmatrix}
\end{align*}

\underline{$1111,31,1111$} : $H_4\ ({}_4F_3)$ 
\qquad$[\Delta(\mathbf m)]=1^{16}$
{\small\begin{align*}
  &\p^{-\mu_3}x^{\lambda_4}u_{H_3}
  \\
  &\begin{Bmatrix}
      -\mu_3+1&0& [0]_{(3)}\\
      -\lambda_4-\mu_2-\mu_3+1                &\lambda_4 \\
      -\lambda_3-\lambda_4-\mu_1-\mu_2-\mu_3+1
        & \lambda_3+\lambda_4+\mu_2+\mu_3\\
      -\lambda_1-\cdots-\lambda_4-\mu_1-\mu_2-\mu_3
        &\lambda_1+\cdots+\lambda_4+\mu_1+\mu_2+\mu_3&\lambda_2+\mu_1+\mu_2+\mu_3
    \end{Bmatrix}
\end{align*}}

\underline{$211,22,1111$} : $EO_4$
\qquad$[\Delta(\mathbf m)]=1^{14}$
{\small\begin{align*}
  &\p^{-\mu_3}(1-x)^{-\lambda'}u_{H_3},\quad
  \lambda'=\lambda_2+\mu_1+\mu_2\\
  &
 \begin{Bmatrix}
      \lambda_2+\mu_1-\mu_2-\mu_3+1&[0]_{(2)}
        & [-\lambda_2-\mu_1-\mu_2+\mu_3]_{(2)}\\
      \lambda_2-\lambda_3-\mu_3+1& \lambda_3+\mu_2+\mu_3\\
      -\lambda_1-\lambda_3-\mu_3
        &\lambda_1+\lambda_3+\mu_1+\mu_2+\mu_3&[0]_{(2)}\\
      -\mu_3+1\\
    \end{Bmatrix}
\end{align*}}%
We have the integral representation of the local solution
corresponding to the exponent at 0:
\[
 \int_0^x\int_0^t\int_0^s (1-t)^{-\lambda_2-\mu_1-\mu_2}(x-t)^{\mu_3-1}
 s^{\lambda_3}(t-s)^{\mu_2-1}
 u^{\lambda_1}(1-u)^{\lambda_2}(s-u)^{\mu-1}du\,ds\,dt.
\]

\underline{$211,22,31,31$}: $I_4$
\qquad$[\Delta(\mathbf m)]=1^6\cdot 2^2${\small
\begin{align*}
  &\p^{-\mu_2}(c_2-x)^{\lambda_3}u_{H_2}\\
  &\begin{Bmatrix}
     [-\mu_2+1]_{(2)}&[0]_{(3)}&[0]_{(3)}&[0]_{(2)}\\
     -\lambda_3-\mu_1-\mu_2+1& & 
        &[\lambda_3+\mu_2]_{(2)}\\
     -\lambda_1-\lambda_2-\lambda_3-\mu_1-\mu_2
      &\lambda_1+\mu_1+\mu_2&\lambda_2+\mu_1+\mu_2
   \end{Bmatrix}
\end{align*}}

\underline{$31,31,31,31,31$}: $P_4$ 
\qquad$[\Delta(\mathbf m)]=1^5\cdot3$
\begin{align*}
 &u_{P_4}=\p^{-\mu}x^{\lambda_0}(1-x)^{\lambda_1}(c_2-x)^{\lambda_2}(c_3-x)^{\lambda_3}
 \\&
 \begin{Bmatrix}
   [-\mu+1]_{(3)} & [0]_{(3)} & [0]_{(3)} & [0]_{(3)} & [0]_{(3)}\\
   -\lambda_0-\lambda_2-\lambda_3-\mu
   &\lambda_0+\mu & \lambda_1+\mu & \lambda_2+\mu & \lambda_3+\mu
 \end{Bmatrix}
\end{align*}

\underline{$22,22,22,31$}: $P_{4,4}$
\qquad$[\Delta(\mathbf m)]=1^8\cdot2$
\begin{gather*}
    \p^{-\mu'}x^{-\lambda_0'}(1-x)^{-\lambda_1'}(c_2-x)^{-\lambda_2'}u_{P_3},
    \quad
    \lambda_j' = \lambda_j+\mu,\quad
    \mu' = \lambda_0+\lambda_1+\lambda_2+2\mu\\
 \begin{Bmatrix}
 [1-\mu']_{(3)}& [\lambda_1+\lambda_2+\mu]_{(2)} & [\lambda_0+\lambda_2+\mu]_{(2)} & [\lambda_0+\lambda_1+\mu]_{(2)}\\
 -\lambda_0-\lambda_1-\lambda_2& [0]_{(2)} &[0]_{(2)} &[0]_{(2)}
 \end{Bmatrix}
\end{gather*}
\subsubsection{Tuple of partitions $:\ 211,211,211$}\label{sec:eq211}
\index{tuple of partitions!rigid!211,211,211}
\qquad$[\Delta(\mathbf m)]=1^{10}\cdot 2$
\[211,211,211 = H_1\oplus H_3:6 = H_2\oplus H_2:4
    = 2H_1\oplus H_2:1\]
From the operations
\begin{align*}
 &\begin{Bmatrix}
   x=\infty&0&1\\
   1-\mu_1 & 0 & 0\\
   -\alpha_1-\beta_1-\mu_1&\underline{\alpha_1+\mu_1}&\beta_1+\mu_1
 \end{Bmatrix}\\
 &\xrightarrow{x^{\alpha_2}(1-x)^{\beta_2}}{}
 \begin{Bmatrix}
   x=\infty&0&1\\
   1-\alpha_2-\beta_2-\mu_1 & \alpha_2 & \beta_2\\
   -\alpha_1-\alpha_2-\beta_1-\beta_2-\mu_1&
      \underline{\alpha_1+\alpha_2+\mu_1}
         &\beta_1+\beta_2+\mu_1
 \end{Bmatrix}\\
  &\xrightarrow{\p^{-\mu_2}}{}
  \begin{Bmatrix}
   x=\infty&0&1\\
   [-\mu_2+1]_{(2)} & [0]_{(2)} & [0]_{(2)}\\
   1-\beta_2-\mu_1-\mu_2 &\alpha_2+\mu_2 &\beta_2+\mu_2\\
   -\alpha_1-\beta_1-\beta_2-\mu_1-\mu_2
   &\underline{\alpha_1+\mu_1+\mu_2}&\beta_1+\beta_2+\mu_1+\mu_2
 \end{Bmatrix}\allowdisplaybreaks\\
  &\longrightarrow
  \begin{Bmatrix}
   x=\infty&0&1\\
   [\lambda_{2,1}]_{(2)} & [\lambda_{0,1}]_{(2)} & [\lambda_{1,1}]_{(2)}\\
   \lambda_{2,2}        & \lambda_{0,2} &\lambda_{1,2}\\
   \lambda_{2,3}        &\lambda_{0,3}&\lambda_{1,3}
 \end{Bmatrix}\quad\text{with}\quad
 \sum_{j=0}^2
 (2\lambda_{j,1}+\lambda_{j,2}+\lambda_{j,3})=3,
\end{align*}
we have the integral representation of the solutions as in the case of 
other examples we have explained and so here we will not discuss them.
The universal operator of type $11,11,11$ is
\begin{align*}
 Q=x^2(1-x)^2\p^2 -(ax+b)x(1-x)\p + (cx^2+dx+e).
\end{align*}
Here we have
\begin{align*}
 b&=\lambda_{0,1}'+\lambda_{0,2}'-1,&e&=\lambda_{0,1}'\lambda_{0,2}',\\
 -a-b&=\lambda_{1,1}'+\lambda_{1,2}'-1,&c+d+e&=\lambda_{1,1}'\lambda_{1,2}',\\
 &&c&=\lambda_{2,1}'\lambda_{2,2}',\\
\lambda_{0,1}'&=\alpha_2,&\lambda_{0,2}'&=\alpha_1+\alpha_2+\mu_1,\\
 \lambda_{1,1}'&=\beta_2,&\lambda_{1,2}'&=\beta_1+\beta_2+\mu_2,\\
  \lambda_{2,1}'&=1-\beta_2-\mu_1-\mu_2,&
  \lambda_{2,2}'&=-\alpha_1-\beta_1-\beta_2-\mu_1-\mu_2
\end{align*}
corresponding to the above second Riemann scheme.
The operator corresponding to the tuple $211,211,211$ is
\begin{align*}
P&=\RAd(\p^{-\mu_2}) Q\allowdisplaybreaks\\
 &=\RAd(\p^{-\mu_2})\Bigl(
   (\vartheta-\lambda_{0,1}')(\vartheta-\lambda_{0,2}')\\
&\quad{}
 +x\bigl(-2\vartheta^2+(2\lambda_{0,1}'+2\lambda_{0,2}'+\lambda_{1,1}'
  +\lambda_{1,2}'-1)\vartheta+\lambda_{1,1}'\lambda_{1,2}'
  -\lambda_{0,1}'\lambda_{0,2}'-\lambda_{2,1}'\lambda_{2,2}'\bigr)
 \\
 &\quad{}
 + x^2(\vartheta+\lambda_{2,1}')(\vartheta+\lambda_{2,2}')\Bigr)
 \allowdisplaybreaks\\
 &=\p^2 (\vartheta-\lambda_{0,1}'-\mu_2)
  (\vartheta-\lambda_{0,2}'-\mu_2)\\
&\quad{}
 +\p(\vartheta-\mu_2+1)
  \bigl(-2(\vartheta-\mu_2)^2+(2\lambda_{0,1}'+2\lambda_{0,2}'+\lambda_{1,1}'
  +\lambda_{1,2}'-1)(\vartheta-\mu_2)\\
 &\quad{}+\lambda_{1,1}'\lambda_{1,2}'
  -\lambda_{0,1}'\lambda_{0,2}'-\lambda_{2,1}'\lambda_{2,2}'\bigr)\\
 &\quad{}
 + (\vartheta-\mu_2+1)(\vartheta-\mu_2+2)(\vartheta+\lambda_{2,1}'-\mu_2)(\vartheta+\lambda_{2,2}'-\mu_2).
\end{align*}

\noindent
The condition for the irreducibility:
\begin{equation*}
 \begin{cases}
 \lambda_{0,1}+\lambda_{1,1}+\lambda_{2,1}\notin\mathbb Z,\\
 \lambda_{0,\nu}+\lambda_{1,1}+\lambda_{2,1}\notin\mathbb Z,\ 
 \lambda_{0,1}+\lambda_{1,\nu}+\lambda_{2,1}\notin\mathbb Z,\ 
 \lambda_{0,1}+\lambda_{1,1}+\lambda_{2,\nu}\notin\mathbb Z\ \ (\nu=2,3),\\
 \lambda_{0,1}+\lambda_{0,2}+\lambda_{1,1}+\lambda_{1,\nu}+\lambda_{2,1}
 +\lambda_{2,\nu'}\notin\mathbb Z\ \ (\nu,\,\nu'\in\{2,3\}).
 \end{cases}
\end{equation*}
There exist three types of direct decompositions of the tuple
and there are 4 direct decompositions which give the 
connection coefficient 
$c(\lambda_{0,3}\!\rightsquigarrow\!\lambda_{1,3})$
by the formula \eqref{eq:connection} in Theorem~\ref{thm:c}:
\begin{align*}
 21\overline{1},21\underline{1},211
 &=00\overline{1},100,100\oplus210,11\underline{1},111\\
 &=11\overline{1},210,111\oplus100,00\underline{1},100\\
 &=10\overline{1},110,110\oplus110,10\underline{1},101\\
 &=10\overline{1},110,101\oplus110,10\underline{1},110
\end{align*}
Thus we have
\begin{align*}
 &c(\lambda_{0,3}\!\rightsquigarrow\!\lambda_{1,3})
 =\frac{\prod_{\nu=1}^2\Gamma(\lambda_{0,3}-\lambda_{0,\nu}+1)}
  {\Gamma(\lambda_{0,3}+\lambda_{1,1}+\lambda_{2,1})\cdot
   \Gamma(1-\lambda_{0,1}-\lambda_{1,3}-\lambda_{2,1})}\\
 &\phantom{c(\lambda_{0,3}\!\rightsquigarrow\!\lambda_{1,3})=}
  \quad\cdot
  \frac{\prod_{\nu=1}^2\Gamma(\lambda_{1,\nu}-\Gamma_{1,3})}
   {\prod_{\nu=2}^3
   \Gamma(\lambda_{0,1}+\lambda_{0,3}+\lambda_{1,1}+\lambda_{1,2}
  +\lambda_{2,1}+\lambda_{2,\nu}-1)}.
\end{align*}

\index{Wronskian}
We can also calculate generalized connection coefficient
defined in Definition~\ref{def:GC}:
\index{connection coefficient!generalized}
\begin{align*}
 &c([\lambda_{0,1}]_{(2)}\!\rightsquigarrow\![\lambda_{1,1}]_{(2)})
 =\frac{
    \prod_{\nu=2}^3\bigl(\Gamma(\lambda_{0,1}-\lambda_{0,\nu}+2)
    \cdot\Gamma(\lambda_{1,\nu}-\lambda_{1,1}-1)\bigr)
  }{
    \prod_{\nu=2}^3\bigl(
    \Gamma(\lambda_{0,1}+\lambda_{1,\nu}+\lambda_{2,1})
    \cdot\Gamma(1-\lambda_{0,\nu}-\lambda_{1,1}-\lambda_{2,1})
    \bigr)
  }.
\end{align*}
This can be proved by the procedure given in Remark~\ref{rem:Cproc} as in the
case of the formula \eqref{eq:e4dc2}.
Note that the gamma functions in the numerator of this formula 
correspond to Remark~\ref{rem:Cproc} (2) and those 
in the denominator correspond to the rigid decompositions
\[
 \begin{split}
  \underline{2}11,\overline{2}11,211
  &=\underline{1}00,\overline{0}10,100\oplus
     \underline{1}11,\overline{2}01,111
  =\underline{1}00,\overline{0}01,100\oplus
     \underline{1}11,\overline{2}10,111\\
  &=\underline{2}10,\overline{1}11,111\oplus
     \underline{0}01,\overline{1}00,100
  =\underline{2}01,\overline{1}11,111\oplus
     \underline{0}10,\overline{1}00,100.
 \end{split}
\]
The equation $Pu=0$ with the Riemann scheme 
$\begin{Bmatrix}
x=\infty & 0 & 1\\
[\lambda_{0,1}]_{(2)} & [0]_{(2)} & [0]_{(2)}\\
\lambda_{0,2} & \lambda_{1,2} & \lambda_{2,2}\\
\lambda_{0,3} & \lambda_{1,3} & \lambda_{2,3}
\end{Bmatrix}$ is isomorphic to the system
\begin{align*}
 \tilde u'&=\frac Ax\tilde u+\frac B{x-1}\tilde u,\quad
 \tilde u=\begin{pmatrix}u_1\\u_2\\u_3\\u_4\end{pmatrix},\ \ u_1=u,
\allowdisplaybreaks\\
 A&=\begin{pmatrix}
    0 & 0 & c_1 & 0\\
    0 & 0 & 0 & c_1\\
    0 & 0 & a_1 & b_1-b_2-c_2\\
    0 & 0 & 0 & a_2
   \end{pmatrix},
\allowdisplaybreaks\\ 
 B&=\begin{pmatrix}
    0 & 0 & 0 & 0\\
    0 & 0 & 0 & 0\\
    -a_1-b_2+c_1& -b_1+b_2+c_2& b_2 &0 \\
    -a_1+a_2+c_2& -a_2-b_1+c_1& a_1-a_2-c_2& b_1
   \end{pmatrix},
\allowdisplaybreaks\\
 &\begin{cases}
  a_1=\lambda_{1,2},\\
  a_2=\lambda_{1,3},\\
  b_1=\lambda_{2,2}-2,\\
  b_2=\lambda_{2,3}-1,\\
  c_1 = -\lambda_{0,1},\\
  c_2=\lambda_{0,1}+\lambda_{0,2}+\lambda_{1,2}+\lambda_{2,2}-1
 \end{cases}
\end{align*}
when $\lambda_{0,1}(\lambda_{0,1}+\lambda_{2,2})
(\lambda_{0,1}+\lambda_{0,2}+\lambda_{1,2}+\lambda_{2,3}-2)\ne0$.
Let $u(x)$ be a holomorphic solution of $Pu=0$ in a neighborhood of $x=0$.
By a direct calculation we have
\begin{align*}
 &u_1(0)=\frac{(a_1-1)(a_2-1)}{(b_1-c_1+1)(b_1-b_2-c_2)c_1}u'(0)+\\
&\ \ 
\frac
{(a_2+b_2+c_2-1)a_1-(c_1+c_2)a_2+(a_2-a_1+c_2)b_1-(c_2+1)b_2-c_2^2+c_1}{
(b_1-c_1+1)(b_1-b_2-c_2)}u(0).
\end{align*}
Since the shift described in Remark~\ref{rem:Cproc} (1) corresponds to the 
shift
\[
 (a_1,a_2,b_1,b_2,c_1,c_2)\mapsto
 (a_1-k,a_2-k,b_1+k,b_2+k,c_1,c_2),
\]
it follows from Theorem~\ref{thm:paralimits} that
\[
 \lim_{k\to\infty}
 c([\lambda_{0,1}]_{(2)}\!\rightsquigarrow\![\lambda_{1,1}]_{(2)})
 \Bigl|_{\substack{\substack\lambda_{0,2}\mapsto \lambda_{0,2}-k,\ 
 \substack\lambda_{0,3}\mapsto \lambda_{0,3}-k\\ 
 \lambda_{1,2}\mapsto \lambda_{1,2}+k,\ 
 \lambda_{1,3}\mapsto \lambda_{1,3}+k}}=1
\]
as in the proof of \eqref{eq:e4dc2} because 
$u_1(0)\sim \frac k{(b_1-b_2-c_2)c_1}u'(0)+Cu(0)$ 
with $C\in \mathbb C$  when $k\to\infty$.
Thus we can calculate this generalized connection coefficient by the procedure 
described in Remark~\ref{rem:Cproc}.

Using \eqref{eq:I0H}, 
we have the series expansion of the local solution at $x=0$ corresponding 
to the exponent $\alpha_1+\mu_1+\mu_2$ for the Riemann 
scheme parametrized by $\alpha_i$, $\beta_i$ and $\mu_i$ with $i=1,\,2$.
\begin{align*}
 &I_0^{\mu_2}x^{\alpha_2}(1-x)^{\beta_2}I_0^{\mu_1}x^{\alpha_1}(1-x)^{\beta_1}
 \allowdisplaybreaks\\&
 =I_0^{\mu_2} \frac{\Gamma(\alpha_1+1)}{\Gamma(\alpha_1+\mu+1)}
  \sum_{n=0}^\infty
  \frac{(\alpha_1+1)_n(-\beta_1)_n}{(\alpha_1+\mu+1)_nn!}
    x^{\alpha_2}(1-x)^{\beta_2}x^{\alpha_1+\mu+n}
\allowdisplaybreaks\\
&=\frac{\Gamma(\alpha_1+1)\Gamma(\alpha_1+\alpha_2+\mu_1+1)
  x^{\alpha_1+\alpha_2+\mu_1+\mu_2}}
  {\Gamma(\alpha_1+\mu_1+1) \Gamma(\alpha_1+\alpha_2+\mu_1+\mu_2+1)}\\
&\quad\cdot
 \sum_{m,n=0}^\infty \frac{(\alpha_1+1)_n(\alpha_1+\alpha_2+\mu_1+1)_{m+n}
   (-\beta_1)_n(-\beta_2)_m}
  {(\alpha_1+\mu_1+1)_{n} (\alpha_1+\alpha_2+\mu_1+\mu_2+1)_{m+n}n!m!}x^{m+n}
\allowdisplaybreaks\\
 &=\frac{\Gamma(\alpha_1+1)\Gamma(\alpha_1+\alpha_2+\mu_1+1)
  x^{\alpha_1+\alpha_2+\mu_1+\mu_2}(1-x)^{-\beta_2}}
  {\Gamma(\alpha_1+\mu_1+1) \Gamma(\alpha_1+\alpha_2+\mu_1+\mu_2+1)}\\
 &\quad \cdot\sum_{m,\,n=0}^\infty
  \frac{(\alpha_1+1)_n(\alpha_1+\alpha_2+\mu_1+1)_n(\mu_2)_m(-\beta_1)_n
  (-\beta_2)_m}
  {(\alpha_1+\mu_1+1)_n(\alpha_1+\alpha_2+\mu_1+\mu_2+1)_{m+n}m!n!}
  x^n\Bigl(\frac x{x-1}\Bigr)^m.
\end{align*}
Note that when $\beta_2=0$, the local solution is reduced to a local solution 
of the equation at $x=0$ satisfied by the hypergeometric series 
${}_3F_2(\alpha'_1,\alpha'_2,\alpha'_3;\beta'_1,\beta'_2;x)$ and when
$\alpha_2=0$, it is reduced to a local solution of the equation corresponding 
to the exponent at $x=1$ with free multiplicity.

Let $u_0(\alpha_1,\alpha_2,\beta_1,\beta_2,\mu_1,\mu_2;x)$ be the local solution
normalized by 
\[
 u_0(\alpha,\beta,\mu;x) - x^{\alpha_1+\alpha_2+\mu_1+\mu_2}
 \in x^{\alpha_1+\alpha_2+\mu_1+\mu_2+1}\mathcal O_0
\]
for generic $\alpha,\beta,\mu$. 
Then we have the recurrence relation
\[
\begin{split}
  u_0(\alpha,\beta_1-1,\beta_2,\mu;x) &= u_0(\alpha,\beta,\mu;x)
  +\frac{(\alpha_1+1)(\alpha_1+\alpha_2+\mu_1+1)}
   {(\alpha_1+\mu_1+1)(\alpha_1+\alpha_2+\mu_1+\mu_2+1)}\\
  &\qquad\cdot u_0(\alpha_1+1,\alpha_2,\beta_ 1-1,\beta_2,\mu;x).
\end{split}
\]

\subsubsection{Tuple of partitions $:\ 211,22,31,31$}
\index{tuple of partitions!rigid!211,22,31,31}
\qquad$[\Delta(\mathbf m)]=1^6\cdot 2$
\begin{align*}
  211,22,31,31
 &=H_1\oplus P_3:4=H_2\oplus H_2:2=2H_1\oplus H_2:2\\
 &=010,10,10,10\oplus 201,12,21,21
 =010,01,10,10\oplus 201,21,21,21\\
 &=001,10,10,10\oplus 210,12,21,21
 =001,01,10,10\oplus 210,21,21,21\\
 &=110,11,11,20\oplus 101,11,20,11
 =110,11,20,11\oplus 101,11,11,20\\
 &=200,20,20,20\oplus 011,02,11,11\\
 &\xrightarrow{\p_{max}}011,02,11,11
\end{align*}
\begin{align*}
&\begin{Bmatrix}
 x=0 & \frac1{c_1} & \frac1{c_2} & \infty\\
 [\lambda_{0,1}]_{(3)} & [\lambda_{1,1}]_{(3)} & [\lambda_{2,1}]_{(2)} & 
 [\lambda_{3,1}]_{(2)}\\
 \lambda_{0,2} & \lambda_{1,2} & \lambda_{2,2} &
 [\lambda_{3,2}]_{(2)}\\
 & & \lambda_{2,3}
\end{Bmatrix}
\allowdisplaybreaks\\
&\xrightarrow{x^{-\lambda_{0,1}}
 (1-c_1x)^{-\lambda_{1,1}}(1-c_2x)^{-\lambda_{2,1}}
 }{}\\
&
\begin{Bmatrix}
 x=0 & \frac1{c_1} & \frac1{c_2} & \infty\\
 [0]_{(3)} & [0]_{(3)} & [0]_{(2)} & 
 [\lambda_{3,1}+\lambda_{0,1}+\lambda_{1,1}+\lambda_{2,1}]_{(2)}\\
 \lambda_{0,2}-\lambda_{0,1} & \lambda_{1,2}-\lambda_{1,1} & 
 \lambda_{2,2}-\lambda_{2,1} &
 [\lambda_{3,2}+\lambda_{0,1}+\lambda_{1,1}+\lambda_{2,1}]_{(2)}\\
 & & \lambda_{2,3}-\lambda_{2,1}
\end{Bmatrix}\\
&\xrightarrow{\p^{-\lambda_1'}}{}
\\
&
\begin{Bmatrix}
 x=0 & \frac1{c_1} & \frac1{c_2} & \infty\\
  0 &  0 \\
 \lambda_{0,2}+\lambda_1'-\lambda_{0,1} 
 & \lambda_{1,2}+\lambda_1'-\lambda_{1,1}& 
 \lambda_{2,2}+\lambda_1'- \lambda_{2,1}&
 [\lambda_{3,2}-\lambda_{3,1}+1]_{(2)}\\
 & & \lambda_{2,3}+\lambda_1'-\lambda_{2,1}
\end{Bmatrix}
\end{align*}
The condition for the irreducibility:
\begin{equation*}
\begin{cases}
 \lambda_{0,1}+\lambda_{1,1}+\lambda_{2,\nu}+\lambda_{3,\nu'}\notin\mathbb Z
 \ \ (\nu\in\{1,2,3\},\ \nu'\in\{1,2\}),\\
 \lambda_{0,1}+\lambda_{0,2}+2\lambda_{1,1}+\lambda_{2,1}+\lambda_{2,\nu}
 +\lambda_{3,1}+\lambda_{3,2}
\notin\mathbb Z \ \ (\nu\in\{2,3\}),
\end{cases}
\end{equation*}
\begin{align*}
 c(\lambda_{0,2}\!\rightsquigarrow\!\lambda_{1,2})
 &=\frac{\Gamma(\lambda_{0,2}-\lambda_{0,1}+1)
        \Gamma(\lambda_{1,2}-\lambda_{1,1})
 (1-\frac{c_2}{c_1})^{\lambda_{2,1}}}
  {\prod_{\nu=2}^3
    \Gamma(\lambda_{0,1}+\lambda_{0,2}+2\lambda_{1,1}+\lambda_{2,1}
    +\lambda_{2,\nu}+\lambda_{3,1}+\lambda_{3,2}-1)},
\allowdisplaybreaks\\
 c(\lambda_{0,2}\!\rightsquigarrow\!\lambda_{2,3})
 &=\prod_{\nu=1}^2
   \frac{\Gamma(\lambda_{2,3}-\lambda_{2,\nu})}
   {\Gamma(1-\lambda_{0,1}-\lambda_{1,1}-\lambda_{2,3}-\lambda_{3,\nu})}\\
 &\quad\cdot
\frac{\Gamma(\lambda_{0,2}-\lambda_{0,1}+1)
 (1-\frac{c_1}{c_2})^{\lambda_{1,1}}}
  {
    \Gamma(\lambda_{0,1}+\lambda_{0,2}+2\lambda_{1,1}+\lambda_{2,1}
    +\lambda_{2,2}+\lambda_{3,1}+\lambda_{3,2}-1)}.
\end{align*}
\subsubsection{Tuple of partitions $:\ 22,22,22,31$}
\index{tuple of partitions!rigid!22,22,22,31}
\qquad$[\Delta(\mathbf m)]=1^8\cdot2$
\begin{align*}
 22,22,22,31
 &=H_1\oplus P_3:8=2(11,11,11,20)\oplus00,00,00,(-1)1\\
 &=10,10,10,10\oplus12,12,12,21
 =10,10,01,10\oplus12,12,21,21\\
 &=10,01,10,10\oplus12,21,12,21
 =10,01,01,10\oplus12,21,21,21\\
 &=01,10,10,10\oplus21,12,12,21
 =01,10,01,10\oplus21,12,21,21\\
 &=01,01,10,10\oplus21,21,12,21
 =01,01,01,10\oplus21,21,21,21\\
 &\overset2\to12,12,12,21
\end{align*}
The condition for the irreducibility:
\begin{equation*}
 \begin{cases}
 \lambda_{0,i}+\lambda_{1,j}+\lambda_{2,k}+\lambda_{3,1}\notin\mathbb Z
 \quad(i,\,j,\,k\in\{1,2\}),\\
 \lambda_{0,1}+\lambda_{0,2}+\lambda_{1,1}+\lambda_{1,2}+
 \lambda_{2,1}+\lambda_{2,2}+\lambda_{3,1}+\lambda_{3,2}\notin\mathbb Z.
 \end{cases}
\end{equation*}
\subsection{Other rigid examples with a small order}\label{sec:ord6Ex}
First we give an example which is not of Okubo type.

\subsubsection{$221,221,221$}
\index{tuple of partitions!rigid!221,221,221}
The Riemann Scheme and the direct decompositions are
\begin{gather*}
 \begin{Bmatrix}
  x=0 &  1 & \infty\\
 [\lambda_{0,1}]_{(2)} & [\lambda_{1,1}]_{(2)} & [\lambda_{2,1}]_{(2)}\\
 [\lambda_{0,2}]_{(2)} & [\lambda_{1,2}]_{(2)} & [\lambda_{2,2}]_{(2)}\\
 \lambda_{0,3} & \lambda_{1,3} & \lambda_{2,3} 
 \end{Bmatrix},\qquad
  \sum_{j=0}^2(2\lambda_{j,1}+2\lambda_{j,2}+\lambda_{j,3})=4,
 \allowdisplaybreaks\\
\begin{aligned}
 \phantom{a}
 [\Delta(\mathbf m)]&=1^{14}\cdot2\\
 22\overline1,22\underline1,221
            &=H_1\oplus 211,211,211:8\qquad\ 6=|2,2,2|\\
            &=H_2\oplus H_3:6\qquad\qquad\qquad 11=|21,22,22|\\
            &=2H_2\oplus H_1:1\\
            &=10\overline1,110,110\oplus120,11\underline1,111
             =01\overline1,110,110\oplus210,11\underline1,111\\
            &=11\overline1,120,111\oplus110,10\underline1,110
             =11\overline1,210,111\oplus110,01\underline1,110\\
            &\to121,121,121
\end{aligned}
\end{gather*}
and a connection coefficient is give by
\begin{align*}
 c(\lambda_{0,3}\!\rightsquigarrow\!\lambda_{1,3})
 &=\prod_{\nu=1}^{2}\biggl(
  \frac{\Gamma(\lambda_{0,3}-\lambda_{0,\nu}+1)}
       {\Gamma(\lambda_{0,\nu}+\lambda_{0,3}+\lambda_{1,1}+\lambda_{1,2}+\lambda_{2,1}+\lambda_{2,2}-1) }\\
 &\qquad\cdot
  \frac{\Gamma(\lambda_{1,\nu}-\lambda_{1,3})}
  {\Gamma(2-\lambda_{0,1}-\lambda_{0,2}-\lambda_{1,\nu}
   -\lambda_{1,3}-\lambda_{2,1}-\lambda_{2,2})}\biggr).
\end{align*}

Using this example we explain an idea to get all 
the rigid decompositions $\mathbf m=\mathbf m'\oplus\mathbf m''$.
Here we note that $\idx(\mathbf m,\mathbf m')=1$.
Put $\mathbf m=221,221,221$.  We may assume 
$\ord\mathbf m'\le\ord\mathbf m''$.

Suppose $\ord\mathbf m'=1$.
Then $\mathbf m'$ is isomorphic to $1,1,1$ and
there exists tuples of indices $(\ell_0,\ell_1,\ell_2)$ such that 
$m'_{j,\nu}=\delta_{j,\ell_j}$.
Then $\idx(\mathbf m,\mathbf m')=m_{0,\ell_0}+m_{1,\ell_1}+m_{1,\ell_2}
 - (3-2)\ord\mathbf m\cdot\ord\mathbf m'$ and we have
$m_{0,\ell_0}+m_{1,\ell_1}+m_{1,\ell_2}=6$.
Hence $(m_{0,\ell_0},m_{1,\ell_1},m_{1,\ell_2})=(2,2,2)$, which
is expressed by $6=|2,2,2|$ in the above.  Since $\ell_j=1$ or $2$ for 
$0\le j\le 2$, it is clear that there exist 8 rigid decompositions with 
$\ord\mathbf m'=1$.

Suppose $\ord\mathbf m'=2$.
Then $\mathbf m'$ is isomorphic to $11,11,11$
and there exists tuples of indices
$(\ell_{0,1},\ell_{0,2},\ell_{1,1},\ell_{1,2},\ell_{2,1},\ell_{2,2})$
which satisfies $\sum_{j=0}^2\sum_{\nu=1}^2 m_{j,\ell_\nu}=
(3-2)\ord\mathbf m\cdot\ord\mathbf m'+1=11$.
Hence we may assume 
$(\ell_{0,1},\ell_{0,2},\ell_{1,1},\ell_{1,2},\ell_{2,1},\ell_{2,2})
=(2,1,2,2,2,2)$ modulo obvious symmetries, which is expressed by
$11=|21,22,22|$.
There exist 6 rigid decompositions with $\ord\mathbf m'=2$.

In general, this method to get all the rigid decompositions of $\mathbf m$
is useful when $\ord\mathbf m$ is not big.
For example if $\ord\mathbf m\le 7$, $\mathbf m'$ is isomorphic
to $1,1,1$ or $11,11,11$ or $21,111,111$.

The condition for the irreducibility is given by Theorem~\ref{thm:irred}
and it is
\[
  \begin{cases}
   \lambda_{0,i}+\lambda_{1,j}+\lambda_{2,k}\notin\mathbb Z&
   (i,\,j,\,k\in\{1,2\}),\\
   \sum_{j=0}^2\sum_{\nu=1}^2\lambda_{j,\nu}
   +(\lambda_{i,3}-\lambda_{i,k})\notin\mathbb Z&
   (i\in\{0,1,2\},\ k\in\{1,2\}).
  \end{cases}
\]
\subsubsection{Other examples}
Theorem~\ref{thm:c} shows that 
the connection coefficients between local solutions of rigid
differential equations which correspond to the 
eigenvalues of local monodromies with free multiplicities 
are given by direct decompositions of the tuples of partitions
$\mathbf m$ describing their spectral types.

We list the rigid decompositions 
$\mathbf m=\mathbf m'\oplus \mathbf m''$
of rigid indivisible $\mathbf m$ in 
$\mathcal P^{(5)}\cup\mathcal P^{(6)}_3$ satisfying
$m_{0,n_0}=m_{1,n_1}=m'_{0,n_0}=m''_{1,n_1}=1$.
The positions of $m_{0,n_0}$ and $m_{1,n_1}$ in $\mathbf m$ to which 
Theorem~\ref{thm:c} applies are indicated by an overline and 
an underline, respectively.
The number of decompositions in each case equals $n_0+n_1-2$ and therefore
the validity of the following list is easily verified.

We show the tuple $\p_{max}\mathbf m$ after $\to$.
The type $[\Delta(\mathbf m)]$ of $\Delta(\mathbf m)$ is calculated by
\eqref{eq:DmInd}, which is also indicated in the following
with this calculation.
For example, when $\mathbf m=311,221,2111$, we have 
$d(\mathbf m)=2$, $\mathbf m'=\p\mathbf m=111,021,0111$,
$[\Delta(s(111,021,0111))]=1^9$, 
$\{m'_{j,\nu}-m'_{j,1}\in\mathbb Z_{>0}\}\cup\{2\}=\{1,1,1,1,2,2\}$
and hence $[\Delta(\mathbf m)]=1^9\times 1^4\cdot 2^2=1^{13}\cdot2^2$,
which is a partition of $h(\mathbf m)-1=17$.
Here we note that $h(\mathbf m)$ is the sum of the numbers
attached the Dynkin diagram \ 
\begin{xy}
 \ar@{-}  *+!D{1} *{\circ} ;
  (5,0)   *+!D{2} *{\circ}="A"
 \ar@{-} "A";(10,0) *+!LD{5} *{\circ}="B"
 \ar@{-} "B";(15,0) *+!D{3} *{\circ}="C"
 \ar@{-} "C";(20,0) *+!D{2} *{\circ}="C"
 \ar@{-} "C";(25,0) *+!D{1} *{\circ}="C"
 \ar@{-} "B";(10,5) *+!LD{3} *{\circ}="D"
 \ar@{-} "D";(10,10) *+!LD{1} *{\circ}
\end{xy} \ corresponding to $\alpha_{\mathbf m}\in\Delta_+$.

All the decompositions of the tuple $\mathbf m$ corresponding to the elements 
in $\Delta(\mathbf m)$ are given, by which we easily get 
the necessary and sufficient condition for the irreducibility
(cf.~Theorem~\ref{thm:irrKac} and \S\ref{sec:RobEx}).%
\index{000Delta1@$[\Delta(\mathbf m)]$}

{\small\begin{align*}
&\text{\large\qquad$\ord\mathbf m=5$}\\
 311,221,2111&=100,010,0001\oplus211,211,2110\qquad\qquad\quad\ \, 6=|3,2,1|\\
             &=100,001,1000\oplus211,220,1111\qquad\qquad\quad\ \, 6=|3,1,2|\\
             &=101,110,1001\oplus210,111,1110\qquad\qquad 11=|31,22,21|\\
             &=2(100,100,1000)\oplus111,021,0111\\
             &\overset{2}\to111,021,0111
\allowdisplaybreaks\\
[\Delta(\mathbf m)]&=1^9\times1^4\cdot2^2=1^{13}\cdot2^2\\
\mathbf m&=H_1\oplus211,211,211:6=H_1\oplus EO_4:1=H_2\oplus H_3:6
 =2H_1\oplus H_3:2
\allowdisplaybreaks\\
 311,22\overline{1},211\underline{1}
  &=211,211,2110\oplus100,010,0001
   =211,121,2110\oplus100,100,0001\\
  &=100,001,1000\oplus211,220,1111\\
  &=210,111,1110\oplus101,110,1001
   =201,111,1110\oplus110,110,1001
\allowdisplaybreaks\\
 31\overline{1},221,211\underline{1}
  &=211,211,2110\oplus100,010,0001
   =211,121,2110\oplus100,100,0001\\
  &=201,111,1110\oplus110,110,1001\\
  &=101,110,1010\oplus210,111,1101
   =101,110,1100\oplus210,111,1011
\allowdisplaybreaks\\[5pt]
 32,211\overline{1},211\underline{1}
  &=22,1111,2110\oplus10,1000,0001
   =10,0001,1000\oplus22,2110,1111\\
  &=11,1001,1010\oplus21,1110,1101
   =11,1001,1100\oplus21,1110,1011\\
  &=21,1101,1110\oplus11,1010,1001
   =21,1011,1110\oplus11,1100,1001\\
  &\overset2\to12,0111,0111\\
[\Delta(\mathbf m)]&=1^9\times 1^7\cdot 2=1^{16}\cdot2\\
\mathbf m&=H_1\oplus H_4:1=H_1\oplus EO_4:6=H_2\oplus H_3:9
  =2H_1\oplus H_3:1
\allowdisplaybreaks\\[5pt]
22\overline{1},22\underline{1},41,41
  &=001,100,10,10\oplus220,121,31,31=001,010,10,10\oplus220,211,31,31\\
  &=211,220,31,31\oplus010,001,10,10=121,220,31,31\oplus100,001,10,10\\
  &\overset2\to021,021,21,21\\
[\Delta(\mathbf m)]&=1^4\cdot2\times1^4\cdot2^3=1^{6}\cdot2^4\\
\mathbf m&=H_1\oplus22,211,31,31:4=H_2\oplus H_3:2=2H_1\oplus P_3:4
\allowdisplaybreaks\\
22\overline{1},221,4\underline{1},41
  &=001,100,10,10\oplus220,121,31,31=001,010,10,10\oplus220,211,31,31\\
  &=111,111,30,21\oplus110,110,11,20
\allowdisplaybreaks\\[5pt]
22\overline{1},32,32,4\underline{1}
  &=101,11,11,20\oplus120,21,21,21=011,11,11,20\oplus210,21,21,21\\
  &=001,10,10,10\oplus 220,22,22,31\\
  &\overset2\to021,12,12,21
\allowdisplaybreaks\\
[\Delta(\mathbf m)]&=1^4\cdot2\times1^3\cdot2^2=1^{7}\cdot2^3
\allowdisplaybreaks\\
\mathbf m&=H_1\oplus22,22,22,31:1=H_1\oplus211,22,31,31:4=
 H_2\oplus P_3:2\\
 &=2H_1\oplus P_3:2
\allowdisplaybreaks\\[5pt]
31\overline{1},31\underline{1},32,41
  &=001,100,10,10\oplus310,211,22,31
   =211,301,22,31\oplus100,001,10,10\\
  &=101,110,11,20\oplus210,201,21,21
   =201,210,21,21\oplus110,101,11,20\\
  &\overset3\to011,011,02,11\\
[\Delta(\mathbf m)]&=1^4\times1^4\cdot2\cdot3=1^8\cdot2\cdot3\\
\mathbf m&=H_1\oplus211,31,22,31:4=H_2\oplus P_3:4\\
  &=2H_1\oplus H_3:1=3H_1\oplus H_2:1
\allowdisplaybreaks\\ 
31\overline{1},311,32,4\underline{1}
  &=001,100,10,10\oplus301,211,22,31\\
  &=101,110,11,20\oplus210,201,21,21
   =101,101,11,20\oplus210,210,21,21
\allowdisplaybreaks\\[5pt]
32,32,4\overline{1},4\underline{1},41
  &=11,11,11,20,20\oplus21,21,30,21,21\\
  &=21,21,21,30,21\oplus11,11,20,11,20\\
  &\overset3\to02,02,11,11,11\\
[\Delta(\mathbf m)]&=1^4\times2^2\cdot 3=1^4\cdot2^2\cdot3\\
\mathbf m&=H_1\oplus P_4:1=H_2\oplus P_3:3=2H_1\oplus P_3:2=3H_1\oplus H_2:1
\allowdisplaybreaks\\[10pt]
&\text{\qquad\large$\ord\mathbf m=6$ and $\mathbf m\in\mathcal P_3$}\\
321,3111,222&=311,2111,221\oplus010,1000,001\qquad\qquad\quad\ \,7=|2,3,2|\\
 &=211,2110,211\oplus110,1001,011\qquad\qquad13=|32,31,22|\\
 &=210,1110,111\oplus111,2001,111\\
 &\overset2\to121,1111,022\to111,0111,012\\
[\Delta(\mathbf m)]&=1^{14}\times1\cdot2^3=1^{15}\cdot2^3\\
\mathbf m&=H_1\oplus311,2111,221:3=
 H_2\oplus211,211,211:6=H_3\oplus H_3:6\\
 &= 2H_1\oplus EO_4:3
\allowdisplaybreaks\\
32\overline{1},311\underline{1},222
            &=211,2110,211\oplus110,1001,011
             =211,2110,121\oplus110,1001,101\\
            &=211,2110,112\oplus110,1001,110\\
            &=111,2100,111\oplus210,1011,111
             =111,2010,111\oplus210,1101,111
\allowdisplaybreaks\\
321,311\overline{1},311\underline{1}
            &=221,2111,3110\oplus100,1000,0001
             =100,0001,1000\oplus221,3110,2111\\
            &=211,2101,2110\oplus110,1010,1001
             =211,2011,2110\oplus110,1101,1001\\
            &=110,1001,1100\oplus211,2110,2011
             =110,1001,1010\oplus211,2110,2101\\
            &\overset3\to021,0111,0111\\
[\Delta(\mathbf m)]&=1^9\times1^7\cdot 2\cdot3=1^{16}\cdot2\cdot3\\
\mathbf m&=H_1\oplus221,2111,311:6=H_1\oplus32,2111,2111:1\\
&=H_2\oplus211,211,211:9=2H_1\oplus H_4:1= 3H_1\oplus H_3:1\\
\allowdisplaybreaks\\
32\overline{1},311\underline{1},3111
            &=221,3110,2111\oplus100,0001,1000
             =001,1000,1000\oplus320,2111,2111\\
            &=211,2110,2110\oplus110,1001,1001
             =211,2110,2011\oplus110,1001,1100\\
            &=211,2110,2011\oplus110,1001,1100
\allowdisplaybreaks\\[5pt]
32\overline{1},32\underline{1},2211
            &=211,220,1111\oplus110,101,1100
             =101,110,1100\oplus220,211,1111\\
            &=111,210,1110\oplus210,111,1101
             =111,210,1101\oplus210,111,1110\\
            &\overset2\to121,121,0211\to101,101,0011\\
[\Delta(\mathbf m)]&=1^{10}\cdot2\times1^4\cdot2^2=1^{14}\cdot2^3\\
\mathbf m&=H_1\oplus311,221,2111:4=H_1\oplus 221,221,221:2\\
 &=H_2\oplus EO_4:2= H_2\oplus211,211,211:4=H_3\oplus H_3:2\\
 &=2H_1\oplus211,211,211:2=2(110,110,1100)\oplus101,101,0011:1
\allowdisplaybreaks\\
321,32\overline{1},221\underline{1}
            &=221,221,2210\oplus100,100,0001
             =110,101,1100\oplus211,220,1111\\
            &=211,211,2110\oplus110,110,0101
             =211,211,1210\oplus110,110,1001\\
            &=210,111,1110\oplus111,210,1101
\allowdisplaybreaks\\[5pt]
41\overline{1},221\underline{1},2211
             &=311,2210,2111\oplus100,0001,0100
              =311,2210,1211\oplus100,0001,1000\\
             &=101,1100,1100\oplus310,1111,1111
              =201,1110,1110\oplus210,1101,1101\\
             &=201,1110,1101\oplus210,1101,1110\\
             &\overset2\to211,0211,0211\to011,001,0011\\
[\Delta(\mathbf m)]&=1^{10}\cdot2\times1^4\cdot 2^3=1^{14}\cdot2^4\\
\mathbf m&=H_1\oplus311,221,2211:8=H_2\oplus H_4:2=H_3\oplus H_3:4\\
&=2H_1\oplus211,211,211:4
\allowdisplaybreaks\\
411,221\overline{1},221\underline{1}
             &=311,2111,2210\oplus100,0100,0001
              =311,1211,2210\oplus100,1000,0001\\
             &=100,0001,0100\oplus311,2210,2111
              =100,0001,1000\oplus311,2210,1211\\
             &=201,1101,1110\oplus210,1110,1101
              =210,1101,1110\oplus201,1110,1101
\allowdisplaybreaks\\[5pt]
41\overline{1},222,2111\underline{1}
             &=311,221,21110\oplus100,001,00001
              =311,212,21110\oplus100,010,00001\\
             &=311,122,21110\oplus100,100,00001
              =201,111,11100\oplus210,111,10011\\
             &=201,111,11010\oplus210,111,10101
              =201,111,10110\oplus210,111,11001\\
             &\overset2\to211,022,01111\to111,012,00111\\
[\Delta(\mathbf m)]&=1^{14}\times1^4\cdot 2^3=1^{18}\cdot2^3\\
\mathbf m&=H_1\oplus311,221,2111:12=H_3\oplus H_3:6=2H_1\oplus EO_4:3
\allowdisplaybreaks\\[5pt]
42,221\overline{1},2111\underline{1}
             &=32,2111,21110\oplus10,0100,00001
              =32,1211,21110\oplus10,1000,00001\\
             &=10,0001,10000\oplus32,2210,11111
              =31,1111,11110\oplus11,1100,10001\\
             &=21,1101,11100\oplus21,1110,10011
              =21,1101,11010\oplus21,1110,10101\\
             &=21,1101,10110\oplus21,1110,11001\\
             &\overset2\to22,0211,01111\to12,0111,00111\\
[\Delta(\mathbf m)]&=1^{14}\times1^6\cdot 2^2=1^{20}\cdot2^2\\
\mathbf m&=H_1\oplus32,2111,2111:8=H_1\oplus EO_4:2
 =H_2\oplus H_4:4\\
&=H_3\oplus H_3:6=2H_1\oplus EO_4:2
\allowdisplaybreaks\\[5pt]
33,311\overline{1},2111\underline{1}
             &=32,2111,21110\oplus01,1000,00001
              =23,2111,21110\oplus10,1000,00001\\
             &=22,2101,11110\oplus11,1010,10001
              =22,2011,11110\oplus11,1100,10001\\
             &=11,1001,11000\oplus22,2110,10111
              =11,1001,10100\oplus22,2110,11011
\allowdisplaybreaks\\
             &=11,1001,10010\oplus22,2110,11101\\
             &\overset2\to13,1111,01111
\allowdisplaybreaks\\
[\Delta(\mathbf m)]&=1^{16}\times1^4\cdot2^2=1^{20}\cdot2^2\\
\mathbf m&=H_1\oplus32,2111,2111:8=H_2\oplus EO_4:12
=2H_1\oplus H_4:2
\allowdisplaybreaks\\[5pt]
32\overline{1},311\underline{1},3111
              &=221,3110,2111\oplus100,0001,1000
               =001,1000,1000\oplus320,2111,2111\\
              &=211,2110,2110\oplus110,1001,1001
               =211,2110,2101\oplus110,1001,1010\\
              &=211,2110,2011\oplus110,1001,1100\\
              &\overset3\to021,0111,0111\\
[\Delta(\mathbf m)]&=1^{9}\times1^7\cdot2\cdot3=1^{16}\cdot2\cdot3\\
\mathbf m&=H_1\oplus221,2111,311:6=H_1\oplus32,2111,2111:1\\
&=H_2\oplus 211,211,211:9=2H_1\oplus H_4:1=3H_1\oplus H_3:1
\allowdisplaybreaks\\
321,311\overline{1},311\underline{1}
              &=100,0001,1000\oplus221,3110,2111
               =221,2111,3110\oplus100,1000,0001\\
              &=211,2101,2110\oplus110,1010,1001
               =211,2011,2110\oplus110,1100,1001\\
              &=110,1001,1100\oplus211,2110,2011
               =110,1001,1010\oplus211,2110,2101
\allowdisplaybreaks\\[5pt]
33,221\overline{1},221\underline{1}
            &=22,1111,2110\oplus11,1100,1001
             =22,1111,1210\oplus11,1100,0101\\
            &=21,1101,1110\oplus12,1110,1011
             =12,1101,1110\oplus21,1110,1011\\
            &=11,1001,1100\oplus22,1210,1111
             =11,0101,1100\oplus22,2110,1111\\
            &\overset1\to23,1211,1211\to21,1011,1011\\
[\Delta(\mathbf m)]&=1^{16}\cdot2\times1^4=1^{20}\cdot2\\
\mathbf m&=H_1\oplus32,2111,2111:8=H_2\oplus EO_4:8=
H_3\oplus H_3:4\\
&=2(11,1100,1100)\oplus11,0011,0011:1
\end{align*}}

We show all the rigid decompositions of the following 
simply reducible partitions of order 6, which also correspond to the 
reducibility of the universal models.\index{tuple of partitions!simply reducible}
{\begin{align*}
 42,222,111111
            &=32,122,011111\oplus 10,100,100000\\
            &=21,111,111000\oplus 21,111,000111\\
            &\overset1\to32,122,011111\to22,112,001111\to12,111,000111\\
[\Delta(\mathbf m)]&=1^{28}\\
\mathbf m&=H_1\oplus EO_5:18=H_3\oplus H_3:10
\allowdisplaybreaks\\[5pt]
  33,222,21111
  &=23,122,11111\oplus10,100,10000\\  
  &=22,112,10111\oplus11,110,11000\\
  &=21,111,11100\oplus12,111,10011\\
  &\overset1\to23,122,11111\to22,112,01111\to12,111,00111\\
[\Delta(\mathbf m)]&=1^{24}\\
\mathbf m&=H_1\oplus EO_5:6=H_2\oplus EO_4:12=H_3\oplus H_3:6
\end{align*}}
\subsection{Submaximal series and minimal series }\label{sec:Rob}
\index{tuple of partitions!rigid!(sub)maximal series}
The rigid tuples $\mathbf m=\{m_{j,\nu}\}$ satisfying
\begin{equation}
  \#\{m_{j,\nu}\,;\,0<m_{j,\nu}<\ord\mathbf m\}\ge\ord\mathbf m +5
\end{equation}
are classified by Roberts \cite{Ro}.
They are the tuples of type $H_n$ and $P_n$ which satisfy
\begin{equation}
  \#\{m_{j,\nu}\,;\,0<m_{j,\nu}<\ord\mathbf m\}=2\ord\mathbf m+2
\end{equation}
and those of 13 series $A_n=EO_n$, 
$B_n$, $C_n$, $D_n$, $E_n$, $F_n$, $G_{2m}$, $I_n$, $J_n$, $K_n$, $L_{2m+1}$, 
$M_n$, $N_n$ called \textsl{submaximal series}
\index{tuple of partitions!rigid!(sub)maximal series} which satisfy
\begin{equation}
 \#\{m_{j,\nu}\,;\,0<m_{j,\nu}<\ord \mathbf m\}=\ord\mathbf m+5.
\end{equation}
The series $H_n$ and $P_n$ are called \textsl{maximal series}.

We examine these rigid series
and give enough information to analyze the series,
which will be sufficient to construct differential equations including
their confluences,
integral representation and series expansion of solutions and
get connection coefficients and the condition of their reducibility.

In fact from the following list we easily get 
all the direct decompositions and 
Katz's operations decreasing the order.
The number over an arrow indicates the difference of the orders.
We also indicate Yokoyama's reduction for systems of Okubo normal form
using extension and restriction,
which are denoted $E_i$ and $R_i$ $(i=0,1,2)$, respectively (cf.~\cite{Yo2}).
Note that the inverse operations of $E_i$ are $R_i$, respectively.
In the following we put
\begin{equation}
\begin{split}
 u_{P_m}&=\p^{-\mu}x^{\lambda_0}(1-x)^{\lambda_1}(c_2-x)^{\lambda_2}
          \cdots(c_{m-1}-x)^{\lambda_{m-1}},\\
 u_{H_2}&=u_{P_2},\\
 u_{H_{m+1}}&=\p^{-\mu^{(m)}}x^{\lambda^{(m)}_0}u_{H_m}.
\end{split}
\end{equation}
We give all the decompositions
\begin{equation}\label{eq:allDec}
 \mathbf m=\bigl(\idx(\mathbf m',\mathbf m)\cdot\mathbf m'\bigr)\oplus\mathbf m''
\end{equation}
for $\alpha_{\mathbf m'}\in\Delta(\mathbf m)$.
Here we will not distinguish between $\mathbf m'\oplus \mathbf m''$
and $\mathbf m''\oplus \mathbf m'$ when $\idx(\mathbf m',\mathbf m)=1$.
Moreover note that the inequality assumed for the formula $[\Delta(\mathbf m)]$ 
below assures that the given tuple of partition is monotone.
\subsubsection{$B_n$}\label{sec:Bn}
(${B_{2m+1}}=\text{III}_m$, $B_{2m}=\text{II}_m$, $B_3=H_3$, $B_2=H_2$)
\index{000Delta1@$[\Delta(\mathbf m)]$}%
\begin{align*}
 u_{B_{2m+1}}&=\p^{-\mu'}(1-x)^{\lambda'}u_{H_{m+1}}\\
 m^21,m+11^m,m1^{m+1} &=10,10,01\oplus mm-11,m1^m,m1^m\\
 &=01,10,10\oplus m^2,m1^m,m-11^{m+1}\\
 &=1^20,11,11\oplus (m-1)^21,m1^{m-1},m-11^m
\allowdisplaybreaks\\
 [\Delta(B_{2m+1})]&=1^{(m+1)^2}\times1^{m+2}\cdot m^2
 =1^{m^2+3m+3}\cdot m^2\\
&\hspace{-1.08cm}
\begin{aligned}
 B_{2m+1}&=H_1\oplus B_{2m}&&:2(m+1) \\
  &=H_1\oplus C_{2m}&&:1\\
  &=H_2\oplus B_{2m-1}&&:m(m+1)\\
  &=mH_1\oplus H_{m+1}&&:2
\end{aligned}
\allowdisplaybreaks\\[5pt]
 u_{B_{2m}}&=\p^{-\mu'}x^{\lambda'}(1-x)^{\lambda''}u_{H_m}\\
 mm-11,m1^m,m1^m&=100,01,10\oplus (m-1)^21,m1^{m-1},m-11^m\\
 &=001,10,10\oplus mm-10,m-11^m,m-11^m\\
 &=110,11,11\oplus m-1m-21,m-11^{m-1},m-11^{m-1}
\allowdisplaybreaks\\
[\Delta(B_{2m})]&=1^{m^2}\times1^{2m+1}\cdot(m-1)=1^{(m+1)^2}\cdot(m-1)\cdot m
\\
&\hspace{-.72cm}
\begin{aligned} B_{2m}&=H_1\oplus B_{2m-1}&&:2m \\
 &= H_1\oplus C_{2m-2}&&:1 \\
 &= H_2\oplus B_{2m-2}&&:m^2\\
 &= (m-1)H_1\oplus H_{m+1}&&:1\\
 &= mH_1\oplus H_m&&:1
\end{aligned}
\allowdisplaybreaks\\
 B_{2m+1} &\overset{m}{\underset{R2E0}\longrightarrow} H_{m+1},\quad
 B_n\overset{1}\longrightarrow B_{n-1},\quad B_n\overset{1}\longrightarrow C_{n-1}
\allowdisplaybreaks\\
 B_{2m} &\overset{m}{\underset{R1E0}\longrightarrow} H_{m},\quad
 B_{2m} \overset{m-1}\longrightarrow H_{m+1}
\end{align*}

\subsubsection{An example}\label{sec:RobEx}
Using the example of type $B_{2m+1}$, 
we explain how we get explicit results from the data written in \S\ref{sec:Bn}.

The Riemann scheme of type $B_{2m+1}$ is
\begin{gather*}
 \begin{Bmatrix}
 \infty & 0 & 1\\
 [\lambda_{0,1}]_{(m)} & [\lambda_{1,1}]_{(m+1)} & [\lambda_{2,1}]_{(m)}\\
 [\lambda_{0,2}]_{(m)} &  \lambda_{1,2}          & \lambda_{2,2}\\
 \lambda_{0,3}         &  \vdots                 & \vdots\\
                       & \lambda_{1,m+1}         & \lambda_{2,m+2}
 \end{Bmatrix},
 \allowdisplaybreaks\\
 \sum_{j=0}^p\sum_{\nu=1}^{n_j} m_{j,\nu}\lambda_{j,\nu}=2m
 \qquad(\text{Fuchs relation}).
\end{gather*}

Theorem~\ref{thm:irrKac} says that 
the corresponding equation is irreducible if and only if any value of the 
following linear functions is not an integer.
\begin{align*}
 L^{(1)}_{i,\nu}&:=
   \lambda_{0,i}+\lambda_{1,1}+\lambda_{2,\nu}\qquad(i=1,2,\ \ \nu=2,\dots,m+2),\\
 L^{(2)}&:=
  \lambda_{0,3}+\lambda_{1,1}+\lambda_{2,1},\\
 L^{(3)}_{\mu,\nu}&:=
  \lambda_{0,1}+\lambda_{0,2}+\lambda_{1,1}+\lambda_{1,\mu}
   +\lambda_{2,1}+\lambda_{2,\nu}-1\\
 &\qquad\qquad\qquad\quad(\mu=2,\dots,m+1,\ \ \nu=2,\dots,m+2),\\
 L^{(4)}_i&:=
   \lambda_{0,i}+\lambda_{1,1}+\lambda_{2,1}\qquad(i=1,2).
\end{align*}
Here $L^{(1)}_{i,\nu}$ (resp.~$L^{(2)}$ etc.) correspond to the terms 
$10,01,01$ and $H_1\oplus B_{2m}:2(m+1)$ 
(resp.~$01,10,10$ and $H_1\oplus C_{2m}:1$ etc.) in \S\ref{sec:Bn}.

It follows from Theorem~\ref{thm:univmodel}  and Theorem~\ref{thm:irrKac} 
that the Fuchsian differential equation with the above Riemann scheme belongs 
to the universal equation $P_{B_{2m+1}}(\lambda)u=0$ if
\[
  L^{(4)}_i\notin\{-1,-2,\ldots,1-m\}\quad(i=1,\ 2).
\]

Theorem~\ref{thm:c} says that the connection coefficient
$c(\lambda_{1,m+1}\rightsquigarrow\lambda_{2,m+2})$ equals
\[
 \frac
 {\prod_{\mu=1}^m\Gamma(\lambda_{1,m+1}-\lambda_{1,\mu}+1)\cdot
  \prod_{\mu=1}^{m+1}\Gamma(\lambda_{2,\nu}-\lambda_{1,m+2})}
 {\prod_{i=1}^2\Gamma(1-L^{(1)}_{i,m+2})\cdot
  \prod_{\nu=2}^{m+1}\Gamma(L^{(3)}_{m+1,\nu})\cdot
  \prod_{\mu=2}^m\Gamma(1-L^{(3)}_{\mu,m+2})}
\]
and
\begin{align*}
  c(\lambda_{1,m+1}\rightsquigarrow\lambda_{0,3})&
 =\frac
 {\prod_{\mu=1}^m\Gamma(\lambda_{1,m+1}-\lambda_{1,\mu}+1)\cdot
  \prod_{i=1}^{2}\Gamma(\lambda_{0,i}-\lambda_{0,3})}
 {\Gamma(1-L^{(2)})\cdot
  \prod_{\nu=2}^{m+1}\Gamma(L^{(3)}_{m+1,\nu})},\\
   c(\lambda_{2,m+2}\rightsquigarrow\lambda_{0,3})&
 = \frac
   {\prod_{\nu=1}^{m+1}\Gamma(\lambda_{2,m+2}-\lambda_{1,\nu}+1)\cdot
   \prod_{i=1}^{2}\Gamma(\lambda_{0,i}-\lambda_{0,3})}
   {\prod_{i=1}^2 \Gamma(L^{(1)}_{m+2})\cdot
   \prod_{\nu=2}^{m+1}\Gamma(L^{(3)}_{m+1,\nu})}.
\end{align*}

It follows from Theorem~\ref{thm:sftUniv} that the universal operators%
\index{shift operator}%
\index{Fuchsian differential equation/operator!universal operator}
\[
\begin{aligned}
 &P_{H_1}^{0}(\lambda) && P_{H_1}^{2}(\lambda)
 &&P_{B_{2m}}^{0}(\lambda) && P_{B_{2m}}^{1}(\lambda)
 &&P_{B_{2m}}^{2}(\lambda)
 &&P_{C_{2m}}^{1}(\lambda)
 && P_{C_{2m}}^{2}(\lambda)\\  
 &P_{H_2}^{1}(\lambda) && P_{H_2}^{2}(\lambda) 
 &&P_{B_{2m-1}}^{0}(\lambda) && P_{B_{2m-1}}^{1}(\lambda)
 &&P_{B_{2m-1}}^{2}(\lambda)
\end{aligned}
\]
define shift operators $R_{B_{2m+1}}(\epsilon,\lambda)$ under
the notation in the theorem.

We also explain how we get the data in \S\ref{sec:Bn}.
Since $\p_{max}:B_{2m+1}=\mathbf m:=mm1,m+11^m,m1^{m+1}\to
H_{m+1}=\mathbf m':=0m1,11^m,01^{m+1}$, the equality \eqref{eq:DmInd} shows
\begin{align*}
 [\Delta(B_{2m+1})]&=[\Delta(H_{m+1})]
 \cup\{d_{1,1,1}(\mathbf m)\}\cup\{m'_{j,\nu}-m'_{j,1}>0\}\\
 &=1^{(m+1)^2}\times m^1\times 1^{m+2}\cdot m^1=1^{(m+1)^2}\times 1^{m+2}\cdot m^2
 = 1^{m^2+3m+3}\cdot m^2.
\end{align*}
Here we note that $\{m'_{j,\nu}-m'_{j,1}>0\}=\{m,1,1^{m+1}\}=1^{m+2}\cdot m^1$
and $[\Delta(H_{m+1})]$ is given in \S\ref{sec:GHG}.
\smallskip

We check \eqref{eq:sumDelta} for $\mathbf m$ as follows:\\
$h(\mathbf m)=%
2(1+\cdots+m)+(2m+1)+2(m+1)+1$\\
$\phantom{h(\mathbf m)}=m^2+5m+4,$\\
$\sum_{i\in[\Delta(\mathbf m)]}i =
 (m^2+3m+3)+2m=m^2+5m+3$.

\vspace{-1.7cm}\hspace{7.6cm}
{\small \begin{xy}
 \ar@{-}  *+!D{1} *{\circ} ;
  (5,0)   *+!D{2} *{\circ}="A"
 \ar@{-} "A";(8,0)
 \ar@{.} (8,0);(13,0)
 \ar@{-} (13,0);(16,0) *+!D{m} *{\circ}="A"
 \ar@{-} "A";(21,0) *+!U{2m+1} *{\circ}="A"
 \ar@{-} "A";(26,0) *+!D{\ \ m+1} *{\circ}="C"
 \ar@{-} "C";(29,0)
 \ar@{.} (29,0);(34,0)
 \ar@{-} (34,0);(37,0) *+!D{2} *{\circ}="C"
 \ar@{-} "C";(42,0) *+!D{1} *{\circ}="C"
 \ar@{-} "A";(21,5) *+!LD{m+1} *{\circ}="A"
 \ar@{-} "A";(21,10) *+!LD{1} *{\circ}
\end{xy}}

The decompositions $mH_1\oplus H_{m+1}$ and $H_1\oplus B_{2m}$ etc.\ in 
\S\ref{sec:Bn} are easily obtained and we should show that they are all 
the decompositions \eqref{eq:allDec}, whose number is given by 
$[\Delta(B_{2m+1})]$.
There are 2 decompositions of type $mH_1\oplus H_{m+1}$, namely,
$B_{2m+1}=mm1,m+11^m,m1^{m+1}=m(100,10,10)\oplus\cdots=m(010,10,10)\oplus\cdots$,
which correspond to $L^{(4)}_i$ for $i=1$ and $2$.
Then the other decompositions are of type $\mathbf m'\oplus\mathbf m''$
with rigid tuples $\mathbf m'$ and $\mathbf m''$ whose number equals $m^2+3m+3$.
The numbers of decompositions $H_1\oplus B_{2m}$ etc.\ given in \S\ref{sec:Bn}
are easily calculated which correspond to $L^{(1)}_{i,\nu}$ etc.\ and we can check 
that they give the required number of the decompositions.
\subsubsection{$C_n$}
($C_4=EO_4$, $C_3=H_3$, $C_2=H_2$)
\begin{align*}
 u_{C_{2m+1}}&=\p^{-\mu'}x^{\lambda'}u_{H_{m+1}}\\
 m+1m,m1^{m+1},m1^{m+1}&=10,01,10\oplus
  m^2,m1^m,m-11^{m+1}\\
 &=11,11,11\oplus m(m-1),m-11^mm-11^m
\allowdisplaybreaks\\
[\Delta(C_{2m+1})]&=1^{(m+1)^2}\times1^{2m+2}\cdot m\cdot(m-1)\\
&=1^{(m+1)(m+3)}\cdot m\cdot(m-1)\\
&\hspace{-1.08cm}
\begin{aligned}
 C_{2m+1}&=H_1\oplus C_{2m}&&:2m+2\\
 &=H_2\oplus C_{2m-2}&&:(m+1)^2\\
 &=mH_1\oplus H_{m+1}&&:1\\
 &=(m-1)H_1\oplus H_{m+2}&&:1
\end{aligned}
\allowdisplaybreaks\\[5pt]
 u_{C_{2m}}&=\p^{-\mu'}x^{\lambda'}
 (1-x)^{-\lambda_1-\mu-\mu^{(2)}-\cdots-\mu^{(m)}}
 u_{H_{m+1}}
\allowdisplaybreaks\\
 m^2,m1^m,m-11^{m+1}&=1,10,01\oplus mm-1,m-11^{m-1},m-11^{m-1}\\
 &=1^2,11,11\oplus (m-1)^2,m-11^{m-1},m-21^m
\allowdisplaybreaks\\
[\Delta(C_{2m})]&=1^{(m+1)^2}\times1^{m+1}\cdot (m-1)^2=1^{m^2+3m+2}\cdot(m-1)^2\\
&\hspace{-.725cm}
\begin{aligned}
 C_{2m}&=H_1\oplus C_{2m-1}&&:2m+2\\
 &=H_2\oplus C_{2m-2}&&:m(m+1)\\
 &=(m-1)H_1\oplus H_{m+1}&&:2
\end{aligned}
\allowdisplaybreaks\\
 C_{2m+1}&\overset{m}{\underset{R2E0R0E0}\longrightarrow} H_{m+1},\quad
 C_{2m+1}\overset{m-1}\longrightarrow H_{m+2}\\
 C_{2m}&\overset{m-1}{\underset{R1E0R0E0}\longrightarrow} H_{m+1},\quad
 C_n\overset{1}\longrightarrow C_{n-1}
\end{align*}

\subsubsection{$D_n$} ($D_6=X_6$ : Extra case, $D_5=EO_5$)
\begin{align*}
 u_{D_5}&=\p^{-\mu_5}(1-x)^{-\lambda_3-\mu_3-\mu_4}u_{E_4}\\
 u_{D_6}&=\p^{-\mu_6}(1-x)^{-\lambda_1-\mu-\mu_5}u_{D_5}\\
 u_{D_n}&=\p^{-\mu_n}(1-x)^{-\lambda'_n}u_{D_{n-2}}\quad(n\ge7)
\allowdisplaybreaks\\
 (2m-1)2,2^m1,2^{m-2}1^5&=10,01,10\oplus(2m-2)2,2^m,2^{m-3}1^6\\
 &=10,10,01\oplus(2m-2)2,2^{m-1}1^2,2^{m-3}1^4\\
 &=(m-1)1,1^m0,1^{m-2}1^2\oplus m1,1^m1,1^{m-2}1^3
\allowdisplaybreaks\\
m\ge2\ \Rightarrow\ [\Delta(D_{2m+1})]&=1^{6m+2}\cdot 2^{(m-1)(m-3)}\times
  1^6\cdot 2^{2m-3}=1^{6m+8}\cdot 2^{m(m-2)}\\
&\hspace{-1.12cm}
\begin{aligned}
 D_{2m+1}&=H_1\oplus D_{2m}&&:m-2\\
   &=H_1\oplus E_{2m}&&:5m\\
   &=H_m\oplus H_{m+1}&&:10\\
   &=2H_1\oplus D_{2m-1}&&:m(m-2)
\end{aligned}
\allowdisplaybreaks\\[5pt]
 (2m-2)2,2^m,2^{m-3}1^6&=10,1,01\oplus
 (2m-3)2,2^{m-1}1,2^{m-3}1^5\\
 &=(m-1)1,1^m,1^{m-3}1^3\oplus(m-1)1,1^m,1^{m-3}1^3
\allowdisplaybreaks\\
m\ge3\ \Rightarrow\ [\Delta(D_{2m})]&=1^{6m+6}\cdot 2^{(m-1)(m-4)}\times
  1^6\cdot 2^{2m-4}=1^{6m+10}\cdot 2^{m(m-3)}\\
&\hspace{-.74cm}
\begin{aligned}
  D_{2m}&=H_1\oplus D_{2m-1}&&:6m\\
  &=H_m\oplus H_m&&:10\\
  &=2H_1\oplus D_{2m-2}&&:m(m-3)
\end{aligned}
\allowdisplaybreaks\\
 D_n &\overset{2}{\underset{R2E0}\longrightarrow} D_{n-2},\quad
 D_n\overset{1}\longrightarrow D_{n-1},\quad
 D_{2m+1}\overset{1}\longrightarrow E_{2m}
\end{align*}

\subsubsection{$E_n$} ($E_5=C_5$, $E_4=EO_4$, $E_3=H_3$)
\begin{align*}
 u_{E_3}&= x^{-\lambda_0-\mu-\mu_3}\p^{-\mu_3}(1-x)^{\lambda'_3}u_{H_2}\\
 u_{E_4}&= \p^{-\mu_4}u_{E_3}\\
 u_{E_n}&=\p^{-\mu_n}(1-x)^{\lambda'_n}u_{E_{n-2}}\quad(n\ge 5)
\allowdisplaybreaks\\
 (2m-1)2,2^{m-1}1^3,2^{m-1}1^3&=10,01,10\oplus (2m-2)2,2^{m-1}1^2,2^{m-2}1^4\\
 &=(m-1)1,1^{m-1}1,1^{m-1}1\oplus m1,1^{m-1}1^3,1^{m-1}1^2\\
 &=(m-2)1,1^{m-1}0,1^{m-1}0\oplus (m+1)1,1^{m-1}1^2,1^{m-1}1^3
\allowdisplaybreaks\\
m\ge2\ \Rightarrow\ [\Delta(E_{2m+1})]&=1^{6m-2}\cdot 2^{(m-2)^2}\times
  1^6\cdot 2^{2m-3}=1^{6m+4}\cdot 2^{(m-1)^2}\\
&\hspace{-1.08cm}
\begin{aligned}
E_{2m+1}&=H_1\oplus E_{2m}&&:6(m-1)\\
 &=H_{m-1}\oplus H_{m+2}&&:1\\
 &=H_m\oplus H_{m+1}&&:9\\
 &=2H_1\oplus E_{2m-1}&&:(m-1)^2
\end{aligned}
\allowdisplaybreaks\\[5pt]
 (2m-2)2,2^{m-1}1^2,2^{m-2}1^4&=10,10,01\oplus (2m-3)2,2^{m-2}1^3,2^{m-2}1^3\\
 &=10,01,10\oplus (2m-3)2,2^{m-1}1,2^{m-3}1^5\allowdisplaybreaks\\
 &=(m-2)1,1^{m-1}0,1^{m-2}1\oplus m1,1^{m-1}1^2,1^{m-2}1^3\\
 &=(m-1)1,1^{m-1}1,1^{m-2}1^2\oplus (m-1)1,1^{m-1}1,1^{m-2}1^2
\allowdisplaybreaks\\
m\ge2\ \Rightarrow\ [\Delta(E_{2m})]&=1^{6m-4}\cdot 2^{(m-2)(m-3)}\times
  1^6\cdot 2^{2m-4}=1^{6m+2}\cdot 2^{(m-1)(m-2)}\\
&\hspace{-0.72cm}
\begin{aligned}
 E_{2m}&=H_1\oplus E_{2m-1}&&:4(m-1)\\
       &=H_1\oplus D_{2m-1}&&:2(m-2)\\
       &=H_{m-1}\oplus H_{m+1}&&:4\\
       &=H_m\oplus H_m&&:6\\
       &=2H_1\oplus E_{2m-2}&&:(m-1)(m-2)
\end{aligned}
\allowdisplaybreaks\\
 E_n &\overset{2}{\underset{R2E0}\longrightarrow} E_{n-2},\quad
 E_n\overset{1}\longrightarrow E_{n-1},\quad
 E_{2m}\overset{1}\longrightarrow D_{2m-1}
\end{align*}

\subsubsection{$F_n$} ($F_5=B_5$, $F_4=EO_4$, $F_3=H_3$)
\begin{align*}
 u_{F_3}&=u_{H_3}\\
 u_{F_4}&=\p^{-\mu_4}(1-x)^{-\lambda_1-\lambda_0^{(3)}-\mu^{(3)}}
          u_{F_3}\\
 u_{F_n}&=\p^{-\mu_n}(1-x)^{\lambda'_n}u_{F_{n-2}}\quad(n\ge 5)
\allowdisplaybreaks\\
 (2m-1)1^2,2^m1,2^{m-1}1^3&=10,10,01\oplus (2m-2)1^2,2^{m-1}1^2,2^{m-1}1^2\\
 &=10,01,10\oplus (2m-2)1^2,2^m,2^{m-2}1^4\\
 &=(m-1)1,1^m0,1^{m-1}1\oplus m1,1^m1,1^{m-1}1^2
\allowdisplaybreaks\\
m\ge1\ \Rightarrow\ [\Delta(F_{2m+1})]&=1^{4m+1}\cdot 2^{(m-1)(m-2)}\times
  1^4\cdot 2^{2m-2}=1^{4m+5}\cdot 2^{m(m-1)}\\
&\hspace{-1.05cm}
\begin{aligned}
 F_{2m+1}&=H_1\oplus G_{2m}&&:3m\\
         &=H_1\oplus F_{2m}&&:m-1\\
         &=H_m\oplus H_{m+1}&&:6\\
         &=2H_1\oplus F_{2m-1}&&:m(m-1)
\end{aligned}
\allowdisplaybreaks\\[5pt]
 (2m-2)1^2,2^m,2^{m-2}1^4&=10,1,01\oplus (2m-3)1^2,2^{m-1}1,2^{m-2}1^3\\
 &=(m-1)1,1^m,1^{m-2}1^2\oplus(m-1)1,1^m,1^{m-2}1^2
\allowdisplaybreaks\\
m\ge2\ \Rightarrow\ [\Delta(F_{2m})]&=1^{4m+2}\cdot 2^{(m-1)(m-3)}\times
  1^4\cdot 2^{2m-3}=1^{4m+6}\cdot 2^{m(m-2)}\\
&\hspace{-0.7cm}
\begin{aligned}
 F_{2m}&=H_1\oplus F_{2m-1}&&:4m\\
       &=H_m\oplus H_m&&:6\\
       &=2H_1\oplus F_{2m-2}&&:m(m-2)
\end{aligned}
\allowdisplaybreaks\\
 F_n &\overset{2}{\underset{R2E0}\longrightarrow} F_{n-2},\quad
 F_n\overset{1}\longrightarrow F_{n-1},\quad
 F_{2m+1}\overset{1}\longrightarrow G_{2m}
\end{align*}

\subsubsection{$G_{2m}$} ($G_4=B_4$)
\begin{align*}
 u_{G_2}&=u_{H_2}\\
 u_{G_{2m}}&=\p^{-\mu_{2m}}(1-x)^{\lambda'_{2m}}u_{G_{2m-2}}
\allowdisplaybreaks\\
 (2m-2)1^2,2^{m-1}1^2,2^{m-1}1^2&=10,01,01\oplus
  (2m-3)1^2,2^{m-1}1,2^{m-2}1^3\\
 &=(m-2)1,1^{m-1}0,1^{m-1}0\oplus m1,1^{m-1}1^2,1^{m-1}1^2
\allowdisplaybreaks\\
m\ge2\ \Rightarrow\ [\Delta(G_{2m})]&=1^{4m-2}\cdot 2^{(m-2)^2}\times
  1^4\cdot 2^{2m-3}=1^{4m+2}\cdot 2^{(m-1)^2}\\
&\hspace{-0.72cm}
\begin{aligned}
 G_{2m}&=H_1\oplus F_{2m-1}&&:4m\\
       &=H_{m-1}\oplus H_{m+1}&&:2\\
       &=2H_1\oplus G_{2m-2}&&:(m-1)^2
\end{aligned}
\allowdisplaybreaks\\
 G_{2m} &= H_1\oplus F_{2m-1} = H_{m-1}\oplus H_{m+1}
\allowdisplaybreaks\\
 G_{2m} &\overset{2}{\underset{R2E0}\longrightarrow} G_{2(m-1)},\quad
 G_{2m}\overset{1}\longrightarrow F_{2m-1}
\end{align*}

\subsubsection{$I_n$}
($I_{2m+1}=\text{III}^*_m$, $I_{2m}=\text{II}^*_m$, $I_3=P_3$)
\begin{align*}
 u_{I_{2m+1}}&=\p^{-\mu'}x^{\lambda'}(c-x)^{\lambda''}u_{H_{m}}
\allowdisplaybreaks\\
 &\hspace{-1cm}(2m)1,m+1m,m+11^{m},m+11^{m}\\
 &=10,10,10,01\oplus (2m-1)1,mm,m1^{m},m+11^{m-1}\\
 &=20,11,11,11\oplus (2m-2)1,mm-1,m1^{m-1},m1^{m-1}
\allowdisplaybreaks\\
 [\Delta(I_{2m+1})]&=1^{m^2}\times
  1^{2m}\cdot m\cdot(m+1)=1^{m^2+2m}\cdot m\cdot(m+1)\\
&\hspace{-0.9cm}
\begin{aligned}
 I_{2m+1}&=H_1\oplus I_{2m}&&:2m\\
  &=H_2\oplus I_{2m-1}&&:m^2\\
  &=mH_1\oplus H_{m+1}&&:1\\
  &=(m+1)H_1\oplus H_m&&:1
\end{aligned}
\allowdisplaybreaks\\[5pt]
 u_{I_{2m}}&=\p^{-\mu'}(1-cx)^{\lambda''}u_{H_{m}}\\
 &\hspace{-1cm}(2m-1)1,mm,m1^{m},m+11^{m-1}\\
 &=10,01,01,10\oplus (2m-2)1,mm-1,m1^{m-1},m1^{m-1}\\
 &=20,11,11,11\oplus (2m-3)1,m-1m-1,m-11^{m-1},m1^{m-2}
\allowdisplaybreaks\\
 [\Delta(I_{2m})]&=1^{m^2}\times
  1^{m}\cdot m^2=1^{m(m+1)}\cdot m^2\\
&\hspace{-0.6cm}
\begin{aligned}
 I_{2m}&=H_1\oplus I_{2m-1}&&:2m\\
       &=H_2\oplus I_{2m-2}&&:m(m-1)\\
       &=mH_1\oplus H_m&&:2
\end{aligned}
\allowdisplaybreaks\\
 I_{2m+1}&\overset{m+1}\longrightarrow H_m,\quad 
 I_{2m+1}\overset{m}\longrightarrow H_{m+1},\quad 
 I_{2m}\overset{m}\longrightarrow H_m,\quad 
 I_{n}\overset{1}\longrightarrow I_{n-1}\\
 I_{2m+1}&\underset{R1E0}\longrightarrow I_{2m}
 \underset{R2E0}\longrightarrow I_{2m-2}
\end{align*}
\subsubsection{$J_n$} ($J_4=I_4$, $J_3=P_3$)
\begin{align*}
 u_{J_2}&=(c-x)^{\lambda'}u_{H_2}\\
 u_{J_3}&=u_{P_3}\\
 u_{J_n}&=\p^{-\mu'_n}x^{\lambda'_n}u_{J_{n-2}}\quad(n\ge4)
\allowdisplaybreaks\\
 &\hspace{-1cm}(2m)1,(2m)1,2^{m}1,2^{m}1\\
 &=10,10,01,10\oplus (2m-1)1,(2m-1)1,2^{m},2^{m-1}11\\
 &=(m-1)1,m0,1^m0,1^m0\oplus(m+1),m1,1^m1,1^m1
\allowdisplaybreaks\\
[\Delta(J_{2m+1})]&=1^{2m}\cdot2^{(m-1)^2}\times
  1^2\cdot 2^{2m-1}=1^{2m+2}\cdot2^{m^2}\\
&\hspace{-1cm}
\begin{aligned}
 J_{2m+1}&=H_1\oplus J_{2m}&&:2m\allowdisplaybreaks\\
         &=H_m\oplus H_{m+1}&&:2\allowdisplaybreaks\\
         &=2H_1\oplus J_{2m-2}&&:m^2
\end{aligned}
\allowdisplaybreaks\\[5pt]
 &\hspace{-1cm}(2m-1)1,(2m-1)1,2^m,2^{m-1}1^2\\
 &=10,10,1,01\oplus (2m-2)1,(2m-2)1,2^{m-1}1,2^{m-1}1\\
 &=(m-1)1,m0,1^m,1^{m-1}1\oplus m0,(m-1)1,1^m,1^{m-1}1
\allowdisplaybreaks\\
 [\Delta(J_{2m})]&=1^{2m}\cdot2^{(m-1)(m-2)}\times
  1^2\cdot2^{2m-2}=1^{2m+2}\cdot2^{m(m-1)}\\
&\hspace{-.65cm}
\begin{aligned}
 J_{2m}&=H_1\oplus J_{2m-1}&&:2m\allowdisplaybreaks\\
       &=H_m\oplus H_m&&:2\allowdisplaybreaks\\
       &=2H_1\oplus J_{2m-2}&&:m(m-1)\\
\end{aligned}
\allowdisplaybreaks\\
 J_{n}&\overset{2}{\underset{R2E0}\longrightarrow} J_{n-2}\ (n\ge6),\quad
  J_{n}\overset{1}\longrightarrow J_{n-1}
\end{align*}
\subsubsection{$K_n$} ($K_5=M_5$, $K_4=I_4$, $K_3=P_3$)
\begin{align*}
 &u_{K_{2m+1}}=\p^{\mu+\lambda'+\lambda''}(c'-x)^{\lambda'}(c''-x)^{\lambda''}u_{P_m}
\allowdisplaybreaks\\
 &m+1m,m+1m,(2m)1,(2m)1,(2m)1,\ldots\in \mathcal P_{m+3}^{(2m+1)}\\
 &\quad= 11,11,11,20,20,\ldots\oplus mm-1,mm-1,(2m-1)0,(2m-2)1,(2m-2)1,\ldots
\allowdisplaybreaks\\
 &[\Delta(K_{2m+1})]=1^{m+1}\cdot(m-1)\times
  m^2\cdot(m+1)=1^{m+1}\cdot(m-1)\cdot m^2\cdot(m+1)\\
 &
\qquad\begin{aligned}
 K_{2m+1}&=H_2\oplus K_{2m-1}&&:m+1\\
         &=(m-1)H_1\oplus P_{m+2}&&:1\\
         &=mH_1\oplus P_{m+1}&&:2\\
         &=(m+1)H_1\oplus P_m&&:1
\end{aligned}
\allowdisplaybreaks\\[5pt]
 &u_{K_{2m}}=\p^{-\mu'}(c'-x)^{\lambda'}u_{P_m}\\
 &mm,mm-11,(2m-1)1,(2m-1)1,\ldots\in\mathcal P_{m+2}^{(2m)}\\
 &\quad=01,001,10,10,10,\ldots\oplus mm-1,mm-10,(2m-2)1,(2m-2)1,\ldots
\allowdisplaybreaks\\
 &\quad=11,110,11,20,20,\ldots\oplus m-1m-1,m-1m-21,(2m-2)0,(2m-3)1,\ldots
\allowdisplaybreaks\\
 &[\Delta(K_{2m})]=1^{m+1}\cdot(m-1)\times
  1\cdot(m-1)\cdot m^2=1^{m+2}\cdot(m-1)^2\cdot m^2
\allowdisplaybreaks\\
&\qquad\begin{aligned}
 K_{2m}&=H_1\oplus K_{2m-1}&&:2\\
       &=H_2\oplus K_{2m-2}&&:m\\
       &=(m-1)H_1\oplus P_{m+1}&&:2\\
       &=mH_1\oplus P_m&&:2
\end{aligned}
\allowdisplaybreaks\\
 &\ \ K_{2m+1}\overset{m+1}\longrightarrow P_{m},\quad
  K_{2m+1}\overset{m}{\underset{R1}\longrightarrow} P_{m+1},\quad
  K_{2m+1}\overset{m-1}\longrightarrow P_{m+2}\\
 &\ \ K_{2m}\overset{m}{\underset{R1}\longrightarrow} P_{m},\quad 
  K_{2m}\overset{m-1}\longrightarrow P_{m+1},\quad
  K_{2m}\overset{1}\longrightarrow K_{2m-1}
\end{align*}
\subsubsection{$L_{2m+1}$} ($L_5=J_5$, $L_3=H_3$)
\begin{align*}
 &u_{L_{2m+1}}=\p^{-\mu'}x^{\lambda'}u_{P_{m+1}}
\allowdisplaybreaks\\
 &mm1,mm1,(2m)1,(2m)1,\ldots\in\mathcal P_{m+2}^{(2m+1)}\\
 &\quad=001,010,10,10,\ldots\oplus mm0,mm-11,(2m-1)1,(2m-1)1,\ldots\\
 &\quad=110,110,11,20,\ldots\oplus m-1m-10,m-1m-11,(2m-1)0,(2m-2)1,\ldots
\allowdisplaybreaks\\
 &[\Delta(L_{2m+1})]=1^{m+2}\cdot m\times
  1^{2}\cdot m^3=1^{m+4}\cdot m^4
\\&\quad\ \ \ 
\begin{aligned}
 L_{2m+1}&=H_1\oplus K_{2m}&&:4\\
         &=H_2\oplus L_{2m-1}&&:m\\
         &=mH_1\oplus P_{m+1}&&:4
\end{aligned}
\allowdisplaybreaks\\
 &\ \ L_{2m+1}=H_1\oplus K_{2m},\quad L_{2m+1}=H_2\oplus L_{2m-1}\\
 &\ \ L_{2m+1}\overset{m}{\underset{R2E0}\longrightarrow} P_{m+1},\quad
  L_{2m+1}\overset{1}\longrightarrow K_{2m}
\end{align*}
\subsubsection{$M_n$} ($M_5=K_5$, $M_4=I_4$, $M_3=P_3$)
\begin{align*}
  &u_{M_{2m+1}}=\p^{\mu+\lambda'_3+\cdots+\lambda'_{m+2}}
    (c_3-x)^{\lambda'_3}\cdots(c_{m+2} -x)^{\lambda'_{m+2}}u_{H_2}
\allowdisplaybreaks\\
  &(2m)1,(2m)1,(2m)1,(2m-1)2,(2m-1)2,\ldots\in \mathcal P^{(2m+1)}_{m+3}\\
  &\quad=m-11,m0,m0,m-11,m-11,\ldots
   \oplus m+10,m1,m1,m1,m1,\ldots\\
  &\quad
   =m-10,m-10,m-10,m-21,m-21,\ldots\\
  &\quad\quad
   \oplus m+11,m+11,m+11,m+11,m+11,\ldots
\allowdisplaybreaks\\
 &[\Delta(M_{2m+1})]=1^{4}\times
  2^{m}\cdot(2m-1)=1^4\cdot2^m\cdot(2m-1)
\allowdisplaybreaks\\
&\qquad\begin{aligned}
 M_{2m+1}&=P_{m-1}\oplus P_{m+2} &&:1\\
         &=P_{m}\oplus P_{m+1} &&:3\\
         &=2H_1\oplus M_{2m-1} &&:m\\
         &=(2m-1)H_1\oplus H_2 &&:1
 \end{aligned}
\allowdisplaybreaks\\[5pt]
  &u_{M_{2m}}=\p^{-\mu'}(c_3-x)^{\lambda'_3}\cdots(c_{m+1} -x)^{\lambda'_{m+1}}
   u_{H_2}\\
  &(2m-2)1^2,(2m-1)1,(2m-1)1,(2m-2)2,\ldots\in\mathcal P^{(2m)}_{m+2}\\
  &\quad=01,10,10,10,\ldots\oplus (2m-2)1,(2m-2)1,(2m-2)1,(2m-3)2,\ldots
\allowdisplaybreaks\\
  &\quad=m-21,m-10,m-10,m-21,\ldots\oplus m1,m1,m1,m1,\ldots\\
  &\quad=m-11,m-11,m0,m-11,\ldots\oplus m-11,m0,m-11,m-11,\ldots
\allowdisplaybreaks\\
  &[\Delta(M_{2m})]=1^{4}\times
  1^{2}\cdot2^{m-1}\cdot(2m-2)=1^6\cdot2^{m-1}\cdot(2m-2)\\
&\qquad\begin{aligned}
  M_{2m}&=H_1\oplus M_{2m-1} &&:2 \\
        &=P_{m-1}\oplus P_{m+1} &&:2\\
        &=P_m\oplus P_m&&:2\\
        &=2H_1\oplus M_{2m-2} &&:m-1\\
        &=(2m-2)H_1\oplus H_2 &&:1\\
 \end{aligned}
\allowdisplaybreaks\\
   &\ \ M_{n}\overset{n-2}\longrightarrow H_2,\quad 
   M_{n}\overset{2}\longrightarrow M_{n-2},\quad
   M_{2m}\overset{1}{\underset{R1E0}\longrightarrow} M_{2m-1}
   \underset{R1}\longrightarrow M_{2m-3}
\end{align*}
\subsubsection{$N_n$}\label{sec:minseries}
($N_6=\text{IV}^*$, $N_5=I_5$, $N_4=G_4$, $N_3=H_3$)
\index{tuple of partitions!rigid!minimal series}%
\begin{align*}
 &u_{N_{2m+1}}=\p^{-\mu'}x^{\lambda'}(c_3-x)^{\lambda'_3}\cdots(c_{m+1}-x)^{\lambda'_{m+1}}u_{H_2}
\allowdisplaybreaks\\
 &(2m-1)1^2,(2m-1)1^2,(2m)1,(2m-1)2,(2m-1)2,\ldots\in \mathcal P^{(2m+1)}_{m+2}\\
 &\quad =10,01,10,10,10\ldots\\
 &\quad\quad\oplus
  (2m-2)1^2,(2m-1)1,(2m-1)1,(2m-2)2,(2m-2)2,\ldots\\
 &\quad =m-11,m-11,m0,m-11,m-11,\ldots
  \oplus
  m1,m1,m1,m1,m1,\ldots
\allowdisplaybreaks\\
 &[\Delta(N_{2m+1})]=1^{4}\times
  1^{4}\cdot2^{m-1}\cdot(2m-1)=1^{8}\cdot2^{m-1}\cdot(2m-1)\\
&\qquad\begin{aligned}
 N_{2m+1}&=H_1\oplus M_{2m}&&:4\\
         &=P_m\oplus P_{m+1}&&:4\\
         &=2H_1\oplus N_{2m-1}&&:m-1\\
         &=(2m-1)H_1\oplus H_2&&:1\\
\end{aligned}
\allowdisplaybreaks\\[5pt]
 &u_{N_{2m}}=\p^{-\mu'}x^{\lambda'_0}(1-x)^{\lambda'_1}
  (c_3-x)^{\lambda'_3}\cdots(c_{m}-x)^{\lambda'_{m}}u_{H_2}
 \quad(m\ge2)\\
 &(2m-2)1^2,(2m-2)1^2,(2m-2)1^2,(2m-2)2,(2m-2)2,\ldots\in \mathcal P^{(2m)}_{m+1}
 \\
 &\quad =01,10,10,10,10\ldots\\
 &\quad\quad
  \oplus (2m-2)1,(2m-3)1^2,(2m-3)1^2,(2m-3)2,(2m-3)2,\ldots\allowdisplaybreaks\\
 &\quad =m-11,m-11,m-11,m-11,m-11,\ldots\\
 &\quad\quad
  \oplus m-11,m-11,m-11,m-11,m-11,\ldots
\allowdisplaybreaks\\
 &[\Delta(N_{2m})]=1^{4}\times
  1^{6}\cdot 2^{m-2}\cdot(2m-2)=1^{10}\cdot2^{m-2}\cdot(2m-2)
\allowdisplaybreaks\\
&\qquad\begin{aligned}
 N_{2m}&=H_1\oplus N_{2m-1}&&:6\\
       &=P_m\oplus P_m&&:4\\
       &=2H_1\oplus N_{2m-2}&&:m-2\\
       &=(2m-2)H_1\oplus H_2&&:1
\end{aligned}
\allowdisplaybreaks\\
 &N_{n}\overset{n-2}\longrightarrow H_2,\quad
  N_{n}\overset{2}\longrightarrow N_{n-2},\quad
  N_{2m+1}\overset{1}{\underset{R1E0}\longrightarrow} M_{2m},\quad
  N_{2m}\overset{1}{\underset{R1E0}\longrightarrow} N_{2m-1}
\end{align*}
\subsubsection{minimal series}\label{sec:minS}%
\index{tuple of partitions!rigid!minimal series}
The tuple $11,11,11$ corresponds to Gauss hypergeometric series, which has
three parameters.  Since the action of additions is easily analyzed, we consider
the number of parameters of the equation corresponding to a rigid tuple
$\mathbf m=\bigl(m_{j,\nu}\bigr)_{\substack{0\le j\le p\\1\le \nu\le n_j}}
\in\mathcal P_{p+1}^{(n)}$ modulo additions and the Fuchs condition equals
\begin{equation}\label{eq:NSpara}
 n_0+n_1+\cdots+n_p - (p+1).
\end{equation}
Here we assume that $0<m_{j,\nu}<n$ for $1\le \nu\le n_j$ and $j=0,\dots,p$.

We call the number given by \eqref{eq:NSpara} the \textsl{effective length}\/ 
of $\mathbf m$.
The tuple $11,11,11$ is the unique rigid tuple of partitions whose effective 
length equals 3.
Since the reduction $\p_{max}$ never increase the effective length and 
the tuple $\mathbf m\in\mathcal P_3$ satisfying $\p_{max}=11,11,11$ is
$21,111,111$ or $211,211,211$, it is easy to see that the non-trivial
rigid tuple $\mathbf m\in\mathcal P_3$ whose effective length is smaller than 
$6$ is $H_2$ or $H_3$.

The rigid tuple of partitions with the effective length 4 is also uniquely
determined by its order, which is%
\index{00P@$P_{p+1,n},\ P_n$}
\begin{equation}
 \begin{split}
  P_{4,2m+1}&: m+1m,m+1m,m+1m,m+1m\\
  P_{4,2m}&:   m+1m-1,mm,mm,mm
 \end{split}
\end{equation}
with $m\in\mathbb Z_{>0}$.
Here $P_{4,2m+1}$ is a generalized Jordan-Pochhammer tuple in
Example~\ref{ex:JPH} i).\index{Jordan-Pochhammer!generalized}

In fact, if $\mathbf m\in\mathcal P$ is rigid with the effective length $4$,
the argument above shows $\mathbf m\in\mathcal P_4$ and $n_j=2$ for $j=0,\dots,3$.
Then $2 = \sum_{j=0}^3 m_{j,1}^2+\sum_{j=0}^3(n-m_{j,1})^2-2n^2$ and 
$\sum_{j=0}^3(n-2m_{j,1})^2=4$ and therefore $\mathbf m=P_{4,2m+1}$ or $P_{4,2m}$.

We give decompositions of $P_{4,n}$:
\begin{align*}
 &
 m+1,m;m+1,m;m+1,m;m+1,m\\
 &\quad=k,k+1;k+1,k;k+1,k;k+1,k\\
 &\qquad\oplus m-k+1,m-k-1;m-k,m-k;m-k,m-k;m-k,m-k\\
 &\quad=2(k+1,k;k+1,k;k+1,k;\ldots)\\
 &\qquad\oplus m-2k-1,m-2k;m-2k-1,m-2k;m-2k-1,m-2k;\ldots\\
 &[\Delta(P_{4,2m+1})]=
   1^{4m-4}\cdot2^{m-1}\times 1^4\cdot 2=1^{4m}\cdot2^m\\
 &\qquad\begin{aligned}
  P_{4,2m+1}&=P_{4,2k+1}\oplus P_{4,2(m-k)}&&:4\quad(k=0,\ldots,m-1)\\
   &=2P_{4,2k+1}\oplus P_{4,2m-4k-1}&&:1\quad(k=0,\ldots,m-1)
 \end{aligned}
\end{align*}
Here $P_{k,-n}=-P_{k,n}$ and in the above decompositions 
there appear ``tuples of partitions" with negative
entries corresponding formally to elements in $\Delta^{re}$ with 
\eqref{eq:Kazpart} (cf.~Remark~\ref{rem:length}~i)).

It follows from the above decompositions that the Fuchsian equation with 
the Riemann scheme
\begin{gather*}
 \begin{Bmatrix}
   \infty & 0 & 1 & c_3\\
   [\lambda_{0,1}]_{(m+1)} & [\lambda_{1,1}]_{(m+1)} 
     & [\lambda_{2,1}]_{(m+1)} & [\lambda_{3,1}]_{(m+1)}\\
   [\lambda_{0,2}]_{(m)} & [\lambda_{1,2}]_{(m)} 
     & [\lambda_{2,1}]_{(m)} & [\lambda_{3,2}]_{(m)}
 \end{Bmatrix}\\
 \sum_{j=0}^4\bigl((m+1)\lambda_{j,1}+m\lambda_{j,2}\bigr)=2m
 \qquad(\text{Fuchs relation}).
\end{gather*}
is irreducible if and only if
\[
  \sum_{j=0}^4\sum_{\nu=1}^2
  \bigl(k+\delta_{\nu,1}+(1-2\delta_{\nu,1})\delta_{j,i}\bigr)
  \lambda_{j,\nu}\notin\mathbb Z\qquad(i=0,1,\ldots,5,\ k=0,1,\ldots,m).
\]

When $\mathbf m=P_{4,2m}$, we have the following.
\begin{align*}
 &
 m+1,m-1;m,m;m,m;m,m\\
 &\quad=k+1,k;k+1,k;k+1,k;k+1,k\\
 &\qquad\oplus m-k,m-k-1;m-k-1,m-k;m-k-1,m-k;m-k-1,m-k\\
 &\quad=2(k+1,k-1;k,k;k,k;k,k)\\
 &\qquad\oplus m-2k-1,m-2k+1;m-2k,m-2k;m-2k,m-2k;m-2k;m-2k\\
 &[\Delta(P_{4,2m})]=
   1^{4m-4}\cdot 2^{m-1}\times 1^4=1^{4m}\cdot 2^{m-1}\\
 &\ \ \begin{aligned}
  P_{4,2m}&=P_{4,2k+1}\text{\small$(=k+1,k;k+1,k;\ldots)$}
    \oplus P_{4,2m-2k+1}&&\!:4\ \ (k=0,\ldots,m-1)\\
   &=2P_{4,2k}\oplus P_{4,2m-4k}&&\!:1\ \ (k=1,\ldots,m-1)
 \end{aligned}
\end{align*}
\[
 P_{4,n}\xrightarrow1 P_{4,n-1},\  P_{4,2m+1}\xrightarrow2 P_{4,2m-1}
\]
\smallskip

Roberts \cite{Ro} classifies the rigid tuples $\mathbf m\in\mathcal P_{p+1}$ 
so that 
\begin{equation}\index{tuple of partitions!rigid!minimal series}
 \frac1{n_0}+\cdots+\frac1{n_p}\ge p-1.
\end{equation}
They are tuples $\mathbf m$ in
4 series $\alpha$, $\beta$, $\gamma$, $\delta$, which
are close to the tuples $r\tilde E_6$, $r\tilde E_7$, $r\tilde E_8$ and 
$r\tilde D_4$, namely, $(n_0,\dots,n_p)=(3,3,3)$, $(2,2,4)$, $(2,3,6)$ and 
$(2,2,2,2)$, respectively (cf.~\eqref{eq:R2ineq}), 
and the series are called \textsl{minimal series}.
Then $\delta_n=P_{4,n}$ and the tuples in the other three series belong 
to $\mathcal P_3$.
For example, the tuples $\mathbf m$ of type $\alpha$ are
\begin{equation}
 \begin{aligned}
  \alpha_{3m}&=m+1mm-1,m^3,m^3,&\alpha_3&=H_3,\\
  \alpha_{3m\pm1}&=m^2m\pm1,m^2m\pm1,m^2m\pm1,&\alpha_4&=B_4,
 \end{aligned}
\end{equation}
which are characterized by the fact that their effective lengths 
equal 6 when $n\ge 4$.
As in other series, we have the following:
\begin{align*}
 \alpha_n&\xrightarrow 1\alpha_{n-1},\ \ \alpha_{3m+1}\xrightarrow2 \alpha_{3m-1}\\
 [\Delta(\alpha_{3m})]&=[\Delta(\alpha_{3m-1})]\times 1^5,\ 
 [\Delta(\alpha_{3m-1})]=[\Delta(\alpha_{3m-2})]\times 1^4,\\ 
 [\Delta(\alpha_{3m-2})]&=[\Delta(\alpha_{3m-4})]\times 1^6\cdot 2\\
 [\Delta(\alpha_{3m-1})]&=[\Delta(\alpha_2)]\times 1^{10(m-1)}\cdot 2^{m-1}
                        =1^{10m-6}\cdot 2^{m-1}\\
 [\Delta(\alpha_{3m})]&=1^{10m-1}\cdot2^{m-1}\\
 [\Delta(\alpha_{3m-2})]&=1^{10m-10}\cdot2^{m-1}
\end{align*}
\begin{align*}
 &\alpha_{3m}=m+1mm-1,m^3,m^3\\
   &\ \ =kkk-1,k^2k-1,k^2k-1\\
   &\quad\oplus (m-k+1)(m-k)(m-k),(m-k)^2(m-k+1),(m-k)^2(m-k+1)\\
   &\ \ =k+1k-1k,k^3,k^3\\
   &\quad\oplus(m-k+1)(m-k)(m-k-1),(m-k)^3,(m-k)^3\\
   &\ \ =2(k+1kk-1,k^3,k^3)\\
   &\quad\oplus(m-2k-1)(m-2k)(m-2k+1),(m-2k)^3,(m-2k)^3
\allowdisplaybreaks\\
 &\begin{aligned}
  \alpha_{3m}&=\alpha_{3k-1}\oplus\alpha_{3(m-k)+1}&&:9\quad(k=1,\dots,m)\\
  &=\alpha_{3k}\oplus\alpha_{3(m-k)}&&:1\quad(k=1,\ldots,m-1)\\
  &=2\alpha_{3k}\oplus\alpha_{3(m-2k)}&&:1\quad(k=1,\ldots,m-1)
 \end{aligned}\allowdisplaybreaks\\
 &\alpha_{3m-1}=mmm-1,mmm-1,mmm-1\\
   &\ \ =kk-1k-1,kk-1k-1,kk-1k-1\\
   &\quad\oplus (m-k)(m-k+1)(m-k),(m-k)(m-k+1)(m-k),\cdots\\
   &\ \ =k+1kk-1,k^3,k^3\\
   &\quad\oplus(m-k-1)(m-k)(m-k),(m-k)(m-k)(m-k-1),\cdots\\
   &\ \ =2(kkk-1,kkk-1,kkk-1)\\
   &\quad\oplus(m-2k)(m-2k)(m-2k+1),(m-2k)(m-2k)(m-2k+1),\cdots
\allowdisplaybreaks\\
  &\begin{aligned}
  \alpha_{3m-1}
    &=\alpha_{3k-2}(=k,k-1,k-1;\cdots)\oplus\alpha_{3(m-k)+1}&&:4\quad(k=1,\dots,m)\\
  &=\alpha_{3k}\oplus\alpha_{3(m-k)-1}&&:6\quad(k=1,\ldots,m-1)\\
  &=2\alpha_{3k-1}\oplus\alpha_{3(m-2k)+1}&&:1\quad(k=1,\ldots,m-1)
  \end{aligned}
\allowdisplaybreaks\\
 &\alpha_{3m-2}=mm-1m-1,mm-1m-1,mm-1m-1\\
   &\ \ =kkk-1,kkk-1,kkk-1\\
   &\quad\oplus (m-k)(m-k-1)(m-k),(m-k)(m-k-1)(m-k),\cdots\\
   &\ \ =k+1kk-1,k^3,k^3\\
   &\quad\oplus(m-k-1)(m-k-1)(m-k),(m-k)(m-k-1)(m-k-1),\cdots\\
   &\ \ =2(kk-1k-1,kk-1k-1,kk-1k-1)\\
   &\quad\oplus(m-2k)(m-2k+1)(m-2k+1),(m-2k)(m-2k+1)(m-2k+1),\cdots
\allowdisplaybreaks\\
  &\begin{aligned}
  \alpha_{3m-2}
    &=\alpha_{3k-1}(=k,k-1,k-1;\cdots)\oplus\alpha_{3(m-k)-1}&&:4\quad(k=1,\dots,m-1)
      \\
  &=\alpha_{3k}\oplus\alpha_{3(m-k)-2}&&:6\quad(k=1,\ldots,m-1)\\
  &=2\alpha_{3k-2}\oplus\alpha_{3(m-2k)+2}&&:1\quad(k=1,\ldots,m-1)
 \end{aligned}
\end{align*}

The analysis of the other minimal series
\begin{align*}
 \beta_{4m,2}&=(2m+1)(2m-1),m^4,m^4&\beta_{4,2}&=H_4\\
 \beta_{4m,4}&=(2m)^2,m^4,(m+1)m^2(m-1)&\beta_{4,4}&=EO_4\\
 \beta_{4m\pm1}&=(2m)(2m\pm1),(m\pm1)m^3,(m\pm1)m^3&\beta_5&=C_5,\ \beta_3=H_3\\
 \beta_{4m+2}&=(2m+1)^2,(m+1)^2m^2,(m+1)^2m^2
\allowdisplaybreaks\\[4pt]
 \gamma_{6m,2}&=(3m+1)(3m-1),(2m)^3,m^6&\gamma_{6,2}&=D_6=X_6\\
 \gamma_{6m,3}&=(3m)^2,(2m+1)(2m)(2m-1),m^6&\gamma_{6,3}&=EO_6\\
 \gamma_{6m,6}&=(3m)^2,(2m)^3,(m+1)m^4(m-1)
\allowdisplaybreaks\\
 \gamma_{6m\pm1}&=(3m)(3m\pm1),(2m)^2(2m\pm1),m^5(m\pm1)&\gamma_5&=EO_5
\allowdisplaybreaks\\
 \gamma_{6m\pm2}&=(3m\pm1)(3m\pm1),(2m)(2m\pm1)^2,m^4(m\pm1)^2&\gamma_4&=EO_4
\allowdisplaybreaks\\
 \gamma_{6m+3}&=(3m+2)(3m+1),(2m+1)^3,(m+1)^3m^3&\gamma_3&=H_3
\end{align*}
and general $P_{p+1,n}$ will be left to 
the reader as an exercise.
\subsubsection{Relation between series}\label{sec:Bseries}
We have studied the following sets of families of spectral types
of Fuchsian differential equations which are closed
under the irreducible subquotients in the Grothendieck group.
\begin{align*}
 &\{H_n\}&&\text{(hypergeometric family)}\allowdisplaybreaks\\
 &\{P_n\}&&\text{(Jordan-Pochhammer series)}\allowdisplaybreaks\\
 &\{A_n=EO_n\}&&\text{(even/odd family)}\allowdisplaybreaks\\
 &\{B_n,\ C_n,\ H_n\}&&\text{(3 singular points)}\allowdisplaybreaks\\
 &\{C_n,\ H_n\}&&\text{(3 singular points)}\allowdisplaybreaks\\
 &\{D_n,\ E_n,\ H_n\}&&\text{(3 singular points)}\allowdisplaybreaks\\
 &\{F_n,\ G_{2m},\ H_n\}&&\text{(3 singular points)}\allowdisplaybreaks\\
 &\{I_n,\ H_n\}&&\text{(4 singular points)}\allowdisplaybreaks\\
 &\{J_n,\ H_n\}&&\text{(4 singular points)}\allowdisplaybreaks\\
 &\{K_n,\ P_n\}&&\text{($[\tfrac{n+5}{2}]$ singular points)}\allowdisplaybreaks\\
 &\{L_{2m+1},\ K_n,\ P_n\}&&\text{($m+2$ singular points)}\allowdisplaybreaks\\
 &\{M_n,\ P_n\}&&\text{($[\tfrac{n+5}{2}]$ singular points)}
  \qquad\supset\{M_{2m+1},\ P_n\}\allowdisplaybreaks\\
 &\{N_n,\ M_n,\ P_n\}&&\text{($[\tfrac{n+3}{2}]$ singular points)}
  \qquad\supset\{N_{2m+1},\ M_n.\ P_n\}
  \allowdisplaybreaks\\
 &\{P_{4,n}=\delta_n\}&&\text{(4 effective parameters)}\allowdisplaybreaks\\
 &\{\alpha_n\}&&\text{(6 effective parameters and 3 singular points)}
\end{align*}

\index{tuple of partitions!rigid!Yokoyama's list}
Yokoyama classified $\mathbf m=
\bigl(m_{j,\nu}\bigr)_{\substack{0\le j\le p\\1\le \nu\le n_j}}
\in\mathcal P_{p+1}$ such that
\begin{align}
&\text{$\mathbf m$ is irreducibly realizable},
\allowdisplaybreaks\\
&m_{0,1}+\cdots+m_{p-1,1}=(p-1)\ord\mathbf m\quad
(\text{$\mathbf m$ is of Okubo type}),
\allowdisplaybreaks\\
&m_{j,\nu}=1\quad(0\le j\le p-1,\ 2\le\nu\le n_j).
\end{align}
\index{Okubo type}%
The tuple $\mathbf m$ satisfying the above conditions is
in the following list given by \cite[Theorem 2]{Yo}
(cf.~\cite{Ro}).
\medskip

\begin{tabular}{|c|c|c|c|c|}\hline
Yokoyama&type&order& p+1 &tuple of partitions\\ \hline\hline
$\text{I}_n$&$H_n$&$n$ & $3$ & $1^n,n-11,1^n$\\ \hline
$\text{I}^*_n$&$P_n$ & $n$ & $n+1$ & $n-11,n-11,\ldots,n-11$ \\ \hline
$\text{II}_n$&$B_{2n}$&$2n$&$3$ & $n1^n, n1^n,nn-11$\\ \hline
$\text{II}^*_n$ &$I_{2n}$ &$2n$&$4$ & $n1^n,n+11^{n-1},2n-11,nn$\\ \hline
$\text{III}_n$&$B_{2n+1}$&$2n+1$ &$3$&$ n1^{n+1},n+11^n,nn1$\\ \hline
$\text{III}^*_n$&$I_{2n+1}$&$2n+1$& $4$& $ n+11^n,n+11^n,(2n)1,n+1n$\\ \hline
$\text{IV}\phantom{^*}$&$F_6$&$6$& $3$ &$21111,411,222$\\ \hline
$\text{IV}^*$&$N_6$&$6$& $4$ &$411,411,411,42$\\ \hline
\end{tabular}
\subsection{Appell's hypergeometric functions}\label{sec:ApEx}
First we recall the Appell hypergeometric functions.
\index{hypergeometric equation/function!Appell}
\begin{align}
 F_1(\alpha;\beta,\beta';\gamma;x,y)&=\sum_{m,n=0}^\infty
 \frac{(\alpha)_{m+n}(\beta)_m(\beta')_n}
      {(\gamma)_{m+n}m!n!}x^my^n,\label{eq:ApF1}
\allowdisplaybreaks\\
 F_2(\alpha;\beta,\beta';\gamma,\gamma';x,y)&=\sum_{m,n=0}^\infty
 \frac{(\alpha)_{m+n}(\beta)_m(\beta')_n}
      {(\gamma)_m(\gamma')_nm!n!}x^my^n,
\allowdisplaybreaks\\
 F_3(\alpha,\alpha';\beta,\beta';\gamma;x,y)&=\sum_{m,n=0}^\infty
 \frac{(\alpha)_m(\alpha')_n(\beta)_m(\beta')_n}
      {(\gamma)_{m+n}m!n!}x^my^n,\label{eq:F3}
\allowdisplaybreaks\\
 F_4(\alpha;\beta;\gamma,\gamma';x,y)&=\sum_{m,n=0}^\infty
 \frac{(\alpha)_{m+n}(\beta)_{m+n}}
      {(\gamma)_m(\gamma')_nm!n!}x^my^n.
\end{align}
They satisfy the following equations
\begin{align}
  \Bigl((\vartheta_x+\vartheta_y+\alpha)(\vartheta_x+\beta)-\p_x(\vartheta_x+\vartheta_y+\gamma-1)\Bigr)F_1&=0,\\
  \Bigl((\vartheta_x+\vartheta_y+\alpha)(\vartheta_x+\beta)-\p_x(\vartheta_x+\gamma-1)\Bigr)F_2&=0,
\allowdisplaybreaks\\
  \Bigl((\vartheta_x+\alpha)(\vartheta_x+\beta)-\p_x(\vartheta_x+\vartheta_y+\gamma-1)\Bigr)F_3&=0,\\
  \Bigl((\vartheta_x+\vartheta_y+\alpha)(\vartheta_x+\vartheta_y+\beta)
  -\p_x(\vartheta_x+\gamma-1)\Bigr)F_4&=0.
\end{align}
Similar equations hold under the symmetry $x\leftrightarrow y$
with $(\alpha,\beta,\gamma)\leftrightarrow (\alpha',\beta',\gamma')$.

\subsubsection{Appell's $F_1$}
\index{hypergeometric equation/function!Appell!$F_1$}
First we examine $F_1$.  Put
\begin{align*}
u(x,y)&:=\int_0^xt^\alpha(1-t)^\beta(y-t)^{\gamma-1}
      (x-t)^{\lambda-1} dt\qquad(t=xs)\\
      &=\int_0^1 x^{\alpha+\lambda+1}s^\alpha(1-xs)^\beta(y-xs)^{\gamma-1}
        (1-s)^{\lambda-1} ds
        \\
      &=x^{\alpha+\lambda}y^{\gamma-1}\int_0^1
       s^\alpha(1-s)^{\lambda-1}(1-xs)^\beta\Bigl(1-\frac yxs\Bigr)^{\gamma-1} ds,\\
h_x&:= x^\alpha(x-1)^\beta(x-y)^{\gamma-1}.
\end{align*}
Since the left ideal of $\overline W[x,y]$ is not necessarily generated
by a single element, we want to have good generators of
$\RAd(\p_x^{-\lambda})\circ \RAd(h_x)\bigl(W[x,y]\p_x+W[x,y]\p_y\bigr)$
and we have
\begin{align*}
P&:=\Ad(h_x)\p_x=\p_x-\frac{\alpha}{x}-\frac{\beta}{x-1}-\frac{\gamma-1}{x-y},\\
Q&:=\Ad(h_x)\p_y=\p_y+\frac{\gamma-1}{x-y},\\
R&:=xP+yQ=x\p_x+y\p_y-(\alpha+\gamma-1)-\frac{\beta x}{x-1},
\allowdisplaybreaks\\
S&:=\p_x(x-1)R=
(\vartheta_x+1)(\vartheta_x+\vartheta_y-\alpha-\beta-\gamma+1)
-\p_x(\vartheta_x+\vartheta_y-\alpha-\gamma+1)
\allowdisplaybreaks\\
T&:=\p_x^{-\lambda}\circ S\circ\p_x^{\lambda}\\
 &=(\vartheta_x-\lambda+1)(\vartheta_x+\vartheta_y-\alpha-\beta-\gamma-\lambda+1)
-\p_x(\vartheta_x+\vartheta_y-\alpha-\gamma-\lambda+1)
\end{align*}
with
\begin{align*}
a=-\alpha-\beta-\gamma-\lambda+1,\ b=1-\lambda,\ 
   c=2-\alpha-\gamma-\lambda.
\end{align*}
This calculation shows the equation $Tu(x,y)=0$ and 
we have a similar equation by changing 
$(x,y,\gamma,\lambda)\mapsto(y,x,\lambda,\gamma)$.
Note that $TF_1(a;b,b';c;x,y)=0$ with $b'=1-\gamma$.

Putting
\begin{align*}
v(x,z)&=I_{0,x}^\mu(x^\alpha(1-x)^\beta(1-zx)^{\gamma-1})\\
    &=\int_0^xt^\alpha(1-t)^\beta(1-zt)^{\gamma-1}
 (x-t)^{\mu-1} dt\\
      &=x^{\alpha+\mu}\int_0^1s^\alpha(1-xs)^\beta(1-xzs)^{\gamma-1}
 (1-s)^{\mu-1} ds,\allowdisplaybreaks
\intertext{we have}
 &\qquad u(x,y)=y^{\gamma-1}v(x,\tfrac1y),\\
t^\alpha(1-t)^\beta&(1-zt)^{\gamma-1}=
 \sum_{m,n=0}^\infty
  \frac{(-\beta)_m(1-\gamma)_n}{m!n!}
   t^{\alpha+m+n}z^n,\allowdisplaybreaks\\
v(x,z)&=\sum_{m,n=0}^\infty
    \frac{\Gamma(\alpha+m+n+1)(-\beta)_m(1-\gamma)_n}
    {\Gamma(\alpha+\mu+m+n+1)m!n!}
    x^{\alpha+\gamma+m+n}z^n\\
  &=x^{\alpha+\mu}\frac{\Gamma(\alpha+1)}{\Gamma(\alpha+\mu+1)}
    \sum_{m,n=0}^\infty
    \frac{(\alpha+1)_{m+n}(-\beta)_m(1-\gamma)_n}
    {(\alpha+\mu+1)_{m+n}m!n!}x^{m+n}z^n\\
  &=x^{\alpha+\mu}\frac{\Gamma(\alpha+1)}{\Gamma(\alpha+\mu+1)}
    F_1(\alpha+1;-\beta,1-\gamma;\alpha+\mu+1;x,xz).
\end{align*}

Using a versal addition to get the Kummer equation, we introduce
the functions
\index{hypergeometric equation/function!Appell!confluence}
\begin{align*}
 v_c(x,y)&:=\int_0^xt^\alpha(1-ct)^{\frac \beta c}(y-t)^{\gamma-1}
 (x-t)^{\lambda-1},\\
 h_{c,x}&:=x^\alpha(1-cx)^{\frac\beta c}(x-y)^{\gamma-1}.
\end{align*}
Then we have
\begin{align*}
 &R:=\Ad(h_{c,x})(\vartheta_x+\vartheta_y)
 =\vartheta_x+\vartheta_y-(\alpha+\gamma-1)+\frac{\beta x}{1-cx},
 \allowdisplaybreaks\\
 &S:=\p_x(1-cx)R\\
 &\ \ \ =(\vartheta_x+1)
  \bigl(\beta-c(\vartheta_x+\vartheta_y-\alpha-\gamma+1)
  \bigr) + \p_x(\vartheta_x+\vartheta_y-\alpha-\gamma+1),
 \allowdisplaybreaks\\
 &T:=\Ad(\p^{-\lambda})R\\
 &\ \ \ =(\vartheta_x-\lambda+1)
  \bigl(\beta-c(\vartheta_x+\vartheta_y-\lambda-\alpha-\gamma+1)
  \bigr) + \p_x(\vartheta_x+\vartheta_y-\lambda-\alpha-\gamma+1)
\end{align*}
and hence $u_c(x,y)$ satisfies the differential equation
\begin{align*}
 &\Bigl(x(1-cx)\p_x^2+y(1-cx)\p_x\p_y\\
 &\quad{}+\bigl(2-\alpha-\gamma-\lambda
  +(\beta+\lambda-2+c(\alpha+\gamma+\lambda-1))x\bigr)\p_x
  +(\lambda-1)\p_y\\
 &\quad{}-
  (\lambda-1)(\beta+c(\alpha+\gamma+\lambda-1))\Bigr)u=0.
\end{align*}
\subsubsection{Appell's $F_4$}
\index{hypergeometric equation/function!Appell!$F_4$}
To examine $F_4$ we consider the function
\[
 v(x,y):=\int_\Delta s^{\lambda_1}t^{\lambda_2}(st-s-t)^{\lambda_3}
 (1-sx-ty)^\mu ds\,dt
\]
and the transformation
\begin{equation}\label{eq:1-tx}
  J^\mu_x(u)(x)
 :=\int_\Delta u(t_1,\dots,t_n)(1-t_1x_1-\dots-t_nx_n)^\mu
  dt_1\cdots dt_n
\end{equation}
for function $u(x_1,\dots,u_n)$.
For example the region $\Delta$ is given by
\begin{align*}
 v(x,y)&=
 \int_{s\le0,\ t\le0} 
s^{\lambda_1}t^{\lambda_2}(st-s-t)^{\lambda_3}
 (1-sx-ty)^\mu ds\,dt.
\end{align*}
Putting $s\mapsto s^{-1}$, $t\mapsto t^{-1}$ and $|x|+|y|<c<\frac{1}{2}$, Aomoto \cite{Ao} shows
\begin{equation}
 \begin{split}
 &\int_{c-\infty i}^{c+\infty i}
  \int_{c-\infty i}^{c+\infty i} s^{-\gamma}t^{-\gamma^\prime}
 (1-s-t)^{\gamma+\gamma^\prime-\alpha-2}
 \left(1-\frac{x}{s}-\frac{y}{t}\right)^{-\beta}ds\,dt\\
 &\quad=-\frac{4\pi^2\Gamma(\alpha)}{\Gamma(\gamma)\Gamma(\gamma^\prime)
  \Gamma(\alpha-\gamma-\gamma^\prime+2)}
  F_4(\alpha;\beta;\gamma,\gamma^\prime;x,y),
 \end{split}
\end{equation}
which follows from the integral formula
\begin{equation}
 \begin{split}
  &\frac1{(2\pi i)^n}\int_{\frac1{n+1}-\infty i}^{\frac1{n+1}+\infty i}\cdots
   \int_{\frac1{n+1}-\infty i}^{\frac1{n+1}+\infty i}
   \prod_{j=1}^{n}t_j^{-\alpha_j}
   \Bigl(1-\sum_{j=1}^nt_j\Bigr)^{-\alpha_{n+1}}dt_1\cdots dt_n\\
  &\qquad=\frac{\Gamma\bigl(\sum_{j=1}^{n+1}\alpha_j-n\bigr)}
  {\prod_{j=1}^{n+1}\Gamma\bigl(\alpha_j\bigr)}.
 \end{split}
\end{equation}

Since
\begin{align*}
J^\mu_x(u)&=J^{\mu-1}_x(u) - \sum x_\nu J^{\mu-1}_x(x_\nu u)
\intertext{and}
\frac{d}{dt_i}\bigl(u(t)(1-\sum t_\nu x_\nu)^{\mu}\bigr)\\
&\hspace{-2cm}=\frac{d u}{dt_i}(t)(1-\sum t_\nu x_\nu)^{\mu}
 - \mu u(t)x_i(1-\sum t_\nu x_\nu)^{\mu-1},
\end{align*}
we have
\begin{align}
 J^\mu_x(\p_i u)(x)
 &=\mu x_i J^{\mu-1}_x(u)(x)\notag\\
 &=-x_i\int t_i^{-1}u(t)\frac{d}{dx_i}
   \bigl(1-\textstyle\sum x_\nu t_\nu\bigr)^{\mu}dt\notag\\
 &=-x_i\frac{d}{d x_i} J^\mu_x\bigl(\frac u{x_i}\bigr)(x),\notag\\
 J^\mu_x\bigl(\p_i(x_iu)\bigr)&=-x_i\p_i J^\mu_x(u)
 ,\notag\\
 J^\mu_x(\p_i u)&=\mu x_iJ^{\mu-1}_x(u)\notag\\
 &=\mu x_iJ^\mu_x(u)+\mu x_i\textstyle\sum x_\nu J^{\mu-1}_x(x_\nu u)\notag\\
 &=\mu x_iJ^\mu_x(u)+x_i\textstyle\sum J^\mu_x\bigl(\p_\nu(x_\nu u)\bigr)
\notag\\
 &=\mu x_iJ^\mu_x(u) - x_i\textstyle\sum x_\nu \p_\nu J^\mu_x(u)\notag
\intertext{and therefore}
 J^\mu_x(x_i\p_i u)&=(-1-x_i\p_i)J^\mu_x(u),\\
 J^\mu_x(\p_i u)
 &=x_i\bigl(\mu-\textstyle\sum x_\nu\p_\nu\bigr)J^\mu_x(u).
\end{align}
Thus we have
\begin{prop}\label{prop:1-tx}
For a differential operator
\begin{equation}
  P =\sum_{\substack{\alpha=(\alpha_1,\dots,\alpha_n)\in\mathbb Z^n_{\ge0}\\
   \beta=(\beta_1,\dots,\beta_n)\in\mathbb Z^n_{\ge0}}}
   c_{\alpha,\beta}
    \p_{1}^{\alpha_1}\cdots\p_{n}^{\alpha_n}
    \vartheta_{1}^{\beta_1}\cdots\vartheta_{n}^{\beta_n},
\end{equation}
we have 
\begin{equation}
 \begin{split}
 J^\mu_x\bigl(Pu(x)\bigr)&=J^\mu_x(P)J^\mu_x\bigl(u(x)\bigr),\\
 J^\mu_x(P) :\!&= \sum_{\alpha,\,\beta}
   c_{\alpha,\beta}
  \prod_{k=1}^n\bigl(x_k(\mu-\sum_{\nu=1}^n\vartheta_{\nu})\bigr)^{\alpha_k}
  \prod_{k=1}^n\bigl(-\vartheta_{k}-1\bigr)^{\beta_k}.
\end{split}
\end{equation}
\end{prop}
Using this proposition, 
we obtain the system of differential equations satisfied by $J^\mu_x(u)$
from that satisfied by $u(x)$. 
Denoting the Laplace transform of the variable
$x=(x_1,\dots,x_n)$ by $\Lap_x$ (cf.\ Definition~\ref{def:Lap}), we have
\begin{equation}
J_x^\mu L_x^{-1}(\vartheta_i)=\vartheta_i,\quad
J_x^\mu L_x^{-1}(x_i)=x_i\bigl(\mu-\sum_{\nu=1}^n\vartheta_\nu\bigr).
\end{equation}

We have
\begin{align}
 \Ad\bigl(x^{\lambda_1} y^{\lambda_2} (xy - x - y)^{\lambda_3}\bigr)\p_x
&=\p_x -\frac{\lambda_1}x - \frac{\lambda_3(y-1)}{xy-x-y},
\allowdisplaybreaks\notag\\
 \Ad\bigl(x^{\lambda_1} y^{\lambda_2} (xy - x - y)^{\lambda_3}\bigr)\p_y
&=\p_y -\frac{\lambda_2}y - \frac{\lambda_3(x-1)}{xy-x-y},
\allowdisplaybreaks\notag\\
 \Ad\bigl(x^{\lambda_1} y^{\lambda_2} (xy - x - y)^{\lambda_3}\bigr)
  &\bigl(x(x-1)\p_x\bigr)\notag\\
  &\hspace{-2cm}=x(x-1)\p_x-\lambda_1(x-1)
 - \frac{\lambda_3(x-1)(xy-x)}{xy-x-y},
\allowdisplaybreaks\notag\\
 \Ad\bigl(x^{\lambda_1} y^{\lambda_2} (xy - x - y)^{\lambda_3}\bigr)
 &\bigl(x(x-1)\p_x-y\p_y\bigr)\notag\\
 &\hspace{-2cm}=x(x-1)\p_x-y\p_y-\lambda_1(x-1)
 -\lambda_2 - \lambda_3(x-1)\notag\\
 &\hspace{-2cm}=x\vartheta_x-\vartheta_x-\vartheta_y-(\lambda_1+\lambda_3)x
  +\lambda_1-\lambda_2+\lambda_3,
\allowdisplaybreaks\notag\\
 \begin{split}
 \p_x\Ad\bigl(x^{\lambda_1} y^{\lambda_2} (xy - x - y)^{\lambda_3}\bigr)
 &\bigl(x(x-1)\p_x-y\p_y\bigr)
\\
 &\hspace{-2cm}=\p_xx(\vartheta_x-\lambda_1-\lambda_3)
   -\p_x(\vartheta_x+\vartheta_y-\lambda_1+\lambda_2-\lambda_3)
 \end{split}\label{eq:AP40}
\end{align}
and
\begin{align*}
 &J^\mu_{x,y}\bigl
 (\p_xx(\vartheta_x-\lambda_1-\lambda_3)
  -\p_x(\vartheta_x+\vartheta_y-\lambda_1+\lambda_2-\lambda_3)\bigr)\\
 &\quad=\vartheta_x(1+\vartheta_x+\lambda_1+\lambda_3)
   -x(-\mu+\vartheta_x+\vartheta_y)(2+\vartheta_x+\vartheta_y+\lambda_1-\lambda_2+\lambda_3).
\end{align*}
Putting
\begin{align*}
 T&:=(\vartheta_x+\vartheta_y-\mu)
    (\vartheta_x+\vartheta_y+\lambda_1-\lambda_2+\lambda_3+2)
    -\p_x(\vartheta_x+\lambda_1+\lambda_3+1)
\end{align*}
with
\begin{align*}
\quad\alpha=-\mu,\quad \beta=\lambda_1-\lambda_2+\lambda_3+2,\quad
\gamma=\lambda_1+\lambda_3+2,
\end{align*}
we have $Tv(x,y)=0$ and moreover it satisfies a similar equation by replacing
$(x,y,\lambda_1,\lambda_3,\gamma)$ by $(y,x,\lambda_3,\lambda_1,\gamma')$.
Hence $v(x,y)$ is a solution of the system of differential
equations satisfied by $F_4(\alpha;\beta;\gamma,\gamma';x,y)$.

In the same way we have
\begin{gather}
 \Ad\bigl(x^{\beta-1}y^{\beta'-1}(1-x-y)^{\gamma-\beta-\beta'-1}\bigr)\vartheta_x
 =\vartheta_x-\beta+1+\frac{(\gamma-\beta-\beta'-1) x}{1-x-y},\notag
 \allowdisplaybreaks\\
 \begin{split}
 &\Ad\bigl(x^{\beta-1}y^{\beta'-1}(1-x-y)^{\gamma-\beta-\beta'-1}\bigr)
 (\vartheta_x-x(\vartheta_x+\vartheta_y))\\
 &\quad=\vartheta_x-x(\vartheta_x+\vartheta_y)
 -\beta+1+(\gamma-3)x\\
 &\quad=(\vartheta_x-\beta+1)-x(\vartheta_x+\vartheta_y-\gamma+3),
 \end{split}\label{eq:F120}\allowdisplaybreaks\\
 \begin{split}
 &J^\mu_{x,y}\bigl(\p_x(\vartheta_x-\beta+1)-\p_xx(\vartheta_x+\vartheta_y-\gamma+3)\bigr)
 \notag\\
 &\quad=x(-\vartheta_x-\vartheta_y+\mu)(-\vartheta_x-\beta)
  +\vartheta_x(-2-\vartheta_x-\vartheta_y-\gamma+3)
   \notag\\
 &\quad=x\Bigl((\vartheta_x+\vartheta_y-\mu)(\vartheta_x+\beta)
  -\p_x(\vartheta_x+\vartheta_y+\gamma-1)\Bigr).\notag
 \end{split}
\end{gather}
which is a differential operator killing $F_1(\alpha;\beta,\beta';\gamma;x,y)$
by putting $\mu=-\alpha$ and in fact we have
\begin{align*}
 &\iint_{\substack{s\ge 0,\ t\ge 0\\1-s-t\ge0}}
  s^{\beta-1}t^{\beta'-1}(1-s-t)^{\gamma-\beta-\beta'-1}
  (1-sx-ty)^{-\alpha}ds\,dt\allowdisplaybreaks\\
 &\quad=\iint_{\substack{s\ge 0,\ t\ge 0\\1-s-t\ge0}}\sum_{m,\,n=0}^\infty
    s^{\beta+m-1}t^{\beta'+n-1}(1-s-t)^{\gamma-\beta-\beta'-1}
    \frac{(\alpha)_{m+n}x^my^n}{m!n!}ds\,dt\allowdisplaybreaks\\
 &\quad=\sum_{m,\,n=0}^\infty
   \frac{\Gamma(\beta+m)\Gamma(\beta'+n)\Gamma(\gamma-\beta-\beta')}
   {\Gamma(\gamma+m+n)}\cdot\frac{(\alpha)_{m+n}}{m!n!}x^my^n\allowdisplaybreaks\\
 &\quad=\frac{\Gamma(\beta)\Gamma(\beta')\Gamma(\gamma-\beta-\beta')}
  {\Gamma(\gamma)} F_1(\alpha;\beta,\beta';\gamma;x,y).
\end{align*}
Here we use the formula
\begin{equation}
 \iint_{\substack{s\ge 0,\ t\ge 0\\1-s-t\ge0}}
  s^{\lambda_1-1}t^{\lambda_2-1}(1-s-t)^{\lambda_3-1}ds\,dt
 =\frac{\Gamma(\lambda_1)\Gamma(\lambda_2)\Gamma(\lambda_3)}
  {\Gamma(\lambda_1+\lambda_2+\lambda_3)}.
\end{equation}
\subsubsection{Appell's $F_3$}
\index{hypergeometric equation/function!Appell!$F_3$}
Since
\begin{align*}
 T_3:\!&=
 J^{-\alpha'}_yx^{-1} J^{-\alpha}_x
 \bigl(\p_x(\vartheta_x-\beta+1)-\p_xx(\vartheta_x+\vartheta_y-\gamma+3)\bigr)\\
 &=J^{-\alpha'}_y\bigl((-\vartheta_x-\alpha)(-\vartheta_x-\beta)
  +\p_x(-\vartheta_x+\vartheta_y-\gamma+2)\bigr)\\
 &=(\vartheta_x+\alpha)(\vartheta_x+\beta)
  -\p_x(\vartheta_x+\vartheta_y+\gamma-1)
\end{align*}
with \eqref{eq:F120}, the operator $T_3$ kills the function
\begin{align*}
 &\iint_{\substack{s\ge 0,\ t\ge 0\\ 1-s-t\ge 0}}
 s^{\beta-1}t^{\beta'-1}(1-s-t)^{\gamma-\beta-\beta'-1}(1-xs)^{-\alpha}
 (1-yt)^{-\alpha'}ds\,dt\\
 &\quad=\iint_{\substack{s\ge 0,\ t\ge 0\\ 1-s-t\ge 0}}\sum_{m,\,n=0}^\infty
 s^{\beta+m-1}t^{\beta'+n-1}(1-s-t)^{\gamma-\beta-\beta'-1}
 \frac{(\alpha)_m(\alpha')_nx^my^n}{m!n!}ds\,dt\\
 &\quad =\sum_{m,\,n=0}^\infty
 \frac{\Gamma(\beta+m)\Gamma(\beta'+n)\Gamma(\gamma-\beta-\beta')(\alpha)_m(\alpha')_n}
 {\Gamma(\gamma+m+n)m!n!}x^my^n\\
 &\quad=\frac{\Gamma(\beta)\Gamma(\beta')\Gamma(\gamma-\beta-\beta')}
  {\Gamma(\gamma)}
  F_3(\alpha,\alpha';\beta,\beta';\gamma;x,y).
\end{align*}
Moreover since
\begin{align*}
T'_3:\!&=\Ad(\p_x^{-\mu})\Ad(\p_y^{-\mu'})
 \bigl
 ((\vartheta_x+1)(\vartheta_x-\lambda_1-\lambda_3)
  -\p_x(\vartheta_x+\vartheta_y-\lambda_1+\lambda_2-\lambda_3)\bigr)\\
 &=(\vartheta_x+1-\mu)(\vartheta_x-\lambda_1-\lambda_3-\mu)
  -\p_x(\vartheta_x+\vartheta_y-\lambda_1+\lambda_2-\lambda_3-\mu-\mu')
\end{align*}
with \eqref{eq:AP40} and
\begin{align*}
\alpha&=-\lambda_1-\lambda_3-\mu,\quad
\beta=1-\mu,\quad
\gamma=-\lambda_1+\lambda_2-\lambda_3-\mu-\mu'+1,
\end{align*}
the function
\begin{equation}\label{eq:F3I2}
 u_3(x,y)
  :=\int_\infty^y\int_\infty^x s^{\lambda_1} t^{\lambda_2}(st-s-t)^{\lambda_3}
  (x-s)^{\mu-1}(y-t)^{\mu'-1}ds\,dt
\end{equation}
satisfies $T'_3u_3(x,y)=0$. Hence $u_3(x,y)$ is a solution of
the system of the equations that 
$F_3(\alpha,\alpha';\beta,\beta';\gamma;x,y)$ satisfies.
\subsubsection{Appell's $F_2$}
\index{hypergeometric equation/function!Appell!$F_2$}
Since
\begin{align*}
 &\p_x\Ad\bigl(x^{\lambda_1-1}(1-x^{\lambda_2-1})\bigr)x(1-x)\p_x\\
 &\quad=
 \p_x x(1-x)\p_x-(\lambda_1-1)\p_x+\p_x(\lambda_1+\lambda_2-2)x\\
 &\quad=\p_xx(-\vartheta_x+\lambda_1+\lambda_2-2)
   +\p_x(\vartheta-\lambda_1+1)
\end{align*}
and
\begin{align*}
T_2:\!&=J^\mu_{x,y}\bigl(\p_xx(-\vartheta_x+\lambda_1+\lambda_2-2)
   +\p_x(\vartheta_x-\lambda_1+1)\bigr)\\
&
 =-\vartheta_x(\vartheta_x+1+\lambda_1+\lambda_2-2)
  +x(\mu - \vartheta_x-\vartheta_y)(-1-\vartheta_x-\lambda_1+1)
\\&
 =x\Bigl((\vartheta_x+\lambda_1)(\vartheta_x+\vartheta_y-\mu)
  -\p_x(\vartheta_x+\lambda_1+\lambda_2-1)\Bigr)
\end{align*}
with
\[
\alpha=-\mu,\quad\beta=\lambda_1,\quad\gamma=\lambda_1+\lambda_2,
\]
the function
\begin{align*}
 u_2(x,y)&:=\int_0^1\!\int_0^1
 s^{\lambda_1-1}(1-s)^{\lambda_2-1}
 t^{\lambda'_1-1}(1-t)^{\lambda'_2-1}
 (1-xs-yt)^{\mu}ds\,dt\\
 &=\int_0^1\!\int_0^1\sum_{m,\,n=0}^\infty s^{\lambda_1+m-1}(1-s)^{\lambda_2-1}
    t^{\lambda_1'+n-1}(1-t)^{\lambda_2'-1}\frac{(-\mu)_{m+n}}{m!n!}x^my^n ds\,dt\\
 &=\sum_{m,\,n=0}^\infty
  \frac{\Gamma(\lambda_1+m)\Gamma(\lambda_2)}
    {\Gamma(\lambda_1+\lambda_2+m)}
  \frac{\Gamma(\lambda_1'+n)\Gamma(\lambda_2')}
    {\Gamma(\lambda_1'+\lambda_2'+m)}
  \frac{(-\mu)_{m+n}}{m!n!}x^my^n\\
 &=\frac{\Gamma(\lambda_1)\Gamma(\lambda_2)
   \Gamma(\lambda_1')\Gamma(\lambda_2')}
   {\Gamma(\lambda_1+\lambda_2)\Gamma(\lambda_1'+\lambda_2')}
   \sum_{m,\,n=0}^\infty\frac{(\lambda_1)_m(\lambda_1')_n(-\mu)_{m+n}}
   {(\lambda_1+\lambda_2)_m(\lambda_1'+\lambda_2')_nm!n!}x^my^n
\end{align*}
is a solution of the equation $T_2u=0$
that $F_2(\alpha;\beta,\beta';\gamma,\gamma';x,y)$ satisfies.

Note that the operator $\tilde T_3$ transformed from $T'_3$ by the coordinate 
transformation $(x,y)\mapsto (\frac1x,\frac1y)$ equals
\begin{align*}
 \tilde T_3&=
  (-\vartheta_x+\alpha)(-\vartheta_x+\beta)
  -x(-\vartheta_x)(-\vartheta_x-\vartheta_y+\gamma-1)\\
  &=(\vartheta_x-\alpha)(\vartheta_x-\beta)
  -x\vartheta_x(\vartheta_x+\vartheta_y-\gamma+1)
\end{align*}
and the operator
\begin{align*}
 \Ad(x^{-\alpha} y^{-\alpha'})\tilde T_3&=
  \vartheta_x(\vartheta_x+\alpha-\beta)
  -x(\vartheta_x+\alpha)(\vartheta_x+\vartheta_y+\alpha+\alpha'-\gamma+1)
\end{align*}
together with the operator obtained by the transpositions
$x\leftrightarrow y$, $\alpha\leftrightarrow\alpha'$ and 
$\beta\leftrightarrow \beta'$
defines the system of the equations satisfied by the functions
\begin{equation}
 \begin{cases}
   F_2(\alpha+\alpha'-\gamma+1;\alpha,\alpha';
  \alpha-\beta+1,\alpha'-\beta'+1;x,y),\\
   x^{-\alpha'}y^{-\alpha'}F_3(\alpha,\alpha';\beta,\beta';\gamma;\frac1x,\frac1y),
 \end{cases}
\end{equation}
which also follows from the integral representation \eqref{eq:F3I2}
with the transformation $(x,y,s,t)\mapsto(\frac1x,\frac1y,\frac1s,\frac1t)$.
\subsection{\texttt{Okubo} and \texttt{Risa/Asir}}\label{sec:okubo}
Most of our results in this paper are constructible and
they can be explicitly calculated and implemented in computer
programs.

The computer program \texttt{okubo} \cite{O5} written by 
the author handles combinatorial calculations in this paper 
related to tuples of partitions.
It generates basic tuples (cf.~\S\ref{sec:Exbasic})
and rigid tuples (cf.~\S\ref{sec:rigidEx}),
calculates the reductions originated by Katz and Yokoyama,
the position of accessory parameters in the universal operator
(cf.~Theorem~\ref{thm:univmodel} iv)) and direct decompositions etc.

The author presented Theorem~\ref{thm:c} in the case when $p=3$
as a conjecture in the fall of 2007, which was proved in May in
2008 by a completely different way from the proof given in 
\S\ref{sec:C1}, which is a generalization of the original proof of 
Gauss's summation formula of the hypergeometric series
explained in \S\ref{sec:C2}. 
The original proof of Theorem~\ref{thm:c} in the case when $p=3$
was reduced to the combinatorial equality \eqref{eq:numdec}.
The author verified \eqref{eq:numdec} by \texttt{okubo} 
and got the concrete connection coefficients for the rigid
tuples $\mathbf m$ satisfying $\ord\mathbf m\le 40$.
Under these conditions 
($\ord\mathbf m\le 40,\, p=3,\, m_{0,n_0}=m_{1,n_1}=1$) 
there are 4,111,704 independent connection coefficients 
modulo obvious symmetries  and it took
about one day to got all of them by a personal computer with \texttt{okubo}. 

Several operations on differential operators such as
additions and middle convolutions defined in \S\ref{sec:frac}
can be calculated by a computer algebra and the author wrote a 
program for their results under \texttt{Risa/Asir}, 
which gives a reduction procedure of the operators 
(cf.~Definition~\ref{def:redGRS}),
integral representations and series expansions of the 
solutions (cf.~Theorem~\ref{thm:expsol}), 
connection formulas (cf.~Theorem~\ref{thm:conG}), 
differential operators (cf.~Theorem~\ref{thm:univmodel} iv)), 
the condition of their reducibility (cf.~Corollary~\ref{cor:irred} i)), 
recurrence relations (cf.~Theorem~\ref{thm:shifm1} ii))
etc.\ for any given spectral type or Riemann scheme \eqref{eq:IGRS} and 
displays the results using \TeX.
This program for Risa/Asir written by the author contains many useful functions
calculating rational functions, Weyl algebra and matrices.
These programs can be obtained from

\quad\texttt{http://www.math.kobe-u.ac.jp/Asir/asir.html}

\quad\texttt{ftp://akagi.ms.u-tokyo.ac.jp/pub/math/muldif}

\quad\texttt{ftp://akagi.ms.u-tokyo.ac.jp/pub/math/okubo}.
\section{Further problems}\label{sec:prob}
\subsection{Multiplicities of spectral parameters}
Suppose a Fuchsian differential equation and its middle convolution
are given.
Then we can analyze the corresponding transformation of a global 
structure of its local solution associated with an 
eigenvalue of the monodromy generator at a singular point if 
the eigenvalue is free of multiplicity.

When the multiplicity of the eigenvalue is larger than one,
we have not a satisfactory result for the transformation
(cf.~Theorem~\ref{thm:conG}).
The value of a generalized connection coefficient defined by 
Definition~\ref{def:GC} may be interesting.
Is the procedure in Remark~\ref{rem:Cproc}  always valid?
In particular, is there a general result assuring 
Remark~\ref{rem:Cproc} (1) (cf.~Remark~\ref{rem:Cgamma})?
Are the multiplicities of zeros of the generalized connection 
coefficients of a rigid Fuchsian differential equation free?
\subsection{Schlesinger canonical form}\label{prob:Sch}
Can we define a natural \textsl{universal} Fuchsian system of Schlesinger 
canonical form \eqref{eq:SCF} with a given realizable spectral type?
Here we recall Example~\ref{ex:univSch}.

Let $P_{\mathbf m}$ be the universal operator in
Theorem~\ref{thm:univmodel}.
Is there a natural system of Schlesinger canonical form 
which is isomorphic to the equation $P_{\mathbf m}u=0$ 
together with the explicit correspondence between them?
\subsection{Apparent singularities}
Katz \cite{Kz} proved that any irreducible rigid local system 
is constructed from the trivial system by successive applications of
middle convolutions and additions and it is proved in this paper 
that the system is realized by a single differential equation 
without an apparent singularity.

In general, an irreducible local system cannot be realized by a
single differential equation without an apparent singularity but
it is realized by that with apparent singularities.
Hence it is expected that there exist some natural operations of
single differential equations with apparent singularities which
correspond to middle convolutions of local systems or systems 
of Schlesinger canonical form.

The Fuchsian ordinary differential equation satisfied by an 
important special function often hasn't an apparent singularity 
even if the spectral type of the equation is not rigid.
Can we understand the condition that a $W(x)$-module 
has a generator so that it satisfies a differential equation 
without an apparent singularity?
Moreover it may be interesting to study the existing of contiguous 
relations among differential equations with fundamental spectral 
types which have no apparent singularity. 
\subsection{Irregular singularities}
Our fractional operations defined in \S\ref{sec:frac} give transformations
of ordinary differential operators with polynomial coefficients, which have
irregular singularities in general.
The reduction of ordinary differential equations under these operations is 
a problem to be studied.
Note that versal additions and middle convolutions construct such differential 
operators from the trivial equation.

A similar result as in this paper is obtained for certain classes of 
ordinary differential equations with irregular singularities (cf.~\cite{Hi}).

A ``versal" path of integral in an integral representation of the solution
and a ``versal" connection coefficient and Stokes multiplier should be studied.  
Here ``versal" means a natural expression corresponding to the versal addition.

We define a complete model with a given spectral type as follows.
For simplicity we consider differential operators without singularities 
at the origin.
For a realizable irreducible tuple of partitions
$\mathbf m=\bigl(m_{j,\nu}\bigr)_{\substack{0\le j\le p\\1\le\nu\le n_j}}$ 
of a positive integer $n$ Theorem~\ref{thm:univmodel} constructs the universal
differential operator
\begin{equation}\label{eq:univcj}
 P_{\mathbf m}=\prod_{j=1}^p(1-c_jx)^n
\cdot \frac{d^n}{dx^n}+\sum_{k=0}^{n-1} a_k(x,c,\lambda,g)\frac{d^k}{dx^k}
\end{equation}
with the Riemann scheme
\begin{align*}
 &\begin{Bmatrix}
  x=\infty & \frac1{c_1} & \cdots &\frac1{c_p}\\
   [\lambda_{0,1}]_{(m_{0,1})} & [\lambda_{1,1}]_{(m_{1,1})} 
    & \cdots &[\lambda_{p,1}]_{(m_{p,1})} \\
  \vdots & \vdots & \vdots & \vdots\\
  [\lambda_{0,n_0}]_{(m_{0,n_0})}  & [\lambda_{1,n_1}]_{(m_{1,n_1})} & 
    \cdots &[\lambda_{p,n_p}]_{(m_{p,n_p})} 
 \end{Bmatrix}
\end{align*}
and the Fuchs relation
\begin{align*}
\sum_{j=0}^p\sum_{\nu=1}^{n_j}m_{j,\nu}
  \lambda_{j,\nu}=n-\frac{\idx\mathbf m}2.
\end{align*}
Here $c=(c_0,\dots,c_p)$, $\lambda=(\lambda_{j,\nu})$ and
$g=(g_1,\dots,g_N)$ are parameters. 
We have $c_ic_j(c_i-c_j)\ne 0$ for $0\le i<j\le p$. 
The parameters $g_j$ are called accessory parameters and we have
$\idx\mathbf m=2-2N$.
We call the Zariski closure $\overline P_{\mathbf m}$ of $P_{\mathbf m}$
in $W[x]$ the \textsl{complete model}
\index{differential equation/operator!complete model} of differential
operators with the spectral type $\mathbf m$, whose dimension equals
$p+\sum_{j=0}^p n_j + N-1$.
It is an interesting problem to analyze the complete model 
$\overline P_{\mathbf m}$.

When $\mathbf m=11,11,11$, the complete model equals
\[
  (1-c_1x)^2(1-c_2x)^2\tfrac{d^2}{dx^2}
 -(1-c_1x)(1-c_2x)(a_{1,1}x+a_{1,0})\tfrac d{dx}
  +a_{0,2}x^2+a_{0,1}x+a_{0,0},
\]
whose dimension equals 7.
Any differential equation defined by the operator belonging
to this complete model is transformed into a Gauss hypergeometric
equation, a Kummer equation, an Hermite equation or an airy equation
by a suitable gauge transformation and a coordinate transformation.
A good understanding together with a certain completion of our operators
is required even in this fundamental example.
It is needless to say that the good understanding is important in 
the case when $\mathbf m$ is fundamental.
\subsection{Special parameters}
Let $P_{\mathbf m}$ be the universal operator of the form \eqref{eq:univcj}
for an irreducible tuple of partition $\mathbf m$.
When a decomposition $\mathbf m=\mathbf m'+\mathbf m''$ with realizable
tuples of partitions $\mathbf m'$ and $\mathbf m''$ is given, 
Theorem~\ref{thm:prod} gives the values of the parameters of
$P_\mathbf m$ corresponding to the product $P_{\mathbf m'}P_{\mathbf m''}$.
A $W(x,\xi)$-automorphism of $P_{\mathbf m}u=0$ gives 
a transformation of the parameters $(\lambda,g)$, which is a contiguous 
relation and called Schlesinger transformation in the case of systems of 
Schlesinger canonical form. 
How can we describe the values of the parameters obtained in this way 
and characterize their position in all the values of the parameters when 
the universal operator is reducible?
In general, they are not all even in a rigid differential equation.
A direct decomposition $32,32,32,32=12,12,12,12\oplus
2(10,10,10,10)$ of a rigid tuples $32,32,32,32$ gives this example
(cf.~\eqref{eq:32idx}).

Analyse the reducible differential equation with an irreducibly realizable
spectral type.
This is interesting even when $\mathbf m$ is a rigid tuple.
For example, describe the monodromy of its solutions.

Describe the characteristic exponents of the generalized Riemann scheme
with an irreducibly realizable spectral type such that there exists a 
differential operator with the Riemann scheme which is
outside the universal operator (cf.~Example \ref{ex:outuniv} and 
Remark~\ref{rem:inuniv}).
In particular, when the spectral type is not fundamental nor simply reducible,
does there exist such a differential operator?

The classification of rigid and simply reducible spectral types coincides with
that of indecomposable objects described in \cite[Theorem~2.4]{MWZ}.
Is there some meaning in this coincidence?

Has the condition \eqref{eq:CSR} a similar meaning in the case of 
Schlesinger canonical form?
What is the condition on the local system or a (single) Fuchsian differential 
equation which has a realization of a system of Schlesinger canonical form? 

Give the condition so that the monodromy group is finite.
Give the condition so that the centralizer of the monodromy
is the set of scalar multiplications.

Suppose $\mathbf m$ is fundamental.
Study the condition so that the connection coefficients is a quotient
of the products of gamma functions as in Theorem~\ref{thm:c}
or the solution has an integral representation only by using elementary
functions.

\subsection{Shift operators}
Calculate the polynomial function $c_{\mathbf m}(\epsilon;\lambda)$ of $\lambda$
defined in Theorem~\ref{thm:shiftC}.
Is it square free?
See Conjecture~\ref{conj:shift}.

Is the shift operator $R_{\mathbf m}(\epsilon,\lambda)$ Fuchsian?

Study the shift operators given in Theorem~\ref{thm:sftUniv}.

Study the condition on the characteristic exponents and accessory parameters 
assuring the existence of a shift operator for a Fuchsian differential operator 
with a fundamental spectral type.

Study the shift operator or Schlesinger transformation of a system of 
Schlesinger canonical form with a fundamental spectral type.
When is it not defined or when is it not bijective?

\subsection{Several variables}
We have analyzed Appell hypergeometric equations in \S\ref{sec:ApEx}.
What should be the geometric structure of singularities of more general 
system of equations when it has a good theory?

Describe or define operations of differential operators that are
fundamental to analyze good systems of differential equations.

A series expansion of a local solution of a rigid ordinal differential equation
indicates that it may be natural to think that the solution is a restriction 
of a solution of a system of differential equations with several variables
(cf.~Theorem~\ref{thm:expsol} and \S\ref{sec:PoEx}--\ref{sec:GHG}).
Study the system.

\subsection{Other problems}
\begin{itemize}
\item
For a rigid decomposition $\mathbf m=\mathbf m'\oplus\mathbf m''$, 
can we determine whether $\alpha_{\mathbf m'}\in \Delta(\mathbf m)$ or
$\alpha_{\mathbf m''}\in \Delta(\mathbf m)$ (cf.~Proposition~\ref{prop:wm} iv))?

\item
Are there analyzable series $\mathcal L$ of rigid tuples of partitions
different from the series given in \S\ref{sec:Rob}?
Namely, $\mathcal L\subset \mathcal P$, the elements of $\mathcal L$
are rigid, the number of isomorphic classes of
$\mathcal L\cap \mathcal P^{(n)}$ are bounded for $n\in\mathbb Z_{>0}$
and the following condition is valid.

Let $\mathbf m=k\mathbf m'+\mathbf m''$ with $k\in\mathbb Z_{>0}$ and
rigid tuples of partitions $\mathbf m$, $\mathbf m'$ and $\mathbf m''$.
If $\mathbf m\in\mathcal L$, 
then $\mathbf m'\in\mathcal L$ and $\mathbf m''\in\mathcal L$.
Moreover for any $\mathbf m''\in\mathcal L$, this decomposition 
$\mathbf m=k\mathbf m'+\mathbf m''$ exists with 
$\mathbf m\in\mathcal L$,  $\mathbf m'\in\mathcal L$ and $k\in\mathbb Z_{>0}$.
Furthermore $\mathcal L$ is indecomposable.  Namely, if 
$\mathcal L=\mathcal L'\cup\mathcal L''$ so that 
$\mathcal L'$ and $\mathcal L''$
satisfy these conditions, then $\mathcal L'=\mathcal L$ or 
$\mathcal L''=\mathcal L$.

\item
Characterize the ring of automorphisms and that of endomorphisms
of the localized Weyl algebra $W(x)$.

\item
In general, different procedures of the reduction of the universal
operator $P_{\mathbf m}u=0$ give different integral representations
and series expansions of its solution (cf.~Example~\ref{ex:localGauss},
Remark~\ref{rem:Irep} and the last part of \S\ref{sec:PoEx}).  
Analyze the difference.
\end{itemize}
\section{Appendix}\label{sec:appendix}
In this section we give a theorem which is proved by K.~Nuida.
The author greatly thanks to K.~Nuida for allowing the author 
to put the theorem with its proof in this section.

Let $(W,S)$ be a \textsl{Coxeter system}.
\index{Coxeter system}
Namely, $W$ is a group with the set $S$ of generators and
under the notation $S=\{s_i\,;\,i\in I\}$, 
the fundamental relations among the generators are
\begin{equation}
 s_k^2=(s_is_j)^{m_{i,j}}=e\text{ \ and \ }m_{i,j}=m_{j,i}
 \text{ \ for \  }\forall i,\ j,\ k\in I\text{ \ satisfying \ }i\ne j.
\end{equation}
Here $m_{i,j}\in\{2,3,4,\ldots\}\cup\{\infty\}$ and the condition
$m_{i,j}=\infty$ means $(s_is_j)^m\ne e$ for any $m\in\mathbb Z_{>0}$. 
Let $E$ be a real vector space with the basis set $\Pi=\{\alpha_i\,;\,i\in I\}$
and define a symmetric bilinear form $(\ \ |\ \ )$ on $E$ by
\begin{equation}
 (\alpha_i|\alpha_i)=2\text{ \ and \ }
 (\alpha_i|\alpha_j)=-2\cos\frac{\pi}{m_{i,j}}.
\end{equation}
Then the \textsl{Coxeter group} $W$ is naturally identified with the 
reflection group generated by the reflections $s_{\alpha_i}$ with respect to 
$\alpha_i$ $(i\in I)$.
The set $\Delta_\Pi$ of the roots of $(W,S)$ equals $W\Pi$,
which is a disjoint union of the set of positive roots $\Delta_\Pi^+:=
\Delta_\Pi\cap \sum_{\alpha\in\Pi}\mathbb Z_{\ge 0}\alpha$ and the set of negative
roots $\Delta_\Pi^-:=-\Delta_\Pi^+$.
For $w\in W$ the length $L(w)$ is the minimal number $k$
with the expression $w=s_{i_1}s_{i_2}\cdots s_{i_k}$ ($i_1,\dots,i_k\in I$).
Defining $\Delta_\Pi(w):=\Delta_\Pi^+\cap w^{-1}\Delta_\Pi^-$, 
we have $L(w)=\#\Delta_\Pi(w)$.

Fix $\beta$ and $\beta'\in\Delta_\Pi$ and put
\begin{equation}
W^\beta_{\!\beta'}:=\{w\in W\,;\,\beta'=w\beta\}
\text{ \ and \ }W^\beta:=W^\beta_\beta.
\end{equation}

\begin{thm}[K.~Nuida]\label{thm:Nuida}
Retain the notation above.
Suppose\/ $W^\beta_{\!\beta'}\ne\emptyset$ and
\begin{equation}\label{eq:noloop}
\begin{split}
  &\text{there exist no sequence $s_{i_1},s_{i_2},\ldots s_{i_k}$ of 
  elements of $S$ such that}\\
  &\begin{cases}
   k\ge 3,\\
   s_{i_\nu}\ne s_{i_\nu'}\quad(1\le \nu<\nu'\le k),\\
    m_{i_\nu, i_{\nu+1}}\text{ and }m_{i_1,i_k}
    \text{\ are odd integers}\quad(1\le\nu<k).
  \end{cases}
\end{split}
\end{equation}
Then an element $w\in W^\beta_{\!\beta'}$ is uniquely determined
by the condition 
\begin{equation}\label{eq:minCox}
  L(w)\le L(v)\quad(\forall v\in W^\beta_{\!\beta'}).
\end{equation}
\end{thm}
\begin{proof}
Put $\Delta_\Pi^\beta:=\{\gamma\in\Delta_\Pi^+\,;\,(\beta|\gamma)=0\}$.
First note that the following lemma.
\begin{lem}\label{lem:Coxeter}
If\/ $w\in W^\beta_{\!\beta'}$ satisfies \eqref{eq:minCox}, then\/
$w\Delta_\Pi^\beta\subset\Delta_\Pi^+$.
\end{lem}

In fact, if $w\in W^\beta_{\beta'}$ satisfies \eqref{eq:minCox} and
there exists $\gamma\in\Delta_\Pi^\beta$ satisfying $w\gamma\in\Delta_\Pi^-$,
then there exists $j$ for a minimal expression 
$w=s_{i_1}\cdots s_{i_{L_\Pi(w)}}$ such that 
$s_{i_{j+1}}\cdots s_{j_{L_\Pi(w)}}\gamma=\alpha_{i_j}$, which implies 
$W^\beta_{\beta'}\ni v
:=ws_\gamma=s_{i_1}\cdots s_{i_{j-1}}s_{i_{j+1}}\cdots s_{i_{L_\Pi(w)}}$
and contradicts to \eqref{eq:minCox}.

It follows from \cite{Br} that the assumption \eqref{eq:noloop}
implies that $W^\beta$ is generated by 
$\{s_\gamma\,;\,\gamma\in \Delta_\Pi^\beta\}$.
Putting $\Pi^\beta=\Delta_\Pi^\beta\setminus\{r_1\gamma_1+r_2\gamma_2\in
\Delta_\Pi^\beta \,;\,\gamma_2\notin\mathbb R \gamma_1,\ 
\gamma_j\in \Delta_\Pi^\beta\text{ and }r_j>0\text{ for }j=1,2\}$ and
$S^\beta=\{s_\gamma\,;\,\gamma\in \Pi^\beta\}$, the pair $(W^\beta,S^\beta)$ is 
a Coxeter system and moreover the minimal length of the expression of 
$w\in W^\beta$ by the product of the elements of $S^\beta$ equals 
$\#\bigl(\Delta_\Pi^\beta\cap w^{-1}\Delta_\Pi^-\bigr)$ 
(cf.~\cite[Theorem~2.3]{Nu}).

Suppose there exist two elements $w_1$ and $w_2\in W^\beta_{\!\beta'}$ satisfying
$L(w_j)\le L(v)$ for any $v\in W^\beta_{\!\beta'}$ and $j=1$, $2$.
Since $e\ne w_1^{-1}w_2\in W^\beta$, 
there exists $\gamma\in \Delta_\Pi^\beta$ such that 
$w_1^{-1}w_2\gamma \in \Delta_\Pi^-$.
Since $-w_1^{-1}w_2\gamma\in \Delta^\beta_\Pi$,
Lemma~\ref{lem:Coxeter} assures 
$-w_2\gamma=w_1(-w_1^{-1}w_2\gamma)\in\Delta_\Pi^+$, 
which contradicts to Lemma~\ref{lem:Coxeter}.
\end{proof}
The above proof shows the following corollary.
\begin{cor}\label{cor:Nuida}
Retain the assumption in\/ {\rm Theorem~\ref{thm:Nuida}.}
For an element $w\in W^\beta_{\beta'}$,
the condition \eqref{eq:minCox} is equivalent to 
$w\Delta^\beta_\Pi\subset \Delta_\Pi^+$.

Let $w\in W^\beta_{\beta'}$ satisfying \eqref{eq:minCox}.
Then
\begin{equation}
 W^\beta_{\beta'}=w\bigl\langle s_\gamma\,;\,(\gamma|\beta)=0,\ 
\gamma\in\Delta_\Pi^+\bigr\rangle.
\end{equation}
\end{cor}

\printindex
\end{document}